\documentclass[reqno, 11pt]{amsart}
\usepackage{bm,amsmath,amsthm,amssymb,mathtools,verbatim,amsfonts,tikz-cd,mathrsfs}
\usepackage{enumitem}
\usepackage{float}
\usepackage{adjustbox}

\usepackage[hidelinks,colorlinks=true,linkcolor=blue, citecolor=black,linktocpage=true]{hyperref}
\usepackage{graphicx}
\usepackage[alphabetic]{amsrefs}
\usepackage{caption}
\usepackage{morefloats}
\usepackage[top=1.4in, bottom=1.2in, left=1.4in, right=1.4in,marginpar=1in]{geometry}
\usepackage{subfigure}

\usetikzlibrary{matrix,arrows,decorations.pathmorphing}

\usepackage{chngcntr}
\counterwithin{figure}{section}

\newtheorem{thm}{Theorem}[section]
\newtheorem{prop}[thm]{Proposition}
\newtheorem{conj}[thm]{Conjecture}
\newtheorem{cor}[thm]{Corollary}
\newtheorem{lem}[thm]{Lemma}

\theoremstyle{definition}
\newtheorem{define}[thm]{Definition}

\theoremstyle{remark}
\newtheorem{rem}[thm]{Remark}
\newtheorem{example}[thm]{Example}

 \newcommand{\fg}{\mathit{fg}}

\newcommand{\ve}[1]{\boldsymbol{\mathbf{#1}}}

\newcommand{\R}{\mathbb{R}}

\newcommand{\Z}{\mathbb{Z}}
\newcommand{\N}{\mathbb{N}}
\newcommand{\Q}{\mathbb{Q}}

\newcommand{\rmE}{\mathrm{E}}
\newcommand{\rmI}{\mathrm{I}}

\renewcommand{\d}{\partial}
\renewcommand{\subset}{\subseteq}
\renewcommand{\supset}{\supseteq}
\renewcommand{\tilde}{\widetilde}
\renewcommand{\bar}{\overline}

\newcommand{\iso}{\cong}

\DeclareMathOperator{\co}{{co}}

\DeclareMathOperator{\Aut}{{Aut}}

\DeclareMathOperator{\End}{{End}}
\DeclareMathOperator{\ev}{{ev}}

\DeclareMathOperator{\gr}{{gr}}

\DeclareMathOperator{\Hom}{{Hom}}

\DeclareMathOperator{\id}{{id}}
\DeclareMathOperator{\ind}{{ind}}
\DeclareMathOperator{\Int}{{int}}

\DeclareMathOperator{\MCG}{{MCG}}
\DeclareMathOperator{\Mod}{{mod}}
\DeclareMathOperator{\MOD}{{Mod}}
\DeclareMathOperator{\Mor}{{Mor}}

\DeclareMathOperator{\Spin}{{Spin}}

\DeclareMathOperator{\RoJ}{{RoJ}}
\DeclareMathOperator{\RiJ}{{RiJ}}

\DeclareMathOperator{\Span}{{Span}}
\DeclareMathOperator{\Sym}{{Sym}}

\DeclareMathOperator{\Tw}{{Tw}}

\newcommand{\lk}{\mathrm{lk}}

\newcommand{\bA}{\mathbb{A}}
\newcommand{\bB}{\mathbb{B}}

\newcommand{\bD}{\mathbb{D}}
\newcommand{\bE}{\mathbb{E}}
\newcommand{\bF}{\mathbb{F}}

\newcommand{\bH}{\mathbb{H}}
\newcommand{\bI}{\mathbb{I}}

\newcommand{\bO}{\mathbb{O}}
\newcommand{\bP}{\mathbb{P}}

\newcommand{\bT}{\mathbb{T}}
\newcommand{\bU}{\mathbb{U}}

\newcommand{\bX}{\mathbb{X}}

\newcommand{\cA}{\mathcal{A}}
\newcommand{\cB}{\mathcal{B}}
\newcommand{\cC}{\mathcal{C}}
\newcommand{\cD}{\mathcal{D}}
\newcommand{\cE}{\mathcal{E}}
\newcommand{\cF}{\mathcal{F}}
\newcommand{\cG}{\mathcal{G}}
\newcommand{\cH}{\mathcal{H}}
\newcommand{\cI}{\mathcal{I}}
\newcommand{\cJ}{\mathcal{J}}
\newcommand{\cK}{\mathcal{K}}
\newcommand{\cL}{\mathcal{L}}
\newcommand{\cM}{\mathcal{M}}
\newcommand{\cN}{\mathcal{N}}
\newcommand{\cO}{\mathcal{O}}

\newcommand{\cQ}{\mathcal{Q}}
\newcommand{\cR}{\mathcal{R}}
\newcommand{\cS}{\mathcal{S}}
\newcommand{\cT}{\mathcal{T}}
\newcommand{\cU}{\mathcal{U}}
\newcommand{\cV}{\mathcal{V}}
\newcommand{\cW}{\mathcal{W}}
\newcommand{\cX}{\mathcal{X}}
\newcommand{\cY}{\mathcal{Y}}
\newcommand{\cZ}{\mathcal{Z}}

\newcommand{\frA}{\mathfrak{A}}

\newcommand{\frC}{\mathfrak{C}}

\newcommand{\frM}{\mathfrak{M}}

\newcommand{\fra}{\mathfrak{a}}

\newcommand{\frc}{\mathfrak{c}}

\newcommand{\frm}{\mathfrak{m}}

\newcommand{\frs}{\mathfrak{s}}

\newcommand{\scA}{\mathscr{A}}

\newcommand{\scH}{\mathscr{H}}

\newcommand{\scK}{\mathscr{K}}

\newcommand{\scS}{\mathscr{S}}
\newcommand{\scT}{\mathscr{T}}
\newcommand{\scU}{\mathscr{U}}
\newcommand{\scV}{\mathscr{V}}
\newcommand{\scW}{\mathscr{W}}

\newcommand{\cCFL}{\mathcal{C\!F\!L}}
\newcommand{\cCFK}{\mathcal{C\hspace{-.5mm}F\hspace{-.3mm}K}}

\newcommand{\CF}{\mathit{CF}}

\newcommand{\HF}{\mathit{HF}}

\newcommand{\CFK}{\mathit{CFK}}

\newcommand{\HFL}{\mathit{HFL}}

\newcommand{\HLE}{\mathit{EHL}}

\newcommand\HFhat{\widehat{\HF}}

\newcommand{\PD}{\mathit{PD}}

\newcommand{\xs}{\ve{x}}
\newcommand{\ys}{\ve{y}}
\newcommand{\zs}{\ve{z}}
\newcommand{\ws}{\ve{w}}

\newcommand{\ps}{\ve{p}}
\newcommand{\qs}{\ve{q}}
\newcommand{\as}{\ve{\alpha}}
\newcommand{\bs}{\ve{\beta}}
\newcommand{\gs}{\ve{\gamma}}
\newcommand{\ds}{\ve{\delta}}

\renewcommand{\a}{\alpha}
\renewcommand{\b}{\beta}
\newcommand{\g}{\gamma}
\newcommand{\dt}{\delta}
\newcommand{\veps}{\varepsilon}

\newcommand{\app}{\mathrm{app}}
\newcommand{\red}{\mathrm{red}}

\usepackage{leftidx}

\DeclareMathOperator{\Cone}{{Cone}}
\DeclareMathOperator{\St}{{St}}

\numberwithin{equation}{section}

\newcommand{\ar}{\mathrm{a.r.}}

\newcommand{\llsquare}{[\hspace{-.5mm}[}
\newcommand{\rrsquare}{]\hspace{-.5mm}]}

\allowdisplaybreaks

\DeclareMathOperator{\LTS}{\mathrm{LTS}}

\DeclareMathOperator{\val}{val}

\newcommand{\cell}{\mathrm{cell}}

\newcommand{\hatbox}{\mathrel{\hat{\boxtimes}}}
\newcommand{\hatotimes}{\mathrel{\hat{\otimes}}}
\newcommand{\vecotimes}{\mathrel{\vec{\otimes}}}

\newcommand{\otimesstar}{\otimes^*}
\newcommand{\otimesshriek}{\otimes^!}

\title[Bordered link surgery]{Bordered manifolds with torus boundary and the link surgery formula}

\author{Ian Zemke}
\address{Department of Mathematics\\University of Oregon\\ Eugene, OR, USA}
\email{izemke@uoregon.edu}
\thanks{IZ was partially supported by NSF grants DMS-1703685 and 2204375 and a Sloan Fellowship.}
\begin{document}

\begin{abstract}
In this paper, we develop a theory of bordered $\HF^-$ using the link surgery formula of Manolescu and Ozsv\'{a}th. We interpret their link surgery complexes as type-$D$ modules over an associative algebra $\cK$, which we introduce. We prove a connected sum formula, which we interpret as an $A_\infty$-tensor product over our algebra $\cK$.  Topologically, this connected sum formula may be viewed as a formula for gluing along torus boundary components.  We compute several important examples.  We show that the dual knot formula of Hedden--Levine and Eftekhary may be interpreted as the $DA$-bimodule for a particular diffeomorphism of the torus. As another example, if $K_1$ and $K_2$ are knots in $S^3$, and $Y$ is obtained by gluing the complements of $K_1$ and $K_2$ together using an orientation reversing diffeomorphism of their boundaries, then our theory may be used to compute $\CF^-(Y)$ from $\CFK^\infty(K_1)$ and $\CFK^\infty(K_2)$. We additionally compute the type-$D$ modules for rationally framed solid tori. Our theory also computes the Heegaard Floer homology of all 3-manifolds which bound the plumbing of a tree of disk bundles over 2-spheres. In a subsequent article, we use this work to verify N\'{e}methi's conjecture about lattice homology.
\end{abstract}

\maketitle

\tableofcontents

\section{Introduction}

In this paper, we define several algebraic objects associated to compact, oriented 3-manifolds with parametrized torus boundary, and prove a gluing formula.  For the purposes of this paper, we use the following definition (compare \cite{LOTBordered}*{Section~1.1}):
\begin{define}
 A \emph{bordered manifold with torus boundary} consists of a compact and oriented manifold $M$ such that $\d M\iso \bT^2$ and $M$ is equipped with a choice of oriented basis $(\mu,\lambda)$ of $H_1(\d M)$.
\end{define}

If $M_1$ and $M_2$ are two bordered manifolds with torus boundaries, a natural gluing operation is to glue $\mu_1$ to $\mu_2$ and glue $\lambda_1$ to $-\lambda_2$. This is motivated by the following fact. Suppose that $M_1$ and $M_2$ are the complements of two oriented knots $K_1$ and $K_2$ in $S^3$, which are given integral framings $\lambda_1$ and $\lambda_2$. Let $\phi\colon \d M_1\to \d M_2$ denote the orientation reversing diffeomorphism which sends $\mu_1$ to $\mu_2$ and $\lambda_1$ to $-\lambda_2$.
Then there is a diffeomorphism
\[
M_1\cup_\phi M_2\iso S_{\lambda_1+\lambda_2}^3(K_1\# K_2).
\]
  This basic topological fact is reviewed in Appendix \ref{appendix}.

 The algebraic objects we define are reformulations of Manolescu and Ozsv\'{a}th's link surgery formula \cite{MOIntegerSurgery}, recast using the framework of type-$D$ and type-$A$ modules of Lipshitz, Ozsv\'{a}th and Thurston \cite{LOTBordered}.

 Using the above topological observation about surgeries on connected sums, our gluing formula amounts to a formula for the behavior of the link surgery formula under connected sums.

\subsection{The knot surgery algebra}

We begin by describing our associative algebra $\cK$, which is an algebra over the idempotent ring $\ve{I}=\ve{I}_0\oplus \ve{I}_1$. Here each $\ve{I}_i$ is a copy of $\bF=\Z/2\Z$. We define 
\[
\ve{I}_0\cdot \cK \cdot \ve{I}_0=\bF[\scU,\scV]\quad \text{and} \quad \ve{I}_1\cdot \cK\cdot \ve{I}_1=\bF[\scU,\scV,\scV^{-1}],
\] 
where $\bF[\scU,\scV]$ denotes a polynomial ring in two variables and $\bF[\scU,\scV,\scV^{-1}]$ is its localization at $\scV$. We set $\ve{I}_0\cdot \cK \cdot \ve{I}_1=\{0\}$.  There are two special algebra elements $\sigma,\tau\in \ve{I}_1\cdot \cK\cdot \ve{I}_0$, and we define
\[
\ve{I}_1\cdot \cK\cdot \ve{I}_0=\bF[\scU,\scV,\scV^{-1}]\otimes \langle \sigma\rangle \oplus 
\bF[\scU,\scV,\scV^{-1}]\otimes \langle \tau\rangle
\]
where $\langle \sigma\rangle$ and $\langle \tau\rangle$ denote copies of $\bF$. The elements $\sigma$ and $\tau$ satisfy the relations
\[
\sigma\scU=\scU\sigma\quad \text{and} \quad \sigma \scV=\scV \sigma
\]
\[
\tau  \scU=\scV^{-1} \tau\quad \text{and} \quad \tau  \scV=\scU\scV^2 \tau.
\]
We typically write $U$ for the product $\scU\scV$. 

The algebra $\cK$ is an associative algebra, i.e. we can think of it as an $A_\infty$-algebra which has vanishing actions $\mu_j$ for $j\neq 2$. Although seemingly asymmetric between $\scU$ and $\scV$, as well as between $\sigma$ and $\tau$, the above algebra admits a more symmetrical description. See Remark~\ref{rem:symmetry}. We work with an asymmetric presentation of the algebra $\cK$ because it mirrors an asymmetric construction of the Heegaard Floer surgery formulas which frequently simplifies computations.

\subsection{Surgery formulas as \texorpdfstring{$\cK$}{K}-modules}

 Ozsv\'{a}th and Szab\'{o} \cite{OSIntegerSurgeries} proved that if $K\subset S^3$ is a knot, then the Heegaard Floer homology groups $\HF^-(S^3_\lambda(K))$ are computable from the knot Floer complex of $K$ \cite{OSKnots} \cite{RasmussenKnots}. If $\lambda$ is an integral framing on $K$, their construction produced a mapping cone complex $\bX_\lambda(K)$, which satisfied
\[
H_*(\bX_\lambda(K))\iso \ve{\HF}^-(S_\lambda^3(K)).
\]
(The bold font denotes coefficients in the power series ring $\bF\llsquare U\rrsquare $).

Manolescu and Ozsv\'{a}th \cite{MOIntegerSurgery} extended this formula to links in $S^3$. They proved that if $L\subset S^3$ is a link with integral framing $\Lambda$, then there is a complex $\cC_{\Lambda}(L)$, defined using the link Floer homology of $L$  \cite{OSLinks}, such that
\[
H_*(\cC_{\Lambda}(L))\iso \ve{\HF}^-(S^3_{\Lambda}(L)).
\]

In Section~\ref{sec:mapping-cone-formula}, we describe how to naturally reinterpret the data of  Ozsv\'{a}th and Szab\'{o}'s mapping cone complex $\bX_\lambda(K)$  both in terms of a right type-$D$ module $\cX_\lambda(K)^\cK$ and a left type-$A$ module ${}_{\cK} \cX_\lambda(K)$. 

In Section~\ref{sec:link-surgery-A-infty}, we extend this interpretation to the link surgery formula. If $\ell$ is a positive integer, we define the $\ell$-component \emph{link algebra} to be
\[
 \cL_{\ell}=\otimes^{\ell}_{\bF} \cK=\cK\otimes_{\bF}\cdots \otimes_{\bF} \cK.
 \]
 To an $\ell$-component link $L\subset S^3$, we define in Section~\ref{sec:link-surgery-A-infty} a type-$D$ module
\[
\cX_{\Lambda}(L)^{\cL_\ell}.
\]

\begin{rem}
We think of the map $v$ in Ozsv\'{a}th and Szab\'{o}'s mapping cone formula as being encoded by the algebra element $\sigma\in \cK$. The map $h$ is encoded by the algebra element $\tau\in \cK$.
\end{rem}

The description of the link surgery formula given by Manolescu and Ozsv\'{a}th is an infinite direct product of free modules over $\bF\llsquare U_1,\dots, U_\ell\rrsquare $. In particular, it is very large and sometimes challenging to do direct computations with. On the other hand, the type-$D$ module $\cX_{\Lambda}(L)^{\cL_\ell}$ that we describe contains all of the algebraic information of the link surgery formula, but is always finite-dimensional over $\bF$. For example, the complement of the $0$-framed trefoil has the type-$D$ module shown in Figure~\ref{fig:trefoil}.

\begin{figure}[ht]
\begin{tikzcd}[labels=description, column sep=2cm]
& \xs^0
	\ar[dl, "\scV"]
	\ar[dr, "\scU"]
\\[-.2cm]
\ys^0 
	\ar[dr, "\scV^{-2}(U\sigma+\tau)"]
&& \zs^0
	\ar[dl, "\sigma+U \tau"]
\\[.5cm]
&\zs^1
\end{tikzcd}
\caption{The type-$D$ module for the $0$-framed trefoil complement. The superscript indicates the idempotent.}
\label{fig:trefoil}
\end{figure}

In our paper, the most fundamental object is the type-$D$ module $\cX_{\Lambda}(L)^{\cL_{\ell}}$. To change a type-$D$ algebra factor of $\cK$ to a type-$A$ algebra factor, we tensor with an algebraically defined ``identity bimodule''
\[
{}_{\cK| \cK} [\bI^{\Supset}].
\]
Ignoring completions,  ${}_{\cK| \cK} [\bI^{\Supset}]$ is a type-$A$ module over $\cK\otimes_{\bF} \cK$. (We write $\cK|\cK$ instead of $\cK\otimes_{\bF} \cK$ to indicate a specific behavior with respect to completions, which we call the \emph{split Alexander condition}).

The bimodule ${}_{\cK| \cK} [\bI^{\Supset}]$ extends to a type-$AA$ bimodule ${}_{\cK| \cK}[\bI^{\Supset}]_{\bF[U]}$, which is also useful to consider.
Unlike the type-$D$ modules $\cX_{\Lambda}(L)^{\cL_{\ell}}$, which are all finitely generated over $\bF$, the type-$A$ and $DA$ bimodules in our theory are usually not finitely generated.

Manolescu and Ozsv\'{a}th's link surgery formula $\cC_{\Lambda}(L)$ is recovered by tensoring our type-$D$ module $\cX_{\Lambda}(L)^{\cL_{\ell}}$ with $\ell$-copies of the 0-framed solid torus module, which we denote by ${}_{\cK} [\cD_0]_{\bF[U]}$. Here, a ``0-framed solid torus'' refers to the complement of a 0-framed unknot in $S^3$.

 Algebraic completions (e.g. direct products and power series coefficients) interact with the algebra $\cK$ in subtle and interesting ways. The algebra $\cK$ has a natural filtration by right ideals
 \[
J_0\supset J_1\supset \cdots.
 \]\
 This filtration equips $\cK$ with a topology where the subspaces $J_n$ form a basis of opens centered at 0. This equips $\cK$ with the structure of a \emph{linear topological chiral algebra}, which is important for our theory.  In Section~\ref{sec:Algebra-K}, we define algebraic categories of \emph{Alexander modules}, which consist of linear topological vector spaces with appropriately continuous actions of $\cK$. These are the categories where our link surgery modules live.

\subsection{Pairing theorems}

One of the central results of our paper is a connected sum formula for the Manolescu--Ozsv\'{a}th integer surgery formula. Suppose $L_1$ and $L_2$ are two links in $S^3$, which have framings $\Lambda_1$ and $\Lambda_2$, as well as distinguished components $K_1\subset L_1$ and $K_2\subset L_2$. In Section~\ref{sec:link-surgery-A-infty}, we describe how to compute $\cC_{\Lambda_1+\Lambda_2}(L_1\# L_2)$ in terms of $\cC_{\Lambda_1}(L_1)$ and $\cC_{\Lambda_2}(L_2)$. Here, $L_1\# L_2$ denotes the connected sum of $L_1$ and $L_2$ along $K_1$ and $K_2$, and $\Lambda_1+\Lambda_2$ denotes the framing where we sum the framings on $K_1$ and $K_2$ and leave the others unchanged. We will interpret this connected sum formula as an $A_\infty$ tensor product over the algebra $\cK$. We use the algebraic formalism of Lipshitz, Ozsv\'{a}th and Thurston \cite{LOTBordered} \cite{LOTBimodules} \cite{LOTDiagonals}, though our proof of the pairing theorem is independent of their proof in the standard bordered setting \cite{LOTBordered}.

Our formula is simplest to state when we are working with knots:

\begin{thm}\label{thm:intro-pairing-theorem}
 Suppose that $(Y_1,K_1)$ and $(Y_2,K_2)$ are two knots in integer homology 3-spheres with integral framings $\lambda_1$ and $\lambda_2$. Then
\[
\bX_{\lambda_1+\lambda_2}(Y_1\# Y_2, K_1\#K_2)\simeq \cX_{\lambda_1}(Y_1,K_1)^{\cK}\hatbox  {}_{\cK} \cX_{\lambda_2}(Y_2,K_2). 
\]
\end{thm}
Here, $\hatbox $ denotes a completed version of Lipshitz, Ozsv\'{a}th and Thurston's box tensor product \cite{LOTBordered}.  There is also a version of the above theorem for links:
\begin{thm}\label{thm:pairing} Suppose that $L_1$ and $L_2$ are two links in $S^3$ with integral framings $\Lambda_1$ and $\Lambda_2$, and that we have two distinguished components $K_1\subset L_1$ and $K_2\subset L_2$. Then
\[
\cC_{\Lambda_1+\Lambda_2}(L_1\#L_2)\simeq \cX_{\Lambda_1}(L_1)^{\cK}\hatbox {}_{\cK} \cX_{\Lambda_2}(L_2).
\]
\end{thm}

In Section~\ref{sec:link-surgery-A-infty} we describe several more general versions of the pairing theorem. One version allows us to compute 
\begin{equation}
\cX_{\Lambda_1+\Lambda_2}(L_1\#L_2)^{\cL_{\ell_1+\ell_2-1}}\hatbox {}_{\cK}[\cD_{0}]_{\bF[U]}.
\label{eq:intro-expanded-pairing}
\end{equation}
 In the above, $\cD_0$ denotes the module for the 0-framed solid torus. In the tensor product, we are tensoring the $\cK$-action on $\cD_0$ with the $\cK$-factor of $\cL_{\ell_1+\ell_2-1}$ which corresponds to $K_1\# K_2$.
 The reader should think of tensoring $\cD_0$ as corresponding to gluing in a solid torus $K_1\#K_2$. See Remark~\ref{rem:tensor-DD-with-A} for more details about this tensor product.
 
    We will show in Section~\ref{sec:AAA-identity} that the module in Equation~\eqref{eq:intro-expanded-pairing} is homotopy equivalent to an appropriate tensor product of the three modules
\[
\cX_{\Lambda_1}(L_1)^{\cL_{\ell_1}},\quad \cX_{\Lambda_2}(L_2)^{\cL_{\ell_2}}, \quad \text{and} \quad {}_{\cK| \cK} [\bI^{\Supset}]_{\bF[U]}.
\]
We think of tensoring with ${}_{\cK| \cK} [\bI^{\Supset}]_{\bF[U]}$ as corresponding to gluing torus boundary components together.

In some sense, the tensor product formulas above use up the two actions of $\cK$, and hence are not convenient for taking connected sums of more than two links along a single component. To take the connected sum of more than two components, we define \emph{pair-of-pants} bimodules
\[
{}_{\cK| \cK} W_{\a\b,\a}^{\cK}\quad \text{and} \quad {}_{\cK| \cK} W_{\a \b, \b}^{\cK}.
\]
In Section~\ref{sec:pair-of-pants} we study these bimodules, and we prove that they can be used to compute the type-$D$ module over $\cK$ after taking the connected sum of link components. There are additional modules for other combinations of $\a$ and $\b$ subscripts. There are multiple modules because of different choices of \emph{arc systems} in the construction of the surgery formula. (See Section~\ref{sec:arc-systems-intro} for more details).

In Section~\ref{sec:bordered-pair-of-pants}, we relate the pair-of-pants bimodules to a surgery description of the 3-manifold $P\times S^1$, where $P$ is a pair-of-pants (i.e. a twice punctured disk).

   \begin{rem} The above pairing theorems give formulas for gluing two bordered manifolds along torus boundary components.  Another important operation is to glue two boundary components of a single bordered 3-manifold together (i.e. self-gluing). Our pairing theorems do not cover this. We will investigate this in a future work.
   \end{rem}

\subsection{Splicing knot complements}

A basic topological operation is to glue the complements of two knots together using an orientation reversing diffeomorphism of their boundaries. 3-manifolds obtained in this manner are referred to as \emph{splices}. Dehn surgery is a special case of splicing.

Splicing in the context of Heegaard Floer homology has been studied by many authors. See \cite{EftekharySplices} \cite{HLSplices} \cite{HanselmanSplices} \cite{KLTSplicing} for some examples. Using bordered Heegaard Floer homology \cite{LOTBordered}, it is possible to compute $\HFhat$ of spliced manifolds. Despite the interest in splicing, there is to date no general description of $\HF^-$ of splices.

In this paper, we give a description of $\HF^-$ of splices in terms of knot Floer homology. If $\phi\colon \bT^2\to \bT^2$ is an orientation preserving diffeomorphism, then we will describe a bimodule ${}_{\cK} [\phi]^{\cK}$. We think of this bimodule as the bordered invariant for $[0,1]\times \bT^2$ with basis $(\mu,-\lambda)$ on $\{0\}\times \bT^2$ (where we write $\mu$ and $\lambda$ for a standard basis of $H_1(\bT^2)$), and $(\phi(\mu),\phi(\lambda))$ on $\{1\}\times \bT^2$.
 
 \begin{thm}
  Let $K_1$ and $K_2$ be knots in $S^3$ and let $M_1$ and $M_2$ be their complements. Let $\phi\colon \d M_1\to -\d M_2$ be an orientation preserving diffeomorphism. Then
  \[
  \ve{\CF}^-(M_1\cup_\phi M_2)_{\bF\llsquare U\rrsquare }\simeq \cX_{\lambda_1}(K_1)^{\cK}\hatbox  {}_{\cK} [\phi]^{\cK}\hatbox  {}_{\cK} \cX_{\lambda_2}(K_2)_{\bF\llsquare U\rrsquare }.
  \]
 \end{thm}
 
 The bimodule ${}_{\cK} [\phi]^{\cK}$ is constructed by factoring the diffeomorphism $\phi$ into the following generators of $\MCG(\bT^2)$:
\begin{enumerate}
\item $\phi(\mu)= \mu$ and $\phi(\lambda)=\lambda\pm \mu$.
\item $\phi(\mu)= \pm \lambda$ and $\phi(\lambda)=\mp \mu$.
\end{enumerate}
In both of the above cases, the bimodule ${}_{\cK} [\phi]^{\cK}$ is obtained from the surgery complex for the Hopf link. We remark that the Hopf link has complement $[0,1]\times \bT^2$, so taking a connected sum with the Hopf link and then surgering on the original knot has the effect of changing the boundary parametrization.

In general, we do not yet have a closed formula for the bimodule ${}_{\cK} [\phi]^{\cK}$ for an arbitrary diffeomorphism $\phi$, nor do our techniques imply independence of the bimodule ${}_{\cK} [\phi]^{\cK}$ from the above presentation. We plan to address both of these questions in a future work.

\begin{rem}
 When $K_1$ and $K_2$ are knots in $S^3$, the modules $\cX_{\lambda_1}(K_1)^{\cK}$ and ${}_{\cK}\cX_{\lambda_2}(K_2)_{\bF\llsquare U\rrsquare }$ are determined by the full knot Floer complexes $\CFK^\infty(K_1)$ and $\CFK^\infty(K_2)$, up to homotopy equivalence. 
\end{rem}

\subsection{Dual knots}

As a special case, our techniques recover the dual knot formula of Eftekhary \cite{EftekharyDuals} and Hedden--Levine   \cite{HeddenLevineSurgery} by taking a connected sum with the Hopf link. We recall that they computed the knot Floer complex of the dual knot $\mu\subset S_\lambda^3(K)$ in terms of the knot Floer complex $\CFK^\infty(K)$. Note that the dual knot $\mu\subset S^3_\lambda(K)$ may be obtained by taking the connected sum of $K$ with a Hopf link, and then performing $\lambda$-framed surgery on $K$.

In our present paper case, taking the connected sum of $K$ with a Hopf link and performing $\lambda$-framed surgery corresponds to taking the tensor product of $\cX_{\lambda}(K)^{\cK}$ with the Hopf link surgery complex, viewed as a $DA$-bimodule ${}_{\cK} \cH_{\Lambda}^{\cK}$. Using homological perturbation theory we describe a smaller model, which we denote by  ${}_{\cK}\cZ_{\Lambda}^{\cK}$. Taking the box tensor product with this minimal model and then restricting to idempotent $\ve{I}_0$ recovers the model of Eftekhary and Hedden--Levine (modulo differing notational conventions). We sketch the equivalence of our model of the dual knot complex with the Eftekhary and Hedden--Levine model in Section~\ref{sec:comparison-EHL}. 

Our analysis of the Hopf link complex builds off of joint work of the author with Hendricks, Hom and Stoffregen in \cite{HHSZDuals}, where an extension of the dual knot formula is proven for Hendricks and Manolescu's involutive knot Floer homology \cite{HMInvolutive}.

\subsection{Examples}

We compute many basic examples in this paper. The most fundamental is the type-$D$ module for the Hopf link. We compute this in Section~\ref{sec:Hopf-links-comp}, and compute a minimal model of the corresponding $DA$-bimodule in Section~\ref{sec:minimal-model-Hopf}. This example is important for applications and computations. Combined with the tensor product formulas, it also gives a computation of the Heegaard Floer homology of all 3-manifolds obtained as the boundary of a plumbing of disk bundles over $S^2$, plumbed along a tree. This is explored further in \cite{ZemLattice}, wherein the equivalence of Heegaard Floer homology and the lattice homology of N\'{e}methi \cite{NemethiAR} \cite{NemethiLattice} is proven.

In Section~\ref{sec:solid-tori} we describe the type-$D$ modules of rationally framed solid tori $\cD_{p/q}^{\cK}$ for $p/q\in \Q\cup \{\infty\}$. The type-$D$ modules $\cD_{p/q}^{\cK}$ can be viewed as naturally recovering the rational surgeries complex $\bX_{p/q}(K)$ of Ozsv\'{a}th and Szab\'{o} \cite{OSRationalSurgeries} in the sense that 
\[
\bX_{p/q}(K)\iso \cD_{p/q}^{\cK}\hatbox {}_{\cK}\cX_{0}(K)
\]
for any knot $K$ in $S^3$. 

\begin{rem} If $L\subset S^3$, then by tensoring type-$D$ modules of the form $\cD_{p/q}^{\cK}$ with the type-$A$ module ${}_{\cK|\cdots| \cK} \cX(L)$, we obtain a version of the link surgery formula which allows for rational surgeries on $L$.
\end{rem}

 As another example, we show that there is a homotopy equivalence of type-$D$ modules
\[
\cD_\infty^{\cK}\simeq \Cone(f^1\colon \cD_n^{\cK}\to \cD_{n+1}^{\cK}),
\]
for some type-$D$ morphism $f^1$. Tensoring with this homotopy equivalence recovers the surgery exact triangle in Heegaard Floer homology.  This example is reminiscent of the original setting of bordered Heegaard Floer homology \cite{LOTBordered}*{Section~11.2}

\subsection{Arc systems}

\label{sec:arc-systems-intro}

We also prove several important technical results about the link surgery formula of Manolescu and Ozsv\'{a}th \cite{MOIntegerSurgery}. We recall that the link surgery formula requires, for each knot component $K_i\subset L$, a choice of embedded arc $\scA_i\subset S^3$ which connects two basepoints on $K_i$. We furthermore require the arcs to be pairwise disjoint and to be disjoint from $L$ except at their endpoints. We call any such collection $\scA$ a \emph{system of arcs}. To a system of arcs, Manolescu and Ozsv\'{a}th's construction produces a chain complex $\cC_{\Lambda}(L,\scA)$.

Given a knot $K\subset Y$, there are two natural choices of arc systems on $K$. These are gotten by connecting the two basepoints $w,z\in K$ with an arc that runs parallel to $K$. We say the arc system is \emph{beta parallel} if the  arc runs from $z$ to $w$ and is oriented in the same direction as $K$. We say that the arc system is \emph{alpha parallel} if it runs from $z$ to $w$ and is oriented oppositely to $K$. Given a Heegaard knot splitting $\Sigma$ for $K$, the arc system is \emph{alpha parallel} if the arc is parallel to the subarc of $K$ which lies in the alpha handlebody, and similarly for beta parallel arcs. To a link $L\subset Y$, there are $2^{|L|}$ arc systems which can be constructed via this procedure, where we choose each arc to be either alpha parallel or beta parallel. Our theory also allows arc systems where some of the arcs are neither alpha parallel or beta parallel. Although these seem less natural, they appear naturally in our proof of the connected sum formula.

There is an analogous asymmetry in Lipshitz, Ozsv\'{a}th and Thurston's theory \cite{LOTBordered}. In their theory, for a bordered 3-manifold $Y$ they consider Heegaard diagrams with boundary where some set of the attaching curves are properly embedded arcs. For each boundary component of $Y$, they must choose whether the arcs which abut that boundary component are alpha arcs or beta arcs. See Figure~\ref{fig:45} for an example of a bordered Heegaard diagram with alpha arcs.

\begin{figure}[h]
\begingroup%
  \makeatletter%
  \providecommand\color[2][]{%
    \errmessage{(Inkscape) Color is used for the text in Inkscape, but the package 'color.sty' is not loaded}%
    \renewcommand\color[2][]{}%
  }%
  \providecommand\transparent[1]{%
    \errmessage{(Inkscape) Transparency is used (non-zero) for the text in Inkscape, but the package 'transparent.sty' is not loaded}%
    \renewcommand\transparent[1]{}%
  }%
  \providecommand\rotatebox[2]{#2}%
  \newcommand*\fsize{\dimexpr\f@size pt\relax}%
  \newcommand*\lineheight[1]{\fontsize{\fsize}{#1\fsize}\selectfont}%
  \ifx\svgwidth\undefined%
    \setlength{\unitlength}{209.51187254bp}%
    \ifx\svgscale\undefined%
      \relax%
    \else%
      \setlength{\unitlength}{\unitlength * \real{\svgscale}}%
    \fi%
  \else%
    \setlength{\unitlength}{\svgwidth}%
  \fi%
  \global\let\svgwidth\undefined%
  \global\let\svgscale\undefined%
  \makeatother%
  \begin{picture}(1,0.51802976)%
    \lineheight{1}%
    \setlength\tabcolsep{0pt}%
    \put(0,0){\includegraphics[width=\unitlength,page=1]{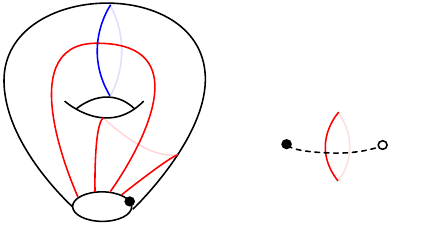}}%
    \put(0.80722388,0.1231257){\color[rgb]{0,0,0}\makebox(0,0)[lt]{\lineheight{1.25}\smash{\begin{tabular}[t]{l}$\scA$\end{tabular}}}}%
    \put(0,0){\includegraphics[width=\unitlength,page=2]{fig45.pdf}}%
  \end{picture}%
\endgroup%

\caption{ Left: A bordered Heegaard diagram for a solid torus in the Lipshitz-Ozsv\'{a}th-Thurston theory. Right: A Heegaard diagram of an unknot with a beta parallel arc system $\scA$ (passing through the alpha curve). In our theory, the corresponding bordered manifold would be represented by the module for an unknot (whose complement is a solid torus). The choice of alpha arcs on the boundary of the bordered Heegaard diagram on the left is analogous to the choice of a beta parallel arc system on the right.}
\label{fig:45}
\end{figure}

 For system of arcs $\scA$ we obtain a surgery complex $\cC_{\Lambda}(L,\scA)$ and a type-$D$ module $\cX_{\Lambda}(L,\scA)^{\cL}.$   Manolescu and Ozsv\'{a}th's result may be interpreted as showing if each arc of $\scA$ is alpha-parallel or beta-parallel, then
\[
\ve{\HF}^-(S^3_{\Lambda}(L))\iso H_*( \cC_{\Lambda}(L,{\scA}))
\]
as modules over $\bF\llsquare U\rrsquare $.  They only consider arc systems which are beta-parallel. In our work, we study the dependence on the arcs system, which turns out to be rather subtle. As a first step, we prove the following:
\begin{thm}\label{thm:independence-of-arcs} If $\scA$ is any system of arcs for $L$, then there is relatively graded isomorphism of vector spaces
\[
\ve{\HF}^-(S^3_{\Lambda}(L))\iso H_*( \cC_{\Lambda}(L,{\scA})).
\]
If the arc $a_i\in \scA$ is either alpha or beta-parallel, then the isomorphism is of $\bF\llsquare U\rrsquare $-modules, where $U$ acts by $U_i$ on $H_*(\cC_{\Lambda}(L,{\scA}))$.
\end{thm}

The isomorphism between the link surgery complexes for different systems of arcs is somewhat subtle. For example, we will construct two systems of arcs $\scA_1$ and $\scA_2$ on the Hopf link $H$ for which
\begin{equation}
\cX_{\Lambda}(H,{\scA_1})^{\cL_2}\not \simeq \cX_{\Lambda}(H,{\scA_2})^{\cL_2}\label{eq:intro-Hopf-arcs}
\end{equation}
as type-$D$ modules. See Section~\ref{sec:alpha-beta-model}. On the other hand, for type-$D$ modules associated to knots, alpha parallel and beta parallel arc systems give the same modules. This follows from \cite{ZemLattice}*{Remark~5.7}.

In Section~\ref{sec:basepoint-moving-maps}, we will describe a general formula for the effect of changing the arcs in the link surgery formula. One important special case is when we want to change an arc from alpha-parallel to beta-parallel, or vice versa. This may be achieved by tensoring with a rank 1 $DA$-bimodule ${}_{\cK} \cT^{\cK}$. In Section~\ref{sec:alpha-to-beta-transformer} we study this bimodule. It is algebraically related to the basepoint moving map investigated by Sarkar \cite{SarkarMovingBasepoints} and later by the author \cite{ZemQuasi}.

The arc systems also play an important role in the connected sum theorems and our construction of the pair-of-pants modules in Section~\ref{sec:pair-of-pants}. Note that in the Lipshitz, Ozsv\'{a}th, Thurston bordered theory, the gluing formula requires one to tensor an ``alpha bordered'' module with an ``alpha bordered'' module.  From a Heegaard diagrammatic perspective, this is natural because it is most natural to glue an ``alpha bordered'' Heegaard diagram to an ``alpha bordered'' bordered Heegaard diagram to get a Heegaard diagram of the glued manifold (since we cannot glue alpha arcs to beta arcs).

It is helpful to extend the Lipshitz, Ozsv\'{a}th and Thurston convention of gluing alpha-to-alpha and beta-to-beta into our theory. In our theory, gluing is accomplished by way of taking the tensor product with algebraically defined bimodules
\[
{}_{\cK|\cK} [\bI^{\Supset}],\quad {}_{\cK|\cK} W_{\a\b,\b}^{\cK} \quad \text{and} \quad {}_{\cK|\cK} W_{\a\b,\a}^{\cK}.
\]
(There are further modules for all combinations of $\alpha$ and $\beta$ subscripts). For connected sums involving ${}_{\cK|\cK}[\bI^{\Supset}]$, the most basic theorem we prove requires one module to be alpha bordered, and one to be beta bordered. Therefore, it is most natural to think of the module ${}_{\cK|\cK} [\bI]^{\Supset}$ as alpha bordered on one component, and beta bordered on the other. The labeling of the modules ${}_{\cK|\cK}W_{\a\b,\b}^{\cK}$, ${}_{\cK|\cK} W_{\a\b,\a}^{\cK}$ indicates the borderedness. For example, tensoring two beta bordered type-$D$ modules for knots into ${}_{\cK|\cK} W_{\a\b,\a}^{\cK}$ gives the alpha-bordered module for their connected sum.

\subsection{General bordered manifolds with torus boundary}
 
 The connected sum formula of Theorem~\ref{thm:pairing} and the topological interpretation of surgeries on connected sums in terms of gluing torus boundary components motivate a definition of a potential invariant of bordered 3-manifold with torus boundary components.
 
 We suppose that $M$ is a bordered 3-manifold with $n$ toroidal boundary components. We may describe $M$ as surgery on a link $L$, in the complement of an unlink $\bU_n$ in $S^3$, in a way which is compatible with the parametrization of $\d M$.
 
 We may define a type-$D$ module $\cX(M)^{\cL_n}$ as follows. We begin with the link surgery formula for $L\cup \bU_n$, interpreted as a type-$D$ module over our algebra $\cL_{n+\ell}$, where $\ell=|L|$. We write $\cX_{\Lambda}(L\cup\bU_n)^{\cL_{n+\ell}}$ for this module. We define $\cX(M)^{\cL_n}$ by tensoring the module $\cX_{\Lambda}(L\cup \bU_n)^{\cL_{n+\ell}}$ with $\ell$ type-$A$ modules for 0-framed solid tori. This corresponds topologically to performing Dehn surgery on the components $L$, while identifying the components of $\d \nu(\bU_n)$ with $\d M$. Theorem~\ref{thm:pairing} may be interpreted as a pairing theorem for gluing two bordered 3-manifolds with torus boundary components together. We prove in Section~\ref{sec:finite-generation} that the modules $\cX(M)^{\cL_n}$, as defined above, are always homotopy equivalent to finitely generated type-$D$ modules, i.e., the underlying vector space $\cX(M)$ may be taken to be finite-dimensional over $\bF$.

 In a future work, we hope to understand in what sense $\cX(M)^{\cL_n}$, defined as above, is an invariant of the bordered manifold $M$. In light of Equation~\eqref{eq:intro-Hopf-arcs}, the modules will in general depend on the choice of system of arcs $\scA$ in $S^3$. We make the following conjecture:
 
 \begin{conj}\footnote{In a subsequent work \cite{ZemExactTriangle}, we prove the above conjecture for arc systems where each arc is either alpha parallel or beta parallel.} If $M$ is a bordered manifold with toroidal boundary components which is equipped with a system of arcs $\scA\subset M$, then the type-$D$ module $\cX(M,\scA)^{\cL}$ is an invariant of $(M,\scA)$ up to homotopy equivalence.
 \end{conj}

 \begin{rem}
  Our bordered modules $\cX(M)^{\cK}$ give useful surgery formulas for homologically essential knots in 3-manifolds with $b_1>0$. This is a setting where the techniques of Ozsv\'{a}th and Szab\'{o} \cite{OSIntegerSurgeries} do not seem to have a simple adaptation. As a concrete example, the $\infty$-framed solid torus may be viewed as the knot complement of a fiber in $S^1\times \{pt\}\subset S^1\times S^2$. We compute in Section~\ref{sec:solid-tori} the associated type-$D$ module, which is concentrated in idempotent 1 and has two generators, and has $\delta^1$ encoded by the diagram
  \[
  \begin{tikzcd}[labels=description, column sep=2cm]
  \xs^1\ar[r, "1+\scV"]& \ys^1.
  \end{tikzcd}
  \]
 \end{rem}

\subsection{Acknowledgments}

The author would like to thank A. Alfieri, R. Lipshitz, C. Manolescu, and P. Ozsv\'{a}th  for helpful conversations. The author would like to thank K. Hendricks, J. Hom and M. Stoffregen for helpful conversations and the collaboration \cite{HHSZDuals}, where some of the results on the minimal model of the Hopf link complex were discovered. The author would also like to thank collaborators M. Borodzik and B. Liu for interesting discussions and helpful insights on related work. Finally, the author would like to thank the anonymous referee for their careful and thoughtful comments.

\section{Hypercubes and Heegaard Floer homology}

In this section, we describe some background on Heegaard Floer homology and homological algebra.

\subsection{Background on Heegaard Floer homology}

We will assume that the reader is familiar with the basics of Heegaard Floer homology \cite{OSDisks}, as well as its refinements for knots \cite{OSKnots} \cite{RasmussenKnots} and links \cite{OSLinks}. We provide a brief review to establish notation.

\begin{define}[\cite{OSLinks}*{Definition~3.1}]\,
\begin{enumerate}
\item  A \emph{multi-pointed Heegaard diagram} consists of a tuple
\[
(\Sigma,\as,\bs,\ws)
\]
as follows. Here, $\Sigma$ is a surface of genus $g$ and $\ws=\{w_1,\dots, w_n\}$ is a nonempty collection of basepoints. Also, $\as=\{\a_1,\dots, \a_{g+n-1}\},$ is a collection of pairwise disjoint embedded curves on $\Sigma$ which are linearly independent in $H_1(\Sigma\setminus\ws)$. We assume $\bs=\{\b_1,\dots, \b_{g+n-1}\}$ satisfies the same condition. In particular, $\Sigma\setminus \as$ consists of $n$ components, each of which contains a single basepoint $w_i\in \ws$, and similarly for $\Sigma\setminus \bs$.
 \item A \emph{Heegaard link diagram with free basepoints} consists of a tuple
 \[
  (\Sigma,\as,\bs,\ws,\zs,\ps)
 \] 
 satisfying the following. We assume that $\ws=\{w_1,\dots, w_\ell\}$, $\zs=\{z_1,\dots, z_\ell\}$ and $\ps=\{p_1,\dots, p_m\}$ are collections of basepoints. Let $n=\ell+m$. We assume that $\as=\{\a_1,\dots, \a_{g+n-1}\}$ and $\bs=\{\b_1,\dots, \b_{g+n-1}\}$ are collections of pairwise disjoint, embedded curves which each span a rank $g$ summand of $H_1(\Sigma)$.  We assume that each component of $\Sigma\setminus \as$ (resp. $\Sigma\setminus \bs$) contains either a single basepoint of $\ps$, or contains one basepoint from $\ws$ and one basepoint from $\zs$. 
\end{enumerate}
\end{define}

\begin{rem}
\item Occasionally, it is helpful to allow the collections $\ws$ and $\zs$ to contain free basepoints. In this case, we assume that if $p\in \ws\cup \zs$ is the sole basepoint in a component of $\Sigma\setminus \as$, then it is the sole basepoint in a component of $\Sigma\setminus \bs$. 
\end{rem}

The curves $\as,\bs\subset \Sigma$ determine two handlebodies, $U_{\a}$ and $U_{\b}$, with boundary $\Sigma$. A Heegaard link diagram with free basepoints determines a 3-manifold $Y$ and an oriented link $L\subset Y$. In $U_{\a}$, the link $L$ is a push-off of a collection of smoothly embedded, pairwise disjoint arcs in $\Sigma\setminus \as$ which connect the link-basepoints of $\ws$ to the link-basepoints of $\zs$. In $U_{\b}$, we assume $L$ is a push off of a collection of smoothly embedded, pairwise disjoint arcs which connect the link-basepoints of $\zs$ to the link-basepoints of $\ws$. We follow the orientation convention that $L$ intersects $\Sigma$ positively at $\zs$ and negatively at $\ws$.

\begin{rem} Typically we assume that if $w_i$ and $z_j$ are in a component of $\Sigma\setminus \as$, then they are also in a component of $\Sigma\setminus \bs$. This ensures that $(\Sigma,\as,\bs,\ws,\zs,\ps)$ represents a link with $\ell$ components and each component has 2 basepoints.
\end{rem}

 If $\cH=(\Sigma,\as,\bs,\ws)$ is a multi-pointed Heegaard diagram for $Y$ and $\frs\in \Spin^c(Y)$, we define $\CF^-(\cH,\frs)$ to be the free $\bF[U_1,\dots, U_n]$-module generated by intersection points $\xs$ of the Lagrangian tori
 \[
 \bT_{\a}=\a_1\times \dots \times \a_{g+n-1}\quad \text{and} \quad \bT_{\b}=\b_1\times \cdots \times \b_{g+n-1}
 \]
 in the symmetric product $\Sym^{g+n-1}(\Sigma)$, such that $\frs_{\ws}(\xs)=\frs$.  See \cite{OSDisks}*{Section~2.6} for the definition of the map $\frs_{\ws}\colon \bT_{\a}\cap \bT_{\b}\to \Spin^c(Y)$. 
 
 The differential counts pseudo-holomorphic representatives of Maslov index 1 flowlines in $\Sym^{g+n-1}(\Sigma)$ via the formula
 \begin{equation}
 \d(\xs)=\sum_{\substack{\phi\in \pi_2(\xs,\ys) \\ \mu(\phi)=1}} \#(\cM(\phi)/\R) U_1^{n_{w_1}(\phi)}\cdots U_n^{n_{w_n}(\phi)}\cdot \ys,
 \label{eq:differential-basic}
 \end{equation}
 extended equivariantly over $\bF[U_1,\dots, U_n]$.
 
 We will frequently make use of Lipshitz's \emph{cylindrical reformulation} of Heegaard Floer homology \cite{LipshitzCylindrical}. In this formulation, the differential counts holomorphic curves of potentially higher genus in $\Sigma\times [0,1]\times \R$. We refer the reader to \cite{LipshitzCylindrical} for additional background.

 \begin{define}[Cf. \cite{OSDisks}*{Section~4.2}] 
 \label{def:weakly-admissible}
 \item
 \begin{enumerate}
 \item A \emph{periodic domain} on a Heegaard multi-diagram $\cH=(\Sigma,\gs_1,\dots,\gs_n,\ws)$ is a 2-chain $P$ on $\Sigma$ such that $\d P$ is an integral linear combination of the curves from $\gs_1,\dots, \gs_n$ and $n_{\ws}(P)=0$. 
 \item  A Heegaard diagram for a pointed 3-manifold $\cH=(\Sigma,\as,\bs,\ws)$ is \emph{weakly admissible} if each non-trivial periodic domain has both positive and negative multiplicities.
 \item  If $\cH=(\Sigma,\as,\bs,\ws,\zs, \ve{p})$ is a Heegaard link diagram with free basepoints, then we say that $\cH$ is weakly admissible if $(\Sigma,\as,\bs,\qs)$ is weakly admissible for each collection $\qs\subset \ws\cup \zs\cup \ps$ such that $\qs$ contains all of the free basepoints and exactly one link basepoint from each link component.
\item We say that a Heegaard multi-diagram $(\Sigma,\gs_1,\dots, \gs_n,\qs)$ is \emph{weakly admissible} if each periodic domain has positive and negative multiplicities. A Heegaard multi-diagram with link basepoints $(\Sigma,\gs_1,\dots, \gs_n, \ws, \zs, \ps)$ is called weakly admissible if it is weakly admissible for each collection $\qs\subset \ws\cup \zs\cup \ps$ as above.
 \end{enumerate}
 \end{define}

 Given a weakly admissible Heegaard diagram $\cH$, we can define a completed version of Heegaard Floer homology $\ve{\CF}^-(\cH)$, which is freely generated by $\xs\in \bT_{\a}\cap \bT_{\b}$ (with no restriction on $\Spin^c$ structures) over the ring $\bF\llsquare U_1,\dots, U_n\rrsquare $. 
 
We now discuss link Floer homology. In this paper, we use a slightly different version of link Floer homology than appeared in \cite{OSLinks}. If $(\Sigma,\as,\bs,\ws,\zs)$ is a Heegaard link diagram with no free basepoints and with $|\ws|=|\zs|=\ell$, we use a version of the link Floer complex which is a free, finitely generated chain complex over the ring $\bF[\scU_1,\dots, \scU_\ell,\scV_1,\dots,\scV_\ell]$. This version was considered in \cite{ZemQuasi} and \cite{ZemCFLTQFT}.
We define $\cCFL(\Sigma,\as,\bs,\ws,\zs,\frs)$ to be generated over $\bF[\scU_1,\dots, \scU_\ell,\scV_1,\dots, \scV_\ell]$ by intersection points $\xs\in \bT_{\a}\cap \bT_{\b}$ such that $\frs_{\ws}(\xs)=\frs$. The differential is similar to~\eqref{eq:differential-basic}, except that we weight a holomorphic curve by 
\[
\scU_1^{n_{w_1}(\phi)}\scV_1^{n_{z_1}(\phi)}\cdots \scU_\ell^{n_{w_\ell}(\phi)}\scV_\ell^{n_{z_\ell}(\phi)}.
\] 

In the case of a doubly pointed Heegaard knot diagram, $\cH=(\Sigma,\as,\bs,w,z)$, we write $\cCFK(\cH,\frs)$ for the complex which is freely generated over $\bF[\scU,\scV]$ by intersection points $\xs\in \bT_{\a}\cap \bT_{\b}$.

\subsection{Hypercubes and hyperboxes of chain complexes}

We now describe Manolescu and Ozsv\'{a}th's algebraic formalism of \emph{hypercubes}  and \emph{hyperboxes} of chain complexes \cite{MOIntegerSurgery}*{Section~5}. Define
\[
\bE_n:=\{0,1\}^n\subset \R^n,
\]
with the convention that $\bE_0=\{0\}$. 
If $\ve{d}=(d_1,\dots, d_n)\in \Z^{\ge 0}$, we define
\[
\bE(\ve{d}):=\{(i_1,\dots, i_n)\in \Z^n: 0\le i_j\le d_j, \forall j\}.
\]
If $\veps,\veps'\in \bE(\ve{d})$, we write $\veps\le \veps'$ if the inequality holds for each coordinate of $\veps$ and $\veps'$. We write $\veps<\veps'$ if $\veps\le \veps'$ and strict inequality holds at one or more coordinates.

\begin{define} An \emph{$n$-dimensional hypercube of chain complexes} $(C_{\veps}, D_{\veps,\veps'})_{\veps\in \bE_n}$ consists of the following:
\begin{enumerate}
\item A group $C_{\veps}$ for each $\veps\in \bE_n$.
\item For each pair of indices $\veps,\veps'\in \bE_n$ such that $\veps\le \veps'$, a linear map
\[
D_{\veps,\veps'}\colon C_{\veps}\to C_{\veps'}.
\]
\end{enumerate}
Furthermore, we assume the following compatibility condition holds whenever $\veps\le \veps''$:
\begin{equation}
\sum_{\substack{\veps'\in \bE_n\\ \veps\le \veps'\le \veps''}} D_{\veps',\veps''}\circ D_{\veps,\veps'}=0.
\label{def:hypercube}
\end{equation}
\end{define}

\begin{define} 
If $\ve{d}\in (\Z^{>0})^n$, a \emph{hyperbox of chain complexes of size $\ve{d}$} consists of a collection of groups $(C_{\veps})_{\veps\in \bE(\ve{d})}$, together with a linear map $D_{\veps,\veps'}$ whenever $\veps\le \veps'$ and $|\veps'-\veps|_{L^\infty}\le 1$. Furthermore, we assume that the analog of \eqref{def:hypercube} holds whenever $\veps\le \veps''$ and  $|\veps''-\veps|_{L^\infty}\le 1$.
\end{define}

Hypercubes and hyperboxes of the same size form $dg$-categories. If $\cC=(C_{\veps}, D_{\veps,\veps'})_{\veps\in \bE_n}$ and $\cC'=(C_{\veps}', D_{\veps,\veps'}')_{\veps\in \bE_n}$ are two hypercubes of the same dimension, then a \emph{morphism of hypercubes} $F$ consists of a collection of linear maps 
\[
F_{\veps,\veps'}\colon C_{\veps}\to C'_{\veps'},
\] ranging over all $\veps, \veps'\in \bE_n$ such that $\veps\le \veps'$. We write $\Hom(\cC,\cC')$ for the space of morphisms. There is a natural morphism differential $\d_{\Mor}$ on $\Hom(\cC,\cC')$, given by
\begin{equation}
\d_{\Mor}(F)_{\veps,\veps''}=\sum_{\substack{\veps'\in \bE_n \\ \veps\le \veps'\le \veps''}} D'_{\veps',\veps''}\circ F_{\veps,\veps'}+F_{\veps',\veps''}\circ D_{\veps,\veps'}.
\label{eq:morphism-differential-hypercubes}
\end{equation}

\begin{rem}
In Section~\ref{sec:hypercube-algebra} we will give an alternate perspective on hypercubes, and show that the category of hypercubes of chain complexes is equivalent to the category of type-$D$ modules over a certain algebra.
\end{rem}

\subsection{Compressing hyperboxes}
\label{sec:compression}
Manolescu and Ozsv\'{a}th \cite{MOIntegerSurgery}*{Section~5} describe an operation called \emph{compression} which takes a hyperbox $\cC=(C_{\veps},D_{\veps,\veps'})_{\veps\in \bE(\ve{d})}$ of size $\ve{d}=(d_1,\dots, d_n)$ and returns an $n$-dimensional hypercube $\widehat{\cC}=(\widehat{C}_{\veps}, \widehat{D}_{\veps,\veps'})_{\veps\in \bE_n}$. The underlying complexes are given by the formula
\[
\widehat{C}_{(\veps_1,\dots, \veps_n)}=C_{(d_1\veps_1,\dots, d_n\veps_n)}.
\]
The hypercube structure maps are more complicated. We first illustrate the construction in the 1-dimensional case. Suppose $\cC$ is the 1-dimensional hyperbox of chain complexes
\[
\begin{tikzcd}
C_0
	\ar[r, "f_{0,1}"]
&
C_1
	\ar[r, "f_{1,2}"]
&
\cdots
	\ar[r, "f_{n-1,n}"]
&
C_n.
\end{tikzcd}
\]
The compression of the above hyperbox is the hypercube
\[
\begin{tikzcd}[column sep=2.5cm]
C_0
\ar[r, "f_{n-1,n}\circ \cdots \circ f_{0,1}"]
&
C_n.
\end{tikzcd}
\]
In Section~\ref{sec:hypercube-algebra} we give an alternate description of compression in the 1-dimensional case in terms of box-tensor products of $DA$-bimodules over a certain algebra.

 The same description also works for hyperboxes of size $(1,\dots, 1,d)$. That is, we view $\cC$ as a 1-dimensional hyperbox where the complex at each point of $\bE(d)$ is a hypercube of dimension $n-1$. Compression in this case is given by composition of hypercube morphisms.
 
For a hyperbox $\cC$ of size $(d_1,\dots, d_n)$, the compression may be defined by iterating  the above construction. Doing so requires a choice of ordering of the axis directions. For concreteness, we pick the standard ordering of $1,\dots, n$. Given a hyperbox $\cC$ we define the partial compression of the $i$-th axis direction, denoted $\frc_i(\cC)$, to be the following hyperbox of size $(d_1,\dots, d_{i-1},1,d_{i+1},\dots, d_n)$. We view the $\cC$ as a 1-dimensional hyperbox (of hyperboxes) of size $(d_i)$, where the complex at each point in $\bE(d_i)$ is a hyperbox of size $(d_1,\dots, d_{i-1},d_{i+1},\dots, d_n)$. The hyperbox at point $j\in \{0,\dots, d_i\}$ has underlying group consisting of the sum of all complexes at points $\veps\in \bE(\ve{d})$ with $\veps_i=j$. We apply the composition perspective of compression for 1-dimensional hypercubes to obtain $\frc_i(\cC)$. We then define the compression
\[
\widehat{\cC}=(\frc_n\circ \cdots \circ \frc_1)(\cC).
\]
Compare \cite{Liu2Bridge}*{Section~4.1.2}.

Of course, we can compress the hyperbox $\cC$ using any ordering of the axis directions. The following is well-known, and we do not claim originality for the following, though we include it for completeness:

\begin{lem} If $\cC$ is a hypercube of chain complexes and $\widehat{\cC}$ is its compression, then $\widehat{\cC}$ is independent up to homotopy equivalence of hypercubes from the choice of ordering of the axis directions.
\end{lem}
\begin{proof}  Given the inductive nature of the construction, it suffices to prove the case when $n=2$. Let $\cC$ be a hyperbox of size $(d_1,d_2)$. Write $\frc_2(\frc_1(\cC))$ and $\frc_1(\frc_2(\cC))$ for the two compressions of $\cC$. 
 Note that the length 1 maps of $\frc_2(\frc_1(\cC))$ and $\frc_1(\frc_2(\cC))$ coincide, so it suffices to consider the diagonal maps. Since we are considering hypercubes of dimension 2, it suffices to show that the two diagonal maps are themselves chain homotopic since such a chain homotopy can be used to build a morphism of hypercubes $F\colon \frc_2(\frc_1(\cC))\to \frc_1(\frc_2(\cC))$ which is a homotopy equivalence. The components of $F$ which preserve cube points act by the identity. We define $F$ to have no components which increase only one cube point. We will define $F$ to have a component from cube point $(0,0)$ to cube point $(1,1)$, which acts by the above chain homotopy between the length 2 maps of $\frc_2(\frc_1(\cC))$ and $\frc_1(\frc_2(\cC))$. Assuming we can construct such a chain homotopy, the map $F$ is clearly a homotopy equivalence. We now describe how to construct this chain homotopy.

 Write $C_*(P_2)$ for the following chain complex over $\bF=\Z/2$. There are two generators in degree 0, denoted $e_1e_2$ and $e_2e_1$. We define $\d(e_ie_j)=0$. There is a single generator in degree $1$, denoted $h_{1,2}$, which satisfies
\[
\d h=e_1e_2+e_2e_1.
\]
We view $e_1e_2$ and $e_2e_1$ as the paths along the boundary of a 2-dimensional cube from $(0,0)$ to $(1,1)$, and we think of $h_{1,2}$ as the diagonal. We think of $C_*(P_2)$ as the cell complex for the 2-dimensional permutohedron.

 We now define another chain complex $C_*(P_{d_1,d_2})$. The generators consist of concatenations  of $e_1$, $e_2$ and $h_{1,2}$ such that the total number of $1$ subscripts appearing is $d_1$, and the total number of $2$ subscripts appearing is $d_2$. The differential sums over all ways of breaking a single $h_{1,2}$ into $(e_1e_2+e_2e_1)$. We think of the generators of $C_*(P_{d_1,d_2})$ as corresponding to paths in $[0,d_1]\times [0,d_2]$ from $(0,0)$ to $(d_1,d_2)$ which consist of concatenations of three types of segments (horizontal unit length segments, vertical unit length segments, or diagonal segments of $L^1$-length $2$).
 
 Note that the hyperbox structure maps of $\cC$ give a chain map
 \[
 C_*(P_{d_1,d_2})\to \Hom(C_{0,0}, C_{d_1,d_2}).
 \]
The two compression procedures give chain maps 
\[
\frc_{1,2}, \frc_{2,1}\colon C_*(P_2)\to C_*(P_{d_1,d_2}).
\]
 Note that $\frc_{i,j}(e_1e_2)=e_1^{d_1} e_2^{d_2}$ and $\frc_{i,j}(e_2e_1)=e_2^{d_2}e_1^{d_1}$ for both choices of $i,j$. In order to show that the diagonals of $\widehat{\cC}_{1,2}$ and $\widehat{\cC}_{2,1}$ are chain homotopic, it suffices to show that the maps $\frc_{1,2}$ and $\frc_{2,1}$ are themselves chain homotopic. To do this, it suffices to show that the homology of $P_{d_1,d_2}$ satisfies
 \[
 H_n(P_{d_1,d_2})\iso \begin{cases}\bF & \text{ if } n=0\\ 
 0& \text{ if } n>0  \end{cases}.
 \]
To establish this, we describe a filtration $\cF_{\le n}$ on $C_*(P_{d_1,d_2})$. Given an arrow sequence $\omega\in P_{d_1,d_2}$, view $\omega$ as giving an embedded piecewise linear curve in $[0,d_1]\times [0,d_2]\subset \R^2$. Consider the region $R$ between $\omega$ and the arrow sequence $e_1^{d_1} e_2^{d_2}$. We define the filtration level of $\omega$ to be the number of unit squares in $[0,d_1]\times [0,d_2]$ with integral corners and whose interiors intersect $R$ non-trivially. Note that the subspace $\cF_{\le n}$ is preserved by $\d$.

We now define a homotopy $H\colon C_*(P_{d_1,d_2})\to C_{*+1}(P_{d_1,d_2})$. The map $H$ sums over all ways of replacing one subword $e_2e_1$ with $h_{1,2}$. We observe that $\bI+[\d, H]$ acts by the identity on $e_1^{d_1} e_2^{d_2}$. We leave it to the reader to verify that if $\omega$ is a word with filtration level $n>0$, then $(\bI+[\d, H])(\omega)$ is a sum of arrow words of filtration level at most $n-1$. Therefore $(\bI+[\d, H])^N$ gives a homotopy equivalence between $C_*(P_{d_1,d_2})$ and $\bF$ (concentrated in degree 0) for $N\gg 0$. This completes the proof.
\end{proof}

\subsection{$A_\infty$-categories and twisted complexes}
\label{sec:twisted-complexes}

We now recall the notion of a \emph{twisted complex} in an $A_\infty$-category. We refer to \cite{SeidelFukaya}*{Section~3.l} for more background on this construction. 

We recall that an \emph{$A_\infty$-category} $\cC$ consists of a set of objects $\cO(\cC)$, a collection of morphisms $\Hom(L,L')$ for $L,L'\in \cO(\cC)$, together with a collection of higher compositions $(\mu_n)_{n\ge 1}$, as follows. We assume that if $L_0,L_1\in \cO(\cC)$, then $\Hom(L_0,L_1)$  is a vector space. We assume that for each sequence $L_0,\dots, L_n\in \cO(\cC)$, the map $\mu_n$ takes the form
\[
\mu_n\colon \Hom(L_0,L_1)\otimes \cdots \otimes \Hom(L_{n-1},L_n)\to \Hom(L_0,L_1).
\]
Furthermore, we assume that for all $\xs_{1},\dots, \xs_{n}$ (where $\xs_i\in \Hom(L_{i-1},L_i)$) the following associativity condition holds:
\begin{equation}
\sum_{i=1}^n \sum_{j=i}^{n} \mu_{n-j+i}(\xs_{1},\dots, \xs_{i-1}, \mu_{j-i+1}(\xs_{i},\dots, \xs_{j}), \xs_{j+1},\dots, \xs_{n})=0.
\label{eq:associativity-condition}
\end{equation}

If $\cC$ is an $A_\infty$-category, the \emph{additive enlargement} $\Sigma \cC$ of $\cC$ is defined as follows. The objects of $\Sigma\cC$ consist of formal sums of the form $\bigoplus_{i\in I} X_i\otimes V_i$. Here, $X_i$ denote objects of $\cC$,  $V_i$ denote graded vector spaces over $\bF=\Z/2$, and $I$ denotes a finite index set. The symbols $\otimes$ and $\bigoplus$ are just notation. (For our purposes, we can take $V_i=\bF$ for all $i$).

A morphism in $\Sigma \cC$ from  $\bigoplus_{i\in I}X_i\otimes V_i$ to $\bigoplus_{j\in J}Y_j\otimes W_j$ consists of a collection $(\theta_{i,j}\otimes f_{i,j})_{i\in I, j\in J}$ where $\theta_{i,j}\colon X_i\to X_j$ is a morphism in $\cC$ and $f_{i,j}\colon V_i\to W_j$ is a linear map of vector spaces. We will typically write such a morphism as $\sum_{(i,j)\in I\times J} \theta_{i,j}\otimes f_{i,j}$.

If $\cC$ is an $A_\infty$-category, then $\Sigma \cC$ is naturally also an $A_\infty$-category, with compositions
\[
\begin{split}&\mu^{\Sigma \cC}_n \left(\sum_{i_0,i_1} \theta_{i_0,i_1}\otimes f_{i_0,i_1},\dots, \sum_{i_{n-1},i_{n}} \theta_{i_{n-1},i_n}\otimes f_{i_{n-1},i_n}, \right)\\
:=&\sum_{i_0,i_1,\dots, i_n} \mu_n^{\cC}(\theta_{i_0,i_1},\dots, \theta_{i_{n-1},i_n})\otimes(f_{i_{n-1},i_n}\circ \cdots 
\circ f_{i_0,i_1}).
\end{split}
\]
It is straightforward to verify that $\Sigma \cC$ is an $A_\infty$-category.

Finally, we define the category of twisted complexes $\Tw \cC$. A \emph{twisted complex} $(X,\delta)$ in $\cC$ consists of an object $X=(X_i,V_i)_{i\in I}$ in $\Sigma \cC$ where $I$ is a finite partially ordered set, together with an endomorphism $\delta\in \Hom_{\Sigma \cC}(X,X)$. If $\delta=\sum_{i,j}\delta_{i,j}\otimes f_{i,j}$ we assume that  $\delta_{i,j}=0$ unless $i<j$. Furthermore, we assume the following version of the Mauer-Cartan equation is satisfied:
\[
\sum_{n\ge 1} \mu_n^{\Sigma \cC}(\underbrace{\delta,\dots, \delta}_n)=0.
\]
A morphism in $\Tw \cC$ is the same as a morphism in $\Sigma \cC$. Finally, we note that $\Tw \cC$ is itself an $A_\infty$-category, with compositions
\[
\mu_n^{\Tw}(\theta_1,\dots, \theta_n):=\sum_{i_0,\dots, i_n\ge 0} \mu_{n+i_0+\cdots+i_n}^{\Sigma \cC} (\underbrace{\delta,\dots, \delta}_{i_0}, \theta_1,\underbrace{\delta,\dots, \delta}_{i_1}, \theta_2,\dots, \theta_n, \underbrace{\delta,\dots, \delta}_{i_n}).
\]

\subsection{Hypercubes of attaching curves}

The notion of a hypercube of chain complexes adapts to other categories. We now describe the appropriate notion of a hypercube in the Fukaya category. See \cite{MOIntegerSurgery}*{Section~8} and \cite{LOTDoubleBranchedI}*{Section~3}, where such objects are referred to as \emph{hyperboxes of Heegaard diagrams}, and \emph{chain complexes of attaching curves}, respectively. These may naturally be viewed as twisted complexes in the Fukaya category, using the framework described in Section~\ref{sec:twisted-complexes}.

\begin{define}\label{def:hypercube-attaching-curves}
 A \emph{hypercube of handleslide equivalent attaching curves} 
 \[
 \cL_{\b}=(\bs_{\veps}, \Theta_{\veps,\veps'})_{\veps\in \bE_n}
 \]
  on $(\Sigma,\ws)$ consists of a collection of beta attaching curves $\bs_{\veps}$, which are each pairwise related by a sequence of handleslides and isotopies, together with distinguished morphisms $\Theta_{\veps,\veps'}\in \ve{\CF}^-(\Sigma,\bs_{\veps},\bs_{\veps'},\ws)$ whenever $\veps<\veps'$. We assume, furthermore, that the Heegaard multi-diagram containing all $2^n$ beta curves is weakly admissible (see Definition~\ref{def:weakly-admissible}). Finally, we assume that for each pair $\veps<\veps'\in \bE_n$, the following relation is satisfied:
\begin{equation}
\sum_{\veps=\veps_1<\cdots<\veps_n=\veps'} f_{\b_{\veps_1},\dots,\b_{\veps_n}}(\Theta_{\veps_1,\veps_2},\dots, \Theta_{\veps_{n-1},\veps_n})=0.\label{eq:compatibility-hypercube-lagrangians}
\end{equation}
In the above, $f_{\b_{\veps_1},\dots, \b_{\veps_n}}$ denotes the holomorphic polygon counting map.
\end{define}

Sometimes in Heegaard Floer theory it is helpful to have a notion of a hypercube of \emph{beta} attaching curves, or a hypercube of \emph{alpha} attaching curves. We define a hypercube of beta attaching curves to be the object described in Definition~\ref{def:hypercube-attaching-curves}, however we define a \emph{hypercube of alpha attaching curves} as the following modification: for each $\veps<\veps'$ we have a chain $\Theta_{\veps',\veps}\in \ve{\CF}^-(\Sigma,\as_{\veps'}, \as_{\veps},\ws)$, and the compatibility condition in Equation~\eqref{eq:compatibility-hypercube-lagrangians} holds, as long as we write the indices in the opposite order. Clearly any hypercube of attaching curves can be viewed as either a hypercube of alpha attaching curves or beta attaching curves (by possibly reversing the order of indices), but the terminology is occasionally helpful.

If $\cL_{\a}$ and $\cL_{\b}$ are hypercubes of handleslide equivalent attaching curves on $(\Sigma,\ws)$, of dimension $m$ and $n$ respectively, then there is an $(n+m)$-dimensional hypercube of chain complexes 
\[
(C_{(\veps,\nu)},D_{(\veps,\nu),(\veps',\nu')})_{(\veps,\nu)\in \bE_{n}\times \bE_m}=\ve{\CF}^-(\Sigma,\cL_{\a},\cL_{\b}, \ws),
\]
 whenever the diagram containing all attaching curves is weakly admissible.  For each $(\veps,\nu)\in \bE_{n+m}$, we set
\[
C_{(\veps,\nu)}=\ve{\CF}^-(\Sigma,\as_\veps, \bs_\nu,\ws).
\]
The structure maps are given by the following formula:
\begin{equation}
\begin{split}
&D_{(\veps,\nu),(\veps',\nu')}(\xs)\\
=&\sum_{\substack{\veps=\veps_1<\cdots<\veps_i=\veps' \\
\nu=\nu_1<\cdots<\nu_j=\nu'}}
f_{\a_{\veps_i},\dots, \a_{\veps_1},\b_{\nu_1},\dots, \b_{\nu_j}}(\Theta_{\veps_i,\veps_{i-1}},\dots, \Theta_{\veps_2,\veps_1},\ve{x}, \Theta_{\nu_1,\nu_2},\dots, \Theta_{\nu_{j-1},\nu_j}).
\end{split}
\label{eq:pairing-formula-hypercubes}
\end{equation}
It is straightforward to see that $\ve{\CF}^-(\Sigma,\cL_{\a},\cL_{\b},\ws)$ is a hypercube of chain complexes.

The case when there are link basepoints $\ws$, $\zs$ basepoints is a straightforward modification of the above construction, using variables $\scU_i$ and $\scV_i$ in the definition of the Floer complexes.

 We observe that a hypercube of attaching curves $\cL=(\gs_\veps,\theta_{\veps,\veps'})_{\veps\in \bE_n}$ may naturally be interpreted as a twisted complex of Lagrangians in the Fukaya category. The index set $I$ in the definition of a twisted complex is the cube $\bE_n$, and the partial order is the natural ordering of cube points.
 
 Correspondingly, there is a natural notion of a morphism between two hypercubes of attaching curves, and the resulting category is an $A_\infty$-category.  If $\cL=(\gs_{\veps}, \xs_{\veps,\veps'})_{\veps\in \bE_n}$ and $\cL'=(\gs'_{\veps}, \ys_{\veps,\veps'})_{\veps\in \bE_n}$ are $n$-dimensional hypercubes of attaching curves, then a morphism $\Phi$ from $\cL$ to $\cL'$ 
consists of a collection of chains $\Phi_{\veps,\veps'}\in \ve{\CF}^-(\gs_{\veps},\gs_{\veps'}',\ws)$ ranging over all $\veps,\veps'\in \bE_n$ with $\veps\le \veps'$. Of course, the set of morphisms is none other than the Floer complex $\ve{\CF}^-(\Sigma,\cL,\cL',\ws)$.

Suppose $\cL_1,\dots, \cL_n$ are $n$-dimensional hypercubes of attaching curves, where $\cL_i=(\gs_{i,\veps}, (\Theta_i)_{\veps,\veps'})_{\veps\in \bE_n}$. The composition operator
\[
\mu_n^{\Tw}\colon \ve{\CF}^-(\cL_1,\cL_2)\otimes \cdots \otimes \ve{\CF}^-(\cL_{n-1},\cL_n)\to \ve{\CF}^-(\cL_1,\cL_n)
\]
is given as follows. If $\Phi_{i,i+1}\in \ve{\CF}^-(\cL_{i},\cL_{i+1})$, then
\[
\begin{split}
&\mu_n^{\Tw}(\Phi_{1,2},\dots, \Phi_{n-1,n})\\
=&\sum_{i_1,\dots, i_n\ge 0} \mu_{i_1+\dots+i_n+n}(\underbrace{\Theta_1,\dots, \Theta_1}_{i_1}, \Phi_{1,2},\underbrace{\Theta_2,\dots, \Theta_2}_{i_2}, \Phi_{2,3},\dots, \Phi_{n-1,n}, \underbrace{\Theta_n,\dots, \Theta_{n}}_{i_n}).
\end{split}
\]
In the above, $\Theta_i=\sum_{\veps<\veps'} (\Theta_i)_{\veps,\veps'}$, viewed as an element in group $\bigoplus_{\veps<\veps'} \ve{\CF}^-(\gs_{i,\veps}, \gs_{i,\veps'})$, and similarly for $\Phi_{i,i+1}$. Additionally, $\mu_{i_1+\cdots+i_n+n}$ is declared to vanish on non-composable Floer chains, e.g. $\mu_2(\xs,\ys)=0$ if $\xs\in \ve{\CF}^-(\gs_0,\gs_1)$ and $\ys\in \ve{\CF}^-(\gs_1',\gs_2)$ where $\gs_1\neq \gs_1'$.

\subsection{More hypercubes}

We now collect several additional constructions and basic properties. Firstly, if $\cC=(C_\veps,D_{\veps,\veps'})_{\veps\in \bE_n}$ and $\cC'=(C'_{\nu}, D'_{\nu,\nu'})_{\nu\in \bE_m}$ are two hypercubes of chain complexes over a ring $R$, we may define their \emph{external tensor product}
\[
\cE:=\cC\otimes_R \cC',
\]
as follows. The hypercube $\cE=(E_{(\veps,\nu)},\delta_{(\veps,\nu), (\veps',\nu')})_{(\veps,\nu)\in \bE_n\times \bE_m}$ is the tensor product in the ordinary sense: As groups, we set
\[
E_{(\veps,\nu)}=C_{\veps}\otimes_R C_{\nu}'
\]
and the structure maps are given by
\[
\delta_{(\veps,\nu),(\veps',\nu')}=\begin{cases}D_{\veps,\veps'}\otimes \bI& \text{ if } \nu=\nu'\\
\bI\otimes D_{\nu,\nu'}'& \text{ if } \veps=\veps'\\
0& \text{ if } \veps<\veps' \text{ and } \nu<\nu'.
\end{cases}
\]

Additionally, if $\cC$ and $\cC'$ are two hyperboxes of size $\ve{d}_0\times \{d\}$ and $\ve{d}_0\times \{d'\}$, such that the restrictions $\cC|_{\bE(\ve{d}_0)\times \{d\}}$ and $\cC'|_{\bE(\ve{d}_0)\times \{0\}}$ coincide (in both underlying groups and structure morphisms), then we may stack $\cC$ and $\cC'$ to obtain a hyperbox of size $\ve{d}_0\times \{d+d'\}$, which we denote $\St(\cC,\cC')$.  The following lemma records some interactions between stacking and the external tensor product.

\begin{lem}\label{lem:stacking-v-compressing}
\item
\begin{enumerate}
\item\label{stack-v-compress-1} Stacking hyperboxes of attaching curves commutes with pairing: If $\cL_{\a}$, $\cL_{\b}$ and $\cL_{\b'}$ are hyperboxes of attaching curves, and $\cL_{\b}$ and $\cL_{\b'}$ may be stacked along a codimension 1 face, then
\[
\St(\ve{\CF}^-(\cL_{\a}, \cL_{\b}), \ve{\CF}^-(\cL_{\a}, \cL_{\b'}))\iso \ve{\CF}^-(\cL_{\a}, \St(\cL_\b,\cL_{\b'})).
\]
\item\label{stack-v-compress-2} Stacking hyperboxes of chain complexes commutes with tensoring hyperboxes: If $\cC$ and $\cC'$ are stackable hyperboxes of chain complexes, and $\cD$ is another hyperbox of chain complexes, then
\[
\St(\cC,\cC')\otimes \cD\iso \St(\cC,\cD)\otimes \St(\cC',\cD).
\]
\item \label{stack-v-compress-3}  Compressing hyperboxes of chain complexes commutes with tensor products of hyperboxes. If $\cC$ and $\cD$ are hyperboxes of chain complexes, and $\widehat{\cC}$ denotes compression, then there is a choice of ordering of the axis directions used to compute the compression so that
\[
\widehat{\cC}\otimes \widehat{\cD}\iso \widehat{\cC\otimes \cD}.
\]
\end{enumerate}
\end{lem}
\begin{proof} Claim~\eqref{stack-v-compress-1} is essentially the definition of the pairing between hyperboxes. Similarly~\eqref{stack-v-compress-2} is clear. For claim~\eqref{stack-v-compress-3}, we use the inductive perspective on compression from Section~\ref{sec:compression}. Suppose that $\cC$ is a hyperbox of size $(d_1,\dots, d_n)$ and $\cD$ is a hyperbox of size $(d_{1}',\dots, d_m')$. We use the partial compression operations $\frc$ described in Section~\ref{sec:compression}. To show the main claim, it suffices to show that
\begin{equation}
\frc_i(\cC\otimes \cD)\iso \frc_i(\cC)\otimes \cD\quad \text{and} \quad \frc_{j+n}(\cC\otimes \cD)=\cC\otimes \frc_j(\cD)
\label{eq:partial-compression-and-tensor-product}
\end{equation}
if $1\le i\le n$ and $1\le j\le m$. We focus on the equation $\frc_i(\cC\otimes \cD)\iso \frc_i(\cC)\otimes \cD$ when $1\le i\le n$. Recall that the differential on $\cC\otimes \cD$ is the tensor product $D_{\cC}\otimes \bI+\bI\otimes D_{\cD}$. Therefore if we view $\cC\otimes \cD$ as a hyperbox of size $(d_i)$, then only the components of $D_{\cC}\otimes \bI$ will increase the $\bE(d_i)$ index, while the components of $\bI_{\cC}\otimes D_{\cD}$ will preserve the $\bE(d_i)$ index. From this observation, Equation~\eqref{eq:partial-compression-and-tensor-product} is straightforward, and the main claim follows.
\end{proof}

A more sophisticated extension of Lemma~\ref{lem:stacking-v-compressing} concerns commuting a tensor product operation on attaching curves and the pairing operation on hypercubes of attaching curves. See Propositions~\ref{prop:disjoint-unions-hypercubes-main} and~\ref{prop:connected-sums}.

\section{\texorpdfstring{$A_\infty$}{A-infinity}-modules}

In this section, we recall some background about $A_\infty$ modules. We follow the notation of \cite{LOTBordered} \cite{LOTBimodules}. See \cite{KellerNotes} \cite{KellerNotesAddendum}  for further background and historical context.

\subsection{Type-$A$ and type-$D$ modules}
\label{sec:type-A/D}
In this section and throughout the paper, we work over characteristic 2.

\begin{define} Suppose $\cA$ is an associative algebra. A \emph{left $A_\infty$-module} ${}_{\cA} M$ over $\cA$ is a left $\ve{k}$-module $M$ equipped with $\ve{k}$-module map
\[
m_{j+1}\colon {\cA}^{\otimes j}\otimes_{\ve{k}}M\to M
\]
for each $j\ge 0$, such that if $a_n,\dots, a_1\in \cA$ and $\xs\in M$, then
\[
\sum_{j=0}^n m_{n-j+1}(a_n,\dots,a_{j+1}, m_{j+1}(a_{j},\dots, a_{1},\xs))+\sum_{k=1}^{ n-1} m_{n}(a_n,\dots, a_{k+1}a_{k},\dots, a_1, \xs)=0.
\]
\end{define}

An $A_\infty$-module ${}_{\cA} M$ is \emph{strictly unital} if $m_2(1,\ve{x})=\ve{x}$ for all $\ve{x}\in M$, and such that for each $j>1$, $m_{j+1}(a_j,\dots, a_1,\ve{x})=0$ if $a_i\in \ve{k}$ for some $i$.

\begin{define}
If $\cA$ is an associative algebra over $\ve{k}$, a right type-$D$ module $N^{\cA}$ is right $\ve{k}$-module $N$, equipped with a $\ve{k}$-linear structure map
\[
\delta^1\colon N\to N\otimes_{\ve{k}} \cA,
\]
which satisfies 
\[
(\id_N\otimes \mu_2)\circ (\delta^1\otimes \id_{\cA})\circ \delta^1=0.
\]
\end{define}
Note that in the above, we are only considering associative algebras, instead of $dg$-algebras or $A_\infty$-algebras. See \cite{LOTBimodules}*{Definition~2.2.23} for the general definition.

\subsection{Bimodules}

We also need to consider $A_\infty$-bimodules. Our exposition and notation follows \cite{LOTBimodules}, though we restrict to the case of associative algebras (as opposed to $dg$-algebras or $A_\infty$-algebras).

\begin{define}
Suppose that $\cA$ and $\cB$ are algebras over $\ve{j}$ and $\ve{k}$, respectively. A type-$DA$ bimodule, denoted ${}_{\cA} M^{\cB}$, consists of a $(\ve{j}$,$\ve{k})$-module $M$, equipped with $(\ve{j},\ve{k})$-linear structure morphisms
\[
\delta_{j+1}^1\colon \cA^{\otimes j} \otimes_{\ve{j}} M\to M\otimes_{\ve{k}} \cB,
\]
for $j\ge 0$. These are required to  satisfy the following structure relation:
\[
\begin{split}
&\sum_{j=0}^n ((\id_M\otimes \mu_2)\circ (\delta^1_{n-j+1}\otimes \id_{\cB}))(a_n,\dots,a_{j+1}, \delta^1_{j+1}(a_{j},\dots, a_{1},\xs))\\
+&\sum_{k=1}^{n-1} \delta^1_{n}(a_n,\dots, a_{k+1}a_k,\dots, a_1, \xs)=0.
\end{split}
\]
\end{define}

\subsection{The Lipshitz--Ozsv\'{a}th--Thurston box tensor product}

In this section, we recall the \emph{box tensor product} of Lipshitz, Ozsv\'{a}th and Thurston \cite{LOTBordered}*{Section~2.4}.

Suppose that $\cA$ is an associative algebra over $\ve{k}$. If $N^{\cA}$ is a type-$D$ module and ${}_{\cA} M$ is a type-$A$ module, then the \emph{box tensor product} $N^{\cA}\boxtimes {}_{\cA} M$ is the chain complex $N\otimes_{\ve{k}} M$, with differential $\d^{\boxtimes}$
as follows.  Following \cite{LOTBordered}, it is convenient to make the inductive notation that 
\begin{equation}
\delta^j\colon N\to N\otimes {\cA}^{\otimes j}
\label{eq:def-delta^j}
\end{equation}
 is the map given inductively by $\delta^j=( \delta^{1}\otimes \id_{\cA^{\otimes j-1}})\circ \delta^{j-1}$ and $\delta^0=\id$. Then the differential has the formula
\[
\d^{\boxtimes}(\xs\otimes \ys)=\sum_{j=0}^\infty (\id_N\otimes m_{j+1})(\delta^j(\xs),\ys).
\]
The differential is usually depicted via the following diagram:
\begin{equation}
\d^{\boxtimes}(\xs\otimes \ys)=\sum
\begin{tikzcd}[row sep=.3cm]
\xs
	\ar[d]
& \ys 
\ar[dddd] 
\\
\delta^1
	\ar[d]
	\ar[dddr]
	&
\\
\vdots 
	\ar[ddr,Rightarrow]
	\ar[d]
\\
\delta^1
	\ar[dd]
\ar[dr]
\\
\, &m_*\ar[d]\\
\,&\, 
\end{tikzcd}
\label{eq:def-box-tensor}
\end{equation}
According to \cite{LOTBordered}*{Lemma~2.30}, the map $\d^{\boxtimes}$ is a differential whenever one of $M$ or $N$  satisfies a \emph{boundedness} assumption; see \cite{LOTBimodules}*{Definitions~2.2.18 and~2.2.29}.

There are also tensor products between various types of bimodules. The following diagram indicates schematically the box tensor products of the form $AD\boxtimes AD$, $AD\boxtimes A$ and $D\boxtimes AD$:
\[
\delta^{AD\boxtimes AD}=\hspace{-.3cm}\begin{tikzcd}[row sep=.3cm, column sep=.4cm]
\ve{a}\ar[d,Rightarrow]&\xs
	\ar[dd]
	& 
\ys
	\ar[ddddd] \\
\Delta
	\ar[dr,Rightarrow]
	\ar[ddr,Rightarrow]
	\ar[dddr,Rightarrow]
&\\
&\delta^1_* \ar[rddd] \ar[d]
\\
& \vdots \ar[ddr,Rightarrow]\ar[d]
\\
&\delta_*^1
	\ar[dd]
	\ar[dr]&
	\\
&\, 
&
\delta^1_*
	\ar[d]
	\ar[dr]\\
&\,&\,&\,
\end{tikzcd} \delta^{AD\boxtimes A}=\hspace{-.3cm}\begin{tikzcd}[row sep=.3cm, column sep=.4cm]
\ve{a}\ar[d,Rightarrow]&\xs
	\ar[dd]
	& 
\ys
	\ar[ddddd] \\
\Delta
	\ar[dr,Rightarrow]
	\ar[ddr,Rightarrow]
	\ar[dddr,Rightarrow]
&\\
&\delta^1_* \ar[rddd]\ar[d]
\\
& \vdots \ar[ddr,Rightarrow]\ar[d]
\\
&\delta_*^1
	\ar[dd]
	\ar[dr]&
	\\
&\, 
&
m_*
	\ar[d]\\
&\,&\,&\,
\end{tikzcd}
 \delta^{D\boxtimes AD}= \hspace{-.5cm}
\begin{tikzcd}[row sep=.3cm, column sep=.4cm]
&\xs
	\ar[dd]
	& 
\ys
	\ar[ddddd] \\
&\\
&\delta^1 \ar[rddd]\ar[d]
\\
& \vdots \ar[ddr,Rightarrow]\ar[d]
\\
&\delta^1
	\ar[dd]
	\ar[dr]&
	\\
&\, 
&
\delta^1
	\ar[d]
	\ar[dr]\\
&\,&\,&\,
\end{tikzcd} 
\]

\subsection{Morphisms of modules and bimodules}

We now recall the definitions of morphisms of type-$D$, type-$A$ and type-$DA$ modules and bimodules from \cite{LOTBimodules}. We assume that all algebras are associative algebras with vanishing differential. We refer the reader to \cite{LOTBimodules} for the case of $dg$ or $A_\infty$-algebras.

Suppose $\cA$ is an algebra over $\ve{k}$ and $ N^{\cA}$ and $M^{\cA}$ are two type-$D$ modules. A morphism $f^1\colon  N^{\cA}\to  M^{\cA}$ is a $\ve{k}$-linear map
\[
f^1\colon N\to M\otimes_{\ve{k}} {\cA}.
\]
There is a morphism differential, given by
\begin{equation}
\d_{\Mor}(f^1)=(\id_M\otimes\mu_2)\circ (f^1\otimes \id_{\cA})\circ \delta^1+(\id_M\otimes\mu_2)\circ (\delta^1\otimes \id_{\cA})\circ f^1.
\label{eq:morphism-differential-type-D}
\end{equation}
The map $\d_{\Mor}$ squares to zero. A type-$D$ \emph{homomorphism} (or \emph{cycle}) is a morphism which satisfies $\d_{\Mor}(f^1)=0$.

A morphism $f_*$ between two type-$A$ modules ${}_{\cA} M$ and ${}_{\cA} N$ is a collection of $\ve{k}$-linear maps
\[
f_{j+1} \colon  \cA^{ \otimes j} \otimes_{\ve k} M\to N.
\]
There is a morphism differential $\d_{\Mor}(f_*)$ given by
\[
\begin{split}
\d_{\Mor}(f_*)_{n+1}(a_n,\dots, a_1, \xs)=&\sum_{j=0}^{n} m_{n-j+1}(a_n,\dots, a_{j+1}, f_{j+1}(a_{j},\dots, a_1, \xs))\\
+&\sum_{j=0}^{n} f_{n-j+1}(a_n,\dots, a_{j+1}, m_{j+1}(a_{j},\dots, a_1, \xs))\\
+&\sum_{i=1}^{n-1} f_{n}(a_n,\dots, a_{i+1}a_{i}, \dots, a_1,\xs).
\end{split}
\]

Finally, if ${}_\cA M^\cB$ and ${}_\cA N^\cB$ are $DA$-bimodules, then a morphism $f_*^1$ from ${}_{\cA} M^{\cB}$ to ${}_{\cA} N^{\cB}$ is a collection of maps
\[
f^1_{j+1}\colon \cA^{\otimes j}\otimes M\to N\otimes \cB
\]
for $j\ge 0$. The morphism differential is given by
\[
\begin{split}
&\d_{\Mor}(f^1_*)_{n+1}(a_n,\dots, a_1, \xs)\\
=&\sum_{j=0}^n (\id_N\otimes \mu_2) (\delta_{n-j+1}^1\otimes \id_{\cB})(a_n,\dots, a_{j+1}, f_{j+1}^1(a_{j},\dots, a_1, \xs))\\
+&\sum_{j=0}^n (\id_N\otimes \mu_2) (f_{n-j+1}^1\otimes \id_{\cB})(a_n,\dots, a_{j+1}, \delta^1_{j+1}(a_{j},\dots, a_1, \xs))\\
+&\sum_{i=1}^{n-1} f_{n}^1(a_n,\dots, a_{i+1}a_{i}, \cdots a_1,\xs).
\end{split}
\]
Of course, the morphism differential $\d_{\Mor}$ also squares to zero for morphisms of type-$A$ and $DA$-morphisms. See \cite{LOTBimodules}*{Definition~2.2.43}.

If $f_*^1\colon {}_\cA M^\cB\to {}_\cA N^\cB$ and $g_*^1 \colon {}_\cA N^\cB\to {}_\cA Q^\cB$ are $DA$-bimodule morphisms, Lipshitz--Ozsv\'{a}th--Thurston define the composition $f_*^1\circ g_*^1$ via the following diagram:
\begin{equation}
g_*^1\circ f_*^1=\begin{tikzcd}[row sep=.4cm]\ve{a}\ar[d, Rightarrow]& \ve{x} \ar[dd]\\
\Delta\ar[dr,Rightarrow,bend right=8]\ar[ddr,Rightarrow,bend right=4] &\,\\
& f^1_*\ar[d]\ar[ddr,bend left=8]\\
& g^1_*\ar[d]\ar[dr, bend left =4]\\
&\,&\mu_2
\end{tikzcd}
\label{eq:compose-morphisms}
\end{equation}
In Equation~\eqref{eq:compose-morphisms}, $\Delta$ denotes the co-multiplication
\[
\Delta\colon \cA^{\otimes n}\to \bigoplus_{i=0}^{n} \cA^{\otimes i}\otimes \cA^{\otimes (n-i)}
\]
given by
\[
\Delta(a_1|\dots|a_n)=\sum_{i=1}^n (a_1|\cdots|a_i)\otimes (a_{i+1}|\cdots |a_n).
\]
In the above, $|$ denotes tensor product.

We now discuss pairing morphisms of modules and bimodules. We focus on the case of $DA$-bimodule morphisms. The other types of pairings which are relevant to our paper are defined by forgetting about the input algebra or output algebra in the diagrams below.
The natural algebraic operation is to pair a bimodule morphism with the identity (as opposed to pairing two morphisms together). Suppose that ${}_\cA N^\cB$ and ${}_\cA M^\cB$ and ${}_\cB P^\cC$ are $DA$-bimodules, and
\[
f_*^1 \colon {}_\cA N^\cB\to {}_\cA M^\cB
\]
is a morphism of $DA$-bimodules. There is a morphism of $DA$-bimodules $f_*^1\boxtimes \bI\colon N\boxtimes P\to M\boxtimes P$, defined by the following diagram
\[
f_*^1\boxtimes \bI=\begin{tikzcd}[row sep=.2cm]
\ve{a} \ar[d,Rightarrow]&\xs\ar[dd]& \ys\ar[ddddd] \\
\Delta\ar[dr, bend right=10, Rightarrow] \ar[ddr,bend right=12,Rightarrow]\ar[dddr,bend right=14,Rightarrow]  \\
\,&\delta\ar[d] \ar[dddr, bend left=10,Rightarrow]&\,\\
&f_*^1\ar[d] \ar[ddr, bend left=6] \\
&\delta \ar[dd]\ar[dr,Rightarrow]&\,\\
&\,  &\delta^1 \ar[d] \ar[dr]&\\
&\,&\,&\,\\
\end{tikzcd}
\]
In the above, the map $\delta$ denotes the sum over $n\ge 0$ of the maps 
\[
\delta_{j+1}^n\colon\underbrace{\cA\otimes \cdots \otimes \cA }_j\otimes M\to M\otimes \underbrace{\cB\otimes \cdots \otimes \cB}_n
\]
obtained by composing $\delta_*^1$ repeatedly, but not multiplying the $\cB$-outputs. (Note that by convention $\delta_*^0=\bI$).

If instead $g_*^1\colon {}_\cB P^\cC\to {}_\cB Q^\cC$ is a morphism of type-$DA$ modules, then $\bI\boxtimes g_{*}^1$ is given by the diagram
\[
\begin{tikzcd}[row sep=.3cm, column sep=.4cm]
\ve{a}\ar[dr,Rightarrow]&\xs\ar[d]& \ys \ar[dd] \\
&\delta\ar[dd] \ar[dr,Rightarrow]&\\
&\, &g^1_*\ar[d] \ar[dr]\\
&\,&\,&\,
\end{tikzcd}
\]

\subsection{External tensor products}
\label{sec:extension-of-scalars}

In this section, we recall a basic algebraic operation called the \emph{external tensor product}. We consider these for type-$D$ and type-$DA$ modules over associative algebras, and we prove some basic properties. 

 We begin with the setting of type-$D$ modules. Let $\cA$ and $\cB$ algebras over rings $\ve{i}$ and $\ve{j}$ of characteristic 2, and let $N^{\cA}=(N,\delta^1)$ and $M^{\cB}=(M,d^1)$ be type-$D$ modules. Then the \emph{external tensor product}, denoted $(N\otimes_{\bF} M)^{\cA\otimes_{\bF} \cB}$, is the group $M\otimes_{\bF} N$ equipped with the structure map $\id_N\otimes 1_{\cA}\otimes d^1+\delta^1\otimes \id_M\otimes 1_{\cB}$ (with tensor factors reordered).

The external tensor product of two type-$DA$ bimodules is more complicated. In general, one must choose a cellular diagonal of the associahedron. We follow the presentation Lipshitz, Ozsv\'{a}th and Thurston \cite{LOTBordered} \cite{LOTDiagonals}. The main goal of this section is to give a concrete formula (not involving diagonals of the associahedron) of the external tensor product of two $DA$-bimodules, which we use throughout the paper. This formula appears in the statement of Proposition~\ref{prop:commute-extension-scalars}.

 Our description is in terms of a more basic operation, which we call  \emph{extension of scalars} (and is also inspired by \cite{LOTBimodules}). Suppose that $\cA$,  $\cB$ and $\cC$ are algebras over rings $\ve{i}$, $\ve{j}$ and $\ve{k}$, respectively. If ${}_{\cA} N^{\cB}$ is a bimodule, then we will describe a bimodule ${}_{\cA\otimes \cC} \cE(N;\cC)^{\cB\otimes \cC}$ which we say is formed by extension of scalars. We will relate this operation with an external tensor product with the identity bimodule ${}_{\cC}[\bI]^{\cC}$ in Proposition~\ref{prop:extension=external}.

 As an $(\ve{i}\otimes \ve{k}, \ve{j}\otimes \ve{k})$-module, we declare $\cE(N;\cC):=N\otimes_{\bF} \ve{k}$. If $k\in \ve{k}$, then we define
\begin{equation}
\delta_{j+1}^1( a_j|c_j,\dots, a_1|c_1,\xs|k)= \delta_{j+1}^1(a_j,\dots, a_1,\xs)\otimes k\otimes (c_j\cdots c_1).
\label{eq:extension-of-scalars-def}
\end{equation}
Note that $\delta_1^1(\ve{x}|k)=\delta_1^1(\ve{x})\otimes k\otimes 1_{\cC}.$ (Note we are implicitly reordering the $k$ factor and the $\cB$ factor above).

Given a morphism of $DA$-bimodules $f^1_{*}\colon {}_\cA N^\cB\to {}_{\cA} M^{\cB}$, there is a morphism 
\[
\cE(f)^1_{*}\colon {}_{\cA\otimes \cC} \cE(N;\cC)^{\cB\otimes \cC}\to {}_{\cA\otimes \cC} \cE(M;\cC)^{\cB\otimes \cC}
\]
given by replacing $\delta_{j+1}^1$ with $f_{j+1}^1$ in Equation~\eqref{eq:extension-of-scalars-def}.

\begin{rem}If ${}_{\cA} M^{\cB}$ is operationally bounded (i.e. $\delta_j^1=0$ for large $j$), then the same is true for $\cE(M,\cC)$. A similar statement holds for morphisms.
\end{rem}

\begin{lem}
\label{lem:extension-properties-1}
 Suppose that $\cA$, $\cB$, and $\cC$ are algebras and suppose that ${}_{\cA}N^{\cB}$ is a $DA$-bimodule.
\begin{enumerate}
\item ${}_{\cA\otimes \cC}\cE(N;\cC)^{\cB\otimes \cC}$ satisfies the $DA$-bimodule structure relations.
\item If $f^1_*$ is a morphism of $DA$-bimodules, then $\d_{\Mor}(\cE(f^1_*))=\cE(\d_{\Mor}(f_*^1))$.
\item If ${}_{\cA} M^{\cB}$ and ${}_{\cB} N^{\cD}$ are bimodules and ${}_{\cB} N^{\cD}$ is operationally bounded, then there is a canonical isomorphism of $DA$-bimodules
\[
{}_{\cA\otimes \cC}\cE({}_{\cA} M^{\cB}\boxtimes {}_{\cB} N^\cD;\cC)^{\cD\otimes \cC}\iso {}_{\cA\otimes \cC} \cE(M;\cC)^{\cB\otimes \cC}\boxtimes {}_{\cB\otimes \cC}\cE(N;\cC)^{\cD\otimes\cC}.
\]
\item With respect to the above isomorphism, we also have 
\[
\cE(f_*^1\boxtimes \bI)=\cE(f_*^1)\boxtimes \bI\quad \text{and} \quad \cE(\bI \boxtimes g_{*}^1)=\bI\boxtimes \cE(g_*^1),
\]
for any $DA$-bimodule morphisms $f_*^1$ and $g_*^1$.
\end{enumerate}
\end{lem}
\begin{proof} The first claim follows from the bimodule relations for $N$, together with the fact that multiplication on $\cC$ is associative.

For the second claim, we note that each term of $\d_{\Mor}(\cE(f_*^1))$, when evaluated on $(b_n|c_n),\dots, (b_1|c_1),\xs|1_{\ve{k}}$, is equal to a corresponding term of $\d_{\Mor}(f_*^1)$. It is easy to see that all of the terms cancel.

For the third claim, the map is $f^1_1(\xs\otimes \ys)=\xs\otimes \ys\otimes 1_{\cA\otimes \cC}$ and $f_j^1=0$ if $j>1$. It is easy to see that this is an isomorphism of $DA$ bimodules.

The final claim is verified in a similar manner.
\end{proof}

\begin{rem}
\label{rem:tensor-DD-with-A}
The external tensor product gives a way to tensor a type-$D$ module $\cX^{\cA\otimes \cB}$ with a type-$A$ module ${}_{\cA} \cY$. Namely, we take the ordinary box tensor product of $\cX^{\cA\otimes \cB}$ with ${}_{\cA\otimes \cB}\cE(\cY, \cB)^{\cB}$. The reader may recognize that this procedure is identical to Lipshitz, Ozsv\'{a}th and Thurston's box tensor product between a type-$DD$ module and a type-$A$ module \cite{LOTBimodules}*{Figure~4}.
\end{rem}

We now consider an additional commutativity property of extension of scalars:

\begin{prop}
\label{prop:commute-extension-scalars}Suppose that ${}_{\cA} M^{\cB}$ and ${}_{\cC} N^{\cD}$ are $DA$-bimodules over associative algebras. Then 
\[
{}_{\cA\otimes \cC}\cE(M, \cC)^{\cB\otimes \cC}\boxtimes {}_{\cB\otimes \cC}\cE(N,\cB)^{\cB\otimes \cD}\quad \text{and}\quad  {}_{\cA\otimes \cC} \cE(N,\cA)^{\cA\otimes \cD}\boxtimes {}_{\cA\otimes \cD}\cE(M,\cD)^{\cB\otimes \cD}
\]
are homotopy equivalent as $DA$-bimodules. In fact, both are homotopy equivalent to the external tensor product of ${}_{\cA} M^{\cB}$ and ${}_{\cC} N^{\cD}$ along a cellular diagonal of the associahedron (defined below). 
\end{prop}

\begin{rem} If one of $M$ and $N$ has the property that $\delta_j^1=0$ for $j>2$, then the above Proposition may be easily proven directly since the structure maps on the two tensor products are identical.
\end{rem}

 As alluded to in its statement, Proposition~\ref{prop:commute-extension-scalars} is most naturally proven by considering a more general tensor product operation on bimodules,  called the \emph{external tensor product}. The construction requires a choice of a cellular diagonal of Stasheff's associahedron \cite{StasheffKn}. We recall some background about diagonals presently. Our exposition will follow the work of Lipshitz, Ozsv\'{a}th and Thurston \cite{LOTDiagonals}, though we refer the reader to earlier work of Saneblidze-Umble \cite{SaneblidzeUmble} on diagonals of the associahedron. See also \cite{MarklShnider} and \cite{Loday}.
 
 We recall that the $(n-2)$-dimensional associahedron, denoted $K_n$, has a cell decomposition which admits a simple combinatorial description, as follows. We consider a chain complex $C_*^{\cell}(K_n)$ over $\bF$ whose generators are identified with planar trees $T$, which have  have $n+1$ valence 1 vertices, which are partitioned into $n$ ``input'' vertices and one ``output'' vertex. Furthermore, trees $T$ are assumed to have no vertices of valence 2. The dimension of a generator $[T]$ is given by
 \[
\deg(T)= \sum_{v\in V_{\Int}(T)} (\val(v)-3)
 \]
 where $V_{\Int}(T)$ denotes the interior vertices of $T$.   The differential on $C_*^{\cell}(K_n)$ is defined by setting $\d(T)$ to be the sum of all trees obtained by splitting an interior vertex into two interior vertices by adding a single edge. 
 
 The external tensor product operation has as input a pair of $DA$-bimodules ${}_{\cA} M^{\cB}$ and ${}_{\cC} N^{\cD}$, and has as output a $DA$-bimodule ${}_{\cA\otimes \cC} (M\otimes_\Gamma N)^{\cB\otimes \cD}$. The underlying vector space of $M\otimes_\Gamma N$ is the ordinary tensor product $M\otimes_{\bF} N$. The $DA$-bimodule structure map $\delta_{j+1}^1$ depends on an extra piece of data, called a \emph{cellular diagonal of the associahedron} $\Gamma$:

\begin{define}[\cite{LOTDiagonals}*{Definition~2.13}]
\label{def:cellular-diagonal} A \emph{cellular diagonal of the associahedron} $\Gamma$ consists of a collection of degree preserving chain maps
\[
\Gamma_n\colon C_*^{\cell}(K_{n})\to C_*^{\cell}(K_{n})\otimes_\bF C_*^{\cell}(K_{n}),
\]
for $n\ge 2$, which are compatible with concatenation of trees (in the sense described below), and such that $\Gamma_2$ is the canonical isomorphism between $C_*^{\cell}(K_2)\iso \bF$ and $C_*^{\cell}(K_2)\otimes_{\bF} C_*^{\cell}(K_2)\iso \bF$.
\end{define}

   The compatibility condition is as follows. Firstly, there is a natural splicing operation 
\[
\circ_i \colon C_*^{\cell}(K_n)\otimes C_*^{\cell}(K_m)\to C_*^{\cell}(K_{n+m-1})
\]
where $T\circ_i S$ joins the output of $S$ into the $i$-th input of $T$. We define $\circ_i$ also for tensors of trees by the formula
\[
(T_1\otimes T_2)\circ_i (S_1\otimes S_2)=(T_1\otimes_i S_1)\otimes (T_2\otimes_i S_2).
\]
The compatibility condition of Definition~\ref{def:cellular-diagonal} is that if $T$ and $S$ have $n$ and $m$ inputs, respectively, then
\[
\Gamma_{n+m-1}(T\circ_i S)=\Gamma_n(T)\circ_i \Gamma_m(S).
\]

 We write $\Psi_n\in C_*^{\cell}(K_n)$ for the tree with $n$ inputs and exactly one interior vertex. The generator $\Psi_n$ is called the \emph{corolla} with $n$ inputs and has homological degree $n-2$. This is the maximal non-trivial grading of $C_*^{\cell}(K_n)$.

If ${}_{\cA} M^{\cB}=(M,\delta_*^1)$ is a $DA$-bimodule, then to each planar tree $T$ with $(n+1)$-inputs, we may form an evaluation map $\delta_T^1\colon \cA^{\otimes n} \otimes M\to M\otimes \cB$ as follows. We first compose the maps $\delta_j^1$ according to the tree $T$. This naturally gives a map from $\cA^{\otimes n} \otimes M$ to $M\otimes  \cB^{\otimes k}$, where $k$ is the number of vertices in the right-most path of $T$. We then apply $\mu_2^{\cB}$ repeatedly
to obtain a map from $\cA^{\otimes n}\otimes M$ to $M\otimes \cB$.

This gives a map
\[
\ev\colon C_*^{\cell}(K_{n+1})\to \Hom_{(\ve{k},\ve{j})}(\cA^{\otimes n}\otimes  M, M\otimes \cB).
\]
We equip $\Hom_{(\ve{k},\ve{j})}(\cA^{\otimes n}\otimes  M, M\otimes \cB)$ with a differential $\d$ given by
\[
\d(f_{j+1}^1)=\delta_1^1\circ f_{j+1}^1+f_{j+1}^1\circ \delta_1^1.
\]
The statement that the structure maps $\delta_*^1$ of ${}_{\cA} M^\cB$ satisfy the $DA$-bimodule relations is equivalent to the statement that $M^{\cB}\boxtimes {}_{\cB}\cB$ is a chain complex, and that $\ev$ is a chain map.

We define $\delta_1^1$ on $M\otimes_{\bF} N$ via the Leibniz rule, and for $n>0$ we define the structure map $\delta_{n+1}^1$  as $(\ev\otimes \ev)(\Gamma_{n+1}(\Psi_{n+1}))$. The fact that a $DA$-bimodule structure relations hold is immediate.

Since we are focused only on associative algebras, it is helpful to instead focus our attention on a slightly more coarse definition of a diagonal. 

\begin{define}
For $n\ge 2$, we define a chain complex $M^{\red}_{n,*}$ as the following modification of $C_*^{\cell}(K_n)$. We call $M^{\red}_{n,*}$ the complex of \emph{left reduced module trees}. The vector space $M^{\red}_{n,*}$ is generated by planar trees $T$ with $n$-inputs, such that all vertices of valence $4$ or more occur along the right most path in $T$ from an input to the output. Additionally, we quotient by the relation that $T\circ_i S=T\circ_i S'$ whenever $i<n$ and $S$ and $S'$ have degree 0 (i.e. only valence 3 internal vertices) and the same number of inputs. See Figure~\ref{fig:40} for examples. 
\end{define}

\begin{figure}[ht]
\begingroup%
  \makeatletter%
  \providecommand\color[2][]{%
    \errmessage{(Inkscape) Color is used for the text in Inkscape, but the package 'color.sty' is not loaded}%
    \renewcommand\color[2][]{}%
  }%
  \providecommand\transparent[1]{%
    \errmessage{(Inkscape) Transparency is used (non-zero) for the text in Inkscape, but the package 'transparent.sty' is not loaded}%
    \renewcommand\transparent[1]{}%
  }%
  \providecommand\rotatebox[2]{#2}%
  \newcommand*\fsize{\dimexpr\f@size pt\relax}%
  \newcommand*\lineheight[1]{\fontsize{\fsize}{#1\fsize}\selectfont}%
  \ifx\svgwidth\undefined%
    \setlength{\unitlength}{273.09102937bp}%
    \ifx\svgscale\undefined%
      \relax%
    \else%
      \setlength{\unitlength}{\unitlength * \real{\svgscale}}%
    \fi%
  \else%
    \setlength{\unitlength}{\svgwidth}%
  \fi%
  \global\let\svgwidth\undefined%
  \global\let\svgscale\undefined%
  \makeatother%
  \begin{picture}(1,0.32367806)%
    \lineheight{1}%
    \setlength\tabcolsep{0pt}%
    \put(0,0){\includegraphics[width=\unitlength,page=1]{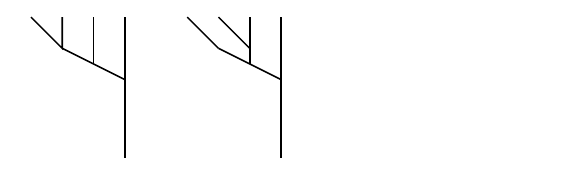}}%
    \put(0.270223,0.17129654){\makebox(0,0)[lt]{\lineheight{1.25}\smash{\begin{tabular}[t]{l}$=$\end{tabular}}}}%
    \put(0,0){\includegraphics[width=\unitlength,page=2]{fig40.pdf}}%
    \put(0.9240259,0.16876524){\makebox(0,0)[lt]{\lineheight{1.25}\smash{\begin{tabular}[t]{l}$=0$\end{tabular}}}}%
    \put(0,0){\includegraphics[width=\unitlength,page=3]{fig40.pdf}}%
  \end{picture}%
\endgroup%

\caption{Equalities of trees in $M_{n,*}^{\red}$.}
\label{fig:40}
\end{figure}

We define the differential $\d$ on $M^{\red}_{n,*}$ the same as in $C_*^{\cell}(K_n)$. It is not hard to see that $\d^2=0$.

For the purposes of defining the diagonal tensor product of two $DA$-bimodules over associative algebras, we will use the following notion:

\begin{define}A \emph{left-reduced module diagonal} is a collection of grading preserving chain maps
\[
\Gamma_n\colon M_{n,*}^{\red}\to M_{n,*}^{\red}\otimes M_{n,*}^{\red}
\]
which are compatible with splicing of trees.
\end{define}

\begin{lem}
\label{lem:M-acyclic} The complex $M^{\red}_{n,*}$ has homology $\bF$ in degree 0, and $0$ in other degrees.
\end{lem}
\begin{proof} We claim that the homology is generated by a tree 
\[
T_{\mathrm{gen}}:=\Psi_2\circ_1 T_0,
\] where $T_0$ is a degree 0 tree with $n-1$ inputs (all such trees are equal in $M_{n,*}^{\red}$). To see this, we argue as follows. We define a map
\[
H\colon M_{n,*}^{\red}\to M_{n,*+1}^{\red}
\]
as follows. If $T$ is a tree, we say the \emph{right-most path} of $T$ is the path from the right-most input of $T$ to its root. We write $RP(T)$ for this path. Given a tree $T$, if $T$ has more than one interior vertex along $RP(T)$, then we contract the top two interior vertices along this path to define a tree $T'$. We set $H(T)=T'$. Otherwise we define $H(T)=0$.

We claim that if $N\gg 0$, the map $(\bI+[\d, H])^N$ is projection onto the span of $T_{\mathrm{gen}}$. Note that this clearly proves the claim. To see this, we argue as follows.  Given a tree, we define the \emph{right-path weight} to be
\[
w_R(T):=\sum_{v\in RP(T)} (\val(v)-2) 
\]
We define the subspace $\cF_j\subset M_{n,*}^\red$ to be the span of trees $T$ with $w_R(T)\le j$.

It is straightforward to check that if $j>1$, then $\bI+[\d, H]$ maps $\cF_{j}$ into $\cF_{j-1}$. Furthermore, $\bI+[\d,H]$ is the identity on $\cF_1$. Therefore, for $N$ large enough, $(\bI+[\d,H])^N$ gives a homotopy equivalence onto $\cF_1$ (viewing $\cF_1$ as a subcomplex with vanishing differential). However $\cF_1$ is exactly the span of $T_{\mathrm{gen}}$, so the proof is complete.
\end{proof}

For our theory, we need to show that if ${}_\cA M$ and ${}_\cB N$ are $A_\infty$-modules over associative algebras $\cA$ and $\cB$, then ${}_{\cA\otimes \cB} (M\otimes_\Gamma N)$ is independent of the choice of left-reduced module diagonal $\Gamma$. The argument is identical to the work in \cite{LOTDiagonals}*{Section~3}, though we include it for completeness. Firstly, we introduce the following notion:

\begin{define}
 A \emph{left-reduced module morphism tree} consists of a planar input tree $T$ which has exactly one distinguished vertex along its right-most path, and whose vertices of valence 4 or more are all along the right-most path. We call the distinguished vertex the \emph{morphism vertex}. This vertex has valence 2 or greater, whereas the other interior vertices have valence 3 or greater. We quotient this set of trees by the relation that $T\circ_i S=T\circ_i S'$ whenever $T$ has $n$ inputs, $i<n$, and $S$ and $S'$ have degree 0 and the same number of inputs.  
\end{define}

The degree of a left reduced module tree $T$ is
\[
\deg(T)=1+\sum_{v\in V_0(T)} (\val(v)-3).
\]

We write $X^{\red}_{n,*}$ for the vector space of left-reduced module morphism trees. We can naturally equip $X^{\red}_{n,*}$ with a differential which sums over all ways of splitting one interior vertex into two interior vertices while leaving a valid left-reduced module morphism tree.

\begin{lem}\label{lem:Xn*-homology} The chain  complex $X_{n,*}^{\red}$ has homology equal to $\bF$ in degree 0, and 0 in other degrees.
\end{lem} 
The proof of Lemma~\ref{lem:Xn*-homology} is very similar to the proof of Lemma~\ref{lem:M-acyclic},  and we leave the details to the reader.

\begin{define} If $\Gamma$ and $\Gamma'$ are two left-reduced module diagonals, then a left-reduced module morphism diagonal from $\Gamma$ to $\Gamma'$ consists of a collection of chain maps
\[
\hat{\Gamma}_n\colon X_{n,*}^\red\to X_{n,*}^{\red}\otimes X_{n,*}^{\red}.
\]
such that $\hat{\Gamma}_1$ is an isomorphism, and $\hat{\Gamma}_n$ is compatible with splicing in the following sense. If $S\in X_{n,*}^{\red}$ and $T\in M_{m,*}^{\red}$, then
\[
\hat{\Gamma}_{n+m-1}(S\circ_n T)=\hat{\Gamma}_n(S)\circ_n \Gamma_m(T)\quad \text{and}\quad \hat{\Gamma}_{n+m-1}(T\circ_m S)=\Gamma'_m(T)\circ_m \hat{\Gamma}_n(S).
\]
Finally, if $T_0\in M_{m,0}^{\red}$ is a degree 0 tree and $S\in X_{n,*}^{\red}$, then
\[
\hat{\Gamma}_{n+m-1}(S\circ_i T_0)=\hat{\Gamma}_{n}(S) \circ_i (T_0\otimes T_0).
\]
\end{define}

Using Lemma~\ref{lem:Xn*-homology}, left-reduced module morphism diagonals may be constructed by a straightforward inductive argument.

\begin{rem}
\label{rem:tensor-morphisms-diagonals}
If $A$ and $B$ are associative algebras and $f_*\colon {}_{A} M_1\to {}_A M_2$ and $g_*\colon {}_B N_1\to {}_B N_2$ are two $A_\infty$-module morphisms and $\hat{\Gamma}$ is a left-reduced module morphism diagonal, then we can define an $A_\infty$-module morphism
\[
(f_*\otimes_{\hat{\Gamma}} g_*)\colon {}_{A\otimes B} (M_1\otimes_{\Gamma} N_1)\to {}_{A\otimes B} (M_2\otimes_{\Gamma'} N_2).
\]
Applying the above construction to the case that $M_1=M_2=M$ and $N_1=N_2=N$, we obtain a chain map $I_*$ from $(M\otimes_\Gamma N)$ to $(M\otimes_{\Gamma'} N)$ which is a homotopy equivalence because $I_1=\bI_{M\otimes N}$ is an isomorphism. We may similarly use $\hat{\Gamma}$ to tensor $DA$-bimodule morphisms together when the underlying algebras are associative algebras. 
\end{rem}

The motivation for left-reduced module diagonals is that it is straightforward to describe explicitly a diagonal over them. We now define two left-reduced module diagonals
\[
L\Gamma_n, R\Gamma_n\colon M_n^{\red}\to M_n^{\red}\otimes M_n^{\red}.
\]
We focus on $L\Gamma_n$, since $R\Gamma_n$ is obtained by switching the coordinates of the codomain.

  To describe $L\Gamma_n$, we define the \emph{right-join} of two input trees $T_1$ and $T_2$, denoted $\RiJ(T_1,T_2)$, to be the tree obtained by 
 stacking the output of $T_2$ into the right most input vertex of $T_1$. (This is the mirror of Lipshitz, Ozsv\'{a}th and Thurston's \emph{left-join}, as adapted to our handedness conventions on modules). If $T_1,\dots, T_n$ is a collection of trees, we write
 \[
 \RiJ(T_1,\dots, T_n)=\RiJ(T_1,\dots, \RiJ(T_{n-1},T_n)),
 \]
 
 If $T_1,\dots, T_n$ are $n$ trees, the \emph{root-join} of $T_1,\dots, T_n$, denoted $\RoJ(T_1,\dots, T_n)$, is obtained by adjoining the outputs of $T_1,\dots, T_n$ into a single corolla $\Psi_n$. See \cite{LOTDiagonals}*{Definition 2.28}. 
 
 For $n>1$, we write $S_n\in H_0^{\cell}(K_n)$ for the generator (thought of as a tree with $n$ inputs and only valence 3 interior vertices).

 The trees-pairs appearing as summands of $L\Gamma_n(\Psi_n)$ consist of  the sum of all pairs $T\otimes S$ where 
 \[
 T=\RiJ(\Psi_{n_k+1},\dots, \Psi_{n_1+1})
 \]
  for some numbers $n_1,\dots, n_k>0$ (where $k>0$) satisfying
 \[
 n_1+\cdots+n_k+1=n,
 \] 
 and 
 \[
 S=\RoJ(S_{n_1},\dots, S_{n_k},\downarrow)
 \]
 This root join is taken in such a way that $\downarrow$ corresponds to the module input of the tree. See Figure~\ref{fig:32} for an example.

 \begin{lem}\label{lem:RGamma-LGamma} The maps $R\Gamma_n$ and $L\Gamma_n$, defined above, are left-reduced module diagonal.
 \end{lem}
The proof is straightforward and left to the reader.

\begin{figure}[ht]
\begingroup%
  \makeatletter%
  \providecommand\color[2][]{%
    \errmessage{(Inkscape) Color is used for the text in Inkscape, but the package 'color.sty' is not loaded}%
    \renewcommand\color[2][]{}%
  }%
  \providecommand\transparent[1]{%
    \errmessage{(Inkscape) Transparency is used (non-zero) for the text in Inkscape, but the package 'transparent.sty' is not loaded}%
    \renewcommand\transparent[1]{}%
  }%
  \providecommand\rotatebox[2]{#2}%
  \newcommand*\fsize{\dimexpr\f@size pt\relax}%
  \newcommand*\lineheight[1]{\fontsize{\fsize}{#1\fsize}\selectfont}%
  \ifx\svgwidth\undefined%
    \setlength{\unitlength}{153.59987514bp}%
    \ifx\svgscale\undefined%
      \relax%
    \else%
      \setlength{\unitlength}{\unitlength * \real{\svgscale}}%
    \fi%
  \else%
    \setlength{\unitlength}{\svgwidth}%
  \fi%
  \global\let\svgwidth\undefined%
  \global\let\svgscale\undefined%
  \makeatother%
  \begin{picture}(1,0.42613792)%
    \lineheight{1}%
    \setlength\tabcolsep{0pt}%
    \put(0,0){\includegraphics[width=\unitlength,page=1]{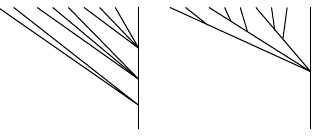}}%
    \put(0.50666561,0.20251566){\makebox(0,0)[lt]{\lineheight{1.25}\smash{\begin{tabular}[t]{l}$\otimes$\end{tabular}}}}%
    \put(0,0){\includegraphics[width=\unitlength,page=2]{fig32.pdf}}%
  \end{picture}%
\endgroup%

\caption{An example of a trees contributing to $L\Gamma_9(\Psi_9)$.}
\label{fig:32}
\end{figure}

\begin{proof}[Proof of Proposition~\ref{prop:commute-extension-scalars}]
 We observe that, immediately from the definitions, we have isomorphisms
 \[
 \begin{split}
{}_{\cA\otimes \cC} \cE(M,\cC)^{\cB\otimes \cC}\boxtimes{}_{\cB\otimes \cC} \cE(N,\cB)^{\cB\otimes \cD}&\iso {}_{\cA\otimes \cC} (N\otimes_{R\Gamma} M)^{\cB\otimes \cD}\\
{}_{\cA\otimes \cC} \cE(N,\cA)^{\cA\otimes \cD}\boxtimes {}_{\cA\otimes \cD}\cE(M,\cD)^{\cB\otimes \cD}&\iso {}_{\cA\otimes \cC} (N\otimes_{L\Gamma} M)^{\cB\otimes \cD}
\end{split}
\]
where $R\Gamma$ and $L\Gamma$ denote the left-reduced module diagonals from Lemma~\ref{lem:RGamma-LGamma}. This implies the main claim, because any two choices of module diagonals give homotopy equivalent bimodules by the argument described in Remark~\ref{rem:tensor-morphisms-diagonals}.
 \end{proof}

 We now prove a result which relates our extension of scalars operation with the external tensor product:
 
   \begin{prop}
   \label{prop:extension=external} Let $\cA$, $\cB$ and $\cC$ be associative algebras. If ${}_{\cA} N^{\cB}$ is a $DA$-bimodule, then ${}_{\cA\otimes \cC}\cE(N,\cC)^{\cB\otimes \cC}$ is isomorphic to the external tensor product of ${}_{\cA} N^{\cB}$ and ${}_{\cC} [\bI]^{\cC}$, for any choice of cellular diagonal of the associahedron (or more generally, module diagonal or left-reduced module diagonal).
   \end{prop}
   The key to the above result is the following lemma:
    \begin{lem}\label{lem:diagonal-asymmetric-terms} Suppose that $\Gamma$ is a cellular diagonal of the associahedron (or a module diagonal, or a left-reduced module diagonal). For each $n$, write $\Gamma_n(\Psi_n)=\Psi_n\otimes X+Y,$
    where $X$ has degree 0 and $Y\in C_*^\cell(K_n)\otimes C_{>0}^\cell(K_n)$. Write $X=T_1+\cdots +T_j$, where each $T_i$ is a tree of degree 0. Then $j\equiv 1 \pmod 2$.
    \end{lem}
    \begin{proof} The argument is the same, regardless of what kind of diagonal $\Gamma$ is, so we assume that $\Gamma$ is an ordinary cellular diagonal of the associahedron. There is a quasi-isomorphism $\Pi\colon C_*^\cell(K_n)\to \bF$, which is 0 on $C_{>0}^{\cell}(K_n)$. The map $\Pi$ sends a tree in grading 0 to $1\in \bF$, and is 0 on all other trees. It is straightforward to check that $\Pi$ is a chain map. By tensoring, we obtain a quasi-isomorphism $\bI\otimes \Pi\colon C_*^{\cell}(K_n)\otimes C_*^{\cell}(K_n)\to C_*^{\cell}(K_n)$.
    
     The claim in the statement equivalent to the claim that $(\bI\otimes \Pi)( \Gamma_n(\Psi_n))=\Psi_n$. This claim is equivalent to the claim that $(\bI\otimes \Pi)( \Gamma_n(\Psi_n))$ is non-zero, since $\Psi_n$ is the only tree in its grading.  We argue by induction. The case that $n=2$ is clear.
    
    By induction and the fact that $\Gamma$ respects splicing, we observe that $(\bI\otimes \Pi)(\Gamma_n(\d \Psi_n))=\d \Psi_{n}$. On the other hand, 
    \[
    (\bI\otimes \Pi)(\Gamma_n(\d \Psi_n))=\d \big( (\bI\otimes \Pi)(\Gamma_n(\Psi_n))\big).
    \]
    For this to be $\d \Psi_n$, we must have that $(\bI\otimes \Pi)(\Gamma_n( \Psi_n))$ is non-zero, so the proof is complete.
    \end{proof}
   
   \begin{proof} [Proof of Proposition~\ref{prop:extension=external}] The proof is the same, regardless of what kind of diagonal $\Gamma$ is, so we assume that $\Gamma$ is an ordinary cellular diagonal of the associahedron. 
   The map $\delta_{n+1}^1$ on the external tensor product is determined by the structure map on $N$ as well as the chain $\Gamma_{n+1}(\Psi_{n+1})$. If $T_i\otimes S_j$ is any summand of $\Gamma_{n+1}(\Psi_{n+1})$, where $T_i$ and $S_j$ are trees with $n+1$ inputs of degree $i$ and $j$, respectively, then the corresponding summand of $\delta_{n+1}^1$ on $\cE(N,\cC)$ will vanish unless $j=0$ and $i=n-1$ since $\cC$ is an associative algebra. In particular, the only contributions to $\delta_{n+1}^1$ will be from tree pairs of the form $\Psi_{n+1}\otimes T_0$, where $T_0$ is a degree 0 tree. By Lemma~\ref{lem:diagonal-asymmetric-terms} there are an odd number of such summands of $\Gamma_{n+1}(\Psi_{n+1})$. It follows that the external tensor product is exactly the extension of scalars construction we gave earlier.
   \end{proof}

\subsection{Hypercubes as type-$D$ modules}
\label{sec:hypercube-algebra}
In this section, we present a helpful perspective of hypercubes and hyperboxes in terms of type-$D$ modules over certain algebras which we call the \emph{cube} and \emph{box} algebras. Many natural hyperbox operations, such as compression, have natural interpretations as $A_\infty$-module operations with respect to these algebras. 

 We will not use the cube and box algebras explicitly in this paper, but they are related to our reinterpretation of the Manolescu and Ozsv\'{a}th's link surgery formula in terms of type-$D$ modules over the algebra $\cL$.

We begin with the 1-dimensional cube algebra $\cC_1$. The algebra $\cC_1$ is an algebra over the idempotent ring $\ve{I}_0\oplus \ve{I}_1$, where $\ve{I}_\veps\iso \bF$. If $\veps\in \{0,1\}$, we write $i_{\veps}$ for the generator of $\ve{I}_{\veps}$. We set $ \ve{I}_{\veps}\cdot \cC_1\cdot \ve{I}_{\veps}=\ve{I}_{\veps}$, and we set 
\[
\ve{I}_{1}\cdot \cC_{1}\cdot \ve{I}_{0}=\langle \theta_{0,1} \rangle,
\]
i.e. we set $\ve{I}_{1}\cdot \cC_{1}\cdot \ve{I}_{0}\iso \bF$, generated by an element $\theta_{0,1}$.

Next, we introduce the \emph{$n$-dimensional cube algebra} $\cC_n$. The definition is 
\[
\cC_n=\cC_1\otimes_{\bF}\cdots \otimes_{\bF} \cC_1=\otimes_{\bF}^n \cC_1.
\]
Note that this is an algebra over the idempotent ring $\ve{E}_n=\otimes^{n}_{\bF} \ve{I}$.  If $\veps\in \{0,1\}^n$, we will write $\ve{E}_{\veps}$ for $\ve{I}_{\veps_1}\otimes \cdots \otimes \ve{I}_{\veps_n}$. We will write $i_{\veps}$ for the generator of $\ve{E}_{\veps}\iso \bF$.

The importance of the cube algebra is the following:

\begin{lem}
\label{lem:cube-algere=hypercubes}
 The category of $n$-dimensional hypercubes of chain complexes is equivalent to the category of type-$D$ modules over $\cC_n$.
\end{lem}
\begin{proof} We describe how to obtain a type-$D$ module over $\cC_n$ from a hypercube of chain complexes. Suppose that $X=(X_\veps,D_{\veps,\veps'})_{\veps\in \bE_n}$ is a hypercube of chain complexes. We define a type-$D$ module $\cX^{\cC_n}$ as follows. The underlying group of $\cX$ is the total space of $X$, i.e. $\cX=\bigoplus_{\veps\in \bE_n} X_{\veps}$. This is naturally a module over the idempotent ring $\ve{E}_n$.

 We define $\delta^1\colon \cX\to \cX\otimes_{\rmE_n} \cC_n$ as follows. If $\veps\le \veps'$, there is an algebra element $\theta_{\veps,\veps'}$ which is the tensor product of $\theta_{0,1}$ over all $i$ such that $\veps_i<\veps'_i$. If $\veps=\veps'$, we set $\theta_{\veps,\veps}=i_{\veps}$.
 With this notation, we set
 \[
 \delta^1(\xs)=\sum_{\veps\le \veps'} D_{\veps,\veps'}(\xs)\otimes \theta_{\veps,\veps'}.
 \]
The type-$D$ structure relations for $(\cX,\delta^1)$ are straightforward to verify.

If $X=(X_{\veps},D_{\veps,\veps'})$ and $Y=(Y_{\veps}, D'_{\veps',\veps'})$ are two $n$-dimensional hypercubes of chain complexes, it is straightforward to extend the above construction to give an isomorphism between morphism sets
\[
F\colon \Mor_{\bE_n}(X,Y)\to \Mor^{\cC_n}\left(\cX^{\cC_n},\cY^{\cC_n}\right),
\]
where $\Mor_{\bE_n}$ denotes the group of hypercube morphisms and $\Mor^{\cC_n}$ denotes the group of type-$D$ morphisms. Clearly $F$ is an isomorphism and $F(f\circ g)=F(f)\circ F(g)$. Hence, we have constructed a functor from the category of $n$-dimensional hypercubes to the category of type-$D$ modules over $\cC_n$. An inverse functor is straightforward to describe. 
\end{proof}

The above construction extends naturally to the setting of hyperboxes, as we now describe. We begin with the 1-dimensional box algebra $\cB_{(d)}$ of size $d>0$. This is an algebra over the idempotent ring 
\[
\ve{I}_{(d)}:=\ve{I}_0\oplus \ve{I}_1\oplus \cdots \oplus \ve{I}_d,
\]
where each $\ve{I}_{i}\iso \bF$. We set
\[
\ve{I}_{i}\cdot \cB_{(d)}\cdot \ve{I}_{i}=\bF.
\]
We set
\[
\ve{I}_{j}\cdot \cB_{(d)}\cdot \ve{I}_i=\begin{cases}\langle \theta_{i,i+1} \rangle& \text{ if } j=i+1\\
\{0\} & \text{ otherwise}.
\end{cases}
\]
Multiplication is given by setting $\theta_{j,j+1}\cdot \theta_{i,i+1}=0$ for all $i$ and $j$.

If $\ve{d}=(d_1,\dots, d_n)$ is a tuple of positive integers, we define the \emph{hyperbox algebra} of size $\ve{d}$, denoted $\cB_{\ve{d}}$ to be the tensor product
\[
\cB_{\ve{d}}:=\cB_{(d_1)}\otimes_{\bF}\cdots \otimes_{\bF} \cB_{(d_n)}.
\]
We have the following straightforward extension of Lemma~\ref{lem:cube-algere=hypercubes}:

\begin{lem} 
The category of hyperboxes of size $\ve{d}$ is equivalent to the category of type-$D$ modules over the algebra $\cB_{\ve{d}}$. 
\end{lem}

Manolescu and Ozsv\'{a}th's compression operation also has a conceptually simple description in terms of the cube and box algebras. With the present perspective, compression takes the form of a $DA$ bimodule
\[
{}_{\cB_{(d)}}[\frC_{(d)}]^{\cC_1},
\]
defined as follows.
 
 If $d=1$, we define $\frC_{(d)}$ to be the identity bimodule. Henceforth, assume $d>1$.  As a group $\frC_{(d)}$ is a copy of $\ve{I}=\ve{I}_0\oplus \ve{I}_1$. The right action of $\ve{I}$ is the obvious one. The left action of the box idempotent ring $\ve{I}_{(d)}$ is given by having $\ve{I}_{(0)}$ and $\ve{I}_{(d)}$ act by $\ve{I}_0$ and $\ve{I}_1$, respectively, and having $\ve{I}_{j}$ act trivially for $j=1,\dots, d-1$. 

The bimodule $[\frC_{(d)}]$ has $\delta_j^1=0$ for all $j$ except $\delta_{2}^1$ and $\delta_{d+1}^1$. We set $\delta_2^1$ to be the unique map satisfying $\delta_2^1(i_{\veps}, i_\nu)=(i_\veps\cdot i_\nu)\otimes 1$, where $\veps\in \{0,\dots, d\}$ and $\nu\in \{0,1\}$, and such that $\delta_2^1(\theta_{i,i+1},i_\nu)=0$ for all $i$. We define
\[
\delta_{d+1}^1(\theta_{d-1,d},\dots, \theta_{0,1}, i_0)=i_1\otimes \theta.
\]

\begin{example} Consider a hyperbox of dimension 1:
\[
X=\left(\begin{tikzcd}
C_0
	\ar[r, "f_{1,0}"]
&
C_1
\ar[r, "f_{2,1}"]
&
\cdots
\ar[r, "f_{d-1,d}"]
&
C_d
\end{tikzcd}\right)
\]
The associated type-$D$ module over $\cB_{(d)}$ takes the following form:
\[
\cX^{\cB_{(d)}}=\left(\begin{tikzcd}[column sep=2cm]
C_0
	\ar[r, "f_{0,1}\otimes \theta_{0,1}"]
&
C_1
\ar[r, "f_{1,2}\otimes \theta_{1,2}"]
&
\cdots
\ar[r, "f_{d-1,d}\otimes \theta_{d-1,d}"]
&
C_d
\end{tikzcd}
\right)
\]
The underlying group of the box tensor product $\cX^{\cB_{(d)}} \boxtimes {}_{\cB_{(d)}} [\frC_{(d)}]^{\cC_1}$ is $C_0\oplus C_d$. One easily computes that the differential on this box tensor product is
\[
\cX^{\cB_{(d)}} \boxtimes {}_{\cB_{(d)}} [\frC_{(d)}]^{\cC_1}=\left(\begin{tikzcd}[column sep=3cm] C_0 \ar[r, "(f_{d-1,d}\circ \cdots \circ f_{0,1} )\otimes \theta"] &C_{d}
\end{tikzcd}\right).
\]
The right hand side is the compression of $X$, viewed as a type-$D$ module over $\cC_1$, because it coincides with the function composition description of compression that is discussed in Section~\ref{sec:compression}
\end{example}

More generally, if $\ve{d}=(d_1,\dots, d_n)$ is an $n$-tuple of positive integers, we can define a bimodule ${}_{\cB_{\ve{d}}}[\frC_{\ve{d}}]^{\cC_n}$ as the external tensor product of the $DA$-bimodules ${}_{\cB_{(d_i)}} [\frC_{(d_i)}]^{\cC_1}$ ranging over $i=1,\dots, n$. Note that defining the external tensor product requires a choice of ordering of the axis directions, cf. Section~\ref{sec:extension-of-scalars}.

\begin{lem} If $X$ is a hyperbox of size $\ve{d}=(d_1,\dots, d_n)$, and $\cX^{\cB_{\ve{d}}}$ is the corresponding type-$D$ module, then the compression of $X$ coincides with the hypercube corresponding to $\cX^{\cB_{\ve{d}}}\boxtimes {}_{\cB_{\ve{d}}} [\frC_{\ve{d}}]^{\cC_{n}}$. 
\end{lem}

\begin{proof} The bimodule ${}_{\cB_{\ve{d}}}[\frC_{\ve{d}}]^{\cC_n}$ requires a choice of module diagonal to construct. Alternatively, by the proof of Proposition~\ref{prop:commute-extension-scalars}, we can construct ${}_{\cB_{\ve{d}}} [\frC_{{\ve{d}}}]^{\cC_n}$  using the extension of scalars construction. Write $\ve{d}_i=(1,\dots, 1,d_i,\dots, d_n)$. We apply the extension of scalars construction to ${}_{\cB_{d_i}}[\frC_{(d_i)}]^{\cC_1}$ to obtain a bimodule
\[
{}_{\cB_{\ve{d}_i}}[\frC_{\ve{d}_i, i}]^{\cB_{\ve{d}_{i+1}}}.
\]
 Using Proposition~\ref{prop:commute-extension-scalars}, we see that
\begin{equation}
{}_{\cB_{\ve{d}}}[\frC_{\ve{d}}]^{\cC_n}\simeq {}_{\cB_{\ve{d}}}[\frC_{\ve{d}_1, 1}]^{\cB_{\ve{d}_{2}}}\boxtimes \cdots \boxtimes 
{}_{\cB_{\ve{d}_n}}[\frC_{\ve{d}_n, n}]^{\cC_n}
\label{eq:decompose-compression-bimodule}
\end{equation}
Tensoring with ${}_{\cB_{\ve{d}_i}}[\frC_{\ve{d}_i,i}]^{\cB_{\ve{d}_{i+1}}}$ has the same effect as applying the partial compression operation $\frc_i$ from Section~\ref{sec:compression}, so the main claim follows from Equation~\eqref{eq:decompose-compression-bimodule}.
\end{proof}

\section{Homological perturbation theory}

In this section, we describe several versions of the homological perturbation lemma. In this section and throughout the rest of the paper, we assume that every ring and module has characteristic 2.

\subsection{\texorpdfstring{$A_\infty$}{A-infinity}-modules}

 Kadeishvili proved that $A_\infty$-algebra structures could be transferred along homotopy equivalences of the underlying chain complex, when coefficients were in a field \cite{Kadeishvili_Ainfinity}. A formula for the perturbed $A_\infty$-module structure in terms of trees may be found in work of Kontsevich and Soibelman \cite{KontsevichSoibelman}*{Section~6.4}. See \cite{KellerNotes}*{Section~3.3} or \cite{SeidelFukaya}*{Remark~1.15} for more historical context and development of ideas. The following is a statement of the homological perturbation lemma for $A_\infty$-modules.

\begin{lem}\label{lem:homological-perturbation-modules}
Suppose that $\cA$ is an associative algebra over a ring $\ve{k}$, 
${}_{\cA} M$ is an $A_\infty$-module,  $(Z,m_1^Z)$ is a chain complex over $\ve{k}$, and that we have three maps of left $\ve{k}$-modules
\[
i\colon Z\to M,\quad \pi\colon M\to Z\quad \text{and} \quad h\colon M\to M
\]
satisfying the following:
\begin{enumerate}
\item $i$ and $\pi$ are chain maps.
\item $\pi\circ i=\id_Z$.
\item $i\circ \pi=\id_M+\d_{\Mor}(h).$
\item $h\circ i=0$.
\item $\pi\circ h=0$.
\item $h\circ h=0$.
\end{enumerate}
In the above, $\d_{\Mor}(h)$ denotes the morphism differential for chain complexes, i.e. $\d_{\Mor}(h)=h\circ m_1^M+m_1^M\circ h$. Then there is an $A_\infty$-module structure on $Z$, extending $m_1^Z$, such that the maps $\pi$, $i$ and $h$ extend to $A_\infty$-module morphisms which satisfy all of the above relations as $A_\infty$-module maps.
\end{lem}

One extremely useful aspect of the homological perturbation lemma is that all of the maps have a concrete description. The structure maps on $Z$ are given by the diagrams shown in Figure~\ref{fig:homological-perturbation}. Therein $m_{>1}$ denotes the $A_\infty$ structure maps of $M$, with $m_1$ excluded.
\begin{figure}[ht]
\[
m^Z(\ve{a},\xs)=\begin{tikzcd}[column sep=.1cm,row sep=.4cm]
\ve{a}\ar[d, Rightarrow]& \ve{x} \ar[d]\\
 \Delta\ar[dr,bend right=10, Rightarrow]\ar[dd,Rightarrow]&i\ar[d]\\
\,& m_{>1} \ar[d]\\
\Delta\ar[dr,bend right=10, Rightarrow]\ar[dd,Rightarrow]&h\ar[d]\\
\,&m_{>1}\ar[d]\\
\vdots \ar[ddr,bend right=20,Rightarrow] & \vdots \ar[d]\\
\,&h\ar[d]\\
\,&m_{>1}\ar[d]\\
\,&\pi\ar[d] \\
\, &\, 
\end{tikzcd}
\quad 
i(\ve{a}, \xs)=\begin{tikzcd}[column sep=.1cm,row sep=.4cm]
\ve{a}\ar[d, Rightarrow]& \ve{x} \ar[d]\\
 \Delta\ar[dr,bend right=10, Rightarrow]\ar[dd,Rightarrow]&i\ar[d]\\
\,& m_{>1} \ar[d]\\
\Delta\ar[dr,bend right=10, Rightarrow]\ar[dd,Rightarrow]&h\ar[d]\\
\,&m_{>1}\ar[d]\\
\vdots \ar[ddr,bend right=20,Rightarrow] & \vdots\ar[d]\\
\,&h\ar[d]\\
\,&m_{>1}\ar[d]\\
\,&h\ar[d] \\
\, &\, 
\end{tikzcd}
\quad
\pi(\ve{a}, \xs)=\begin{tikzcd}[column sep=.1cm,row sep=.4cm]
\ve{a}\ar[d, Rightarrow]& \ve{x} \ar[d]\\
 \Delta\ar[dr,bend right=10, Rightarrow]\ar[dd,Rightarrow]&h\ar[d]\\
\,& m_{>1} \ar[d]\\
\Delta\ar[dr,bend right=10, Rightarrow]\ar[dd,Rightarrow]&h\ar[d]\\
\,&m_{>1}\ar[d]\\
\vdots \ar[ddr,bend right=20,Rightarrow] & \vdots\ar[d]\\
\,&h\ar[d]\\
\,&m_{>1}\ar[d]\\
\,&\pi\ar[d] \\
\, &\, 
\end{tikzcd}
\quad h(\ve{a},\xs)=
\begin{tikzcd}[column sep=.1cm,row sep=.4cm]
\ve{a}\ar[d, Rightarrow]& \ve{x} \ar[d]\\
 \Delta\ar[dr,bend right=10, Rightarrow]\ar[dd,Rightarrow]&h\ar[d]\\
\,& m_{>1} \ar[d]\\
\Delta\ar[dr,bend right=10, Rightarrow]\ar[dd,Rightarrow]&h\ar[d]\\
\,&m_{>1}\ar[d]\\
\vdots \ar[ddr,bend right=20,Rightarrow] & \vdots\ar[d]\\
\,&h\ar[d]\\
\,&m_{>1}\ar[d]\\
\,&h\ar[d] \\
\, &\, 
\end{tikzcd}
\]
\caption{The maps appearing in the homological perturbation lemma for $A_\infty$-modules.}
\label{fig:homological-perturbation}
\end{figure}

\subsection{$DA$-bimodules}

There is a homological perturbation lemma for $DA$-bimodules, which we will use extensively. The statement and proof are similar to the statement for $A_\infty$-modules. If $({}_{\cA} M^{\cB},\delta_j^1)$ is a $DA$-bimodule, we will write $M^{\cB}$ for the associated type-$D$ module which has the same generators as ${}_{\cA} M^{\cB}$, and whose structure map is given by the map $\delta_1^1$ from ${}_{\cA} M^{\cB}$.

\begin{lem}\label{lem:homological-perturbation-DA-modules} Suppose that $\cA$ and $\cB$ are associative algebras, and that ${}_{\cA} M^{\cB}$ is a type-$DA$ bimodule,  $Z^{\cB}$ is a type-$D$ module over $\cB$, and that we have three morphisms of type-$D$ modules
\[
i^1\colon Z^{\cB}\to M^{\cB},\quad \pi^1\colon M^{\cB}\to Z^{\cB}\quad \text{and} \quad h^1\colon M^{\cB}\to M^{\cB}
\]
satisfying the following:
\begin{enumerate}
\item $i^1$ and $\pi^1$ are homomorphisms of type-$D$ modules (i.e. $\d_{\Mor}(i^1)$ and $\d_{\Mor}(\pi^1)$ vanish). 
\item $\pi^1\circ i^1=\id_Z$.
\item $i^1\circ \pi^1=\id_M+\d_{\Mor}(h^1).$
\item $h^1\circ i^1=0$.
\item $\pi^1\circ h^1=0$.
\item $h^1\circ h^1=0$.
\end{enumerate}
In the above, $\d_{\Mor}(h^1)$ denotes the morphism differential for type-$D$ modules over $\cB$; see Equation~\eqref{eq:morphism-differential-type-D}. Then there is an induced $DA$-bimodule structure on ${}_{\cA} Z^{\cB}$, extending $Z^{\cB}$, such that the maps $\pi^1$, $i^1$ and $h^1$ extend to $DA$-bimodule morphisms which satisfy all of the above relations as $DA$-bimodule morphisms.
\end{lem}

We refer the reader to \cite{OSBorderedKauffmanStates}*{Lemma~2.12} for a proof. Note that the maps are similar to those appearing in Figure~\ref{fig:homological-perturbation}. The structure map $\delta_{j+1}^1$ is indicated in Figure~\ref{fig:structure-map-homological-perturbation-bimodules}.
\begin{figure}
\begin{equation}
\delta_{j+1}^1(\ve{a},\xs)=\begin{tikzcd}[row sep=.4cm]
\ve{a}\ar[d, Rightarrow]& \ve{x} \ar[d]\\
 \Delta\ar[dr,bend right=10, Rightarrow]\ar[dd,Rightarrow]&i^1\ar[d] \ar[ddddddddr,bend left=15]\\
\,& \delta_{>1}^1 \ar[d] \ar[dddddddr,bend left=13]\\
\Delta\ar[dr,bend right=10, Rightarrow]\ar[dd,Rightarrow]&h^1\ar[d]\ar[ddddddr,bend left=10]\\
\,&\delta^1_{>1}\ar[d]\\
\vdots \ar[ddr,bend right=20,Rightarrow] & \vdots \ar[d] \ar[ddddr,Rightarrow,bend left=9]\\
\,&h^1\ar[d] \ar[dddr,bend left=8]\\
\,&\delta^1_{>1}\ar[d] \ar[ddr,bend left=7]\\
\,&\pi^1\ar[dd] \ar[dr,bend left=6] \\
\, &\,& \Pi \ar[d]\\
\,&\,& \,
\end{tikzcd}
\end{equation}
\caption{The structure relation $\delta_{j+1}^1$ from the homological perturbation lemma for $DA$-bimodules.}
\label{fig:structure-map-homological-perturbation-bimodules}
\end{figure}
The other maps $i_j^1$, $\pi_j^1$ and $h_j^1$ are constructed by a similar modification of the other maps in Figure~\ref{fig:homological-perturbation}. Here, $\delta_{>1}^1$ denotes $\delta_{j}^1$ for some $j>1$. The final $\Pi$ map means to multiply all of the inputs using several applications of $\mu_2$.

\subsection{Hypercubes}

In this section, we review a version of the homological perturbation lemma for hypercubes. See  \cite{HHSZDuals}*{Section~2.7} for more details. Lemma~\ref{lem:homological-perturbation-cubes} is similar to other formulations of the homological perturbation lemma in the context of filtered chain complexes.  See, e.g.,  \cite{HK_Homological_Perturbation}.   A similar, though slightly less explicit, construction for transferring hypercube structures along homotopy equivalences is described by Liu \cite{Liu2Bridge}*{Section~5.6}.

\begin{lem}\label{lem:homological-perturbation-cubes}
Suppose that $\cC=(C_\veps,D_{\veps,\veps'})_{\veps\in \bE_n}$ is a hypercube of chain complexes, and $(Z_\veps)_{\veps\in \bE_n}$ is a collection of chain complexes, indexed by $\veps\in \bE_n$. Furthermore, suppose that for each $\veps\in \bE_n$, we have chosen maps 
\[
\pi_{\veps}\colon C_\veps\to Z_\veps\qquad i_{\veps}\colon Z_\veps\to C_\veps \qquad h_\veps\colon C_\veps\to C_\veps,
\]
satisfying
\[
\pi_{\veps}\circ i_{\veps}=\id, \quad i_{\veps}\circ \pi_{\veps}=\id+\d_{\Mor}(h_\veps), \quad  h_\veps\circ h_\veps=0, \quad \pi_{\veps}\circ h_{\veps}=0\quad \text{and} \quad h_{\veps}\circ i_{\veps}=0
\]
and such that $\pi_{\veps}$ and $i_{\veps}$ are chain maps.
With the above data chosen, there are canonical hypercube structure maps $\delta_{\veps,\veps'}\colon Z_{\veps}\to Z_{\veps'}$ so that $\cZ=(Z_\veps,\delta_{\veps,\veps'})$ is a hypercube of chain complexes, and also there are morphisms of hypercubes
\[
\Pi\colon \cC\to \cZ, \quad I\colon  \cZ\to \cC\quad \text{and}\quad H\colon \cC\to \cC
\]
such that $I$ and $\Pi$ are chain maps,  and satisfy
\[
\Pi\circ I=\id, \quad I\circ \Pi=\id+\d_{\Mor}(H), \quad  H\circ H=0, \quad \Pi\circ H=0\quad \text{and} \quad H\circ I=0.
\]
\end{lem}

It is important to understand that the structure maps $\delta_{\veps,\veps'}$ and the morphisms $\Pi$ and $I$ are determined completely by an explicit formula. We now describe $\delta_{\veps,\veps'}$. Suppose that $\veps<\veps'$ are points in $\bE_n$.   The hypercube structure maps $\delta_{\veps,\veps'}$ are given by the following sum:
\[
\delta_{\veps,\veps'}:=\sum_{\veps=\veps_1<\cdots<\veps_j=\veps'}  \pi_{\veps'}\circ D_{\veps_{j-1},\veps_j}\circ h_{\veps_{j-1}}\circ D_{\veps_{j-2},\veps_{j-1}}\circ \cdots \circ h_{\veps_2}\circ D_{\veps_1,\veps_2}\circ i_{\veps}.
\]
The component of the map $I$ sending coordinate $\veps$ to coordinate $\veps'$ is given via the formula
\[
I_{\veps,\veps'}:=\sum_{\veps=\veps_1<\cdots<\veps_j=\veps'}  h_{\veps_{j}}\circ D_{\veps_{j-1},\veps_{j}}\circ \cdots \circ h_{\veps_2}\circ D_{\veps_1,\veps_2}\circ i_{\veps}.
\]
Similarly, $\Pi$ is given by the formula
\[
\Pi_{\veps,\veps'}:=\sum_{\veps=\veps_1<\cdots<\veps_j=\veps'}  \pi_{\veps'}\circ D_{\veps_{j-1},\veps_j}\circ h_{\veps_{j-1}}\circ D_{\veps_{j-2},\veps_{j-1}}\circ \cdots\circ D_{\veps_1,\veps_2} \circ h_{\veps_1}.
\]

\subsection{Hypercubes of $DA$-bimodules}

There is a version of the homological perturbation lemma which naturally generalizes both the homological perturbation lemmas for $DA$-bimodules and hypercubes of chain complexes.

\begin{define} A \emph{hypercube of $DA$-bimodules} consists of a $DA$-bimodule $(C,\delta_*^1)$, such that $C$ decomposes as $C=\bigoplus_{\veps\in \bE_n} C_{\veps}$, where each $C_{\veps}$ is itself a $(\ve{j},\ve{k})$-module. Furthermore, the structure map $\delta_{j+1}^1$ decomposes as a sum of maps 
\[
\delta_{j+1,\veps,\veps'}^1\colon \cA^{\otimes j}\otimes C_{\veps}\to C_{\veps'}\otimes \cB.
\]
 where $j\ge 0$, and $\veps\le \veps'$.
 \end{define}
 
  The above definition is a special case of Lipshitz, Ozsv\'{a}th and Thurston's notion of a \emph{filtered $DA$-bimodule} \cite{LOTDoubleBranchedI}*{Section~2.2}.

\begin{lem}
\label{lem:homological-perturbation-DA-hypercube}
Suppose that $\cA$ and $\cB$ are $dg$-algebras.  Suppose also that ${}_{\cA}\cC^{\cB}=({}_{\cA} C_\veps^{\cB},D_{*,\veps,\veps'}^1)_{\veps\in \bE_n}$ is a hypercube of $DA$-bimodules, and suppose that for each $\veps\in \bE_n$ we have chosen type-$D$ modules $(Z_{\veps}^{\cB})_{\veps\in \bE_n}$, as well as morphisms of type-$D$ modules
\[
\pi^1_{\veps}\colon C_{\veps}^{\cB}\to Z_{\veps}^{\cB},\quad i_{\veps}^1\colon Z_{\veps}^{\cB}\to C_{\veps}^{\cB}\quad \text{and}\quad h_{\veps}^1\colon C_{\veps}^{\cB}\to C_{\veps}^{\cB}
\]
satisfying 
\[
\pi_{\veps}^1\circ i_{\veps}^1=\id_{Z_\veps},\quad i_{\veps}^1\circ \pi_{\veps}^1+\id_{C_\veps}=\d_{\Mor}(h^1_\veps),\quad h_{\veps}^1\circ h_{\veps}^1=0, \quad h^1_{\veps}\circ i^1_{\veps}=0,\quad \text{and}\quad \pi^1_{\veps}\circ h^1_{\veps}=0,
\]
and such that $\pi_{\veps}^1$ and $i_{\veps}^1$ are homomorphisms of type-$D$ modules (i.e. cycles). Such data determines structure maps $\delta_{j+1,\veps,\veps'}^1\colon \cA^{\otimes j}\otimes Z_{\veps}\to Z_{\veps'}\otimes \cB$, which make ${}_{\cA}\cZ^{\cB}=({}_{\cA}Z_{\veps}^{\cB}, \delta^1_{*,\veps,\veps'})$ a hypercube of $DA$-bimodules. Furthermore, there are morphisms of hypercubes of $DA$-bimodules
\[
\Pi_*^1\colon {}_{\cA}\cC^{\cB}\to {}_{\cA}\cZ^{\cB}\quad I_*^1\colon {}_{\cA}\cZ^{\cB}\to {}_{\cA}\cC^{\cB} \quad \text{and} \quad H_*^1\colon {}_{\cA} \cC^{\cB}\to {}_{\cA} \cC^{\cB},
\]
which satisfy analogous relations to $i_\veps^1$, $\pi_{\veps}^1$ and $h_{\veps}^1$.
\end{lem}

We now describe the maps $\delta_{*}^1$, $\Pi_{*}^1$, $I_*^1$ and $H_*^1$. We focus on the structure map $\delta_{j+1,\veps,\veps'}^1$. This map is defined by modifying Figure~\ref{fig:structure-map-homological-perturbation-bimodules} as follows. In place of each $\delta_{>1}^1$ labeled therein, we are allowed one of two maps:
\begin{enumerate}
\item A hypercube map $D_{j+1,\veps,\veps'}^1$ for $\cC$, such that $\veps<\veps'$. Furthermore, we assume only that $j\ge 0$ (so that we allow the case that there are no algebra inputs here).
\item An internal structure map $D_{>1,\veps,\veps}^1$ of some ${}_{\cA} C_\veps^{\cB}$ (so in this case we do not allow instances of the internal differential $D_{1,\veps,\veps}^1$).
\end{enumerate}

The maps $\Pi_*^1$, $I_*^1$ and $H_*^1$ are defined by small modifications, similar to Figure~\ref{fig:homological-perturbation}. We leave the proof of the above lemma to the reader, as it is a straightforward variation of Lemma~\ref{lem:homological-perturbation-DA-modules}.

\section{Linear topological spaces}
\label{sec:completions}

In this section, we recall basic background on algebraic completions. See \cite{AtiyahMacdonald}*{Section~10} and \cite{LefschetzAT}*{Chapter~2} for related background. This material is used to define our module categories in Section~\ref{sec:Algebra-K}.

\begin{define} A \emph{linear topological vector space} consists of a vector space $\cX$ equipped with a topology which admits a basis of open sets centered at 0 consisting of vector subspaces, such that furthermore the addition map $\cX\times \cX \to \cX$ is continuous.
\end{define}

If $R$ is a ring, an \emph{$R$-module with linear topology} is defined similarly, except that there is a basis of 0 consisting of $R$-submodules.

Vector spaces with linear topologies are usually specified by picking a decreasing filtration consisting of subspaces $X_i\subset \cX$ indexed by $i\in I$, where $I$ is some directed and partially ordered set. Such a filtration determines a topology, where we declare a basis to be the sets of the form $x+X_i$, where $i\in I$ and $x\in \cX$.

Given a filtration $(X_{i})_{i\in I}$ on a vector space $\cX$, the \emph{completion} is defined as the inverse limit
\[
\ve{\cX}:=\varprojlim_{i\in I} \cX/X_i.
\]

Rather than working with inverse limits, it is sometimes easier to work with Cauchy sequences and nets. Since some of the spaces we are working with are not first countable, we work with Cauchy \emph{nets} instead of the perhaps more familiar Cauchy \emph{sequences}.

We recall that a \emph{net} in a topological space $\cX$ is a map $\ve{x}\colon I\to \cX$, where $I$ is a directed set. We usually write $\ve{x}=(x_i)_{i\in I}$ for a net. When $\cX$ is a linear topological vector space, we say that a net $(x_i)_{i\in I}$ is \emph{Cauchy} if for each open subspace $E\subset \cX$, there is an $i_0\in I$ such that if $i,j\ge i_0$, then $x_i-x_j\in E.$

If $(x_i)_{i\in I}$ and $(y_j)_{j\in J}$ are two nets in $\cX$, their sum is the net $(x_i+y_j)_{(i,j)\in I\times J}$. Two Cauchy nets are \emph{equivalent} if their difference converges to $0\in \cX$ as a net. We note that the collection of nets in $\cX$ forms a proper class, though the set of Cauchy equivalence classes is a set. The following is well known and is elementary to prove:

\begin{lem} If $\cX$ is a linear topological space, then the completion $\ve{\cX}$ is canonically isomorphic to the vector space of equivalence classes of Cauchy nets in $\cX$.
\end{lem}

\subsection{Linear maps}
\label{sec:Cauchy}

If $\cX$ and $\cY$ are linear topological vector spaces, then a continuous map $f\colon \cX\to \cY$ naturally induces a map on completions. We recall the following  basic principle, whose proof is elementary:

\begin{lem}
\label{lem:continuous-iff-at-0} Suppose that $\cX$ and $\cY$ are linear topological spaces. A linear map $f\colon \cX\to \cY$ is continuous if and only if it is continuous at $0$.
\end{lem}

\begin{define}\label{def:linear-topological-category}
 We define the category of \emph{linear topological vector spaces} over $\bF$, denoted $\LTS_\bF$, as follows. The objects are linear topological vector spaces over $\bF$. We define the set of morphisms from $\cX$ to $\cY$ in this category to be the set of continuous linear maps from $\ve{\cX}$ to $\ve{\cY}$ (i.e. continuous linear maps between the completed spaces).
\end{define}

\begin{rem}\label{rem:canonically-isomorphic-X-veX} By definition, $\cX$ and $\ve{\cX}$ are isomorphic in the category $\LTS_{\bF}$ for any linear topological $\bF$-vector space $\cX$. 
\end{rem}

\subsection{Direct products}

We now describe some helpful perspectives on the direct sum and product in the context of linear topological spaces.  These are important for our work because of the role that direct products play in the surgery formula from \cite{MOIntegerSurgery}.

\begin{define}\label{def:cofinite-basis}
 If $(X_i)_{i\in I}$ is a family of linear topological spaces, the \emph{direct product} of topological vector spaces $\cX=\prod_{i\in I} X_i$ coincides with the direct product in the category of topological spaces. More explicitly, a subspace $E\subset \prod_{i\in I} X_i$ is open if and only if there is some finite set $S\subset I$ and a set of open subspaces $U_{s}\subset X_{s}$ ranging over $s\in S$, such that
 \[
  \left(\prod_{i\in I\setminus S} X_i \right)\oplus \left(\prod_{s\in S} U_s\right)\subset E.
 \]
 We declare the product topology on $\bigoplus_{i \in I} X_i$ to be the subspace topology induced by the inclusion $\bigoplus_{i\in I} X_i\subset \prod_{i\in I} X_i$.
\end{define}
Note that if $\cX=\bigoplus_{i\in I} X_i$ is equipped with the product topology, then $\ve{\cX}\iso \prod_{i\in I} \ve{X}_i$.

In our work, we consider frequently the case where each $X_i$ is equipped with the discrete topology. In this case, if $S\subset I$ is a finite set, we write 
\[
\cX_{\co(S)}=\bigoplus_{i\not \in S} X_i
\]
for the fundamental basis of opens.

\begin{example}
 We illustrate the case that $I=\Z$ and $X_i=\bF$. We consider $\cX=\bigoplus_{i\in I} X_i$. Write $x_i$ for the generator of $X_i$. The basis of open subspaces as $\cX_{\co(\{-n,\dots, n\})}=\Span(x_i: |i|>n)$. We can identify $\cX/\cX_{\co(\{-n,\dots, n\})}$ with $\Span(x_{-n},\dots, x_n)$. The inverse limit of $\cX/\cX_{\co(\{-n,\dots, n\})}$ consists of all tuples $(a_i)_{i\in \Z}$ where $a_i\in X_i$. This is the direct product $\prod_{i\in \Z} X_i$. 
\end{example}

The following example illustrates the meaning of a continuous map between direct products equipped with the product topologies:
\begin{example}
 Suppose that $F\colon \prod_{i\in I} X_i\to \prod_{j\in J} Y_j$ is a linear map and each $X_i$ and $Y_j$ is equipped with the discrete topology. Write $F_{j,i}$ for $\Pi_j\circ F\circ I_i$, where $\Pi_j$ and $I_i$ are projections and inclusions. Then $F$ is continuous with respect to the product topologies if and only if for each $j$, there are only finitely many $i$ such that $F_{j,i}\neq 0$, and furthermore $F=\sum_{j,i} F_{j,i}$.
\end{example}

\subsection{Linear compactness}

We now recall the notion of linear compactness. This definition is important for the module categories in Section~\ref{sec:Algebra-K} since some of the operations we define (such as various versions of the external tensor product) require the underlying spaces of our modules to be linearly compact. All of the link surgery modules we construct are linearly compact.

Many basic statements about compact sets from point-set topology have analogs for linearly compact vector spaces,  though linear compactness and point-set compactness are not equivalent in general.  There are several equivalent definitions of linear compactness, though the following is the most useful for our purposes:

\begin{define} A linear topological space $\cX$ is \emph{linearly compact} if each open subspace $U\subset \cX$ has finite codimension, i.e. $\dim(\cX/U)<\infty$. 
\end{define}

\begin{rem}
\begin{enumerate} 
\item The notion of linear compactness may be found in the classical work of Lefschetz \cite{LefschetzAT}*{Chapter~2}. Note that Lefschetz formulates linear compactness in terms of the finite intersection property for closed affine subspaces, though this turns out to be equivalent. Since we will not revisit the work of Lefschetz, we will not describe the equivalence.
\item The direct product of linearly compact spaces is linearly compact. In particular, $\prod_{\N} \bF$ is linearly compact.
\end{enumerate}
\end{rem}

\subsection{Tensor products}
\label{sec:tensor-prods}
Given linear topological vector spaces $\cX$ and $\cY$, there are many different ways to topologize the tensor product $\cX\otimes \cY$. This phenomenon goes back to work of Grothendieck \cite{GrothendieckNuclear} in the setting of functional analysis and Banach spaces. Despite fundamental differences between the Banach spaces and linear topological vector spaces, similar phenomena occur in both settings.

In this paper, we focus on three natural topologies on tensor products:
\begin{enumerate}
\item The \emph{standard topology} $\cX\otimes^! \cY$.
\item The \emph{$*$-topology} $\cX\otimes^* \cY$.
\item The \emph{chiral-topology} $\cX\vecotimes \cY$.
\end{enumerate}
The following inclusions are continuous
\[
\cX\otimesstar \cY\subset \cX\vecotimes \cY\subset \cX\otimesshriek \cY.
\]
 The topologies $\otimes^*$ and $\vecotimes$ are due to Beilinson \cite{Beilinson-Tensors}. Additional exposition and analysis of these constructions is given by Positelski \cite{Positelski-Linear}*{Sections~12 and 13}, wherein many basic properties, subtleties and counterexamples to natural conjectures are described. These topologies are defined as follows:

\begin{define} Suppose that $\cX$ and $\cY$ are linear topological spaces. 
\begin{enumerate}
\item A subspace $E\subset \cX\otimes^! \cY$ is open if and only if the following holds: There are open subspaces $U\subset \cX$ and $V\subset \cY$ such that 
$U\otimes \cY\subset E$ and $\cX\otimes V\subset E.$
\item A subspace $E\subset \cX\otimes^* \cY$ is open if and only if the following holds: There are open subspaces $U\subset \cX$ and $V\subset \cY$ such that 
$
U\otimes V\subset E;$ and for each $x\in \cX$ and $y\in \cY$, there are open subspaces $U_y\subset \cX$ and $V_x\subset \cY$ such that
$ x\otimes V_x\subset E$ and $U_y\otimes y\subset E.$
\item A subspace $E\subset \cX\vecotimes \cY$ is open if and only if the following holds: There is an open subspace $U\subset \cX$ so that 
$U\otimes \cY\subset E;$ and for each $x\in \cX$, there is an open $V_x\subset \cY$ such that $ x\otimes V_x\subset E.$ 
\end{enumerate}
\end{define}

The topologies $\otimesstar$ and $\otimesshriek $ are symmetric between $\cX$ and $\cY$, though $\vecotimes$ is asymmetric. The motivation for $\vecotimes$ is that it is a natural topology for an algebra whose topology is given by one-sided ideals. This perspective is of particular importance because our algebra $\cK$ is naturally topologized using a family of right ideals. See Lemma~\ref{lem:simple-properties-K-Rn}.

\begin{rem}
\label{rem:equivalent-bases} If $\cX$ and $\cY$ are equipped with filtrations $(X_i)_{i\in \N}$ and $(Y_i)_{i\in \N}$, such that $X_0=\cX$ and $Y_0=\cY$, then the topology on $\cX\otimes^! \cY$ is equivalent to the one induced by the filtration $(\cX\otimes \cY)_n=\sum_{i+j=n} X_i\otimes Y_j$. This is because
\[
X_{2n}\otimes \cY\oplus \cX\otimes Y_{2n}\subset \sum_{i+j=2n} X_i\otimes Y_j\subset X_n\otimes \cY\oplus\cX\otimes Y_n
\]
 This is not in general true if we work with filtrations over $\Z$ instead of $\N$, or if we relax the assumption that $X_0=\cX$ and $Y_0=\cY$.
\end{rem}

\subsection{Associativity and commutativity relations}

In this section, we recall some basic properties about the tensor products $\otimes^!$, $\vecotimes$ and $\otimes^*$. We begin with associativity:

\begin{lem} If $\circ$ is one of $\{*,!,\to\}$, then the canonical map
\[
\cX\otimes^{\circ} (\cY\otimes^{\circ} \cZ)\iso (\cX\otimes^{\circ} \cY)\otimes^{\circ} \cZ
\]
is a homeomorphism.
\end{lem}

The reader may consult \cite{Positelski-Linear}*{Proposition~13.4} for details. In contrast,  the different tensor product operations are not mutually associative. Nonetheless, there are some relations between different parenthesizations of the tensor products. One important and well-known example is the following (we give a proof for the benefit of the reader):

\begin{lem}
\label{lem:split-arrow-shriek}
Suppose that $\cX_1$, $\cX_2$, $\cY_1$, and $\cY_2$ are linear topological vector spaces.
 The natural map
\[
(\cX_1\otimesshriek \cY_1)\vecotimes (\cX_2\otimesshriek \cY_2)\to
 (\cX_1\vecotimes \cX_2)\otimesshriek (\cY_1\vecotimes \cY_2)
\]
is continuous.
\end{lem}
\begin{proof}
A subspace $E\subset (\cX_1\vecotimes \cX_2)\otimesshriek (\cY_1\vecotimes \cY_2)$ is open if
\begin{enumerate}[label=($s$)]
\item There are open $U_1\subset \cX_1$ and $V_1\subset \cY_1$, and for each $x_1\in \cX_1$ and $y_1\in \cY_1$ there are open $U_2^{x_1}\subset \cX_2$ and $V_2^{y_1}\subset \cY_2$ so that
$\cX_1\otimes \cX_2\otimes V_1\otimes \cY_2,$  $\cX_1\otimes \cX_2\otimes y_1\otimes V_2^{y_1}$, $U_1\otimes \cX_2\otimes \cY_1\otimes \cY_2$, and $x_1\otimes U_2^{x_1}\otimes \cY_1\otimes \cY_2$ are contained in  $E$.
\end{enumerate}
A subspace $E\subset (\cX_1\otimesshriek \cY_1)\vecotimes (\cX_2\otimesshriek \cY_2)$ is open if
\begin{enumerate}[label=($t$-\arabic*)]
\item There are open $U_1\subset \cX_1$ and $V_1\otimes \cY_1$ so that $U_1\otimes \cY_1\otimes \cX_2\otimes \cY_2$, and  $\cX_1\otimes V_1\otimes \cX_2\otimes \cY_2$ are contained in $E$.
\item For each $x_1\in \cX_1$ and $y_1\in \cY_1$, there is an open $U_2^{x_1,y_1}\subset \cX_2$ so that $x_1\otimes y_1\otimes U_2^{x_1,y_1}\otimes \cY_2\subset E$.
\item For each $x_1\in \cX_1$, $y_1\in \cY_1$, $x_2\in \cX_2$, there is an open $V_2^{x_1,y_1,x_2}\subset \cY_2$ so that $x_1\otimes y_1\otimes x_2\otimes  V_2^{x_1,y_1,x_2}\subset E.$
\end{enumerate}
Any subspace $E$ satisfying $(s)$ satisfies ($t$-1), ($t$-2) and ($t$-3), so the stated map is continuous.
\end{proof}

Applying Lemma~\ref{lem:split-arrow-shriek} repeatedly gives the following:

\begin{cor}
\label{cor:re-order-arrow-shriek}
 Suppose that $\cX_1,\dots, \cX_n$ and $\cY_1,\dots, \cY_n$ are linear topological vector spaces.
\begin{enumerate}
\item The map 
\[
(\cX_1\otimesshriek \cY_1)\vecotimes \cdots \vecotimes (\cX_n\otimesshriek \cY_n)\to
 (\cX_1\vecotimes\cdots\vecotimes \cX_n)\otimesshriek (\cY_1\vecotimes\cdots\vecotimes \cY_n)
 \]
 is continuous.
\item 
The map
\[
(\cX_1\otimes^!\cdots \otimes^! \cX_n)\vecotimes ( \cY_1\otimes^!\cdots \otimes^! \cY_n)\to (\cX_1\vecotimes \cY_1)\otimes^! \cdots \otimes^!(\cX_n\vecotimes \cY_n)
\]
is continuous. 
\end{enumerate}
\end{cor}

The following is  well-known:

\begin{lem}
\label{lem:lin-comact-discrete-change-completion}
Suppose that $\cX$ and $\cY$ are filtered spaces. If either
\begin{enumerate}
\item $\cX$ is linearly compact, or
\item $\cY$ is discrete,
\end{enumerate}
then $\cX\vecotimes \cY\iso \cX\otimesshriek \cY$.
\end{lem}
\begin{proof} Consider the case that $\cX$ is linearly compact and suppose that $E\subset \cX\vecotimes \cY$ is an open subspace. By assumption, there is some open subspace $U\subset \cX$ so that $U\otimes \cY\subset E$. Since $\cX$ is linearly compact, we may pick $x_1,\dots, x_n\in \cX$ spanning a complementary vector space to $U$. By assumption, there are open subspaces $V_1,\dots,V_n\subset \cY$ so that $x_i\otimes V_i\subset E$. Therefore
\[
U\otimes \cY+\cX\otimes ( V_1\cap\dots \cap V_n)\subset U\otimes \cY+x_1\otimes V_1+\cdots+ x_n\otimes V_n\subset E
\]
so the map $\cX\otimesshriek \cY\to \cX\vecotimes \cY$ is continuous. The map in the opposite direction is obviously continuous, so they are isomorphic.

The claim when $\cY$ is discrete is similar. 
\end{proof}

\begin{rem} If both $\cX$ and $\cY$ are linearly compact, or both are discrete, then all three tensor products $\cX\otimesshriek \cY$, $\cX\vecotimes \cY$ and $\cX\otimesstar \cY$ coincide. Similarly, if $\cX$ is linearly compact and discrete (i.e. finite dimensional), then all three tensor products coincide.
\end{rem}

\begin{rem} Even if $\cX$ and $\cY$ are first countable, it may not be the case that $\cX\vecotimes \cY$ is first countable. For example, if $\cX=\bigoplus_{i\in \Z} \bF$ is given the discrete topology and $\cY=\bF\llsquare U\rrsquare $ (equipped with the $U$-adic topology) then it is not hard to see that $\cX\vecotimes \cY$ is not first countable. See \cite{Positelski-Linear}*{Example~13.1 and Lemma~7.4}.
\end{rem}

Morphisms may also be tensored. We state the following well-known result:

\begin{lem}
\label{lem:tensor-morphisms}
All three tensor product operations $\otimes^!$, $\otimes^*$ and $\vecotimes$ are functorial with respect to morphisms, i.e., if $f$ and $g$ are continuous morphisms, then $f\otimes^\circ g$ is also continuous for $\circ\in \{!,*,\to\}$. 
\end{lem}

See  \cite{Positelski-Linear}*{Lemma~12.3} for a detailed exposition.

\subsection{The tree topology}
\label{sec:tree-completions}
 In this section, we introduce a topology on a tensor product of linear topological spaces indexed by vertices of certain directed trees.
   We use this notion when considering split Alexander modules in Section~\ref{sec:split-modules}.

 We say that a tree $\Gamma$ is \emph{strongly directed} if each edge is given an orientation, and each vertex has at most one outgoing edge.  We write $v>v'$ if there is an oriented sequence of edges from $v$ to $v'$. A strongly directed tree has a unique minimal vertex, which we refer to as the \emph{root}.
 
  If $(\cX_v)_{v\in V(\Gamma)}$ is a collection of linear topological spaces, we will describe a linear topology on the tensor product
\[
\cX_\Gamma:=\bigotimes_{v\in V(\Gamma)} \cX_v.
\]
The construction works more generally when the spaces $\cX_v$ are equipped with commuting actions of rings $R_e$ for each edge $e$ adjacent to $v$. (For our purposes, $R_e$ will be an idempotent ring of a link algebra).

We describe the topology after introducing some notation. We consider a set of symbols $\{p,\cO,X\}$. We call a function
\[
s\colon V(\Gamma)\to \{p, \cO, X\}
\]
 an \emph{admissible labeling} if it satisfies the following:
\begin{enumerate}[ label=($s$-\arabic*), ref=$s$-\arabic*]
\item $s$ is not constantly $p$. 
\item If $s(v)=X$, then $s(v')=X$ for all $v'< v$.
\item If $s(v)=p$, then $s(v')=p$ for all $v'>v$.
\item If $s(v)=\cO$, then $s(v')=p$ for all $v'>v$, and $s(v')=X$ for all $v'<v$.
\item If $s(v)=X$, then there is a vertex $v'$ with an edge pointing to $v$ such that $s(v')\in \{\cO,X\}$.
\end{enumerate}
See Figure~\ref{fig:33} for an example of an admissible labeling. Note that the last axiom prohibits any maximal leaf from being labeled by $X$. Furthermore, an admissible labeling must label at least one vertex with $\cO$. 

\begin{figure}[ht]
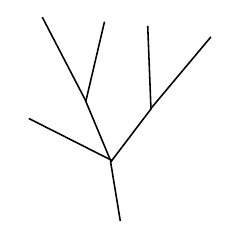
\caption{An admissible labeling of a strongly directed tree.}
\label{fig:33}
\end{figure}

Suppose that $s$ is an admissible labeling of $\Gamma$. We define a \emph{point-enhancement} $\ve{p}$ of $s$ to be a 
choice of element $\ve{p}(v)\in \cX_v$ for each $v\in s^{-1}(p)$. We say an \emph{open-enhancement} $\cU$ of $s$ is a choice of open subspaces $\cU(v)\subset \cX_v$ for each $v\in s^{-1}(\cO)$. Given point and open-enhancements, $\ve{p}$ and $\cU$, of $s$, there is a subspace $\cW(\cU,\ve{p})\subset \cX_{\Gamma}$, obtained by tensoring the spans of $\ve{p}(v)$ from $\ve{p}$, with the open subspaces $\cU(v)$, as well as $\cX_v$ ranging over $v\in s^{-1}(X)$.

We now define the \emph{tree topology} on $\cX_{\Gamma}$, when $\Gamma$ is a strongly directed tree:

\begin{define} Suppose that $\Gamma$ is a strongly directed tree. We say a subspace $E\subset \cX_{\Gamma}$ is open if and only if for each admissible labeling $s$ of $\Gamma$ and each point-enhancement $\ve{p}$ of $s$, there is a open-enhancement $\cU_{\ve{p}}$ of $s$ so that
\[
\cW(\cU_{\ve{p}},\ve{p})\subset E. 
\]
\end{define}

\begin{example} If $\Gamma$ is a linearly ordered graph with vertices $v_n>v_{n-1}>\cdots >v_1$, and $\cX_{v_i}$ are spaces, then 
\[
\cX_{\Gamma}=\cX_{v_n} \vecotimes \cdots \vecotimes \cX_{v_1}.
\]
\end{example}

\begin{rem}\label{rem:graphs-subtle} We warn the reader that there are some related (but subtly non-isomorphic) topologies that one can define on these tensor products. As an example, consider a graph $\Gamma$ with a central root vertex $v_0$, and two rays pointing into $v_0$, labeled by vertices $w_1,\dots, w_n$ and $u_1,\dots, u_m$.  We designate $v_0$ is minimal, and $w_n$ and $u_m$ are maximal. See below:

\begin{center}
\begingroup%
  \makeatletter%
  \providecommand\color[2][]{%
    \errmessage{(Inkscape) Color is used for the text in Inkscape, but the package 'color.sty' is not loaded}%
    \renewcommand\color[2][]{}%
  }%
  \providecommand\transparent[1]{%
    \errmessage{(Inkscape) Transparency is used (non-zero) for the text in Inkscape, but the package 'transparent.sty' is not loaded}%
    \renewcommand\transparent[1]{}%
  }%
  \providecommand\rotatebox[2]{#2}%
  \newcommand*\fsize{\dimexpr\f@size pt\relax}%
  \newcommand*\lineheight[1]{\fontsize{\fsize}{#1\fsize}\selectfont}%
  \ifx\svgwidth\undefined%
    \setlength{\unitlength}{140.47412999bp}%
    \ifx\svgscale\undefined%
      \relax%
    \else%
      \setlength{\unitlength}{\unitlength * \real{\svgscale}}%
    \fi%
  \else%
    \setlength{\unitlength}{\svgwidth}%
  \fi%
  \global\let\svgwidth\undefined%
  \global\let\svgscale\undefined%
  \makeatother%
  \begin{picture}(1,0.19999389)%
    \lineheight{1}%
    \setlength\tabcolsep{0pt}%
    \put(0.05559539,0.16914066){\makebox(0,0)[t]{\lineheight{1.25}\smash{\begin{tabular}[t]{c}$w_n$\end{tabular}}}}%
    \put(0.73558181,0.16919116){\makebox(0,0)[t]{\lineheight{1.25}\smash{\begin{tabular}[t]{c}$w_1$\end{tabular}}}}%
    \put(0.04785686,0.01517827){\makebox(0,0)[t]{\lineheight{1.25}\smash{\begin{tabular}[t]{c}$u_m$\end{tabular}}}}%
    \put(0.7355818,0.00567081){\makebox(0,0)[t]{\lineheight{1.25}\smash{\begin{tabular}[t]{c}$u_1$\end{tabular}}}}%
    \put(0.95734316,0.0836283){\makebox(0,0)[t]{\lineheight{1.25}\smash{\begin{tabular}[t]{c}$v_0$\end{tabular}}}}%
    \put(0,0){\includegraphics[width=\unitlength,page=1]{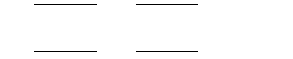}}%
    \put(0.39379995,0.16519111){\makebox(0,0)[t]{\lineheight{1.25}\smash{\begin{tabular}[t]{c}$\cdots$\end{tabular}}}}%
    \put(0.39379995,0.00501927){\makebox(0,0)[t]{\lineheight{1.25}\smash{\begin{tabular}[t]{c}$\cdots$\end{tabular}}}}%
    \put(0,0){\includegraphics[width=\unitlength,page=2]{fig35.pdf}}%
  \end{picture}%
\endgroup%

\end{center}

 Note that the topology
\[
\cY:=\left((\cX_{w_1}\vecotimes \cdots \vecotimes \cX_{w_n})\otimes^* (\cX_{u_1}\vecotimes \cdots \vecotimes \cX_{u_m})\right)\vecotimes \cX_{v_0}
\]
is similar to the tree topology $\cX_{\Gamma}$, but is not isomorphic in general. The open subspaces of $\cY$ may also be described in terms of admissible labelings of $\Gamma$, however the description has some subtle differences. For example, given an admissible labeling $s$ with point-enhancements $\ve{p}$, the open subspaces in an open-enhancement $\cU_{\ve{p}}$ for $\cX_{\Gamma}$ may vary as we change $\ve{p}$. For $\cY$, the open set at a vertex labeled $\cO$ can only depend on the values of $\ve{p}$ for vertices above that vertex. On the other hand, the map
\[
\cX_{\Gamma}\to \cY
\]
is continuous. Compare Proposition~\ref{prop:contraction} and Remark~\ref{ref:non-associative}, below.
\end{rem}

\begin{rem}\label{rem:small-star}
 Suppose that $\Gamma$ is a star with a minimal root $v_0$, and $n$ vertices $w_1,\dots, w_n$, each connected to $v_0$ (and no additional vertices).
 
 \begin{center}
\begingroup%
  \makeatletter%
  \providecommand\color[2][]{%
    \errmessage{(Inkscape) Color is used for the text in Inkscape, but the package 'color.sty' is not loaded}%
    \renewcommand\color[2][]{}%
  }%
  \providecommand\transparent[1]{%
    \errmessage{(Inkscape) Transparency is used (non-zero) for the text in Inkscape, but the package 'transparent.sty' is not loaded}%
    \renewcommand\transparent[1]{}%
  }%
  \providecommand\rotatebox[2]{#2}%
  \newcommand*\fsize{\dimexpr\f@size pt\relax}%
  \newcommand*\lineheight[1]{\fontsize{\fsize}{#1\fsize}\selectfont}%
  \ifx\svgwidth\undefined%
    \setlength{\unitlength}{68.4160047bp}%
    \ifx\svgscale\undefined%
      \relax%
    \else%
      \setlength{\unitlength}{\unitlength * \real{\svgscale}}%
    \fi%
  \else%
    \setlength{\unitlength}{\svgwidth}%
  \fi%
  \global\let\svgwidth\undefined%
  \global\let\svgscale\undefined%
  \makeatother%
  \begin{picture}(1,0.39492257)%
    \lineheight{1}%
    \setlength\tabcolsep{0pt}%
    \put(0.90587805,0.33450286){\makebox(0,0)[t]{\lineheight{1.25}\smash{\begin{tabular}[t]{c}$w_n$\end{tabular}}}}%
    \put(0.09383648,0.33450286){\makebox(0,0)[t]{\lineheight{1.25}\smash{\begin{tabular}[t]{c}$w_1$\end{tabular}}}}%
    \put(0.5056969,0.01030574){\makebox(0,0)[t]{\lineheight{1.25}\smash{\begin{tabular}[t]{c}$v_0$\end{tabular}}}}%
    \put(0.50127048,0.34993131){\makebox(0,0)[t]{\lineheight{1.25}\smash{\begin{tabular}[t]{c}$\cdots$\end{tabular}}}}%
    \put(0,0){\includegraphics[width=\unitlength,page=1]{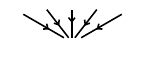}}%
  \end{picture}%
\endgroup%

 \end{center}
 
  If $\cX_{v_0},\cX_{w_1},\dots, \cX_{w_n}$ are spaces indexed by the vertices, then
\[
\cX_{\Gamma}=(\cX_{w_1}\otimes^*\cdots \otimes^* \cX_{w_n})\vecotimes \cX_{v_0}. 
\]
\end{rem}

\begin{prop}\label{prop:contraction}
 Suppose that $\Gamma$ is a strongly directed tree and $\Gamma_0\subset \Gamma$ is a connected subtree.  Let $(\cX_v)_{v\in V(\Gamma)}$ be a collections of spaces, indexed by vertices of $\Gamma$ and let $\cY$ be another linear topological space. Suppose that 
\[
f\colon \cX_{\Gamma_0}\to \cY
\]
is a continuous map. Let $\Gamma'=\Gamma/\Gamma_0$ be the strongly directed tree obtained by contracting all vertices of $\Gamma_0$ to a single vertex, denoted $[\Gamma_0]$. For $v\in V(\Gamma')$, write $\cY_v$ for $\cX_v$ if $v\in V(\Gamma)\setminus V(\Gamma_0)$, and $\cY$ if $v=[\Gamma_0]$.
Then the map 
\[
f\otimes \id\colon \cX_{\Gamma}\to \cY_{\Gamma'}
\]
is continuous.
\end{prop}

See Figure~\ref{fig:47} for a schematic.

\begin{figure}[H]
\begingroup%
  \makeatletter%
  \providecommand\color[2][]{%
    \errmessage{(Inkscape) Color is used for the text in Inkscape, but the package 'color.sty' is not loaded}%
    \renewcommand\color[2][]{}%
  }%
  \providecommand\transparent[1]{%
    \errmessage{(Inkscape) Transparency is used (non-zero) for the text in Inkscape, but the package 'transparent.sty' is not loaded}%
    \renewcommand\transparent[1]{}%
  }%
  \providecommand\rotatebox[2]{#2}%
  \newcommand*\fsize{\dimexpr\f@size pt\relax}%
  \newcommand*\lineheight[1]{\fontsize{\fsize}{#1\fsize}\selectfont}%
  \ifx\svgwidth\undefined%
    \setlength{\unitlength}{166.23561228bp}%
    \ifx\svgscale\undefined%
      \relax%
    \else%
      \setlength{\unitlength}{\unitlength * \real{\svgscale}}%
    \fi%
  \else%
    \setlength{\unitlength}{\svgwidth}%
  \fi%
  \global\let\svgwidth\undefined%
  \global\let\svgscale\undefined%
  \makeatother%
  \begin{picture}(1,0.71827364)%
    \lineheight{1}%
    \setlength\tabcolsep{0pt}%
    \put(0,0){\includegraphics[width=\unitlength,page=1]{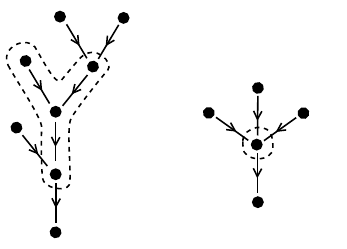}}%
    \put(0.21968103,0.28542232){\makebox(0,0)[lt]{\lineheight{1.25}\smash{\begin{tabular}[t]{l}$\Gamma_0$\end{tabular}}}}%
    \put(0.80051838,0.2678026){\makebox(0,0)[lt]{\lineheight{1.25}\smash{\begin{tabular}[t]{l}$[\Gamma_0]$\end{tabular}}}}%
  \end{picture}%
\endgroup%

\caption{The operation considered in Proposition~\ref{prop:contraction}. On the left is $\Gamma$ and on the right $\Gamma'$. The regions enclosed by dashes indicate $\Gamma_0$ and $[\Gamma_0]$.}
\label{fig:47}
\end{figure}

\begin{proof} We let $\cV'\subset \cY_{\Gamma'}$ be open. Let $s$ be an admissible labeling of $\Gamma$ and $\ve{p}$ a point-enhancement of $s$. Our goal is to construct an open-enhancement $\cU_{\ve{p}}$ of $s$ so that 
\[
(f\otimes \id)(\cW(\cU_{\ve{p}},\ve{p}))\subset \cV'.
\]
Our construction of $\cU_{\ve{p}}$ varies depending on $s$. 

We partition the admissible labelings $s$ of $\Gamma$ as follows:
\begin{enumerate}
\item ($p$-type) $s|_{\Gamma_0}\equiv p$.
\item ($\cO$-type) $s|_{\Gamma_0}$ is not constantly $p$, and all edges pointing into $\Gamma_0$ take value $p$. 
\item ($X$-type)  There is at least one edge pointing into $\Gamma_0$ which takes value $\cO$ or $X$. 
\end{enumerate}

If $s$ is an admissible labeling, then we define a labeling $s'$ of $V(\Gamma')$ as follows. For $v\neq [\Gamma_0]$, we can view $v$ as a vertex of both $\Gamma$ and $\Gamma'$. We set $s'(v)=s(v)$. If $v=[\Gamma_0]$, we define $s'(v)$ to be $p$ if $s$ is $p$-type, $\cO$ if $s$ is $\cO$-type, and $X$ if $s$ is $X$-type.

 We make the following claims:
\begin{enumerate}
\item If $s$ is an admissible labeling of $\Gamma$ then $s'$ is an admissible labeling of $\Gamma'$.
\item If $s$ is $\cO$-type, then $s|_{\Gamma_0}$ is an admissible labeling. 
\end{enumerate}
We leave it to the reader to verify both of the above two claims, as they are straightforward.

We now describe $\cU_{\ve{p}}$. The construction depends on whether $s$ is $p$, $\cO$ or $X$ type. Suppose first that $s$ is $p$-type. We let $s'$ be the admissible labeling of $\Gamma'$ constructed above, and we define a point-enhancement $\ve{p}'$ of $s'$ by using the points from $\ve{p}$ for vertices in $V(\Gamma)\setminus V(\Gamma_0)$, and using $f(\ve{p}(v_1)\otimes \cdots \otimes \ve{p}(v_n))$ for the vertex $[\Gamma_0]$, where $V(\Gamma_0)=\{v_1,\dots, v_n\}$.
Since $\cV'$ is open, there is an open-enhancement $\cU_{\ve{p}'}$ of $s'$ so that 
\begin{equation}
\cW(\cU_{\ve{p}'},\ve{p}')\subset \cV'.\label{eq:open-enhancement-case-p}
\end{equation}
 We can use the same open subspaces from $\cU_{\ve{p}'}$ to construct an open-enhancement $\cU_{\ve{p}}$ of $s$.  Equation~\eqref{eq:open-enhancement-case-p} is equivalent to the statement that
 \[
 (f\otimes \id)(\cW(\cU_{\ve{p}}, \ve{p}))\subset \cV'.
 \]
 
 The construction of $\cU_{\ve{p}}$ when $s$ is $X$-type is similar. In this case, given a point-enhancement $\ve{p}$ of $s$, we define a point-enhancement $\ve{p}'$ of $s'$ by using the points from $\ve{p}$ for the vertices in $V(\Gamma')\setminus \{[\Gamma_0]\}=V(\Gamma)\setminus V(\Gamma_0)$. Since $\cV'$ is open,  there is an open-enhancement $\cU_{\ve{p}'}$ of $s'$ so that
 \[
 \cW(\cU_{\ve{p}'}, \ve{p}')\subset \cV'.
 \]
 We construct an open-enhancement of $s$ by using the subspaces from $\cU_{\ve{p}'}$ for $v\not \in V(\Gamma_0)$, and using the entire subspaces $\cX_v$ for $v\in V(\Gamma_0)$ with $s(v)=\cO$. Since $s'([\Gamma_0])=X$, we have that
 \[
(f\otimes \id) (\cW(\cU_{\ve{p}}, \ve{p}))\subset \cW(\cU_{\ve{p}'}, \ve{p}')
 \]
 which is contained in $\cV'$ by construction.
 
 Finally, we consider the case that $s$ is $\cO$-type (or equivalently $s'([\Gamma_0])=\cO$). We let $\ve{p}$ be a point-enhancement of $s$. We let $\ve{p}'$ be the point-enhancement of $s'$ obtained by restricting $\ve{p}$ to the vertices of $V(\Gamma')\setminus \{[\Gamma_0]\}$. Since $\cV'$ is open, there is an open-enhancement $\cU_{\ve{p}'}$ of $s'$ so that
 \[
 \cW(\cU_{\ps'}, \ve{p}')\subset \cV'.
 \]
 Let $U'_0\subset \cY$ denote the open subspace $\cU_{\ps'}([\Gamma_0])$. Since $f$ is continuous, there is an open subspace $U_0\subset \cX_{\Gamma_0}$ such that $f(U_0)\subset U'_0$.
 
Write $s_0$ for $s|_{V(\Gamma_0)}$. As described earlier, $s_0$ is an admissible labeling. Let $\ve{p}_0$ denote the restriction of $\ve{p}$ to $\Gamma_0$. Since $U_0$ is open, there is an open-enhancement $\cU_{0}$ of $s_0$ so that 
\[
\cW(\cU_{0}, \ve{p}_0)\subset U_0.
\]

We now define the open-enhancement $\cU_{\ve{p}}$ of $s$ as follows. For vertices $w\in V(\Gamma)\setminus V(\Gamma_0)$ with $s(w)=\cO$, we set $\cU_{\ps}(w)=\cU_{\ps'}(w)$. For the vertices $w\in V(\Gamma_0)$ with $s(w)=\cO$, we set $\cU_{\ve{p}}(w)=\cU_0(w)$. By construction
\[
(f\otimes \id)(\cW(\cU_{\ve{p}},\ve{p}))\subset \cW(\cU_{\ve{p}'}, \ve{p}')\subset \cV'.
\]

It follows that $(f\otimes \id)^{-1}(\cV')$ is an open subspace, so the proof is complete.
 \end{proof}

\begin{rem} Note that we can extend the above result in several ways.
\begin{enumerate}
\item If $(\cX_{v})_{v\in V(\Gamma)}$ is a family of linear topological spaces with $\Gamma_0\subset \Gamma$, and if $f\colon \ve{\cX}_{\Gamma_0}\to \ve{\cY}$ is a continuous map between completions, then there is a well-defined, continuous map 
\[
f\otimes \id \colon \ve{\cX}_{\Gamma}\to \ve{\cY}_{\Gamma'}.
\]
To see this, note that it suffices to define a map from $\cX_{\Gamma}$ to $\ve{\cY}_{\Gamma'}$. In this case, observe that if $x\in \cX_{\Gamma_0}$, then by definition $f(x)\in \ve{\cY}$ is the limit of a Cauchy net $\{y_{\b}\}_{\b\in B}$ where $y_\b\in \cY$. If $\xs$ is in the tensor product of the $\cX_v$ for $v\not \in V(\Gamma_0)$, then we claim that $\{ y_{\b}\otimes \xs\}_{\b\in B}$ is a Cauchy net in $\cY_{\Gamma'}$. It suffices to prove this when $\xs$ is an elementary tensor. Let $s'$ be the admissible labeling on $\Gamma/\Gamma_0$ which assigns $\cO$ to $[\Gamma_0]$, assigns $X$ to each vertex below $[\Gamma_0]$, and assigns $p$ to all other vertices. Let $\ve{p}'$ be the point-enhancement of $s'$ which assigns to each vertex $w$ with $s(w')=p$ the corresponding factor from $\xs$. If  $\cV'\subset \cY_{\Gamma'}$ is an open subspace, then by definition there is an open-enhancement $\cU_{\ve{p}'}$ so that
\[
\cW(\cU_{\ve{p}'},\ve{p}')\subset \cV'.
\]
The open-enhancement $\cU_{\ve{p}'}$ consists of a single open subspace $U_{0}\subset \cY$, since $[\Gamma_0]$ is the only $\cO$-labeled vertex. For all sufficiently large $\b$, we have $y_{\b'}-y_{\b''}\in U_{0}$ whenever $\b',\b''\ge \b$. It follows that $ y_{\b'}\otimes \xs-y_{\b''}\otimes \xs\in \cW(\cU_{\ve{p}'}, \ve{p}')$ so $\{y_{\b}\otimes \xs\}_{\b\in B}$ is a Cauchy sequence in $\cY$. We therefore define $(f\otimes \id)(x\otimes \xs)$ to be the limit of the Cauchy net $\{y_\b\otimes \xs\}_{\b\in B}$. A similar argument to the previous paragraph shows that this definition of $(f\otimes \id)$ well-defined (independent of the choice of net $\{y_{\b}\}$). A straightforward extension of Proposition~\ref{prop:contraction} shows that $(f\otimes \id)\colon \ve{\cX}_{\Gamma}\to \ve{\cY}_{\Gamma'}$ is continuous.
\item It follows from the previous argument that the completion $\cX_{\Gamma}$ is unchanged if we replace one $\cX_v$ in our family of spaces by its completion $\ve{ \cX}_v$. To see this, let $\cX_\Gamma'$ denote the tree complex where one $\cX_v$ is replaced by $\ve{\cX}_v$. There is a natural map $\cX_{\Gamma}\to \cX_{\Gamma}'$. Applying the previous remark, there is a natural map in the opposite direction, because the identity map gives a continuous map from the completion of $\ve{\cX}_v$ to the completion of $\cX_v$. These maps are easily seen to be inverses.
\end{enumerate}
\end{rem}

We highlight an important application of Proposition~\ref{prop:contraction}:
\begin{rem}
\label{ref:non-associative}
If $\Gamma$ is a strongly directed graph and $(\cX_{v})_{v\in V(\Gamma)}$ is a family of spaces, we can form another collection of spaces by ``adding parentheses'', in the following sense. Let $\Gamma_0\subset \Gamma$ be a tree, and consider the tree $\Gamma/\Gamma_0$ obtained by collapsing to vertices and edged of $\Gamma_0$ to a single vertex, denoted $[\Gamma_0]$. We can define a collection of spaces indexed by $\Gamma/\Gamma_0$, by setting $\cY_v=\cX_v$ if $v\not \in V(\Gamma_0)$ and $\cY_{[\Gamma_0]}=\cX_{\Gamma_0}$. Proposition~\ref{prop:contraction} shows that the natural map
\[
\bI\colon \cX_{\Gamma}\to \cY_{\Gamma/\Gamma_0}
\] 
is continuous.
\end{rem}

In some cases, there is an isomorphism $\cX_{\Gamma}\iso \cY_{\Gamma/\Gamma_0}$. We focus on the following simple case:

\begin{lem}
\label{lem:retopologize-bridge}
Let $\Gamma$ be a strongly directed tree and $(\cX_v)_{v\in V(\Gamma)}$ a family of spaces and let $v_0$ and $v_1$ be edges of $\Gamma$ which are connected by an edge. Assume that $v_1>v_0$ and that $v_0$ has no other incoming edges. See Figure~\ref{fig:49}. Let $\Gamma_0$ consist of subgraph with vertices $v_0$ and $v_1$ and the edge connecting them. Let $\cY_{\Gamma/\Gamma_0}$ be the space above. If $\cX_{v_1}$ is linearly compact, then there is an isomorphism
\[
\cX_{\Gamma}\iso \cY_{\Gamma/\Gamma_0}.
\]
\end{lem}
\begin{figure}[h]
\begingroup%
  \makeatletter%
  \providecommand\color[2][]{%
    \errmessage{(Inkscape) Color is used for the text in Inkscape, but the package 'color.sty' is not loaded}%
    \renewcommand\color[2][]{}%
  }%
  \providecommand\transparent[1]{%
    \errmessage{(Inkscape) Transparency is used (non-zero) for the text in Inkscape, but the package 'transparent.sty' is not loaded}%
    \renewcommand\transparent[1]{}%
  }%
  \providecommand\rotatebox[2]{#2}%
  \newcommand*\fsize{\dimexpr\f@size pt\relax}%
  \newcommand*\lineheight[1]{\fontsize{\fsize}{#1\fsize}\selectfont}%
  \ifx\svgwidth\undefined%
    \setlength{\unitlength}{88.09987887bp}%
    \ifx\svgscale\undefined%
      \relax%
    \else%
      \setlength{\unitlength}{\unitlength * \real{\svgscale}}%
    \fi%
  \else%
    \setlength{\unitlength}{\svgwidth}%
  \fi%
  \global\let\svgwidth\undefined%
  \global\let\svgscale\undefined%
  \makeatother%
  \begin{picture}(1,1.19175268)%
    \lineheight{1}%
    \setlength\tabcolsep{0pt}%
    \put(0,0){\includegraphics[width=\unitlength,page=1]{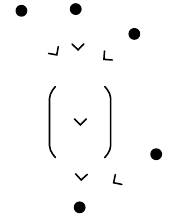}}%
    \put(0.44196041,0.68593068){\color[rgb]{0,0,0}\makebox(0,0)[t]{\lineheight{1.25}\smash{\begin{tabular}[t]{c}$v_1$\end{tabular}}}}%
    \put(0.45286658,0.37312358){\color[rgb]{0,0,0}\makebox(0,0)[t]{\lineheight{1.25}\smash{\begin{tabular}[t]{c}$v_0$\end{tabular}}}}%
    \put(0,0){\includegraphics[width=\unitlength,page=2]{fig49.pdf}}%
  \end{picture}%
\endgroup%

\caption{A tree from Lemma~\ref{lem:retopologize-bridge}. Here, $\Gamma_0$ denotes the parenthesized subtree.}
\label{fig:49}
\end{figure}

\begin{proof} Remark~\ref{ref:non-associative} shows that the natural map
\[
\bI \colon \cX_{\Gamma}\to \cY_{\Gamma/\Gamma_0}
\]
is continuous.  Therefore it suffices to show that the natural map
\[
\bI\colon \cY_{\Gamma/\Gamma_0}\to \cX_{\Gamma}
\]
is continuous. Equivalently, it suffices to show that each open subspace of $\cX_{\Gamma}$ contains an open subspace of $\cY_{\Gamma/\Gamma_0}$. Let $\cV\subset \cX_{\Gamma}$ be an open subspace.

We consider admissible labelings $s$ on $\Gamma/\Gamma_0$. We wish to show that for each admissible labeling $s'$ of $\Gamma/\Gamma_0$ and each point-enhancement $\ve{p}'$ of $s'$, there is an open-enhancement $\cU_{\ve{p}'}$ of $s'$ so that
\[
\cW(\cU_{\ve{p}'},\ve{p}')\subset \cV.
\]
We break the proof into three cases, depending on whether $s([\Gamma_0])$ is $p$, $\cO$ or $X$.

 If $s([\Gamma_0])=X$, then we define a labeling $s$ on $\Gamma$ by setting $s(v_0)=s(v_1)=X$ and setting $s(w)=s'(w)$ for other vertices. If $\ve{p}'$ is a point-enhancement of $s'$, then we can view $\ve{p}'$ also as a point-enhancement of $s$, for which we write $\ve{p}$. Since $\cV$ is open, there is an open-enhancement $\cU_{\ve{p}}$ so that
\[
\cW(\cU_{\ve{p}},\ve{p})\subset \cV.
\]
Since $v_0$ and $v_1$ are labeled $X$ by $s'$, we use the same open subspaces as in $\cU_{\ve{p}}$ to form an open-enhancement $\cU_{\ve{p}'}$ of $s'$. We observe that
\[
\cW(\cU_{\ve{p}'},\ve{p}')=\cW(\cU_{\ve{p}},\ve{p})\subset \cV.
\]

Next, we consider the case that $s'([\Gamma_0])=p$. In this case, we form an admissible labeling $s$ on $\Gamma$ by setting $s(v_0)=s(v_1)=p$, and defining $s$ and $s'$ to coincide at other vertices. We let $\ve{p}'$ be a point-enhancement of $s'$. There is a point assigned assigned to $[\Gamma_0]$, which consists of an element of the tensor product $\cX_{v_0}\otimes \cX_{v_1}$. We can write this as a finite sum of elementary tensors $\sum_{i=1}^n p_i\otimes q_i\in \cX_{v_0}\otimes \cX_{v_1}$. For each elementary tensor $p_i\otimes q_i$, we obtain a point-enhancement $\ve{p}_i$ of $\Gamma$ by using the elements of $\ve{p}'$ for $p$-labeled vertices other than $v_0$ and $v_1$, and by using $p_i$ for $v_0$ and $q_i$ for $v_1$. Since $\cV$ is open, there is an open-enhancement $\cU_{\ve{p}_i}$ of $s$ so that
\[
\cW(\cU_{\ve{p}_i},\ve{p}_i)\subset \cV.
\]
We construct an open-enhancement $\cU_{\ve{p}'}$ of $s'$ as follows. For a vertex $v\in V(\Gamma)\setminus \{v_0,v_1\}$ we set $\cU_{\ve{p}'}(v)=\cU_{\ve{p}_1}(v)\cap \cdots \cap \cU_{\ve{p}_n}(v)$.  We observe that
\[
\cW(\cU_{\ve{p}'},\ve{p}')\subset \cW(\cU_{\ve{p}_1},\ve{p}_1)+\cdots +\cW(\cU_{\ve{p}_n},\ve{p}_n)\subset \cV.
\]

Finally, we consider the case that $s'([\Gamma_0])=\cO$. If $s'([\Gamma_0])=\cO$, then there are two labelings to consider on $\Gamma$. We denote these by $s_{p,\cO}$ and $s_{\cO,X}$. The labeling $s_{p,\cO}$ has $s_{p,\cO}(v_1)=p$ and $s_{p,\cO}(v_0)=\cO$. The labeling $s_{\cO,X}$ has $s_{\cO,X}(v_1)=\cO$ and $s_{\cO,X}(v_0)=X$. Both $s_{p,\cO}$ and $s_{\cO,X}$ agree with $s'$ for other vertices. It is straightforward to see that $s_{p,\cO}$ and $s_{\cO,X}$ are admissible.

Since the $s'$ and $s_{\cO,X}$ have the same $p$-labeled vertices, the collection $\ve{p}'$ induces a point-enhancement $\ve{p}$ (using the same points) of $s_{\cO,X}$. Since $\cV$ is open, there is an open-enhancement $\cU_{\ve{p}}$ of $s_{\cO,X}$ so that
\[
\cW(\cU_{\ve{p}},\ve{p})\subset \cV.
\]
Let $U_{v_1}$ be the open set assigned to $v_1$.

Since $\cX_{v_1}$ is linearly compact, by definition there are elements $p_1,\dots, p_n\in \cX_{v_1}$ so that
\begin{equation}
\Span_{\bF}(p_1,\dots, p_n)+U_{v_1}=\cX_{v_1}.\label{eq:complement-linear-compact}
\end{equation}

We now construct point-enhancements $\ve{p}_1,\dots, \ve{p}_n$ of $s_{p,\cO}$. We use the points from $\ve{p}'$ for the $p$-labeled vertices other than $v_0$, and we use $p_i$ for $v_0$. By definition, there are open-enhancements $\cU_{\ve{p}_1},\dots, \cU_{\ve{p}_n}$ so that
\[
\cW(\cU_{\ve{p}_i},\ve{p}_i)\subset \cV.
\]
Write $U_{v_0,i}=\cU_{\ve{p}_i}(v_0)$.

We now define an open-enhancement $\cU_{\ve{p}'}$ of $s'$. For a vertex $v\neq [\Gamma_0]$ assigned $\cO$, we use the open set
\[
\cU_{\ve{p}}(v)\cap \cU_{\ve{p}_1}(v)\cap \cdots \cap \cU_{\ve{p}_n}(v).
\]

For the vertex $[\Gamma_0]$, we use the open set
\[
U_{v_1}\otimes \cX_{v_2}+\cX_{v_1}\otimes ( U_{v_0,1}\cap \cdots \cap U_{v_0,n}),
\]
which is open in $\cX_{\Gamma_0}\iso \cX_{v_1}\vecotimes \cX_{v_0}$ by Lemma~\ref{lem:lin-comact-discrete-change-completion} because $\cX_{v_1}$ is linearly compact.

We observe that
\begin{equation}
\cW(\cU_{\ve{p}'},\ve{p}')\subset \cW(\cU_{\ve{p}},\ve{p})+\cW(\cU_{\ve{p}_1},\ve{p}_1)+\cdots +\cW(\cU_{\ve{p}_n},\ve{p}_n)\subset \cV.\label{eq:def-open-enhancement-lin-compact-lemma}
\end{equation}
The above equation is proven as follows. For vertices $V(\Gamma/\Gamma_0)$ which are labeled $p$ by $s'$, we use the same points in all of $\ve{p}'$, $\ve{p}$ and $\ve{p}_i$. For the vertices of $V(\Gamma/\Gamma_0)$ which are labeled $\cO$, other than $[\Gamma_0]$, we use the intersection of the opens from $\cU_{\ve{p}}$, $\cU_{\ve{p}_1},$\dots, $\cU_{\ve{p}_n}$, so these factors are contained in the corresponding factors of each summand in the middle of Equation~\eqref{eq:def-open-enhancement-lin-compact-lemma}. Finally, for $v_0$ and $v_1$, we observe that
\[
U_{v_1}\otimes \cX_{v_2}+\cX_{v_1}\otimes (U_{v_0,1}\cap \cdots \cap U_{v_0,n})\subset U_{v_1}\otimes \cX_{v_2}+p_1\otimes U_{v_0,1}+\cdots p_n\otimes U_{v_0,n}
\]
by Equation~\eqref{eq:complement-linear-compact}. Together, the above implies Equation~\eqref{eq:def-open-enhancement-lin-compact-lemma}. We conclude that $\cV$ is open in $\cX_{\Gamma/\Gamma_0}$, completing the proof. 
\end{proof}

\section{The link surgery formula}

In this section, we give some background on the link surgery formula of Manolescu and Ozsv\'{a}th \cite{MOIntegerSurgery}, and prove some important preliminary results.

\subsection{Statement of the link surgery formula}
\label{sec:statement-of-link-surgery}
In this section, we state Manolescu and Ozsv\'{a}th's link surgery formula. We will describe more detail about the construction  in Section~\ref{sec:basic-systems}.

If $L\subset S^3$ is a link, write $K_1,\dots, K_\ell$ for the components of $L$. Manolescu and Ozsv\'{a}th define the affine lattice over $\Z^\ell$
\[
\bH(L):=\prod_{i=1}^\ell \left( \frac{\lk(K_i,L\setminus K_i)}{2}+\Z\right).
\]

Whenever $M\subset L$, Manolescu and Ozsv\'{a}th define a reduction map
\[
\psi^{ M}\colon \bH(L)\to \bH(L\setminus M)
\]
as follows. If $K_{i_1}\cup \cdots \cup K_{i_j}=L$, then they set
\[
\psi^{M}(\ve{s})=(\psi^M_{i_1}(s_i),\dots, \psi^M_{i_j}(s_j))
\]
where
\[
\psi^{M}_{i}(s_i)=s_i-\frac{\lk(K_i,M)}{2}.
\]
Manolescu and Ozsv\'{a}th also define a version of the map $\psi^{M}$ for oriented sublinks of $L$. See \cite{MOIntegerSurgery}*{Section~3.7}. The above map $\psi^M$ corresponds to the case that all components of $M$ are oriented consistently with $L$. 

\begin{define}
Suppose $\cH=(\Sigma,\as,\bs,\ws,\zs,\ps)$ is a Heegaard link diagram for $(Y,L)$ with free basepoints $\ps$, and link basepoints $\ws\cup \zs$. We say that $\cH$ is \emph{link minimal} if each link component of $L$ has exactly one basepoint from $\ws$, and one basepoint from $\zs$. 
\end{define}

Let $\cH=(\Sigma,\as,\bs,\ws,\zs)$ be a link-minimal Heegaard diagram for an oriented link $L\subset S^3$ with no free-basepoints. If $M\subset L$, we write $\cH^{L\setminus M}$ for the diagram obtained by deleting $z_i\in \zs$ for each component $K_i\subset L$. Note that Manolescu also define a reduction $\cH^{L\setminus \vec{M}}$ when $\vec{M}$ is equipped with an orientation. In this case, we delete $z_i$ for each positively oriented component of $M$, and we delete $w_i$ for each negatively oriented component of $M$. The diagram $\cH^{L\setminus M}$ is a link minimal diagram with $|M|$ free-basepoints.

Recall that if $(\Sigma,\as,\bs,\ws,\zs)$ is a link-minimal Heegaard link diagram for $L$, then there is an $\ell$-component Alexander grading
\[
A=(A_1,\dots, A_\ell)\colon \bT_{\a}\cap \bT_{\b}\to \bH(L).
\]

Suppose $\cH$ is a link minimal Heegaard diagram for $L\subset S^3$ and $\cH$ has no free basepoints. If $\ve{s}\in \bH(L)$, then Manolescu and Ozsv\'{a}th define
\[
\frA^-(\cH^{L\setminus M},\ve{s})
\] 
to be the free  $\bF\llsquare U_1,\dots, U_\ell\rrsquare $-module generated by monomials
\[
U_1^{i_1}\cdots U_{n}^{i_\ell} \cdot \xs
\]
where $\xs\in \bT_{\a}\cap \bT_{\b}$, which satisfy $i_j\ge 0$ for all $j$, and $A_j(\xs)\le s_j+i_j$ whenever $K_j\subset L\setminus M$. Manolescu and Ozsv\'{a}th equip $\frA^-(\cH^{L\setminus M},\ve{s})$ with the differential which counts Maslov index 1 holomorphic disks, weighted by the product of $U_i^{n_{w_i}(\phi)}$. 

We will also occasionally write
\[
A^-(\cH^{L\setminus M},\ve{s})
\]
for the free $\bF[U_1,\dots, U_n]$-module which has the same generators as $\frA^-(\cH^{L\setminus M},\ve{s})$.

Let $\Lambda$ denote an integral framing on $L$. The underlying group of the link surgery complex is defined to be
\[
\cC_{\Lambda}(L)=\bigoplus_{M\subset L} \prod_{\ve{s}\in \bH(L)} \frA^-(\cH^{L\setminus M}, \psi^{M}(\ve{s})).
\]
 The differential $\cD\colon \cC_{\Lambda}(L)\to \cC_{\Lambda}(L)$ decomposes as a sum
 \[
 \cD=\sum_{N\subset L} \sum_{\vec{N} \in \Omega(N)} \Phi^{\vec{N}},
 \]
 where $\Omega(N)$ denotes the set of orientations on $N$.

Given an oriented sublink $\vec{N}\subset L$ (with orientation potentially different than $L$), Manolescu and Ozsv\'{a}th define \[
\Lambda_{L,\vec{N}}\in \Z^\ell\iso H_1(S^3\setminus L)
\]
to be the sum of the longitudes of the negatively oriented components of $\vec{N}$.

Suppose $\vec{N}\subset L$ is an oriented sublink, and that $L_0\subset L$ is a sublink which contains $N$. Write $L_1=L_0\setminus N$. In this case, the summand $\Phi^{\vec{N}}$ of $\cD$ maps   $\frA^-(\cH^{L_0},\psi^{L\setminus L_0}(\ve{s}))$
to $\frA^-(\cH^{L_1},\psi^{L\setminus L_1}(\ve{s}+\Lambda_{L,\vec{N}}))$.

We write $\Phi_{L_0,L_1}^{\vec{N}}$ for the component of $\Phi^{\vec{N}}$ as follows:
\[
\Phi_{L_0,L_1}^{\vec{N}}\colon \prod_{\ve{s}\in \bH(L)}\frA^-(\cH^{L_0}, \psi^{L\setminus L_0}(\ve{s}))\to \prod_{\ve{s}\in \bH(L)}\frA^-(\cH^{L_1}, \psi^{L\setminus L_1}(\ve{s}+\Lambda_{L,\vec{N}})).
\]

Finally, we note that it is often convenient to view $\cC_{\Lambda}(L,\scA)$ as a one dimensional mapping cone, so we establish some notation for this purpose. If $J\subset L$ is a component, and $\nu\in \{0,1\}$, then we write $\cC_{\Lambda}^\nu(L)$ for the subcube consisting of  $\cC_{\veps}\subset \cC_{\Lambda}(L,\scA)$ ranging over $\veps$ with $J$ component equal to $\nu$. We may decompose
\[
\cC_{\Lambda}(L)\iso \Cone\left(
\begin{tikzcd}[column sep=2cm]
\cC_\Lambda^0(L) \ar[r, "F^J+F^{-J}"]&\cC_\Lambda^1(L)
\end{tikzcd}
 \right).
\]
In the above, $F^J$ (resp. $F^{-J}$) denotes the sum of $\Phi^{\vec{N}}$ for every oriented sublink $\vec{N}$ of  $L$ such that $J\subset \vec{N}$ (resp. $-J\subset \vec{N}$).

 We will review the construction of the hypercube maps $\Phi^{\vec{N}}$ in more detail in  Section~\ref{sec:basic-systems}.

\subsection{Gradings and algebraic reduction maps}

In this section, we state several grading formulas which are useful when working with the link surgery formula.

\begin{define}
Suppose that $\cH=(\Sigma,\as,\bs,\ws,\zs,\ps)$ is a link minimal Heegaard diagram, where $\ws\cup \zs$ are link basepoints and $\ps$ are free basepoints. We call a subset $\cW\subset \ws \cup \zs\cup \ps$ a \emph{complete set of basepoints} on $\cH$ if $\ps\subset \cW$ and $\cW$ contains exactly one basepoint from each link component. 
\end{define}

Suppose $\cH$ is a diagram for a link $L$ in $Y$, each of whose components is null-homologous, and $\frs\in \Spin^c(Y)$ is torsion. Given a complete set of basepoints $\cW$ on $\cH$, there is a well defined Maslov grading $\gr_{\cW}$ on $\cCFL(\cH,\frs)$. The grading $\gr_{\cW}$ is induced from the absolute grading of $\CF^-(Y,\cW,\frs)$ by declaring the natural quotient map $\cCFL(\cH,\frs)\to \CF^-(Y,\cW,\frs)$ to be grading preserving. There is also an $\ell$-component Alexander grading $A^{\cW}=(A_1^{\cW},\dots, A_\ell^{\cW})$. Note that a complete set of basepoints also determines an orientation on $L$, via the convention that $L$ intersects the Heegaard surface $\Sigma$ negatively at the link basepoints in $\cW$. In particular, we may equivalently specify the Alexander grading $A^L$ by picking an orientation on $L$.

The variable $\scU_i$ has $\gr_{\cW}$-grading $-2$ if $w_i\in \cW$, and $0$ if $w_i\not\in \cW$. Similarly $\scV_i$ has $\gr_{\cW}$-grading $-2$ if $z_i\in \cW$ and $0$ otherwise. The variable $U_i$ for a free basepoint $p_i\in \ps$ has $\gr_{\cW}$-grading $-2$.

\begin{rem}\label{rem:grading-conventions}
There are two natural conventions for defining the absolute grading for multi-pointed 3-manifolds. We recall that adding a basepoint has the effect on $\widehat{\HF}$ of tensoring with a 2-dimensional vector space $V=\bF\oplus \bF$. See \cite{OSLinks}*{Section~6}. In this paper, we use the convention that the top degree element of $V$ has absolute grading 0. This is different than the TQFT convention of \cite{ZemAbsoluteGradings}, where $V$ is supported in gradings $1/2$ and $-1/2$.
\end{rem}

The following is implicit in the work of Manolescu and Ozsv\'{a}th \cite{MOIntegerSurgery} (see in particular \cite{MOIntegerSurgery}*{Section~12}). In particular, the results are not new to our paper, however the presentation is helpful.

\begin{lem}
\label{lem:grading-changes}
 Let $L=K_1\cup \cdots \cup K_\ell$ denote an $\ell$-component link in a 3-manifold $Y$, such that each component is null-homologous, and suppose $Y$ is equipped with a torsion $\Spin^c$ structure $\frs$. Let $\cH=(\Sigma,\as,\bs,\ws,\zs,\ps)$ be a link minimal Heegaard diagram, where $\ws$ and $\zs$ are link basepoints and $\ps$ are free basepoints. Let $\cC$ denote the link Floer complex $\cCFL(\cH,\frs)$ over $\bF[\scU_1,\dots, \scU_\ell, \scV_1,\dots, \scV_\ell,U_1,\dots, U_n]$, where $\ell=|\ws|=|\zs|$ and $n=|\ps|$. 
 \begin{enumerate} 
 \item\label{gradings-1} If $\cW$ is complete collection of basepoints on $\cH$, and $z_i\in \zs$ but $z_i\not \in \cW$, then the quotient map 
 \[
 \cQ_{\scV_i}\colon \cC\to \cC/(\scV_i-1)
 \]
  is grading preserving with respect to $\gr_{\cW}$.  Similarly, the quotient map
  \[
  \cQ_{\scU_i}\colon  \cC\to \cC/(\scU_i-1),
  \]
  is grading preserving if $w_i\in \ws$ but $w_i\not \in \cW$.
 \item\label{gradings-2} Suppose that $K_j$ is a component of $L$, and let $\cW_0\cup \{w_j\}$ be a complete collection of basepoints, and write $\{w_j,z_j\}=K_j\cap (\ws\cup \zs)$.  Let $L$ and $K_j$ have orientations induced by $\cW_0\cup \{w_j\}$. Then
 \[
\gr_{\cW_0\cup \{w_j\}}-\gr_{\cW_0\cup \{z_j\}}= 2A_j^{\cW_0\cup \{w_j\}}-\lk(L\setminus K_j, K_j).
 \]
 \item\label{gradings-3} Suppose that $\cW$ and $\cW'$ are two complete collections of basepoints which coincide at index $j$. Then
 \[
 A^{\cW}_j=A^{\cW'}_j.
 \]
  \item\label{gradings-4} Let $K_i\subset L$, and orient both $K_i$ and $L$  to intersect $\Sigma$ negatively at $\ws$, and suppose $i\neq j$. Give $\cC$ and $\cC/(\scV_i-1)$ the Alexander grading $A_j$ induced by the basepoints $\ws$ and $\ws\setminus \{w_i\}$. For $j\neq i$,
  \[
 A_j(\cQ_{\scV_i})= -\frac{1}{2}\lk(K_i,K_j).
  \]
 \item \label{gradings-5} Let $K_i\subset L$, and orient $K_i$ and $L$ to intersect $\Sigma$ negatively at $\ws$, and suppose $j\neq i$. Give $\cC$ and $\cC/(\scU_i-1)$ the Alexander grading $A_j$ induced by the basepoints $\ws$ and $\ws\setminus \{w_i\}$, respectively. Then,  
 \[
 A_j(\cQ_{\scU_i})=\frac{1}{2}\lk(K_i,K_j).
 \]
 \end{enumerate}
\end{lem}

\begin{proof} Following \cite{ZemAbsoluteGradings}, the Alexander and Maslov gradings may be defined by representing $Y\setminus N(L)$ as Dehn surgery on a framed link $J$ in the complement of an $\ell$-component unlink in $S^3$. We pick a Heegaard triple $\cT=(\Sigma,\as,\bs_0,\bs,\ws,\zs)$ which represents the 2-handle cobordism $X(J)$. We pick a triangle class $\psi\in \pi_2(\xs,\ys,\zs)$ on $\cT$. In \cite{ZemAbsoluteGradings}*{Section~5.5}, the following formulas for the Maslov and Alexander gradings are proven:
\[
\begin{split}
\gr_{\cW}(\zs)=&\gr_{\cW}(\xs)+\gr_{\cW}(\ys)-\frac{1}{2}(g(\Sigma)+|\cW|-1)-\mu(\psi)+2n_{\cW}(\psi)\\
&+\frac{c_1(\frs_{\cW}(\psi))^2-2\chi(X(J))-3\sigma(X(J))}{4}\label{eq:Maslov-equation-def}\\
A_j^{\cW}(\zs)=&A_j^{\cW}(\xs)+A_j^{\cW}(\ys)+(n_{\cW}-n_{(\ws\cup \zs\cup \ps)\setminus\cW})(\psi)+\frac{\langle c_1(\frs_{\cW}(\psi)), \hat{S}_j\rangle-[\hat{S}]\cdot [\hat{S}_j]}{2}.
\end{split}
\]
In the above, if $W$ is obtained by attaching 2-handles to $S^3\times [0,1]$ in the complement of an unlink with components $u_1\cup \cdots \cup u_\ell$, then $S_j$ is the surface $u_j\times [0,1]$ and $S=S_1+\cdots+S_\ell$. The classes $\hat{S}$ and $\hat{S}_j$ are obtained by capping these surfaces with Seifert surfaces in either end. 

With the above in place, Claim~\eqref{gradings-1} follows because the formula for $\gr_{\cW}$ does not change if we delete a basepoint which is not in $\cW$.

Consider now Claim~\eqref{gradings-2}. An identical argument to \cite{ZemAbsoluteGradings}*{Lemma~3.8} implies that if $\psi$ is a homology class of triangles on the surgery diagram for the above cobordism, then
\begin{equation}
\frs_{\cW_{0}\cup \{w_j\}}(\psi)-\frs_{\cW_0\cup \{z_j\}}(\psi)=\PD[S_j].
\label{eq:difference-Spinc-structures}
\end{equation}

 We may use the formula in Equation~\eqref{eq:Maslov-equation-def} to evaluate $\gr_{\cW_0\cup \{w_j\}}-\gr_{\cW_0\cup \{z_j\}}$. Firstly, we recall that by straightforward computation $\widehat{\HFL}(\as,\bs_0,\ws,\zs,\ps)$ is isomorphic to $(\bF\oplus \bF)^{\otimes (|\ws|+|\ps|-1)}$. This computation can be obtained by using topological invariance of the Heegaard Floer group, and picking a convenient genus 0 Heegaard diagram, so that each $w_i$ is immediately adjacent to $z_i$, and so that the attaching curves come in isotopic pairs. Similarly, $\widehat{\HFL}(\bs_0,\bs,\ws,\zs,\ps)$ is isomorphic to $(\bF\oplus \bF)^{\otimes (g(\Sigma)-|J|+|\ws|+|\ps|-1)}$. Furthermore, for both groups, the non-zero elements of homology decompose into homogeneously graded summands which have the property that
 \begin{equation}
 \gr_{\cW}(\xs)=\gr_{\cW'}(\xs)\quad \text{and} \quad A^{\cW}(\xs)=A^{\cW'}(\xs)=0\label{eq:simple-generators-grading-formulas}
 \end{equation}
 for all complete collections $\cW,\cW'\subset \ws\cup \zs\cup \ps$. Therefore, in our grading formulas, we may assume that the intersection points $\xs$ and $\ys$ satisfy Equation~\eqref{eq:simple-generators-grading-formulas}. Using Equation ~\eqref{eq:difference-Spinc-structures}, we therefore obtain
\[
\begin{split}
&\gr_{\cW_0\cup \{w_j\}}(\zs)-\gr_{\cW_0\cup \{z_j\}}(\zs)
\\
=&\frac{c_1^2(\frs_{\cW_0\cup \{w_j\}}(\psi))-\left(c_1(\frs_{\cW_0\cup \{w_j\}}(\psi)) -2\PD[S_j]\right)^2}{4}-2(n_{z_j}-n_{w_j})(\psi).
\end{split}
\]
We may rearrange the above to see that
\[
\begin{split} &\gr_{\cW_0\cup \{w_j\}}(\zs)-\gr_{\cW_0\cup \{z_j\}}(\zs)
\\=&
\langle c_1(\frs_{\cW_0\cup \{w_j\}}(\psi)), \hat{S}_j\rangle-[\hat{S}_j]\cdot [\hat{S}_j]-2(n_{z_j}-n_{w_j})(\psi)\\
=&\langle c_1(\frs_{\cW_0\cup \{w_j\}}(\psi)), \hat{S}_j\rangle-[\hat{S}]\cdot [\hat{S}_j]-2(n_{z_j}-n_{w_j})(\psi)+[\hat{S} \setminus \hat{S}_j]\cdot [\hat{S}]\\
=&2 A_j^{\cW_0\cup \{w_j\}}(\zs)+[\hat{S}\setminus\hat{S}_j]\cdot [\hat{S}_j]\\
=&2 A_j^{\cW_0\cup \{w_j\}}(\zs)-\lk(L\setminus K_j,K_j).
\end{split}
\]
In the last line, we used the equality $[\hat{S}\setminus \hat{S}_j]\cdot [\hat S_j]=-\lk(L\setminus K_j, K_j)$, which follows since we capped with negative Seifert surfaces of $L$ (and since we are using the outward normal first convention for boundary orientations).

We now consider Claim~\eqref{gradings-3}. We consider the case when $\cW$ and $\cW'$ differ only at a single index $i\neq j$. Assume that $\cW$ contains $w_i$ and that $\cW'$ contains $z_i$. Write $L'$ for $L$ with the orientation of $K_i$ reversed. Computing directly from the definition, shows that
\[
2(A_j^{\cW}-A_j^{\cW'})=\langle c_1(\frs_{\cW}(\psi))-c_1(\frs_{\cW'}(\psi)), \hat{S}_j\rangle -\hat{S} \cdot \hat S_j+\hat S' \cdot \hat S_j.
\]
Here $\hat S'=\hat S-2 \hat S_i$. As before $\frs_{\cW}-\frs_{\cW'}=\PD[S_i]$. Hence, the above equation gives
\[
A_j^{\cW}-A_j^{\cW'}=0 
\]

 Consider claim \eqref{gradings-4} now. For this claim, we apply claim ~\eqref{gradings-2}, with $\cW_0=\cW\setminus \{w_j\}$. Since $z_i\not \in \cW_0\cup \{w_j\}$ and $z_i\not \in \cW_0\cup \{z_j\}$, by part~\eqref{gradings-1} we know that $\cQ_{\scV_i}$ is homogeneously graded with respect to $\gr_{\cW_0\cup \{w_j\}}$ and $\gr_{\cW_0\cup \{z_j\}}$. Using claim~\eqref{gradings-2}, we obtain
\[
2\Delta ( A_j)=\lk(L\setminus (K_i\cup K_j), K_j)-\lk(L\setminus K_j,K_j)=-\lk(K_i,K_j).
\]

We now consider claim~\eqref{gradings-5}. By claim~\eqref{gradings-3}, it is sufficient to compute the grading change of $A^{\cW'}_j$ where $\cW'=(\ws\setminus \{w_i\})\cup \ps\cup \{z_i\}$. We use claim~\eqref{gradings-4}, but note that with respect to $\cW'$, the orientation of $K_i$ is reversed, while the orientation of $K_j$ is unchanged.
\end{proof}

\subsection{On the algebra of the link surgery formula}

In this section, we reformulate the algebra of the link surgery formula slightly. We begin with a convenient description of the underlying group:

\begin{lem}\label{lem:module-structure-surgery-hypercube-groups} Suppose that $\cH$ is a link minimal Heegaard diagram for a link $L$ in $S^3$, which has no free basepoints. Let $M$ be a sublink of $L$. Write $S_M\subset \bF[\scU_1,\dots, \scU_{\ell},\scV_1,\dots, \scV_{\ell}]$ for the multiplicatively closed subset generated by $\scV_i$ for $i$ such that $K_i\subset M$. Then there is an $\bF[U_1,\dots, U_\ell]$-equivariant chain isomorphism
\[
\bigoplus_{\ve{s}\in \bH(L)} A^-(\cH^{L\setminus M}, \psi^{M}(\ve{s}))\iso S_M^{-1} \cdot\cCFL(\cH),
\]
where we view $U_i$ as acting by $\scU_i\scV_i$ on the right-hand side. Furthermore, if $\ve{s}\in \bH(L)$, this isomorphism intertwines the summand $A^-(\cH^{L\setminus M}, \psi^{M}(\ve{s}))$ with the subspace of $S_M^{-1}\cdot \cCFL(\cH)$ in Alexander multi-grading $\ve{s}$.
\end{lem}
\begin{proof}
 The module $S_M^{-1}\cdot \cCFL(\cH)$ decomposes over Alexander gradings $\ve{s}\in \bH(L)$. Hence, it is sufficient to identify the subspace in Alexander grading $\ve{s}$ with
 \[
A^-(\cH^{L\setminus M}, \psi^{M}(\ve{s}))
 \]
 in a way which commutes with the action of $\bF[U_1,\dots, U_\ell]$.  We recall that $A^-(\cH^{L\setminus M}, \psi^{M}(\ve{s}))$ is generated by monomials $U_1^{i_1}\cdots U_{\ell}^{i_\ell}\cdot \ve{x}$  such that $i_j\ge 0$ for all $j$, and 
 \[
 A_j^{L\setminus M}(\ve{x})-i_j\le \psi^{M}_j(s_j)
 \]
 for each $j$ such that $K_j\not \in M$. Here, we are writing $\ve{s}=(s_1,\dots, s_\ell)$. Note that by definition of $\psi_j^{M}(s_j)$ (see \cite{MOIntegerSurgery}*{Section~3.7}), we have
 \[
A^L_j(\ve{x})-A^{L\setminus M}_j(\ve{x})=s_j-\psi_j^{M}(s_j).
 \]
 Compare Lemma~\ref{lem:grading-changes}.
Hence we may think of $A^-(\cH^{L\setminus M}, \psi^{M}(\ve{s}))$ as being generated by monomials $U_1^{i_1}\cdots U_{\ell}^{i_\ell} \cdot \xs$ such that $i_j\ge 0$ for all $j$ and such that
\[
A_j^L(\ve{x})-i_j\le s_j
\]
whenever $K_j\not \in M$. We may define a map
\[
A^-(\cH^{L\setminus M}, \psi^{M}(\ve{s}))\to S_M^{-1}\cdot \cCFL(\cH)
\]
via the formula
\[
U_1^{i_1}\cdots U_{\ell}^{i_\ell}\cdot \xs\mapsto \scU_1^{i_1}\cdots \scU_{\ell}^{i_\ell}\scV_1^{s_1-A_1^L(\xs)+i_1}\cdots \scV_\ell^{s_{\ell}-A^L_{\ell}(\xs)+i_\ell}\cdot \xs.
\]
\end{proof}

\begin{rem}
\label{rem:localization-natural}
 In their construction of the link surgery formula, Manolescu and Ozsv\'{a}th also inclusion maps $\cI^{\vec{N}}_{\ve{s}}$ between the complexes $\frA^-(\cH,\ve{s})$ and some generalizations of these complexes. See \cite{MOIntegerSurgery}*{Section~3.8}. We note that the map $\cI^{\vec{N}}_{\ve{s}}$ can be identified with the map for localizing at the variables $\scU_i$, for $i$ such that $+K_i\subset \vec{N}$, and localizing at variables $\scV_i$, for $i$ such that $-K_i\subset \vec{N}$. 
 \end{rem}

We now discuss some algebraic properties of the maps in the link surgery hypercube. Manolescu and Ozsv\'{a}th prove that the hypercube maps in their surgery formula $\cC_{\Lambda}(L)$ are $\bF[U_1,\dots, U_\ell]$-equivariant. By the previous lemma, there is a more refined action of $\bF[\scU_1,\dots, \scU_\ell, \scV_1,\dots, \scV_\ell]$ on the complex $\cC_{\Lambda}(L)$, though the hypercube differential does not commute with this action. We now describe how the hypercube maps interact with the action of this larger ring.

Firstly, we define a ring homomorphism
\[
\phi^\tau\colon \bF[\scU,\scV]\to \bF[\scU,\scV,\scV^{-1}]
\]
via the formula
\[
\phi^\tau(\scU)=\scV^{-1}\quad \text{and} \quad \phi^\tau(\scV)=\scU\scV^2.
\]
We write $\phi^\sigma\colon \bF[\scU,\scV]\to \bF[\scU,\scV,\scV^{-1}]$ for the canonical inclusion.

We view $S_M^{-1}\cdot\bF[\scU_1,\dots, \scU_{\ell}, \scV_1,\dots, \scV_\ell]$ as the tensor product over $\bF$ of $\ell$ different rings, which are each isomorphic to $\bF[\scU,\scV]$ or $\bF[\scU,\scV,\scV^{-1}]$.

If $N\subset L\setminus M$, then we define the ring homomorphism
\[
\phi^{\vec{N}}\colon S_M^{-1}\cdot  \bF[\scU_1,\dots, \scU_{\ell}, \scV_1,\dots, \scV_\ell]\to S_{M\cup N}^{-1}\cdot \bF[\scU_1,\dots, \scU_{\ell},\scV_1,\dots, \scV_{\ell}] 
\]
by tensoring the map $\phi^{\tau}$ for each component $K_i\subset \vec{N}$ which is oriented oppositely to $L$, tensoring $\phi^{\sigma}$ for each component $K_i\subset \vec{N}$ which is oriented consistently with $L$, and tensoring the identity map for the remaining components of $L$.

\begin{lem}\label{lem:module-structure-surgery-hypercube-morphisms}
Suppose that  $M\subset L$ and $N\subset L\setminus M$. Let $\vec{N}$ denote $N$ with a choice of orientation. 
If $a\in S_{M}^{-1}\cdot\bF[\scU_1,\dots, \scU_\ell, \scV_1,\dots, \scV_\ell]$, and $\ve{x}\in S_{M}^{-1}\cdot \cCFL(\cH)$,  then
\[
\Phi^{\vec{N}}(a\cdot \ve{x})=\phi^{\vec{N}}(a)\cdot \Phi^{\vec{N}}(\ve{x}).
\]
\end{lem}
\begin{proof} We focus on the length 1 maps in the surgery hypercube. The maps of higher length are analyzed using the same argument. In this case, write $N=K_i$, where $K_i$ is a component of $L$. To simplify the notation further, we focus on the case that $M=\emptyset$. Let $K_i$, (resp. $-K_i$), denote $K_i$ equipped with the orientation which coincides with (resp. is opposite to) the orientation from $L$.  The map $\Phi^{K_i}_{L,L\setminus K_i}$ is  the canonical inclusion map
\[
\cCFL(\cH)\hookrightarrow \scV_i^{-1}\cdot \cCFL(\cH),
\]
which is clearly $\bF[\scU_1,\dots, \scU_\ell, \scV_1,\dots, \scV_{\ell}]$-equivariant. Note that $\phi^{\sigma_i}$ is just the canonical inclusion of $\bF[\scU_1,\dots, \scU_\ell, \scV_1,\dots, \scV_{\ell}]$ into $\scV_i^{-1}\cdot \bF[\scU_1,\dots, \scU_{\ell}, \scV_1,\dots, \scV_{\ell}]$.

 The map $\Phi^{-K_i}_{L,L\setminus K_i}$ is more complicated, as we now describe. Similarly to Lemma~\ref{lem:module-structure-surgery-hypercube-groups}, there are canonical isomorphisms
\[
\begin{split}
\theta_i\colon \scU_i^{-1}\cdot \cCFL(\cH)&\to \cCFL(\cH)/(\scU_i-1)\otimes \bF[\bH_i(L)], \quad \text{and}
\\
\rho_{i}\colon \scV_i^{-1}\cdot \cCFL(\cH)&\to \cCFL(\cH)/(\scV_i-1)\otimes \bF[\bH_i(L)].
\end{split}
\]
Here $\bF[\bH_i(L)]$ denotes the $\bF[T_i,T_i^{-1}]$-module generated by $T_i^\alpha$ where $\alpha\in \bH_i(L)$. Here, $T_i$ is a formal variable. We also view $\cCFL(\cH)/(\scU_i-1)$ and $\cCFL(\cH)/(\scV_i-1)$ as $\bF[U_i,T_i,T_i^{-1}]$-modules, where $U_i$ acts by $\scV_i$ on $\cCFL(\cH)/(\scU_i-1)$, and by $\scU_i$ on $\cCFL(\cH)/(\scV_i-1)$.

The map $\theta_i$ is given as follows. Write 
\[
\cQ_{\scU_i} \colon \scU_i^{-1}\cCFL(\cH)\to \cCFL(\cH)/(\scU_i-1)
\]
for the quotient map, which sends $\scU_i$ to $1$ and which sends $\scV_i$ to $U_i$. 
Then $\theta_i$ is given by the formula
\[
\theta_i(\xs)=\cQ_{\scU_i}(\xs)\otimes T_i^{A_i^L(\xs)}.
\]
The map $\rho_i$ is defined similarly. 

We write $\cI_i$ for the inclusion of $\cCFL(\cH)$ into $\scU_i^{-1}\cdot \cCFL(\cH)$.

Moving the basepoint $z_i$ to $w_i$ along a subarc of $K_i$ determines a homotopy equivalence
\[
\Psi_{z_i\to w_i}\colon \cCFL(\cH)/(\scU_i-1)\to \cCFL(\cH)/(\scV_i-1)
\]
which preserves $A_j^{L\setminus K_i}$ for $j\neq i$. The map $\Psi_{z_i\to w_i}$ is $\bF[U_i,T_i,T_i^{-1}]$-equivariant, and is equivariant with respect to the variables for other link components.

The map $\Phi_{L,L\setminus K_i}^{-K_i}$ is defined as the composition
\[
\Phi_{L,L\setminus K_i}^{-K_i}:= \rho_i^{-1}\circ (\Psi_{z_i\to w_i}\otimes T_i^{\lambda_i})\circ \theta_i\circ \cI_i.
\]

We now consider the interaction of $\Phi_{L,L\setminus K_i}^{-K_i}$ with $\scU_i$ and $\scV_i$. The map $\cI_i$ commutes with $\scU_i$ and $\scV_i$. The map $\theta_i$ has the property that
\[
\theta_i(\scU_i^n \cdot \xs)=T_i^{-n}\cdot \theta_i(\xs)\quad \text{and} \quad \theta_i(\scV_i^n\cdot \xs)=U_i^nT_i^n\cdot\theta_i(\xs) .
\]
Similarly
\[
\rho_i(\scU_i^n \cdot \xs)=U_i^nT_i^{-n} \cdot \rho_i(\xs)\quad \text{and} \quad \rho_i(\scV_i^n \cdot \xs)=T_i^n \cdot \rho_i(\xs),
\]
which implies that
\[
\rho_i^{-1}(T_i^n \cdot \xs)=\scV_i^n\cdot\rho_i^{-1}(\xs)\quad \text{and} \quad \rho_i^{-1}(U_i^n\cdot  \xs)=\scU_i^n\scV_i^n \cdot  \rho_i^{-1}(\xs).
\]
From these relations, it follows that
\[
\Phi_{L,L\setminus K_i}^{-K_i}(\scU_i \cdot \ve{x})=\scV_i^{-1} \cdot \Phi_{L,L\setminus K_i}^{-K_i}(\ve{x})\quad \text{and} \quad \Phi_{L,L\setminus K_i}^{-K_i}(\scV_i\cdot \ve{x})=\scU_i \scV_i^2\cdot \Phi_{L,L\setminus K_i}^{-K_i}(\ve{x}),
\]
completing the proof. 
\end{proof}

\begin{example} 
We now describe the surgery complex for the unknot $\bO$ in our present notation, and compare it to the traditional notation of \cite{OSIntegerSurgeries}. The complex $\cCFK(\bO)$ has one generator, $1$, over $\bF[\scU,\scV]$. In our notation, $v$ sends $\scU^i\scV^j\in \cCFK(\bO)$ to $\scU^i \scV^j\in \scV^{-1} \cCFK(\bO)$. Similarly $h_{\lambda}$ sends $\scU^i\scV^j$ to $\scV^{\lambda} \phi^\tau(\scU^i\scV^j)=\scU^j \scV^{2j-i+\lambda}$.

Traditionally, one works with the mapping cone formula over $\bF[U]$. We identify both $A(\bO,s)\subset \cCFK(\bO)$ and $B(\bO,s)\subset \scV^{-1} \cCFK(\bO)$ with $\bF[U]$ (where $U$ acts by $\scU\scV$). We identify $\scU^i\scV^j\in \cCFK(\bO)$ with $U^{\min(i,j)}\in A(\bO,s)$, and we identify $\scU^i\scV^j\in \scV^{-1} \cCFK(\bO)$ with $U^{i}\in B(\bO,s)$. We can view the previous identification as giving two isomorphisms
\[
 \cCFK(\bO)\iso \bigoplus_{s\in \Z} \bF[U]\quad \text{and} \quad   \scV^{-1}\cCFK(\bO)\iso \bigoplus_{s\in \Z} \bF[U].
\]
In the original notation of Ozsv\'{a}th and Szab\'{o}
\[
v(U^i)=\begin{cases} U^{i-s}& \text{if } s\le  0,\\
U^i& \text{if } s\ge 0
\end{cases}
\qquad
h_{\lambda}(U^i)=\begin{cases}U^i &\text{if } s\le 0,\\
U^{i+s}& \text{if } s\ge 0.
\end{cases}
\]
The map $h_{\lambda}$ sends $A(\bO,s)$ to $B(\bO,s+\lambda)$. It is straightforward to see that the above maps form a commutative diagram
\[
\begin{tikzcd}\cCFK(\bO)\ar[d, "\iso"] \ar[r, "f"] &\scV^{-1} \cCFK(\bO)\ar[d, "\iso"]\\
\displaystyle\bigoplus_{s\in \Z} A(\bO,s)\ar[r, "f"]& \displaystyle\bigoplus_{s\in \Z} B(\bO,s)
\end{tikzcd}
\]
whenever $f$ is one of $v$ or $h_{\lambda}$.
\end{example}

\section{The knot surgery algebra}

\label{sec:Algebra-K}

In this section, we introduce our algebras $\cK$ and $\cL$. We also describe natural filtrations on $\cK$ and $\cL$. Using these filtrations, we describe our module categories.

\subsection{The algebras \texorpdfstring{$\cK$}{K} and \texorpdfstring{$\cL$}{L} }
\label{sec:knot-algebra}
We recall the \emph{knot surgery algebra} $\cK$ from the introduction. We define $\cK$ to be an algebra over the idempotent ring
\[
\ve{I}=\ve{I}_0\oplus \ve{I}_1,
\]
where each of $\ve{I}_i$ is rank 1 over $\bF=\Z/2$. We write $i_0$ and $i_1$ for the generators of $\ve{I}_0$ and $\ve{I}_1$, respectively.  We set
\[
\ve{I}_0\cdot \cK \cdot \ve{I}_0=\bF[\scU,\scV]\quad \text{and} \quad \ve{I}_1\cdot \cK\cdot \ve{I}_1\iso \bF[\scU,\scV,\scV^{-1}].
\]
Also
\[
\ve{I}_0\cdot \cK\cdot \ve{I}_1=0.
\]
We define 
\[
\ve{I}_1\cdot \cK\cdot \ve{I}_0=\bF[\scU,\scV,\scV^{-1}]\otimes \langle \sigma,\tau\rangle.
\]
That is, elements of $\ve{I}_1\cdot \cK\cdot \ve{I}_0$ can be written as sums of monomials of the form $\scU^i \scV^j \sigma$ and $\scU^i \scV^j \tau$ where $i\ge 0$ and $j\in \Z$. These algebra elements satisfy the relation
\[
\sigma\cdot a=\phi^\sigma(a)\cdot \sigma \quad \text{and} \quad
\tau\cdot a =  \phi^\tau(a)\cdot \tau,
\]
for $a\in \bF[\scU,\scV]=\ve{I}_0\cdot \cK\cdot \ve{I}_0$. Here $\phi^\tau\colon \bF[\scU,\scV]\to \bF[\scU,\scV,\scV^{-1}]$ is the algebra homomorphism satisfying $\phi^\tau(\scU)=\scV^{-1} $ and $\phi^\tau(\scV)=\scU \scV^2$. The map $\phi^\sigma\colon \bF[\scU,\scV]\to \bF[\scU,\scV,\scV^{-1}]$ is the canonical inclusion.

In particular, $\ve{I}_1\cdot \cK\cdot \ve{I}_0$ is generated by two special algebra elements $\sigma$ and $\tau$, together with the left action of $\bF[\scU,\scV,\scV^{-1}]$, which satisfy the relations
\[
\sigma \scU=\scU \sigma\quad \text{and} \quad \sigma \scV=\scV \sigma
\]
\[
\tau  \scU=\scV^{-1} \tau\quad \text{and} \quad \tau \scV=\scU\scV^2 \tau.
\]
The link algebra $\cL_\ell$ is defined as
\[
\cL_\ell:= \cK\otimes_{\bF}\cdots \otimes_{\bF} \cK.
\]
We often write just $\cL$, when $\ell$ is determined by context. We view $\cL_{\ell}$ as being an algebra over the idempotent ring
\[
\ve{E}_\ell:=\ve{I}\otimes_{\bF}\cdots \otimes_{\bF} \ve{I}.
\]

\begin{rem}
\label{rem:symmetry}
 The algebra $\cK$ admits a more symmetric basis, as follows. We may view $\ve{I}_1\cdot \cK \cdot \ve{I}_1$ as being $\bF[U,T,T^{-1}]$, where $U=\scU\scV$ and $T=\scV$. Then the above relations  become
\[
\sigma \scU=UT^{-1} \sigma\quad \text{and} \quad \sigma \scV=T \sigma
\]
\[
\tau \scU=T^{-1}\tau\quad \text{and} \quad \tau  \scV=U T \tau.
\]
\end{rem}
\begin{rem}
The above symmetry may be packaged into an algebra homomorphism
\[
\cE\colon \cK\to \cK,
\]
as follows. On $\ve{I}_0\cdot \cK\cdot \ve{I}_0$, we set
\[
\cE(\scU)=\scV\quad \text{and} \quad \cE(\scV)=\scU.
\]
On $\ve{I}_1\cdot \cK\cdot \ve{I}_1$, we set
\[
\cE(U)=U\quad \text{and} \quad \cE(\scV^i)=\scV^{-i}.
\]
On $\ve{I}_1\cdot \cK\cdot \ve{I}_0$, we set
\[
\cE(\sigma)=\tau\quad \text{and} \quad \cE(\tau)=\sigma.
\]
\end{rem}

\subsection{Topologies on $\cK$ and $\cL$}
\label{sec:filtration-K}

In this section, we describe the topologies on $\cK$ and $\cL$ which we use throughout the paper.

\begin{define}
\label{def:knot-topology}
Suppose that $n\in \N$ is fixed. We define $J_n\subset \cK$ to be the $\bF$ span of following set of generators:
\begin{enumerate}
\item In $\ve{I}_0\cdot \cK\cdot \ve{I}_0$, the generators $\scU^i\scV^j$, for $i\ge n$ or $j\ge n$ (i.e.  $\max(i,j)\ge n$).
\item In $\ve{I}_1\cdot \cK\cdot \ve{I}_0$, the generators $\scU^i\scV^j \sigma$ for $i\ge n$ or $j\ge n$.
\item In $\ve{I}_1\cdot\cK\cdot \ve{I}_0$, the generators $\scU^i\scV^j\tau$ for 
$j\le 2i-n$ or $i\ge n$.
\item In $\ve{I}_1\cdot\cK\cdot \ve{I}_1$, the generators $\scU^i\scV^j$ where $i\ge n$.
\end{enumerate}
The subspace $J_n$ is illustrated in Figure~\ref{fig:26}.
\end{define}

\begin{figure}[ht]
\begingroup%
  \makeatletter%
  \providecommand\color[2][]{%
    \errmessage{(Inkscape) Color is used for the text in Inkscape, but the package 'color.sty' is not loaded}%
    \renewcommand\color[2][]{}%
  }%
  \providecommand\transparent[1]{%
    \errmessage{(Inkscape) Transparency is used (non-zero) for the text in Inkscape, but the package 'transparent.sty' is not loaded}%
    \renewcommand\transparent[1]{}%
  }%
  \providecommand\rotatebox[2]{#2}%
  \newcommand*\fsize{\dimexpr\f@size pt\relax}%
  \newcommand*\lineheight[1]{\fontsize{\fsize}{#1\fsize}\selectfont}%
  \ifx\svgwidth\undefined%
    \setlength{\unitlength}{287.25050657bp}%
    \ifx\svgscale\undefined%
      \relax%
    \else%
      \setlength{\unitlength}{\unitlength * \real{\svgscale}}%
    \fi%
  \else%
    \setlength{\unitlength}{\svgwidth}%
  \fi%
  \global\let\svgwidth\undefined%
  \global\let\svgscale\undefined%
  \makeatother%
  \begin{picture}(1,0.81139697)%
    \lineheight{1}%
    \setlength\tabcolsep{0pt}%
    \put(0,0){\includegraphics[width=\unitlength,page=1]{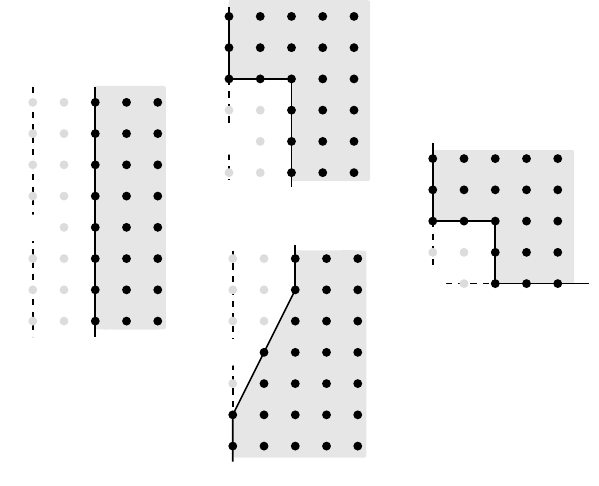}}%
    \put(0.84074571,0.2743729){\makebox(0,0)[t]{\lineheight{1.25}\smash{\begin{tabular}[t]{c}$\ve{I}_0\cdot J_n\cdot  \ve{I}_0$\end{tabular}}}}%
    \put(0.15914382,0.20604017){\makebox(0,0)[t]{\lineheight{1.25}\smash{\begin{tabular}[t]{c}$\ve{I}_1\cdot J_n \cdot \ve{I}_1$\end{tabular}}}}%
    \put(0.49691376,0.00306822){\makebox(0,0)[t]{\lineheight{1.25}\smash{\begin{tabular}[t]{c}$J_n^\tau$\end{tabular}}}}%
    \put(0.49691376,0.49777232){\makebox(0,0)[t]{\lineheight{1.25}\smash{\begin{tabular}[t]{c}\\$J_n^\sigma$\end{tabular}}}}%
    \put(0.72291402,0.3303176){\makebox(0,0)[t]{\lineheight{1.25}\smash{\begin{tabular}[t]{c}$1$\end{tabular}}}}%
    \put(0.38870066,0.2159748){\makebox(0,0)[t]{\lineheight{1.25}\smash{\begin{tabular}[t]{c}$\tau$\end{tabular}}}}%
    \put(0.38242115,0.56676279){\makebox(0,0)[t]{\lineheight{1.25}\smash{\begin{tabular}[t]{c}$\sigma$\end{tabular}}}}%
    \put(0.05513169,0.42209528){\makebox(0,0)[t]{\lineheight{1.25}\smash{\begin{tabular}[t]{c}$1$\end{tabular}}}}%
  \end{picture}%
\endgroup%

\caption{The subspace $J_n\subset \cK$, indicated by the shaded regions.  Here $J_n^\sigma$ and $J_n^\tau$ denote the subspaces of $J_n$ which are in the $\cK$-span of $\sigma$ and $\tau$. The dots indicate monomials in the algebra. The horizontal direction indicates the power of $\scU$, and the vertical direction indicates the power of $\scV$. }
\label{fig:26}
\end{figure}

We observe that
\[
J_0=\cK, \quad J_{n+1}\subset J_n,\quad \text{and} \quad \bigcap_{n\in \N} J_n=\{0\}.
\]
We will use the subspaces $J_n$ as a fundamental basis of $0$ in our topology on $\cK$. The condition that $\bigcap_{n\in \N} J_n=\{0\}$ is equivalent to the inclusion $\cK\to \ve{\cK}$ being injective.

\begin{rem} The subspaces $J_n^\sigma$ and $J_n^\tau$ can also be described as follows:
\[
\begin{split} J_n^\sigma&=(U^n)\cdot \sigma+\sigma\cdot (\scU^n,\scV^n)\\
J_n^\tau&=(U^n) \cdot \tau+\tau\cdot (\scU^n,\scV^n).
\end{split}
\]
Here $(U^n)\subset \ve{I}_1\cdot \cK\cdot \ve{I}_1$ denotes the ideal generated by $U=\scU\scV$, and $(\scU^n,\scV^n)\subset \ve{I}_0\cdot \cK\cdot \ve{I}_0$ denotes the ideal spanned by the elements $\scU^n$ and $\scV^n$. 
\end{rem}

We now consider the link algebra. For our purposes, the most natural topology on the link algebra is the $!$-topology:
\[
\cL_\ell=\cK\otimesshriek_{\bF}\cdots \otimesshriek_{\bF}\cK.
\]
This topology is first countable, with a basis of opens given by
\[
\cJ_{\ell,n}:=\sum_{i=0}^{\ell-1} \cK^{\otimes i} \otimes_{\bF} J_n\otimes_{\bF} \cK^{\otimes \ell-i-1}\subset \cL_{\ell}.
\]

A fundamental result to our paper is the following:

\begin{prop}
\label{prop:multiplication-continuous} The map $\mu_2\colon \cK\vecotimes_{\rmI} \cK\to \cK$ is continuous.
\end{prop}

\begin{rem}
\label{rem:continuity-and-K}
 Multiplication is not continuous on $\cK\otimesshriek_{\rmI} \cK$. As an example $\scV^{-i}\otimes \scV^i\sigma\to 0$ in $\cK\otimes_{\rmI} \cK$ as $i\to \infty$, whereas $\mu_2(\scV^{-i}\otimes \scV^i\sigma)=\sigma\not\to 0$. 
\end{rem}

 Proposition~\ref{prop:multiplication-continuous} follows immediately from the following lemma and the definition of $\vecotimes$ in Section~\ref{sec:tensor-prods}:

\begin{lem}\label{lem:simple-properties-K-Rn} Let $J_n\subset \cK$ be the subspace defined in Equation~\eqref{def:knot-topology}.
\begin{enumerate}
\item For all $n$, $\mu_2(J_n\otimes \cK)\subset J_n$ (i.e. $J_n$ is a right ideal of $\cK$).
\item Suppose $x\in \cK$ and $n\in \N$ are fixed. Then for sufficiently large $m$, 
\[
\mu_2( x \otimes J_m)\subset  J_n.
\]
\end{enumerate}
\end{lem}
\begin{proof} We consider the first claim. We wish to show that $\mu_2(x\otimes a)\in J_n$ whenever $x\in J_n$ and $a\in \cK$. We break the argument into cases, depending on the idempotents of $x$ and $a$. If $x,a\in \ve{I}_1\cdot \cK \cdot\ve{I}_1$, the claim follows because $w_U(x\cdot  a)\ge w_U(x)+w_U(a)$, where $w_U$ is the $U$-adic weight (i.e. $w_U(a)=\max\{i: \exists a',a=U^i\cdot a'\}$). The same argument applies if $x\in \ve{I}_1\cdot \cK\cdot \ve{I}_1$ and $a\in \ve{I}_1\cdot  \cK \cdot \ve{I}_0$. The situation that $x,a\in \ve{I}_0\cdot  \cK\cdot  \ve{I}_0$ is also clear. We now consider the case that $x\in \ve{I}_1\cdot \cK\cdot \ve{I}_0$ and $a\in \ve{I}_0\cdot \cK \cdot \ve{I}_0$. Suppose first that $x=\scU^i\scV^j \sigma$ and $a=\scU^s\scV^t$ (for $s,t\ge 0$). Then $x\cdot a=\scU^{i+s}\scV^{j+t}\sigma$, which is clearly also in $J_n$. Similarly, if $x=\scU^i\scV^j\tau$ and $a=\scU^s\scV^t$, then
\[
x\cdot a=\scU^{i+t} \scV^{j-s+2t}\tau.
\]
By assumption, either $i\ge n$ or $j\le 2i-n$. If $i\ge n$, then $i+t\ge n$ so $x\cdot a\in J_n$. The second case is equivalent to $2i-j-n\ge 0$. We note that 
\[
2(i+t)-(j-s+2t)-n=s+2i-j-n\ge 0
\]
so $x\cdot a\in J_n$. This completes the first part of the claim.

We leave the second claim to the reader, as it may be handled similarly.
\end{proof}

Combining Corollary~\ref{cor:re-order-arrow-shriek} and Proposition~\ref{prop:multiplication-continuous} we also obtain:
\begin{cor}\label{cor:link-continuous} Suppose $\cL$ is topologized as $\cK\otimes^!\cdots \otimes^! \cK$. Then the map
\[
\mu_2\colon \cL\vecotimes_{\rmE} \cL\to \cL
\]
is continuous.
\end{cor}

\subsection{Alexander modules}

In this section, we describe our categories of type-$D$, $A$ and $DA$ modules. We will refer to these categories as the categories of \emph{Alexander modules}. 

\begin{define}  A \emph{type-$D$ Alexander module} consists of a pair  $(\cX,\delta^1)$ where $\cX$ is a linear topological right $\ve{E}$-module and 
\[
\delta^1\colon \cX\to \cX\vecotimes_{\rmE} \cL
\]
is linear topological morphism which satisfies
\[
(\id\otimes \mu_2)\circ (\delta^1\otimes \id_{\cL})\circ \delta^1=0.
\]
\end{define}

Recall that a linear topological morphism is the same as a continuous linear map on completions.
If $\cX^\cL$ and $\cY^{\cL}$, we define $\Mor(\cX^{\cL}, \cY^{\cL})$ to be the space of linear topological morphisms
\[
f^1\colon \cX\to \cY\vecotimes_{\rmE} \cL.
\]
We write $\MOD^{\cK}_{\fra}$ for the category of Alexander type-$D$ modules.

Type-$A$ Alexander modules are similar:
\begin{define}  A  \emph{type-$A$ Alexander module} $(\cX,m_j)$ over $\cL$ consists of a linear topological left $\ve{E}$-module $\cX$ and a collection of linear topological morphisms
\[
m_{j+1}\colon \cL\vecotimes_{\rmE} \cdots \vecotimes_{\rmE} \cL\vecotimes_{\rmE} \cX\to \cX,
\]
satisfying the $A_\infty$-relations.
\end{define}

Type-$DA$ Alexander modules are defined by the obvious amalgamation of the above notions.

\subsection{Split Alexander modules}
\label{sec:split-modules}

We now define the notion of a \emph{split} Alexander type-$A$ or $DA$ module. Consider the link algebra $\cL_\ell$ and suppose that $\bP$ is a partition of $\{1,\dots, \ell\}$. Suppose that $\cX$ is a linear topological left $\ve{E}_\ell$-module. A \emph{$\bP$-split continuous} linear topological morphism
\[
f_{j+1}\colon \cL_\ell\otimes_{\rmE} \cdots \otimes_{\rmE} \cL_\ell\otimes_{\rmE} \cX\to \cX
\]
is a continuous map on completions, where we give $cL_{\ell}^{\otimes j}\otimes \cX$ the tree topology from Section~\ref{sec:tree-completions}, where $\cX$ is the central vertex and we have $|\bP|$-rays of length $j$ extending from this vertex. If $p_1,\dots, p_{|\bP|}$ are the sizes of the elements of $\bP$, then each non-root vertex of the $i$-th ray is labeled with $\cL_{p_i}$ (topologized using the $!$-topology).

Note that $\bP$-split continuity is preserved by $A_\infty$-composition of such maps by Proposition~\ref{prop:contraction}.

\begin{define}
Suppose that $\bP$ is a partition of $\{1,\dots, \ell\}$. A \emph{$\bP$-split type-$A$ Alexander  module} consists of a linear topological $\ve{E}=\ve{E}_\ell$-module $\cX$, equipped with a $\bP$-split linear topological morphism
\[
m_{j+1}\colon \cL_{\ell}\otimes_{\rmE}\cdots \otimes_{\rmE} \cL_{\ell}\otimes_{\rmE} \cX\to \cX,
\]
which satisfies the $A_\infty$-module relations.
\end{define}

If $\bP$ is a partition of $\{1,\dots, \ell\}$ and $i_1,\dots, i_n$ are the sizes of the partition elements (so $i_1+\cdots+i_n=\ell$), we will write
\[
{}_{\cL_{i_1}|\cdots|\cL_{i_n}} \cX
\]
for a $\bP$-split Alexander module.

Once we introduce box tensor products of split Alexander modules, we will show in Lemma~\ref{lem:split-identity-bimodule} that the split-Alexander condition is weaker than the Alexander condition. In particular, any type-$A$ Alexander module ${}_{\cL} \cX$ is also a split type-$A$ Alexander module for any partition.

$\bP$-Split Alexander type-$DA$ modules can be defined similarly to the above, though we typically need to assume that the underlying space $\cX$ of the module is linearly compact for the type-$DA$ structure relations to be meaningful, and for $A_\infty$-composition of morphisms to be defined. (Note that all of the link surgery modules we consider in this paper are linearly compact). We explain this in the case of a $\{ \{1\},\{2\}\}$-split Alexander module ${}_{\cK|\cK} \cX^{\cK}$. Consider the spaces involved in the composition of $\delta_{i+1}^1$, $\delta_{j+1}^1$ and $\bI_{\cX}\otimes \mu_2$, applied to 
\[
\underbrace{(\cK|\cK)\otimes_{\rmE_2} \cdots \otimes_{\rmE_2} (\cK|\cK)}_{i+j}\otimes_{\rmE_2} \cX.
\]
We topologize the above space using the tree topology for tree where there are $2$ length $i+j$ rays attached to a lowest vertex. The lowest vertex is assigned space $\cX$. Proposition~\ref{prop:contraction} shows that $\delta_{i+1}^1\otimes \id$ gives a map from the above space to
\[
\underbrace{(\cK|\cK)\otimes_{\rmE_2} \cdots \otimes_{\rmE_2} (\cK|\cK)}_{j}\otimes_{\rmE_2} (\cX\vecotimes \cK).
\]
The above space is topologized using a tree with a single lowest vertex, and 2 length $i$ rays extending from this ray. The lowest vertex is assigned the space $\cX\vecotimes \cK$. Lemma~\ref{lem:retopologize-bridge} implies that if $\cX$ is linearly compact, then the above topological vector space is isomorphic to the same tensor product, equipped with the tree topology where the tree $\Gamma$ has 2 length $i$ rays connected to a vertex $v_1$, which is above a single vertex $v_0$. We assign $v_1$ the vector space $\cX$, and $v_0$ the vector space $\cK$. See Figure~\ref{fig:47} for a schematic when $i=j=1$.

\begin{figure}[h]
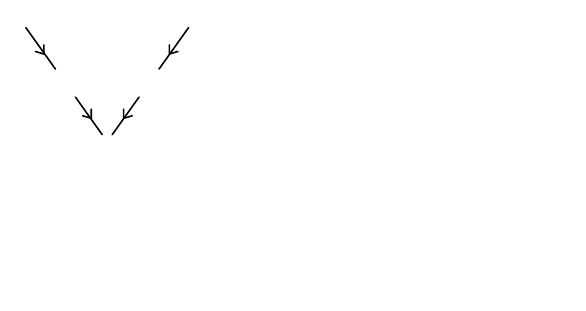
\caption{Maps appearing in the $DA$-structure relation for a split Alexander $DA$-bimodule. Here, the arrows labeled $\iso$ correspond to the isomorphism from Lemma~\ref{lem:retopologize-bridge}, which requires $\cX$ to be linearly compact.}
\label{fig:47}
\end{figure}

\subsection{External tensor products}

We now discuss external tensor products, focusing on subtleties involving completions. 

We first discuss type-$D$ modules. Suppose that $\cX^{\cL_m}$ and $\cY^{\cL_n}$ are type-$D$ Alexander modules. To form the external tensor products, we assume that $\cX$ and $\cY$ are linearly compact, so that $\cX\vecotimes \cL_m\iso \cX\otimes^! \cL_m$ by Lemma~\ref{lem:lin-comact-discrete-change-completion}, and similarly for $\cY\vecotimes \cL_n$. Under this assumption, we may define the structure map on the tensor product as $\delta^1_{\cX}\otimes \id_{\cY}\otimes 1_{\cL_n}+\id_{\cX}\otimes 1_{\cL_m}\otimes \delta^1_{\cY}$, with tensor factors reordered, as indicated below:
\[
\cX\otimes^! \cY\to (\cX\vecotimes \cL_m)\otimes^!(\cY\vecotimes \cL_n)\iso \cX\otimes^! \cL_m \otimes^! \cY\otimes^! \cL_n\iso (\cX\otimes^! \cY)\vecotimes (\cL_m\otimes^! \cL_n).
\]

Next, we consider tensoring (split) type-$A$ modules. Suppose that $\bP$ is a partition of $\{1,\dots, s\}$ and $\bP'$ is a partition of $\{1,\dots, r\}$. Let us write $i_1,\dots, i_m$ for the sizes of the elements of $\bP$, and $j_1,\dots, j_n$ for those of $\bP'$.  Suppose that we have $\bP$ and $\bP'$ split modules
\[
{}_{\cL_{i_1}|\cdots|\cL_{i_m}} \cX\quad \text{and} \quad {}_{\cL_{j_1}|\cdots|\cL_{j_n}} \cY.
\]
We will describe two versions of the external tensor product, namely
\[
{}_{\cL_{i_1}|\cdots|\cL_{i_m+j_1}|\cdots |\cL_{j_n}}(\cX\otimes^! \cY)\quad \text{and} \quad {}_{\cL_{i_1}|\cdots|\cL_{i_m}|\cL_{j_1}|\cdots|\cL_{j_n}} (\cX \otimes^! \cY).
\]
We topologize $\cL_{r+s}^{\otimes j} \otimes \cX\otimes \cY$ using the tree topology for a star with $m+n-1$ (resp. $m+n$ rays). We label the root $\cX\otimes^! \cY$, and we label each other vertex with a $\otimes^!$-tensor product of copies of $\cK$.

To construct the external tensor product, one must also pick a cellular diagonal of the associahedron. This construction is recalled in Section~\ref{sec:extension-of-scalars}. The fact that the external tensor product behaves well with completions is implied by the following result:

\begin{prop}\label{prop:external-tensor-product-type-A}
 Suppose that $\bP$ and $\bP'$ are partitions of $\{1,\dots, r\}$ and $\{1,\dots, s\}$, respectively, and that
\[
f_{j+1}\colon \cL_r^{\otimes j} \otimes \cX\to \cX\quad \text{and} \quad g_{j+1} \colon \cL_s^{\otimes j} \otimes \cY\to \cY
\]
are  $\bP$ and $\bP'$ split continuous, then their tensor product $f_{j+1}\otimes g_{j+1}$ is both $\bP\sqcup \bP'$ and $(\bP\sqcup \bP')_{p\sim p'}$ split continuous, for any $p\in \bP$ and $p'\in \bP'$. (Here we write $(\bP\sqcup \bP')_{p\sim p'}$ for the partition obtained by merging $p$ and $p'$ in $\bP\sqcup \bP'$).
\end{prop}

We will prove Proposition~\ref{prop:external-tensor-product-type-A} for the module ${}_{\cL_{i_1}|\cdots|\cL_{i_m+j_1}|\cdots |\cL_{j_n}}(\cX\otimes^! \cY)$. We will later see in Lemma~\ref{lem:split-identity-bimodule} that the module ${}_{\cL_{i_1}|\cdots|\cL_{i_m}|\cL_{j_1}|\cdots |\cL_{j_n}}(\cX\otimes^! \cY)$ is a box tensor product of this module with a split Alexander bimodule ${}_{\cL_{i_m}|\cL_{j_1}}[\bI]^{\cL_{i_m+j_1}}$, and hence is also a split Alexander module.

The proof Proposition~\ref{prop:external-tensor-product-type-A} is an immediate consequence of the following general result about tree topologies:

\begin{lem}
\label{lem:zipper-lemma}  Let $\Gamma_n^k$ denote the star which has $n$ rays of length $k$, which all point towards a root vertex $v_0$. Let $(\cX_v)_{v\in V(\Gamma_n^k)}$ and $(\cY_v)_{v\in V(\Gamma_m^k)}$ denote two families of spaces. Pick a distinguished ray from each of $\Gamma_{n}^k$ and $\Gamma_m^k$. Form a new collection of spaces
 $(\cZ_v)_{v\in \Gamma_{m+n-1}^k}$ as follows. View $\Gamma_{m+n-1}^k$ as being obtained by merging vertices of the same height along the distinguished rays.
 We define $\cZ_v$ to be either $\cX_v$ or $\cY_v$ for $v$ not along the merged ray. For $v$ along the merged ray, we define $\cZ_v$ to be $\cX_v\otimes^! \cY_v$. 
 Then the natural map
\[
\cZ_{\Gamma_{m+n-1}^k}\to \cX_{\Gamma_n^k}\otimes^! \cY_{\Gamma_m^k}
\]
is continuous. 
\end{lem}
\begin{figure}[H]
\begin{center}
\begingroup%
  \makeatletter%
  \providecommand\color[2][]{%
    \errmessage{(Inkscape) Color is used for the text in Inkscape, but the package 'color.sty' is not loaded}%
    \renewcommand\color[2][]{}%
  }%
  \providecommand\transparent[1]{%
    \errmessage{(Inkscape) Transparency is used (non-zero) for the text in Inkscape, but the package 'transparent.sty' is not loaded}%
    \renewcommand\transparent[1]{}%
  }%
  \providecommand\rotatebox[2]{#2}%
  \newcommand*\fsize{\dimexpr\f@size pt\relax}%
  \newcommand*\lineheight[1]{\fontsize{\fsize}{#1\fsize}\selectfont}%
  \ifx\svgwidth\undefined%
    \setlength{\unitlength}{311.28925259bp}%
    \ifx\svgscale\undefined%
      \relax%
    \else%
      \setlength{\unitlength}{\unitlength * \real{\svgscale}}%
    \fi%
  \else%
    \setlength{\unitlength}{\svgwidth}%
  \fi%
  \global\let\svgwidth\undefined%
  \global\let\svgscale\undefined%
  \makeatother%
  \begin{picture}(1,0.31428909)%
    \lineheight{1}%
    \setlength\tabcolsep{0pt}%
    \put(0,0){\includegraphics[width=\unitlength,page=1]{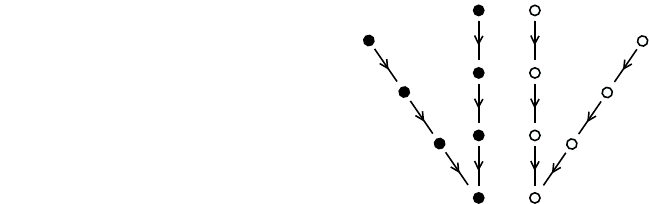}}%
    \put(0.78225966,0.13526189){\makebox(0,0)[t]{\lineheight{1.25}\smash{\begin{tabular}[t]{c}$\otimes^!$\end{tabular}}}}%
    \put(0,0){\includegraphics[width=\unitlength,page=2]{fig38.pdf}}%
    \put(0.21219853,0.00822698){\makebox(0,0)[t]{\lineheight{1.25}\smash{\begin{tabular}[t]{c}$\otimes^!$\end{tabular}}}}%
    \put(0,0){\includegraphics[width=\unitlength,page=3]{fig38.pdf}}%
    \put(0.21181373,0.10324444){\makebox(0,0)[t]{\lineheight{1.25}\smash{\begin{tabular}[t]{c}$\otimes^!$\end{tabular}}}}%
    \put(0,0){\includegraphics[width=\unitlength,page=4]{fig38.pdf}}%
    \put(0.21267362,0.19998182){\makebox(0,0)[t]{\lineheight{1.25}\smash{\begin{tabular}[t]{c}$\otimes^!$\end{tabular}}}}%
    \put(0,0){\includegraphics[width=\unitlength,page=5]{fig38.pdf}}%
    \put(0.21267362,0.29671908){\makebox(0,0)[t]{\lineheight{1.25}\smash{\begin{tabular}[t]{c}$\otimes^!$\end{tabular}}}}%
    \put(0,0){\includegraphics[width=\unitlength,page=6]{fig38.pdf}}%
  \end{picture}%
\endgroup%

\caption{A schematic of Lemma~\ref{lem:zipper-lemma}. Solid dots denote $\cX_v$ and open dots denote $\cY_v$.}
\end{center}
\label{fig:38}
\end{figure}
\begin{proof} We abbreviate $\cX_{\Gamma_n^k}$, $\cY_{\Gamma_m^k}$ and $\cZ_{\Gamma_{m+n-1}^k}$ by $\cX$, $\cY$ and $\cZ$, respectively.

 Let 
\[
U_{\cX}\otimes \cY+\cX\otimes V_{\cY}\subset \cX\otimes^! \cY
\]
 be an open subspace, where $U_{\cX}\subset \cX$ and $V_{\cY}\subset \cY$ are open subspaces. Let $s$ be an admissible labeling of $\Gamma_{n+m-1}^k$. Our goal is to show that if $\ve{p}$ is a point-enhancement of $s$, then there is an open enhancement $\cU_{\ve{p}}$ so that
 \[
 \cW(\cU_{\ve{p}},\ve{p})\subset U_{\cX}\otimes \cY+\cX\otimes V_{\cY}.
 \]
 
 There are two induced labelings $s_{\cX}$ and $s_{\cY}$ of $\Gamma_{n}^k$ and $\Gamma_{m}^{k}$, obtained by restricting $s$. We make the following claims:
\begin{enumerate}
\item At least one of $s_{\cX}$ and $s_{\cY}$ is admissible.
\item If there is an $\cO$-labeled vertex along the merged ray, then both $s_{\cX}$ and $s_{\cY}$ are admissible.
\end{enumerate}
The above claims are straightforward to prove, so we leave the argument to the reader.

Consider first the case that the merged ray has no vertices labeled $\cO$. In this case, the root is labeled $X$, but all other vertices along the merged ray are labeled $p$. Also, in this case, at least one of $s_{\cX}$ and $s_{\cY}$ is admissible. Suppose, without loss of generality, that $s_{\cX}$ is admissible. Write $v_1,\dots, v_{k-1}$ for the $p$-labeled vertices along the merged ray. If $v$ is a $p$-labeled vertex along the merged ray, write $\ve{p}(v)$ as a sum of finitely many elementary tensors. For each $v\in \{v_1,\dots, v_{k-1}\}$, write $x_{v,1},\dots, x_{v,i_v}$ for these elementary tensors.  We consider the $N:=i_{v_1}\dots i_{v_{k-1}}$ point enhancements $\ve{p}_{\cX,j}$, where $\ve{p}_{\cX,j}(w)=\ve{p}(w)$ if $w$ is not along the merged ray, and such that $\ve{p}_{\cX,j}$ enumerate all combinations of the $x_{v_i,r}$  for $i\in \{1,\dots, k-1\}$ and $r\in \{1,\dots, i_{v_i}\}$. 

Since $U_{\cX}$ is open, for each $\ve{p}_{\cX,i}$, there is an open enhancement $\cU_{\cX,i}\subset \cX$ so that 
\[
\cW(\cU_{\cX,i},\ve{p}_{\cX,i})\subset U_{\cX}.
\]
We define an open enhancement of $s$ as follows. If $v$ is in $\Gamma_n^k$ (the tree for $\cX$), we define
\[
\cU_{\ve{p}}(v)=\bigcap_{i=1}^{N} \cU_{\cX,i}(v).
\]
If $v$ is in $\Gamma_n^k$ (the tree for $\cY$), set $\cU_{\ve{p}}(v)=\cY_v$. Then
\[
\cW(\cU_{\ve{p}},\ve{p})\subset \cW(\cU_{\cX,1},\ve{p}_{\cX,1})\otimes \cY+\cdots+ \cW(\cU_{\cX,N}, \ve{p}_{\cX,N})\otimes \cY \subset U_{\cX}\otimes \cY+\cX\otimes V_{\cY},
\]
as we wanted to show.

Next, we consider the case that $s$ labels a vertex $v_0$ along the merged ray by $\cO$. In this case, both $s_{\cX}$ and $s_{\cY}$ are admissible. Consider a point enhancement $\ve{p}$ of $s$. Write $v_1,\dots, v_l$ for the vertices above $v_0$ (which are labeled $p$ by $s$). Similar to before, by separating the elements of $\ve{p}$ into elementary tensors, we obtain point enhancements $\ve{p}_{\cX,i}$ and $\ve{p}_{\cY,j}$ for $i=1,\dots, N$ and $j=1,\dots, M$. For each $i$ and $j$, there are open enhancements $\cU_{\cX,i}$ and $\cU_{\cY,j}$ of $s_{\cX}$ and $s_{\cY}$ (respectively) so that
\[
\cW(\cU_{\cX,i},\ve{p}_{\cX,i})\subset U_{\cX}\quad \text{and} \quad \cW(\cU_{\cY,j},\ve{p}_{\cX,j})\subset V_{\cY}. 
\]
We then define an open enhancement $\cU_{\ve{p}}$ of $s$ as follows.  Suppose $s(v)=\cO$. If $v$ is in $\Gamma_n^k$ not in the merged ray, we set
\[
\cU_{\ve{p}}(v)=\bigcap_{i=1}^N \cU_{\cX,i}(v).
\]
If $v$ is in $\Gamma_m^k$ but not in the merged ray, we set
\[
\cU_{\ve{p}}(v)=\bigcap_{j=1}^M \cU_{\cY,j}(v).
\]
If $v$ is in the merged ray, we set
\[
\cU_{\ve{p}}(v)=\left(\bigcap_{i=1}^N \cU_{\cX,i}(v)\right)\otimes \cY+\cX\otimes \left(\bigcap_{j=1}^M \cU_{\cY,j}(v)\right).
\]
We observe that
\[
\begin{split}
\cW(\cU_{\ve{p}},\ve{p})\subset &\left(\cW(\cU_{\cX,1},\ve{p}_{\cX,1})+\cdots+\cW(\cU_{\cX,N},\ve{p}_{\cX,N})\right)\otimes \cY\\
&+\cX\otimes \left(\cW(\cU_{\cY,1},\ve{p}_{\cY,1})+\cdots \cW(\cU_{\cY,M},\ve{p}_{\cY,M}) \right)\\
\subset& U_{\cX}\otimes \cY+\cX\otimes V_{\cY}. 
\end{split}
\]
This completes the proof.
\end{proof}

Taking the external tensor product of two type-$DA$ Alexander modules, when the underlying spaces of the modules are linearly compact, is a straightforward extension.

\begin{rem} We can see more directly that if  ${}_{\cL_i} \cX$ and ${}_{\cL_j} \cY$ are type-$A$ Alexander modules, then ${}_{\cL_{i+j}}(\cX\otimes^! \cY)$ is also an Alexander module. Consider two continuous maps
\[
f_{j+1}\colon \vec{\cT}^j \cL_n \vecotimes \cX\to \cX\quad \text{and} \quad g_{j+1}\colon \vec{\cT}^j \cL_m \vecotimes \cX\to \cX,
\]
where 
\[
\vec{\cT}^j \cL_{n}=\underbrace{\cL_n\vecotimes \cdots \vecotimes \cL_n}_n.
\]
 The external tensor product of $f_{j+1}$ and $g_{j+1}$ is continuous since it may be written as the following sequence of continuous maps:
\[
\begin{split} &\vec{\cT}^j (\cL_n\otimes^! \cL_m) \vecotimes (\cX\otimes^! \cY)\\
\to &\left(\vec{\cT}^j \cL_n \vecotimes \cX\right)\otimes^! \left(\vec{\cT}^j \cL_m\vecotimes \cY\right)\\
\to&\cX\otimes^! \cY.
\end{split}
\]
The first map is from Corollary~\ref{cor:re-order-arrow-shriek} and the second map is $f_{j+1}\otimes g_{j+1}$. 
\end{rem}

\subsection{Box tensor products}

In this section, we describe the interaction of Lipshitz, Ozsv\'{a}th and Thurston's box tensor product operation with Alexander modules. The simplest case is boxing an Alexander type-$D$ module $\cX^{\cL}$ with an Alexander type-$A$ module ${}_{\cL} \cY$. In this case, 
\[
\delta^1\colon \cX\to \cX\vecotimes_{\rmE} \cL \quad \text{and} \quad m_{j+1}\colon \cL\vecotimes_{\rmE} \cdots \vecotimes_{\rmE} \cL\vecotimes_{\rmE} \cY\to \cY,
\]
are continuous linear maps on the completions. Hence, assuming an  appropriate boundedness condition so that only finitely many terms contribute to the box tensor product (e.g. $m_j=0$ for $j\gg 0$), the differential $\d^\boxtimes$ will be a sum of finitely many continuous maps. Type-$DA$ modules present no additional complications.

Box tensor products involving split Alexander type-$A$ and $DA$ modules are more subtle. If ${}_{\cL_{i_1}|\cdots|\cL_{i_n}} \cY$ is a split Alexander type-$A$ module, it seems only possible to tensor with an Alexander type-$D$ module which is itself an external tensor product of $n$ modules 
\[
\cX_1^{\cL_{i_1}},\dots, \cX_n^{\cL_{i_n}}.
\]
In this case, we write
\[
(\cX^{\cL_{i_1}}_1,\dots,\cX^{\cL_{i_n}}_n)\hatbox {}_{\cL_{i_1}|\cdots|\cL_{i_n}} \cY
\]
for the resulting tensor product. Ignoring completions, the tensor product is formed by taking the external tensor product of the type-$D$ modules $\cX_j^{\cL_{i_j}}$ to obtain a type-$D$ module over $\cL_{i_1+\cdots+i_n}$. This type-$D$ module is then tensored with the type-$A$ module using the standard definition (see Equation~\eqref{eq:def-box-tensor}).
We topologize the underlying space of the tensor product 
\[
(\cX_1\otimes \cdots\otimes  \cX_n)\otimes \cY
\]
using the graph topology from Section~\ref{sec:tree-completions} for a tree with one root and $n$ rays. Here, $\cY$ is labeled as the root, and $\cX_1,\dots, \cX_n$ each form a ray, with an edge pointing to $\cY$. Proposition~\ref{prop:contraction} shows that the structure operations are continuous. 

We may also form the tensor product when $\cY$ or some of the $\cX_i$ are split $DA$ Alexander bimodules. The output will be a split Alexander $DA$-bimodule, whose tensor splitting of the incoming algebras is the concatenation of the tensor-splittings for each of the $\cX_i$. To illustrate, if ${}_{\cL_{i_1}|\cL_{i_2}} \cY^{\cL}$, ${}_{\cL_{j_1}|\cdots|\cL_{j_m}} \cX_1^{\cL_{i_1}}$ and ${}_{\cL_{k_1}|\cdots|\cL_{k_t}} \cX_2^{\cL_{i_2}}$ are split Alexander $DA$-bimodules, then their tensor product, $N$, will be a split module as indicated in the following equation:
\[
\left({}_{\cL_{j_1}|\cdots|\cL_{j_m}} \cX_1^{\cL_{i_1}} ,{}_{\cL_{k_1}|\cdots|\cL_{k_t}} \cX_2^{\cL_{i_2}}\right)\hatbox  {}_{\cL_{i_1}|\cL_{i_2}} \cY^{\cL}={}_{\cL_{j_1}|\cdots |\cL_{j_m}|\cL_{k_1}|\cdots |\cL_{k_t}} N^{\cL}.
\]

\begin{rem} The box tensor product operation of split Alexander modules seems definable using our techniques only when the underlying vector spaces of our modules are linearly compact. In this case, the underlying vector spaces of the tensor product $(\cX^{\cK}, \cY^{\cK})\hatbox{}_{\cK|\cK} \cZ$ is the standard tensor product $\cX\otimes^! \cY\otimes^! \cZ$ by Lemma~\ref{lem:lin-comact-discrete-change-completion}. Furthermore, assuming linear compactness, the completed box tensor product operation satisfies the following associativity property:
\[
(\cX^{\cK}, \cY^{\cK})\hatbox {}_{\cK|\cK} \cZ\iso \cX^{\cK}\hatbox {}_{\cK} (({}_{\cK} \bI^{\cK},\cY^{\cK}) \hatbox {}_{\cK|\cK} \cZ)\iso \cY^{\cK} \hatbox {}_{\cK} ((\cX^{\cK},{}_{\cK} [\bI]^{\cK})\hatbox{}_{\cK|\cK} \cZ). 
\]
Note that all of the bordered link surgery modules considered in this paper are countable direct products of $\bF$, so are linearly compact. 
\end{rem}

We now describe a split type-$DA$ Alexander bimodule ${}_{\cK|\cK} [\bI]^{\cL_2}$ which transforms any Alexander module ${}_{\cL_2} \cX$ into a split module ${}_{\cK|\cK} \cX$. The structure maps are the same as for the identity module.

\begin{lem}\label{lem:split-identity-bimodule}
 The bimodule ${}_{\cK|\cK} [\bI]^{\cL_2}$ is a split Alexander bimodule. More generally, if $\bP$ is a partition of $\{1,\dots, \ell\}$, and $\ell_1,\dots, \ell_j$ are the sizes of the partition elements, then ${}_{\cL_{\ell_1}|\cdots|\cL_{\ell_j}} [\bI]^{\cL_{\ell}}$ is a split Alexander bimodule. 
\end{lem}
\begin{proof} Consider first ${}_{\cK|\cK}[\bI]^{\cL_2}$. In this case, the claim follows from Remark~\ref{rem:small-star} and the fact that the natural map
\[
\cK\otimes^*_{\bF} \cK\to \cK\otimes^!_{\bF} \cK
\]
is continuous. The case of general $\ell$ and $\bP$ is proven by the same argument.
\end{proof}

\begin{rem} Lemma~\ref{lem:split-identity-bimodule} implies that the ordinary Alexander module condition is stronger than the split Alexander module condition. In particular, given a type-$A$ Alexander module ${}_{\cL_{\ell}} \cX$, we may always view $\cX$ as a split Alexander module for any partition of $\{1,\dots, \ell\}$. 
\end{rem}

\subsection{Finitely generated $\cK$-modules}

In this section, we discuss the category of finitely generated type-$D$ modules over $\cL$.

We denote by $\ve{\cK}$ the completion of the algebra $\cK$ with respect to the topology described above. As a vector space, we have the following isomorphisms:
\[
\begin{split}
\ve{I}_0\cdot \ve{\cK} \cdot \ve{I}_0&\iso \bF\llsquare \scU,\scV\rrsquare \\
 \ve{I}_1\cdot \ve{\cK}\cdot \ve{I}_0&\iso \bF[\scV,\scV^{-1}\rrsquare  \llsquare \scU\rrsquare  \langle \tau\rangle \oplus \bF\llsquare \scV,\scV^{-1}] \llsquare \scU\rrsquare  \langle \sigma \rangle\\
\ve{I}_1\cdot \ve{\cK} \cdot \ve{I}_1&\iso \bF[\scV,\scV^{-1}] \llsquare \scU\rrsquare .
\end{split}
\]
We may consider the completion of the link algebra $\ve{\cL}$ as well.

Since multiplication is continuous by Proposition~\ref{prop:multiplication-continuous} and Corollary~\ref{cor:link-continuous}, we obtain a well-defined map on the completed algebras
\[
\mu_2\colon \ve{\cL}\vecotimes_{\rmE} \ve{\cL}\to \ve{\cL}
\]

Since $\ve{\cL}$ is an algebra, we may also consider the ordinary category of  type-$D$ modules over this algebra (i.e. ordinary modules and algebras, with no topologies). Of particular interest to us is the category of finitely generated type-$D$ modules.
 We write $\MOD^{\ve{\cL}}_{\fg}$ for this category.

 There is a related category, $\MOD^{\cL}_{\fg,\fra}$, consisting of Alexander  type-$D$ modules over $\cL$ which are finitely generated over $\bF$.

\begin{prop}
 The categories $\MOD^{\cL}_{\fg,\fra}$ and $\MOD^{\ve{\cL}}_{\fg}$ are equivalent.
\end{prop}
\begin{proof} We define functors in both directions.  The functor  
 \[
 F\colon\MOD_{\fg, \fra}^{\cL}\to \MOD_{\fg}^{\ve{\cL}}
 \]
 is obtained by taking completions, noting that the completion of $\cX\vecotimes \cL$ coincides with the ordinary tensor product $\ve{\cX}\otimes \ve{\cL}$ when $\cX$ is finitely generated. (Note that when $\cX$ is finite dimensional, $\cX\iso \ve{\cX}$ if and only if $\cX$ is Hausdorff).

 The functor 
  \[
  G\colon \MOD_{\fg}^{\ve{\cL}}\to  \MOD_{\fg, \fra}^{\cL}
  \] 
  sends a type-$D$ module $(\cX, \delta^1)$ to itself, where $\cX$ is equipped with the discrete topology.
   This is well-defined due to the fact that if $\cX$ and $\cY$ are equipped with the discrete topology, then a map
 \[
 f^1\colon \cX\to \cY\vecotimes \ve{\cL}
 \]
 is automatically continuous, since $\{0\}\subset \cX$ is open. We observe that $F\circ G$ is the identity functor.
 
 On the other hand, $G\circ F$ maps $(\cX,\delta^1)$ to $(\ve{\cX}, \delta^1)$. Note that in $\Mod_{fg, \fra}^{\cL}$, the objects $(\cX,\delta^1)$ to $(\ve{\cX}, \delta^1)$ are canonically isomorphic, since morphisms are defined to be continuous linear maps on  completions.
 See Definition~\ref{def:linear-topological-category} and Remark~\ref{rem:canonically-isomorphic-X-veX}. In particular, there is a natural transformation from $G\circ F$ to the identity. The proof is complete.
\end{proof}

\section{Link surgery modules over $\cK$ and $\cL$}

In this section, we describe some basic link surgery modules over the algebras $\cK$ and $\cL$. We begin by defining the type-$D$ and type-$A$ modules for a solid torus and the algebraic merge module. Subsequently, we interpret the knot and link surgery formulas as type-$D$ modules over $\cK$ and $\cL$.

 \subsection{The type-$D$ module for a solid torus}
 \label{sec:solid-torus-module}
 If $\lambda$ is an integer, we refer to the complement of a $\lambda$-framed unknot in $S^3$ as the \emph{$\lambda$-framed solid torus}. We will define a type-$D$ module $\cD_{\lambda}^{\cK}$ for the $\lambda$-framed solid torus as follows.  We set
 \[
 \cD_\lambda^{\cK} \cdot \ve{I}_0=\langle \xs^0\rangle\quad \text{and} \quad   \cD_\lambda^{\cK} \cdot \ve{I}_1=\langle \xs^1\rangle,
 \]
 where $\langle \xs^\veps\rangle=\bF$, spanned by a generator $\xs^\veps$. Here, $\veps\in \{0,1\}$ denotes the idempotent. We define the structure map via the formula
 \[
 \delta^1(\xs^0)=\xs^1\otimes (\sigma+\scV^n \tau).
 \]

\subsection{The type-$A$ module for a solid torus}
\label{sec:framing}
Suppose that $\lambda$ is an integer. We now define the type-$A$ module 
${}_{\cK}\cD_\lambda$ for the solid torus as follows.  We set
\[
\ve{I}_0\cdot \cD_\lambda=\bF[\scU,\scV]\quad \text{and} \quad \ve{I}_1\cdot \cD_\lambda=\bF[\scU,\scV,\scV^{-1}].
\]
We define the type-$A$ structure map $m_j$ on $\cD_\lambda$ to be 0 unless $j=2$. We define $m_2$ on $\cD_\lambda$ as follows. If $f\in \ve{I}_i\cdot \cK\cdot \ve{I}_i$ and $x\in \ve{I}_j\cdot \cD_j$, then we define $m_2(f,x)$ to be $f\cdot x$ (ordinary multiplication of polynomials) if $i=j$ and to be 0 otherwise.  If $x\in \ve{I}_0\cdot \cD_{\lambda}$, we define
\[
m_2(\sigma,x)= \phi^\sigma(x)\in \bF[\scU,\scV,\scV^{-1}]=\ve{I}_1\cdot \cD_\lambda,
\]
where $\phi^\sigma$ is the canonical inclusion of localization.  Similarly, we define
\[
m_2(\tau,x)=\scV^\lambda \cdot \phi^\tau(x)\in \bF[\scU,\scV,\scV^{-1}],
\]
where $\cdot$ denotes ordinary multiplication of polynomials.

We view $\cD_{\lambda}$ as being the direct sum of the 1-dimensional spans of monomials, and give $\cD_{\lambda}$ the product topology with respect to this decomposition. See Definition~\ref{def:cofinite-basis}. With respect to this topology, the completion is given by
\[
\ve{I}_0\cdot \ve{\cD}_\lambda=\bF\llsquare \scU,\scV\rrsquare ,\quad \text{and} \quad  \ve{I}_1\cdot \ve{\cD}_\lambda=\bF\llsquare \scU,\scV,\scV^{-1}\rrsquare .
\]
(Note that the latter object is defined only as a vector space, and not as a ring, since we are taking completions with respect to $\scV$ and $\scV^{-1}$).

\begin{lem}\label{lem:continuity-framing} The map
\[
m_2\colon \cK\vecotimes_{\rmI} \cD_{\lambda}\to \cD_{\lambda}
\]
is continuous.
\end{lem}

\begin{proof}Write $\cD$ for $\cD_{\lambda}$. Continuity amounts to two claims. The first is that if $E\subset  \cD$ is an open subspace, then there is some $n$ so that $m_2(J_n\otimes \cD)\subset E$. Additionally, we need to show that if $E$ is as above and $a\in \cK$, then there is some open $V_a\subset \cD$ so that $m_2(a\otimes V_a)\subset E$. By definition, we may assume that $E$ is $\cD_{\co(S)}$ for some finite set of basis elements in $\cD$. (Recall that $\cD_{\co(S)}$ is the span of the basis elements not in $S$).

For the first claim, it suffices to show that if $n$ is sufficiently large, then $J_n\otimes \cD$ is mapped into $\cD_{\co(S)}$. This may be proven by considering each idempotent of $\cK$ separately. To ensure $(\ve{I}_0\cdot J_n\cdot \ve{I}_0)\otimes \cD$ is sent into $\cD_{\co(S)}$, it is sufficient to pick $n$ larger than $\max(p,q)$ ranging over all monomials $\scU^p\scV^q$ appearing in $\ve{I}_0\cdot S$.  For $(\ve{I}_1\cdot J_n\cdot \ve{I}_1)\otimes \cD$ to be sent into $\cD_{\co(S)}$, we pick $n$ larger than $p$ for each monomial $\scU^p\scV^q$ appearing  in $\ve{I}_1\cdot S$. Finally, to ensure that $(\ve{I}_1\cdot J_n \cdot\ve{I}_0)\otimes \cD$ is sent into $\cD_{\co(S)}$, it is sufficient to pick $n$ larger than $\max(p,q, 2p-q)$, ranging over monomials $\scU^p\scV^q$ appearing in a summand of $\ve{I}_1\cdot S$. To see that this is sufficient, note that if $\scU^i\scV^j\tau\otimes \scU^s\scV^t\in J_n\otimes \cD$, then $\max(i, 2i-j)\ge n$. We observe that $\scU^i\scV^j\tau \scU^s\scV^t=\scU^{i+t}\scV^{j-s+2t}\tau$ lies in $\cD_{\co(S)}$, since 
\[
\max(i+t, 2(i+t)-(j-s+2t))=\max(i+t,2i-j+s)\ge \max(i,2i-j) \ge n.
\]
The fact that $\scU^i \scV^j \sigma\otimes \scU^s \scV^t\in J_n\otimes \cD$ is sent into $\cD_{\co(s)}$ is similar. 

Next, we need to show that if $a\in \cK$ is fixed, then there is a finite subset $T$ of generators of $\cD$ so that $a\otimes \cD_{\co(T)}$ is mapped into $\cD_{\co(S)}$. Without loss of generality, we assume that $a$ is a monomial. We observe the following basic fact: if $a\in \ve{I}_{\veps'}\cdot \cK\cdot \ve{I}_{\veps}$ and $y\in \ve{I}_{\veps'}\cdot \cD$ are non-zero monomials, then there is at most 1 (in particular finitely many) $x\in \ve{I}_{\veps}\cdot \cD$ such that $m_2(a,x)=y$. In particular, there exists a finite set of generators $T$ such that $a\otimes \cD_{\co(T)}$ is sent into $\cD_{\co(S)}$, completing the proof.
\end{proof}

\begin{rem}\label{rem:m2-not-continuous}
 Note that $m_2$ is \emph{not} continuous as a map from $\cK\otimes^! \cD$ to $\cD$. Compare Remark~\ref{rem:continuity-and-K}. In fact, $m_2$ does not induce a map from the completion of $\cK\otimes^! \cD_\lambda$ to $\ve{\cD}_{\lambda}$. For example, in the completion of $\cK\otimes^! \cD_{\lambda}$ the sum $\sum_{i=0}^\infty \scV^i \otimes \scV^{-i} i_1$ converges since $\scV^{-i}i_1\to 0$ in the topology of $\cD_{\lambda}$, while $m_2$ cannot be defined on such infinite sums. 
\end{rem}

The module for a solid torus extends to a type-$AA$ bimodule:
\[
{}_{\cK} [\cD_{\lambda}]_{\bF[U]}.
\]
Here, we have $U$ act on the right via polynomial multiplication with $\scU\scV$ (summed over both idempotents).
 Since $[\scU\scV, \tau]=0$ and $[\scU\scV,\sigma]=0$, the construction above defines an $AA$-bimodule structure on $\cD_\lambda$. Note that only $m_{1,1,0}$ and $m_{0,1,1}$ are non-trivial.

 Additionally, we can also think of the module for a solid torus as corresponding instead to a $DA$-bimodule ${}_{\cK} [\cD_{\lambda}]^{\bF[U]}$, in such a way that
 \[
 {}_{\cK} [\cD_{\lambda}]^{\bF[U]}\hatbox {}_{\bF[U]}\bF[U]_{\bF[U]}\iso {}_{\cK} [\cD_{\lambda}]_{\bF[U]}.
 \] 
 For this description, the generators of $\ve{I}_0\cdot  {}_{\cK} [\cD_{\lambda}]^{\bF[U]}$ are $\scU^i$ and $\scV^j$ for $i,j\ge 0$. The generators of $\ve{I}_1\cdot  {}_{\cK} [\cD_{\lambda}]^{\bF[U]}$ are $\scV^i$ for $i\in \Z$. This module has $\delta_j^1=0$ unless $j=2$. We illustrate the structure map $\delta_2^1$ of ${}_{\cK}[\cD_{0}]^{\bF[U]}$ in Figure~\ref{fig:solid-torus-DA} and leave the structure relations to the reader for other framings.
 
 \begin{figure}[ht]
 \[
 \begin{tikzcd}[labels=description,column sep=1.3cm, row sep=1.5cm]
 \cdots
 &[-1cm]
 \scU^2
 	\ar[r, bend left=20, "\scV|U"]
 	\ar[from=r, bend left=20, "\scU|1"]
 	\ar[d,bend left,"\sigma|U^2",pos=.6]
 	\ar[d, bend right,"\tau|1", pos=.3]
 &
 \scU
 	\ar[r, bend left=20, "\scV|U"]
 	\ar[from=r, bend left=20, "\scU|1"]
 	\ar[d,bend left,"\sigma|U",pos=.6]
 	\ar[d, bend right,"\tau|1", pos=.3]
 &
 1
 	\ar[r, bend left=20, "\scV|1"]
 	\ar[from=r, bend left=20, "\scU|U"]
 	\ar[d,bend left,"\sigma|1",pos=.6]
 	\ar[d, bend right,"\tau|1", pos=.3]
 &
 \scV
 	\ar[r, bend left=20, "\scV|1"]
 	\ar[from=r, bend left=20, "\scU|U"]
 	\ar[d,bend left,"\sigma|1",pos=.6]
 	\ar[d, bend right,"\tau|U", pos=.3]
 &
  \scV^2
 	\ar[d,bend left,"\sigma|1",pos=.6]
 	\ar[d, bend right,"\tau| U^2", pos=.3]	
 &[-1cm]
 \cdots
 \\
 \cdots
 &
 \scV^{-2}
 	\ar[r, bend left=20, "\scV|1"]
 	\ar[from=r, bend left=20, "\scU|U"]
 &
 \scV^{-1}
 	\ar[r, bend left=20, "\scV|1"]
 	\ar[from=r, bend left=20, "\scU|U"]
 &
 1
 	\ar[r, bend left=20, "\scV|1"]
 	\ar[from=r, bend left=20, "\scU|U"]
 &
  \scV
  	\ar[r, bend left=20, "\scV|1"]
  	\ar[from=r, bend left=20, "\scU|U"]
 &
 \scV^2
 &
 \cdots
 \end{tikzcd}
 \]
 \caption{The $DA$-bimodule ${}_{\cK} \cD_0^{\bF[U]}$. An arrow decorated with $a|b$ from $\xs$ to $\ys$ means that $\delta_2^1(a,\xs)=\ys\otimes b$.}
 \label{fig:solid-torus-DA}
 \end{figure}

 \subsection{The merge bimodule}
 
 \label{sec:merge}
 
 In this section, we define a bimodule ${}_{\cK|\cK}M^{\cK}$, which we call the \emph{merge bimodule}.

 Whenever $R$ is a commutative algebra over $\bF$, there is a morphism of algebras $\phi\colon  R\otimes_{\bF} R\to R$, given by $\phi(a\otimes b)=ab$. A ring homomorphism $\psi\colon A\to B$ always determines a bimodule ${}_{A} [\psi]^B$ whose underlying vector space is the ground ring, and has structure map $\delta_1^1=0$ and $\delta_2^1(a\otimes 1)=1\otimes \psi(a)$. In particular, there is a natural bimodule ${}_{R\otimes R} [\phi]^{R}$ for any commutative ring $R$.
 
  The algebra $\cK$ is not commutative, though the bimodule ${}_{\cK|\cK} M^{\cK}$ is somewhat analogous.
 
 We now describe the bimodule $M$. To make the notation clear, in this section we will write $|$ for $\otimes_{\bF}$ and $\otimes$ for $\otimes_{\rmI}$.
 
 As a vector space, $M$ is $\ve{I}=\ve{I}_0\oplus \ve{I}_1$. There is a left action of $ \ve{I}| \ve{I}$, given by
 \[
 (i_1|  i_2)\cdot i=i_1i_2\cdot i.
 \]
  The right $\ve{I}$-action is the standard action.
  
On $M$, only $\delta_2^1$ and $\delta_3^1$ are non-trivial. The map $\delta_2^1$ is defined as follows. Suppose that $a_1|a_2$ is an elementary tensor in either
  \[
  (\ve{I}_0|\ve{I}_0)\cdot (\cK|\cK) \cdot (\ve{I}_0|\ve{I}_0)  \quad \text{or} \quad  (\ve{I}_1|\ve{I}_1)\cdot (\cK|\cK) \cdot  (\ve{I}_1|\ve{I}_1).
  \]
In this case, we set
  \[
  \delta_2^1(a_1|a_2 \otimes i)=i\otimes a_1a_2.
  \]
  On any other elementary tensor, we set $\delta_2^1$ to vanish.
 
 We define $\delta_3^1$ as follows. We set
 \[
 \delta_3^1(1| \sigma,\sigma| 1\otimes i_0)=i_1\otimes \sigma,\quad \text{and} \quad \delta_3^1(1| \tau,\tau| 1\otimes i_0)=i_1\otimes \tau.
 \]
 More generally, if $a,b,c,d$ are monomials concentrated in single idempotents, then we set
 \[
 \delta_3^1( a | b\sigma, c\sigma| d, i_0)=i_0\otimes abc\sigma d
 \]
 and similarly for the $\tau$ terms. We set
 \[
 \delta_3^1(\sigma| 1,1| \sigma\otimes i_0)=0,
 \]
 and similarly if $\tau$ replaces $\sigma$. The $DA$-bimodule relations are straightforward to verify.

   \begin{lem}\label{lem:merge-continuous} The merge module ${}_{\cK|\cK} M^\cK$ is a split Alexander module (using the discrete partition on the two incoming algebra factors). 
   \end{lem}

   \begin{rem} The structure maps of the merge module are \emph{not} continuous when we view the bimodule as a non-split module over $\cK\otimes^!\cK$ or $\cK\otimes^* \cK$, i.e. as  ${}_{\cK\otimes^! \cK} M^{\cK}$ or ${}_{\cK\otimes^* \cK} M^{\cK}$. To see this, observe that discontinuity with respect to the topology on $\cK\otimes_{\bF}^*\cK $ implies discontinuity with respect to $\cK\otimes^!_{\bF}\cK$ since there is a continuous map $\cK\otimes^*_{\bF} \cK\to \cK\otimes_{\bF}^! \cK$. Note that a sequence of tensors 
  \[
  \ve{x}_n=(x_n| y_n)\otimes (z_n| w_n)\in (\cK\otimes_{\bF}^* \cK)\vecotimes_{\rmE} ( \cK\otimes_{\bF}^* \cK)
  \]
    will converge to 0 if $x_n$ is constant in $n$ and $y_n\to 0$. In particular, we need not make any assumptions about $z_n$ or $w_n$. We note that $\scV^n\sigma\to 0$ in $\cK$, as each of the ideals $J_m$ (which form a basis of open sets  centered at $0$ for our topology on $\cK$) contains $\scV^n \sigma$ for all sufficiently large $n$. Therefore, 
   \[
   (1|\scV^n \sigma)\otimes (\scV^{-n} \sigma| 1)\to 0.
   \]
 On the other hand,
   \[
   \delta_3^1\left(1| \scV^n \sigma,\scV^{-n} \sigma| 1, i_0\right)=i_1\otimes \sigma\not\to 0,
   \]
   so $\delta_3^1$ is not continuous.
   \end{rem}

   \begin{proof}[Proof of Lemma~\ref{lem:merge-continuous}] We consider first $\delta_2^1$. In this case, the only non-trivial actions on $i_0$ or $i_1$ are from pairs $a| b$ where $a$ and $b$ are both in $\ve{I}_0\cdot \cK\cdot \ve{I}_0$ or both in $\ve{I}_1\cdot\cK\cdot \ve{I}_1$. In this case, the claim is clear.

We now consider $\delta_3^1$. Recall that verifying the split Alexander condition amounts to showing that $\delta_3^1$ is continuous if we use the tree topology from Section~\ref{sec:tree-completions} using the following tree: 

\begin{center}
\begingroup%
  \makeatletter%
  \providecommand\color[2][]{%
    \errmessage{(Inkscape) Color is used for the text in Inkscape, but the package 'color.sty' is not loaded}%
    \renewcommand\color[2][]{}%
  }%
  \providecommand\transparent[1]{%
    \errmessage{(Inkscape) Transparency is used (non-zero) for the text in Inkscape, but the package 'transparent.sty' is not loaded}%
    \renewcommand\transparent[1]{}%
  }%
  \providecommand\rotatebox[2]{#2}%
  \newcommand*\fsize{\dimexpr\f@size pt\relax}%
  \newcommand*\lineheight[1]{\fontsize{\fsize}{#1\fsize}\selectfont}%
  \ifx\svgwidth\undefined%
    \setlength{\unitlength}{96.5411208bp}%
    \ifx\svgscale\undefined%
      \relax%
    \else%
      \setlength{\unitlength}{\unitlength * \real{\svgscale}}%
    \fi%
  \else%
    \setlength{\unitlength}{\svgwidth}%
  \fi%
  \global\let\svgwidth\undefined%
  \global\let\svgscale\undefined%
  \makeatother%
  \begin{picture}(1,0.29742232)%
    \lineheight{1}%
    \setlength\tabcolsep{0pt}%
    \put(0.51968432,0.25260218){\makebox(0,0)[t]{\lineheight{1.25}\smash{\begin{tabular}[t]{c}$\cK$\end{tabular}}}}%
    \put(0.05775951,0.25260218){\makebox(0,0)[t]{\lineheight{1.25}\smash{\begin{tabular}[t]{c}$\cK$\end{tabular}}}}%
    \put(0.05775951,0.00730339){\makebox(0,0)[t]{\lineheight{1.25}\smash{\begin{tabular}[t]{c}$\cK$\end{tabular}}}}%
    \put(0.51968432,0.00730339){\makebox(0,0)[t]{\lineheight{1.25}\smash{\begin{tabular}[t]{c}$\cK$\end{tabular}}}}%
    \put(0,0){\includegraphics[width=\unitlength,page=1]{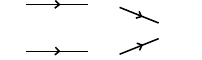}}%
    \put(0.81249908,0.11988534){\makebox(0,0)[lt]{\lineheight{1.25}\smash{\begin{tabular}[t]{l}$\ve{I}$\end{tabular}}}}%
  \end{picture}%
\endgroup%

\end{center}

 We define a function
   \[
   \tilde{\mu}\colon \cK\otimes_{\bF}^* \cK\to \cK
   \]
      given on elementary tensors of monomials by
      \[
      \tilde{\mu}(a \sigma, b\sigma)=ab\sigma\quad \text{and} \quad \tilde{\mu}(a\tau, b\tau)=ab\tau,
      \]
      with $\tilde{\mu}$ vanishing on other elementary tensors of monomials. Additionally, we consider a map
      \[
      \Pi\colon (\cK\otimes_{\rmI}\cK)\otimes_{\bF} (\cK\otimes_{\rmI}\cK)\to (\cK\otimes_{\rmI}\cK)\otimes_{\bF} (\cK\otimes_{\rmI}\cK)
      \]
      such that $\Pi$ is the identity on the elementary tensors of monomials
      \[
      (a\sigma\otimes a')| (b\otimes b'\sigma)\quad \text{and} \quad (a\tau\otimes a')| (b\otimes b'\tau),
      \]
      and such that $\Pi$ vanishes on other elementary tensors of monomials.
      
      We observe that
      \[
      \delta_3^1(\ve{a},1)=1\otimes (\tilde{\mu}\circ (\mu_2| \mu_2)\circ \Pi)(\ve{a}).
      \]
      The map $\Pi$ is clearly continuous with respect to the tree topology, and furthermore, $\mu_2|\mu_2$ is continuous by Propositions~\ref{prop:multiplication-continuous} and ~\ref{prop:contraction}. Note that if we topologize $(\cK|\cK)\otimes_{\rmI|\rmI} \ve{I}$ with the tree topology, it is isomorphic to $\cK\otimes^*_{\bF} \cK$. (See Remark~\ref{rem:small-star}). Hence, it suffices to show that $\tilde{\mu}$ is continuous. This amounts to the following:
      \begin{enumerate}[label=($m$-\arabic*), ref=$m$-\arabic*]
      \item\label{eq:tilde-mu-1}For each $n\in \N$, there are $k,m\in \N$ so that
      \[
      \tilde{\mu}(J_k\otimes J_m)\subset J_n.
      \]
      \item \label{eq:tilde-mu-2} For each $n\in \N$ and $x\in \cK$, there is an $m\in \N$ so that
      \[
      \tilde{\mu}(x\otimes J_{m})\subset J_n\quad \text{and} \quad \tilde{\mu}(J_{m}\otimes x)\subset J_n.
      \]
      \end{enumerate}
      Both of the arguments are elementary, and we illustrate the idea. Recall that by out construction in Definition~\ref{def:knot-topology}, $\scU^i\scV^j \sigma\in J_n$ if and only if $i \ge 0$ and $\max(i,j)\ge n$. If $\scU^i\scV^j\sigma$ and $\scU^s \scV^t\sigma$ are two monomials in $J_n$, then at least one of two situations occurs: both $j,t\ge n$; or one of $i$ and $s$ is at least $n$. If $j,t\ge n$, then $\scU^{i+s} \scV^{j+t}\sigma\in J_n$ since $j+t\ge 2n$. If one of $i$ and $s$ is at least $n$, then $i+s\ge n$. A symmetric analysis holds for the multiples of $\tau$, and we see that
      \[
      \tilde{\mu}(J_n\otimes J_n)\subset J_n,
      \]
      so \eqref{eq:tilde-mu-1} holds. A similar and elementary argument implies ~\eqref{eq:tilde-mu-2}.
         \end{proof}

There is an asymmetry in the merge module. Recall that we set
\[
\delta_3^1(1|\sigma, \sigma|1,i_0)=i_1\otimes \sigma \quad \text{and} \quad\delta_3^1(\sigma|1,1|\sigma, i_0)=0 .
\]
We define a different module ${}_{\cK|\cK} N^{\cK}$ by instead setting
\[
\delta_3^1(1|\sigma, \sigma|1,i_0)=0 \quad \text{and} \quad \delta_3^1(\sigma|1,1|\sigma, i_0)=i_1\otimes \sigma.
\]
We make a similar change to the map $\delta_3^1$ on $\tau$ inputs.
 
  This construction turns out to produce homotopy equivalent bimodules:
 
 \begin{lem}\label{lem:M-n-well-defined} There is a homotopy equivalence of $DA$-bimodules
 \[
 {}_{\cK|\cK} M^{\cK}\simeq {}_{ \cK|\cK} N^{\cK}.
 \]
 \end{lem}
 \begin{proof} We construct a morphism $f_{j+1}^1$ between the two bimodules. We set $f_1^1(\ve{x})=\ve{x}\otimes 1$. The only other non-vanishing term is $f_2^1$. This map is determined by the relations
 \[
 f_2^1( \sigma| \sigma, i_0)=i_1\otimes \sigma \quad \text{and} \quad f_2^1(\tau|\tau,i_0)=i_1\otimes \tau.
 \]
  The $DA$-bimodule morphism structure relations for two algebra inputs which both change idempotent are given by the schematic
 \[
 \begin{tikzcd}[column sep=.5cm,row sep=.3cm]a_2\ar[ddrr]& a_1\ar[ddr]& \ve{x}\ar[d]\\
 &&f_1^1\ar[d]\ar[ddr]\\
 &&\delta_3^1\ar[d]\ar[dr]\\
 &&\,&\mu_2
 \end{tikzcd}
 \hspace{-.2cm}+\hspace{-.2cm}
 \begin{tikzcd}[column sep=.5cm,row sep=.3cm]a_2\ar[drr]& a_1\ar[dr]& \ve{x}\ar[d]\\
 &&\delta_3^1\ar[d]\ar[ddr]\\
 &&f_1^1\ar[d]\ar[dr]\\
 &&\,&\mu_2
 \end{tikzcd}
 \hspace{-.2cm}+\hspace{-.2cm}
 \begin{tikzcd}[column sep=.5cm,row sep=.3cm]a_2\ar[dr]& a_1\ar[d]& \ve{x}\ar[dd]\\
 &\mu_2\ar[dr]&\,\\
 &&f_2^1\ar[d]\ar[dr]\\
 &&\,&\mu_2
 \end{tikzcd}
 \hspace{-.2cm}=0
 \]
 which is easily verified. The structure relations for $n\neq 2$ algebra inputs are also easily verified.
 \end{proof}

 \subsection{The type-$A$ identity bimodule}
 \label{sec:AAidentity}
 It is helpful to introduce another bimodule, which we denote by
 \[
 {}_{\cK| \cK} [\bI^{\Supset}]:={}_{\cK| \cK} M^{\cK}\hatbox {}_{\cK}\cD_0.
 \]
 We call $ {}_{\cK| \cK} [\bI^{\Supset}]$ the \emph{type-$A$ identity module}.  We will use ${}_{\cK| \cK} [\bI^{\Supset}]$ to transform type-$D$ actions into type-$A$ actions.  The type-$A$ module ${}_{\cK| \cK} [\bI^{\Supset}]$ appears in the pairing theorem, stated in Theorem~\ref{thm:pairing}.

 There is a further refinement which takes into account the right $\bF[U]$-action on $\cD_0$. We define
 \[
 {}_{\cK| \cK}[\bI^{\Supset}]_{\bF[U]}:={}_{\cK|\cK} M^{\cK}\hatbox {}_{\cK}[\cD_0]_{\bF[U]}.
 \]
 Verifying the pairing theorem with the bimodule ${}_{\cK|\cK} [\bI^{\Supset}]_{\bF[U]}$ turns out to be more subtle than for ${}_{\cK|\cK}[\bI^{\Supset}]$. Nonetheless we prove the pairing theorem with ${}_{\cK|\cK} [\bI^{\Supset}]_{\bF[U]}$ in Section~\ref{sec:AAA-identity} once we introduce the pair-of-pants bimodules.

\subsection{The knot surgery formula over $\cK$}
\label{sec:mapping-cone-formula}

We now describe how to view the mapping cone formula of Ozsv\'{a}th and Szab\'{o} in terms of the algebra $\cK$. 

Recall that Ozsv\'{a}th and Szab\'{o}'s mapping cone formula \cite{OSIntegerSurgeries}*{Theorem~1.1} states that if $\lambda$ is an integral framing on $K\subset S^3$, then
\[
\ve{\CF}^-(S^3_\lambda(K))\iso \bX_\lambda(K)= \Cone\left(\bA(K)\xrightarrow{v+h_\lambda} \bB(K)\right).
\]
Recall, as in Lemma~\ref{lem:module-structure-surgery-hypercube-groups}, that $\bA(K)$ may be identified with a completion of $\cCFK(K)$ and $\bB(K)$ may be identified with a completion of $\scV^{-1} \cCFK(K)$.

We will describe the following algebraic perspectives of the mapping cone formula:
\begin{enumerate}
\item A chain complex $\bX_\lambda(K)$ over $\bF\llsquare U\rrsquare $.
\item A right type-$D$ module $\cX_{\lambda}(K)^{\cK}$.
\item A left type-$A$ module ${}_{\cK} \cX_\lambda(K)$.
\end{enumerate}

These will satisfy the following relations:
\begin{equation}
\bX_{\lambda}(K)\iso \cX_{\lambda}(K)^{\cK}\hatbox {}_{\cK} [\cD_{0}]\simeq \cD_0^{\cK} \hatbox {}_{\cK} \cX_{\lambda}(K).\label{eq:box-with-D}
\end{equation}

We begin with the description as a type-$D$ module. Let $\ve{x}_1,\dots, \ve{x}_n$ be a free basis of $\cCFK(K)$ over  $\bF[\scU,\scV]$. We declare the underlying $\ve{I}$-module of $\cX_{\lambda}(K)^{\cK}$ to be
\[
\cX_{\lambda}(K)^{\cK}=\Span_{\bF}(\ve{x}_1,\dots, \ve{x}_n)\otimes_{\bF} \ve{I}. 
\]
In particular, over $\bF$ each $\xs_i$ contributes one generator in each idempotent. We denote these generators by $\xs_i^0$ and $\xs_i^1$.

The map $\delta^1$ on $\cX_{\lambda}(K)^{\cK}$ is as follows. Firstly, there are internal $\delta^1$ summands from the differential of $\cCFK(K)$. If $\d(\xs)$ contains a summand of $\ys\cdot  \scU^{n}\scV^m$, then $\delta^1(\xs^\veps)$ contains a summand of $\ys^\veps\otimes \scU^n \scV^m$, for $\veps\in \{0,1\}$. 

In Ozsv\'{a}th and Szab\'{o}'s mapping cone formula, $v$ is the canonical inclusion of $\cCFK(K)$ into $\scV^{-1}\cCFK(K)$. Correspondingly, $\delta^1(\xs^0)$ contains a summand of the form $\xs^1\otimes \sigma$.

If $\ys \cdot \scU^i \scV^j$ is a summand of $h_\lambda(\ve{x})$ then we define $\delta^1(\xs^0)$ to have a summand of the form
\[
\ys^1\otimes \scU^i \scV^j \tau .
\]

\begin{lem}
\label{lem:XK-type-D-relations}
$\cX_{\lambda}(K)^{\cK}$ is a type-$D$ module.
\end{lem}
\begin{proof}

The proof is a formal consequence of the fact that $\d^2=0$ on $\cCFK(K)$, that $v$ and $h_{\lambda}$ are chain maps, and that $v$ and $h_{\lambda}$ satisfy the equivariance properties of  Lemma~\ref{lem:module-structure-surgery-hypercube-morphisms}.

In more detail, suppose that $\xs_1,\dots, \xs_n$ is a free basis of $\cCFK(K)$, and consider 
\begin{equation}
((\id_{\cX}\otimes \mu_2)\circ( \delta^1\otimes \id_{\cK})\circ \delta^1)(\xs^{\veps}_i)
\label{eq:d^2=X-knot}
\end{equation}
For concreteness, consider the case when $\veps=0$. There are three sources of terms in the above expression: those arising from two applications of $\d$, those arising from one $v$ and one $\d$, and those arising from one $h_{\lambda}$ and one $\d$. Consider the terms with two $\d$ terms. Write $\d(\xs_i)=\sum_{j=1}^n  \xs_j \cdot f_{j,i}$. The terms with two $\d$ terms are simply
\[
\sum_{k=1}^n \xs_k^0\otimes \bigg(\sum_{j=1}^n f_{k,j}f_{j,i}\bigg),
\]  
which is 0 since $\d^2=0$ on $\CFK(K)$. Similarly, the terms corresponding to one $\d$ and one $v$ are given by the formula
\[
\sum_{j=1}^n \xs_j^1\otimes ( \sigma f_{j,i}+f_{j,i}\sigma)
\]
which vanishes because $\sigma f_{j,i}=f_{j,i}\sigma$ in $\cK$. 

Finally, we consider the terms corresponding to one $\d$ and one $h_{\lambda}$. Write $h_{\lambda}(\xs_i)=\sum_{j=1}^n \xs_j \cdot g_{j,i}$, where $g_{j,i}\in \bF[\scU,\scV,\scV^{-1}]$. The corresponding terms of Equation~\eqref{eq:d^2=X-knot} are
\begin{equation}
\sum_{k=1}^n \xs_k^1\otimes \bigg(\sum_{j=1}^n g_{k,j} \tau f_{j,i} +f_{k,j} g_{j,i}\tau \bigg) =\sum_{k=1}^n \xs_k^1\otimes \bigg(\sum_{j=1}^n g_{k,j} \phi^\tau(f_{j,i}) +f_{k,j} g_{j,i} \bigg)\tau\label{eq:chain-map-h-expand} 
\end{equation}
However $\sum_{j=1}^n g_{k,j} \phi^\tau(f_{j,i}) +f_{k,j} g_{j,i}$ is the $\xs_k$ coefficient of $[\d, h_{\lambda}](\xs_i)$, since $h_{\lambda}$ satisfies $h_{\lambda}(\ve{x}\cdot a )= h_{\lambda}(\ve{x})\cdot \phi^\tau(a)$, by  Lemma~\ref{lem:module-structure-surgery-hypercube-morphisms}. In particular, Equation~\eqref{eq:chain-map-h-expand} vanishes. The case when $\veps=1$ is similar.
\end{proof}

Having defined the type-$D$ module $\cX_{\lambda}(K)^{\cK}$, we may now define the type-$A$  module
\[
{}_{\cK} \cX_{\lambda}(K):=\cX_{\lambda}(K)^{\cK}\hatbox {}_{\cK| \cK} [\bI^{\Supset}].
\]

\subsection{The link surgery formula over $\cL$}
\label{sec:link-surgery-A-infty}
 In this section, we describe how to view Manolescu and Ozsv\'{a}th's link surgery formula as a type-$D$ module.

Recall that we define
\[
\cL_\ell:=\cK_1\otimes_{\bF} \cdots \otimes_{\bF} \cK_\ell.
\]
where each $\cK_i$ denotes a copy of $\cK$. 

Let $L$ be an $\ell$-component link in $S^3$ with integral framing $\Lambda$. The link surgery complexes require a choice of auxiliary data, which we call a \emph{system of arcs} $\scA$. We will discuss these in detail in Section~\ref{sec:systems-of-arcs} below. If $\scA$ is chosen, Manolescu and Ozsv\'{a}th's construction produces a chain complex $\cC_{\Lambda}(L,\scA)$ over $\bF[U_1,\dots, U_\ell]$. In this section, we will describe how to construct a type-$D$ module 
\[
\cX_{\Lambda}(L,\scA)^{\cL_\ell},
\]
based on their construction of $\cC_{\Lambda}(L,\scA)$.  

We view the link algebra $\cL_\ell$ as being an algebra over the idempotent ring
\[
\ve{E}_\ell:=\underbrace{\ve{I}\otimes_{\bF}\cdots \otimes_{\bF} \ve{I}}_\ell.
\]
We can view $\ve{E}_\ell$ as being a ring with $2^\ell$ elementary idempotents, each identified with a point in the cube $\bE_\ell$. If $\veps\in \bE_\ell$, we write
\[
\ve{E}_\ell:=\ve{I}_{\veps_1}\otimes \cdots \otimes \ve{I}_{\veps_\ell}\iso \bF
\]
for the corresponding idempotent. 

If $\veps'\ge \veps$, we can understand $\ve{E}_{\veps'}\cdot \cL_\ell \cdot \ve{E}_{\veps}$ as follows. Write $I_{\veps',\veps}=\{i_1,\dots, i_n\}=(\veps'-\veps)^{-1}(1)$, i.e. the set of indices $i$ where $\veps'_i>\veps$. Then $\ve{E}_{\veps'}\cdot \cL_\ell \cdot \ve{E}_{\veps}$ is generated by elements of the form:
\[
\a_1\cdots \a_\ell \phi_{i_1}\cdots \phi_{i_n}
\]
where:
\begin{enumerate}
\item $a_i$ is in $\bF[\scU_i,\scV_i]$ if $\veps'_i=0$.
\item $a_i$ is in $\bF[\scU_i,\scV_i,\scV_i^{-1}]$ if $\veps'_i=1$.
\item Each $\phi_{i_j}$ is either $\sigma_{i_j}$ or $\tau_{i_j}$. 
\end{enumerate}

If $\ve{x}_1,\dots, \ve{x}_n$ is a free basis of $\cCFL(L)$  over $\bF[\scU_1,\dots, \scU_\ell,\scV_1,\dots, \scV_\ell]$, then we define the generators of $\cX_{\Lambda}(L,\scA)^{\cL_\ell}$ to be
\[
\Span_{\bF}( \ve{x}_1,\dots, \ve{x}_n) \otimes_{\bF} \ve{E}_\ell
\]
In particular, if $\xs$ is a basis element of $\cCFL(L)$, we have a generator $\xs^\veps$ for each $\veps\in \bE_{\ell}$.

If $M\subset L$, write $\veps(M)\in \bE_\ell$ for the coordinate such that $\veps(M)_i=0$ if $L_i\not \in M$ and $\veps(M)_i=1$ if $L_i\in M$.

Suppose that $\ve{x}$ is a basis element of $\cCFL(L)$ and $M\subset L$. By Lemma~\ref{lem:module-structure-surgery-hypercube-groups}, we may view $\ve{x}^{\veps(M)}$ as an element of the group $\cC_{\Lambda}(L,\scA)=\prod_{\ve{s}\in \bH(L)} \frA(\cH^{L\setminus M},\psi^{M}(\ve{s}))$. Suppose also that $\vec{N}$ is an oriented sublink of  $L\setminus M$, and that $\Phi^{\vec{N}}(\xs^{\veps(M)})$ has a summand of $\ve{y}^{\veps(M\cup N)}\cdot f$, where we view $f$ as an element of the $2\ell$-variable polynomial ring, localized at the variables $\scV_i$ such that $L_i\subset N$.
 We may naturally view $f$ as being an element of
\[
\ve{E}_{\veps(M\cup N)}\cdot \cL_\ell\cdot \ve{E}_{\veps(M\cup N)}.
\]

 There is an algebra element $t^{\vec{N}}_{\veps(M),\veps(M\cup N)}\in \ve{E}_{\veps(M\cup N)}\cdot \cL_\ell \cdot \ve{E}_{\veps(M)}$ which is the tensor of $\sigma_i$ for $i$ such that $L_i\subset \vec{N}$ and $L_i$ is oriented the same as $L$, $\tau_i$ for $i$ such that $L_i\subset \vec{N}$ and $L_i$ is oriented oppositely from $L$, and $1$ for $i$ such that $L_i\not \in N$. With this notation, we declare $\delta^1(\xs^{\veps(M)})$ to have the summand
\[
\ys^{\veps(M\cup N)}\otimes f \cdot t^{\vec{N}}_{\veps(M),\veps(M\cup N)}.
\]

\begin{lem} 
$\cX_{\Lambda}(L,\scA)^{\cL_\ell}$ is a type-$D$ module.
\end{lem}
\begin{proof} The proof is similar to Lemma~\ref{lem:XK-type-D-relations}. 
We can decompose the structure map $\delta^1$ as
\[
\delta^1=\sum_{\vec{M}\subset L} \delta^1_{\vec{M}}
\]
where $\delta^1_{\vec{M}}$ consists of all terms which are weighted by idempotent-preserving multiples of the algebra element $t^{\vec{M}}$.

 The map $\delta^1_{\vec{M}}$ encodes the map $\Phi^{\vec{M}}$ on the surgery formula, in the sense that if $\Phi^{\vec{M}}(\xs)$ has a summand $a\cdot \ys$ for some monomial $a$ in the $\scU_i$ and $\scV_i^{\pm 1}$, then $\delta^1_{\vec{M}}$ has a summand $\ys\otimes a t^{\vec{M}}$.

Let $\vec{N}\subset L$ be an oriented link, and consider the components of
 \begin{equation}
 \bigg(
(\id_{\cX}\otimes \mu_2)\circ( \delta^1\otimes \id_{\cL_\ell})\circ \delta^1\bigg)(\xs^{\veps})
 \label{eq:structure-relation-X-L-n}
 \end{equation}
 whose algebra elements are weighted by idempotent-preserving multiples of $t^{\vec{N}}$. These can be identified with
 \begin{equation}
 \sum_{\substack{\vec{M}_1\cup \vec{M}_2= \vec{N}\\ \vec{M}_1\cap \vec{M}_2=\emptyset}} (\bI_{\cX}\otimes \mu_2)\circ(\delta_{\vec{M}_2}^1\otimes \bI_{\cK})\circ \delta_{\vec{M}_1}^1.
 \label{eq:structure-relation-X-L-n-refined}
 \end{equation}
 We consider a component of the above map corresponding to sublinks $\vec{M}_1$ and $\vec{M}_2$, applied to some generator $\xs$. Suppose that $\delta_{\vec{M}_1}^1(\xs)$ contains a summand $\ys\otimes a t^{\vec{M}_1}$. Such a summand corresponds to a summand of $\ys\cdot  a$ in $\Phi^{\vec{M}_1}(\xs)$. Next, consider a summand $\zs\otimes b t^{\vec{M}_2}$ in $\delta_{\vec{M}_2}^1(\ys)$. In the composition appearing in Equation~\eqref{eq:structure-relation-X-L-n-refined}, there is a corresponding term of 
 \[
 \zs\otimes b t^{\vec{M}_2} a t^{\vec{M}_1}=\zs\otimes b \phi^{\vec{M}_2}(a) t^{\vec{N}}.
 \]
 (The second equality follows from the definition of the algebra $\cL_\ell$).
  Using Lemma~\ref{lem:module-structure-surgery-hypercube-morphisms}, which states that $\Phi^{\vec{M}_2}(a \ys)=\phi^{\vec{M}_2}(a)\cdot \Phi^{\vec{M}_2}(\ys)$,  we see that there is a corresponding summand in $(\Phi^{\vec{M}_2}\circ \Phi^{\vec{M}_1})(\xs)$ of $\zs\cdot b \phi^{\vec{M}_2}(a)$.
 Therefore, the fact Equation~\eqref{eq:structure-relation-X-L-n-refined} vanishes follows from the fact that
\[
 \sum_{\substack{\vec{M}_1\cup \vec{M}_2= \vec{N}\\ \vec{M}_1\cap \vec{M}_2=\emptyset}} \Phi^{\vec{M}_2}\circ \Phi^{\vec{M}_1}=0,
\]
 which is proven in \cite{MOIntegerSurgery}*{Proposition~9.4}.
 \end{proof}

Given a collection of integers $\Lambda=(\lambda_1,\dots, \lambda_\ell)$, there is a type-$A$ module
\[
{}_{\cL_\ell}\cD_{\Lambda}
\]
defined by taking the external tensor product of the modules ${}_{\cK} [\cD_{\lambda_i}]$ from Section~\ref{sec:framing}. There is also a right action of $\bF[U_1,\dots, U_\ell]$, and $\cD_{\Lambda}$ may be viewed as an $AA$-bimodule 
\[
{}_{\cL_\ell} [\cD_{\Lambda}]_{\bF[U_1,\dots,U_\ell]}.
\]

The following is immediate from the definition of $\cX_{\Lambda}(L,\scA)^{\cL_\ell}$:
\begin{prop}
 There is a chain isomorphism
\[
\cC_{\Lambda}(L)_{\bF[U_1,\dots, U_\ell]}\iso \cX_{\Lambda}(L)^{\cL_\ell}\hatbox {}_{\cL_\ell} [\cD_{(0,\dots, 0)}]_{\bF[U_1,\dots, U_\ell]}.
\]
\end{prop}

We will also need to consider type-$A$ and type-$DA$ versions of the link surgery formula. An important special case is when we have a distinguished component $K_i\subset L$. We define a bimodule ${}_{\cK} \cX_{\Lambda}(L)^{\cL_{\ell-1}}$, where the $\cK$-input corresponds to $K_i$ by setting
\[
{}_{\cK} \cX_{\Lambda}(L)^{\cL_{\ell-1}}:=\cX_{\Lambda}(L)^{\cL_\ell}\hatbox {}_{\cK| \cK} [\bI^{\Supset}],
\]
where the tensor product is taken along the algebra component corresponding to $K_i$. In the above equation, ${}_{\cK| \cK}[\bI^{\Supset}]$ is the bimodule from Section~\ref{sec:AAidentity}. More generally, we may turn more algebra components from type-$D$ to type-$A$ by tensoring additional copies of ${}_{\cK| \cK}[\bI^{\Supset}]$.

\section{\texorpdfstring{$\sigma$}{sigma}-basic systems of Heegaard diagrams}
\label{sec:basic-systems}

In this section, we define the notion of a \emph{$\sigma$-basic} system of Heegaard diagrams. The notion is a very small generalization of Manolescu and Ozsv\'{a}th's \emph{basic system}.
These are collections of Heegaard diagrams which can be used to compute the link surgery formula. We use the small distinction in notation to disambiguate the notions, and to highlight the algebraic significance of these systems. For these systems, the maps $\Phi^{\vec{N}}$ vanish unless $\vec{N}=+K$ (i.e. a single component, oriented consistently with $L$), or $\vec{N}$ contains only negatively oriented components relative to $L$. We also describe how to construct a $\sigma$-basic system for the connected sum of two links.

\subsection{Systems of arcs}
\label{sec:systems-of-arcs}

In this section, we define the notion of a \emph{system of arcs}, which is a necessary piece of auxiliary data used to define the link surgery formula.

\begin{define}
\label{def:system-of-arcs}
 Suppose that $(S^3,L,\ws,\zs)$ is an oriented, multi-pointed link, which is link minimal (i.e. each component of $L$ contains exactly one basepoint of $\ws$, and one basepoint of $\zs$).
\begin{enumerate}
\item  A \emph{system of arcs} for $L$ consists of a set  $\scA=\{\scA_1,\dots, \scA_\ell\}$ of pairwise disjoint arcs satisfying that $\d \scA_i=\scA_i\cap L=\{w_i,z_i\}$.
\item We say that an arc $\scA_i$ is \emph{beta-parallel} if $\scA_i$ is isotopic to a push-off of the subarc of $L_i$ which goes from $z_i$ to $w_i$.
\item We say that $\scA_i$ is \emph{alpha-parallel} if $\scA_i$ is isotopic to a push-off of the subarc of $L_i$ which goes from $w_i$ to $z_i$.
\end{enumerate}
\end{define}

Note that an arc $\scA_i$ is beta-parallel if for any Heegaard diagram of $(Y,L,\ws,\zs)$, $\scA_i$ is isotopic to the subarc of $K_i\subset L$ which lies in the beta handlebody $U_{\b}$. A similar remark holds for alpha-parallel arcs.

Given a link minimal $(L,\ws,\zs)$ there are two distinguished systems of arcs, $\scA_{\a}$ and $\scA_{\b}$, which are the ones which have only alpha-parallel arcs, or only beta-parallel arcs, respectively.

\begin{rem} 
Manolescu and Ozsv\'{a}th's notion of \emph{good sets of trajectories} \cite{MOIntegerSurgery}*{Definition~8.26} corresponds to our notion of a system of alpha-parallel arcs.
\end{rem}

\subsection{$\sigma$-basic systems of Heegaard diagrams}

In this section, we define the notion of a \emph{$\sigma$-basic system of Heegaard diagrams}. Our notion is a small adaptation of Manolescu and Ozsv\'{a}th's construction.

\begin{define} Suppose that $(L,\ws,\zs)$ is a minimally pointed, $n$-component link in a 3-manifold $Y$, and suppose that $\scA=(\scA_1,\dots, \scA_\ell)$ is a system of arcs for $(Y,L,\ws,\zs)$. A \emph{$\sigma$-basic system of Heegaard diagrams} for $(Y,L,\scA)$ consists of the following data:
\begin{enumerate}
\item A tuple of positive integers $\ve{d}=(d_1,\dots, d_\ell)$ where $n=|L|$.
\item A collection of Heegaard diagrams $\scH(M)=(\cH_{\veps})_{\veps\in \bE(\ve{d})}.$ We assume all Heegaard diagrams have the same underlying Heegaard surface $\Sigma$, which we also assume contains the arcs $\scA_1,\dots, \scA_\ell$. We write $\cH=(\Sigma,\as_{\veps}, \bs_{\veps}, \ws_{\veps}, \zs_{\veps}, \ve{p}_{\veps})$, where $\ws_{\veps}$ and $\zs_{\veps}$ are link basepoints and $\ve{p}_{\veps}$ are free basepoints. 
\item For each component $i\in \{1,\dots, n\}$, we have a formal labeling of each of the intervals $[0,1],\dots, [d_{i}-1,d_i]$ as being either an \emph{alpha incrementing interval}, a \emph{beta incrementing interval} or a \emph{surface isotopy interval}.
\item A decomposition of each arc $\scA_j$ into a concatenation of oriented subarcs $l_{i,1},\dots, l_{i,j_i}$, as well as a choice of surface isotopies $\phi_{i,1},\dots, \phi_{i,j_i}$. We assume $\phi_{i,j}$ is supported in a small neighborhood of $l_{i,j}$ and moves the initial endpoint of $l_{i,j}$ to the terminal endpoint. 
\end{enumerate}
We assume that the following conditions are satisfied:
\begin{enumerate}
\item $\ws_{\veps}=\{w_i: \veps_i=0\}$ and $\zs_{\veps}=\{z_i: \veps_i=0\}$.
\item Let $\cH_{\veps}=(\Sigma,\as_{\veps},\bs_{\veps},\ws_{\veps},\zs_{\veps},\ps_{\veps})$ and let $e_i=(0,\dots, 1,\dots 0)\in \{0,1\}^n$ denote a length 1 vector (i.e. a direction). As described above, the interval $I=[\veps_i,\veps_i+e_i]$ is labeled as either an alpha incrementing interval, a beta incrementing interval, or a surface isotopy interval.
\begin{enumerate}
\item If $I$ is an alpha incrementing segment, then $\bs_{\veps}=\bs_{\veps+e_i}$ while $\as_{\veps+e_i}$ and $\as_{\veps}$ differ. Furthermore, $\ps_{\veps+e_i}=\ps_{\veps}$ unless $\veps_i=0$, in which case $\ps_{\veps+e_i}=\ps_{\veps}\cup \{z_i\}$.
\item If $I$ is a beta incrementing interval, then the same holds as above except with the alpha and beta curves interchanged.
\item If $I$ is a surface isotopy segment, then $\cH_{\veps+e_i}$ is the image of $\cH_{\veps}$ under one of the surface isotopies $\phi_{i,j}$  described above. Similarly, $\ws_{\veps+e_i}=\ws_{\veps}$ and $\zs_{\veps+e_i}=\zs_\veps$. If $\veps_i>0$, then $\ps_{\veps+e_i}=\phi_{i,j}(\ps_{\veps})$. If $\veps_i=0$, then $\ps_{\veps+e_i}=\phi_{i,j}(\ps_{\veps}\cup z_i)$.  
\end{enumerate} 
\item The faces on opposite sides of the box have the same Heegaard diagrams. In more detail, if $\veps\in \bE(\ve{d})$ and $\veps_i=0$, then
\[
\as_{\veps+d_ie_i}=\as_\veps,\quad \bs_{\veps+d_ie_i}=\bs_{\veps},\quad \ps_{\veps+d_ie_i}=\ps_{\veps}\cup \{w_i\}. 
\]
\end{enumerate}
\end{define}

We illustrate an example of a $\sigma$-basic system in Figure~\ref{fig:41}.

\begin{figure}[h]
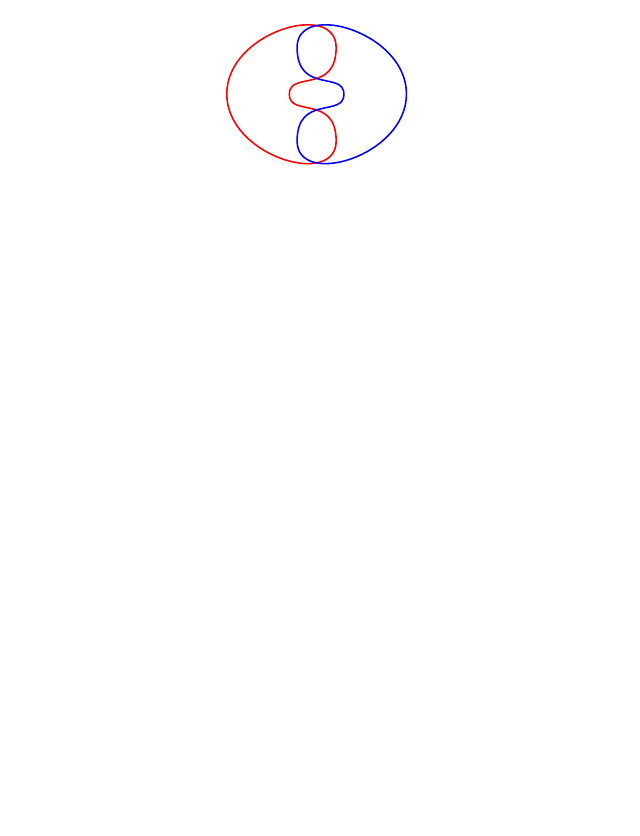
\caption{An example of a $\sigma$-basic system of Heegaard diagrams for the Hopf link. At the top, we show the underlying Heegaard link diagram of the Hopf link, with an arc system consisting of two beta parallel arcs $\scA_1$ and $\scA_2$ (shown as dashed arcs). For this example, the subdivisions of $\scA_1$ and $\scA_2$ are trivial (i.e. consist of a single arc). The labels of the arrows indicate whether the corresponding subintervals are ``alpha incrementing intervals'', ``beta incrementing intervals'' or ``surface isotopies intervals''. We label the $\ws$ basepoints as solid circles,  the $\zs$ basepoints as open circles, and the free basepoints as squares.}
\label{fig:41}
\end{figure}

\begin{rem} The terminology \emph{$\sigma$-basic system} refers to the fact that when we use such systems of Heegaard diagrams to construct the link surgery formula, the resulting complexes have $\Phi^{+K_i}$ equal to inclusions for localization, and $\Phi^{\vec{M}}=0$ if $\vec{M}$ has more than one component and contains a positively oriented component. In particular, using these diagrams we obtain models of the type-$D$ modules $\cX_{\Lambda}(L)^{\cL\ell}$ where the components weighted by $\sigma_i$ are very simple. Our definition is slightly more general than Manolescu and Ozsv\'{a}th's notion of a basic system of Heegaard diagrams, which has the same algebraic properties with respect to the differential.
\end{rem}

The following definition plays an important role in our proof of the connected sum formula for surgery hypercubes:

\begin{define}
\label{def:algebraically-rigid} Suppose that $\cL_{\b}=(\bs_{\veps},\theta_{\veps,\veps'})_{\veps\in \bE_{n}}$ is a weakly admissible hypercube of handleslide equivalent beta-attaching curves on a pointed Heegaard surface $(\Sigma,\ws,\zs, \ve{p})$ (where $\ws,\zs$ are link basepoints and $\ve{p}$ are free basepoints). We say $\cL_{\b}$ is \emph{algebraically rigid} if each $\bT_{\b_\veps}\cap \bT_{\b_{\veps'}}$ has $2^{g(\Sigma)+|\ws|+|\ve{p}|-1}$ intersection points (the minimal possible number). We make a parallel definition for alpha hyperboxes. We say a hyper\emph{box} of attaching curves is algebraically rigid if each sub-hypercube is algebraically rigid. 

Note that a hypercube of attaching curves being algebraically rigid is a property of the attaching curves, and has nothing to do with the chains $\theta_{\veps,\veps'}$.  Therefore, abusing notation slightly, we will say a $\sigma$-basic system of Heegaard diagrams $\scH$ is algebraically rigid if it has the property that $\bT_{\b_{\veps}}\cap \bT_{\b_{\veps'}}$ has $2^{g(\Sigma)+|\ws_\veps|+|\ps_{\veps}|}$ intersection points whenever $|\veps'-\veps|_{L^\infty}\le 1$ and none of the intervals $[\veps,\veps+e_i]$ is a surface isotopy interval for $i$ such that $\veps_i<\veps_i'$. We make the same assumption for the alpha curves.
\end{define}

\subsection{Meridional $\sigma$-basic systems}

\label{sec:meridional}

In this section, we describe a particular family of $\sigma$-basic systems which are useful in our proof of the pairing system. Similar diagrams are ubiquitous in Heegaard Floer theory. For example they are the diagrams which Ozsv\'{a}th and Szab\'{o} use to prove the mapping cone formula \cite{OSIntegerSurgeries}.

\begin{define}
\label{def:meridional-basic-system}
We say that a $\sigma$-basic system of Heegaard diagrams $\scH$ for $(Y,L)$, with a system of arcs $\scA$, is a \emph{meridional $\sigma$-basic system} if the following hold:
\begin{enumerate}
\item Each arc of $\scA$ is either beta-parallel or alpha-parallel.
\item The diagram $\scH(\emptyset)=(\Sigma,\as,\bs,\ws,\zs)$ has the following property. Suppose $K_i\in L$ is a component with corresponding arc $\scA_i\in \scA$ and basepoints $w_i,z_i$. Then $\scA_i$ is either alpha-parallel or beta-parallel. If $\scA_i$ is alpha-parallel, then $\scA_i$ is disjoint from the alpha curves, and intersects a single beta curve $\b^s_i$. A similar assumption is made if $\scA_i$ is beta-parallel. In particular, if $|L|=\ell$, there are $n$ special alpha and beta circles, denoted $\a_i^s$ and $\beta_j^s$.
\item In each Heegaard diagram of $\scH$, there are similarly $\ell$ special alpha and beta curves, and the rest are designated as non-special. The non-special alpha curves of each subcube of $\scH$ consist of small Hamiltonian translates of the non-special curves of $\cH$.
\item If $K_i\subset L$ is a component whose arc $\scA_i$ is beta-parallel, then the component of $\Sigma\setminus \bs$ which contains the basepoints of $K_i$ is a punctured disk which contains a special beta circle $\b^s_i$ as two of its boundary components. Gluing the two $\b^s_i$ boundary components together, we obtain a punctured torus. We assume that $K_i$-axis direction of $\scH$ realizes the isotopy of $\b^s_i$ in a loop around this punctured torus (handlesliding $\b^s_i$ over the boundary punctures). We make an analogous assumption for components with beta-parallel decoration.
\end{enumerate}
\end{define}

\begin{lem}\label{lem:admissibility-meridional-system} Suppose $L$ is an $\ell$-component link in $S^3$ and $\scA$ is a system of arcs for $L$ such that each arc is either alpha-parallel or beta-parallel. Then there exists a link minimal, weakly admissible, algebraically rigid $\sigma$-basic system of Heegaard diagrams for $(S^3,L,\scA)$.
\end{lem}
\begin{proof}
Except for the claim about admissibility, existence of such a $\sigma$-basic system is clear since the construction is given in Definition~\ref{def:meridional-basic-system}.

Admissibility is verified as follows. We first verify admissibility for each hypercube of alpha and beta Lagrangians appearing in the construction. For concreteness, suppose that $\cL_{\a}$ is a hypercube of alpha attaching curves appearing in the construction. For appropriately chosen Hamiltonian translates, admissibility of $\cL_{\a}$ is equivalent to admissibility of the diagram where we delete the tuples of curves which appear only as small translates of each other, and surger $\Sigma$ along these curves as well. The resulting Heegaard diagram is a union of tori, and admissibility is straightforward to arrange. See Figure~\ref{fig:3}.

Admissibility of each pair $(\cL_{\a},\cL_{\b})$ may be arranged by winding one of the cubes (say $\cL_{\a}$) relative to $\cL_{\b}$, following the procedure described in \cite{OSDisks}*{Section~4.2.2} and \cite{OSLinks}*{Section~3.4}.
\end{proof}

\begin{figure}[ht]
\centering
\begingroup%
  \makeatletter%
  \providecommand\color[2][]{%
    \errmessage{(Inkscape) Color is used for the text in Inkscape, but the package 'color.sty' is not loaded}%
    \renewcommand\color[2][]{}%
  }%
  \providecommand\transparent[1]{%
    \errmessage{(Inkscape) Transparency is used (non-zero) for the text in Inkscape, but the package 'transparent.sty' is not loaded}%
    \renewcommand\transparent[1]{}%
  }%
  \providecommand\rotatebox[2]{#2}%
  \newcommand*\fsize{\dimexpr\f@size pt\relax}%
  \newcommand*\lineheight[1]{\fontsize{\fsize}{#1\fsize}\selectfont}%
  \ifx\svgwidth\undefined%
    \setlength{\unitlength}{300.33225914bp}%
    \ifx\svgscale\undefined%
      \relax%
    \else%
      \setlength{\unitlength}{\unitlength * \real{\svgscale}}%
    \fi%
  \else%
    \setlength{\unitlength}{\svgwidth}%
  \fi%
  \global\let\svgwidth\undefined%
  \global\let\svgscale\undefined%
  \makeatother%
  \begin{picture}(1,1.03365747)%
    \lineheight{1}%
    \setlength\tabcolsep{0pt}%
    \put(0,0){\includegraphics[width=\unitlength,page=1]{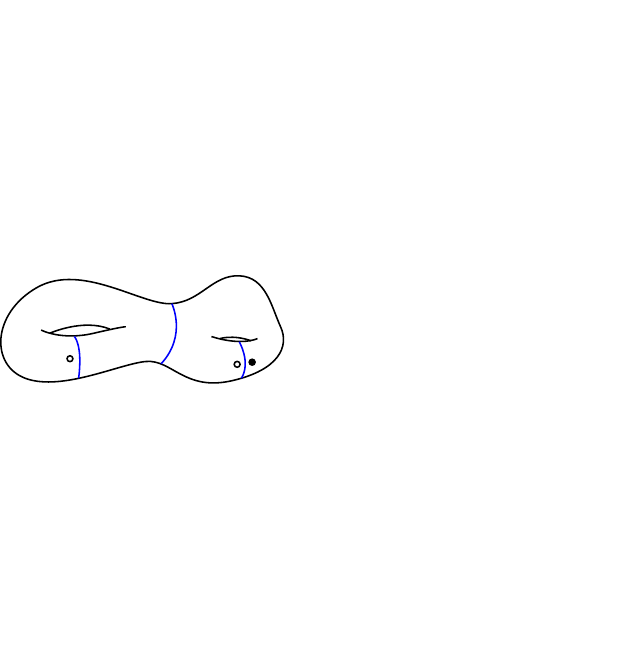}}%
    \put(0.11169747,0.47191066){\makebox(0,0)[rt]{\lineheight{1.25}\smash{\begin{tabular}[t]{r}$z_i$\end{tabular}}}}%
    \put(0,0){\includegraphics[width=\unitlength,page=2]{fig28.pdf}}%
    \put(0.38951106,0.77963574){\color[rgb]{0,0,1}\makebox(0,0)[lt]{\lineheight{1.25}\smash{\begin{tabular}[t]{l}$\beta_i$\end{tabular}}}}%
    \put(0.34648151,0.74897009){\makebox(0,0)[rt]{\lineheight{1.25}\smash{\begin{tabular}[t]{r}$z_i$\end{tabular}}}}%
    \put(0.40758396,0.74776112){\makebox(0,0)[lt]{\lineheight{1.25}\smash{\begin{tabular}[t]{l}$w_i$\end{tabular}}}}%
    \put(0.70543537,0.47699724){\makebox(0,0)[lt]{\lineheight{1.25}\smash{\begin{tabular}[t]{l}$w_i$\end{tabular}}}}%
    \put(0,0){\includegraphics[width=\unitlength,page=3]{fig28.pdf}}%
    \put(0.268407,0.39510599){\makebox(0,0)[lt]{\lineheight{1.25}\smash{\begin{tabular}[t]{l}$\phi$\end{tabular}}}}%
  \end{picture}%
\endgroup%

\caption{Some of the beta curves and a basepoint translation from Lemma~\ref{lem:admissibility-meridional-system}. The first arrow is a basepoint translation, and the remainder are holomorphic polygon counting maps.  The dashed arc on the top is $c_i$.}
\label{fig:3}
\end{figure}

\subsection{Basic systems for general systems of arcs}

We now consider general systems of arcs. Generalizing Lemma~\ref{lem:admissibility-meridional-system}, we prove the following:

\begin{lem}\label{lem:algebraic-rigid-general-basic-system} Suppose that $\scA$ is a system of arcs for $L\subset S^3$. Then there is a an algebraically rigid $\sigma$-basic system of Heegaard diagrams for $(L,\scA)$.
\end{lem}
\begin{proof} The proof is similar to the proof of Lemma~\ref{lem:admissibility-meridional-system}. We begin with a minimally pointed Heegaard diagram $(\Sigma,\as,\bs,\ws,\zs)$ for $(Y,L)$ such that each arc $\scA_i$ is embedded in $\Sigma$ and such that the arcs $\scA_i$ are pairwise disjoint. By performing an isotopy of $\as$ and $\bs$ supported in a neighborhood of the arcs $\scA_i$, we may assume that each $\scA_i$ is the concatenation of two arcs $A_i^\a$ and $A_i^\b$ such that $A_i^\a$ is disjoint from $\bs$ and $A_i^{\b}$ is disjoint from $\as$. We assume that $\d A_i^\b=\{z_i,x_i\}$ and $\d A_i^\a=\{w_i,x_i\}$, for some point $x_i\in \Sigma$.  

We now stabilize the Heegaard diagram $2|L|$ times, as shown in Figure~\ref{fig:29}. For each $\scA_i$, we introduce two stabilizations. One of the stabilizations corresponds to attaching a 1-handle with feet near $z_i$ and $x_i$. We introduce a new alpha and a new beta curve. The alpha curve runs parallel to $A_i^{\b}$. The new beta curve, denoted $\b_i^s$, is the belt-sphere of the 1-handle. Similarly, the other stabilization corresponds to attaching a 1-handle with feet near $x_i$ and $w_i$. Here, the new beta curve runs parallel to $A_i^{\a}$. The new alpha curve, $\a_i^s$, is the belt sphere of the 1-handle.  

We write $(\Sigma',\as',\bs',\ws,\zs)$ for the stabilized Heegaard diagram. We handleslide the arcs $\scA_i$ into the new stabilization tubes to obtain arcs $\scA_i'\subset \Sigma'$. The arcs $\scA_1',\dots, \scA_{\ell}'$ are still isotopic to the arcs in $\scA$.

By definition of a Heegaard link diagram, the basepoints $w_i,z_i\in \Sigma$ are contained in a single component $A_i\subset\Sigma\setminus \as$ and also a single component $B_i\subset \Sigma\setminus \bs$.  Note that the point $x_i\in \Sigma$ is contained in $A_i$ since there is a path $A_i^\b$ from $x_i$ to $z_i$ which is disjoint from $\as$. Similarly $x_i\in B_i$. 
Write $A_i'$ and $B_i'$ for the analogous components of $\Sigma'\setminus \as'$ and $\Sigma'\setminus \bs'$. Note that $A_i'$ is obtained by removing four disks from $A_i$, and similarly for $B_i'$.

We define our $\sigma$-basic system of Heegaard diagrams as follows. In the $i$-th axis direction, we first move $\b_i^s$ around $B_i$ by a sequence of handleslides across the curves of $\bs$. This is similar to the case of meridional $\sigma$-basic systems in Lemma~\ref{lem:admissibility-meridional-system}. We perform a sequence of isotopies and handleslides to move $\b_s^s$ to the other side of $z_i$, so that $z_i$ may be moved to $x_i$ without crossing any attaching curves. Write $z_i'$ for $x_i$, thought of as a basepoint. The next step in the $i$-th axis direction is to move $\a_i^s$ around $A_i$ so that it moves to the other side of $z_i'$, allowing $z_i'$ to be moved to $w_i$.

Note that since $\b_i^s$ moves only around $B_i'$, and $\a_i^s$ moves only around $A_i'$, the $\b_i^s$ may be moved around simultaneously (for different $i$) and similarly different $\a_i^s$ may be moved simultaneously. Additionally, the original curves $\as$ and $\bs$ are also unchanged in this process. In particular, we may build the different axis directions of the hypercube simultaneously to build the $\sigma$-basic system of Heegaard diagrams. By a similar argument to Lemma~\ref{lem:admissibility-meridional-system}, we may choose the constituent hypercubes of attaching circles to be algebraically rigid and weakly admissible.
\end{proof}

\begin{figure}[ht]
\centering
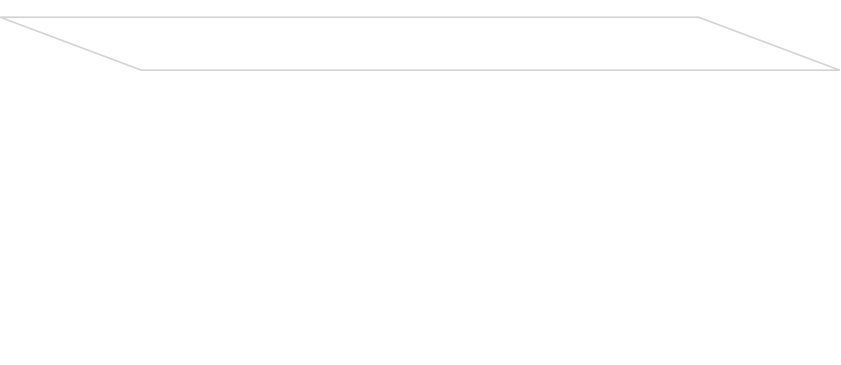
\caption{Top: the Heegaard surface $(\Sigma,\as,\bs,\ws,\zs)$ and the arc $\scA_i$ from Lemma~\ref{lem:algebraic-rigid-general-basic-system}. Bottom: the stabilized Heegaard surface $(\Sigma',\as',\bs',\ws,\zs)$ and the arc $\scA_i'$. Also shown are the special curves $\a_i^s$ and $\b_i^s$. }
\label{fig:29}
\end{figure}

\subsection{Tensor products of hypercubes of attaching curves}
\label{sec:tensor-product-def-hypercube}

 If $\cL_{\b}$ is a hypercube of attaching curves on $(\Sigma,\ws,\zs)$, and $\cL_{\b'}$ is a hypercube of attaching curves on $(\Sigma',\ws,',\zs')$, of dimensions $n$ and $m$, respectively, we now define an $(n+m)$-dimensional hypercube $\cL_{\b}\otimes \cL_{\b'}$ on $(\Sigma\sqcup \Sigma', \ws\cup \ws',\zs\cup \zs')$. We note that the construction of $\cL_{\b}\otimes \cL_{\b'}$ is essentially identical to Lipshitz, Ozsv\'{a}th and Thurston's construction of a \emph{connected sum of chain complexes of attaching circles} \cite{LOTDoubleBranchedII}*{Definition~3.40}.

 We begin with the underlying attaching curves of $\cL_{\b}\otimes \cL_{\b'}$, which we denote by $(\ds_{(\veps,\nu)})_{(\veps,\nu)\in \bE_n\times \bE_m}$, where $n=\dim (\cL_{\b})$ and $m=\dim (\cL_{\b'})$. We set $\ds_{(\veps,\nu)}$ to be a small translation of $\bs_{\veps}\cup \bs'_{\nu}$. We assume the small translations are chosen so that the copies of $\bs_{\veps}$ from $\ds_{(\veps,\nu)}$ and $\ds_{(\veps,\nu')}$ intersect in $2^{g(\Sigma)+|\ws|-1}$ points, and similarly for the translated curves on $\Sigma'$.
 
 Before defining the morphisms of $\cL_{\b}\otimes \cL_{\b'}$, we begin with some helpful terminology:

\begin{define} If $\veps,\veps'\in \bE_n$, we say that there is an \emph{arrow} from $\veps$ to $\veps'$ if $\veps< \veps'$.
\begin{enumerate}
\item We say that an arrow from $(\veps,\nu)\in \bE_{n}\times \bE_m$ to $(\veps',\nu')$ is \emph{mixed} if $\veps<\veps'$ and $\nu<\nu'$.
\item We say that an arrow from $(\veps,\nu)$ to $(\veps',\nu')$ is \emph{non-mixed} if $\veps=\veps'$ or $\nu=\nu'$.
\end{enumerate}
\end{define}

We now describe the morphisms of $\cL_{\b}\otimes \cL_{\b'}$. We define the chains in $\cL_{\b}\otimes \cL_{\b'}$ for each mixed arrow to be zero. For the arrow from $(\veps,\nu)$ to $(\veps,\nu')$, where $\nu<\nu'$, we use the chain
\[
\Theta_{\veps,\veps}^+\otimes \Theta_{\nu,\nu'},
\]
where $\Theta_{\veps,\veps}^+$ denotes the top degree generator (recall that we have perturbed one copy of $\bs_\veps$ slightly in the definition), and $\Theta_{\nu,\nu'}$ denotes the chain from $\cL_{\b'}$. Similarly, for the arrow from $(\veps,\nu)$ to $(\veps',\nu)$, where $\veps<\veps'$, we use the chain
\[
\Theta_{\veps,\veps'}\otimes \Theta_{\nu, \nu}^+.
\]

\begin{rem}
 We will only use the above construction in the case that $\cL_{\b}$ and $\cL_{\b'}$ are algebraically rigid. In this case, $\cL_{\b}\otimes \cL_{\b'}$ is also algebraically rigid, and the hypercube relations are automatic.
\end{rem}
\subsection{Basic systems and connected sums}

\label{sec:basic-systems-connected-sums}

In this section, we define a connected sum operation on $\sigma$-basic systems of Heegaard diagrams. 

We begin by defining a connected sum operation for systems of arcs. Let $L_1$ and $L_2$ be links in $S^3$, with distinguished components $K_1$ and $K_2$, respectively.  If $\scA_1$ and $\scA_2$ are two systems of arcs for $L_1$ and $L_2$, we can form their connected sum $\scA_1\# \scA_2$ by taking the connected sum of $L_1$ and $L_2$ at $w_1\in L_1$ and $z_2\in L_2$. We then concatenate the arcs for the components $K_1$ and $K_2$ to obtain the arc for $K_1\#K_2$.  \emph{A-priori}, this operation is asymmetric between $L_1$ and $L_2$. 

 The most important special cases of this construction are the following:
\begin{enumerate}[label=($\#$-\arabic*), ref=$\#$-\arabic*]
\item \emph{$\a\a$ connected sums}:  Suppose $\scA_{1}$ and $\scA_{2}$ are systems of arcs for $L_1$ and $L_2$, such that the arcs for $K_1$ and $K_2$ are both alpha-parallel. We form a  system of arcs for $L_1\# L_2$ by using an alpha-parallel arc for $K_1\# K_2$, and using the remaining arcs without change.
\item \emph{$\b\b$ connected sums}: These are analogous to $\alpha\alpha$.
\item\label{alpha-beta-connected-sums} \emph{$\a\b$ connected sums}: Suppose that $\scA_{1}$ and $\scA_{2}$ are systems of arcs for $L_1$ and $L_2$, such that the arc for $K_1$ is alpha-parallel while the arc for $K_2$ is beta-parallel. We form the system $\scA_{1}\# \scA_{2}$ by using the co-core of the connected sum band as the arc for $K_1\#K_2$.
\end{enumerate}
These special cases are illustrated in Figure~\ref{fig:25}.

\begin{figure}[ht]
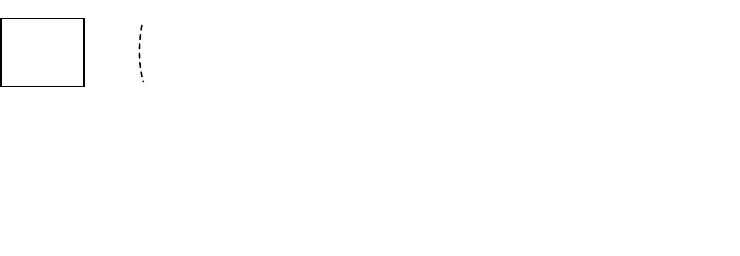
\caption{The connected sums of two alpha-parallel arcs (top row), and the connected sum of one alpha-parallel arc and one beta-parallel arc (bottom row).}
\label{fig:25}
\end{figure}

We now describe how to realize the above connected sum operations on the level of $\sigma$-basic systems of Heegaard diagrams. Suppose that we have $\sigma$-basic systems $\scH_1$ and $\scH_2$ for $L_1$ and $L_2$. We focus on the largest hyperbox of the system, since the smaller hyperboxes are determined by the compatibility condition with respect to inclusion.

We take the connected sum of the Heegaard surfaces $\Sigma_1$ and $\Sigma_2$ at $w_1\in K_1$ and $z_2\in K_2$, deleting those two basepoints, and leaving $z_1$ and $w_2$. See Figure~\ref{fig:31}.

We write $\scH^{(0)}_1,\scH_{1}^{(1)}\subset \scH_1$ for the codimension 1 subboxes where the $K_1$-component is 0 (resp. where the $K_1$-component is maximal). Write $\scH^{(0)}_{2}$ and $\scH_{2}^{(1)}$ for the analogous subboxes of $\scH_2$. 

 We first construct a hyperbox of Heegaard diagrams on $\Sigma_1\# \Sigma_2$, which we denote  $\scH_1\# \scH_2^{(0)}$. We build this hyperbox similarly to the tensor product procedure from Section~\ref{sec:tensor-product-def-hypercube}.  We build an analogous hyperbox of Heegaard diagrams $\scH^{(1)}_1\otimes \scH_2$ of dimension $(\ell_1+\ell_2-1)$. Note that since these hypercubes take place on the connected sum $\Sigma_1\# \Sigma_2$, as opposed to the disjoint union, we may not be able to build the chains in the constituent hypercubes of attaching curves using the tensor product operation. Nonetheless, we may always use the curves so-constructed, and fill in the morphisms of these hypercubes using the standard filling construction of Manolescu-Ozsv\'{a}th  \cite{MOIntegerSurgery}*{Lemma~8.6}. 

Both $\scH_1\otimes \scH_2^{(0)}$ and $\scH_{1}^{(1)}\otimes \scH_2$ share $\scH^{(1)}_1\otimes \scH_2^{(0)}$ as a codimension 1 subface. In particular, we may stack $\scH_1\otimes \scH_2^{(0)}$ and $\scH^{(1)}_1\otimes \scH_2$. Hence, for the connected sum $L_1\# L_2$, we use
\begin{equation}
\scH_1\#\scH_2:=\St\left(\scH_1\otimes \scH^{(0)}_2,\scH^{(1)}_1\otimes \scH_2\right),
\label{eq:connected-sum-basic-system}
\end{equation}
where $\St$ denotes stacking hyperboxes. This construction yields a $\sigma$-basic system of Heegaard diagrams for $L_1\#L_2$, whose largest hyperbox is given by Equation~\eqref{eq:connected-sum-basic-system}.

\begin{rem}
 If $\scH_1$ and $\scH_2$ are algebraically rigid, then $\scH_1\# \scH_2$ is also algebraically rigid.
\end{rem}

\begin{rem} Note that $\St\left(\scH_1^{(0)} \otimes \scH_2, \scH_1^{(1)}\otimes \scH_2\right)$ does not naturally define a $\sigma$-basic system of Heegaard diagrams. This is because the the Heegaard diagrams must encode an isotopy of the basepoint $z_1$ to the basepoint $w_2$ on the connected sum of the Heegaard surfaces. The stacking in Equation~\eqref{eq:connected-sum-basic-system} naturally encodes an isotopy (concatenating the isotopy from $z_1$ to $w_1$, the isotopy of the basepoint across the connected sum tube from $w_1$ to $z_2$, and the isotopy of $z_2$ to $w_2$). The alternate stacking does not naturally encode an isotopy of this sort.
\end{rem}

   \begin{figure}[ht]
   \centering
   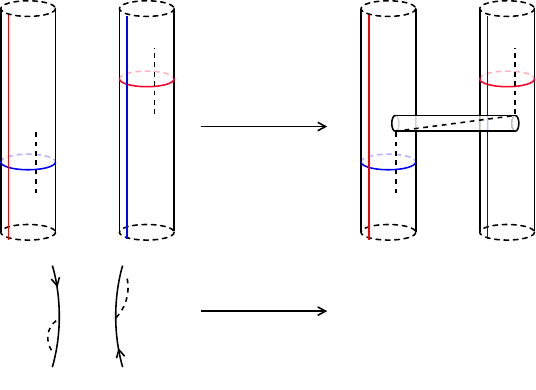
   \caption{Forming the connected sum of two $\sigma$-basic systems of Heegaard diagrams. The dashed arcs denote the arcs in our $\sigma$-basic system. On the left, we begin with an alpha-parallel arc for $K_1$ and a beta-parallel arc for $K_2$. On the right, we have the arc for the connected sum.}
   \label{fig:31}
   \end{figure}

\section{Hypercubes and disjoint unions}

\label{sec:disjoint-unions}

In this section, we consider hypercubes of attaching curves on disconnected Heegaard surfaces. Recall that in Section~\ref{sec:tensor-product-def-hypercube}, we defined a tensor product operation on hypercubes of attaching curves. The goal of this section is to prove the following tensor product formula for pairing such hypercubes of attaching curves:

\begin{prop}
\label{prop:disjoint-unions-hypercubes-main}Suppose that $\cL_{\b}$ and $\cL_{\b'}$ are admissible hypercubes of handleslide equivalent beta attaching curves on $(\Sigma,\ws,\zs)$ and $(\Sigma',\ws',\zs')$, respectively. Suppose additionally that $\cL_{\a}$ and $\cL_{\a'}$ are hypercubes of alpha attaching curves, satisfying the same assumptions. Suppose also that if $\as$ and $\bs$ are curves in $\cL_{\a}$ and $\cL_{\b}$, then $(\Sigma,\as,\bs,\ws,\zs)$ is the diagram of a link in a rational homology 3-sphere, and similarly for $\cL_{\a'}$ and $\cL_{\b'}$.
\begin{enumerate}
\item \label{prop:disjoint-unions-1} If the small translates in the construction of $\cL_{\b}\otimes \cL_{\b'}$ are chosen suitably small, then $\cL_{\b}\otimes \cL_{\b'}$ is a hypercube of attaching curves. The same holds for $\cL_{\a}\otimes \cL_{\a'}$.
\item \label{prop:disjoint-unions-2} If the translations used in the constructions of $\cL_{\a}\otimes \cL_{\a'}$ and $\cL_{\b}\otimes \cL_{\b'}$ are suitably small, then the canonical map of vector spaces
\[
\begin{split}
&\ve{\CF}^-_{J\sqcup J;}(\Sigma\sqcup\Sigma', \cL_{\a}\otimes \cL_{\a'}, \cL_{\b}\otimes \cL_{\b'},\ws\cup \ws',\zs\cup \zs')\\
\iso &\ve{\CF}^-_J(\Sigma,\cL_{\a},\cL_{\b},\ws,\zs)\otimes_{\bF} \ve{\CF}^-_{J'}(\Sigma', \cL_{\a'},\cL_{\b'},\ws',\zs'),
\end{split}
\]
is a chain isomorphism.
\end{enumerate}
The same statements hold if $\Sigma$ and $\Sigma'$ have free basepoints in addition to link basepoints.
\end{prop}
\begin{rem} In our proof, the assumption that $(\Sigma,\as,\bs,\ws,\zs)$ is a diagram for a link in a rational homology 3-sphere is only used to simplify several statements about admissibility. For general three-manifolds, the statement above holds as long as we restrict to a finite set of $\Spin^c$ structures on the hypercubes and assume the diagrams satisfy a stronger version of admissibility (see  \cite{OSDisks}*{Definition~4.10}).
\end{rem}

 Our proof of Proposition~\ref{prop:disjoint-unions-hypercubes-main} follows from the same line of reasoning as \cite{LOTDoubleBranchedII}*{Section~3.5}, though we present it for the benefit of the reader.

\subsection{Small translate theorems}

We now review a preliminary technical result concerning Heegaard diagrams with repeated attaching curves. These are based on the results of \cite{LOTDoubleBranchedII}*{Section~3} and \cite{HHSZExact}*{Section~11}.

Suppose $\cD=(\Sigma,\as_j,\dots, \as_1, \bs_1,\dots, \bs_k,\ws,\zs)$ is a multi-pointed Heegaard multi-diagram, the curves $\bs_1,\dots, \bs_k$ are pairwise handleslide-equivalent, and the curves $\as_1,\dots, \as_j$ are pairwise handleslide-equivalent.  Let $\bI=(i_1,\dots, i_j)$ and $\bI'=(i'_1,\dots, i'_k)$ be tuples of positive integers.
For $s\in \{1,\dots, j\}$, pick curves $\as_s^{1},\dots, \as_s^{i_s}$ which are small Hamiltonian translates of $\as_s$. Similarly, for  $t\in \{1,\dots, k\}$, pick attaching curves $\bs_t^{1},\dots, \bs_t^{i_t'}$ which are small translates of $\bs_t$. We define the diagram
\[
{}_{\bI}\cD_{\bI'}:=(\Sigma,\as_j^{i_j},\dots, \as_j^{1},\dots,\as_1^{i_1},\dots, \as_1^{1},\bs_1^{1},\dots, \bs_1^{i_1'},\dots, \bs_k^{1},\dots, \bs_k^{i_k'},\ws,\zs).
\]
 We assume that each $\bT_{\a_s^{n}}\cap \bT_{\a_s^{m}}$ contains exactly $2^{g(\Sigma)+|\ws|-1}$ points, for each $s$, $n$ and $m$, and similarly for the beta-translates.  

\begin{lem}\label{lem:nearest-point} Suppose $\cD=(\Sigma,\as_j,\dots, \as_1, \bs_1,\dots, \bs_k,\ws,\zs)$ is a multi-diagram such that the $\as_i$ are all pairwise handleslide equivalent, and the $\bs_i$ are also all pairwise handleslide equivalent.  Assume that $(\Sigma, \as_1,\bs_1,\ws,\zs)$ represents a link in a rational homology 3-sphere. Assume that $\cD$ is admissible for each complete collection of basepoints $\cW\subset \ws\cup \zs$. Let ${}_{\bI} \cD_{\bI'}$ be the diagram constructed above. Assume the translations are chosen so that ${}_{\bI} \cD_{\bI'}$ is also weakly admissible. Suppose that $(J_y)_{y\in K_{\ell-1}}$ is a generically chosen family of almost complex structures for counting holomorphic $\ell$-gons, where $\ell$ is the total number of attaching curves on ${}_{\bI} \cD_{\bI'}$. Here, $K_{\ell-1}$ is Stasheff's associahedron on $\ell-1$ inputs. Equivalently, $K_{\ell-1}$ is the moduli space of complex disks with $\ell$ boundary marked points, one of which is distinguished as the ``output''. If the translations in the construction are chosen suitably small, then the following holds: If $\psi$ is a class of $\ell$-gons representing $\frs$ such that the input of each $\ve{\CF}^-(\as_{s}^{i+1},\as_s^{i})$ and $\ve{\CF}^-(\bs_t^{i},\bs_t^{i+1})$ is the top degree generator and $\psi$ has a $J_y$-holomorphic representative for some $y\in K_{\ell-1}$, then 
\[
\mu(\psi)\ge \min(0,3-j-k).
\]
Equivalently, if $S:= \sum_{s=1}^j (i_s-1)+\sum_{t=1}^{k} (i_t'-1)$ denotes the number of special inputs from small translates, then 
\[
\mu(\psi)\ge \min(0,3-\ell+S).
\]
\end{lem}

\begin{rem}\label{rem:QHS-admissibility-finiteness} The condition that $(\Sigma,\as_1,\bs_1)$ represents a rational homology 3-sphere is to obtain finiteness in the number of classes that the polygon maps count. In this case, the existence of the Maslov gradings $\gr_{\cW}$ for each complete collection $\cW\subset \ws\cup \zs$ will ensure that for each $N$, there are only finitely many nonnegative classes in a given
\[
\pi_2(\Theta^{\a}_{n,n-1},\dots, \Theta^\a_{2,1}, \xs, \Theta^{\b}_{1,2},\dots, \Theta^{\b}_{m-1,m}, \ys)
\]  with Maslov index $N$. Similar arguments can be made for diagrams of general 3-manifolds, though one would need to require a stronger version of admissible (such as strong $\frs$-admissibility). See \cite{OSDisks}*{Definition~4.10}.
\end{rem}

We now briefly sketch the proof of Lemma~\ref{lem:nearest-point}. We assume for simplicity of notation that $j=1$ and $i_1=1$, so that there are only beta-translates. Consider a class of $\ell$-gons $\psi$ on ${}_{\bI}\cD_{\bI'}$, such that for each $i$, $t$, the input from $\bT_{\b_{t}^{i}}\cap \bT_{\b_t^{i+1}}$ is the top degree generator. There is a well-defined approximating class $\psi^{\app}$ on $\cD$, and $\mu(\psi)=\mu(\psi^{\app})$ (at this step, one needs the special inputs to be the top degree generator; see \cite{HHSZExact}*{Lemma~11.3}, which generalizes easily from triangles to $\ell$-gons). For each $i$ and $j$, we pick a sequence $\{\bs_{t,n}^i\}_{n\in \N}$ such that as $n\to \infty$, the curves $\bs_{t,n}^i$ approach $\bs_t$ (in the sense that there are maps $\bs_{t,n}^i$ is the time 1 translation by a Hamiltonian vector field $X_{H_{i,t,n}}$ for a sequence of functions $H_{i,t,n}$ on $\Sigma$ which converge to 0 in the $C^\infty$ topology). Write $\cD_n'$ for the diagram obtained by replacing each $\bs_t^i$ with $\bs_{t,n}^i$. Given a sequence  $u_n$ of holomorphic representatives of $\psi$ on $\cD_n'$, we may extract a subsequence which converges to representative of $\psi^{\app}$, which is a $(k+1)$-gon. In particular, by applying transversality for $(k+1)$-gons, one obtains that
\[
\mu(\psi)\ge \min(0,3-(k+1)).
\]

As noted in Remark~\ref{rem:QHS-admissibility-finiteness}, we use the assumption that $(\Sigma,\as,\bs,\ws,\zs)$ is a diagram for a rational homology sphere so that there are only finitely many classes of $\ell$-gons of index $3-\ell$ representing a given $\Spin^c$ structure, and we apply the above argument to each class.

We refer the reader to \cite{LOTDoubleBranchedII}*{Section~3} and \cite{HHSZExact}*{Section~11} for additional details.

There is an additional refinement of the above lemma for the case when $\cD=(\Sigma,\as,\bs,\ws,\zs)$ is an ordinary Heegaard diagram, and we have only one translate $\bs'$ of the curves $\bs$. In this case, if $\bs'$ are chosen to be suitably small Hamiltonian translates of $\bs$, then there is a canonical nearest point map
\[
\Phi_{np}\colon \bT_{\a}\cap \bT_{\b}\to \bT_{\a}\cap \bT_{\b'}.
\]
Extending this map linear over the variables, we obtain a map
\[
\Phi_{np}\colon \ve{\CF}^-(\Sigma,\as,\bs,\ws,\zs)\to \ve{\CF}^-(\Sigma,\as,\bs',\ws,\zs).
\]

\begin{lem} \label{lem:nearest-point-map} Suppose that $(\Sigma,\as,\bs,\ws,\zs)$ is a diagram for a link in a rational homology 3-sphere which is weakly admissible for each complete collection $\cW\subset \ws\cup \zs$, and suppose that $\bs'$ are suitably small translates of the curves $\bs$. Then we have an equality of maps
\[
\Phi_{np}(-)=f_{\a,\b,\b'}(-,\Theta_{\b,\b'}^+).
\]
\end{lem}
The proof in the case of $\widehat{\CF}$ was given by Lipshitz, Ozsv\'{a}th and Thurston in \cite{LOTDoubleBranchedII}*{Lemma~3.38}. This was extended to $\ve{\CF}^-$ (mostly by adapting the notation of \cite{LOTDoubleBranchedII}) in \cite{HHSZExact}*{Proposition~11.1}.

\subsection{Disconnected Heegaard surfaces}
\label{sec:disconnected-Heegaard-surfaces}

Suppose that $\cL_{\b}$ and $\cL_{\b'}$ are hypercubes of handleslide equivalent attaching curves on $(\Sigma,\ws,\zs)$ and $(\Sigma',\ws',\zs')$, of dimension $n$ and $m$, respectively. In this section, we prove Part~\eqref{prop:disjoint-unions-1} of Proposition~\ref{prop:disjoint-unions-hypercubes-main}, i.e. for suitable choices of translates, the diagram $\cL_{\b}\otimes \cL_{\b'}$ is a hypercube of attaching curves (compare \cite{LOTDoubleBranchedII}*{Proposition~3.52}).

We begin with a technical result, from which we will derive the hypercube relations. In the following, we will write $\kappa=(\veps,\nu)$ for points in $\bE_{n+m}\iso \bE_n\times \bE_m$. Also, we write $(\ds_{\kappa})_{\kappa\in \bE_{n+m}}$ for the attaching curves of $\bE_{n+m}$.  If $\kappa_1<\dots<\kappa_\ell$ is an increasing sequence of indices, we write $f_{\kappa_1,\dots, \kappa_\ell}$ for the holomorphic $\ell$-gon map $f_{\dt_{\kappa_1},\dots, \dt_{\kappa_\ell}}$. We similarly write $f_{\veps_1,\dots, \veps_\ell}$ and $f_{\nu_1,\dots, \nu_\ell}$ for holomorphic polygon maps on $\Sigma$ and $\Sigma'$.  We write $\Theta_{\kappa_{i-1},\kappa_i}$ for the chains on $\cL_{\b}\otimes \cL_{\b'}$. We recall that these are given by
\[
\Theta_{\kappa_{i},\kappa_{i+1}}=
\begin{cases} \Theta_{\veps_i,\veps_{i+1}}\otimes \Theta_{\nu_i}^+ & \text{ if } \nu_i=\nu_{i+1}\\
\Theta_{\veps_i}^+\otimes \Theta_{\nu_i,\nu_{i+1}} & \text{ if } \veps_i=\veps_{i+1}\\
0& \text{ otherwise}.
\end{cases}
\]

\begin{lem}\label{lem:technical-Lb-Lb'-hypercube-relations} Let $\cL_\beta$ and $\cL_{\b'}$ are hypercubes of handleslide equivalent attaching curves on $(\Sigma,\ws,\zs)$ and $(\Sigma',\ws',\zs')$, respectively, and let $\cL_{\b}\otimes\cL_{\b'}$ denote the hypercube on $(\Sigma\sqcup \Sigma',\ws\cup \ws', \zs\cup \zs')$ described in Section~\ref{sec:tensor-product-def-hypercube}. Suppose that $\kappa_1<\cdots<\kappa_\ell$ is an increasing sequence in $\bE_{n+m}$, and let $\Theta_{\kappa_{i-1},\kappa_i}$ denote the chains in $\cL_\b\otimes \cL_{\b'}$. If $\ell\neq 3$, then
\begin{equation}
\begin{split}
&f_{\kappa_1,\dots, \kappa_\ell}(\Theta_{\kappa_1,\kappa_2},\dots, \Theta_{\kappa_{\ell-1},\kappa_\ell})
\\
&=\begin{cases}
 f_{\veps_1,\dots, \veps_\ell}(\Theta_{\veps_1,\veps_2},\dots, \Theta_{\veps_{\ell-1},\veps_\ell})\otimes \Theta^+_{\nu_1,\nu_1}
	&	 \text{if } \nu_1=\nu_\ell,
\\
\Theta_{\veps_1,\veps_1}^+\otimes f_{\nu_1,\dots, \nu_\ell}(\Theta_{\nu_1,\nu_2},\dots, \Theta_{\nu_{\ell-1},\nu_\ell})
	& \text{if } \veps_1=\veps_\ell,\\
0& \text{otherwise}.
\end{cases}
\end{split}
\label{eq:main-computation-disjoint-union-hypercubes-v2}
\end{equation}
 When $\ell=2$, we interpret $f_{\kappa_1,\kappa_2}$ as the ordinary Floer differential.
When $\ell=3$, Equation~\eqref{eq:main-computation-disjoint-union-hypercubes-v2} holds when $\nu_1=\nu_3$ or $\veps_1=\veps_3$. If $\ell=3$ and instead $\veps_1<\veps_2=\veps_3$ and $\nu_1=\nu_2<\nu_3$, then
\[
f_{\kappa_1,\kappa_2,\kappa_3}(\Theta_{\kappa_1,\kappa_2},\Theta_{\kappa_2,\kappa_3})=\Theta_{\veps_1,\veps_2}\otimes \Theta_{\nu_2,\nu_3},
\]
where we identify complexes whose attaching curves are small approximations of each other. Similarly if $\veps_1=\veps_2<\veps_3$ and $\nu_1<\nu_2=\nu_3$ we have
\[
f_{\kappa_1,\kappa_2,\kappa_3}(\Theta_{\kappa_1,\kappa_2},\Theta_{\kappa_2,\kappa_3})=\Theta_{\veps_2,\veps_3}\otimes \Theta_{\nu_1,\nu_2}.
\]
\end{lem}
\begin{proof} Consider first the claim when $\ell=2$ (i.e. for holomorphic disks). Suppose that $(\veps_1,\nu_1)<(\veps_2,\nu_2)$ and consider $\Theta_{\kappa_1,\kappa_2}$. By assumption, $\Theta_{\kappa_1,\kappa_2}$ is only non-trivial if $\veps_1=\veps_2$ or $\nu_1=\nu_2$, so assume for concreteness that $\veps_1<\veps_2$ and $\nu_1=\nu_2$. By Definition, $\Theta_{\kappa_1,\kappa_2}=\Theta_{\veps_1,\veps_2}\otimes \Theta_{\nu_1,\nu_1}^+$. The differential on the disjoint union of two diagrams is clearly tensorial, so 
\[
\d(\Theta_{\veps_1,\veps_2}\otimes \Theta_{\nu_1,\nu_1}^+)=\d\Theta_{\veps_1,\veps_2}\otimes \Theta_{\nu_1,\nu_1}^++\Theta_{\veps_1,\veps_2}\otimes \d \Theta_{\nu_1,\nu_1}^+=\d \Theta_{\veps_1,\veps_2}\otimes \Theta_{\nu_1,\nu_1}^+,
\]
proving the claim in this case. The case that $\ell=3$ is similarly straightforward to verify. 

We assume now that $\ell>3$. The map $f_{\kappa_1,\dots, \kappa_\ell}$ counts holomorphic $\ell$-gons of Maslov index $3-\ell$. The holomorphic curves map into $(\Sigma\sqcup \Sigma')\times D_\ell$. To define the polygon counting map, we pick a family of almost complex structures $(\cJ_y)_{y\in K_{\ell-1}}$ on $(\Sigma\sqcup \Sigma')\times D_{\ell}$. Such a family of almost complex structures induces two families, $(J_y)_{y\in K_{\ell-1}}$ and $(J_y')_{y\in K_{\ell-1}}$, on $\Sigma\times D_\ell$ and $\Sigma'\times D_{\ell}$, respectively. 

The holomorphic polygon map $f_{\kappa_1,\dots, \kappa_\ell}$ on the disjoint union may be equivalently described by counting pairs $(u,u')$ representing pairs of classes of $\ell$-gons $(\psi,\psi')$ satisfying 
\begin{equation}
\mu(\psi)+\mu(\psi')=3-\ell,\label{eq:basic-index-additivity}
\end{equation}
 which have the same almost complex structure parameter $y\in K_{\ell-1}$. In particular, we may naturally view the moduli space for a class $(\psi,\psi')$ (as would be counted by the polygon map $f_{\kappa_1,\dots, \kappa_{\ell}}$) as the fibered product
 \begin{equation}
 \begin{split}
&\bigcup_{y\in K_{\ell-1}} \cM_{\cJ_y}(\psi,\psi')\times \{y\}\\
=&\left(\bigcup_{y\in K_{\ell-1}} \cM_{J_y}(\psi)\times \{y\}\right) \times_{\ev} \left(\bigcup_{y\in K_{\ell-1}} \cM_{J_y'}(\psi')\times \{y\}\right),
\end{split}
\label{eq:moduli-space-fibered-product}
 \end{equation}
 where $\ev$ is the evaluation map to $K_{\ell-1}$, which sends a pair $(u,y)$ to the parameter $y$.

Consider the last line of Equation~\eqref{eq:main-computation-disjoint-union-hypercubes-v2}. This equation concerns sequences $\kappa_1<\cdots<\kappa_\ell$ which increment both the $\bE_n$-coordinates and the $\bE_m$-coordinates.  Consider a class of $\ell$-gons $(\psi,\psi')$ on $(\Sigma\sqcup \Sigma')\times D_\ell$, potentially counted by the $\ell$-gon map $f_{\kappa_1,\dots, \kappa_\ell}$. Let $k$ be the number of inputs of $\psi$ which do not increment the $\bE_n$-coordinate. Let $k'$ be the number of inputs of $\psi'$ which do not increment the $\bE_m$-coordinate.

 By Lemma~\ref{lem:nearest-point}, if $(\psi,\psi')$ has a holomorphic representative, we must have
\begin{equation}
\mu(\psi)\ge \min(3-\ell+k,0)\quad \text{and} \quad \mu(\psi')\ge \min(3-\ell+k',0), \label{eq:twin-inequalities-Maslov}
\end{equation}
because each input of $\psi$ which does not increase the $\bE_n$-coordinate contributes a top degree generator in a pair of translated curves as input. Since $\mu(\psi)+\mu(\psi')=3-\ell$, by assumption, for a holomorphic representative to exist we must have
\[
\min(3-\ell+k,0)+\min(3-\ell+k',0)\le 3-\ell.
\]
We add $2(\ell-3)$ to the above to obtain
\begin{equation}
\min(k,\ell-3)+\min(k',\ell-3)\le \ell-3. \label{eq:inequality-kk'}
\end{equation}
On the other hand, there are $\ell-1$ total inputs to the map $f_{\kappa_1,\dots, \kappa_{\ell}}$, and each input increments either $\bE_n$ or $\bE_m$, but not both (by definition of $\cL_{\b}\otimes \cL_{\b'}$). Hence 
\begin{equation}
k+k'=\ell-1.
\label{eq:k+k'=l-1}
\end{equation} Combining the above equation with Equation~\eqref{eq:inequality-kk'}, we deduce that it is not possible that both $k\le \ell-3$ and $k'\le \ell-3$, since then Equation~\eqref{eq:inequality-kk'} would  imply $k+k'\le \ell-3$, which contracts Equation~\eqref{eq:k+k'=l-1}. Hence, we assume without loss of generality that $k> \ell-3$. Equation~\eqref{eq:inequality-kk'} implies that
\[
\min(k',\ell-3)\le 0,
\]
which implies that either $k'=0$ or $\ell\in \{2,3\}$. Since we assumed that $\ell>3$, we conclude that $k'=0$. Equation~\eqref{eq:k+k'=l-1} implies that $k=\ell-1$. Symmetrically, if instead we assumed $k'>\ell-3$, we would have $(k,k')=(0,\ell-1)$. This proves the last line of Equation~\eqref{eq:main-computation-disjoint-union-hypercubes-v2}.

We now consider the first two lines of Equation~\eqref{eq:main-computation-disjoint-union-hypercubes-v2}. These correspond to the cases that $(k,k')$ is $(\ell-1,0)$ or $(0,\ell-1)$. Consider the case that $(k,k')=(0,\ell-1)$. In this case, the holomorphic curve count occurs on the diagram 
\[
(\Sigma,\bs_{\veps_1},\dots,\bs_{\veps_\ell},\ws,\zs)\sqcup (\Sigma', \bs'_{\nu},\dots, \bs'_{\nu},\ws',\zs'),
\]
where the $\ell$ attaching curves on the right-hand diagram are small translates and we write $\nu$ for $\nu_1$.  Equation~\eqref{eq:twin-inequalities-Maslov} implies in this case that $\mu(\psi)\ge 3-\ell$ and $\mu(\psi')\ge 0$. Together with Equation~\eqref{eq:basic-index-additivity} we see that 
\begin{equation}
\mu(\psi)=3-\ell \quad \text{and} \quad  \mu(\psi')=0.
\label{eq:index-equations-psi-psi'}
\end{equation}

If $\psi'\in \pi_2(\Theta^+_{\nu,\nu},\cdots, \Theta^+_{\nu,\nu},\zs)$, the Maslov index formula for the grading implies that
\[
\mu(\psi')=n_{\ws'}(\psi')+\gr_{\ws'}(\Theta_{\nu,\nu}^+,\zs).
\]
Since $\mu(\psi')=0$, we conclude that $n_{\ws'}(\psi')=0$ and $\zs=\Theta_{\nu,\nu}^+$.

For the main claim, it suffices to show that
\[
\# \bigcup_{y\in K_{\ell-1}}\cM_{J_y}(\psi)\equiv \sum_{\substack{
\psi'\in \pi_2(\Theta_{\nu,\nu}^+,\dots, \Theta_{\nu,\nu}^+)
\\
\mu(\psi')=0\\
y\in K_{\ell-1}
}} \cM_{\cJ_y}(\psi,\psi') \pmod{2}.
\]
Equation~\eqref{eq:index-equations-psi-psi'} implies that 
\[
\dim \bigcup_{y\in K_{\ell-1}} \cM_{J_y}(\psi)\times \{y\}=0\quad \text{and} \quad \dim \bigcup_{y\in K_{\ell-1}} \cM_{J_y'}(\psi')\times \{y\}=\dim(K_{\ell-1})=\ell-3.
\]
By the fibered product equation in Equation~\eqref{eq:moduli-space-fibered-product} and the above dimension counts, it therefore suffices to show that the map
\[
\ev \colon \bigcup_{
\substack{
\psi'\in \pi_2(\Theta_{\nu,\nu}^+,\dots, \Theta_{\nu,\nu}^+)
\\
\mu(\psi')=0\\
y\in K_{\ell-1}
}}\cM_{J_y}(\psi')\to K_{\ell-1}
\]
has odd degree. (Compare \cite{LOTDoubleBranchedII}*{Lemma~3.50}) This is proven by considering the preimage of a path $\gamma\colon [0,1]\to K_{\ell-1}$, such that $\gamma(0)=y$, $\gamma(t)\in \Int K_{\ell-1}$ for $t\in [0,1)$ and $\gamma(1)$ is a point in $\d K_{\ell-1}$ which realizes a degeneration of a holomorphic $\ell$-gon into $\ell-2$ triangles (i.e. a point in the lowest dimensional boundary strata of $K_{\ell-1}$). The codimension 1 degenerations along the image of $(0,1)$ under $\g$ consist of index 1 holomorphic disks breaking off. The holomorphic disks which break off at the input generator cancel in pairs, since each $\Theta_{\nu,\nu}^+$ is a cycle. There are no disks which bubble off in the output, since they would leave an index $-1$ $\ell$-gon in $\pi_2(\Theta_{\nu,\nu}^+,\dots, \Theta_{\nu,\nu}^+,\ys)$, for some $\ys$, as well as an index $1$ disk $\phi\in\pi_2(\ys,\Theta_{\nu,\nu}^+)$ with $n_{\ws}(\phi)=0$. The existence of such a degeneration would imply that $\gr(\ys,\Theta_{\nu,\nu}^+)=1$, which is impossible since there are no generators with grading higher than $\Theta_{\nu,\nu}^+$.

The cardinality of the limit at $t=1$ corresponds to the component of $\Theta_{\nu,\nu}^+$ in the $\ell-2$ fold composition of maps of the form $f_{\b'_{\nu},\b'_{\nu},\b'_{\nu}}(\Theta_{\nu,\nu}^+,-)$, applied to the element $\Theta_{\nu,\nu}^+$ (e.g. by the grading preserving invariance of the Heegaard Floer complex of connected sums of $S^1\times S^2$). Clearly this is 1, modulo 2. This establishes the first line of~\eqref{eq:main-computation-disjoint-union-hypercubes-v2}. The second line follows from the same reasoning, establishing Equation~\eqref{eq:main-computation-disjoint-union-hypercubes-v2} and completing the proof.
\end{proof}

We now use the previous Lemma to finish our proof that $\cL_{\b}\otimes \cL_{\b'}$ is a hypercube of attaching curves:

\begin{proof}[Proof of Part~\eqref{prop:disjoint-unions-1} of Proposition~\ref{prop:disjoint-unions-hypercubes-main}]
Suppose that $\kappa<\kappa'$ are points in $\bE_{n+m}$. We wish to show the hypercube relations for $\cL_{\b}\otimes \cL_{\b'}$. Write $\kappa=(\veps,\nu)$ and $\kappa'=(\veps',\nu')$. There are two cases to consider:
\begin{enumerate}
\item $\veps=\veps'$ or $\nu=\nu'$. I.e. the arrow $(\kappa,\kappa')$ is \emph{non-mixed}.
\item $\veps<\veps'$ and $\nu<\nu'$. I.e. the arrow $(\kappa,\kappa')$ is \emph{mixed}.
\end{enumerate}
In the case of a non-mixed arrow, Lemma~\ref{lem:technical-Lb-Lb'-hypercube-relations} shows that the hypercube relations for $\cL_{\b}\otimes \cL_{\b'}$ follow immediately from the hypercube relations on $\cL_{\b}$ and $\cL_{\b'}$.

For a mixed arrow $(\kappa,\kappa')$, Lemma~\ref{lem:technical-Lb-Lb'-hypercube-relations} implies that there are exactly two terms which contribute to the hypercube relation. These correspond to the two broken arrow sequences $(\kappa,\kappa_0,\kappa')$ where $\kappa_0\in \{(\veps,\nu'),(\veps',\nu)\}$. By Lemma~\ref{lem:technical-Lb-Lb'-hypercube-relations} both of these sequences contribute $\Theta_{\veps,\veps'}\otimes \Theta_{\nu,\nu'}$, which cancel, so the hypercube relations are satisfied.
\end{proof}

We are now able to prove the remainder of Proposition~\ref{prop:disjoint-unions-hypercubes-main}:

\begin{proof}[Proof of Part~\eqref{prop:disjoint-unions-2} Proposition~\ref{prop:disjoint-unions-hypercubes-main}]
The proof is in the same spirit as Part~\eqref{prop:disjoint-unions-1} of the proposition. To simplify the notation, we will assume that $\cL_{\a}$ and $\cL_{\a'}$ are both 0-dimensional, and consist of ordinary sets of attaching circles $\as$ and $\as'$. If $(\veps,\nu)<(\veps',\nu')$ are points in $\bE_{n+m}$, then the hypercube map from $(\veps,\nu)$ to $(\veps',\nu')$ in 
\begin{equation}
\ve{\CF}^-(\Sigma\sqcup \Sigma', \as\cup \as', \cL_{\b}\otimes \cL_{\b'},\ws\cup \ws')\label{eq:hypercube-restated}
\end{equation}
 is obtained by summing over all increasing sequences 
 \[
 (\veps,\nu)=(\veps_1,\nu_1)<\cdots< (\veps_\ell,\nu_\ell)=(\veps',\nu')
 \]
 the holomorphic $(\ell+1)$-gon map which has special inputs from $\cL_{\b}\otimes \cL_{\b'}$. We say such a sequence is \emph{mixed} if $\veps<\veps'$ and $\nu<\nu'$. We claim that in the pairing of the two hypercubes of attaching curves, mixed sequences make trivial contribution. This is an approximation argument similar to Part~\eqref{prop:disjoint-unions-1} of the Proposition, as we now describe.
  
   We recall that $\cL_{\b}\otimes \cL_{\b'}$ has no mixed arrows which are assigned non-zero chain. Hence we may consider only broken arrow paths in $\bE_{n+m}$ where each individual arrow increases exactly one of the $\bE_n$ or $\bE_m$ coordinates. For such a sequence of length $\ell\ge 1$, let $k$ be the number of arrows which do not increase the $\bE_n$ coordinate, and let $k'$ be the number of arrows which do not increase the $\bE_m$ coordinate. In our present situation, $k+k'=\ell-1$. The contribution of this arrow path is a count of holomorphic $(\ell+1)$-gons, with $\ell-1$ special inputs from the hypercube $\cL_{\b}\otimes \cL_{\b'}$, and one input from $\ve{\CF}^-(\as\cup \as', \bs_{\veps}\cup \bs'_{\nu})$. Suppose $(\psi,\psi')$ is a homology class which could potentially contribute. By construction, $\mu(\psi)+\mu(\psi')=2-\ell$. Equation~\eqref{eq:twin-inequalities-Maslov} adapts to show that
\[
\mu(\psi)\ge \min(2-\ell+k,0)\quad \text{and} \quad \mu(\psi')\ge \min(2-\ell+k',0),
\]
so 
\[
\min(2-\ell+k,0)+\min(2-\ell+k',0)\le 2-\ell,
\]
and hence
\[
\min(k,\ell-2)+\min(k',\ell-2)\le \ell-2.
\]
Similar to the argument from  Part~\eqref{prop:disjoint-unions-1}, the only allowable configurations are $(k,k')\in \{(\ell-1,0),(0,\ell-1)\}$. This implies that there are no mixed arrows in the hypercube of Equation~\eqref{eq:hypercube-restated}.

It remains to verify that the non-mixed arrows of Equation~\eqref{eq:hypercube-restated} coincide with the tensor product differential. Arguing similarly to the proof of Part~\eqref{prop:disjoint-unions-1}, it suffices to show that if $\xs\in \bT_{\a}\cap \bT_{\b_{\veps}}$, then the map
\[
\ev \colon \bigcup_{
\substack{
\psi\in \pi_2(\xs,\Theta_{\veps,\veps}^+,\dots, \Theta_{\veps,\veps}^+,\ys)
\\
\mu(\psi)=0\\
y\in K_{\ell}
}}\cM_{J_y}(\psi)\to K_{\ell}
\]
is odd degree if $\ys$ is the canonical nearest point to $\xs$, $\ys=\xs_{np}$, and is even degree otherwise. Similarly to Part~\eqref{prop:disjoint-unions-1}, we consider the preimage of a path $\gamma\colon [0,1]\to K_{\ell}$, which connects a generic $y\in \Int K_{\ell}$ to a point in $\d K_{\ell}$ of maximal codimension. The only possible generic degenerations on the interior consist of an index 1 disk breaking off. Disks breaking off at the $\Theta_{\veps,\veps}^+$ inputs cancel in pairs. Index 1 disks breaking off at the $\xs$ input or the $\ys$ output are impossible, since they leave an index $-1$ $(\ell+1)$-gon which has $\ell-1$ inputs equal to the top degree generator, whose existence would violate Lemma~\ref{lem:nearest-point}. Hence, we identify the preimage over $y$ with the preimage of $\gamma(1)$. By the nearest point map argument of Lemma~\ref{lem:nearest-point-map}, the count is odd if and only if $\ys=\xs_{np}$, completing the proof.
\end{proof}

\section{Hypercubes and connected sums}

\label{sec:hypercubes-connected-sums}
In this section, we prove a connected sum formula for hypercubes of attaching curves. We begin with some preliminary definitions before stating our result in Proposition~\ref{prop:connected-sums}.

\begin{define}
\label{def:matching} A \emph{Heegaard surface with matched link basepoints} $(\Sigma,\ws,\zs,\ps)$ consists of a surface with a finite collection of points $\ws\cup \zs\cup \ps$, which is equipped with a matching function
\[
m\colon  \ws\to \zs,
\]
which is a bijection. We call $\ws\cup \zs$ the \emph{link basepoints}, and $\ps$ the \emph{free basepoints}.
\end{define}

\begin{define}
\label{def:admissible-connected-sum}
Suppose that $(\Sigma,\ws,\zs,\ps)$ and $(\Sigma',\ws',\zs',\ps')$ are two Heegaard surfaces with matched link basepoints, as in Definition~\ref{def:matching}. We say that $(\Sigma\# \Sigma',\ws'',\zs'',\ps'')$ is formed by an \emph{admissible connected sum} if it is formed by one of the two following procedures (possibly with the roles of $\Sigma$ and $\Sigma'$ reversed).
\begin{enumerate}
\item (Connected sum at a free basepoint) The connected sum is taken at a free basepoint $p\in \ps $, and at some point $x'\in \Sigma'\setminus (\ws'\cup \zs')$. The basepoints are
\[
\ws''=\ws\cup \ws', \quad \zs''=\zs\cup \zs', \quad \text{and} \quad \ps''=(\ps\setminus \{p\})\cup \ps'.
\] 
\item (Connected sum at link basepoints) The connected sum is taken at two link basepoints $w\in \ws$ and $z'\in \zs'$. The basepoints are 
\[
\ws''=(\ws\setminus \{w\})\cup \ws',\quad \zs''=\zs\cup (\zs'\setminus \{z'\}),\quad \text{and} \quad \ps''=\ps\cup \ps'.
\] 
\end{enumerate}
\end{define}

Suppose that we form an admissible connected sum $(\Sigma\# \Sigma',\ws'',\zs'',\ps'')$ from the pointed surfaces $(\Sigma,\ws,\zs,\ps)$ and $(\Sigma',\ws',\zs',\ps')$. Suppose $\cL_{\b}$ and $\cL_{\b'}$ are hypercubes of handleslide equivalent attaching curves on $\Sigma$ and $\Sigma'$, respectively. We may construct a hypercube-shaped diagram  $\cL_{\b}\otimes \cL_{\b'}$ on $(\Sigma\# \Sigma',\ws'',\zs'',\ps'')$ via the same procedure as on the disjoint union of the two diagrams in Section~\ref{sec:tensor-product-def-hypercube}. We do not claim that the hypercube relations are satisfied, however we have the following:

\begin{rem}
If $\cL_{\b}$ and $\cL_{\b'}$ are algebraically rigid, then the top degree chains coincide on disjoint unions of diagrams and on admissible connected sums, justifying the use of the notation $\cL_{\b}\otimes \cL_{\b'}$ for both situations. Also, if $\cL_{\b}$ and $\cL_{\b'}$ are algebraically rigid, then $\cL_{\b}\otimes \cL_{\b'}$ is automatically a hypercube of attaching curves. 
\end{rem}

To state our connected sum theorem for hypercubes, we need an additional condition on our hypercubes of attaching curves:

\begin{define}
\label{def:graded-by}
 Suppose that $\cL_{\a}$ is an algebraically rigid hypercube of handleslide equivalent attaching curves on $(\Sigma,\ws,\zs,\ps)$, and $x$ is a point in the complement of $\Sigma\setminus \bigcup_{\a\in \cL_{\a}} \as$.  We say that $x$ is \emph{basepoint-esque} for $\cL_{\a}$ if for every composable sequence $\Theta_{\nu_j,\nu_{j-1}},\dots, \Theta_{\nu_2,\nu_1}$ of chains in $\cL_{\a}$, and every nonnegative class $\psi\in \pi_2(\Theta_{\nu_j,\nu_{j-1}},\dots, \Theta_{\nu_2,\nu_1},\xs)$, we have
\[
\mu(\psi)\ge 2n_x(\psi).
\]
(Note $\xs$ is the outgoing intersection point in the class $\psi$).
 We make a similar definition for hypercubes of algebraically rigid beta hypercubes.
\end{define}

\begin{lem} If $\cL_{\a}$ is an algebraically rigid hypercube of handleslide equivalent attaching curves on $(\Sigma,\ws,\zs,\ps)$, and $x\in \ws\cup \zs\cup \ps$, then $x$ is basepoint-esque for $\cL_{\a}$. 
\end{lem}
\begin{proof} Assume for concreteness that $x\in \ws$. If $\psi\in \pi_2(\Theta_{\nu_j,\nu_{j-1}},\dots, \Theta_{\nu_2,\nu_1},\xs)$ is a class, then the relation between the Maslov index and the absolute grading on Heegaard Floer homology (see \cite{OSTriangles}*{Section~7}) implies
\[
\mu(\psi)=2n_{\ws\cup \ps}(\psi)+\gr_{\ws\cup \ps}(\Theta_{\nu_j,\nu_{j-1}})+\cdots +\gr_{\ws\cup \ps}(\Theta_{\nu_2,\nu_1})-\gr_{\ws\cup \ps}(\xs),
\]
where $\gr_{\ws\cup \ps}$ is the absolute grading, normalized so that the top degree cycle of  
\[
\ve{\HF}^-(\Sigma,\as_{\nu_i},\as_{\nu_{i-1}},\ws,\zs,\ps)
\]
 has grading 0. Since $\cL_{\a}$ is algebraically rigid, we also know that $\gr_{\ws\cup \ps}(\xs)\le 0$ and $\gr_{\ws\cup \ps}(\Theta_{\nu_{i+1},\nu_i})=0$ for all $i$ and hence
 \[
 \mu(\psi)\ge 2n_{\ws\cup \ps}(\psi).
 \] 
 If $\psi$ is a nonnegative class, we have $n_{\ve{w}\cup \ps}(\psi)\ge n_x(\psi)$, so we see that $x$ is basepoint-esque for $\cL_{\a}$.
\end{proof}

\begin{rem}
\label{rem:not-graded-by-p}
 Suppose 
 \[
 \cL_{\a}=\left(\begin{tikzcd} \as\ar[r, "\Theta^+_{\a',\a}"]& \as'\end{tikzcd}\right)
 \] where $\as'$ is obtained by performing a Hamiltonian isotopy to $\as$ which crosses $x$, then $x$ is not basepoint-esque for $\cL_{\a}$  because there is a bigon of index 1 which covers $x$ once. See Figure~\ref{fig:50}.
\end{rem}

\begin{figure}[h]
\begingroup%
  \makeatletter%
  \providecommand\color[2][]{%
    \errmessage{(Inkscape) Color is used for the text in Inkscape, but the package 'color.sty' is not loaded}%
    \renewcommand\color[2][]{}%
  }%
  \providecommand\transparent[1]{%
    \errmessage{(Inkscape) Transparency is used (non-zero) for the text in Inkscape, but the package 'transparent.sty' is not loaded}%
    \renewcommand\transparent[1]{}%
  }%
  \providecommand\rotatebox[2]{#2}%
  \newcommand*\fsize{\dimexpr\f@size pt\relax}%
  \newcommand*\lineheight[1]{\fontsize{\fsize}{#1\fsize}\selectfont}%
  \ifx\svgwidth\undefined%
    \setlength{\unitlength}{148.7961248bp}%
    \ifx\svgscale\undefined%
      \relax%
    \else%
      \setlength{\unitlength}{\unitlength * \real{\svgscale}}%
    \fi%
  \else%
    \setlength{\unitlength}{\svgwidth}%
  \fi%
  \global\let\svgwidth\undefined%
  \global\let\svgscale\undefined%
  \makeatother%
  \begin{picture}(1,0.40693723)%
    \lineheight{1}%
    \setlength\tabcolsep{0pt}%
    \put(0,0){\includegraphics[width=\unitlength,page=1]{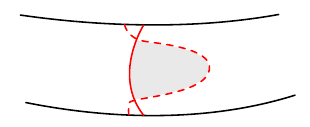}}%
    \put(0.52270566,0.15756543){\color[rgb]{0,0,0}\makebox(0,0)[t]{\lineheight{1.25}\smash{\begin{tabular}[t]{c}$x$\end{tabular}}}}%
    \put(0.40211125,0.15416525){\color[rgb]{1,0,0}\makebox(0,0)[rt]{\lineheight{1.25}\smash{\begin{tabular}[t]{r}$\as$\end{tabular}}}}%
    \put(0.69481466,0.18646738){\color[rgb]{1,0,0}\makebox(0,0)[lt]{\lineheight{1.25}\smash{\begin{tabular}[t]{l}$\as'$\end{tabular}}}}%
  \end{picture}%
\endgroup%

\caption{An example of a point $x$ which is not basepoint-esque from Remark~\ref{rem:not-graded-by-p}. }
\label{fig:50}
\end{figure}

 Suppose that $\cL_{\a}$ is a hypercube of handleslide equivalent attaching curves on $(\Sigma,\ws,\zs,\ps)$ and $x\in \Sigma$ is in the complement of each curve in $\cL_{\a}$. If $\as$ is in $\cL_{\a}$, write $E_x^{\a}$ for the product of the variables of the basepoints in the component of $\Sigma\setminus \as$ containing $x$.

\begin{lem}\label{lem:action-of-variables-well-defined} Let $\cL_{\a}$ be a hypercube of handleslide equivalent attaching curves on $(\Sigma,\ws,\zs,\ps)$ and suppose that $x$ is basepoint-esque for $\cL_\a$. Then $E_x^\a$ is independent of $\as\in \cL_{\a}$.
\end{lem}
\begin{proof} Let $\as_1,\as_2\in \cL_{\a}$. Consider the endomorphism $A_x$ of $\ve{\CF}^-(\as_2,\as_1)$ which counts holomorphic disks of index 1, which are given a multiplicative factor of $n_x(\phi)$. It is straightforward to see that $[\d, A_x]=(E^{\a_1}_x+E^{\a_2}_x)\cdot \id$. On the other hand, if $\Theta_{\a_2,\a_1}$ is a cycle generating the top degree of homology, then $\d(\Theta_{\a_2,\a_1}) =0$ since $\Theta_{\a_2,\a_1}$ is a cycle and $A_x (\Theta_{\a_2,\a_1})=0$ since $x$ is basepoint-esque for $\cL_{\a}$ so there are no Maslov index 1 and nonnegative homology classes of disks in any $\pi_2(\Theta_{\a_2,\a_1},\xs)$ with multiplicity 1 on $x$. Since $\ve{\CF}^-(\as_2,\as_1)$ is free and $0=[\d, A_x](\Theta_{\a_2,\a_1})=\Theta_{\a_2,\a_1}\cdot (E^{\a_1}_x+E^{\a_2}_x)$, it follows that $E^{\a_1}_x+E^{\a_2}_x=0$, completing the proof. 
\end{proof}

The main result of this section is the following:

\begin{prop}\label{prop:connected-sums} Suppose that $\cL_{\a}$ and $\cL_{\b}$ are algebraically rigid hypercubes of handleslide-equivalent attaching curves on $(\Sigma,\ws,\zs,\ps)$ and $\cL_{\a'}$ and $\cL_{\b'}$ are algebraically rigid hypercubes of handleslide-equivalent attaching curves on $(\Sigma',\ws',\zs',\ps')$. Suppose that we form $(\Sigma\# \Sigma',\ws'',\zs'',\ps'')$ by an admissible connected sum of $(\Sigma,\ws,\zs,\ps)$ and $(\Sigma',\ws',\zs',\ps')$ at points $x\in \Sigma$ and $x'\in \Sigma'$. Form the hypercubes $\cL_{\a}\otimes \cL_{\a'}$ and $\cL_{\b}\otimes \cL_{\b'}$ as described above. If $x$ is basepoint-esque for $\cL_{\b}$ and $x'$ is basepoints-esque for $\cL_{\a'}$,  then there is a natural homotopy equivalence of hypercubes
\[
\begin{split}
&\ve{\CF}^-(\Sigma\# \Sigma',\cL_{\a}\otimes \cL_{\a'}, \cL_{\b}\otimes \cL_{\b'}, \ws'',\zs'',\ps'')\\
\simeq   &\ve{\CF}^-(\Sigma, \cL_{\a},\cL_{\b},\ws,\zs,\ps)\otimes_R \ve{\CF}^-(\Sigma',\cL_{\a'},\cL_{\b'},\ws',\zs',\ps').
\end{split}
\]
 Also, $R$ is as follows:
\begin{enumerate}
\item If the connected sum is taken along a link component, then $R=\bF[\scU,\scV]$ is the ring for the link components along which the connected sum is taken.
\item  If the connected sum is taken at a free basepoint $x=p_i\in \ps$, then the tensor product is taken over $\bF[U]$, where we have $U$ act by $U_i$ on the left factor, and $E_{x'}^{\a'}$ on the right factor. (Cf. Lemma~\ref{lem:action-of-variables-well-defined}).
\item  If the connected sum is taken at a free basepoint $x'=p_i'\in \ps'$ in $\ws'\cup \zs'$, the tensor product is taken over $\bF[U]$ where we have $U$ act by $U_i$ on the right factor and $E_{x}^{\b}$ on the left factor.
\end{enumerate}
\end{prop}

Unlike in Section~\ref{sec:disjoint-unions}, the homotopy equivalence of Proposition~\ref{prop:connected-sums} is not given by the identity map on the level of groups. Instead it involves counting holomorphic curves. 

\begin{rem} If instead $x$ is basepoint-esque for $\cL_{\a}$ and $x'$ is basepoint-esque for $\cL_{\b'}$, then there is also a homotopy equivalence of hypercubes realizing a tensor product formula. Note that in the case that if $x$ and $x'$ are basepoint-esque for all of the hypercubes of Lagrangians, then there are two distinct maps of hypercubes realizing the homotopy equivalence. These are referred to as $\Phi^{\bullet<\circ'}$ and $\Phi^{\circ'<\bullet}$ below.
\end{rem}

\subsection{Cylindrical boundary degenerations}

In this section, we recall some important facts about boundary degenerations.

\begin{define}
 Suppose $\bs$ is a set of attaching curves on $(\Sigma,\ws)$, and $\xs\in \bT_{\b}$.  A \emph{cylindrical beta boundary degeneration} at $\xs$ consists of a tuple $(u,S,j)$ such that $(S,j)$ is a Riemann surface and
 \[
 u\colon (S,\d S)\to (\Sigma\times [0,\infty)\times \R, \bs \times \{0\}\times \R)
 \]
 is a map  satisfying the following:
 \begin{enumerate}
 \item $u$ is $(j,J)$-holomorphic.
 \item $u$ is proper.
 \item For each $t\in \R$ and $i\in \{1,\dots, g(\Sigma)+|\ws|-1\}$, the set $u^{-1}(\b_i\times \{0\}\times \{t\})$ consists of a single point.
 \item $u$ has finite energy.
 \item $\pi_{\bH}\circ u$ is non-constant on each component of $S$, where $\bH=[0,\infty)\times \R$.
 \item $S$ has a collection of $n=g(\Sigma)+|\ws|-1$ boundary punctures $p_1,\dots, p_n$. If $\xs=(x_1,\dots, x_n)$, where $x_i\in \b_i$, then
 \[
 \lim_{z\to p_i} (\pi_{\bH}\circ u)(z)=\infty\quad \text{and} \quad \lim_{z\to p_i} (\pi_{\Sigma}\circ u)(z)=x_i.
 \]
 \end{enumerate}
\end{define}

Cylindrical alpha boundary degenerations are defined by analogy. Of fundamental importance is the mod 2 count of boundary degenerations.:

\begin{prop} \label{prop:count-boundary-degenerations} Suppose $\bs$ is a set of attaching curves on $(\Sigma,\ws)$, and $B$ is a Maslov index 2 class of boundary degenerations. For an appropriate choice of almost complex structures on $\Sigma\times [0,\infty)\times \R$, the moduli space of boundary degenerations $\cN(B,\ve{x})$ is transversely cut out. Furthermore, the parametrized moduli space
\[
\bigcup_{\xs\in \bT_{\b}} \cN(B,\ve{x})\times \{\ve{x}\}
\]
is also transversely cut out.  Furthermore
\[
\# \cN(B,\ve{x})/\Aut(\bH)\equiv 1.
\] 
\end{prop}

The case when $|\ws|>1$ is proven by Ozsv\'{a}th and Szab\'{o} \cite{OSLinks}*{Theorem~5.5}. Ozsv\'{a}th and Szab\'{o} also proved that if $|\ws|=1$ there are generically no boundary degenerations for split almost complex structures on $\Sigma\times [0,\infty)\times \R$. However for their boundary degenerations, the parametrized moduli space $\bigcup_{\xs\in \bT_{\b}} \cN(B,\ve{x})\times \{\ve{x}\}$ might in general be non-empty, and there are always broken boundary degenerations at which transversality is not achieved (e.g. the union of a closed curve representing $\Sigma$ and a constant disk at $\xs$). This is sufficient for showing that $\d^2=0$, but is not sufficient for our purposes. The $|\ws|=1$ case was revisited in \cite{HHSZExact}*{Section~7.6}, where it was shown that for appropriately generic choices of almost complex structure, the count of Proposition~\ref{prop:count-boundary-degenerations} holds.

\subsection{The 0-dimensional case}
\label{sec:connected-sums-diagrams}

In this section, we present the proof of Proposition~\ref{prop:connected-sums} in the case that $\cL_{\b}$ and $\cL_{\b'}$ are both 0-dimensional (i.e. we recover the standard connected sum formula of Ozsv\'{a}th and Szab\'{o} \cite{OSProperties}). The analytic details we present here will be the basis of the higher dimensional cases of Proposition~\ref{prop:connected-sums}, which we consider in the subsequent section. Our argument is inspired by work of Ozsv\'{a}th and Szab\'{o} in the bordered setting \cite{OSBorderedHFK}. See also \cite{HHSZExact}*{Section~19.4}.

 Suppose $\cH=(\Sigma,\as,\bs,\ws,\zs)$ and $\cH'=(\Sigma',\as',\bs',\ws',\zs')$ are multi-pointed diagrams, and suppose that $J$ and $J'$ are almost complex structures on $\Sigma\times [0,1]\times \R$ and $\Sigma'\times [0,1]\times \R$, respectively. The notation $J\wedge J'$ means the data of the pair $J$ and $J'$ together with distinguished connected sum points, $x\in \Sigma$ and $x'\in \Sigma'$. We assume that the connected sum is admissible, in the sense of Definition~\ref{def:admissible-connected-sum}.

\begin{rem}
In Proposition~\ref{prop:connected-sums} we additionally had the assumption that $x$ and $x'$ were basepoint-esque for some of the hypercubes of attaching curves. In the present case, $\cL_{\a}$, $\cL_{\a'}$, $\cL_{\b}$ and $\cL_{\b'}$ are all 0-dimensional, and this condition is automatic.
\end{rem}

For notational simplicity, we focus on the case where we have just one $\scU$ variable, and one $\scV$ variable, no free-basepoints, and we are doing a connected sum of link components.  The first step is to do a neck-stretching degeneration along the connected sum tube. The result of this degeneration is a chain complex, freely generated over the ground ring $R=\bF[\scU,\scV]$ by pairs $\xs\times \xs'$, where $\xs\in \bT_{\a}\cap \bT_{\b}$ and $\xs'\in \bT_{\a'}\cap \bT_{\b'}$. The differential is given by the formula
\begin{equation}
\d_{J\wedge J'}(\xs\times \xs')= \sum_{\substack{\ys\times \ys'\in (\bT_{\a}\cap \bT_{\b})\times (\bT_{\a'}\cap \bT_{\b'})\\ \phi\in \pi_2(\xs,\ys) \\
\phi'\in \pi_2(\xs',\ys')\\
n_{x}(\phi)=n_{x'}(\phi')\\ 
\mu(\phi)+\mu(\phi')-2n_x(\phi)=1
}}\# (\cM\cM(\phi,\phi')/\R) \cdot \scU^{n_{\ws''}(\phi+\phi')}\scV^{n_{\zs''}(\phi+\phi')} \cdot \ys\times \ys'
\label{eq:del-J-wedge-J'}
\end{equation}
where $ \cM\cM(\phi,\phi')$ is the \emph{perfectly matched moduli space}, consisting of pairs $(u,u')$ satisfying the following matching condition. In the following, we write $(S^{\qs},j)$ and $(T^{\qs'}, j')$ for the sources curves of $u$ and $u'$, where $(S,j)$ and $(T,j')$ are Riemann surfaces, and $\qs$ and $\qs'$ are ordered collections of $n_{x}(\phi)=n_{x'}(\phi)$ marked points, for which we write $\qs=\{q_1,\dots, q_n\}$ and $\qs'=\{q_1,\dots, q_n'\}$. We define
\begin{equation}
\cM\cM_{J\wedge J'}(\phi,\phi'):=\left\{(u,u')\middle\vert \begin{array}{l} u \text{ is $J$-holomorphic}\\
u'\text{ is $J'$-holomorphic}\\
(\pi_{\Sigma}\circ u)(q_i)=x, (\pi_{\Sigma'}\circ u')(q_i')=x'\\
(\pi_{\bD}\circ u)(q_i)=(\pi_{\bD}\circ u')(q_i')
\end{array}  \right\}.
\label{eq:infinity-matched-moduli-spaces}
\end{equation}
It is not hard to see that this gives a chain complex, which we denote by $\ve{\CF}^-_{J\wedge J'}(\cH,\cH')$. (Compare \cite{HHSZExact}*{Section 19.4}). The construction is inspired by the matched moduli spaces which appear in \cite{LOTBordered}*{Section~9.1}.

If $I$ is a (non-singular) almost complex structure on $\Sigma\# \Sigma'\times [0,1]\times \R$, then one may define a chain homotopy equivalence
\[
\ve{\CF}^-_{I}(\cH\# \cH')\simeq \ve{\CF}^-_{J\wedge J'}(\cH,\cH')
\]
by counting index $0$ curves for a non-cylindrical almost complex structure on $\Sigma\# \Sigma'\times [0,1]\times \R$ which interpolates an ordinary cylindrical almost complex structure on $\Sigma\# \Sigma'\times [0,1]\times (-\infty,t_1]$, and a degenerate almost complex structure (i.e. one with infinite neck length) on $\Sigma\wedge \Sigma'\times [0,1]\times [t_2,\infty)$, for some $t_1\ll 0$ and $t_2\gg 0$. See \cite{HHSZExact}*{Section~19.4} for more details. 

We now describe a chain homotopy equivalence
\[
\Phi^{\bullet<\circ'}\colon \ve{\CF}^-_{J\wedge J'}(\cH,\cH')\to \ve{\CF}^-_J(\cH)\otimes_R \ve{\CF}^-_{J'}(\cH').
\]
If $t_0\in \R$, write 
\[
\lambda_{t_0}:=[0,1]\times \{t_0\}\subset [0,1]\times \R.
\]
Let us write
\[
s\colon \Sigma\times [0,1]\times \R\to [0,1]\quad \text{and} \quad  t\colon \Sigma\times [0,1]\times \R\to \R
\]
for the projection maps.

\begin{define}
\label{def:Phi-matched-disk} Suppose $t_0\in \R$. A \emph{$(\Phi,\bullet<\circ',t_0)$-matched $J\wedge J'$-holomorphic curve pair} consists of a pair of marked $J\wedge J'$ holomorphic disks $(u,u')$, equipped with the following data:
\begin{enumerate}
\item A partition of the marked points of $u$ (resp. $u'$) into three sets, $\ve{S}$, $\ve{C}$ and $\ve{N}$ (resp. $\ve{S}'$, $\ve{C}'$ and $\ve{N}'$).
\item A bijection $\phi\colon \ve{S}\to \ve{S}'$.
\end{enumerate}
We assume the following are satisfied:
\begin{enumerate} 
\item If $q_i$ is a marked point of $u$, then $(\pi_{\Sigma}\circ u)(q_i)=x$, and similarly for all marked points of $u'$
\item If $q\in \ve{S}$ and $q'\in \ve{S}'$ and $\phi(q)=q'$, then 
\[
(t\circ u)(q)=(t\circ u')(q')<0 \quad \text{and} \quad (s\circ u)(q)=(s\circ u')(q')
\]
\item If $q\in \ve{N}$ then $(t\circ u)(q)>t_0$. The analogous statement holds for $q'\in \ve{N}'$.
\item $|\ve{C}|=|\ve{C}'|$. Furthermore, if $q\in \ve{C}$ then $(t\circ u)(q)=t_0$, and similarly for the marked points of $\ve{C}'$. The marked points of $\ve{C}$ and $\ve{C}'$ alternate between those of $\ve{C}$ and those of $\ve{C}'$ along $\lambda_{t_0}$. Finally, the left-most marked point along this line is contained in $\ve{C}$.
\end{enumerate}
\end{define}

See Figure~\ref{fig:5} for a schematic of a $(\Phi,\bullet<\circ',t_0)$-matched curve pair.

\begin{figure}[ht]
\centering
\begingroup%
  \makeatletter%
  \providecommand\color[2][]{%
    \errmessage{(Inkscape) Color is used for the text in Inkscape, but the package 'color.sty' is not loaded}%
    \renewcommand\color[2][]{}%
  }%
  \providecommand\transparent[1]{%
    \errmessage{(Inkscape) Transparency is used (non-zero) for the text in Inkscape, but the package 'transparent.sty' is not loaded}%
    \renewcommand\transparent[1]{}%
  }%
  \providecommand\rotatebox[2]{#2}%
  \newcommand*\fsize{\dimexpr\f@size pt\relax}%
  \newcommand*\lineheight[1]{\fontsize{\fsize}{#1\fsize}\selectfont}%
  \ifx\svgwidth\undefined%
    \setlength{\unitlength}{103.84852612bp}%
    \ifx\svgscale\undefined%
      \relax%
    \else%
      \setlength{\unitlength}{\unitlength * \real{\svgscale}}%
    \fi%
  \else%
    \setlength{\unitlength}{\svgwidth}%
  \fi%
  \global\let\svgwidth\undefined%
  \global\let\svgscale\undefined%
  \makeatother%
  \begin{picture}(1,0.9509635)%
    \lineheight{1}%
    \setlength\tabcolsep{0pt}%
    \put(0,0){\includegraphics[width=\unitlength,page=1]{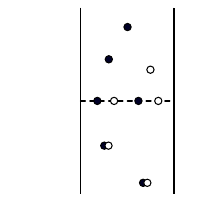}}%
    \put(0.82371089,0.46846761){\makebox(0,0)[lt]{\lineheight{1.25}\smash{\begin{tabular}[t]{l}$\lambda_{t_0}$\end{tabular}}}}%
    \put(0,0){\includegraphics[width=\unitlength,page=2]{fig5.pdf}}%
    \put(0.14784776,0.4534033){\color[rgb]{0,0,0}\makebox(0,0)[rt]{\lineheight{1.25}\smash{\begin{tabular}[t]{r}$\ve{C}$\end{tabular}}}}%
    \put(0.22519521,0.23847907){\color[rgb]{0,0,0}\makebox(0,0)[rt]{\lineheight{1.25}\smash{\begin{tabular}[t]{r}$\ve{S}$\end{tabular}}}}%
    \put(0.22519521,0.68703969){\color[rgb]{0,0,0}\makebox(0,0)[rt]{\lineheight{1.25}\smash{\begin{tabular}[t]{r}$\ve{N}$\end{tabular}}}}%
  \end{picture}%
\endgroup%

\caption{The projection to $[0,1]\times \R$ of marked points of a $(\Phi,\bullet<\circ',t_0)$-matched holomorphic curve pair.  Solid dots indicate the marked points of $u$, while open dots are the marked points of $u'$. The map $\Psi^{\bullet<\circ'}$ counts similar configurations.}
\label{fig:5}
\end{figure}

If $(\phi,\phi')$ are a pair of homology classes, $M$ is a pair of marked source curves for $\phi$ and $\phi'$, and $t_0\in \R$, we write $\cM\cM^{\Phi,\bullet<\circ'}(\phi,\phi',M,t_0)$ for the moduli space of $(\Phi,\bullet<\circ',t_0)$-matched disks representing $\phi$ and $\phi'$. We consider the parametrized moduli space
\[
\cM\cM^{\Phi,\bullet<\circ'}(\phi,\phi',M)=\bigcup_{t_0\in \R} \cM\cM^{\Phi,\bullet<\circ'}(\phi,\phi',M,t_0)\times \{t_0\}.
\]
If $t_0$ is fixed, the expected dimension of $\cM\cM^{\Phi,\bullet<\circ'}(\phi,\phi',M,t_0)$ is given by the formula
\[
\ind(\phi,\phi',M)=\mu(\phi)+\mu(\phi')-2(|\ve{S}|+|\ve{C}|).
\]

The map $\Phi^{\bullet<\circ'}$ counts pairs of curves $(u,u')$ in $\cM\cM^{\Phi,\bullet<\circ'}(\phi,\phi',M)/\R$, ranging over tuples with $\ind(\phi,\phi',M)=0$. Equivalently, we can think of $\Phi^{\bullet<\circ'}$ as counting elements of the non-paramatrized moduli spaces $\cM\cM^{\Phi,\bullet<\circ'}(\phi,\phi',M,t_0)$ for some fixed $t_0$. (The parametrized perspective becomes helpful later).

Next, we explain the $\scU$ and $\scV$ weights of a curve. The basepoints not involved in the connected sum contribute powers of variables as normal. The marked points in $\ve{S}\cup \ve{S}'\cup \ve{C}\cup \ve{C}'$ contribute algebra elements which are the same as would be counted on the connected sum of the two diagrams, where we have removed any basepoints used for the connected sum. For the marked points in $\ve{N}$ and $\ve{N}'$, the algebraic contribution is the same as on the disjoint union of the two diagrams, where we treat the punctures as basepoints.

Concretely, if the connected sum is formed along a link component, and $x=w\in \ws$ and $x'=z'\in \zs'$, then the algebra element contributed by the marked points would be
\begin{equation}
\scU^{|\ve{N}|}\scV^{|\ve{N}'|}.\label{eq:algebra-contribution-North}
\end{equation}
(There would be an additional contributions from the other basepoints not involved in the connected sum). If instead the connected sum is formed at a free basepoint $x=w\in \Sigma$ and some non-basepoint $x'\in \Sigma'$, then the algebra contribution from the marked points would be $\scU^{|\ve{N}|}.$

We define a similar map 
\[
\Psi^{\bullet<\circ'}\colon \ve{\CF}^-_J(\cH)\otimes_R \ve{\CF}^-_{J'}(\cH')\to \ve{\CF}^-_{J\wedge J'}(\cH,\cH')
\]
by counting holomorphic curves with similar conditions, but with the roles of the $\ve{S}$ and $\ve{N}$ labeled marked points switched.

\begin{rem}
Constant disks are counted by both $\Phi^{\bullet<\circ'}$ and $\Psi^{\bullet<\circ'}$.
\end{rem}

\begin{lem}
\label{lem:PhiPsi-chain-maps}
 The maps $\Phi^{\bullet<\circ'}$ and $\Psi^{\bullet<\circ'}$ are chain maps.
\end{lem}

\begin{proof}We focus on $\Phi^{\bullet<\circ'}$ since the argument for $\Psi^{\bullet<\circ'}$ is not substantially different. Our proof is modeled on work of Ozsv\'{a}th and Szab\'{o} in the setting of bordered knot Floer homology \cite{OSBorderedHFK}. The proof is to count the ends of $(\Phi,\bullet<\circ',t_0)$-matched moduli spaces for triples $(\phi,\phi',M)$ with $\ind(\phi,\phi',M)=1$. Modulo the $\R$-action on the parametrized moduli space, such moduli spaces are 1-dimensional. In the present situation, it is sufficient to  fix $t_0=0$, and consider only $(\Phi,\bullet<\circ',0)$-matched holomorphic curves. The ends are constrained generically to the following configurations:
\begin{enumerate}[label=($\Phi$-\arabic*), ref=$\Phi$-\arabic*]
\item \label{dM:Phi-1} Two paired marked points in $\ve{S}$ and $\ve{S}'$ may collide with the line $\lambda_{0}$, away from the marked points in $\ve{C}$ and $\ve{C}'$.
\item \label{dM:Phi-2}A pair of punctures in $\ve{C}$ and $\ve{C}'$ may collide along $\lambda_0$.
\item \label{dM:Phi-3} A puncture of $\ve{N}$ or $\ve{N}'$ may collide with $\lambda_0$ (in the complement of $\ve{C}$ and $\ve{C}'$).
 There are two subcases:
\begin{enumerate}
\item \label{dM:Phi-3-a} After the degeneration, the marked points along $\lambda_0$ do not alternate between those of  $u$ and $u'$.
\item \label{dM:Phi-3-b} After the degeneration, the marked points along $\lambda_0$ do alternate between those of $u$ and $u'$.
\end{enumerate}
\item \label{dM:Phi-4} A puncture of $\ve{C}$ may degenerate into an index 2 beta boundary degeneration at height $t=0$.
\item \label{dM:Phi-5} A puncture of $\ve{C}'$ may degenerate into an index 2 alpha boundary degeneration at height $t=0$.
\item \label{dM:Phi-6} Strip breaking may occur, leaving a $(\Phi,\bullet<\circ',0)$-matched disk of index 0 as well as a holomorphic disk of index 1 in either positive or negative direction. There are two subcases:
\begin{enumerate}
\item The index 1 holomorphic disk degenerates towards $-\infty$, and is perfectly-matched.
\item The index 1 holomorphic disk degenerates towards $+\infty$, and has trivial matching (i.e. the projections of all marked points to $[0,1]\times \R$ are distinct).
\end{enumerate}
\end{enumerate}
Most of these ends appear in canceling pairs, and the rest correspond to the relation $[\d, \Phi^{\bullet<\circ'}]=0$, as we describe presently.

 The ends \eqref{dM:Phi-1} cancel with the ends~\eqref{dM:Phi-2}.

 In an end of type~\eqref{dM:Phi-3-a}, there are two adjacent marked points along $\lambda_0$ which are both from $u$ or both from $u'$. Such a curve appears twice in the boundary of the moduli spaces. See Figure~\ref{fig:1}. 
 
 The ends of type~\eqref{dM:Phi-3-b} cancel with the end of type~\eqref{dM:Phi-4} and~\eqref{dM:Phi-5}. Here we are using the mod 2 count of boundary degenerations from Proposition~\ref{prop:count-boundary-degenerations}. Note also the algebra contributions coincide for the two canceling degenerations of type~\eqref{dM:Phi-3-b} and type~\eqref{dM:Phi-4} or~\eqref{dM:Phi-5} (cf. Equation ~\eqref{eq:algebra-contribution-North}). If the connected sum is taken along a link component, each boundary degeneration which forms along $\lambda_{t_0}$ will contain one puncture from $\ve{C}\cup \ve{C}'$ and also one link basepoint. The degeneration which cancels this boundary degeneration formation consists of a $\ve{N}$ or $\ve{N}'$ puncture colliding with $\lambda_{t_0}$. The algebra weight from the $\ve{N}$ or $\ve{N}'$ puncture coincides with the weight of the basepoint which is lost in the boundary degeneration. The case of a connected sum at a free basepoint is similar.
  
 The remaining ends are ~\eqref{dM:Phi-6}, which correspond exactly to the commutator $[\d,\Phi^{\bullet<\circ'}]$. Summing all ends, we conclude that
 \[
\d \circ \Phi^{\bullet<\circ'}+\Phi^{\bullet<\circ'}\circ \d=0, 
 \]
completing the proof.
\end{proof}

\begin{figure}[ht]
\centering
\begingroup%
  \makeatletter%
  \providecommand\color[2][]{%
    \errmessage{(Inkscape) Color is used for the text in Inkscape, but the package 'color.sty' is not loaded}%
    \renewcommand\color[2][]{}%
  }%
  \providecommand\transparent[1]{%
    \errmessage{(Inkscape) Transparency is used (non-zero) for the text in Inkscape, but the package 'transparent.sty' is not loaded}%
    \renewcommand\transparent[1]{}%
  }%
  \providecommand\rotatebox[2]{#2}%
  \newcommand*\fsize{\dimexpr\f@size pt\relax}%
  \newcommand*\lineheight[1]{\fontsize{\fsize}{#1\fsize}\selectfont}%
  \ifx\svgwidth\undefined%
    \setlength{\unitlength}{199.31810308bp}%
    \ifx\svgscale\undefined%
      \relax%
    \else%
      \setlength{\unitlength}{\unitlength * \real{\svgscale}}%
    \fi%
  \else%
    \setlength{\unitlength}{\svgwidth}%
  \fi%
  \global\let\svgwidth\undefined%
  \global\let\svgscale\undefined%
  \makeatother%
  \begin{picture}(1,1.5644244)%
    \lineheight{1}%
    \setlength\tabcolsep{0pt}%
    \put(0,0){\includegraphics[width=\unitlength,page=1]{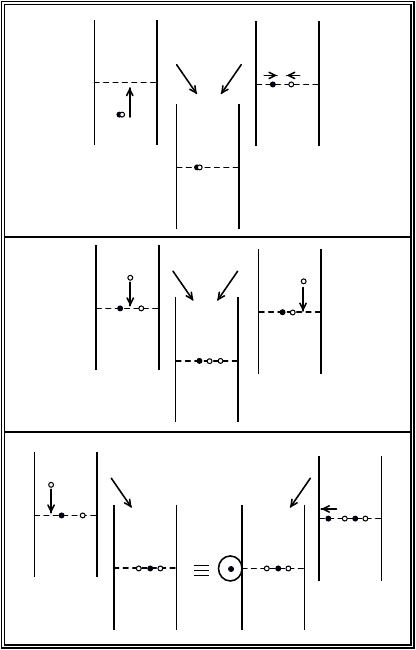}}%
    \put(0.78868351,1.34629059){\makebox(0,0)[lt]{\lineheight{1.25}\smash{\begin{tabular}[t]{l}\eqref{dM:Phi-2}\end{tabular}}}}%
    \put(0.21282375,1.34629066){\makebox(0,0)[rt]{\lineheight{1.25}\smash{\begin{tabular}[t]{r}\eqref{dM:Phi-1}\end{tabular}}}}%
    \put(0.20309398,0.80340384){\makebox(0,0)[rt]{\lineheight{1.25}\smash{\begin{tabular}[t]{r}\eqref{dM:Phi-3-a}\end{tabular}}}}%
    \put(0.15882142,0.13022388){\makebox(0,0)[t]{\lineheight{1.25}\smash{\begin{tabular}[t]{c}\eqref{dM:Phi-3-b}\end{tabular}}}}%
    \put(0.84349151,0.13022388){\makebox(0,0)[t]{\lineheight{1.25}\smash{\begin{tabular}[t]{c}\eqref{dM:Phi-4}\end{tabular}}}}%
    \put(0.79321173,0.80340384){\makebox(0,0)[lt]{\lineheight{1.25}\smash{\begin{tabular}[t]{l}\eqref{dM:Phi-3-a}\end{tabular}}}}%
  \end{picture}%
\endgroup%

\caption{Some cancellations in the proof that $\Phi^{\bullet<\circ'}$ and $\Psi^{\bullet<\circ'}$ are chain maps}
\label{fig:1}
\end{figure}

\begin{lem}
\label{lem:Phi-Psi-homotopy-inverses-diagrams} The maps $\Phi^{\bullet<\circ'}$ and $\Psi^{\bullet<\circ'}$ are homotopy inverses.
\end{lem}
\begin{proof}We define two maps $H_{\wedge}$ and $H_{\sqcup}$, which we prove satisfy
\begin{equation}
\id+\Psi^{\bullet<\circ'}\circ\Phi^{\bullet<\circ'}=[\d, H_\wedge]\quad \text{and} \quad \id+\Phi^{\bullet<\circ'}\circ \Psi^{\bullet<\circ'}=[\d,H_\sqcup].
\label{eq:Phi-Psi-homotopy-equivalence}
\end{equation}
We focus on the map $H_\wedge$, since $H_\sqcup$ is constructed by a straightforward modification.

 The map $H_{\wedge}$ counts certain curve pairs $(u,u')$ where $u$ has five collections of marked points, $\ve{NN}$, $\ve{NC}$, $\ve{M}$, $\ve{SC}$ and $\ve{SS}$, and $u'$ has five collections of marked points $\ve{NN}'$, $\ve{NC}'$, $\ve{M}'$, $\ve{SC}'$ and $\ve{SS}'$.

\begin{define}
 Suppose that $t_0<t_1$ are real numbers. We say a pair of marked holomorphic strips $(u,u')$ is \emph{$(H_\wedge,t_0,t_1)$-matched} if  the following are satisfied:
\begin{enumerate}
\item The marked points of $\ve{NN}$ and $\ve{NN}'$ are perfectly matched. Furthermore their projection to $\R$ lies above $t_1$.
\item The marked points in $\ve{NC}$ and $\ve{NC}'$ both project to $\lambda_{t_1}$. Furthermore, $|\ve{NC}|=|\ve{NC}'|$. As one travels along $\lambda_{t_1}$, the marked points alternate between $|\ve{NC}|$ and $|\ve{NC}'|$, and the left-most marked point is from $\ve{NC}$, while the right most is from $\ve{NC}'$.
\item The marked points of $\ve{M}$ and $\ve{M}'$ have no matching condition. Their projection to $[0,1]\times \R$ lies between $\lambda_{t_0}$ and $\lambda_{t_1}$.
\item The marked points $\ve{SC}$ and $\ve{SC}'$ satisfy the same conditions as $\ve{NC}$ and $\ve{NC}'$, except that they project to $\lambda_{t_0}$.
\item The marked points in $\ve{SS}$ and $\ve{SS}'$ are perfectly matched, and their projection to $\R$ lies below $t_0$. 
\end{enumerate}
\end{define}

\begin{figure}[ht]
\centering
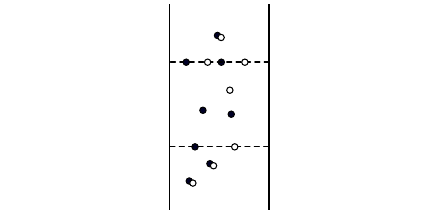
\caption{A $(H_\wedge,t)$-matched curve pair. On the left side, we indicate the regions where the marked points $\ve{NN},$ $\ve{NC},$ $\ve{M},$ $\ve{SC},$ and $\ve{SS}$ are sent.}
\label{fig:21}
\end{figure}

If $(\phi,\phi')$ are two homology classes of disks, equipped with decorated sources and matching data $M$, and $t_0<t_1$ are fixed, then the expected dimension of the moduli space of curves which are $(H_\wedge,t_0,t_1)$-matched and represent $(\phi,\phi',M)$ is given by
\[
\ind(\phi,\phi',M)=\mu(\phi)+\mu(\phi')-2|\ve{SS}|-2|\ve{NN}|-2|\ve{SC}|-2|\ve{NC}|.
\]
Write $\cM\cM(\phi,\phi',M)$ for this moduli space. The parametrized moduli space
\[
\cM\cM(\phi,\phi',M):=\bigcup_{t_0<t_1} \cM\cM(\phi,\phi',M,t_0,t_1)\times \{(t_0,t_1)\}
\]
has a free $\R$-action, corresponding to overall translation (acting diagonally on $t_0$ and $t_1$).

The map $H_\wedge$ counts $\cM\cM(\phi,\phi',M)/\R$ for triples $(\phi,\phi',M)$ satisfying
$
\ind(\phi,\phi',M)=-1.
$
In this case, the expected dimension of $\cM\cM(\phi,\phi',M)$ is 1.

To verify Equation~\eqref{eq:Phi-Psi-homotopy-equivalence}, we count the ends of the parametrized moduli spaces $\cM\cM(\phi,\phi',M)/\R$ where $\ind(\phi,\phi',M)=0$ (so that $\dim \cM\cM(\phi,\phi',M)/\R=1$).

There are ends which occur at $0<t_1-t_0<\infty$. These are the obvious analogs of the ends in the case of one special line \eqref{dM:Phi-1},\dots,\eqref{dM:Phi-5}. They cancel by an identical proof to the case of one special line.

It remains to analyze the curves which appear as $t_1-t_0\to 0$ or $t_1-t_0\to \infty$. The ends which appear as $t_1-t_0\to \infty$ are easy to analyze: they correspond exactly to the composition $\Psi^{\bullet<\circ'}\circ \Phi^{\bullet<\circ'}$. We now focus on the ends which appear as $t_1-t_0\to 0$. These correspond to the lines $\lambda_{t_0}$ and $\lambda_{t_1}$ colliding to form a single line, for which we write $\lambda_{t'}$. The limiting curves live in a moduli space of pairs $(u,u')$ where the points of $\ve{SS}$ are perfectly matched with $\ve{SS}'$ and $\ve{NN}$ are perfectly matched to $\ve{NN}'$, and the rest of the marked points are matched to $\lambda_{t'}$. For fixed $t'$, this moduli space has  expected dimension
\[
\mu(\phi)+\mu(\phi')-2|\ve{SS}|-2|\ve{NN}|-2|\ve{S}\ve{C}|-2|\ve{NC}|-|\ve{M}|-|\ve{M}'|.
\]
By assumption, the above quantity is $-|\ve{M}|-|\ve{M}'|$. The moduli space which is parametrized over all $t'$ has dimension $-|\ve{M}|-|\ve{M}'|+1$, but since it also has a free $\R$ action, we conclude that $-|\ve{M}|-|\ve{M}'|\ge 0$, so $|\ve{M}|=|\ve{M}'|=0$.

There are two remaining cases for the curves appearing as $t_1-t_0\to 0$:
\begin{enumerate}[label=($H_\wedge$-\arabic*), ref=$H_\wedge$-\arabic*]
\item\label{list:H-wedge-ends-0} $|\ve{SC}|=|\ve{NC}|=|\ve{SC}'|=|\ve{NC}'|=0$.
\item\label{list:H-wedge-ends-1} At least one of $|\ve{SC}|,$ $|\ve{NC}|,$ $|\ve{SC}'|$, $|\ve{NC}'|$ is non-empty.
\end{enumerate}
For the ends~\eqref{list:H-wedge-ends-0}, there are no marked points along the special line and the limiting curve pair $(u,u')$ is perfectly matched, similar to the curves counted by $\d_{J\wedge J'}$ in Equation~\eqref{eq:del-J-wedge-J'}, except $(u,u')$ has index 0 (i.e. the expected dimension, ignoring the special line is 0). Using transversality and expected dimension counts, we see that $u$ and $u'$ must both represent the constant class. These ends contribute $\id_{\ve{\CF}^-_{J\wedge J'}(\cH,\cH')}$ to the left equation of ~\eqref{eq:Phi-Psi-homotopy-equivalence}.

We now consider the ends labeled~\eqref{list:H-wedge-ends-1}. We claim that these appear in canceling pairs. Indeed, each such end appears with even multiplicity, corresponding to switching the roles of $\ve{NC}$ and $\ve{SC}$, and switching $\ve{NC}'$ and $\ve{SC}'$ (i.e. having the special lines pass through each other). See Figure~\ref{fig:6}.

\begin{figure}[ht]
\centering
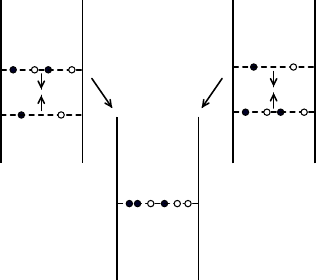
\caption{The cancellation of the ends labeled~\eqref{list:H-wedge-ends-1}.}
\label{fig:6}
\end{figure}

The construction of the homotopy $H_\cup$ is defined via the obvious modification, and a similar analysis goes through to establish the right equation of~\eqref{eq:Phi-Psi-homotopy-equivalence}.
\end{proof}

\subsection{Connected sums and hypercubes}

We now extend the work of the previous section to handle the more general case of hypercubes of dimension greater than 0. We focus on the case that $\cL_{\a}$ and $\cL_{\a'}$ are 0-dimensional, i.e. $\cL_{\a}=\as$ and $\cL_{\a'}=\as'$, while $\cL_\b$ and $\cL_{\b'}$ are of arbitrary dimension.

Similar to the setting of ordinary Floer complexes in Section~\ref{sec:connected-sums-diagrams}, it is straightforward to construct a neck-stretching homotopy equivalence
\[
\ve{\CF}^-_{I}(\Sigma\# \Sigma', \as\cup \as', \cL_{\b}\otimes \cL_{\b'})\simeq \ve{\CF}^-_{J\wedge J'}(\Sigma\wedge \Sigma', \as\cup \as', \cL_{\b}\otimes \cL_{\b'}).
\]
The construction is essentially the same as the construction of a homotopy equivalence of hypercubes for changing the almost complex structure. Compare \cite{HHSZExact}*{Section~14.2}. 

The goal of this section is to describe a homotopy equivalence of hypercubes
\[
\Phi^{\bullet<\circ'}\colon \ve{\CF}^-_{J\wedge J'}(\Sigma\wedge\Sigma', \as\cup \as', \cL_{\b}\otimes \cL_{\b'})\to \ve{\CF}^-_{J\sqcup J'}(\Sigma \sqcup \Sigma', \as\cup \as',\cL_{\b}\otimes \cL_{\b'}).
\]
under certain restrictions on $\cL_{\b}$ and $\cL_{\b'}$. 
In the domain of $\Phi^{\bullet<\circ'}$, the curves counted are perfectly-matched. In codomain, the curves have trivial matching (i.e. they are counted in the normal way for disconnected Heegaard surfaces, as in Section~\ref{sec:disconnected-Heegaard-surfaces}). The map $\Phi^{\bullet<\circ'}$ is an extension of the map defined in the previous section, which we describe momentarily. Note that by Proposition~\ref{prop:disjoint-unions-hypercubes-main}, the codomain of $\Phi^{\bullet<\circ'}$ is homotopy equivalent to the tensor product of $\ve{\CF}^-_J(\Sigma,\as,\cL_{\b})$ and $\ve{\CF}^-_{J'}(\Sigma',\as',\cL_{\b'})$, if the hypercube $\cL_{\b}\otimes \cL_{\b'}$ is constructed with sufficiently small Hamiltonian translations.

Similar to the case of connected sums, we will define the map $\Phi^{\bullet<\circ'}$ to count holomorphic polygons with a special line.

To simplify the exposition and notation, it is helpful to establish a few conventions. Firstly, we recall that the holomorphic $\ell$-gon maps require a choice of a family of almost complex structures $(J_y)_{y\in K_{\ell-1}}$. The space $K_{\ell-1}$ may be identified with the moduli space of disks with $\ell$ marked points along its boundary. For the sake of notation, it is more convenient to consider the moduli space of $\ell-2$ marked points along the line 
\[
\{0\}\times \R\subset [0,1]\times \R.
\]
Overall translation by the $\R$-action gives equivalent elements in the moduli space. For our purposes, it is actually more convenient to not initially quotient by this action, and consider a family of almost complex structures $(J_{y,\tau})_{(y,\tau)\in K_{\ell-1}\times \R}$.  Let $T_r$ denote translation of the $\R$ component of $\Sigma\times [0,1]\times \R$ by $r$ units.
We require that
\[
J_{y,\tau}=(T_r^*)(J_{y,\tau+r}),
\] 
for each $y\in K_{\ell-1}$, and $\tau,r\in \R$. Note that this does \emph{not} imply that each $J_{y,\tau}$ is translation invariant under the $\R$-action on $[0,1]\times \R$. 

With this convention, the normal $\ell$-gon counting maps can be viewed as counting the 1-dimensional parametrized moduli spaces
\[
\bigcup_{(y,\tau)\in K_{\ell-1}\times \R} \cM_{J_y}(\psi)\times \{(y,\tau)\},
\]
modulo their free $\R$-action.

If $t_0\in \R$, the notion of a \emph{$(\Phi,\bullet<\circ',t_0)$}-matched pair of $(\ell+1)$-gons is easily adapted from Definition~\ref{def:Phi-matched-disk}. (Note that to define the map $\Phi^{\bullet<\circ'}$, the indexing is somewhat more natural if we count $(\ell+1)$-gons than $\ell$-gons). Here, $t_0$ denotes the height of the special line, where the marked points $\ve{C}$ and $\ve{C}'$ are matched. If $(\psi,\psi',M)$ is a triple consisting of two homology classes of $(\ell+1)$-gons, and $M$ is a pair of decorated source curves and matching data, and $(y,\tau)\in K_{\ell}\times \R$ and $t_0\in \R$ is fixed, then the expected dimension of the moduli space of $(\Phi,\bullet<\circ',t_0)$-matched $(\ell+1)$-gons, denoted $\cM\cM_{J_{y,\tau}}^{\Phi,\bullet<\circ'}(\psi,\psi',M,t_0)$, is given by
\[
 \ind(\psi,\psi',M)=\mu(\psi)+\mu(\psi')-2|\ve{S}|-2|\ve{C}|. 
\]
This is because the moduli space with $J_{y,\tau}$ fixed but with no puncture constraints has expected dimension $\mu(\psi)+\mu(\psi')$, and the perfect matching below $t_0$ and the matching along the special line give a codimension $2(|\ve{S}|+|\ve{C}|)$ constraint.

If $J=(J_{(y,\tau)})_{(y,\tau)\in K_{\ell}\times \R}$ is a family, we define the parametrized moduli space
\[
\cM\cM_{J}^{\Phi,\bullet<\circ'}(\psi,\psi',M)=\bigcup_{\substack{
(y,\tau)\in K_{\ell}\times \R\\ t_0\in \R}}\cM\cM^{\Phi,\bullet<\circ'}_{J_{(y,\tau)}}(\psi,\psi',M,t_0)\times \{(y,\tau,t_0)\},
\]
which has expected dimension 
\begin{equation}
e\dim \cM\cM_{J}^{\Phi,\bullet<\circ'}(\psi,\psi',M)=\ind(\psi,\psi',M)+\ell,
\label{eq:e-dim-parametrized-w-t-tau}
\end{equation}
 and has a free $\R$-action.

We now define our hypercube morphism $\Phi^{\bullet<\circ'}$.
If $\kappa\le \kappa'$ and $\kappa,\kappa'\in \bE_{n+m}$, we  define the component $\Phi_{\kappa\to \kappa'}^{\bullet<\circ'}$ of $\Phi^{\bullet<\circ'}$ as follows. If $\kappa=\kappa'=(\veps,\nu)$, then the map 
\[
\Phi^{\bullet<\circ'}_{\kappa\to \kappa'}\colon \ve{\CF}^-_{J\wedge J'}(\Sigma\wedge \Sigma', \as\cup \as', \bs_{\veps}\cup \bs'_{\nu})\to \ve{\CF}^-_{J\sqcup J'}(\Sigma\sqcup \Sigma', \as\cup \as', \bs_{\veps}\cup \bs'_{\nu})
\]
is the tensor product map defined in Section~\ref{sec:connected-sums-diagrams}.

If $\kappa_1<\cdots<\kappa_\ell$ is an increasing sequence of points in $\bE_{n+m}$, then we write $\kappa_j=(\veps_j,\nu_j)$ and we define a map 
\[
\varphi^{\bullet<\circ'}_{\kappa_1<\cdots<\kappa_\ell}\colon \ve{\CF}^-_{J\wedge J'}(\Sigma\wedge \Sigma', \as\cup \as', \bs_{\veps_1}\cup \bs'_{\nu_1})\to  \ve{\CF}^-_{J\sqcup J'}(\Sigma\wedge \Sigma', \as\cup \as', \bs_{\veps_\ell}\cup \bs'_{\nu_\ell})
\]
by counting $(\Phi,\bullet<\circ',t_0)$-matched holomorphic $(\ell+1)$-gons which have $\ind(\psi,\psi',M)=1-\ell$, and all of whose special inputs come from $\cL_{\b}\otimes \cL_{\b'}$. By Equation~\eqref{eq:e-dim-parametrized-w-t-tau} the parametrized moduli spaces of such families have expected dimension $1$, and also have a free $\R$-action.

 If $\kappa<\kappa'$, then we define 
\[
\Phi^{\bullet<\circ'}_{\kappa\to \kappa'}:=\sum_{\kappa=\kappa_1<\cdots<\kappa_\ell=\kappa'} \varphi_{\kappa_1<\cdots<\kappa_\ell}.
\]

We may also define a map 
\[
\Psi^{\bullet<\circ'}\colon \ve{\CF}^-_{J\sqcup J'}(\Sigma\sqcup \Sigma', \as\cup \as', \cL_{\b}\otimes \cL_{\b'})\to \ve{\CF}^-_{J\wedge J'}(\Sigma\wedge \Sigma', \as\cup \as', \cL_{\b}\otimes \cL_{\b'})
\]
in the opposite direction via the obvious modification.

\begin{lem}\label{lem:codim-1-Phi-ell} Suppose that $\cL_{\b}$ and $\cL_{\b'}$ are hypercubes of beta attaching curves on $\Sigma$ and $\Sigma'$, respectively, and $x\in \Sigma$ and $x'\in \Sigma'$ are choices of connected sum points. Suppose the following are satisfied:
\begin{enumerate}
\item $\cL_{\b}$ and $\cL_{\b'}$ are algebraically rigid.
\item $x$ is basepoint-esque for $\cL_{\b}$.
\end{enumerate}The following are the generic codimension 1 degenerations in the parametrized moduli spaces of $(\Phi,\bullet<\circ',t_0)$-matched $(\ell+1)$-gon curve pairs on $(\Sigma\wedge \Sigma', \as\cup \as',\cL_\b\otimes \cL_{\b'})$ representing classes with matching data satisfying $\ind(\psi,\psi',M)=2-\ell$ (i.e. of expected dimension $1$ after quotienting by the free $\R$-action), where the special beta inputs are from the hypercube $\cL_{\b}\otimes \cL_{\b'}$:
\begin{enumerate}[label=($\Phi^\ell$-\arabic*), ref=$\Phi^\ell$-\arabic*]
\item\label{Phi-ell-1} Two paired marked points in $\ve{S}$ and $\ve{S}'$ may collide with the special line (away from the marked points $\ve{C}$ and $\ve{C}'$). None of the punctures in $\ve{P}\cup \ve{P}'$ have the same height as the special line. 
\item\label{Phi-ell-2} A pair of punctures of $\ve{C}$ and $\ve{C}'$ may collide along $\lambda_{t_0}$ (and the $\ve{P}$ and $\ve{P}'$ punctures have different heights than $\lambda_{t_0}$).
\item\label{Phi-ell-3} A puncture of $\ve{N}$ or $\ve{N}'$ may collide with $\lambda_{t_0}$ (in the complement of $\ve{C}$ and $\ve{C}'$, and with different height than $\ve{P}$ and $\ve{P}'$).
\item\label{Phi-ell-4} A puncture of $\ve{C}$ may degenerate into an index 2 beta boundary degeneration or alpha boundary degeneration. The height $\lambda_{t_0}$ is different than any puncture in $\ve{P}$ or $\ve{P}'$.
\item\label{Phi-ell-5} The holomorphic pair $(u,u')$ may break into two $\R$-levels, exactly one of which contains the special line.
\item\label{Phi-ell-6} Degenerations may occur which involve the punctures $\ve{P}$ and $\ve{P}'$. These are constrained to the following:
\begin{enumerate}
\item \label{Phi-ell-6a} An index 1 disk pair may bubble at two punctures of $\ve{P}\cup \ve{P}'$. If it degenerates above the special line, then the matching is trivial. If it degenerates below the special line, the matching is perfect.
\item \label{Phi-ell-6b} Four punctures of $\ve{P}\cup \ve{P}'$ collide and bubble off a pair of index 0 holomorphic triangles. If they form below the special line, the matching is perfect. If they occur above the special line, the matching is trivial. 
\end{enumerate}
\end{enumerate}
See Figure~\ref{fig:7} for a schematic of the above degenerations.
\end{lem}

\begin{figure}[ht]
\centering
\begingroup%
  \makeatletter%
  \providecommand\color[2][]{%
    \errmessage{(Inkscape) Color is used for the text in Inkscape, but the package 'color.sty' is not loaded}%
    \renewcommand\color[2][]{}%
  }%
  \providecommand\transparent[1]{%
    \errmessage{(Inkscape) Transparency is used (non-zero) for the text in Inkscape, but the package 'transparent.sty' is not loaded}%
    \renewcommand\transparent[1]{}%
  }%
  \providecommand\rotatebox[2]{#2}%
  \newcommand*\fsize{\dimexpr\f@size pt\relax}%
  \newcommand*\lineheight[1]{\fontsize{\fsize}{#1\fsize}\selectfont}%
  \ifx\svgwidth\undefined%
    \setlength{\unitlength}{248.63065796bp}%
    \ifx\svgscale\undefined%
      \relax%
    \else%
      \setlength{\unitlength}{\unitlength * \real{\svgscale}}%
    \fi%
  \else%
    \setlength{\unitlength}{\svgwidth}%
  \fi%
  \global\let\svgwidth\undefined%
  \global\let\svgscale\undefined%
  \makeatother%
  \begin{picture}(1,0.88914271)%
    \lineheight{1}%
    \setlength\tabcolsep{0pt}%
    \put(0,0){\includegraphics[width=\unitlength,page=1]{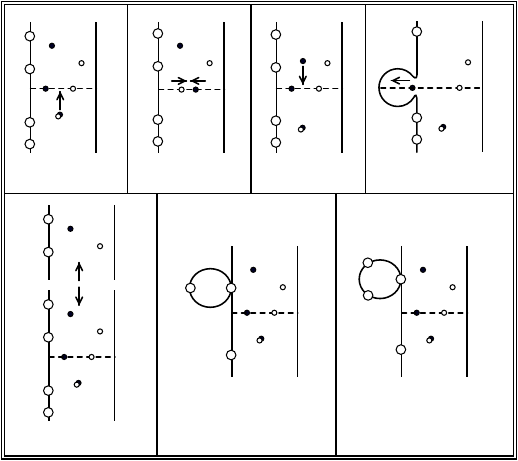}}%
    \put(0.11210108,0.5532816){\makebox(0,0)[t]{\lineheight{1.25}\smash{\begin{tabular}[t]{c}\eqref{Phi-ell-1}\end{tabular}}}}%
    \put(0.1551847,0.03917535){\makebox(0,0)[t]{\lineheight{1.25}\smash{\begin{tabular}[t]{c}\eqref{Phi-ell-5}\end{tabular}}}}%
    \put(0.3578158,0.5532816){\makebox(0,0)[t]{\lineheight{1.25}\smash{\begin{tabular}[t]{c}\eqref{Phi-ell-2}\end{tabular}}}}%
    \put(0.59348934,0.5532816){\makebox(0,0)[t]{\lineheight{1.25}\smash{\begin{tabular}[t]{c}\eqref{Phi-ell-3}\end{tabular}}}}%
    \put(0.85302849,0.5532816){\makebox(0,0)[t]{\lineheight{1.25}\smash{\begin{tabular}[t]{c}\eqref{Phi-ell-4}\end{tabular}}}}%
    \put(0.48410505,0.03917535){\makebox(0,0)[t]{\lineheight{1.25}\smash{\begin{tabular}[t]{c}\eqref{Phi-ell-6a}\end{tabular}}}}%
    \put(0.8383682,0.03917535){\makebox(0,0)[t]{\lineheight{1.25}\smash{\begin{tabular}[t]{c}\eqref{Phi-ell-6b}\end{tabular}}}}%
  \end{picture}%
\endgroup%

\caption{The generic degenerations of $(\Phi,\bullet<\circ',t_0)$-matched holomorphic $\ell$-gon pairs. The large boundary dots denote punctures of $\ve{P}$ and $\ve{P}'$.}
\label{fig:7}
\end{figure}
\begin{proof}The claim that these degenerations have codimension 1 is clear. We claim that there are no further degenerations. There are two points which require explanation:
\begin{enumerate}[label=($P$-\arabic*), ref=$P$-\arabic*]
\item\label{putative-ell-gon-1} There are generically no degenerations involving the punctures $\ve{P}\cup \ve{P}'$ which leave a puncture along $\{0\}\times \R$ of the same height as the special line $\lambda_{t_0}$.
\item \label{putative-ell-gon-2}The generic degenerations involving the punctures $\ve{P}$ or $\ve{P}'$ along $\{0\}\times \R$ (which necessarily occur at heights other than that of the special line by ~\eqref{putative-ell-gon-1}) are constrained to those listed in~\eqref{Phi-ell-6}. In particular, they are constrained only to index 1 disk bubbling and index 0 triangle degenerations. Such curves are perfectly matched if they occur below the special line, and trivially matched if they occur above the special line. Degenerations of pairs of $j$-gons at these punctures are prohibited for $j>3$.
\end{enumerate}
See Figure~\ref{fig:2} for a schematic of a degeneration which we claim is non-generic.

We now argue that degenerations ~\eqref{putative-ell-gon-1} are non-generic. Consider a pair $(u,u')$ of broken curves obtained as the limit of a 1-parameter family of $(\Phi,\bullet<\circ',t_0)$-matched $(\ell+1)$-gons, representing a pair of classes $(\psi,\psi')$, all of whose special inputs are the top degree generators. Assume that each of $u$ and $u'$ consist of a $j$-gon and a $k$-gon for
\begin{equation}
j+k=\ell+3.\label{eq:j-k-gon-breaking}
\end{equation}
 Such a degeneration corresponds to the codimension 1 strata of $K_{\ell}$. Our argument extends easily to further degenerations of into a tree of curves, corresponding to the higher codimension strata of $K_{\ell}$, though we focus on the codimension 1 strata to simplify the notation.

  Suppose that $u$ degenerates into a pair $u_r$ and $u_l$, where $u_r$ is a map into $\Sigma\times [0,1]\times \R$, and we view $u_l$ as a map into $\Sigma \times \bH$ where $\bH=[0,\infty)\times \R$. Here, $u_r$ and $u_l$ have additional boundary punctures, and we think of $u_r$ as a holomorphic $k$-gon, and we think of $u_l$ as a holomorphic $j$-gon.  We assume that $u'$ similarly consists of a pair $u'_l$ and $u_r'$, which are holomorphic $j$-gons and $k$-gons.
  
We assume for the sake of the argument that $j,k \ge 2$. Note that we automatically have $k\ge 2$, since there are the punctures at $\pm \infty$. The case that $j=1$ corresponds to a boundary degeneration. In this case, we may fill in the boundary puncture for the curves $u_r$ and $u_r'$ and obtain an $(\ell+1)$-gon. We leave it to the reader to analyze the case of boundary degenerations by adapting the index arguments we present.

\begin{figure}[ht]
\centering
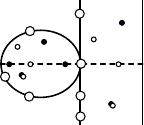
\caption{A putative example of a degeneration of type~\eqref{putative-ell-gon-1}. Big dots are boundary punctures. Small dots are marked points. The left disk is identified with $\bH$.}
\label{fig:2}
\end{figure}

In the limit, write $\ve{N}_l,$ $\ve{N}_r$, (resp. $\ve{C}_l,$ $\ve{C}_r$) (resp. $\ve{S}_l$, $\ve{S}_r$) for the marked points of $u_l$ and $u_r$ that lie above above, (resp. on) (resp. below) the special line. Write $\ve{N}'_l$ (and so forth) for the analogous marked points of $u_l'$ and $u_r'$. Here, we use the same designation of $\ve{N}$, $\ve{C}$ and $\ve{S}$ as from the 1-parameter family (so in principle, it is possible for a $\ve{N}$-marked point to lie on the special line, though it will follow from our argument that these degenerations occur in too high of codimension to be generic).

By construction of the 1-parameter family relevant to the degeneration, we have
\begin{equation}
\mu(\psi)+\mu(\psi')+\ell-2-2(|\ve{C}|+|\ve{S}|)=0.
\label{eq:Maslov-index-l-gon}
\end{equation}
For these families, the parametrized moduli space has expected dimension 2 by Equation~\eqref{eq:e-dim-parametrized-w-t-tau}, and also has a free $\R$-action. We note that $|\ve{N}|=|\ve{N}_l|+|\ve{N}_r|$, and similarly for the other collections of marked points. 

We may view the pair $(u_l,u_l')$ as determining a map to $(\Sigma\sqcup \Sigma')\times[0,\infty)\times \R$.  Excluding the puncture at $\infty$, there are $j-1$ punctures along $\{0\}\times \R$. We normalize the limit so that the special line coincides with $[0,\infty)\times \{0\}$. There is a 1-parameter family of conformal automorphisms of the half plane which preserve this line. They are multiplication by positive real numbers.
 
We claim that
  \begin{equation}
 \mu(\psi_l)+\mu(\psi_l')+j-2-2|\ve{S}_l|-\max(|\ve{C}_l|+|\ve{C}_{l}'|,1)\ge 0.
 \label{eq:expected-dimension-psi-l}
  \end{equation}
  The reasoning for the above equation is as follows.  We may view the limiting curves as being elements of a moduli space of holomorphic curves which is parametrized by the moduli space of $j-1$ marked points along $\{0\}\times \R$. This is a $j-1$ dimensional moduli space, however there is a free action by $\R\times (0,\infty)$. The $\R$-factor translates in the $i$-direction, and the $(0,\infty)$-factor scales by a real positive number. If $|\ve{C}_l|+|\ve{C}_l'|=0$, then both of these actions preserve the moduli space of curves which are parametrized over the moduli space of $j-1$ points along $\{0\}\times \R$, and normalized so that the special line is $\{0\}\times [0,\infty)$. This parametrized moduli space has expected dimension $\mu(\psi_l)+\mu(\psi_l')+j-1-2|\ve{S}_l|$. Hence $\mu(\psi_l)+\mu(\psi_l')+j-3-2|\ve{S}_l|\ge 0$ if this moduli space is non-empty, as claimed. If $|\ve{C}_l|+|\ve{C}_l'|\ge 1$, then the $\R$-action which translates in the $i$-direction no longer preserves the matched moduli space, but the $(0,\infty)$-action does preserve the matched moduli space, so instead we only obtain 
  \[
  \mu(\psi_l)+\mu(\psi_l')+j-2-2|\ve{S}_l|-|\ve{C}_l|-|\ve{C}_l'|\ge 0,
  \]
  as claimed.

Similarly, $(u_r,u_r')$ can be viewed as living in a moduli space of $k$-gons where the marked points $\ve{C}_r$ and $\ve{C}_r'$ lie on the line $\lambda_{t_0}$, the marked points $\ve{S}_r$ and $\ve{S}_r'$ are perfectly matched, and also that there is a boundary puncture which has the same height as the special line. Since the expected dimension of this moduli space must be nonnegative if it admits a representative, we conclude that
\begin{equation}
\mu(\psi_r)+\mu(\psi_r')+k-3-2|\ve{S}_r|-|\ve{C}_r|-|\ve{C}_r'|\ge 0.
\label{eq:expected-dimension-psi-2}
\end{equation}

 We may rearrange Equations~\eqref{eq:Maslov-index-l-gon} and~\eqref{eq:j-k-gon-breaking} to obtain
\begin{equation}
\begin{split}
0=&\mu(\psi)+\mu(\psi')+\ell-2-2(|\ve{C}|+|\ve{S}|)\\
=&\mu(\psi_l)+\mu(\psi_l')+j-2-|\ve{C}_l|-|\ve{C}_l'|-2|\ve{S}_l|\\
&+\mu(\psi_r)+\mu(\psi_r')+k-3-|\ve{C}_r|-|\ve{C}_r'|-2|\ve{S}_r|.
\end{split}
\label{eq:inequality-l-curve-marked-points}
\end{equation}

Next, since $\cL_{\b}$ is algebraically rigid and $x$ is basepoint-esque for $\cL_{\b}$, we obtain
\begin{equation}
\mu(\psi_l)\ge 2n_p(\psi_l)= 2|\ve{C}_l|+2|\ve{S}_l|+2|\ve{N}_l|.\label{eq:graded-p-result1}
\end{equation}
See Definition~\ref{def:graded-by}. In particular, we obtain that
\begin{equation}
\mu(\psi_l)+\mu(\psi_l')+j-2-|\ve{C}_l|-|\ve{C}_l'|-2|\ve{S}_l|\ge \mu(\psi_l')+j-2+|\ve{C}_l|-|\ve{C}_l'|+2|\ve{N}_l|.
\label{eq:refine-using-rigidity}
\end{equation}
Since we are performing the $\bullet<\circ'$ degeneration, we know also that
\begin{equation}
|\ve{C}_l|\ge |\ve{C}'_l|\ge |\ve{C}_l|-1,\label{eq:graded-p-result2}
\end{equation}
since the left-most puncture along the special line will be in $\ve{C}_l$. 
By transversality for ordinary $j$-gons, we know that 
\begin{equation}
\mu(\psi_l')+j-2\ge 1.
\label{eq:maslov-index-basic-j-gons}
\end{equation}
 In particular, from Equations~\eqref{eq:refine-using-rigidity} and~\eqref{eq:maslov-index-basic-j-gons} we obtain that
 \begin{equation}
 \mu(\psi_l)+\mu(\psi_l')+j-2-|\ve{C}_l|-|\ve{C}_l'|-2|\ve{S}_l|\ge 1.\label{eq:expected-dimension-psi-3}
 \end{equation}

Equations~\eqref{eq:expected-dimension-psi-l} and~\eqref{eq:expected-dimension-psi-3} imply that Equation~\eqref{eq:inequality-l-curve-marked-points} is the sum of two nonnegative integers, one of which is at least 1. Hence the equation is never satisfiable. This implies claim~\eqref{putative-ell-gon-1}.

It remains to consider degenerations into polygons where the special line does not occur at the same height as one of the punctures in $\ve{P}\cup \ve{P}'$. In such cases,  $|\ve{C}_l|=|\ve{C}_l'|=0$. Furthermore, Equations~\eqref{eq:expected-dimension-psi-l} and~\eqref{eq:expected-dimension-psi-2} adapt to show that
\begin{equation}
\begin{split}
&\mu(\psi_l)+\mu(\psi_l')+j-3-2|\ve{S}_l|\ge 0\qquad \text{and}\\ &\mu(\psi_r)+\mu(\psi_r')+k-2-2|\ve{S}_r|-|\ve{C}_r|-|\ve{C}_r'|\ge 0.
\end{split}
\label{eq:expected-dimension-psi-l-v2}
\end{equation}
Arguing similarly to Equation~\eqref{eq:inequality-l-curve-marked-points}, we see that both of the two inequalities above are equalities.

We consider first the case that $(u_l,u_l')$ form below the special line.
In this case, we observe that $\psi_l\# \psi_{l}'$ may be viewed as a class on $\Sigma\# \Sigma'$ since they have the same multiplicity at the connected sum point. The excision principle for the index (i.e. deleting disks at the connected sum points) implies that
\[
\mu(\psi_l\# \psi_l')=\mu(\psi_l)+\mu(\psi_l')-2|\ve{S}_l|.
\]
By the absolute grading formula, if $\zs$ is the outgoing intersection point of $\psi_l\# \psi_l'$, we have
\[
\mu(\psi_l\# \psi_l')=n_{\ws''}(\psi_l\# \psi_l')+\gr(\Theta^+,\zs),
\]
since the inputs of $\psi_l\# \psi_l'$ consist of only the top degree intersection points. We conclude that $\mu(\psi_l\# \psi_l')\ge 0$. The first line of Equation~\eqref{eq:expected-dimension-psi-l-v2} (now known to be an equality) now implies that $j-3\le 0$. The only possibilities are $j=2,3$, as claimed. Degenerations of $j$-gons above the special line are handled via an essentially identical argument.
\end{proof}

\begin{lem}\label{lem:chain-maps-Phi-Psi} Under the assumptions in Lemma~\ref{lem:codim-1-Phi-ell}, the maps $\Phi^{\bullet<\circ'}$ and $\Psi^{\bullet<\circ'}$ are chain maps of hypercubes.
\end{lem}
\begin{proof}
We focus on the map $\Phi^{\bullet<\circ'}$, since the analysis for $\Psi^{\bullet<\circ'}$ is nearly identical. Suppose $\kappa\le \kappa'\in \bE_{n+m}$, and consider the hypercube relations for $\Cone(\Phi^{\bullet<\circ'})$. If $\kappa=\kappa'$, then the hypercube relations follow from Lemma~\ref{lem:PhiPsi-chain-maps}, so assume $\kappa<\kappa'$. 

Suppose that $\kappa=\kappa_1<\cdots <\kappa_{\ell}=\kappa'$. We count the ends of 1-dimensional families of $(\Phi,\bullet<\circ',t_0)$-matched $\ell+1$-gons, which have inputs from $\cL_{\b}\otimes \cL_{\b'}$.  The codimension 1 degenerations are analyzed in Lemma~\ref{lem:codim-1-Phi-ell}.  Many of these ends have a similar cancellation pattern as to Lemma~\ref{lem:PhiPsi-chain-maps}. For example the ends~\eqref{Phi-ell-1} and~\eqref{Phi-ell-2} cancel modulo 0. Similarly, each curve satisfying ~\eqref{Phi-ell-3} appears twice in the boundary of the modulo space, unless the limiting marked points alternate between $\Sigma$ and $\Sigma'$, in which case the curves cancel the degenerations in~\eqref{Phi-ell-4}. The curves appearing in~\eqref{Phi-ell-6} cancel modulo 2 because of the hypercube relations for $\cL_{\b}\otimes \cL_{\b'}$ (on both $\Sigma\sqcup \Sigma'$ and $\Sigma\wedge \Sigma'$).

The remaining curves are those which appear in the degenerations labeled~\eqref{Phi-ell-5}. These ends correspond exactly to the hypercube relations for the mapping cone of $\Phi^{\bullet<\circ'}$. This verifies that $\Phi^{\bullet<\circ'}$ is a chain map.
\end{proof}

\begin{rem} In the above, we assumed that $\cL_{\a}\otimes \cL_{\a'}$ was 0-dimensional. In the case that both $\cL_{\a}\otimes \cL_{\a'}$ and $\cL_{\b}\otimes \cL_{\b'}$ have positive dimension, the argument is essentially the same. The only additional subtlety concerns for which hypercubes $x$ and $x'$ are basepoint-esque for. In the above, we used that $x$ was basepoint-esque for $\cL_{\b}$ to obtain Equations~\eqref{eq:graded-p-result1} and~\eqref{eq:graded-p-result2} for degenerations along $\{0\}\times \R$. To obtain the analogous bounds for degenerations along $\{1\}\times \R$ we instead would need $p'$ to be basepoint-esque for $\cL_{\a'}$, as in the statement of Proposition~\ref{prop:connected-sums}.
\end{rem}

\begin{lem}
Under the assumptions in Lemma~\ref{lem:codim-1-Phi-ell}, the maps $\Phi^{\bullet<\circ'}$ and $\Psi^{\bullet<\circ'}$ are homotopy inverses of each other.
\end{lem}
\begin{proof} The proof is similar to the proof of Lemma~\ref{lem:Phi-Psi-homotopy-inverses-diagrams}. Indeed, one can define $(H_\wedge,t_0)$-matched holomorphic $(\ell+1)$-gons by analogy to the case of disks. Counting such curves gives a morphism of hypercubes
\[
H_{\wedge}\colon \ve{\CF}^-(\Sigma\wedge \Sigma',\as\cup \as', \cL_{\b}\otimes \cL_{\b'})\to \ve{\CF}^-(\Sigma\wedge \Sigma',\as\cup \as', \cL_{\b}\otimes \cL_{\b'}).
\]
By a straightforward adaptation of Lemma~\ref{lem:codim-1-Phi-ell} to the case where there are two special lines, and by adapting the argument from Lemma~\ref{lem:Phi-Psi-homotopy-inverses-diagrams} to handle the limiting curves appearing as the special lines collide, we obtain exactly that the following diagram is an $(n+m+2)$-dimensional hypercube of chain complexes:
\[
\begin{tikzcd}
\ve{\CF}^-(\Sigma\wedge \Sigma',\as\cup \as', \cL_{\b}\otimes \cL_{\b'})
	\ar[d, "\id"]
	\ar[r, "\Phi^{\bullet<\circ'}"]
	\ar[dr,dashed, "H_\wedge"]
&
\ve{\CF}^-(\Sigma\sqcup \Sigma',\as\cup \as', \cL_{\b}\otimes \cL_{\b'})
	\ar[d, "\Psi^{\bullet<\circ'}"]
\\
\ve{\CF}^-(\Sigma\wedge \Sigma',\as\cup \as', \cL_{\b}\otimes \cL_{\b'})
	\ar[r, "\id"]
&
\ve{\CF}^-(\Sigma\wedge \Sigma',\as\cup \as', \cL_{\b}\otimes \cL_{\b'})
\end{tikzcd}
\]
The above diagram realizes the relation
\[
\Psi^{\bullet<\circ'}\circ \Phi^{\bullet<\circ'}+\id=[\d, H_{\wedge}],
\]
completing the proof. A homotopy equivalence in the other direction is defined by the obvious adaptation.
\end{proof}

\section{The pairing theorem}

In this section, we prove the pairing theorem. We prove Theorems~\ref{thm:intro-pairing-theorem} and~\ref{thm:pairing} of the introduction, as well as some generalizations and refinements.

\subsection{The statement}

We now state the pairing theorem (i.e. connected sum formula). In this section, we state the pairing theorem for links in terms of a box-tensor product. In Section~\ref{sec:unpacking}, we unpack the statement in terms of the link surgery complexes. In Section~\ref{sec:proof-pairing-thm}, we prove the statement.

\begin{thm}\label{thm:pairing-links}
 Suppose that $L_1$ and $L_2$ are two links in $S^3$, which are equipped with systems of arcs $\scA_1$ and $\scA_2$, respectively. Form their connected sum $L_1\#L_2$ along two distinguished components $K_1\subset L_1$ and $K_2\subset L_2$. Assume the arc for $K_1$ is alpha-parallel and that the arc for $K_2$ is beta-parallel. Let $\scA_1\# \scA_2$ denote the connected sum, in the sense of Section~\ref{sec:basic-systems-connected-sums} (see in particular ~\eqref{alpha-beta-connected-sums}). Then
\[
\cX_{\Lambda_1+\Lambda_2}(L_1\# L_2, \scA_1\# \scA_2)^{\cL_{\ell_1+\ell_2-1}}\simeq \left(\cX_{\Lambda_1}(L_1,\scA_1)^{\cL_{\ell_1}}\otimes_{\bF} \cX_{\Lambda_2}(L_2, \scA_2)^{\cL_{\ell_2}} \right)\hatbox {}_{\cK| \cK} M^{\cK}.
\]
\end{thm}

\begin{rem} In Theorem~\ref{thm:pairing-links}, the arc of $\scA_1\# \scA_2$ for $K_1\#K_2$ is the co-core of the band used to form the connected sum. See Figure~\ref{fig:25}.
\end{rem}

\subsection{Unpacking the statement} 
\label{sec:unpacking}
We now give a concrete reformulation of the pairing theorem from the previous section, in terms of the knot and link surgery complexes.

We begin by unpacking the claim for the Ozsv\'{a}th--Szab\'{o} mapping cone complex $\bX_{\lambda}(Y_1\# Y_2, K_1\#K_2)$. For $i\in \{1,2\}$, the complex $\bX_{\lambda_i}(Y_i,K_i)$ is the mapping cone complex
\[
\bX_{\lambda_i}(Y_i,K_i)=
\Cone\big(\begin{tikzcd}[column sep=3cm]
\bA(K_i)\ar[r, "v^{(Y_i,K_i)}+h_{\lambda_i}^{(Y_i,K_i)}"]& \bB(K_i)
\end{tikzcd}\big)
\]
We recall from Lemma~\ref{lem:module-structure-surgery-hypercube-groups} that we may identify $\bA(K_i)$ with a completion of $\cCFK(Y_i,K_i)$, and that we may identify $\bB(K_i)$ with a completion of $\scV^{-1} \cCFK(Y_i,K_i)$.
From the connected sum formula for knot Floer homology, we know that
\[
\cCFK(Y_1\# Y_2,K_1\#K_2)\iso \cCFK(Y_1,K_1)\otimes_{\bF[\scU,\scV]} \cCFK(Y_2,K_2)
\]
and similarly the localized module $\scV^{-1} \cCFK(Y_1\#Y_2, K_1\# K_2)$ decomposes as a tensor product over the ring $\bF[\scU,\scV,\scV^{-1}]$.

 These tensor products of knot Floer complexes may themselves be described as box tensor products over the algebras $\bF[\scU,\scV]$ and $\bF[\scU,\scV,\scV^{-1}]$. Indeed, if $\xs_1,\dots, \xs_n$ is a free-basis of $\cCFK(Y_1,K_1)$, then we may form a type-$D$ module $\cCFK(Y_1,K_1)^{\bF[\scU,\scV]}$, spanned over $\bF$ by $\xs_1,\dots, \xs_n$, with structure map $\delta^1$ encoding the differential on $\cCFK(Y_1,K_1)$. Similarly, we may view $\cCFK(Y_2,K_2)$ as a type-$A$ module ${}_{\bF[\scU,\scV]} \cCFK(Y_2,K_2)$, freely generated over $\bF[\scU,\scV]$ by intersection points, and with only $m_1$ and $m_2$ non-vanishing. Clearly,
 \[
 \cCFK(Y_1,K_1)^{\bF[\scU,\scV]}\boxtimes {}_{\bF[\scU,\scV]} \cCFK(Y_2,K_2)\iso \cCFK(Y_1,K_1)\otimes_{\bF[\scU,\scV]} \cCFK(Y_2,K_2).
 \]

 The maps $v^{(Y_i,K_i)}$ from the knot surgery formulas of $K_1$ and $K_2$ are the canonical inclusion maps for localization at $\scV_i$. Hence, the map $v$ on the tensor product is just the tensor product $v^{(Y_1,K_1)}\otimes v^{(Y_2,K_2)}$.

On the other hand, the maps $h_{\lambda_1}^{(Y_1,K_1)}$ and $h_{\lambda_2}^{(Y_2,K_2)}$ both satisfy  $h_{\lambda_i}^{(Y_i,K_i)}(a \ve{x})=\phi^\tau(a) h_{\lambda_i}^{(Y_i,K_i)}$ by  Lemma~\ref{lem:module-structure-surgery-hypercube-morphisms}. It is straightforward to see that this implies that their tensor product $h_{\lambda_1}^{(Y_1,K_1)}\otimes h_{\lambda_2}^{(Y_2,K_2)}$ is well defined. 

In the box tensor product $\cX_{\lambda_1}(K_1)^{\cK} \hatbox {}_{\cK} \cX_{\lambda_2}(K_2)$, there is a summand of the differential which corresponds to $\tau$ being output by $\cX_{\lambda_1}(K_1)^{\cK}$, and then input into ${}_{\cK}\cX_{\lambda_2}(K_2)$. This summand is identified with $h_{\lambda_1}^{(Y_1,K_2)}\otimes h_{\lambda_2}^{(Y_2,K_2)}$. Similarly, there is a summand of the differential that corresponds to a $\sigma$ being output by $\cX_{\lambda_1}(K_1)^{\cK}$, and then input into ${}_{\cK} \cX_{\lambda_2}(K_2)$. This contributes $v_1\otimes v_2$ to the differential.

We summarize the above observations with the following lemma:

\begin{lem}
\label{lem:restatement-knots} If $K_1\subset Y_1$ and $K_2\subset Y_2$ are two knots in integer homology 3-spheres, then Theorem~\ref{thm:pairing-links} is equivalent to the statement that the map $h_{\lambda_1+\lambda_2}^{(Y_1\# Y_2, K_1\#K_2)}$ in  Ozsv\'{a}th and Szab\'{o}'s surgery formula $\bX_{\lambda_1+\lambda_2}(Y_1\# Y_2,K_1\# K_2)$ satisfies
\[
h_{\lambda_1+\lambda_2}^{(Y_1\#Y_2,K_1\#K_2)}\simeq h_{\lambda_1}^{(Y_1,K_1)}\otimes h_{\lambda_2}^{(Y_2,K_2)}.
\]
\end{lem}

The pairing theorem for the link surgery formula has a similar restatement.
 Let $\cC_{\Lambda_1}(L_1)$ and $\cC_{\Lambda_2}(L_2)$ be the  link surgery hypercubes of Manolescu and Ozsv\'{a}th. Write
\[
\cC_{\Lambda_i}(L_i)\iso \Cone\left(\begin{tikzcd}[column sep=1.3cm] \cC_0(L_i)\ar[r, "F^{-K_i}+F^{K_i}"] &\cC_1(L_i)\end{tikzcd}\right)
\]
Here $\cC_\nu(L_i)$ consists of the complexes of all points of the cube $\bE_{\ell_i}$ such that the coordinate for $K_i$ is $\nu\in \{0,1\}$.  Also, we are writing $F^{ K_i}$ (resp. $F^{-K_i}$) for the sum of the hypercube maps for all sublinks $\vec{N}\subset L_i$ which contain $K_i$ (resp. $-K_i$).

Note that similar to the case of ordinary link Floer complexes, the connected sum formula \cite{OSKnots}*{Section~7}  provides an identification of vector spaces
\[
\cC_0(L_1\#L_2)\iso \cC_0(L_1)\hatotimes_{\bF[\scU,\scV]} \cC_0(L_2)
\]
and
\[
\cC_1(L_1\# L_2)\iso \cC_1(L_1)\hatotimes_{\bF[\scU,\scV,\scV^{-1}]} \cC_1(L_2).
\]
Here, $\hatotimes$ denotes the completed tensor product. 

The maps $F^{K_i}$ are $\bF[\scU,\scV]$-equivariant, while $F^{-K_i}$ are $[T]$-equivariant. In particular, both tensor products $F^{K_1}\otimes F^{K_2}$ and $F^{-K_1}\otimes F^{-K_2}$ are well-defined. The same logic as with the case of connected sums of knots gives the following restatement of the pairing theorem for the surgery hypercubes:
 
\begin{lem}\label{lem:restatement-links}
 Theorem~\ref{thm:pairing-links} is equivalent to the statement that the surgery hypercube $\cC_{\Lambda_1+\Lambda_2}(L_1\# L_2, \scA_1\# \scA_2)$ is homotopy equivalent to the $(\ell_1+\ell_2-1)$-dimensional hypercube
\[
\Cone\left(\begin{tikzcd}[column sep=3.5cm]\cC_0(L_1, \scA_1)\hatotimes \cC_0(L_2,\scA_2)\ar[r,"F^{K_1}\otimes F^{K_2} +F^{-K_1}\otimes F^{-K_2}"] &\cC_1(L_1,\scA_1)\hatotimes \cC_1(L_2,\scA_2)\end{tikzcd}\right)
\]
 For each $\nu\in \{0,1\}$, the above complexes $\cC_\nu(L_1)\hatotimes \cC_\nu(L_2)$ are equipped with the tensor product differential $D_1^\nu\otimes \id+\id\otimes D_2^\nu$, where $D_j^\nu$ is the total differential of the hypercube $\cC_\nu(L_j)$ (i.e. the sum of the internal differentials as well as the hypercube maps).  
\end{lem}

\subsection{Proof of the pairing theorem}
\label{sec:proof-pairing-thm}
We now prove the pairing theorem, Theorem~\ref{thm:pairing-links}:
\begin{proof}[Proof of Theorem~\ref{thm:pairing-links}]
 We will use the restatement  in terms of mapping cones  from Lemma~\ref{lem:restatement-links}. Our strategy is to pick a $\sigma$-basic system  for which we can apply the hypercube tensor product formulas from Section~\ref{sec:hypercubes-connected-sums}.

We consider first the map $F^{-(K_1\# K_2)}$ and the corresponding Heegaard diagrams. These maps will in general have summands of length greater than 1, but since we are using a $\sigma$-basic system, the only non-trivial summands correspond to sublinks $\vec{M}\subset -(L_1\#L_2)$, all of whose components are oriented oppositely to $L_1\# L_2$. (See the second paragraph of \cite{MOIntegerSurgery}*{Section~8.7}).

 We focus on the largest hyperbox in a $\sigma$-basic system, i.e. the one for $\vec{M}=-(L_1 \# L_2)$, since all of the smaller hypercubes are determined by this one.  We suppose that $\scH_1$ and $\scH_2$ are $\sigma$-basic systems of Heegaard diagrams for $L_1$ and $L_2$, which use the systems of arcs $\scA_1$ and $\scA_2$, respectively. We consider the $\sigma$-basic system $\scH_{\#}$ constructed in  Section~\ref{sec:basic-systems-connected-sums}.
 
We obtain two important hyperboxes and hypercubes:
 \begin{enumerate}
 \item A hyperbox  $\ve{\CF}^-(\scH_\#)$ over \[R=\bF[U_1,\dots, U_{\ell_1-1},U_{\ell_1},U_{\ell_1+1},\dots, U_{\ell_1+\ell_2-1}]
   \]
   where $U_1,\dots, U_{\ell_1}$ are the variables for $L_1$, and $U_{\ell_1},\dots, U_{\ell_1+\ell_2-1}$ are the variables for $L_2$. Here, $U_{\ell_1}$  is viewed as the variable for the special component $K_1\# K_2$.  This hypercube is generated over $R$ by intersection points on each constituent Heegaard diagram. If $\ve{d}$ is the size of the hypercube, then at each point $\veps\in \bE(\ve{d})$, the diagrams of $\scH_{\#}$ are equipped with a complete collection of basepoints $\cW_{\veps}$, consisting of exactly one basepoint for each link component. When $\veps=0$, these coincide with the $\zs$-basepoints. As one increases the coordinates of the cube, the $\zs$-basepoints are moved into the position of the $\ws$-basepoints.   A holomorphic curve representing a class $\phi$ is weighted by the product of $U_i^{n_{p_{i,\veps}}(\phi)}$, where $p_{i,\veps}\in \cW_{\veps}$ is the basepoint corresponding to the link component in index $i$.
   \item A hypercube $\cC_{\Lambda_1+\Lambda_2}^{-(L_1\# L_2)}$, which has the same underlying groups and internal differential as the link surgery hypercube $\cC_{\Lambda_1+\Lambda_2}(L_1\#L_2)$ (computed using $\scH_{\#}$), but has only the differentials for negatively oriented sublinks of $L_1\#L_2$.
 \end{enumerate}
 
 Write $\frc \ve{\CF}^-(\scH_{\#})$ for the compression of $\ve{\CF}^-(\scH_{\#})$ If $\veps\in \bE_{\ell_1+\ell_2-1}$, write $\cC_{\veps}\subset \cC_{\Lambda_1+\Lambda_2}(L)$ and $\bm{C}_{\veps}\subset \frc \ve{\CF}^-(\scH_\#)$ for the underlying chain complexes at the point $\veps\in \bE_{\ell_1+\ell_2-1}$.

  The hypercube structure maps of $\cC_{\Lambda_1+\Lambda_2}^{-(L_1\#L_2)}$ are determined by the hypercube structure maps for $\ve{\CF}^-(\scH_{\#})$, as we now describe. If $\veps\in \bE_{\ell_1+\ell_2-1}$, write 
  \[
  \cS_{\veps}\subset \bF[\scU_1,\scV_1,\dots, \scU_{\ell_1+\ell_2-1}, \scV_{\ell_1+\ell_2-1}]
  \]
  for the multiplicatively closed subset generated by $\scU_i$ for each $i$ such that $\veps_i=0$ and $\scV_i$ for each $i$ such that $\veps_i=1$.

   At each point $\veps\in \bE_{\ell_1+\ell_2-1}$, we write $I_{\veps}$ for the inclusion map
 \[
I_{\veps} \colon  \cC_{\veps}\to \cS_{\veps}^{-1} \cdot \cC_{\veps}.
 \]                                         
 
 For each $\veps$, there is also an isomorphism
 \[
 \theta_\veps\colon \cS_{\veps}^{-1} \cC_{\veps}\to \bm{C}_{\veps}\otimes \bF[\bH(L)].
 \]
 This map is gotten by sending a generator $a\cdot \xs$ (where $\xs$ is an intersection point and $a$ is an algebra element) to $a'\cdot \xs$, where $a'$ is the unique element in  $\bF[U_1,\dots, U_{\ell_1+\ell_2-1}]\otimes \bF[\bH(L)]$ satisfying the following:
 \begin{enumerate}
 \item $A(a')=A(\xs)+A(a)$ (where $A(a')$ is defined by setting $A(U_i)=0$ and $A(T^{\ve{s}})=\ve{s}$).
 \item If $\veps_i=0$, then the power of $U_i$ in $a'$ is equal to the power of $\scV_i$ in $a$.
 \item If $\veps_i=1$, then the power of $U_i$ in $a'$ is equal to the power of $\scU_i$ in $a$. 
 \end{enumerate}
 
 Write $\Phi_{\veps,\veps'}$ for the hypercube structure map for $\cC_{\Lambda_1+\Lambda_2}^{-(L_1 \# L_2)}$ from $\cC_{\veps}$ to $\cC_{\veps'}$, and write $f_{\veps,\veps'}$ for the hypercube structure map for $\frc\ve{\CF}^-(\scH_\#)$ from $\bm{C}_{\veps}$ to $\bm{C}_{\veps'}$. The map $\Phi_{\veps,\veps'}$ is equal to
 \begin{equation}
 \Phi_{\veps,\veps'}= \theta_{\veps'}^{-1} \circ (f_{\veps,\veps'}\otimes T^{\Lambda_{\veps,\veps'}})\circ \theta_{\veps}\circ I_{\veps}.
 \label{eq:Phi-veps-veps'-conn-sum}
 \end{equation}
 In the above, $\Lambda_{\veps,\veps'}\in \Z^{\ell_1+\ell_2-1}$ denotes the sum of the columns of the framing matrix of $L_1 \# L_2$ corresponding to components $K_i$ of $L_1\#L_2$ for which $\veps_i'>\veps_i$.

Hence, to show the main claim it suffices to show that $\frc \ve{\CF}^-(\scH_\#)$ admits a tensor product decomposition analogous to the claimed one for $\cC(\scH_\#)$, and also to show that each summand of $F^{-K_1}\otimes F^{-K_2}$ has Alexander grading consistent with the surgery formula. The claim about Alexander gradings is obvious, so it suffices to address the claim about $\frc\ve{\CF}^-(\scH_\#).$

  We now consider in more detail the $\sigma$-basic system of Heegaard diagrams $\scH_\#$ constructed in Section~\ref{sec:basic-systems-connected-sums}. The hypercubes of attaching curves which appear in this $\sigma$-basic system have a simple description in terms of tensor products of hypercubes of attaching curves. See Equation~\eqref{eq:connected-sum-basic-system}. Recall that this hyperbox was constructed from $\sigma$-basic systems of Heegaard diagrams $\scH_{1}$ and $\scH_{2}$ for $L_1$ and $L_2$, respectively. We may assume that these hypercubes are algebraically rigid by Lemma~\ref{lem:algebraic-rigid-general-basic-system} (see also Lemma~\ref{lem:admissibility-meridional-system}). Recall that in Section~\ref{sec:basic-systems-connected-sums} we defined
  \[
  \scH_{\#}=\St\left(\scH_{1}\# \scH_{2}^{(0)}, \scH_{1}^{(1)}\# \scH_{2}\right).
  \]
  Here, $\scH_{1}^{(0)}$ (resp. $\scH_1^{(1)}$) is the subbox of $\scH_1$ where the coordinate for $K_1$ is 0 (resp. maximal). We define $\scH_2^{(0)}$ and $\scH_2^{(1)}$ similarly.
 Write $x_1\in \Sigma_1$ and $x_2\in \Sigma_2$ for the connected sum points. Since $\scH_1$ is alpha-parallel at $K_1$ (i.e. we use a special beta-curve as a meridian of $K_1$) it is straightforward to see that $x_1$ is basepoint-esque for  all of the constituent alpha-hypercubes of $\scH_1$, in the sense of Definition~\ref{def:graded-by}. (Note that $x_1$ is typically not basepoint-esque for the beta-hyperboxes in the construction, cf. Remark~\ref{rem:not-graded-by-p}).   Similarly since $\scH_2$ is beta-parallel at $K_2$, $x_2$ will be basepoint-esque for all of the constituent beta-hypercubes of $\scH_2$. See Figure~\ref{fig:31}. Note that $x_1$ and $x_2$ will be basepoint-esque for both the alpha and beta hypercubes of $\scH_{i}^{(j)}$.

 By Lemma~\ref{lem:stacking-v-compressing}, the operations of stacking and pairing hyperboxes of attaching curves commute. Using our results about connected sums of hypercubes in Propositions~\ref{prop:disjoint-unions-hypercubes-main} and~\ref{prop:connected-sums}, we obtain a homotopy equivalence between $\ve{\CF}^-(\scH_\#)$ and the hyperbox obtained by stacking
  \[
  \ve{\CF}^-(\scH_1)\otimes_{\bF[U_\ell]} \ve{\CF}^-(\scH^{(0)}_2)\quad \text{and} \quad \ve{\CF}^-(\scH_{1}^{(1)})\otimes_{\bF[U_\ell]} \ve{\CF}^-(\scH_{2}).
  \]
 To compress, we use the inductive construction described in Section~\ref{sec:compression}. In doing so, we may choose to compress the axis corresponding to the special link components $K_1$ and $K_2$ last. We obtain that the compression of $\ve{\CF}^-(\scH_\#)$ is homotopy equivalent to the compression of a hyperbox
  \begin{equation}
  \begin{tikzcd}[column sep=1.2cm]
 \bm{C}_0(L_1)\otimes_{\bF[U_\ell]} \bm{C}_0(L_2)\ar[r,"f_1\otimes \id"] & \bm{C}_1(L_1)\otimes_{\bF[U_\ell]} 
 \bm{C}_0(L_2)\ar[r, "\id\otimes f_2"] & \bm{C}_1(L_1)\otimes_{\bF[U_\ell]} \bm{C}_1(L_2).
  \end{tikzcd}
  \label{eq:unweighted-hypercube-pairing}
  \end{equation}
  In the above, $\bm{C}_{\nu}(L_i)$ denotes the compression of $\ve{\CF}^-(\scH_i^{(\nu)})$, for $\nu\in \{0,1\}$, and $f_i$ is the map so that the compression of $\ve{\CF}^-(\scH_i)$ is isomorphic to $\Cone(f_i)$. This is the analog for $\frc \ve{\CF}^-(\scH_\#)$ of the formula in the main statement for $\cC_{\Lambda_1+\Lambda_2}^{-(L_1\#L_2)}$. Since the other maps in Equation~\eqref{eq:Phi-veps-veps'-conn-sum} are tensorial, we obtain $F^{-(K_1\# K_2)}=F^{-K_1}\otimes F^{-K_2}$, as in the statement.
  
  We now consider the map $F^{K_1\# K_2}$. Since we are using a $\sigma$-basic system, the maps $F^{K_1}$, $F^{K_2}$ and $F^{K_1\# K_2}$ are all the canonical inclusions for localizing at the $\scV$ variable for $K_1$, $K_2$ and $K_1\# K_2$. We claim that the homotopy equivalence described above intertwines $F^{K_1\# K_2}$ with $F^{K_1}\otimes F^{K_2}$. Concretely, this amounts to the claim that the homotopy equivalence, defined above, counts the same curves on $\bm{C}_0(L_1)\otimes_{\bF[U_\ell]} \bm{C}_0(L_2)$ as it does on $\bm{C}_1(L_1)\otimes_{\bF[U_\ell]} \bm{C}_1(L_2)$. This is immediate from the definition. The main claim follows.
  \end{proof}

\section{Arc systems and the link surgery formula}

\label{sec:basepoint-moving-maps}

For any system of arcs $\scA$ for a link $L$, the construction of Manolescu and Ozsv\'{a}th gives a hypercube of chain complex $\cC_{\Lambda}(L,\scA)$, which is a module over $\bF\llsquare U_1,\dots, U_\ell\rrsquare $. When $\scA$ consists only of beta parallel arcs, their link surgery formula states that
\[
H_*(\cC_{\Lambda}(L,\scA))\iso \ve{\HF}^-(Y_{\Lambda}(L)).
\]
When $\scA$ consists only of arcs which are alpha-parallel or beta-parallel, their proof also adapts without substantial change to show the above isomorphism. However, for a general set of arcs, their work does not prove the isomorphism. In this section, we prove the following:

\begin{thm}
\label{thm:basepoint-independence}
Suppose that $\scA$ and $\scA'$ are two systems of arcs for a framed link $L\subset S^3$. Suppose that $\scA$ differs from $\scA'$ by changing only the arc for the component $K_1\subset L$. Then there is a homotopy equivalence of hypercubes
\[
\cC_{\Lambda}(L,\scA)\simeq \cC_{\Lambda}(L,\scA')
\]
which is equivariant over  $\bF\llsquare U_2,\dots, U_\ell\rrsquare $.
\end{thm}

We note that our argument does not imply that the type-$D$ modules $\cX_{\Lambda}(L,\scA)^{\cL}$ are homotopy equivalent for all choices of $\scA$. Instead, our proof gives the following statement about the bordered link modules:

\begin{prop}\label{prop:equivalence-type-D-change-basepoint}
 Let $\scA$ and $\scA'$ be two systems of arcs for $L\subset S^3$, and suppose that $\scA$ differs from $\scA'$ by changing only the arc for $K_1\subset L$. Then there is a homotopy equivalence of type-$D$ modules over $\cL_{\ell-1}$:
\[
\cX_{\Lambda}(L,\scA)^{\cL_\ell}\hatbox {}_{\cK_1}\cD_0\simeq \cX_{\Lambda}(L,\scA')^{\cL_\ell}\hatbox {}_{\cK_1}\cD_0.
\]
In the above,  $\cK_1$ denotes the algebra factor corresponding to the component $K_1\subset L$. 
\end{prop}

\subsection{Basepoint moving maps on hypercubes}
\label{sec:basepoint-moving-maps-subsec}

Suppose that $\cL_{\a}$ and $\cL_{\b}$ are two hypercubes of handleslide equivalent attaching curves on $(\Sigma,\ws,\zs,\ps)$. We suppose that $\g\subset \Sigma$ is an immersed loop, starting and ending at a free basepoint $p$. We now describe an endomorphism $\gamma_*$ of $\ve{\CF}^-(\cL_{\a},\cL_{\b})$.

 We begin by picking a diffeomorphism of $\Sigma$, which corresponds to an isotopy which moves $p$ in a loop along $\g$, and is the identity outside of a small neighborhood of $\g$.
 
There is a canonical map
\[
\phi_*\colon \ve{\CF}^-(\cL_{\a},\cL_{\b})\to \ve{\CF}^-(\phi\cL_{\a},\phi \cL_\b)
\]
obtained by pushing forward an intersection point tautologically under the diffeomorphism $\phi$, and extending equivariantly over $\scU_i$, $\scV_i$ and $U_j$. We now build a hyperbox of handleslide equivalent attaching curves
\begin{equation}
\begin{tikzcd}
\phi\cL_{\a}\ar[r]& \cL_{\a}'\ar[r]& \cL_{\a}.
\end{tikzcd}
\label{eq:hyperbox-phiLa->La}
\end{equation}
We pick $\cL_{\a}'$ by winding the curves of $\cL_{\a}$ sufficiently so that $(\phi \cL_{\a},\cL'_{\a})$ and $(\cL'_{\a},\cL_{\a})$ are both weakly admissible. We use the standard filling procedure of Manolescu and Ozsv\'{a}th \cite{MOIntegerSurgery}*{Lemma~8.6} to fill in the remaining chains of the above hyperbox. 

By pairing the hyperbox in Equation~\eqref{eq:hyperbox-phiLa->La} with $\phi_* \cL_{\b}$ and compressing, we obtain a map
\[
\Psi_{\phi \cL_{\a}\to \cL_{\a}}^{\phi\cL_{\b}}\colon \ve{\CF}^-(\phi \cL_{\a},\phi\cL_\b)\to \ve{\CF}^-(\cL_{\a},\phi \cL_\b).
\]
A map $\Psi_{\cL_{\a}}^{\phi\cL_\b\to \cL_{\b}}\colon \ve{\CF}^-(\cL_{\a},\phi\cL_\b)\to \ve{\CF}^-(\cL_\a,\cL_\b)$ is defined similarly. We define the map $\g_*$ as a composition
\[
\gamma_*:=\Psi_{\cL_{\a}}^{\phi\cL_\b\to \cL_{\b}}\circ \Psi_{\phi \cL_{\a}\to \cL_{\a}}^{\phi\cL_{\b}}\circ \phi_*.
\]

It is not hard to see that the map $\gamma_*$ is well defined up to chain homotopies of hypercube morphisms by repeating the same procedure to build hypercubes of higher dimension. 

\begin{rem}\label{rem:naturality-hypercubes}
The above construction may be packaged into a form of naturality for hypercubes. If $\scH=(\cL_{\a},\cL_{\b})$ and $\scH'=(\cL_{\a'},\cL_{\b'})$ are two hypercubes of handleslide equivalent attaching curves, such that each set of curves in $\cL_{\a}$ is handleslide equivalent to each set of curves in $\cL_{\a'}$ and $\dim(\cL_{\a})=\dim(\cL_{\a'})$, and similarly for $\cL_{\b}$ and $\cL_{\b'}$, then the above construction produces a map $\Psi_{\scH\to \scH'}\colon \ve{\CF}^-(\scH)\to \ve{\CF}^-(\scH')$. Furthermore,
\begin{equation}
\Psi_{\scH'\to \scH}\circ \Psi_{\scH\to \scH'}\simeq \id_{\ve{\CF}^-(\scH)}.
\label{eq:naturality-hypercubes}
\end{equation}
Equation~\eqref{eq:naturality-hypercubes} is proven by using associativity to reduce to the case that $\scH'$ consist of small translates of $\cL_{\a}$ and $\cL_{\b}$, and then by either adapting Lipshitz's proof in the case of triangles \cite{LipshitzCylindrical} or by using the small translate theorems for holomorphic polygons in Lemma~\ref{lem:nearest-point}. 
\end{rem}

We now state our formula for the map $\gamma_*$. The statement is formally similar to the case of the ordinary Floer complexes \cite{ZemGraphTQFT}*{Theorem~D}. Our formula involves two endomorphisms $\Phi_w$ and $\cA_{\g}$ which are defined in Sections~\ref{sec:homology-actions-hypercubes} and~\ref{sec:basepoint-actions-hypercubes}.

\begin{thm}\label{thm:basepoint-moving-hypercubes} Suppose that $\cL_{\a}$ and $\cL_{\b}$ are hypercubes of handleslide equivalent attaching curves on $(\Sigma,\ws,\zs,\ps)$. Let $p\in \ps$ be a free basepoint and let $\g$ be a closed path on $\Sigma$, which is based at $p$. If $\gamma_*$ is the associated basepoint moving hypercube morphism, then
\[
\gamma_*\simeq \id+\Phi_p\circ \cA_\g
\]
as morphisms of hypercubes.
\end{thm}

The proof of the above result will be given in Section~\ref{sec:algebraic-relations}. As a first application, we can compute the effect of changing one arc in a system of arcs for a link $L\subset S^3$:

\begin{cor}
\label{cor:different-path-formula}
Suppose that $L\subset S^3$ is a framed link and that $\scA$ is a system of arcs for $L$. Let $\scH$ be a $\sigma$-basic system of Heegaard diagrams for $(L,\scA)$. Let $\g$ be an embedded closed loop in $S^3$ which is based at some $z_i\in K_i\subset L$ but is otherwise disjoint from $L$ and $\scA$. Assume $\g$ is contained in the Heegaard surface for $\Sigma$. Let $\scA'$ be the system of arcs obtained by concatenating the arc for $K_i$ with $\g$. Let us view $\cC_\Lambda(\scH)$ as
\begin{equation}
\cC_{\Lambda}(\scH)\iso \Cone(\begin{tikzcd}[column sep=2cm]\cC_0(\scH) \ar[r, "F^{K_i}+F^{-K_i}"] &\cC_{1}(\scH)
\end{tikzcd}),
\label{eq:basic-system-C-lambda-path-1}
\end{equation}
where $\cC_i$ is a $|L|-1$ dimensional hypercube and $F^{-K_i}$ denotes the sum of all hypercube maps for oriented sublinks of $L$ which contain $-K_i$, and similarly for $F^{K_i}$. Then there is a $\sigma$-basic system of Heegaard diagrams $\scH'$ for $(L,\scA')$ such that 
\begin{equation}
\cC_{\Lambda}(\scH')\simeq 
\Cone(\begin{tikzcd}[column sep=4cm]
\cC_0(\scH) \ar[r, "F^{K_i}+(\id+\scV^{-1}_i\cA_\g\circ \Phi_{w_i})F^{-K_i}"]& \cC_{1}(\scH)
\end{tikzcd})
\label{eq:basic-system-C-lambda-path-2}
\end{equation}
where $\cA_\g$ and $\Phi_{w_i}$ are viewed as endomorphisms of $\cC_1(\scH)$.
\end{cor}

We now sketch the proof of Theorem~\ref{thm:basepoint-independence} using Corollary~\ref{cor:different-path-formula} and a few other basic results from later in this section:

\begin{proof}[Proof of Theorem~\ref{thm:basepoint-independence}]
The map $\Phi_{w_i}$ will be introduced in more detail in Section~\ref{sec:homology-actions-hypercubes}. According to Remark~\ref{rem:Phi-nullhomotopic}, below, the map $\Phi_{w_i}$ is null-homotopic as a hypercube endomorphism via the homotopy
\[
\Phi_{w_i}=\d_{\Mor}(\d_{\scU_i})=\d \circ \d_{\scU_i}+\d_{\scU_i} \circ \d.
\]
Here, $\d$ denotes the hypercube differential on $\cC_1(\scH)$, and $\d_{\scU_i}$ is the derivative with respect to $\scU_i$ (i.e. $\d_{\scU_i}(\scU_i^n \xs)=n \scU_i^{n-1} \xs$), when $\xs$ is a generator which has no $\scU_i$ powers. See Remark~\ref{rem:Phi-nullhomotopic} for more details on this null-homotopy.

The null-homotopy $\d_{\scU_i}$ of $\Phi_{w_i}$ gives a chain homotopy equivalence between the hypercubes in Equations~\eqref{eq:basic-system-C-lambda-path-1} and~\eqref{eq:basic-system-C-lambda-path-2}. The map $\d_{\scU_i}$ does not commute with $\scU_i$ or $U_i=\scU_i\scV_i$, however it will commute with $\scU_j$ and $\scV_j$ for  $j\neq i$. In particular, the null-homotopy $\d_{\scU_i}$ cannot be used to induce a morphism between $\cX_{\Lambda}(L,\scA)^{\cL_\ell}$ and $\cX_{\Lambda}(L,\scA')^{\cL_\ell}$. However, after tensoring with $\cD_0$, the endomorphism $\bI\otimes \d_{\scU_i}$ can be defined, so we we may package the null-homotopy of $\Phi_{w_i}$ as a homotopy equivalence of type-$D$ modules
\[
\cX_{\Lambda}(L,\scA)^{\cL_\ell}\hatbox {}_{\cK} \cD_0\simeq \cX_{\Lambda}(L,\scA')^{\cL_\ell}\hatbox {}_{\cK} \cD_0
\]
as in the statement of Proposition~\ref{prop:equivalence-type-D-change-basepoint}.
\end{proof}

We now describe the proof of Corollary~\ref{cor:different-path-formula}.

\begin{proof}[Proof of Corollary~\ref{cor:different-path-formula}]
One first builds the surgery hyperbox for sublinks of $-L$. We compress the direction for $K_i$. The resulting hyperbox takes the form
\[
\Cone(\begin{tikzcd}\widetilde{\cC}_0(\scH) \ar[r, "F^{-K_i}"]& \widetilde{\cC}_1(\scH)\end{tikzcd} ),
\]
where $\widetilde{\cC}_j(\scH)$ compresses to $\cC_j(\scH)$.
To build the hyperbox for $\scH'$, we stack the above hyperbox with the basepoint moving map hyperbox for $\g_*$ to form the following hyperbox:
\[
 (\begin{tikzcd}\widetilde{\cC}_0(\scH) \ar[r, "F^{-K_i}"]& \tilde{\cC}_1(\scH) \ar[r, "\g_*"] & \widetilde{\cC}_1(\scH)\end{tikzcd} ).
\]
If $\widetilde{\cC}_1(\scH)$ has dimension $n-1$ and size $\ve{d}$, the hyperbox for $\g_*$ has dimension $n$ and size $\ve{d}\times \{1\}$. 

Theorem~\ref{thm:basepoint-moving-hypercubes} implies that the above hyperbox has homotopy equivalent compression to the compression of
\[
(\begin{tikzcd}\widetilde{\cC}_0(\scH) \ar[r, "F^{-K_i}"]& \widetilde{\cC}_1(\scH) \ar[r, "\scV_i^{-1}\Phi_{w_i}"] &\widetilde{\cC}_1(\scH)\ar[r, "\cA_\g"]& \widetilde{\cC}_1(\scH)\end{tikzcd} ).
\]
The factor of $\scV^{-1}_i$ in front of $\Phi_{w_i}$ is to make the map for sublinks of $L$ which contain $-K_i$ have the correct Alexander grading, since $\cA_\g$ has $A_i$-Alexander grading $0$ while $\Phi_{w_i}$ has $A_i$-Alexander grading $1$. (Note that $\cA_\g$ and $\Phi_{w_i}$ will not necessarily be homogeneous in the other Alexander gradings, since they are hypercube morphisms, and hence may increment the cube directions for other link components). We recall that for these maps, the powers of $\scV_i$ are determined by the overall Alexander grading change of the map.
Compressing the above hyperbox gives the statement.
\end{proof}

\begin{rem}
We observe that in the above corollary, there is some ambiguity in the definition of $\Phi_{w_i}$. In practice, we would like to compute $\Phi_{w_i}$ on the link surgery hypercube, as opposed computing $\Phi_{w_i}$ using a large hyperbox whose compression is the link surgery hypercube. It is an easy consequence of the Leibniz rule that  $\Phi_{w_i}$ commutes with the compression operation, so either choice gives the same answer.
\end{rem}

The following remark indicates some subtleties of the hypercube homology actions $\cA_\g$ and the basepoint moving map formulas:

\begin{rem}  The statement of  Corollary~\ref{cor:different-path-formula} is false if $\g$ is not disjoint from the curves $\scA$ on the Heegaard diagram. An example is the genus 0 diagram for the Hopf link in Figure~\ref{fig:11}. In this case, composing with the basepoint moving map for moving one $w_i$ around one of the components will switch between two models of the Hopf link complex (for different arc systems) which we show are non-isomorphic type-$D$ modules in Section~\ref{sec:Hopf-links-comp}. On the other hand, any loop on this Heegaard diagram bounds a disk since the surface is $S^3$, so Lemma~\ref{lem:homology-action-homotopy-hypercube} can be used to show that the homology action is null-homotopic. In this case, the hypotheses are not satisfied, since the curve $\g=K_2$ is not disjoint from the arc $\scA_1$ for $K_1$.
\end{rem}

\subsection{Basepoint actions on hypercubes}

\label{sec:basepoint-actions-hypercubes}
In this section, we define the basepoint action on hypercubes. The basepoint actions are analogs of natural endomorphisms which appear on the ordinary Heegaard Floer complexes, especially in the context of basepoint moving maps and diffeomorphism maps. See \cite{SarkarMovingBasepoints}, \cite{ZemQuasi} and \cite{ZemGraphTQFT} for example.

Suppose that $\cL_{\a}$ and $\cL_{\b}$ are hypercubes of attaching curves on $(\Sigma,\ws,\zs, \ve{p})$. Here and throughout, we assume $\ws$ and $\zs$ are link basepoints and $\ps$ are free basepoints.  In this section, we will consider the Floer complex $\ve{\CF}^-(\cL_{\a},\cL_{\b})$ which is a free module over $\bF[\scU_1,\dots, \scU_\ell,\scV_1,\dots, \scV_\ell,U_1,\dots, U_t]$, where $\scU_j$ is the variable for $w_j\in \ws$, $\scV_j$ is the variable for $z_j\in \zs$, and $U_j$ is the variable for $p_j\in \ps$.

If $w_i\in \ws$,  we now describe an endomorphism
\[
\Phi_{w_i}\colon \ve{\CF}^-(\cL_{\a},\cL_{\b})\to \ve{\CF}^-(\cL_{\a},\cL_{\b}).
\]
If $z_i\in \zs$, or $p_i\in \ps$, there will be analogous endomorphisms, denoted $\Psi_{z_i}$ and $\Phi_{p_i}$, all defined by essentially the same construction.

We define the map $\Phi_{w_i}$ by formally differentiating the hypercube differential $\d$ on $\ve{\CF}^-(\cL_{\a}, \cL_{\b})$ with respect to $\scU_i$. More precisely, if $\xs$ and $\ys$ are intersection points  $\ve{\CF}^-(\cL_{\a},\cL_{\b})$ and $\d(\xs)$ has a summand of $\ve{a}\cdot \ys$ for some $\ve{a}\in \bF[\scU_1,\scV_1,\dots, \scU_\ell,\scV_\ell,U_1,\dots, U_t]$,  then we define $\Phi_{w_i}(\xs)$ to have a summand of $\d_{\scU_i}(\ve{a})\cdot \ys$. We extend $\Phi_{w_i}$ to all of $\ve{\CF}^-(\cL_{\a},\cL_{\b})$ by declaring it to be $\bF[\scU_1,\scV_1,\dots, \scU_\ell,\scV_\ell,U_1,\dots, U_t]$-equivariant.

\begin{lem}\label{lem:Phi-chain-map} The map $\Phi_{w_i}$ satisfies $\d_{\Mor}(\Phi_{w_i})=0$, where $\d_{\Mor}$ is the morphism differential for hypercube morphisms (see Equation~\eqref{eq:morphism-differential-hypercubes}).
\end{lem}
\begin{proof}  The proof is the same as in the setting of the ordinary Floer chain complexes. See, e.g., \cite{ZemQuasi}*{Lemma~3.1}. This is obtained by taking the equation $\d^2=0$, and differentiating with respect to $\scU_{i}$. The Leibniz rule yields $\Phi_{w_i}\circ \d+\d \circ \Phi_{w_i}=0$.
\end{proof}

\begin{rem}
 In Lemma~\ref{lem:Phi-chain-map}, we are implicitly using the fact that the differential commutes with the action of $\scU_i$. We note that in the link surgery formula $\cC_{\Lambda}(L)$, although the group is naturally a completion of a free-module over $\bF[\scU_i]$, the differential does not commute with $\scU_i$ (cf. Lemma~\ref{lem:module-structure-surgery-hypercube-morphisms}). However, if $w\in K$ and we restrict to a subcube of $\cC_{\Lambda}(L)$ where the coordinate for $K$ is constant, then the differential does commute with $\scU_i$ and the map $\Phi_{w_i}$ can be defined. This is the only case that we consider in the present paper.
\end{rem}

When $\cL_{\a}$ and $\cL_{\b}$ are algebraically rigid, there is an alternate definition of $\Phi_{w_i}$, for which we write $\Phi_w^0$. The map $\Phi_{w_i}^0$ is defined as follows. Suppose that $\cL_{\a}$ is $m$-dimensional and $\cL_{\b}$ is $n$-dimensional. Given $\veps_1<\cdots<\veps_j$ in $\bE_m$ and $\nu_1<\dots<\nu_k$ in $\bE_n$, we define the \emph{$w_i$-weighted polygon counting map}, $f^{w_i}_{\a_{\veps_j},\dots, \a_{\veps_1},\b_{\nu_1},\dots, \b_{\nu_k}}$, to count holomorphic polygons of index $3-j-k$ with a multiplicative weight of 
\[
n_{w_i}(\psi) \scU_1^{n_{w_1}(\psi)}\scV_1^{n_{z_1}(\psi)}\cdots \scU_\ell^{n_{w_\ell}(\psi)} \scV_\ell^{n_{w_\ell}(\psi)}U_1^{n_{p_1}(\psi)}\cdots U_t^{n_{p_t}(\psi)}.
\]  We define
\begin{equation} 
(\phi_{w_i})_{\veps_1<\cdots<\veps_j}^{\nu_1<\cdots<\nu_k}\colon \ve{\CF}^-(\as_{\veps_1},\bs_{\nu_1})\to \ve{\CF}^-(\as_{\veps_j},\bs_{\nu_k})
\label{eq:phiw-veps-nu-def}
\end{equation}
via the formula
\[
\begin{split}
&(\phi_{w_i})_{\veps_1<\cdots<\veps_j}^{\nu_1<\cdots<\nu_k}(\xs)
\\
:=&\scU_i^{-1} f^{w_i}_{\a_{\veps_j},\dots, \a_{\veps_1},\b_{\nu_1},\dots, \b_{\nu_k} }(\Theta_{\a_{\veps_j},\a_{\veps_{j-1}} },\dots, \Theta_{\a_{\veps_2},\a_{\veps_1} }, \xs, \Theta_{\b_{\nu_1},\b_{\nu_2} }, \dots \Theta_{\b_{\nu_{k-1}},\b_{\nu_k} } )
\end{split}
\]
We define $\Phi_{w_i}^0\colon \ve{\CF}^-(\cL_{\a},\cL_{\b})\to \ve{\CF}^-(\cL_{\a},\cL_{\b})$ via the formula
\begin{equation}
\Phi_{w_i}^0:=\sum_{\substack{\nu_1<\dots<\nu_k \\ \veps_1<\cdots<\veps_j}} (\phi_{w_i})_{\veps_1<\cdots<\veps_j}^{\nu_1<\cdots<\nu_k}.
\label{eq:def-map-Phi_w-0}
\end{equation}

\begin{lem}
\label{lem:algebraically-rigid-Phiw0}
 If $\cL_{\a}$ and $\cL_{\b}$ are algebraically rigid hypercubes of handleslide equivalent attaching curves, then $\Phi_{w_i}=\Phi_{w_i}^0$.
\end{lem}
\begin{proof} For algebraically rigid hypercubes of attaching curves, the only special inputs for the hypercube differential are top degree intersection points, which are uniquely determined and have no $\scU_i$-powers. In particular, the $\scU_i$-power of an arrow in the differential of $\ve{\CF}^-(\cL_{\a},\cL_{\b})$ is determined entirely by the $n_{w_i}$ multiplicity of the corresponding holomorphic $\ell$-gon, which is exactly what the map $\Phi_{w_i}^0$ records.
\end{proof}

\begin{rem}\label{rem:Phi-nullhomotopic} Similar to the setting of the 3-manifold invariants, the map $\Phi_{w_i}$ is a null-homotopic hypercube endomorphism, though the natural homotopy does not commute with the action of $\scU_{i}$. A similar comment holds for the endomorphisms $\Psi_{z_i}$ and $\Phi_{p_i}$, when $p_i$ is a free basepoint. As in \cite{ZemGraphTQFT}*{Equation~14.34}, we have
\begin{equation}
\Phi_{w_i}=[\d, \d_{\scU_i}]=\d \circ \d_{\scU_i}+\d_{\scU_i}\circ \d
\label{eq:Phi-w-null-homotopic}
\end{equation}
where $\d$ denotes the hypercube differential and $\d_{\scU_i}$ is the map 
\[
\d_{\scU_i}(\scU_i^k \ve{x})=k \scU_i^{k-1}\ve{x}
\]
extended equivariantly over the actions of the other variables. Conceptually, we can think of Equation~\eqref{eq:Phi-w-null-homotopic} as an application of the Leibniz rule, in the following sense. If we think of $\Phi_{w_i}$ and $\d$ as matrices with coefficients in $\bF[\scU_i]$ and a generator $\xs$ as a column vector with entries in $\bF[\scU_i]$, then $\Phi_{w_i}(\xs)$ and $\d(\xs)$ can be thought of as matrix products. The Leibniz rule for derivatives implies that
\[
\d_{\scU_i}(\d\xs)=(\d_{\scU_i} (\d))(\xs)+ \d (\d_{\scU_i}(\xs)).
\]
By definition $\d_{\scU_i} (\d)=\Phi_{w_i}$, so rearranging the above equation yields the null-homotopy $\Phi_{w_i}=[\d,\d_{\scU_i}]$.

\end{rem}

\subsection{Homology actions on hypercubes}
\label{sec:homology-actions-hypercubes}
We now recall the construction of various homological actions on hypercubes from \cite{HHSZNaturality}*{Section~6.2}. We will describe an action for both closed curves on $\Sigma$ and also arcs which have boundary on two basepoints $\{w_1,w_2\}\subset \ws$.

In the case of closed curves, the construction is an adaptation of the original case \cite{OSDisks}*{Section~4.2.5}. For arcs with boundary, the construction extends \cite{Ni-homological} and \cite{ZemGraphTQFT}*{Section~5}.

We consider first the case of a closed curve $\g\subset \Sigma$. Suppose that $\cL_{\a}$ and $\cL_{\b}$ are hypercubes of attaching curves on $(\Sigma,\ws,\zs)$.  We now describe an endomorphism
\[
\cA_{\g}\colon \ve{\CF}^-(\cL_{\a},\cL_{\b})\to \ve{\CF}^-(\cL_{\a},\cL_{\b})
\]
as follows. Suppose
\begin{equation}
\psi\in \pi_2(\xs_{\a_{\veps_j},\a_{\veps_{j-1}}},\dots, \xs_{\a_{\veps_2},\a_{\veps_{1}}},\xs,\xs_{\b_{\nu_1},\b_{\nu_{2}}},\dots, \xs_{\b_{\nu_{k-1}},\b_{\nu_{k}}},\ys)
\label{eq:psi-A-lambda-def-intersection-points}
\end{equation}
where each $\as_{\veps_{i}}$ and $\bs_{\veps_{i}}$ are curves in $\cL_{\a}$ or $\cL_{\b}$, respectively, and each intersection point labeled $\xs_{\g,\g'}$ is in $\bT_{\g}\cap \bT_{\g'}$. We define the 1-chain $\d_{\a}(\psi)\subset \Sigma$ to be the boundary of $\psi$ which lies in $\as_{\veps_1}\cup \cdots \cup \as_{\veps_j}$. We define a quantity $a(\psi,\g)\in \bF$ via the formula
\[
a(\psi,\g)=\#(\d_{\a} (\psi)\cap \g).
\]
We now define an endomorphism
\[
\cA_\g\colon \ve{\CF}^-(\cL_{\a},\cL_{\b})\to \ve{\CF}^-(\cL_{\a},\cL_{\b}).
\]
The map counts holomorphic polygons as would be counted in the hypercube differential for $\ve{\CF}^-(\cL_{\a},\cL_{\b})$, except with an extra multiplicative weight of $a(\psi,\g)$.

\begin{lem}
\label{lem:homology-action-chain-homotopy} Suppose that $\cL_{\a}$ and $\cL_{\b}$ are hypercubes of attaching curves on $(\Sigma,\ws,\zs,\ps)$ and $\g$ is a closed curve on $\Sigma$. Then $\d_{\Mor}(\cA_{\g})=0$, where $\d_{\Mor}$ is the morphism differential for hypercube morphisms.
\end{lem}
\begin{proof} 
The proof that $\d_{\Mor}(\cA_\g)=0$ follows from a Gromov compactness argument, similar to the proof of the hypercube relations for $\ve{\CF}^-(\cL_{\a},\cL_{\b})$. 

 Suppose that  $\xs\in \bT_{\a_{\veps}} \cap \bT_{\b_\nu}$ and $\ys\in \bT_{\a_{\veps'}}\cap \bT_{\b_{\nu'}}$, for $\veps\le \veps'$ and $\nu\le \nu'$.  We consider a class of polygons 
 \[
 \psi\in \pi_2(\xs_{\a_{\veps_j},\a_{\veps_{j-1}}},\dots, \xs_{\a_{\veps_2},\a_{\veps_{1}}},\xs,\xs_{\b_{\nu_1},\b_{\nu_{2}}},\dots, \xs_{\b_{\nu_{k-1}},\b_{\nu_{k}}},\ys)
 \]
 where each $\xs_{\a_{\veps_k},\a_{\veps_{k-1}}}$ is an intersection point appearing as a summand of a chain from $\cL_{\a}$, and similarly for the beta labeled intersection points. We suppose that
 \[
 \mu(\psi)=4-j-k
 \]
 so that $\cM(\psi)$ is 1-dimensional. We consider the ends of the 1-dimensional moduli space $\cM(\psi)$,

 If $j+k>2$, the ends of $\cM(\psi)$ generically consist of $\psi$ breaking into a pair of holomorphic polygons (with at least 2 sides, each) representing homology classes $\psi_1,\psi_2$ satisfying $\psi=\psi_1+\psi_2$. It is helpful to organize possible degenerations into three subtypes (up to reordering $\psi_1$ and $\psi_2$):
 \begin{enumerate}[label=($e$-\arabic*), ref=$e$-\arabic*]
 \item\label{list:end1} $\psi_1$, $\psi_2$ both have boundary on both the alpha and beta curves.  
 \item\label{list:end2} $\psi_1$ has boundary on both  alpha and beta curves, while $\psi_2$ has boundary only on the alpha curves.
 \item\label{list:end3} $\psi_1$ has boundary on both alpha and beta curves, while $\psi_2$ has boundary only on the beta curves.
 \end{enumerate}
 
 \begin{figure}[h]
 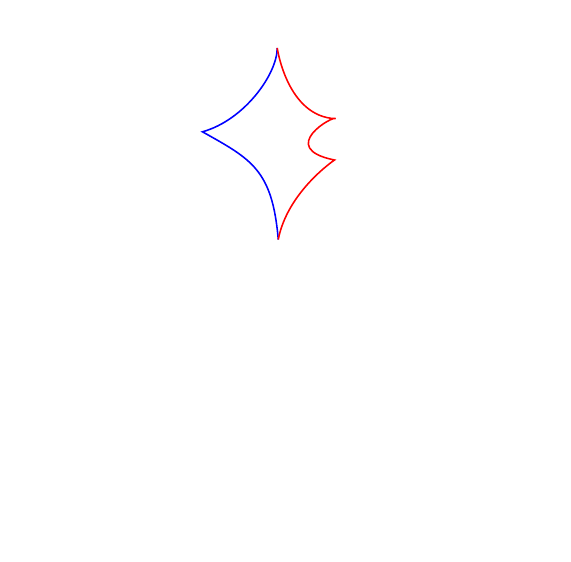
 \caption{The degenerations \eqref{list:end1}, \eqref{list:end2} and \eqref{list:end3} from Lemma~\ref{lem:homology-action-chain-homotopy}.}
 \label{fig:43}
 \end{figure}
 
 We consider the sum over all sequences of intersections points $\xs_{\a_{\veps_j},\a_{\veps_{j-1}}},\dots, \xs_{\a_{\veps_2},\a_{\veps_{1}}}$ and $\xs_{\b_{\nu_1},\b_{\nu_{2}}},\dots, \xs_{\b_{\nu_{k-1}},\b_{\nu_{k}}}$ from the corresponding chains of $\cL_{\a}$ and $\cL_{\b}$, as well as  classes $\psi$, as above, of 
 \[
 a(\psi,\g)\# \d \cM(\psi) \scU_1^{n_{w_1}(\psi)}\scV_1^{n_{z_1}(\psi)}\cdots \scU_\ell^{n_{w_\ell}(\psi)} \scV_\ell^{n_{z_\ell}(\psi)}U_1^{n_{p_1}(\psi)}\cdots U_t^{n_{p_t}(\psi)}.
 \]
(We also multiply by any variables $\scU_i$ or $\scV_i$ appearing in the corresponding chains of $\cL_{\a}$ or $\cL_{\b}$, but we omit this from the above equation). We observe that if $\psi=\psi_1+\psi_2$, as above, then
\begin{equation}
a(\psi,\g)=a(\psi_1,\g)+a(\psi_2,\g).\label{eq:additivity-a-gamma-psi}
\end{equation}
Therefore, when summing over the ends of $\#\d \cM(\psi)$, instead of summing each end with weight $a(\psi,\g)$, we can instead sum each end twice, once with weight $a(\psi_1,\g)$ and once with weight $a(\psi_2,\g)$.

Ends ~\eqref{list:end1}, summed as above, contribute the $\ys$ coefficient of 
\[
(\d\circ \cA_\g+\cA_\g\circ \d)(\xs).
\]

When weighted by $a(\psi_1,\g)$, we claim that the ends ~\eqref{list:end2} and~\eqref{list:end3} cancel modulo 2 by the hypercube structure relation for $\cL_{\a}$ and $\cL_{\b}$. 

We will show that when weighted by $a(\psi_2,\g)$, the ends ~\eqref{list:end2} and~\eqref{list:end3} cancel modulo 2. To establish this, we claim more generally that if $\psi_2$ is a class of polygons on a subdiagram of $(\Sigma,\cL_\a)$ or $(\Sigma,\cL_\b)$, then
\begin{equation}
a(\psi_2,\g)=0,\label{eq:a-gamma=0-non-mixed}
\end{equation}
which clearly will imply the claim. 
If $\psi_2$ is a class of polygons on $(\Sigma,\cL_{\b})$ this claim is trivial since the alpha boundary of $\psi_2$ is trivial. If $\psi_2$ is instead a class on a subdiagram of $(\Sigma,\cL_{\a})$, Equation~\eqref{eq:a-gamma=0-non-mixed} follows from the fact that $\d_{\a}(\psi_2)=\d(D(\psi_2))$, where we are writing $D(\psi_2)$ for the \emph{domain} of the class $\psi_2$. We recall that $D(\psi_2)$ is the formal integral combination of the components of $\Sigma\setminus \bigcup_{\as \in \cL_{\a}} \as$ whose multiplicity at a point $p\in \Sigma\setminus \bigcup_{\as \in \cL_{\a}}  \as$ coincides with $n_p(\psi_2)$.
Therefore,
\begin{equation}
a(\psi_2,\g)=\# \d_{\a}(\psi_2)\cap \g=\# \d(D(\psi_2))\cap \g\equiv \# D(\psi_2)\cap \d\g=0,
\label{eq:a-psi-g=0-pure-alpha}
\end{equation}
since $\d \g=\emptyset$. This establishes Equation~\eqref{eq:a-gamma=0-non-mixed}.

Finally, we consider the ends of the moduli spaces $\psi$, as above, when $k=j=1$ (i.e. the $\psi$ is a class of disks). In this case, ends of the form ~\eqref{list:end2} and~\eqref{list:end3} do not appear, though ends of the form~\eqref{list:end1} contribute the $\ys$ coefficient of $(\cA_{\g}\circ \d+\d\circ \cA_{\g})(\xs)$ as before. In this case, there is an additional type of end which appears when $\xs=\ys$:
 \begin{enumerate}[label=($e$-\arabic*), ref=$e$-\arabic*]
 \setcounter{enumi}{3}
 \item\label{list:end4} A constant holomorphic strip at $\xs$, together with a Maslov index 2 boundary degeneration.
 \end{enumerate}
 See \cite{OSLinks}*{Section~5} for more on boundary degenerations in Heegaard Floer theory. We claim that boundary degenerations make trivial algebraic contribution to the weighted sum of the ends of the moduli spaces under consideration. This follows from the fact that if $\psi_2$ is a class of boundary degenerations, then the same argument as Equation~\eqref{eq:a-psi-g=0-pure-alpha} establishes that $a(\psi_2,\g)=0$.

Summing all ends as above, we conclude that $\d_{\Mor}(\cA_{\g})=0$.
\end{proof}

There is a more streamlined presentation of the proof of the above lemma. It is helpful to view $\g$ itself as a type of Floer morphism from a set of attaching curves $\as$ to itself.  We refer to $\g$ as a \emph{formal endomorphism} of $\as$. We can define holomorphic polygon counts with $\g$ as an input as follows. If $\ds_{1},\dots, \ds_n$ are attaching curves, we define 
\[
f_{\dt_1,\dots, \dt_j,\dt_j,\dots, \dt_n}(\Theta_{1,2},\dots, \Theta_{j-1,j},\g, \Theta_{j,j+1},\dots, \Theta_{n-1,n})
\]
to count holomorphic $n$-gons of Maslov index $3-n$ which are weighted by a factor or $\# \d_{\dt_j}(\psi)\cap \g$. If we formally set $\d \g=0$, then it is straightforward to adapt the argument from the proof of Lemma~\ref{lem:homology-action-chain-homotopy} to see that standard associativity relations hold if we allow for a formal endomorphism as an argument.

We can extend the above formalism by defining a formal endomorphism of hypercubes of attaching curves:
\[
F_\g \colon \cL_{\a}\to \cL_{\a}.
\]
The morphism $L_{\g}$ has only length 1 components, all of which are $\g$. There are no higher length components. We will think of $F_{\g}$ as a morphism of twisted complexes (see Section~\ref{sec:twisted-complexes}).

Equation~\eqref{eq:a-gamma=0-non-mixed} immediately implies
\[
\mu_1^{\Tw}(F_\g)=0.
\]
Furthermore, by definition 
\[
\cA_\g=\mu_2^{\Tw}(F_\g,-).
\]
The $A_\infty$-associativity conditions for the category of twisted complexes implies that  
\begin{equation}
\d_{\Mor}(\cA_\g):=\mu_1^{\Tw}\left( \mu_2^{\Tw}(F_\g,-)\right)+\mu_2^{\Tw}\left(F_\g,\mu_1^{\Tw}(-)\right)=0,
\label{eq:associativie=>A_g-chain-map}
\end{equation}
which is the statement of Lemma~\ref{lem:homology-action-chain-homotopy}.

Another important property of the homology action is the following:

\begin{lem} 
\label{lem:homology-action-homotopy-hypercube}
Suppose that $\cL_{\a}$ and $\cL_{\b}$ are hypercubes of attaching curves on $(\Sigma,\ws,\zs,\ps)$ and that $\g$ is a closed 1-chain on $\Sigma$. Suppose that $C\subset \Sigma$ is an integral 2-chain such that $\d C=\g+S_{\a}+S_{\b}$ where $S_{\a}$ are closed 1-chains which are disjoint from all curves in $\cL_{\a}$ and $S_{\b}$ are closed 1-chains which are disjoint from all curves in $\cL_{\b}$. Then
\[
\cA_\g\simeq 0
\]
as morphisms of hypercubes.
\end{lem}
\begin{proof}
We construct the following diagram to realize the chain homotopy:
\[
\begin{tikzcd} \ve{\CF}^-(\cL_{\a},\cL_{\b})
	\ar[r, "\cA_{\g}"]
	\ar[d, "\id"]
	\ar[dr, "H_C",dashed]
& \ve{\CF}^-(\cL_{\a},\cL_{\b})
	\ar[d, "\id"]
\\
\ve{\CF}^-(\cL_{\a},\cL_{\b})
	\ar[r, "0"]&
\ve{\CF}^-(\cL_{\a},\cL_{\b})
\end{tikzcd}
\]
We define the map $H_C$ to have only length 2 chains in the above diagram (i.e. to increment the $\gamma$ and $\id$ directions, but no other directions), and to send $\xs\in \ve{\CF}^-(\as_{\veps},\bs_{\nu})$ 
to $n_{\xs}(C)\cdot \xs\in \ve{\CF}^-(\as_{\veps},\bs_{\nu})$, and to be equivariant in actions of the variables.

We claim that the hypercube relations are satisfied. To see this, it is sufficient to show that if
\[
\veps_1<\cdots<\veps_n\quad \text{and} \quad \nu_1<\cdots<\nu_m
\]
are increasing sequences in the cubes for $\cL_{\a}$ and $\cL_{\b}$, respectively, then
\begin{equation}
H_C\circ D_{\veps_1<\cdots<\veps_n}^{\nu_1<\cdots<\nu_m}+D_{\veps_1<\cdots<\veps_n}^{\nu_1<\cdots<\nu_m}\circ H_C=(\cA_\g)_{\veps_1<\cdots<\veps_n}^{\nu_1<\cdots<\nu_m}
\label{eq:homology-action-homotopy-hypercube}
\end{equation}
In the above, $D_{\veps_1<\cdots<\veps_n}^{\nu_1<\cdots<\nu_m}$ denotes the summands of the hypercube differential of $\ve{\CF}^-(\cL_{\a},\cL_{\b})$ which counts curves on $(\Sigma,\as_{\veps_n},\dots, \as_{\veps_1},\bs_{\nu_1},\dots, \bs_{\nu_m})$.

In order to prove Equation~\eqref{eq:homology-action-homotopy-hypercube}, we consider a class
\[
\psi\in \pi_2(\Theta_{\veps_n,\veps_{n-1}},\dots, \Theta_{\veps_2,\veps_1} ,\xs,\Theta_{\nu_1,\nu_2},\cdots ,\Theta_{\nu_{m-1},\nu_m},\ys)
\]
  of $(n+m)$-gons, and consider $\d_{\a}(\psi)\cap C\subset \Sigma$. Since $\d_{\a}(\psi)$ is 1-dimensional and $C$ is 2-dimensional, the intersection $\d_{\a}(\psi) \cap C$ is 1-dimensional. We compute $\d (\d_{\a}(\psi) \cap C)$ using the Leibniz rule for intersections, which yields:
\[
\begin{split}
\d(\d_{\a}(\psi)\cap C)\equiv &\# (\d \d_{\a}(\psi))\cap C+\# \d_{\a}(\psi)\cap \d C \pmod{2} \\
\equiv& n_{\ve{x}}(C)+n_{\ve{y}}(C)+\# \d_{\a}(\psi)\cap (\g+S_{\a}+S_{\b}) \pmod{2}.
\end{split}
\]
We first note that $\d_{\a}(\psi)\cap S_{\a}=\d_{\b}(\psi)\cap S_{\b}=\emptyset$. Furthermore $\d_{\a}(\psi)$ and $\d_{\b}(\psi)$ are homologous via the 2-chain $D(\psi)$, so we also have $\# \d_{\a}(\psi)\cap S_{\b}\equiv 0$. Hence
\[
n_{\ve{x}}(C)+n_{\ve{y}}(C)\equiv \#\d_{\a}(\psi)\cap \g\pmod 2.
\]
This implies Equation~\eqref{eq:homology-action-homotopy-hypercube}, completing the proof.
\end{proof}

\subsection{Relative homology actions and hypercubes}
\label{sec:relative-homology-actions}

We now define an endomorphism $\cA_{\lambda}$ for arcs $\lambda$ on $\Sigma$ which have boundary on two basepoints. We focus on the case that $\d \lambda$ consists of two free basepoints, $p_i$ and $p_j$, to simplify the exposition.

In the case of the 3-manifold invariants, if $(Y,\ps)$ is a multi-pointed 3-manifold and $\lambda$ is an arc connecting two basepoints $p_i,p_j\in \ps$, then there is an endomorphism
\[
A_{\lambda}\colon \ve{\CF}^-(Y,\ps)\to \ve{\CF}^-(Y,\ps)
\]
which satisfies
\[
\d_{\Mor}(A_{\lambda})=U_{i}+U_{j}.
\]
See \cite{ZemGraphTQFT}*{Section~5}.

We now describe how to extend the construction of this map into the setting of hypercubes. When $\cL_{\a}$ and $\cL_{\b}$ are hypercubes of strongly equivalent attaching curves on $(\Sigma,\ws,\zs,\ps)$ and $\lambda$ connects two free basepoints $p_i$, $p_j$, we will construct a morphism of hypercubes
\[
\cA_{\lambda}\colon \ve{\CF}^-(\cL_{\a},\cL_{\b})\to \ve{\CF}^-(\cL_{\a},\cL_{\b})
\]
which satisfies
\[
\d_{\Mor}(\cA_{\lambda})=U_i+U_j.
\]

Similar to the case of closed curves $\g$, we will define
\[
\cA_{\lambda}=\mu_2^{\Tw}(F_{\lambda},-)
\]
for some formal endomorphism
\[
F_{\lambda}\colon \cL_{\a}\to \cL_{\a}.
\]
The map $F_{\lambda}$ will have length 1 components equal to the formal endomorphisms $\lambda\colon \as_{\veps}\to \as_{\veps}$. However, unlike for closed curves, we will need to add higher length arrows to the map $F_{\lambda}$. The main issue is that if $\psi$ is a class of $\ell$-gons on a subdiagram of $(\Sigma,\cL_{\a},\ws,\zs,\ps)$, then in contrast to Equation~\eqref{eq:a-gamma=0-non-mixed} we have
\begin{equation}
a(\psi,\lambda)\equiv n_{p_i}(\psi)-n_{p_j}(\psi)\pmod 2,
\label{eq:a(psi,lambda)=npnp}
\end{equation}
which may be non-zero. Therefore the proof of Lemma~\ref{lem:homology-action-chain-homotopy} does not carry over without modification.

To construct $F_{\lambda}$ and understand its formal properties, we have to first expand our notation of a formal endomorphism slightly. If $\as$ is a set of attaching curves, we will also consider formal morphisms of the form $a\cdot \bI\colon \as\to \as$, where $a$ is a polynomial in the $\scU_i$, $\scV_i$ and $U_j$ variables. 
We define holomorphic polygon maps with $a \cdot\bI$ as an input by declaring them to be $\scU_i$, $\scV_i$ and $U_i$-equivariant, and also strictly unital. That is, we declare
\[
f_{\a_1,\dots, \a_j, \a_j,\dots, \a_n}(\xs_1,\dots, a\cdot \bI,\dots, \xs_{n-1})
\]
to vanish unless $n=2$. We define
\[
f_{\a_1, \a_2,\a_2}(\xs, a\cdot\bI)=a\cdot \xs\quad \text{and} \quad f_{\a_1,\a_1,\a_2}(a\cdot\bI, \xs)=a\cdot \xs.
\]
We also declare $\d(\bI)=0$.

We can extend the above construction to define a formal endomorphism $\bI\colon \cL_{\a}\to \cL_{\a}$ for any hypercube of attaching curves $\cL_{\a}$. The morphism $\bI$ has only length 1 components, all of which are of the form $\bI_{\a_\veps}$.

We will construct $F_{\lambda}$ so that
\[
\mu_1^{\Tw}(F_{\lambda})=(U_{i}+U_{j})\cdot \bI.
\]
Once we construct $F_{\lambda}$, we will define
\[
A_{\lambda}(-):=\mu_2^{\Tw}(F_{\lambda},-).
\]

 In the following lemma, we prove that it is always possible to construct the higher length chains of $F_\lambda$ when $\cL_{\a}$ is a hypercube of handleslide equivalent curves:

\begin{lem}\label{lem:d-mor-A-lambda}
Suppose that $\cL_{\a}$ is a hypercube of handleslide equivalent attaching curves on $(\Sigma,\ws,\zs,\ps)$, and let $\lambda$ be an immersed arc on $\Sigma$.
\begin{enumerate}
\item If $\d\lambda=\{p_i,p_j\}$, where $p_i,p_j$ are free basepoints, then $L_\lambda$ may be constructed so that
\[
\mu_1^{\Tw}(F_{\lambda})=U_{i}+U_j.
\]
\item More generally, if $\d \lambda=\{x_1,x_2\}\subset \ws\cup \zs\cup \ps$, then $F_\lambda$ may be constructed so that
\[
\mu_1^{\Tw}(F_{\lambda})=E_{x_1}^\a+E_{x_2}^{\a},
\]
where $E_{x_i}^{\a}$ denotes the product of all variables for basepoints in the component of $\Sigma\setminus \as_{\veps}$ which contains $x_i$. (Note that this is independent of $\veps$, since $\cL_{\a}$ consists of handleslide equivalent attaching curves). 
\end{enumerate}
\end{lem}
\begin{rem}
\label{rem:remark-dmor-A-lambda}
\item
\begin{enumerate}
\item It follows from the associativity relations for $\mu_i^{\Tw}$ (as in Equation~\eqref{eq:associativie=>A_g-chain-map}) that if $\cL_\b$ is another hypercube of attaching curves and $\d \lambda=\{p_i,p_j\}$ consists of free basepoints, then
\[
\d_{\Mor}(\cA_{\lambda})=(U_i+U_j)\cdot \bI
\]
as an endomorphism of $\ve{\CF}^-(\cL_{\a},\cL_{\b})$. The second part of the lemma has a similar consequence.
\item  As an example of the second part of the lemma, if $(\Sigma,\as,\bs,\ws,\zs)$ represents a link $L\subset Y$ for some $\as\in \cL_{\a}$ and $\bs\in \cL_{\b}$, and $x_1=w_i$ and $x_2=w_j$ are two link basepoints on components $K_i$ and $K_j$ of $L$, then
\[
\d_{\Mor}(\cA_{\lambda})=(\scU_{i}\scV_{i}+\scU_{j}\scV_{j})\cdot \id.
\]
\end{enumerate}
\end{rem}
\begin{proof} We focus on the first part of the lemma, as the second part is not substantially different. The proof is described in \cite{HHSZNaturality}*{Section~6.2} when we set $U_i=U_j$, though we review it in detail since we need to prove additional properties of the maps $\cA_{\lambda}$.

 We begin by considering the length 2 arrows which increment the $\lambda$ direction. If $\as_1$ is an immediate successor of $\as_0$ in $\cL_{\a}$, then we wish to construct a chain $\eta_{\a_1,\a_0}$ so that the diagram below is a hypercube (except for the length 1 relations along $\lambda$):
\[
\begin{tikzcd}[labels=description, column sep=2cm, row sep=2cm]
\as_0
	\ar[d, "\Theta_{\a_1,\a_0}"]
	\ar[r, "\lambda"]
	\ar[dr,dashed, "\eta_{\a_1,\a_0}"]
& \as_0
	\ar[d, "\Theta_{\a_1,\a_0}"]
\\
\as_1 \ar[r, "\lambda"]& \as_1
\end{tikzcd}
\]
Concretely, this amounts to finding $\eta_{\a_1,\a_0}\in \ve{\CF}^-(\Sigma,\as_1,\as_0)$ such that
\[
\d \eta_{\a_1,\a_0}=A_\lambda^{\a_0}(\Theta_{\a_1,\a_0})+A^{\a_1}_{\lambda}(\Theta_{\a_1,\a_0}).
\]
Above, the map $A_{\lambda}^{\a_0}$ weights curves by $\# (\d_{\a_0}(\phi)\cap \lambda)$ while $A^{\a_1}_{\lambda}$ weights curves by $\#(\d_{\a_1}(\phi)\cap \lambda)$.

We recall that for the ordinary Floer complexes, 
\[
A_{\lambda}^{\a_0}(\Theta_{\a_1,\a_0})+A^{\a_1}_{\lambda}(\Theta_{\a_1,\a_0})=U_{p_i}\Phi_{p_i}(\Theta_{\a_1,\a_0})+U_{p_j}\Phi_{p_j}(\Theta_{\a_1,\a_0}).
\]
See \cite{ZemGraphTQFT}*{Lemma~5.7}. However $\Phi_{p_i}(\Theta_{\a_1,\a_0})$ and $\Phi_{p_j}(\Theta_{\a_1,\a_0})$ are cycles of Maslov $(\gr_{\ws\cup \ps},\gr_{\zs\cup \ps})$-bigrading $(1,1)$ larger than the canonical generator of $\ve{\HF}^-(\Sigma,\as_1,\as_0)$, and hence are each boundaries. We note that $\eta_{\a_1,\a_0}$ may be chosen to be in the submodule of $\ve{\CF}^-(\Sigma,\as_1,\as_0)$ generated by the ideal $(U_{p_i},U_{p_j})$.

We now define the length 3 morphisms of $F_\lambda$ which increase the $\lambda$-direction. Suppose that $\as_{00}$ and $\as_{11}$ differ by two coordinates of $\cL_{\a}$. We consider the 3-dimensional cube spanned by $\as_{00}$, $\as_{10}$, $\as_{01}$, $\as_{11}$, extended into the $\lambda$ direction. All  length 2 hypercube relations are satisfied. We sum the terms which would appear in the length 3 relation of this hypercube except for contribution of the length 3 arrow. Write
\[
C_{\a_{11},\a_{00}}\in \ve{\CF}^-(\as_{11},\as_{00})
\]
for this element.  We observe three facts:
\begin{enumerate}[label=($C$-\arabic*), ref=$C$-\arabic*]
\item $C_{\a_{11},\a_{00}}$ is a cycle.
\item $C_{\a_{11},\a_{00}}$ has Maslov bigrading equal to that of the canonical top degree generator of $\ve{\HF}^-(\as_{11},\as_{00})$.
\item\label{claim-C-submoduleU} $C_{\a_{11},\a_{00}}$ is in the submodule generated by the ideal $(U_{p_i},U_{p_j})$.
\end{enumerate}
The first two equations are easily verified. To see the third, we observe that the length 3 relation is obtained by summing over all ways of taking a consecutive sequence of morphisms in the cube, and then inserting $\lambda$ between 2 terms or at the start or beginning and then applying the holomorphic polygon maps interpreted as above. As an example, the expression $C_{\a_{11},\a_{00}}$ contains the sum
\[
\begin{split}&f_{\a_{11},\a_{01},\a_{00},\a_{00}}( \Theta_{\a_{11},\a_{01}},\Theta_{\a_{01},\a_{00} },\lambda)\\
+&f_{\a_{11},\a_{01},\a_{01},\a_{00}}( \Theta_{\a_{11},\a_{01}},\lambda,\Theta_{\a_{01},\a_{00} })\\
+&f_{\a_{11},\a_{11},\a_{01},\a_{00}}(\lambda,\Theta_{\a_{11},\a_{01}},\Theta_{\a_{01},\a_{00} }).
\end{split}
\]
The above expression may be collapsed into a count of holomorphic triangles which are weighted by $n_{p_i}(\psi)+n_{p_j}(\psi)$. In particular, the output lies in the submodule generated by $U_{p_i}$ and $U_{p_j}$. The same reasoning holds for the other terms in $C_{\a_{11},\a_{00}}$.

 Since $\as_{11}$ and $\as_{00}$ are handleslide equivalent, the canonical generator of $\ve{\HF}^-(\as_{11},\as_{00})$ has both the maximal $\gr_{\ws\cup \ps}$ and maximal $\gr_{\zs\cup \ps}$-grading of any non-zero element, and furthermore, the canonical element is the unique non-zero element of homology in this grading. Furthermore, it is straightforward to  verify (e.g. by a model computation on a simple diagram) that the quotient map 
\begin{equation}
\Pi\colon \ve{\CF}^-(\as_{11},\as_{00})\to \ve{\CF}^-(\as_{11},\as_{00})/(U_{i},U_{j})
\label{eq:pi-map-quotient-variables}
\end{equation}
is non-vanishing on the canonical generator of top degree. This map $\Pi$ vanishes on $C_{\a_{11},\a_{00}}$ because $C_{\a_{11},\a_{00}}$ is in the ideal $(U_{i},U_j)$. We note that $C_{\a_{11},\a_{00}}$ has the same grading as the top degree element of $\ve{\HF}^-(\as_{11},\as_{00})$, and hence $C_{\a_{11},\a_{00}}$, being zero on homology, must be a boundary in $\ve{\CF}^-(\as_{11},\as_{00})$. We pick any homogeneously graded chain $\omega_{\a_{11},\a_{00}}$ such that $\d \omega_{\a_{11},\a_{00}}=C_{\a_{11},\a_{00}}$ to be the length 3 map in our cube.

Filling in the higher dimensional cubes proceeds in a similar, though slightly simpler manner. In this case, when we want to fill over an $n$-dimensional subcube for $n>3$, all of whose morphisms except for the longest morphism are defined, the corresponding element $C$ is a cycle of bidegree $(n-3,n-3)$ greater than the canonical generator, and hence will automatically be a cycle.
\end{proof}

\begin{rem}
If one restricts to the case that $\cL_{\a}$ is algebraically rigid, then we can take $F_{\lambda}$ to have only length 1 components, which consist of the formal endomorphisms $\lambda$. This is because if $\psi\in \pi_2(\Theta_{\a_{\veps_{n}},\a_{\veps_{n-1}}},\dots, \Theta_{\a_{\veps_2},\a_{\veps_1}}, \ys)$ is a class of polygons with expected dimension 0 on $(\Sigma,\as_{\veps_n},\dots, \as_{\veps_1},\ws,\zs,\ps)$, then the Maslov grading formula in Heegaard Floer theory implies that
\[
\mu(\psi)=\gr_{\ws\cup \ps}(\Theta_{\a_{\veps_n},\a_{\veps_1}},\ys)+2n_{\ws\cup \ps}(\psi).
\]
Since $\cL_{\a}$ is algebraically rigid, $\gr_{\ws\cup \ps}(\Theta_{\a_{\veps_n},\a_{\veps_1}},\ys)\ge 0$. It follows that, if $\mu(\psi)=3-n$ (so that $\cM(\psi)$ is 0-dimensional) then there are no holomorphic representatives of $\psi$ if $n\ge 3$ and if $n=2$ or $n=3$, then we have $n_{\ws\cup \ps}(\psi)=0$. In particular, $n_{p_i}(\psi)=n_{p_j}(\psi)=0$ and hence
\[
\#\d_{\a}(\psi)\cap \lambda=0
\]
by Equation~\eqref{eq:a(psi,lambda)=npnp}. In particular, we can argue as in Lemma~\ref{lem:homology-action-chain-homotopy} to see that $\d(F_{\lambda})=(U_i+U_j)\cdot \bI$, when $F_{\lambda}$ is defined with only length 1 morphisms. Therefore, when $\cL_{\a}$ is algebraically rigid, we can define $\cA_{\lambda}=\mu_2^{\Tw}(F_{\lambda},-)$ (with $F_{\lambda}$ having no higher length arrows) just as we defined $\cA_{\g}$ in Lemma~\ref{lem:homology-action-chain-homotopy}, whereas in general we need to add higher length arrows to $F_{\lambda}$. 
\end{rem}

The following is a technical result which is useful when we prove properties about $\cA_{\lambda}$ (cf. \cite{HHSZNaturality}*{Lemma~6.2}).

\begin{lem}
\label{lem:higher-length-chains-in-submodule-A-lambda} If $\d \lambda=\{x_1,x_2\}$, where $x_i\in \ws\cup \zs\cup \ps$, then all of the higher length morphisms of $F_\lambda$ may be chosen to be in the submodule spanned by $(E_{x_1},E_{x_2})$.
\end{lem}
\begin{proof} We focus on the case that $\d \lambda=\{p_i,p_j\}$ are free basepoints, so $E_{x_1}=U_i$ and $E_{x_2}=U_j$, since the general case is no different. In the proof of Lemma~\ref{lem:d-mor-A-lambda} we observed that the length 2 components of $F_\lambda$ could be chosen to be in submodule generated by the ideal $(U_i,U_j)$. Consider the length 3 arrows of $F_\lambda$ (the argument for higher length chains is essentially the same). Following the proof of Lemma~\ref{lem:d-mor-A-lambda}, let $C_{\a_{11},\a_{00}}$ be the sum of the terms in the length 3 relation, excluding the term $\d\omega_{\a_{11},\a_{00}}$, which has yet to be defined.

We observed in Lemma~\ref{lem:d-mor-A-lambda} that $C_{\a_{11},\a_{00}}$ was a boundary, and we picked some $\omega_{\a_{11},\a_{00}}$ in bigrading $(1,1)$ higher than the canonical generator such that $\d \omega_{\a_{11},\a_{00}}=C_{\a_{11},\a_{00}}$. Write 
\[
\omega_{\a_{11},\a_{00}}=U_i\omega_i+U_j  \omega_j+\omega_0,
\] 
where $\omega_0$ has no powers of $U_i$ or $U_j$. Note that the image $\Pi(\omega)$ of $\omega_0$  in $\ve{\CF}^-(\as_{11},\as_{00})/(U_i,U_j)$ is a cycle because in Claim~\eqref{claim-C-submoduleU} in the proof of Lemma~\ref{lem:d-mor-A-lambda} we showed that $C_{\a_11,\a_{00}}$ is in the submodule spanned by $(U_i,U_j)$. Further, $\omega$ has bigrading $(1,1)$ higher than the top degree. Hence $\Pi(\omega_0)$ is a boundary, i.e. $\Pi(\omega_0)=\d (\xs_0)$. We lift $\xs_0$ to $\ve{\CF}^-(\as_{11},\as_{00})$ to get a chain $\xs$ of homogeneous bigrading $(1,1)$ higher than $\omega_0$, so that
\[
\d \ve{x}=\omega_0+U_{i} \ve{y}_i+U_{j} \ve{y}_j,
\]
for some $\ve{y}_i,\ve{y}_j$. Note that 
\[
\d(\omega_{\a_{11},\a_{00}}+\d \xs)=C_{\a_{11},\a_{00}}
\]
and $\omega_{\a_{11},\a_{00}}+\d \xs$ is in the submodule spanned by $(U_i,U_j)$, so we use $\omega_{\a_{11},\a_{00}}+\d \xs$ in place of $\omega_{\a_{11},\a_{00}}$. The proof for chains of length $n>3$ follows from essentially the same logic.
\end{proof}

\begin{rem}\label{rem:well-defined-F_lambda} Applying the essentially the same construction as above shows that $F_{\lambda}$ is well-defined up to chain homotopy. If $F_{\lambda}$ and $F_{\lambda}'$ are two constructions of this morphism, then we can apply the filling construction to build a diagonal morphism $h$ in the following diagram
\[
\begin{tikzcd}[labels=description, column sep=1cm, row sep=1cm] \cL_{\a} \ar[r, "F_{\lambda}"] \ar[dr,dashed, "h"] \ar[d, "\bI"] & \cL_{\a} \ar[d, "\bI"] \\
\cL_{\a} \ar[r, "F_{\lambda}'"] & \cL_{\a}
\end{tikzcd}
\]
which gives a chain homotopy of generalized endomorphisms
\[
\d_{\Mor}(h)=F_{\lambda}+F_{\lambda}'.
\]
\end{rem}

\subsection{Free stabilization}

In \cite{ZemGraphTQFT}*{Section~6}, the author describes maps for adding and removing basepoints to the Heegaard Floer complexes. These are referred to as \emph{free-stabilization} maps. One can also define a free-stabilization map on the level of hypercubes. These are described in \cite{HHSZNaturality}*{Section~7.1}, and we recall the construction presently.

Suppose that $\cL_{\a}$ and $\cL_{\b}$ are hypercubes of
of attaching curves on a surface $(\Sigma,\ws,\zs,\ps)$, and that $p$ is a point on $\Sigma\setminus (\ws\cup \zs\cup \ps)$, which is not contained in any curves of $\cL_{\a}$ or $\cL_{\b}$. We define stabilized hypercubes $\cL_{\a}^+$ and $\cL_{\b}^+$. We construct the attaching curves of $\cL_{\a}^+$ and $\cL_{\b}^+$ as follows. For each $\as$ in $\cL_{\a}$, we adjoin a new component which consists of a small circle bounding a disk centered at $p$. We define $\cL_{\b}^+$ similarly. We translate these curves slightly to achieve admissibility. If $\veps<\veps'$ are points in the cube for $\cL_{\a}$, and $\Theta_{\a_{\veps'},\a_{\veps}}$ is the corresponding chain of $\cL_{\a}$, then we define corresponding chain in $\cL_{\a}^+$ to be $\Theta_{\a_{\veps'},\a_{\veps}}\otimes \theta^+$, where $\theta^+$ is the top graded generator of the new attaching curves centered at $p$. We make a similar definition for the chains of $\cL_{\b}^+$. 

We define a map
\[
\cS^+_p\colon \ve{\CF}^-(\cL_{\a},\cL_{\b})\to \ve{\CF}^-(\cL_{\a}^+,\cL_{\b}^+)
\]
via the formula $\xs\mapsto \xs\otimes \theta^+$, extended equivariantly over the variables $\scU_{i}$, $\scV_{i}$ and $U_j$.

There is also a map $\cS^-_p$ in the opposite direction, defined dually by the formula $\xs\otimes \theta^+\mapsto 0$ and $\xs\otimes \theta^-\mapsto \xs$, extended equivariantly over the variables $\scU_{i}$, $\scV_{i}$ and $U_j$. 

\begin{lem}\label{lem:Sw-chain-maps}
 The hypercube free-stabilization maps, $\cS_p^+$ and $\cS_p^-$, are chain maps.
\end{lem}

The proof follows from a standard index computation and gluing argument. See \cite{HHSZNaturality}*{Proposition~5.5} for a proof in the general case, and \cite{ZemGraphTQFT}*{Propositions~6.5 and Theorem 6.7} for the case of holomorphic disks and triangles.

The fact that $\cS_p^+$ and $\cS_p^-$ are chain maps of hypercubes may be interpreted as implying a form of naturality for the free-stabilization maps, similar to Remark~\ref{rem:naturality-hypercubes}. Suppose $\scH=(\Sigma,\cL_{\a},\cL_{\b})$ and $\scH'=(\Sigma,\cL_{\a'},\cL_{\b'})$ are diagrams of hypercubes of attaching curves and that the curves of $\cL_{\a}$ and $\cL_{\a'}$ are pairwise handleslide equivalent, and similarly for $\cL_{\b}$ and $\cL_{\b'}$. Write $\scH^+_p$ and ${\scH'}_p^+$ for their stabilizations.
Using the notation of Remark~\ref{rem:naturality-hypercubes}, we see that Lemma~\ref{lem:Sw-chain-maps} implies the (strict) commutation of the diagram
\[
\begin{tikzcd}[column sep=2cm]
\ve{\CF}^-(\scH)\ar[r, "\Psi_{\scH\to \scH'}"]\ar[d, "\cS_p^+"]& \ve{\CF}^-(\scH')\ar[d, "\cS_p^+"]\\
\ve{\CF}^-(\scH^+_p)\ar[r,"\Psi_{\scH^+_p\to {\scH'}_p^+}"]& \ve{\CF}^-({\scH_p'}^+)
\end{tikzcd}
\]

Similar to the setting of the graph TQFT, we have the following relation between the free-stabilization maps and the basepoint actions from Section~\ref{sec:basepoint-actions-hypercubes}.

\begin{lem}
\label{lem:Phi-SS}
For appropriately degenerated almost complex structures, we have $\cS_p^+\cS_p^-=\Phi_p$.
\end{lem}
\begin{proof}
The argument is essentially the same as \cite{HHSZNaturality}*{Proposition~5.5}, though we repeat the argument here for completeness. We consider a class of polygons
\[
\psi\# \psi'\in \pi_2(\Theta_{\a_n,\a_{n-1}}\times \theta^+,\dots, \Theta_{\a_2,\a_1}\times \theta^+,\xs\times \theta,\Theta_{\b_1,\b_2}\times \theta^+,\dots, \Theta_{\b_{m-1},\b_m}\times \theta^+,\ys\times \theta')
\]
 which potentially contributes to the differential on the hypercube $\ve{\CF}^-(\cL_{\a},\cL_{\b})$. Here, $\theta,\theta'\in \{\theta^+,\theta^-\}$. Here, $\psi'$ has is a class in the stabilization region, and $\psi$ is a class on the unstabilized diagram. By a standard computation
\begin{equation}
\mu(\psi\# \psi')=\mu(\psi)+2n_p(\psi')+\gr(\theta,\theta').\label{eq:index-connected-sum-Phi-w}
\end{equation}
See \cite{ZemGraphTQFT}*{Equation~6.3}. The hypercube differential counts solutions with $\mu(\psi\# \psi')=3-m-n$. 

In Equation~\eqref{eq:index-connected-sum-Phi-w}, the term $\gr(\theta,\theta')$ is in $\{-1,0,1\}$. For sufficiently stretched almost complex structure, we know that $\mu(\psi)\ge 3-m-n$ if $m+n\ge 3$, by transversality. Therefore it follows that when $m+n\ge 3$, we have
\[
(\mu(\psi), \gr(\theta,\theta'), n_p(\psi'))\in \{ (3-m-n,0,0), (4-m-n,-1, 0)\}.
\]
In particular, $n_p(\psi)=0$ so such curves do not contribute to $\Phi_p$ because $n_p(\psi')=0$. 

When $m+n=2$, it is possible for $\mu(\psi)=2-m-n=0$ and for $\psi$ to have a representative, though such curves must represent the constant class. The only possibility with $n_p(\psi')=1$ is when $\gr(\theta,\theta')=-1$, $\mu(\psi)=0$ and $n_p(\psi')=1$. These curves have domain equal to the bigon in the stabilization region which covers $p$. These are the only curves which are counted by the hypercube map $\Phi_p$, and we see that
\[
\Phi_p=\cS^+_p \circ \cS_p^-.
\]
\end{proof}

The same neck-stretching argument as in Lemma~\ref{lem:Sw-chain-maps} can be used to prove the following (cf. \cite{ZemGraphTQFT}*{Corollary~6.6}):

\begin{lem}\label{lem:commute-S-A-lam}
 Suppose that $\lambda$ is an arc or closed path which is disjoint from $p$. For an appropriately degenerated almost complex structure, we have
\[
\left[\cA_\lambda, \cS_p^{\pm}\right]= 0.
\]
\end{lem}
We leave the proof of the above result to the reader, as it is very similar to Lemmas~\ref{lem:Sw-chain-maps} and~\ref{lem:Phi-SS}.

Similarly, we have the following analog of \cite{ZemGraphTQFT}*{Proposition~6.14}, whose proof we also leave to the reader:
\begin{lem}\label{lem:commute-free-stabilizations} If $p$ and $p'$ are distinct free basepoints, then $\left[\cS^{\circ_1}_p,\cS^{\circ_2}_{p'}\right]\simeq 0$, where $\circ_1,\circ_2\in \{+,-\}$.
\end{lem}

\subsection{Algebraic relations and the diffeomorphism action}
\label{sec:algebraic-relations}

In this section, we prove some important algebraic relations involving the maps $\cA_\lambda$ and $\cS_p^+$. The relations we prove are analogs of the relations from the graph TQFT \cite{ZemGraphTQFT}.

In the following lemma, we collect several first relations about the maps $\cA_{\lambda}$ and $\cS_p^{\pm}$:

\begin{lem}
\label{lem:basic-graph-TQFT-relations} Suppose that $\cL_{\a}$ and $\cL_{\b}$ are hypercubes of handleslide equivalent attaching curves on $(\Sigma,\ws,\zs,\ps)$.
\begin{enumerate}
\item If $p\in \ps$, then $\Phi_p^2\simeq 0$. The same holds for basepoints in $\ws\cup \zs$.
\item  Suppose that $p,p'$ are distinct free basepoints, and that $\lambda$ is an arc on $(\Sigma,\ws,\zs,\ps)$ such that $\d \lambda=\{p,p'\}$. Then
\[
[\cA_{\lambda},\Phi_p]\simeq [\cA_{\lambda}, \Phi_{p'}]\simeq \id.
\]
\item Suppose that $p\in \ps$  is a free basepoint, and $p'\in \Sigma\setminus (\ws\cup \zs\cup \ps)$ is a point which is not contained in $\cL_{\a}$ or $\cL_{\b}$. If $\lambda$ is an arc with $\d \lambda=\{p,p'\}$, then
\[
\cS_{p'}^- \cA_{\lambda} \cS_{p'}^+\simeq \id
\]
as endomorphisms of $\ve{\CF}^-(\cL_{\a},\cL_{\b},\ws,\zs,\ps)$. Here, we view $\cS_{p'}^+$ as having codomain $\ve{\CF}^-(\cL_{\a}^+,\cL_{\b}^+,\ws,\zs,\ps\cup \{p'\})/(U_p+U_{p'})$. 
\end{enumerate}
\end{lem}
\begin{proof}
The claim that $\Phi_p^2\simeq 0$ follows from the argument in \cite{SarkarMovingBasepoints}*{Lemma~4.4}. Cf. \cite{ZemGraphTQFT}*{Lemma~14.8}. 

 The proof of the second claim is essentially the same as in the case of the ordinary 3-manifold invariants \cite{ZemGraphTQFT}*{Lemma~14.6}. We differentiate the relation $[\d, \cA_{\lambda}]=(U_p+U_{p'})\cdot \id$ from Lemma~\ref{lem:d-mor-A-lambda} with respect to the $U_p$ to obtain the relation $[\Phi_p,\cA_{\lambda}]\simeq\id$. The same logic gives $[\Phi_{p'}, \cA_{\lambda}]\simeq \id$.

The final claim that $\cS_{p'}^- \cA_{\lambda} \cS^+_{p'}\simeq \id$ is proven in \cite{HHSZNaturality}*{Proposition~7.2}. Compare \cite{ZemGraphTQFT}*{Lemma~7.10} for the case of the ordinary Heegaard Floer complexes.
\end{proof}

The following is a hypercube analog of \cite{OSDisks}*{Proposition~4.17} and \cite{ZemGraphTQFT}*{Lemma~5.5}.

\begin{lem}\label{lem:homology-action-squares}Let $\cL_{\a}$ and $\cL_{\b}$ be hypercubes of handleslide equivalent attaching curves on $(\Sigma,\ws,\zs,\ps)$. If $\g$ is a closed curve on $\Sigma$ then \[
\cA_{\g}^2\simeq 0.
\] If $\lambda$ is an arc which connects two distinct free basepoints $p_1,p_2\in \ps$, then  
\[
\cA_{\lambda}^2\simeq U_1\cdot \id\simeq U_{2}\cdot \id.
\]
\end{lem}

\begin{proof}We begin by considering the claim for $\cA_{\g}$, when $\g$ is a closed curve. Our strategy will be to interpret the expression $\mu_2^{\Tw}(F_\g,F_\g)$ and show that
\[
\mu_2^{\Tw}(F_\g,F_\g)=0.
\]
Assuming such a formalism can be established, the associativity relations for twisted complexes would show that
\begin{equation}
\cA_\g\circ \cA_\g=\mu_2^{\Tw}(F_\g, \mu_2^{\Tw}(F_\g,-))=\d_{\Mor}(\mu_3^{\Tw}(F_\g,F_\g,-)). \label{eq:null-homotopy-A-g^2}
\end{equation}

We  now describe how to interpret $\mu_2^{\Tw}(F_\g,F_\g)$ and $\mu_3^{\Tw}(F_\g,F_\g,-)$. To define these, it suffices to explain how to define the holomorphic polygon counting maps when there are two inputs which are $\g$. Let $\ds_1,\dots, \ds_n$ be attaching curves on $(\Sigma,\ws,\zs,\ps)$. We define 
\[
f_{\dt_1,\dots, \dt_i,\dt_i,\dots, \dt_j,\dt_j,\dots, \dt_n}(\Theta_{1,2},\dots, \g,\Theta_{i,i+1},\dots,\Theta_{j-1,j}, \g,\dots, \Theta_{n-1,n})
\]
(where the two $\g$'s are non-adjacent inputs) to count holomorphic $n$-gons of index $3-n$ with a multiplicative weight of 
\[
(\#\d_{\dt_i}(\psi)\cap \g)\cdot (\# \d_{\dt_j}(\psi)\cap \g).
\]
We define
\[
f_{\dt_1,\dots, \dt_i,\dt_i,\dt_i,\dots, \dt_n}(\Theta_{1,2},\dots, \Theta_{i-1,i} ,\g,\g, \Theta_{i,i+1},\dots, \Theta_{n-1,n}).
\]
to count holomorphic triangles with weight $3-n$ and multiplicative weight
\[
\frac{ (\# \d_{\dt_i}(\psi)\cap \g+1)(\# \d_{\dt_i}(\psi)\cap \g)}{2}.
\]
The $A_\infty$-associativity relations are not hard to verify if we interpret $\mu_1(\g)=0$ and $\mu_2(\g,\g)=0$. This allows us to interpret all of the expressions in Equation~\eqref{eq:null-homotopy-A-g^2}. These numerical choices can be motivated by thinking of the homology action as counting curves with marked points along their alpha boundaries. See \cite{LipshitzSarkarSplit}*{Section~3}.

To establish Equation~\eqref{eq:null-homotopy-A-g^2}, using associativity for $\mu_n^{\Tw}$, it suffices to show that
\begin{equation}
\mu_2^{\Tw}(F_\g,F_\g)=0.\label{eq:FgFg=0} 
\end{equation}

To establish Equation~\eqref{eq:FgFg=0}, we argue as follows. To streamline the argument, we define the quantities
\[
a_i^\psi:=\# (\d_{\a_i}(\psi)\cap \g)\quad \text{and} \quad a_{i,i}^{\psi}:=\frac{a_i^\psi(a_i^{\psi}+1)}{2}.
\] 
The length 1 components of $\mu_2^{\Tw}(F_\g,F_\g)$ are equal to $\mu_2(\g,\g)$, which is 0 by definition. The higher length components count sums as in
\[
f_{a_{\veps_1},\dots, \a_{\veps_i},\a_{\veps_i},\dots, \a_{\veps_j},\a_{\veps_j},\dots, \a_{\veps_n}}(\Theta_{1,2},\dots, \g, \dots, \g, \dots, \Theta_{n-1,n}).
\]
Summing over all ways of inserting two $\g$ terms into the above expression, we see that a given class $\psi$ is counted with weight
\[
\sum_{i<j} a_i^{\psi}a_j^{\psi}+\sum_{i=1}^n a_{ii}^{\psi}.
\]
It is straightforward to see that the above quantity is equal, modulo 2, to
\[
\frac{ \left(\# \d_{\a}(\psi)\cap \g\right)\left( \# \d_{\a}(\psi)\cap \g+1\right)}{2}.
\]
By Equation~\eqref{eq:a-gamma=0-non-mixed}, $\# \d_{\a}(\psi)\cap \g\equiv 0\pmod 2$, so we conclude that $\mu_2^{\Tw}(F_\g,F_\g)=0$.

We now move on to the case that $\d \lambda=\{p_1,p_2\}$ and $p_1,p_2$ are two free basepoints. In this case, we will construct a formal endomorphism $h_{\lambda,\lambda}\colon \cL_{\a}\to \cL_{\a}$ such that
\begin{equation}
\mu_2^{\Tw}(F_\lambda,F_\lambda)+\mu_1^{\Tw}(h_{\lambda,\lambda})=U_1\cdot \bI. \label{eq:FlFl=U}
\end{equation}
When defining holomorphic polygon counts with two $\lambda$ terms as inputs, we use the same convention as above, except we use
\[
\mu_2(\lambda,\lambda)=U_j\cdot \bI
\]
and
\[
\mu_1(\lambda)=(U_i+U_j)\cdot \bI.
\]

We return to the general question of verifying the associativity conditions on a diagram $(\Sigma,\ds_1,\dots, \ds_n,\ws,\zs,\ps)$.  We remark that both our definition of $\mu_2(\lambda,\lambda)$ and the definition of our polygon maps when $\lambda$ appears twice are asymmetric between $U_i$ and $U_j$. Note also that the quantity $a_{ii}^{\psi}$ requires a choice of orientation on $\lambda$.  We orient $\lambda$ so that $\# (\d(\psi)\cap \lambda)=n_{p_1}(\psi)-n_{p_2}(\psi)$. This implies that if $\psi_1$ and $\psi_{2}$ are the nonnegative classes of index 2 $\ds_j$-boundary degenerations which cover $p_1$ and $p_2$, respectively, then
\begin{equation}
a_{j,j}^{\psi_1}\equiv 1 \pmod{2} \quad \text{and} \quad a_{j,j}^{\psi_2}\equiv 0 \pmod{2}. \label{eq:def-ajj-phi}
\end{equation}
Furthermore,  if $\psi$ is any other nonnegative class of index 2 boundary degeneration, then
\begin{equation}
a_{j,j}^{\psi}=0.
\label{eq:ajj-boundary-degeneration-other-classes}
\end{equation} 

With these conventions, we now claim that the $A_\infty$-associativity relations hold for polygons with two $\lambda$ inputs.  We consider the desired associativity relations applied to the tuple $(\Theta_{1,2},\dots,\Theta_{i-1,i}, \lambda,\dots, \lambda,\Theta_{j,j+1}, \dots,\Theta_{n-1,n})$. There are two cases to consider: $i<j$ and $i=j$. Suppose that $\psi=\psi_1*\psi_2$ is a class of $n$-gons.  We consider first the case that $i<j$. Here, the associativity relations follow from the fact that
\[
a_i^{\psi_1} a_j^{\psi_1}+a_i^{\psi_1} a_j^{\psi_2}+a_i^{\psi_2} a_j^{\psi_1}+a_i^{\psi_2}a_j^{\psi_2}\equiv a^{\psi_1*\psi_2}_i a^{\psi_1* \psi_2}_j\pmod 2.
\]
Since the right-hand side is independent of the decomposition of $\psi=\psi_1*\psi_2$, we may obtain the associativity relations in this case summing $a^{ \psi}_i a^{ \psi}_j\# \d \cM( \psi)$ (multiplied by the appropriate powers of $\scU_i$, $\scV_i$ and $U_j$) over $n$-gon classes $\psi$ of index $4-n$. If $n=2$, there are additional ends corresponding to boundary degenerations. Since we are considering the case that $i<j$, this corresponds to considering the associativity relations for $(\lambda,\Theta_{1,2},\lambda)$. In this case, boundary degenerations make trivial contribution to the associativity relations since $a^\psi_1\cdot a^\psi_2=0$ for any nonnegative, index 2 class of boundary degenerations $\psi$ (since $\psi$ will only have non-trivial boundary on one of the two sets of attaching curves).

We now consider the associativity relation in the case that $i=j$. We observe the relation
\[
a_{i,i}^{\psi_1}+a_{i,i}^{\psi_2}+a_i^{\psi_1} a_i^{\psi_2}\equiv a_{i,i}^{\psi_1*\psi_2}\pmod 2.
\]
When $n>2$, the associativity relations follow from the above relation by summing over the ends of moduli spaces of $n$-gons with Maslov index $4-n$, with a weight of $a_{i,i}^{\psi}$.  When $n=2$, there are additional ends corresponding to boundary degenerations. These cases are relevant to the tuples $(\lambda,\lambda,\Theta_{1,2})$ and $(\Theta_{1,2},\lambda,\lambda)$. Focus on $(\lambda,\lambda,\Theta_{1,2})$, since the other configuration is handled by a similar argument. The values of $a_{i,i}^\psi$ on index 2 boundary degenerations is described in Equations~\eqref{eq:ajj-boundary-degeneration-other-classes} and ~\eqref{eq:def-ajj-phi}. When weighted by $a_{1,1}^{\psi}$, the total weighted count of boundary degenerations is $U_1$, and this is contributed by the $\ds_1$-degeneration covering $p_1$. Similarly when weighted by $a_{2,2}^{\psi}$, the total count is also $U_1$, and this is contributed by the $\ds_2$-degeneration covering $p_1$. Both of these are encoded by the relation $\mu_2(\lambda,\lambda)=U_1\cdot \bI$. This completes the proof of associativity.

We now construct $h_{\lambda,\lambda}$ by a filling procedure similar to our construction of the morphism $F_{\lambda}$. We assume that the morphisms $F_{\lambda}$ have already been constructed, as in Lemma~\ref{lem:d-mor-A-lambda}. We think of $h_{\lambda,\lambda}$ as being the diagonal morphism in a diagram of the following shape:
\[
\begin{tikzcd}[column sep=1cm, row sep=1cm, labels=description] \cL_{\a} \ar[dr, dashed, "h_{\lambda,\lambda}"] \ar[r, "F_\lambda"] & \cL_{\a} \ar[d, "F_{\lambda}"] \\
& \cL_{\a}
\end{tikzcd}.
\]
As such, we say that the \emph{length} of a component of $h_{\lambda,\lambda}$ from $\as_{\veps}$ to $\as_{\veps'}$ is $|\veps'-\veps|_{L^1}+2$.  

We now construct the components of $h_{\lambda,\lambda}$ by induction on their length. We define the length 2 components of $h_{\lambda,\lambda}$ to vanish. We consider the problem of filling arrows of length $n>2$ of $h_{\lambda,\lambda}$. We illustrate the case that $n=3$ below:
\[
\begin{tikzcd}[labels=description,column sep=2cm, row sep=.2cm]
&
\as_1
	\ar[from=dl, "\lambda"]
	\ar[dr, "\lambda"]
	\ar[dd, "\Theta_{\a_2,\a_1}"]
	\ar[dddr,dashed, "\eta_{\a_2,\a_1}"]
 &\\
\as_1
	\ar[dd, "\Theta_{\a_2,\a_1}"]
	\ar[dr, "\eta_{\a_2,\a_1}",dashed]
	\ar[ddrr,dotted, "\omega_{\a_2,\a_1}",pos=.6]
&& \as_1
	\ar[dd, "\Theta_{\a_2,\a_1}"]
\\[2cm]
& \as_2
	\ar[from=dl, "\lambda"]
	\ar[dr, "\lambda"]
\\
\as_2&& \as_2
\end{tikzcd}
\]
We let $C_{\a_2,\a_1}$ denote the terms appearing in the length 3 relation, except for $\d \omega_{\a_2,\a_1}$ (which has yet to be defined). We make two claims:
\begin{enumerate}
\item $C_{\a_2,\a_1}$ is a cycle of bidegree $(1,1)$ less than the top degree of homology.
\item If $F_{\lambda}$ is constructed as in Lemma~\ref{lem:higher-length-chains-in-submodule-A-lambda} so that its components of length 2 or more lie in the submodule generated by the ideal $(U_1,U_2)$, then $C_{\a_2,\a_1}$ is also in the submodule generated by the ideal $(U_1,U_2)$.
\end{enumerate}
The claim that $C_{\a_2,\a_1}$ is a cycle is an easy consequence of the associativity relations, interpreted above, which we leave to the reader. 

The fact that $C_{\a_2,\a_1}$ may be assumed to be in the submodule spanned by $(U_1,U_2)$ may be seen as follows. By assumption each summand of $C_{\a_2,\a_1}$ which involves $\eta_{\a_2,\a_1}$ is in this submodule. The remaining terms involve counts of holomorphic disks of index 1 which have $\Theta_{\a_2,\a_1}$ as an input. There are multiple ways for such a holomorphic disk to contribute to $C_{\a_2,\a_1}$, depending on the locations of the two $\lambda$-inputs. We observe that the total count of a class $\phi \in\pi_2(\Theta_{\a_2,\a_1},\ys)$ to $C_{\a_2,\a_1}$ is weighted by
\[
a_1^\phi a_2^\phi+a_{1,1}^\phi+a_{2,2}^\phi.
\]
It is straightforward to see that the above quantity is congruent modulo 2 to
\[
\frac{(a_1^\phi+a_2^\phi)(a_1^\phi+a_2^\phi+1)}{2}.
\]
We observe that $a^\phi_1+a_2^\phi=\pm (n_{p_1}(\phi)-n_{p_2}(\phi))$, since $\d D(\phi)=\d_{\a_1}(\phi)+\d_{\a_2}(\phi)$. In particular, if a disk $\phi$ has $n_{p_1}(\phi)=n_{p_2}(\phi)=0$, then it is weighted by 0 in $C_{\a_2,\a_1}$. Hence $C_{\a_2,\a_1}$ is in the submodule spanned by $(U_1,U_2)$.

We have established that $C_{\a_2,\a_1}$ is a cycle in bigrading $(1,1)$ less than the canonical top degree generator, and furthermore is in the span of $U_1$ and $U_2$. We observe that the map $\Pi$ in Equation~\eqref{eq:pi-map-quotient-variables} for quotienting by $U_1$ and $U_2$ is a chain map which is injective on the subspace of homology which lies in the grading of the canonical generator and also in the subspace of grading $(1,1)$ less than the top degree of $\ve{\HF}^-(\as_2,\as_1)$. Hence, $C_{\a_2,\a_1}=\d \omega_{\a_2,\a_1}$ for some $\omega_{\a_2,\a_1}$ in the top degree.

We now argue that $\omega_{\a_2,\a_1}$ may be taken to be in the submodule $(U_1,U_2)$. The proof is essentially the same as Lemma~\ref{lem:higher-length-chains-in-submodule-A-lambda}. The image $\Pi(\omega_{\a_2,\a_1})$ is a cycle, and by possibly adding a top degree generator of $\ve{\HF}^-(\as_2,\as_1)$, we may assume that $[\Pi(\omega_{\a_2,\a_1})]=0$. From here, the argument of Lemma~\ref{lem:higher-length-chains-in-submodule-A-lambda} may be applied verbatim.

From here, we continue filling the higher length chains of $h_{\lambda,\lambda}$. The argument for the higher length chains is essentially identical to the argument for the length 3 chains, so we leave it to the reader.  This verifies Equation~\eqref{eq:FlFl=U}. The main claim follows from this and the associativity relations for morphisms of twisted complexes of attaching curves.
\end{proof}

\begin{rem} We now make an algebraic digression. We note that the map $\cA_{\lambda}$ is in general only defined up to chain homotopy, and is not a cycle unless we set $U_1=U_2$. Some care is required when composing maps which are not cycles. For example, if $f_1$, $f_2$, $g_1$, $g_2$ are in $\End(C)$, (but not necessarily chain maps) and $f_1\simeq f_2$ and $g_1\simeq g_2$, it is not necessarily the case that $f_1\circ g_1\simeq f_2\circ g_2$. For the above lemma to be meaningful, it must be the case that changing $\cA_{\lambda}$ by a chain homotopy also changes $\cA_{\lambda}^2$ by a chain homotopy. We observe the following: If $f_1\simeq f_2$ and $\d_{\Mor}(f_1)=\d_{\Mor}(f_2)$ is central in $\End(C)$, then each $f_i^2$ is a chain map and
\[
f_1^2\simeq f_2^2.
\]
Similarly if $f_1,f_2,g_1,g_2$ are maps so that $f_1\simeq f_2$, $g_1\simeq g_2$ and $\d_{\Mor}(f_i)$ and $\d_{\Mor}(g_i)$ are central in $\End(C)$, then  $[f_i,g_i]$ are chain maps ($i=1,2$) and
\[
[f_1,g_1]\simeq [f_2,g_2].
\]
To see that $f_1^2\simeq f_2^2$, we let $j$ be a chain homotopy between $f_1$ and $f_2$, i.e. $\d_{\Mor}(j)=f_1+f_2$. Then
\[
\d_{\Mor}(f_1\circ j+j\circ f_2)=[ \d_{\Mor}(f_1),j]+f_1^2+f_2^2=f_1^2+f_2^2.
\]
(Note here that $\d_{\Mor}(f_1)=\d_{\Mor}(f_2)$ since $f_1\simeq f_2$). Similarly, if $\d_{\Mor}(j)=f_1+f_2$ and $\d_{\Mor}(h)=g_1+g_2$, then
\[
\begin{split}
\d_{\Mor}([j,g_1]+[f_2,h])=&[f_1+f_2,g_1]+[j,\d_{\Mor}(g_1)]+[f_2,g_1+g_2]+[\d_{\Mor}(f_2), h]\\
=&[f_1,g_1]+[f_2,g_2].
\end{split}
\]
\end{rem}

\begin{lem}
Suppose that $\cL_{\a}$ and $\cL_{\b}$ are hypercubes of handleslide equivalent attaching curves on $(\Sigma,\ws,\zs,\ps)$, where $\ps$ consists of free basepoints. Suppose that $\lambda_1$ and $\lambda_2$ are closed curves or arcs on $\Sigma$ with boundary in $\ps$. Then
\[
[\cA_{\lambda_1},\cA_{\lambda_2}]\simeq \sum_{p_i\in \d \lambda_1\cap \d\lambda_2} U_i.
\]
\end{lem}
\begin{proof} The proof is similar to the proof of Lemma ~\ref{lem:homology-action-squares}. We are going to construct a formal endomorphism
\[
h_{[\lambda_1,\lambda_2]}\colon \cL_{\a}\to \cL_\a
\]
which satisfies
\begin{equation}
\mu_1^{\Tw}(h_{[\lambda_1,\lambda_2]})=\mu_2^{\Tw}(F_{\lambda_1},F_{\lambda_2})+\mu_2^{\Tw}(F_{\lambda_2},F_{\lambda_1}).
\label{eq:[l_1,l_2]-hypercube-formal}
\end{equation}
 We think of constructing $h_{[\lambda_1,\lambda_2]}$ as corresponding to building a diagram of the form
 \begin{equation}
\cL_\a^{[\lambda_1,\lambda_2]}:=
\begin{tikzcd}[labels=description, column sep=1.2cm, row sep=1.2cm]
\cL_{\a}\ar[r, "\lambda_1"]\ar[d, "\lambda_2"]\ar[dr,dashed, "h_{[\lambda_1,\lambda_2]}"] &\cL_{\a}\ar[d,"\lambda_2"]\\
\cL_{\a}\ar[r, "\lambda_1"]& \cL_{\a}
\end{tikzcd}
\label{eq:hypercubeL-a[lambda-1,lambda-2]}
\end{equation}
which satisfies a suitable version of the hypercube relations.

Similar to our proof of Lemma~\ref{lem:homology-action-squares}, we need to describe how to interpret the holomorphic polygon counts where we formally have both $\lambda_1$ and $\lambda_2$ as an input.

We have already defined the holomorphic polygon maps when there is a single input of $\lambda_1$ or $\lambda_2$. When there are two inputs, we make the following definitions. Suppose $\ds_1,\dots, \ds_n$ is a collection of attaching curves on $(\Sigma,\ws,\zs,\ps)$.  If $j<k$, $s,t\in \{1,2\}$, and $s\neq t$, we define 
\[
f_{\dt_1,\dots, \dt_j,\dt_j,\dots, \dt_k, \dt_k,\dots, \dt_n}(\Theta_{1,2},\dots, \lambda_s,\dots, \lambda_t,\dots, \Theta_{n-1,n})
\]
 to count polygons weighted by 
 $\#(\d_{\dt_j}(\psi)\cap \lambda_s)\#(\d_{\dt_k}(\psi)\cap \lambda_t)$. Next,  we define 
 \[
 f_{\dt_1,\dots, \dt_j, \dt_j, \dt_j,\dots, \g_n}(\Theta_{1,2},\dots, \lambda_2,\lambda_1,\dots, \Theta_{n-1,n})
 \]
  to also count curves weighted by
  $\#(\d_{\dt_j}(\psi)\cap \lambda_2)\#(\d_{\dt_j}(\psi)\cap \lambda_1)$. Finally, we define 
  \[ f_{\dt_1,\dots, \dt_j, \dt_j, \dt_j,\dots, \dt_n}(\Theta_{1,2},\dots, \lambda_1,\lambda_2,\dots, \Theta_{n-1,n})=0
  \]

We claim that these satisfy a slightly restricted version of the $A_\infty$-associativity relation. Instead of showing an $A_\infty$-associativity relation on all tuples, we will show that associativity is satisfied on only some generators. We will show that associativity holds on tuples of the form
\begin{equation}
(\Theta_{1,2},\dots,\Theta_{j-1,j}, \lambda_t,\dots, \lambda_s,\Theta_{k,k+1},\dots, \Theta_{n-1,n})
\label{eq:associativity-cubified-1}
\end{equation}
whenever $j<k$ and $t\neq s$. Additionally, we will show that the associativity relations also hold on
\begin{equation}
\begin{split}
&(\Theta_{1,2},\dots,\Theta_{j-1,j},\lambda_1, \lambda_2,\Theta_{j,j+1},\dots, \Theta_{n-1,n})\\
+&(\Theta_{1,2},\dots,\Theta_{j-1,j}, \lambda_2, \lambda_1,\Theta_{j,j+1},\dots, \Theta_{n-1,n})
\end{split}
\label{eq:associativity-cubified-2}
\end{equation}
as long as we make the following conventions: We assume the polygon maps are strictly unital with the $\lambda$-inputs; $\d \lambda_i=0$; and $[\lambda_1,\lambda_2]=\sum_{p_i\in \d \lambda_1\cap \d\lambda_2} U_i$. This weaker notion of associativity is sufficient for the purposes of constructing $h_{[\lambda_1,\lambda_2]}$ and proving Equation~\eqref{eq:[l_1,l_2]-hypercube-formal}. Note that instead of the above approach, we could have instead attempt to prove associativity for general tuples, but this would require defining the polygon maps when there are inputs such as $\lambda_1\cdot\lambda_2$, which is an unnecessary detour.

The proof of associativity, as described above, is broken into several cases. The first case is when $n>2$ and $j<k$. In this case, the generic degenerations of a 1-dimensional family of holomorphic $n$-gons will consist of pairs of an $p$-gon and a $q$-gon for $p+q=n+2$. Write $\phi$ and $\psi$ for the classes of the $p$ and $q$-gons respectively. If $s\in \{1,2\}$, write $a_{j}^\phi$ for $\#(\d_{\dt_j}(\phi)\cap \lambda_1)$ and $b_j^\phi$ for $\#(\d_{\dt_j}(\phi)\cap \lambda_2)$. In this case, the associativity relation applied to tuples as in Equation~\eqref{eq:associativity-cubified-1} follows from the additive relation
\[
a_{j}^\phi b_{k}^\psi+a_{j}^\psi b_{k}^\phi+a_{j}^\phi b_{k}^\phi+a_{j}^\psi b_{k}^\psi=a_{j}^{\phi*\psi} b_{k}^{\phi*\psi}
\]
Associativity for tuples as in Equation~\eqref{eq:associativity-cubified-2}, when $j=k$ and $n>2$, follows from the similar additive relation
\[
a_{j}^\phi b_{j}^\psi+a_{j}^\psi b_{j}^\phi+a_{j}^\phi b_{j}^\phi+a_{j}^\psi b_{j}^\psi=a_{j}^{\phi*\psi} b_{j}^{\phi*\psi}.
\]

Finally, in the case that $n=2$, there are additional ends corresponding to boundary degenerations. These ends contribute in the associativity relations for $(\Theta_{1,2},\lambda_1,\lambda_2)+(\Theta_{1,2},\lambda_2,\lambda_1)$. The relation $[\lambda_1,\lambda_2]=\sum_{p\in \d \lambda_1\cap \d\lambda_2} U_p$ correspond to the fact that $a_2^\phi b_2^\phi\equiv 1$ if $\phi$ is a nonnegative, index 2 class of $\ds_2$-boundary degenerations which covers a point of $\d\lambda_1\cap \d\lambda_2$. The associativity relations for $(\lambda_1,\lambda_2, \Theta_{1,2})+(\lambda_2,\lambda_1, \Theta_{1,2})$ are similar. Finally, we consider the relations for $(\lambda_1,\Theta_{1,2},\lambda_2)$. In this case, boundary degenerations do not make any contribution, since  $a_1^\phi b_2^\phi\equiv 0$ for any boundary degeneration because a boundary degeneration has boundary on only one set of attaching curves. 

Having described the necessary associativity relations, filling the higher length chains of $h_{[\lambda_1,\lambda_2]}$ and verifying Equation~\eqref{eq:[l_1,l_2]-hypercube-formal} follows the same line of reasoning as in the proof of Lemma~\ref{lem:homology-action-squares}. We leave the remaining details of the construction to the reader.
\end{proof}

We prove a final lemma:

\begin{lem}\label{lem:additivity-A-lambda} Suppose that $\lambda_1$ and $\lambda_2$ are two 1-chains on $(\Sigma,\ws,\zs,\ps)$, such that $\d \lambda_1=\{x_1,x_2\}$ and $\d \lambda_2=\{x_2,x_3\}$. Then
\[
\cA_{\lambda_1}+\cA_{\lambda_2}\simeq \cA_{\lambda_1*\lambda_2}.
\]
\end{lem}
\begin{proof}We observe first that for any class of curves $\phi$ which has boundary on $\as_i$, we have
 \[
\# (\d_{\a_i}(\phi)\cap (\lambda_1+\lambda_2))= \# (\d_{\a_i}(\phi)\cap \lambda_1)+\# (\d_{\a_i}(\phi)\cap \lambda_2).
\]
In particular, we may take the formal endomorphism $F_{\lambda_2 *\lambda_1}$ to be $F_{\lambda_2}+F_{\lambda_1}$. With this definition, $\cA_{\lambda_1*\lambda_2}=\cA_{\lambda_1}+\cA_{\lambda_2}$. 
Since $F_{\lambda_1*\lambda_2}$ is well-defined up to chain homotopy by Remark~\ref{rem:well-defined-F_lambda}, the claim in the statement follows for an arbitrary construction of $F_{\lambda_2*\lambda_2}$. 
\end{proof}

We now prove a helpful formula for understanding basepoint moving maps on the link surgery hypercubes. See \cite{ZemGraphTQFT}*{Theorem~14.11} for an analogous result in the setting of the ordinary Heegaard Floer complexes.

\begin{lem} Let $\scH=(\Sigma,\cL_{\a},\cL_{\b},\ws,\zs,\ps)$ where $\cL_{\a}$ and $\cL_{\b}$ are hypercubes of handleslide equivalent attaching curves, $\ws\cup \zs$ are link basepoints and $\ps$ are free basepoints. Suppose that $p_1\in \ps$, and that $p_2\in \Sigma\setminus \ps$ is in the complement of the curves in $\cL_{\a}$ and $\cL_{\b}$. Suppose also that $(\Sigma,\cL_{\a},\cL_{\b},\ws,\zs,\ps)$ and $(\Sigma,\cL_{\a},\cL_{\b},\{p_2\}\cup \ps\setminus \{p_1\})$ are both weakly admissible. Let $\scH_{p_i}^+$ denote the diagram obtained by free-stabilizing $\scH$ at $p_i$.  If $\lambda$ is a path from $p_1$ to $p_2$, then
\[
\cS_{p_1}^- \Psi_{\scH_{p_2}^+\to\scH_{p_1}^+} \cA_\lambda \cS_{p_2}^+
\]
is chain homotopic to the diffeomorphism map which moves $p_1$ to $p_2$ along $\lambda$. In the above, we work on the complexes where we have identified $U_{p_1}$ and $U_{p_2}$ so that $\cA_{\lambda}$ is a chain map. Also $\Psi_{\scH_{p_2}^+\to \scH_{p_1}^+}$ denotes the map defined in Section~\ref{sec:basepoint-moving-maps-subsec}. 
\end{lem}
\begin{proof} We give a slightly simpler proof than appeared in the setting of 3-manifold invariants in \cite{ZemGraphTQFT}*{Theorem~14.11}, as follows.    Let $q$ be a point very near to $p_2$. We will view $\lambda$ as the concatenation of a path $\lambda_1$ from $p_1$ to $q$, and a path $\lambda_2$ from $q$ to $p_2$. In the following, we drop the transition maps $\Psi_{\scH_{p_2}^+\to \scH_{p_1}^+}$ from the notation.  We will write $U$ for  
\[
U=U_{p_1}=U_{p_2}.
\]

Let $\phi_*$ denote the point pushing diffeomorphism map which moves $p_2$ to $p_1$ along $\lambda$, composed with the naturality map. It is sufficient to show that
\begin{equation}
\id\simeq \phi_* \cS_{p_1}^- \cA_{\lambda} \cS_{p_2}^+.
\label{eq:basepoint-moving-map-rephrased}
\end{equation}

As a first step, we claim that map $\cS_{p_1}^- \cA_{\lambda} \cS_{p_2}^+$ is functorial under concatenation of arcs, i.e. that
\begin{equation}
\cS_{p_1}^- \cA_{\lambda} \cS_{p_2}^+\simeq \cS_{q}^- \cA_{\lambda_2} \cS_{p_2}^+\cS_{p_1}^- \cA_{\lambda_1} \cS_{q}^+.\label{eq:functoriality-graph-action}
\end{equation}
To establish Equation~\eqref{eq:functoriality-graph-action}, first note that we may commute free-stabilization maps with other terms by Lemmas~\ref{lem:commute-S-A-lam} and~\ref{lem:commute-free-stabilizations} to obtain
\[
\begin{split}
&\cS_{q}^- \cA_{\lambda_2} \cS_{p_2}^+\cS_{p_1}^- \cA_{\lambda_1} \cS^+_{q}\\
\simeq & \cS_{p_1}^-\cS_{q}^- \cA_{\lambda_2} \cA_{\lambda_1} \cS_q^+\cS_{p_2}^+ 
\end{split}
\]

Next, we use the relations   $\cA_{\lambda_i}^2\simeq U$ and $\cA_{\lambda_1+\lambda_2}=\cA_{\lambda_1}+\cA_{\lambda_2}$  from Lemmas~\ref{lem:homology-action-squares} and~\ref{lem:additivity-A-lambda} to see that
\[
 \cS_{p_1}^- \cS_{q}^-\cA_{\lambda_2} \cA_{\lambda_1} \cS_q^+\cS_{p_2}^+ \simeq \cS_{p_1}^- \cS_{q}^- (\cA_{\lambda_2+\lambda_1} \cA_{\lambda_1}+U) \cS_q^+\cS_{p_2}^+.
\]
We use the relation $\cS_{q}^-\cS_q^+=0$ (which is immediate from the formulas for $\cS_q^+$ and $\cS_q^-$) and the fact that $\lambda_2+\lambda_1=\lambda$  to see that the above is homotopic to
\[
 \cS_{p_1}^-\cS_{q}^- \cA_{\lambda} \cA_{\lambda_1} \cS_q^+\cS_{p_2}^+.
\]
Next, we use the fact that $\lambda$ is disjoint from $q$ to commute $\cS_q^{-}$ with $\cA_{\lambda}$ using Lemma~\ref{lem:commute-S-A-lam}, and then we use the relation $\cS_q^- \cA_{\lambda_1} \cS_q^+\simeq \id$ from Lemma~\ref{lem:basic-graph-TQFT-relations} to see that the above equation is homotopic to
\[
\cS_{p_1}^- \cA_{\lambda} \cS_{p_2}^+.
\]
This establishes Equation~\eqref{eq:functoriality-graph-action}.

To obtain the main statement in Equation~\eqref{eq:basepoint-moving-map-rephrased}, we combine Equation~\eqref{eq:functoriality-graph-action} with the naturality of all of the other maps with respect to diffeomorphisms, via the following manipulation:
\[
\begin{split}
\phi_*  \cS_{p_1}^- \cA_{\lambda} \cS_{p_2}^+&\simeq \phi_* \cS_{q}^- \cA_{\lambda_2} \cS_{p_2}^+ \cS_{p_1}^- \cA_{\lambda_1} \cS_{q}^+
\\
&\simeq \cS_{q}^- \cA_{\phi(\lambda_2)} \phi_*\cS_{p_2}^+\cS_{p_1}^- \cA_{\lambda_1} \cS_{q}^+\\
&\simeq \cS_{q}^- \cA_{\phi(\lambda_2)} \cS_{p_1}^+ \cS_{p_1}^- \cA_{\lambda_1} \cS_{q}^+.
\end{split}
\]
We may easily arrange that $\phi(\lambda_2)$ and $\lambda_1$ are isotopic on $\Sigma$  (rel. boundary). Hence $\cA_{\phi(\lambda_2)}\simeq \cA_{\lambda_1}$ by Lemma~\ref{lem:homology-action-homotopy-hypercube}. The expression is easily shown to be homotopic to the identity using the graph TQFT relations, as we briefly recall. We obtain from Lemma~\ref{lem:Phi-SS} that
\[
\cS_{q}^- \cA_{\lambda_1} \cS_{p_1}^+ \cS_{p_1}^- \cA_{\lambda_1} \cS_{q}^+\simeq \cS_{q}^- \cA_{\lambda_1} \Phi_{p_1} \cA_{\lambda_1} \cS_{q}^+.
\]
The above expression is homotopic to $\cS_{q}^- \cA_{\lambda_1}  \cA_{\lambda_1} \Phi_{p_1} \cS_{q}^++ \cS_{q}^- \cA_{\lambda_1}  \cS_{q}^+$ by Lemma~\ref{lem:basic-graph-TQFT-relations}. By Lemma~\ref{lem:homology-action-squares} we see that $\cS_q^- \cA_{\lambda_1}^2 \Phi_{p_1} \cS_q^+\simeq U \cS_q^- \cS_q^+ \Phi_{p_1}=0$ and by Lemma~\ref{lem:basic-graph-TQFT-relations} we have that $\cS_{q}^- \cA_{\lambda_1} \cS_q^+\simeq \id$.  This completes the proof.
\end{proof}

\begin{rem} The above argument also applies to the setting of ordinary Heegaard Floer complexes to give a shorter proof of the normalization axiom of the graph TQFT \cite{ZemGraphTQFT}*{Theorem~14.11}.
\end{rem}

We now prove the main basepoint moving map result:
\begin{proof}[Proof of Theorem~\ref{thm:basepoint-moving-hypercubes}]
The proof follows from the same manipulation as in the setting of the graph TQFT, see \cite{ZemGraphTQFT}*{Section~14.4}. We produce the argument for the reader. We decompose $\g$ into the concatenation of two arcs $\lambda_1,\lambda_2$, and view $\lambda_1$ as a path from $p$ to a new basepoint $q$, and $\lambda_2$ as a path from $q$ to $p$. In all of the Floer complexes we consider, we identify the variables for $p$ and $q$ with a single variable $U$. We compute 
\[
\begin{split}
\gamma_*&\simeq \cS_{q}^- \cA_{\lambda_2} \cS_p^+ \cS_p^- \cA_{\lambda_1} \cS_q^+\\
&\simeq \cS_{q}^- \cA_{\lambda_2} \Phi_p \cA_{\lambda_1} \cS_q^+\\
&\simeq \cS_{q}^- A_{\lambda_1} \cS_q^++\cS_q^-\Phi_p \cA_{\lambda_2} \cA_{\lambda_1} \cS_q^+\\
&\simeq \id+\Phi_p \cS_q^-(\cA_{\lambda_1}+\cA_\g) \cA_{\lambda_1} \cS_q^+\\
&\simeq \id+\Phi_p \cS_q^- U \cS_q^+ +\Phi_p A_\g \cS_q^-\cA_{\lambda_1} \cS_q^+\\
&\simeq \id+\Phi_p \cA_\g,
\end{split}
\]
completing the proof. 
\end{proof}

\subsection{Relating homology actions and basepoint actions}

Suppose that $(\Sigma,\as,\bs,\ws,\zs)$ is a Heegaard link diagram for $(Y,L)$ and that $K_i\subset L$ is a knot component which has basepoints $w_i\in \ws$ and $z_i\in \zs$. We may represent $K_i$ on $\Sigma$ as an immersed curve $[K_i]$ on $\Sigma$ which is the concatenation of two arcs $\lambda_{\a}$ and $\lambda_{\b}$ which both have boundary $\{w_i,z_i\}$. The arc $\lambda_{\a}$ intersects only alpha curves while $\lambda_{\b}$ intersects only beta curves. See Figure~\ref{fig:44}. If $\phi$ is a class of disks, then 
\begin{equation}
a(\phi,\g)=n_{w_i}(\phi)-n_{z_i}(\phi)\pmod 2.
\label{eq:a-phi-g=nw-nz}
\end{equation}
The map $A_{[K_i]}$ weights a holomorphic disk $\phi$ by $a(\phi,\g)$. On the other hand, we can define a map which counts holomorphic disks weighted by $n_{w_i}(\phi)$. This map is clearly equal to $\scU_i \Phi_{w_i}$. Similarly the map which counts disks weighted by $n_{z_i}(\phi)$ is equal to $\scV_i \Psi_{z_i}$. Therefore Equation~\eqref{eq:a-phi-g=nw-nz} translates into the equality
\[
A_{[K_i]}=\scU_{i} \Phi_{w_i}+\scV_{i}\Psi_{z_i},
\] 
as endomorphisms of $\cCFL(Y,L)$. Compare \cite{ZemCFLTQFT}*{Lemma~14.12}. In this section, we prove a hypercube version of the above statement.

\begin{figure}[h]
\begingroup%
  \makeatletter%
  \providecommand\color[2][]{%
    \errmessage{(Inkscape) Color is used for the text in Inkscape, but the package 'color.sty' is not loaded}%
    \renewcommand\color[2][]{}%
  }%
  \providecommand\transparent[1]{%
    \errmessage{(Inkscape) Transparency is used (non-zero) for the text in Inkscape, but the package 'transparent.sty' is not loaded}%
    \renewcommand\transparent[1]{}%
  }%
  \providecommand\rotatebox[2]{#2}%
  \newcommand*\fsize{\dimexpr\f@size pt\relax}%
  \newcommand*\lineheight[1]{\fontsize{\fsize}{#1\fsize}\selectfont}%
  \ifx\svgwidth\undefined%
    \setlength{\unitlength}{150.05734661bp}%
    \ifx\svgscale\undefined%
      \relax%
    \else%
      \setlength{\unitlength}{\unitlength * \real{\svgscale}}%
    \fi%
  \else%
    \setlength{\unitlength}{\svgwidth}%
  \fi%
  \global\let\svgwidth\undefined%
  \global\let\svgscale\undefined%
  \makeatother%
  \begin{picture}(1,0.75232704)%
    \lineheight{1}%
    \setlength\tabcolsep{0pt}%
    \put(0.12187336,0.29425106){\color[rgb]{0,0,0}\makebox(0,0)[t]{\lineheight{1.25}\smash{\begin{tabular}[t]{c}$w_i$\end{tabular}}}}%
    \put(0.86151893,0.45814998){\color[rgb]{0,0,0}\makebox(0,0)[t]{\lineheight{1.25}\smash{\begin{tabular}[t]{c}$z_i$\end{tabular}}}}%
    \put(0,0){\includegraphics[width=\unitlength,page=1]{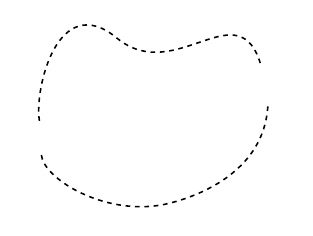}}%
    \put(0.82357124,0.21440287){\color[rgb]{0,0,0}\makebox(0,0)[lt]{\lineheight{1.25}\smash{\begin{tabular}[t]{l}$\lambda_\a$\end{tabular}}}}%
    \put(0.04347815,0.53589639){\color[rgb]{0,0,0}\makebox(0,0)[lt]{\lineheight{1.25}\smash{\begin{tabular}[t]{l}$\lambda_\b$\end{tabular}}}}%
    \put(0,0){\includegraphics[width=\unitlength,page=2]{fig44.pdf}}%
    \put(0.36562022,0.69139009){\color[rgb]{0,0,1}\makebox(0,0)[lt]{\lineheight{1.25}\smash{\begin{tabular}[t]{l}$\bs$\end{tabular}}}}%
    \put(0.47331432,0.01478265){\color[rgb]{1,0,0}\makebox(0,0)[lt]{\lineheight{1.25}\smash{\begin{tabular}[t]{l}$\as$\end{tabular}}}}%
  \end{picture}%
\endgroup%

\caption{A knot shadow $K_i$ (dashed) decomposed as a concatenation of two subarcs $\lambda_{\a}$ and $\lambda_{\b}$.}
\label{fig:44}
\end{figure}

We consider hypercubes of handleslide equivalent attaching curves $\cL_{\a}$ and $\cL_{\b}$ on $(\Sigma,\ws,\zs)$. We suppose that there is a pair of arcs $\lambda_{\a},\lambda_{\b}\subset \Sigma$ such that
\[
\d \lambda_{\a}=\d \lambda_{\b} =\{w_i,z_i\}, \qquad K_i=\lambda_\a*\lambda_\b,
\]
and such that $\lambda_{\a}$ intersects only curves from $\cL_{\a}$ while $\lambda_{\b}$ intersects only curves from $\cL_{\b}$. Write $K_i$ for the concatenation of $\lambda_{\a}$ and $\lambda_{\b}$. We recall from Section~\ref{sec:homology-actions-hypercubes} that there is an endomorphism
\[
\cA_{[K_i]}\colon \ve{\CF}^-(\cL_{\a},\cL_{\b})\to \ve{\CF}^-(\cL_{\a},\cL_{\b}).
\]
Since $K_i$ is a closed curve on $\Sigma$, the map $\cA_{[K_i]}$ counts holomorphic polygons (as would be counted by the ordinary hypercube differential) with a weight of $a(\phi,K_i)$, where $a(\phi,K_i)$ is the sum over $\#(\d_{\a_\veps} \phi \cap K)$ ranging over all $\as_{\veps}$ in $\cL_{\a}$. 

In Section~\ref{sec:homology-actions-hypercubes} we constructed basepoint actions $\Phi_{w_i}$ and $\Psi_{z_i}$ on the hypercube $\ve{\CF}^-(\cL_{\a},\cL_{\b})$. The map $\Phi_{w_i}$ was defined as the formal differential of the hypercube differential with respect to $\scU_i$. We note that in general powers of $\scU_i$ are contributed by both the holomorphic curves in the hypercube differential, and also the powers of $\scU_i$ present in the input chains of the polygon maps used to construct the hypercube. We defined a map $\Phi_{w_i}^0$ only took into account the multiplicities of holomorphic curves at the basepoint $w_i$. We proved in Lemma~\ref{lem:algebraically-rigid-Phiw0} that if $\cL_{\a}$ and $\cL_{\b}$ are algebraically rigid, then $\Phi^0_{w_i}=\Phi_{w_i}$ and $\Psi^0_{w_i}=\Psi_{w_i}$. In general, this is not the case. Nonetheless, we do have the following:

\begin{lem}
\label{lem:homology-action-versus-basepoint-action}
 Suppose that $\cL_{\a}$ and $\cL_{\b}$ are hypercubes of handleslide equivalent attaching curves on $(\Sigma,\ws,\zs)$, and $K_i\subset \Sigma$ is a concatenation of two arcs $\lambda_{\a}$ and $\lambda_{\b}$, as above, where $\lambda_\a$ intersects only the curves in $\cL_{\a}$, and $\cL_{\b}$ intersects only the curves in $\cL_{\b}$. Then $\scU_i \Phi^0_{w_i}+\scV_i \Phi_{z_i}^0$ is a chain map (regardless of whether $\cL_{\a}$ and $\cL_{\b}$ are algebraically rigid) and
\[
\cA_{[K_i]}= \scU_i \Phi_{w_i}^0+\scV_i \Psi_{z_i}^0=\scU_i \Phi_{w_i}+\scV_i\Psi_{z_i},
\]
as hypercube morphisms.
\end{lem}
\begin{proof} Since $a(\phi,K_i)\equiv n_{w_i}(\phi)-n_{z_i}(\phi)\pmod 2$ for any holomorphic polygon counted by the hypercube differential of $\ve{\CF}^-(\cL_{\a},\cL_{\b})$ we conclude that
\[
\cA_{[K_i]}=\scU_{i} \Phi_{w_i}^0+\scV_{i}\Psi_{z_i}^0,
\]
since the two sides count the same holomorphic curves, with the same weights. Since the left-hand side is a chain map by Lemma~\ref{lem:homology-action-chain-homotopy}, the right hand side must also be a chain map.

It remains to show that
\[
\scU_i \Phi_{w_i}^0+\scV_{i} \Psi_{z_i}^0= \scU_i \Phi_{w_i}+\scV_i\Psi_{z_i}.
\]
To see this, we observe that $\scU_i\Phi_{w_i}+\scV_i \Psi_{z_i}$ may be described as summing over sequences $\veps_1<\cdots<\veps_n$ and $\nu_1<\cdots<\nu_m$  the maps 
\begin{equation}
\scU_i\cdot (\phi_{w_i})^{\nu_1<\dots< \nu_m}_{\veps_1<\dots< \veps_n}+\scV_i\cdot  (\psi_{z_i})^{\nu_1<\dots< \nu_m}_{\veps_1<\dots< \veps_n}\label{eq:maps-phi_w-psi_z-rigid}
\end{equation}
(which count holomorphic curves weighted by their multiplicities over $w_i$ and $z_i$, respectively, as in Equation~\eqref{eq:phiw-veps-nu-def}) and also summing expressions of the form
\begin{equation}
\xs\mapsto f_{\a_{\veps_n},\dots, \a_{\nu_1},\b_{\nu_1},\dots, \b_{\nu_m}}(\Theta^{\a}_{n,n-1},\dots, (\Theta^\a_{j+1,j})',\dots,\Theta_{2,1}^\a ,\xs, \Theta^\b_{1,2},\dots, \Theta^\b_{m-1,m}),
\label{eq:non-rigid-terms-phiw-psiz}
\end{equation}
as well as terms where the primed term is a beta-chain. In the above, we write $\Theta_{j+1,j}^{\a}$ for $\Theta_{\a_{\veps_{j+1}},\a_{\veps_j}}$ and $(\Theta^\a_{j+1,j})'$ denotes $(\scU_i \d_{\scU_i}+\scV_i\d_{\scV_i})(\Theta^\a_{j+1,j})$, and similarly for the beta-terms.
That is, we sum over the holomorphic polygon maps where we apply $(\scU_i \d_{\scU_i}+\scV_i\d_{\scV_i})$ to exactly one of the inputs from $\cL_{\a}$ or $\cL_{\b}$.
The sum of maps in Equation~\eqref{eq:maps-phi_w-psi_z-rigid} is exactly the map $\scU_i\Phi_{w_i}^0+\scV_i\Psi_{z_i}^0$, so it suffices to show that the terms in Equation~\eqref{eq:non-rigid-terms-phiw-psiz} vanish. To see this, we will show that if $\Theta$ is any chain in $\cL_{\a}$ or $\cL_{\b}$, then $(\scU_i\d_{\scU_i}+\scV_i \d_{\scV_i})(\Theta)=0$. We argue as follows, focusing on $\cL_{\a}$. We observe that on the diagram $(\Sigma,\cL_{\a},\ws,\zs)$, the basepoints $w_i$ and $z_i$ are immediately adjacent (by virtue of the existence of $\lambda_{\b}$, which is disjoint from the alpha curves). We note that each subdiagram $(\Sigma,\as',\as,\ws,\zs)$ represents an unlink $\bU$ in a connected sum of several copies of $S^1\times S^2$. There is a $|\bU|$-component Alexander grading on $\cCFL(\bU)$, and the top Maslov bidegree generators lie in Alexander grading $(0,\dots, 0)$. By construction, all of the chains of $\cL_{\a}$ also lie in Alexander grading $(0,\dots, 0)$. Write $A_i$ for the component of the Alexander grading corresponding to $K_i$. Since $\lambda_\b$ connects $w_i$ to $z_i$ and is disjoint from all curves of $\cL_{\a}$, if $\xs\in \bT_{\a'}\cap \bT_{\a}$, we have $A_i(\xs)=0$. In particular, any chain in $\cL_{\a}$ must have equal powers of $\scU_i$ and $\scV_i$, i.e., it must be of the form $\scU_i^N \scV_i^N\cdot \ys$, where $\ys$ has no powers of $\scU_i$ or $\scV_i$. Clearly
\[
(\scU_i \d_{\scU_i}+\scV_i \d_{\scV_i})(\scU_i^N \scV_i^N\cdot \ve{y})=0,
\]
as claimed. The same argument works for $\cL_{\b}$, so the proof is complete.
\end{proof}

\subsection{Properties of the $U$-action on link surgery complexes}

We now explore some of the complexities of the surgery theorem when we allow arbitrary arcs for our $\sigma$-basic systems.

\begin{lem}\label{lem:U-actions-homotopic} Suppose that $\scA$ is a system of arcs for $L\subset S^3$, and suppose that the arcs for knot components $K_1,K_2\subset L$ are alpha-parallel or beta-parallel (the case that one is alpha-parallel and the other is beta-parallel is allowed). Then the actions of $U_1$ and $U_2$ are homotopic on $\cX_{\Lambda}(L,\scA)^{\cL}$ in the sense that there is a type-$D$ module morphism $H^1\colon \cX_{\Lambda}(L,\scA)^{\cL}\to \cX_{\Lambda}(L,\scA)^{\cL}$ such that
\[
\d_{\Mor}(H^1)=\id\otimes (U_1+U_2).
\]
In the above, $U_i$ denotes $\scU_i\scV_i$ and $\id\otimes (U_1+U_2)$ means the map which sends $\xs$ to $\xs\otimes (U_1+U_2)$ in all idempotents.
\end{lem}
\begin{proof} Let $\scH$ be a $\sigma$-basic system of Heegaard diagrams for $(L,\scA)$. We may assume that $\scH$ is meridional for $K_1$ and $K_2$, in the sense of Section~\ref{def:algebraically-rigid}.  We begin with the case that $K_1$ and $K_2$ are both alpha-parallel or both beta-parallel. For concreteness, assume that $K_1$ and $K_2$ are both alpha-parallel, so that the arcs of $\scA$ are both disjoint from the alpha curves. We let $\lambda\colon [0,1]\to \Sigma$ denote a path from $z_1$ to $z_2$ which is disjoint from all of the arcs of $\scA$.

Given hypercubes of alpha and beta Lagrangians $\cL_{\a}$ and $\cL_{\b}$ on a pointed surface $(\Sigma,\ws,\zs)$, as well as an arc $\lambda$ connecting two points $x_1,x_2\in \Sigma$, we obtain a map
\begin{equation}
\cA_{\lambda}\colon \ve{\CF}^-(\cL_{\a},\cL_{\b})\to \ve{\CF}^-(\cL_{\a},\cL_{\b})
\label{eq:A-lambda-ref}
\end{equation}
as in Section~\ref{sec:homology-actions-hypercubes}, satisfying
\[
\d_{\Mor}(\cA_{\lambda})=E_1^\a+E_2^\a,
\]
where $E_i^\a$ denotes the product of the variables for the basepoints in the unique alpha degeneration containing $x_i$. Compare Lemma~\ref{lem:d-mor-A-lambda}. We view the link surgery hypercube as a 2-dimensional hypercube by internalizing all axis directions except those of $K_1$ and $K_2$. With this perspective, write
\[
\cC_{\Lambda}(L)=
\begin{tikzcd}[labels=description, row sep=2cm, column sep=3cm]
 \cC_{(0,0)}
  	\ar[r, "F^{K_1}+F^{-K_1}"] 
  	\ar[d, "F^{K_2}+F^{-K_2}"]
  	\ar[dr,dashed, "F^{-K_1,-K_2}"]
  	 & 
  	 \cC_{(1,0)}\ar[d, "F^{K_2}+F^{-K_2}"]
  	 \\
\cC_{(0,1)}\ar[r, "F^{K_1}+F^{-K_1}"]
& \cC_{(1,1)}
\end{tikzcd}
\]

Write $\phi_{1}$ for the point-pushing diffeomorphism which moves $z_1$ to $w_1$, and write $\phi_{2}$ for the map which moves $z_2$ to $w_2$. Write
\[
\lambda_{0,0}=\lambda, \quad \lambda_{1,0}=\phi_1(\lambda),\quad \lambda_{0,1}=\phi_2(\lambda),\quad \lambda_{1,1}=\phi_1(\phi_2(\lambda)).
\]

We now adapt the above construction of $\cA_{\lambda}$ to the setting of the surgery hypercube, which has maps corresponding to basepoint moving maps. We focus on a $\sigma$-basic system of Heegaard diagrams for $L$. We apply the aforementioned construction of the hypercube morphism $\cA_{\lambda}$ to the hypercube consisting of all of the link surgery edge maps with negatively oriented components. We obtain a diagram of the following form, which is a hypercube except for the length 1 relations along arrows labeled $\cA_{\lambda_{\veps}}$ (which instead satisfy $[\d, \cA_{\lambda_{\veps}}]=U_1+U_2$).
\begin{equation}
\begin{tikzcd}[
column sep={3.5cm,between origins},
row sep=.8cm,labels=description
]
\cC_{(0,0)}
	\ar[dd, "F^{-K_2}"]
	\ar[dr,  "\cA_{\lambda_{0,0}}"]
	\ar[rr, "F^{-K_1}"]
	\ar[ddrr,dashed, "F^{-K_1,-K_2}"]
	\ar[dddrrr,dotted,"\omega^{-K_1,-K_2}", sloped]
&&[-1.5cm]
\cC_{(1,0)}
	\ar[dd, "F^{-K_2}"]
	\ar[dr,"\cA_{\lambda_{1,0}}"]
	\ar[dddr,dashed, "j^{-K_2}"]
&
\\
&\cC_{(0,0)}
	\ar[rr,crossing over, "F^{-K_1}"]
&&
\cC_{(1,0)}
	\ar[dd, "F^{-K_2}"]
	\ar[from=ulll,dashed,crossing over, "h^{-K_1}"]
\\[1.8cm]
\cC_{(0,1)}
	\ar[rr, "F^{-K_1}"]
	\ar[dr,"\cA_{\lambda_{0,1}}"]
	\ar[drrr,dashed, "j^{-K_1}",pos=.6]
&&
\cC_{(1,1)}
	\ar[dr, "\cA_{\lambda_{1,1}}"]	
&
\\
&
\cC_{(0,1)}
	\ar[rr, "F^{-K_1}"]
	\ar[from =uu, crossing over, "F^{-K_2}"]
	\ar[from=uuul,dashed, crossing over, "h^{-K_2}"]
	&&
\cC_{(1,1)}
	\ar[from=uull,crossing over, dashed, "F^{-K_1,-K_2}"]
\end{tikzcd}
\label{eq:diagram-homology-action}
\end{equation}

We first claim that the maps labeled $\cA_{\lambda_{\veps}}$ in Equation~\eqref{eq:diagram-homology-action} are all equal. To see this, note that $\cA_{\lambda_{1,0}}$ is obtained by concatenating $\lambda$ with the arc from $z_1$ to $w_1$ which is disjoint from the alpha curves. Therefore $\#\d_{\a}(\psi) \cap \lambda=\# \d_{\a}(\psi)\cap \lambda_{1,0}$ for all classes of polygons $\psi$ used to define the complexes $\cC_{(1,0)}$. Similar arguments apply to the other length 1 maps $\cA_{\lambda_{\veps}}$ so we conclude that the length 1 maps satisfy
\begin{equation}
\cA_{\lambda_{\veps}}=\cA_{\lambda_{0,0}}
\label{eq:A-lambda-veps-equal}
\end{equation}
for all $\veps\in \bE_2$.

 We now consider the length 2 maps, we claim that 
 \begin{equation}
 h^{-K_i}=j^{-K_i} 	\quad \text{for} \quad i=\{1,2\}.
 \label{eq:maps-opposite-faces-equal}
 \end{equation}
  The argument is similar to the argument for the length 1 maps. Namely the maps $h^{-K_2}$ and $j^{-K_2}$ count the same holomorphic curves. Some of the curves contributing to $h^{-K_2}$ are counted with a weight of $a(\psi,\lambda)$ whereas the same curves are counted by $a(\psi,\phi_1(\lambda))$ for $j^{-K_2}$. Since $\phi_1(\lambda)$ is obtained by concatenating $\lambda$ with an arc which passes through only beta curves, we see that $a(\psi,\lambda)=a(\psi,\phi_1(\lambda))$, so these expressions are equal. We conclude that $h^{-K_i}=j^{-K_i}$.
 
 We now define an endomorphism $\cA^{\bX}_{\lambda}$ of the link surgery hypercube $\cC_{\Lambda}(L)$ to be the endomorphism induced by the components of the diagram in Figure~\ref{eq:diagram-homology-action} which come out of the page (i.e. the $\cA_{\lambda_\veps}$ direction). By construction, the maps labeled $h^{-K_i}$ have the same equivariance properties as the link surgery maps $F^{-K_i}$, so the morphism $\cA^{\bX}_{\lambda}$ is induced by a morphism of type-$D$ modules
 \[
\cA^{\cX,1}_{\lambda}\colon \cX_{\Lambda}(L,\scA)^{\cL}\to \cX_{\Lambda}(L,\scA)^{\cL}.
 \]
 (Note that outside of this lemma, we will write just $\cA_{\lambda}$ or $\cA_{\lambda}^1$ for the maps labeled $\cA_{\lambda}^{\bX}$ and $\cA_{\lambda}^{\cX,1}$ above).
 
 The components labeled $\cA_{\lambda_{\veps}}$ preserve the idempotents for $K_1$ and $K_2$. The arrows labeled $-K_1$ are weighted in $H^1$ by  multiples of  the algebra element $\tau_1$. The arrows labeled $-K_2$ are weighted by multiples of $\tau_2$. The dotted arrow labeled $\omega^{-K_1,K_2}$ is weighted by a multiple of $\tau_1\tau_2$. None of the components of $\cA_{\lambda}^{\cX,1}$ are weighted by algebra elements containing a factor of $\sigma_1$ or $\sigma_2$.
 
 As a morphism from the back fact of of Figure~\ref{eq:diagram-homology-action} to the front face, the map $\cA_{\lambda}^{\bX}$ satisfies
 \[
 \d_{\Mor}(\cA_{\lambda}^{\bX})=(U_1+U_2)\cdot \id
 \]
 by Remark~\ref{rem:remark-dmor-A-lambda}. We claim that the same equation is satisfied if we view $\cA_{\lambda}^{\bX}$ as a morphism from $\cC_{\Lambda}(L)$ to itself, and furthermore that the type-$D$ endomorphism $\cA_{\lambda}^{\cX,1}$ satisfies $\d_{\Mor}(\cA_{\lambda}^{\cX,1})=\id\otimes (U_1+U_2)$ as an endomorphism of $\cX_{\Lambda}(L)^{\cL}$. Note that the hypercube differential in the diagram in Figure~\ref{eq:diagram-homology-action} does not include the differentials $F^{K_1}$ or $F^{K_2}$. Since we are using $\sigma$-basic systems, these maps are identified with the maps induced by localization at $\scV_1$ and $\scV_2$. Once we include $F^{K_1}$ and $F^{K_2}$ into the differentials on the front and back faces, the fact that $\d_{\Mor}(\cA_{\lambda}^{\bX})$ is still $(U_1+U_2)\cdot \id$ follows immediately from Equations~\eqref{eq:A-lambda-veps-equal} and~\eqref{eq:maps-opposite-faces-equal}. The same argument works for type-$D$ morphisms. This completes the proof when both components are alpha-parallel.
 
 We now consider the case that $K_1$ is alpha-parallel and $K_2$ is beta-parallel. In this case, we must modify the map $\cA_{\lambda}^{\cX,1}$. Recall that $\cA_{\lambda}^{\cX,1}$ was built from the map $\cA_{\lambda}$. We instead replace this by a map $\cA_{\lambda}^{\cX,1}$ built from $\cA_{\lambda}+\scV_2 \Psi_{z_2}$.

 If we view the complex as a 1-dimensional cone in the $K_2$-direction
  \[
 \cC_{\Lambda}(L,\scA)= \Cone\left(\begin{tikzcd}[column sep=2cm]\cC_{0}\ar[r, "F^{K_2}+F^{ -K_2}"] &\cC_1,
  \end{tikzcd}\right)
  \]
  then we can define endomorphisms $\cA_{\lambda}+\scV_2\Psi_{z_2}\colon \cC_i \to \cC_i$ similar to the case when $K_1$ and $K_2$ were both alpha-parallel. The map $\cA_{\lambda}+\scV_2\Psi_{z_2}\colon \cC_0\to \cC_0$ consists of the components labeled $\cA_{\lambda_{0,0}}$, $\cA_{\lambda_{1,0}}$ and $h^{-K_1}$. As an endomorphism of $\cC_1$, the maps is similar. Since $K_1$ is alpha-parallel, we see that 
 \[
 \cA_{\lambda_{0,0}}=\cA_{\lambda_{1,0}}=\cA_{\lambda},
 \]
 and so the same argument as before shows that
 \[
 \d_{\Mor}(\cA_{\lambda}+\scV_2 \Psi_{z_2})=(U_1+U_2)\cdot \id
 \]
 as endomorphisms of $\cC_i$.
 
 On the other hand, the point pushing map for $K_2$ moves $z_2$ to $w_2$ through exactly the alpha circles which intersect the trace of $K_2$ on the knot diagram. Hence
 \[
 \cA_{\lambda_{0,1}}=\cA_{\lambda_{1,1}}=\cA_{\lambda}+\cA_{[K_2]}.
 \]
 Here, $\cA_{[K_2]}$ denotes the hypercube homology action for the trace of $K_2$ on the Heegaard diagram.
 We conclude that
 \begin{equation}
 F^{ -K_2} \cA_{\lambda}\simeq (\cA_{\lambda}+\cA_{[K_2]}) F^{-K_2}, \label{eq:A-lambda-commutator-descent-map}
 \end{equation}
 as maps from $\cC_0$ to $\cC_1$.  On the other hand, since $F^{-K_2}$ moves $z_2$ to $w_2$, we have
 \begin{equation}
 F^{-K_2}\scV_2\Psi_{z_2}\simeq \scU_2 \Phi_{w_2}F^{-K_2}.
 \label{eq:commutator-Psi-K-descent}
 \end{equation}
 We recall from Lemma~\ref{lem:homology-action-versus-basepoint-action} that $\cA_{[K_2]}\simeq \scU_2 \Phi_{w_2}+\scV_2 \Psi_{z_2}$, as endomorphisms of $\cC_0$. Combining the relations in Equations~\ref{eq:A-lambda-commutator-descent-map}, ~\ref{eq:commutator-Psi-K-descent} with this equation for $\cA_{[K_2]}$, we obtain
 \[
F^{-K_2} (\cA_{\lambda}+\scV_2\Psi_{z_2})\simeq(\cA_{\lambda}+\scV_2 \Psi_{z_2})F^{-K_2},
 \]
 as maps from $\cC_0$ to $\cC_{1}$. Note that since $F^{K_2}$ is the map for localizing at $\scV_2$, we also have
 \[
F^{-K_2} (\cA_{\lambda}+\scV_2\Psi_{z_2})=(\cA_{\lambda}+\scV_2 \Psi_{z_2})F^{-K_2}.
 \]
 Arranging the maps $\cA_{\lambda}+\scV_2 \Psi_{z_2}$ (as maps from $\cC_i$ to $\cC_i$) as well as the above homotopies (which map $\cC_0$ to $\cC_1$) into an endomorphism $H$ of $\cC_{\Lambda}(L)$, we see that $\d_{\Mor}(H)=(U_1+U_2)\cdot \id$. The same logic gives a type-$D$ homotopy, so the proof is complete.
 \end{proof}

\section{The alpha-beta transformer}
\label{sec:alpha-to-beta-transformer}

In this section, we describe a bimodule ${}_{\cK} \cT^{\cK}$, which we call the \emph{alpha-beta transformer bimodule}. We will prove the following:

\begin{thm}
\label{thm:transformer-bimodule} Suppose that $L\subset S^3$ is a framed link with a distinguished component $K$. Let $\scA$ be a system of arcs for $L$ such that the arc for $K$ is alpha-parallel. Let $\scA'$ be a system of arcs where $K$ is beta-parallel, but otherwise the arcs are the same as $\scA$. Then
\[
\cX_{\Lambda}(L,\scA)^{\cL}\simeq \cX_{\Lambda}(L,\scA')^{\cL}\hatbox {}_{\cK} \cT^{\cK}.
\]
The same formula holds if we switch the roles of $\scA$ and $\scA'$.
\end{thm}

\subsection{The bimodule}

\label{sec:transformer-def}

The bimodule $\cT$ has rank 1 over $\ve{I}$, and has $\delta_1^1=0$. The map $\delta_2^1$ is the unique $\ve{I}$-linear map which satisfies
\[
\delta_2^1(a,1)=1\otimes a,
\]
for $a\in \cK$.
(This coincides with the identity bimodule).

Additionally, there is a $\delta_4^1$ term, which is non-vanishing on tuples of the form $(a,b,f\tau ,i_0)$ where $f$, $a$ and $b$ are concentrated in a single idempotent, and which vanishes on other elementary tensors. On such elements, we have the formula
\[
\delta_4^1(a, b , f \tau , i_0)=i_1\otimes \big( \d_{\scU} (a)\cdot  \d_{\scV}(b)\cdot f\tau  \big).
\]

\begin{lem}
\label{lem:transformer-bimodule}
The structure maps on ${}_{\cK} \cT^{\cK}$ satisfy the $DA$-bimodule relations.
\end{lem}
\begin{proof} The only interesting structure relations to be verified are the ones with 5 inputs. There are only two non-vacuous relations to be verified, which are those for sequences of the form $(a,b,c,\tau,i_0)$ and those of the form $(a,b,\tau,c,i_0)$. 

The structure relations for sequences of the form $(a,b,\tau,c,i_0)$ are straightforward, so we consider the tuple $(a,b,c,\tau,i_0)$. The structure diagrams with non-trivial contribution are the following:
\[
\begin{tikzcd}[column sep=0cm,row sep=.75cm]
a
	\ar[drr]
&
b
	\ar[dr]
&
c
	\ar[ddrr]
&
\tau
	\ar[ddr]
&
i_0
	\ar[dd]
\\
&&\mu_2
\ar[drr]
&&\,
\\
&&&&
\delta_4^1
	\ar[d]
	\ar[dr]
\\
&&&&\,&\,
\end{tikzcd}
+
\begin{tikzcd}[column sep=0cm,row sep=.75cm]
a
	\ar[ddrrrr]
&
b
	\ar[drr]
&
c
	\ar[dr]
&
\tau
	\ar[ddr]
&
i_0
	\ar[dd]
\\
&&&\mu_2
\ar[dr]
&\,
\\
&&&&
\delta_4^1
	\ar[d]
	\ar[dr]
\\
&&&&\,&\,
\end{tikzcd}
+\begin{tikzcd}[column sep=0cm,row sep=.75cm]
a
	\ar[ddrrrr]
&
b
	\ar[ddrrr]
&
c
	\ar[dr]
&
\tau
	\ar[d]
&
i_0
	\ar[dd]
\\
&&&\mu_2
	\ar[dr]
&\,
\\
&&&&
\delta_4^1
	\ar[d]
	\ar[dr]
\\
&&&&\,&\,
\end{tikzcd}
+\begin{tikzcd}[column sep=0cm,row sep=.65cm]
a
	\ar[ddrrrr]
&
b
	\ar[drrr]
&
c
	\ar[drr]
&
\tau
	\ar[dr]
&
i_0
	\ar[d]
\\
&&&&\delta_4^1
	\ar[d]
	\ar[ddr]
\\
&&&&
\delta_2^1
	\ar[d]
	\ar[dr]
\\
&&&&\,&\mu_2
\end{tikzcd}
\]
The vanishing of the above equation follows from the algebraic relation
\[
\d_{\scU} (ab)\d_{\scV}(c)+\d_\scU (a) \d_{\scV}(bc)+\d_{\scU}(a)\d_{\scV}(b) c+a \d_{\scU}(b)\d_{\scV}(c)=0,
\]
which is straightforward to verify.
\end{proof}

In Section~\ref{sec:mapping-cone-bimodules} we describe another perspective which makes the $DA$-bimodule relations more straightforward to verify.

\subsection{Basepoint actions as $A_\infty$-morphisms}
\label{sec:basepoint-actions-A_infty-morphisms}

We now give a helpful reformulation of algebraic basepoint actions in terms of $A_\infty$-morphisms. This framework helps us understand the structure maps appearing in the alpha-to-beta transformer bimodule.

If $C$ is a free $\bF[U]$-module, then there is natural basepoint action $\Phi_U\colon C\to C$, defined as follows. We pick a free $\bF[U]$ basis $B$ of $C$. If $\xs\in B$, then we write $\d(\xs)=\sum_{\ys\in B} c_{\ys,\xs}U^{n_{\ys,\xs}}\cdot  \ys$ where $c_{\ys,\xs}\in \bF$ and $n_{\ys,\xs}\in \N$. The map $\Phi_U$ is defined via the formula
\begin{equation}
\Phi_U(\xs)=\sum_{\ys\in B} n_{\ys,\xs}c_{\ys,\xs}U^{n_{\ys,\xs}-1} \cdot \ys.
\label{eq:Phi_U_def}
\end{equation}
The induced map $(\Phi_U)_*$ on $H_*(C)$ is trivial (cf. Lemma~\ref{rem:Phi-nullhomotopic}), though the map $\Phi_U$ is often not null-homotopic through an $\bF[U]$-equivariant homotopy.

Since $C$ is free and the differential is $\bF[U]$-equivariant, we may view $C$ as corresponding to a type-$D$ module $C^{\bF[U]}$. Of course, 
\[
C^{\bF[U]}\iso C^{\bF[U]}\boxtimes {}_{\bF[U]} [\bI]^{\bF[U]}.
\]
There is a $DA$-bimodule endomorphism $\phi_*^1$ of ${}_{\bF[U]} [\bI]^{\bF[U]}$ defined as follows. We set $\phi_1^1=0$, and we set 
\[
\phi_2^1(\a, 1)=1\otimes \d_U(\a).
\]
We set $\phi_j^1=0$ if $j>2$. The $DA$-bimodule relations are straightforward to verify, and easily translated into the Leibniz rule for derivatives.

\begin{lem}\label{lem:Phi-box-tensor-product} The map $\Phi_U$ in Equation~\eqref{eq:Phi_U_def}, viewed as a type-$D$ morphism
\[
\Phi_U\colon C^{\bF[U]}\to C^{\bF[U]},
\]
satisfies
\[
\Phi_U:=\bI_C\boxtimes \phi_*^1.
\]
\end{lem}
The proof is immediate.

We recall that the map $\Phi_U$ satisfies $\Phi_U^2\simeq 0$. See Lemma~\ref{lem:basic-graph-TQFT-relations}. It is helpful to note that this relation may be translated into a fact about the $DA$-bimodule endomorphism $\phi_*^1$:

\begin{lem}\label{lem:phi^2=0-bimodules} The $DA$-bimodule morphism $\phi_*^1$ satisfies
\[
\phi_*^1\circ \phi_*^1\simeq 0.
\]
\end{lem}
\begin{proof} We define a $DA$-bimodule endomorphism 
\[
H_*^1\colon {}_{\bF[U]} [\bI]^{\bF[U]}\to {}_{\bF[U]} [\bI]^{\bF[U]}
\]
 via the formula $H_1^1=0$ and 
\[
H_2^1(U^n,1)= 1\otimes \frac{n(n-1)}{2} U^{n-2}.
\]
The map $\d_{\Mor}(H_*^1)$ has three summands, shown as graphs below.
\[
\d_{\Mor}(H_*^1)=
\hspace{-.6cm}
\begin{tikzcd}[column sep=.1cm, row sep=.4cm]
\,\ar[dr]&\,\ar[d]&[-.4cm] \,\ar[dd]\\
&\mu_2\ar[dr]&
\\[-.2cm]
&&H_2^1
	\ar[d]
	\ar[dr]
\\
&&\,&\mu_2
\end{tikzcd}
+
\begin{tikzcd}[column sep=.1cm, row sep=.4cm]
\,\ar[ddrr]&[-.2cm]\,\ar[dr]& \,\ar[d]\\
&&H_2^1 \ar[d] \ar[ddr]
\\[-.2cm]
&&\delta_2^1
	\ar[d]
	\ar[dr]
	\\
&&\,&\mu_2
\end{tikzcd}
+
\begin{tikzcd}[column sep=.1cm, row sep=.4cm]
\,\ar[ddrr]&[-.2cm]\,\ar[dr]& \,\ar[d]\\
&&\delta_2^1 \ar[d] \ar[ddr]
\\[-.2cm]
&&H_2^1\ar[d]\ar[dr]\\
&&\,&\mu_2
\end{tikzcd}
\]
Hence
\[
\begin{split}
&\d_{\Mor}(H_*^1)(U^n,U^m,1)\\
=&1\otimes \frac{n(n-1)+m(m-1)+(n+m)(n+m-1)}{2} U^{n+m-2}\\
\equiv&1\otimes nm U^{n+m-2}\pmod 2\\
=&(\phi_*\circ \phi_*)_2^1 (U^n,U^m,1)
\end{split}
\]
completing the proof.
\end{proof}

We will also need to consider analogous basepoint actions on complexes over the 2-variable polynomial ring $\bF[\scU,\scV]$. In this case, we may define two $DA$-bimodule endomorphisms of ${}_{\bF[\scU,\scV]} [\bI]^{\bF[\scU,\scV]}$, denoted $\phi_*^1$ and $\psi_*^1$, which differentiate with respect to $\scU$ and $\scV$, respectively. Lemma~\ref{lem:phi^2=0-bimodules} adapts to show that $\phi_*^1\circ \phi_*^1\simeq 0$ and $\psi_*^1\circ \psi_*^1\simeq 0$. Additionally, we have the following:

\begin{lem}
\label{lem:Phi-Psi-commute}
 As endomorphisms of ${}_{\bF[\scU,\scV]}[\bI]^{\bF[\scU,\scV]}$, we have
\[
\phi_*^1\circ \psi_*^1\simeq \psi_*^1\circ \phi_*^1.
\]
\end{lem}
\begin{proof}
We define a $DA$-bimodule endomorphism $H_*^1$ of ${}_{\bF[\scU,\scV]}[\bI]^{\bF[\scU,\scV]}$ whose only non-vanishing action is given by the formula
\[
H_2(\a,1)=1\otimes \d^2_{\scU,\scV}(\a),
\]
where $\d^2_{\scU,\scV}(a)=\d_{\scU}(\d_{\scV}(a))$.
We observe the equality
\[
\d_{\Mor}(H_*^1)=\phi_*^1\circ \psi_*^1+\psi_*^1\circ \phi_*^1
\]
which follows from the easily verified relation
\[
\d_{\scU,\scV}^2(a\cdot b)+a\cdot  \d^2_{\scU,\scV}(b)+\d_{\scU,\scV}^2(a)\cdot b=\d_{\scU}(a)\cdot \d_{\scV}(b)+\d_{\scV}(a)\cdot \d_{\scU}(b).
\]
\end{proof}

\subsection{The transformer as a mapping cone bimodule}
\label{sec:mapping-cone-bimodules}

In order to efficiently prove statements about the transformer bimodule, it is helpful to discuss some generalities about constructing bimodules. If $f_*^1\colon {}_{\cA} M^{\cB}\to {}_{\cA} N^{\cB}$ is a morphism of $DA$-bimodules such that $\d_{\Mor}(f_*^1)=0$, then we may construct another bimodule $\Cone(f_*^1)$ which is also a $DA$-bimodule. If $f_*^1$ and $g_*^1$ are two homotopic bimodule morphisms, then $\Cone(f_*^1)$ and $\Cone(g_*^1)$ are homotopy equivalent.

It is helpful to understand that the identity and transformer bimodules are mapping cone bimodules. For $\veps\in \{0,1\}$, write  ${}_{\cK}[\ve{i}_\veps]^{\cK}$ denote $\ve{I}_\veps\cdot {}_{\cK}[\bI]^{\cK}\cdot \ve{I}_{\veps}$, with the restriction of the structure map from the identity bimodule ${}_{\cK} [\bI]^{\cK}$. There are two morphisms
\[
s_*^1,t_*^1\colon {}_{\cK}[\ve{i}_0]^{\cK}\to {}_{\cK}[\ve{i}_1]^{\cK}
\]
as follows. The map $t_j^1=0$ for $j\neq 2$, and $t_2^1(f\tau,i_0)=i_1\otimes f\tau$ and $t_2^1(f\sigma,i_0)=0$. The map $s_*^1$ is similar, but with the roles of $\sigma$ and $\tau$ reversed. It is not hard to see that the $DA$-bimodule morphisms $s_*^1$ and $t_*^1$ are cycles. Furthermore
\[
{}_{\cK}[\bI]^{\cK}=\Cone(s_*^1+t_*^1\colon {}_{\cK} [\ve{i}_0]^{\cK}\to {}_{\cK} [\ve{i}_1]^{\cK}).
\]
The $A_\infty$ basepoint actions $\phi_*^1$ and $\psi_*^1$ from Section~\ref{sec:basepoint-actions-A_infty-morphisms} also induce endomorphisms of ${}_{\cK} [\ve{i}_\veps]^{\cK}$ for $\veps\in \{0,1\}$. In particular, we observe that the transformer bimodule may be described as
\[
{}_{\cK}\cT^{\cK}=\Cone(s_*^1+(\id+\phi_*^1\circ \psi_*^1)\circ t_*^1\colon {}_{\cK} [\ve{i}_0]^{\cK}\to {}_{\cK} [\ve{i}_1]^{\cK}).
\]
\begin{rem} The above perspective gives a more streamlined proof that ${}_{\cK} \cT^{\cK}$ satisfies the $DA$-bimodule structure relations. Indeed, it is sufficient to verify that $\phi_*^1$, $\psi_*^1$, $s_*^1$ and $t_*^1$ are cycles, which is easy to verify.
\end{rem}

\subsection{Properties of the transformer bimodule}

In this section, we prove some important properties of the transformer bimodule.

We begin by showing that in analogy to Sarkar's map for a Dehn twist \cite{SarkarMovingBasepoints}, the transformer bimodule is involutive:

\begin{lem}
\label{lem:transformer-involutive}The transformer bimodule satisfies
\[
{}_{\cK} \cT^{\cK}\hatbox{}_{\cK} \cT^{\cK}\simeq {}_{\cK} [\bI]^{\cK}.
\]
\end{lem}
\begin{proof} We compute that ${}_{\cK}\cT^{\cK}\hatbox {}_{\cK} \cT^{\cK}$ is equal to the mapping cone bimodule
\[
\Cone\left(s_*^1+(\id+\phi_*^1\circ \psi_*^1\circ \phi_*^1\circ \psi_*^1)\circ t_*^1\colon {}_{\cK} [\ve{i}_0]^{\cK}\to {}_{\cK} [\ve{i}_1]^{\cK}\right).
\]
On the other hand, $\phi_*^1\circ \psi_*^1\circ \phi_*^1\circ \psi_*^1\simeq 0$ by Lemmas~\ref{lem:phi^2=0-bimodules} and~\ref{lem:Phi-Psi-commute}. Hence,  ${}_{\cK} \cT^{\cK}\hatbox {}_{\cK} \cT^{\cK}$ is homotopy equivalent to $\Cone(s_*^1+t_*^1)$, which is ${}_{\cK} [\bI]^{\cK}$.
\end{proof}

Next, we consider the tensor product of the transformer with the modules for a solid torus:
\begin{lem}
\label{lem:tensor-transformer-solid-torus} The transformer bimodule satisfies
\[
\cD_\lambda^{\cK}\hatbox {}_{\cK} \cT^{\cK}\simeq \cD_\lambda^{\cK}\quad\text{and} \quad {}_{\cK} \cT^{\cK}\hatbox {}_{\cK} \cD_\lambda^{\bF[U]}\simeq {}_{\cK} \cD_\lambda^{\bF[U]}.
\]
\end{lem}
\begin{proof} The relation $\cD_\lambda^{\cK}\hatbox {}_{\cK} \cT^{\cK}\simeq \cD_\lambda^{\cK}$ is trivial, so we focus on the relation ${}_{\cK} \cT^{\cK}\hatbox {}_{\cK} \cD_\lambda^{\bF[U]}\simeq {}_{\cK} \cD_\lambda^{\bF[U]}$.  To simplify the notation, it is easier to construct a homotopy equivalence
\[
{}_{\cK} \cT^{\cK}\hatbox {}_{\cK} [\cD_\lambda]_{\bF[U]}\simeq {}_{\cK} [\cD_\lambda]_{\bF[U]},
\]
though the maps we write down can be easily adapted to give the statement about $DA$-bimodules, since they are strictly $\bF[U]$-equivariant.

We will view both bimodules as mapping cone bimodules, in the sense of Section~\ref{sec:mapping-cone-bimodules}. Define
\[
{}_{\cK}[\ve{d}_{\veps}]_{\bF[U]}:=\ve{I}_\veps\cdot {}_{\cK}[\cD_{\lambda}]_{\bF[U]},
\]
for $\veps\in \{0,1\}$. We observe that both ${}_{\cK} \cT^{\cK}\hatbox {}_{\cK} [\cD_{\lambda}]_{\bF[U]}$ and ${}_{\cK} [\cD_{\lambda}]_{\bF[U]}$ are the mapping cones of type-$AA$ module maps from $\ve{d}_0$ to $\ve{d}_1$. Indeed, there are type-$AA$ module maps
\[
\Sigma_*, T_*\colon {}_{\cK}[\ve{d}_0]_{\bF[U]}\to {}_{\cK}[\ve{d}_1]_{\bF[U]},
\]
where $\Sigma_{1,1,0}(f \sigma,\ve{x})=m_2(f\sigma,\ve{x})$ (where $m_2$ is the action from $\cD_\lambda$) and $T_{1,1,0}(f\tau,\ve{x})=m_2(f\tau,\ve{x})$. Furthermore, $\Sigma_{1,1,0}(f\tau,\ve{x})=0$ and $T_{1,1,0}(f\sigma,\ve{x})=0$. Also $\Sigma_*$ and $T_*$ vanish on other configurations of inputs. It is straightforward to verify that $\Sigma_*$ and $T_*$ are 
cycles. Additionally,
\begin{equation}
{}_{\cK}[\cD_{\lambda}]_{\bF[U]}=\Cone\left( \Sigma_*+T_*\colon {}_{\cK} [\ve{d}_0]_{\bF[U]}\to {}_{\cK} [\ve{d}_1]_{\bF[U]}\right)
\label{eq:D-as-AA}
\end{equation}
and 
\begin{equation}
{}_{\cK} \cT^{\cK}\hatbox {}_{\cK}[\cD_{\lambda}]_{\bF[U]}=\Cone\left( \Sigma_*+(\id+\phi_*\circ\psi_*)\circ T_*\colon {}_{\cK} [\ve{d}_0]_{\bF[U]}\to {}_{\cK} [\ve{d}_1]_{\bF[U]}\right).
\label{eq:TD-as-AA}
\end{equation}
Here, we are writing $\phi_*$ and $\psi_*$ for $\phi_*^1\boxtimes \bI_{\cD_\lambda}$ and $\psi_*^1\boxtimes \bI_{\cD_\lambda}$, where $\phi_*^1$ and $\psi_*^1$ are the endomorphisms of ${}_{\cK}[\ve{i}_1]^{\cK}$ described above.

We observe that
\begin{equation}
\phi_*\circ \psi_*\simeq\scV^{-1} \phi_*\circ(\scU \phi_*+\scV \psi_*),\label{eq:phi-psi=phi(UphiVpsi)}
\end{equation}
as endomorphisms of $[\ve{d}_1]$, since $\phi_*^2\simeq 0$ by Lemma~\ref{lem:phi^2=0-bimodules}.

We write $\d_{\scW}=\scU \d_{\scU}+\scV \d_{\scV}$. We define an endomorphism $J_*$ of ${}_{\cK} [\ve{d}_1]_{\bF[U]}$ via the formula
\[
J_{0,1,0}(\ve{x})=\d_{\scW}(\ve{x}).
\]
We set $J_{i,1,j}=0$ unless $i=j=0$. Observe that $J_{0,1,0}$ commutes with the right action of $U$.  We claim that
\[
\d_{\Mor}(J_*)=\scU \phi_*+\scV \psi_*.
\]
To see that, we note that there are exactly two structure trees which contribute to $\d_{\Mor}(J_*)$. These are shown below:
\[
\d_{\Mor}(J_*)_{1,1,0}(a,\xs)=
\begin{tikzcd}[row sep=.3cm]
a \ar[dr] & \xs\ar[d]\\
& \delta_2^1\ar[d]\\
& J_{0,1,0}\ar[d]\\
&\,
\end{tikzcd}
+
\begin{tikzcd}[row sep=.3cm]
a\ar[ddr] & \xs\ar[d]\\
& J_{0,1,0}\ar[d]\\
& \delta_2^1\ar[d]\\
&\,
\end{tikzcd}
\]
Hence
\[
\d_{\Mor}(J_*)_{1,1,0}(a,\xs)= \d_{\scW}(a \cdot \xs)+a\cdot \d_{\scW}(\xs).
\]
By the Leibniz rule we have that the above is $\d_{\scW}(a)\cdot \xs.$
Therefore
\[
\d_{\Mor}(J_*)=\scU \phi_*+\scV \psi_*,
\]
as endomorphisms of $[\ve{d}_1]$. We conclude that $\phi_*\circ \psi_*\simeq 0$ as endomorphisms of $[\ve{d}_1]$ by Equation~\eqref{eq:phi-psi=phi(UphiVpsi)}. Hence the two mapping cones in Equations~\eqref{eq:D-as-AA} and~\eqref{eq:TD-as-AA} are homotopy equivalent. The proof is complete.
\end{proof}

\subsection{Proof of Theorem~\ref{thm:transformer-bimodule}}

We now prove Theorem~\ref{thm:transformer-bimodule}, which states that changing from an alpha-parallel arc to a beta-parallel arc may be achieved by tensoring with the transformer bimodule ${}_{\cK} \cT^{\cK}$. 

\begin{proof}[Proof of Theorem~\ref{thm:transformer-bimodule}]
We use our basepoint moving map formula from Corollary~\ref{cor:different-path-formula}.

Let $\scA$ be an arc system for $L\subset S^3$, and assume that $K_i\subset L$ is a component whose arc $\scA_i$ is beta-parallel. Let $\scA_i'$ denote an arc for $K_i$ which is alpha-parallel, and let $\scA'$ denote the arc system which has $\scA_i'$ in place of $\scA$.

We pick a Heegaard link diagram $\cH=(\Sigma,\as,\bs,\ws,\zs)$ for $(S^3,L)$ so that all of the arcs of $\scA$ are embedded in $\Sigma$ and pairwise disjoint. We assume the arc $\scA_i$ is disjoint from $\bs$. The arc $\scA_i'$ may also be viewed as being embedded on $\Sigma$. With respect to this embedding, $\scA_i'$ is disjoint from $\as$. In principle $\scA_i'$ might intersect some of the other arcs of $\scA$. However, by stabilizing the Heegaard diagram $(\Sigma,\as,\bs,\ws,\zs)$, we may assume that $\scA_i'$ is disjoint from all of the other arcs of $\scA$. 

We build a hyperbox of Heegaard diagrams using the diagram $\cH$ as the initial diagram. If $K_i\subset L$, then the $K_i$ direction of this hypercube involves the composition of a tautological diffeomorphism map to move $z_i$ along $\scA_i$, followed by a holomorphic polygon counting maps which change the alpha and beta curves back to their original position. We do this in each axis direction to build a $\sigma$-basic system used for the link surgery formula.

For this $\sigma$-basic system, we may change $\scA_i$ to $\scA_i'$ by composing the $K_i$ direction with a final basepoint moving hypercube which moves $w_i$ in a loop along the curve $[K_i]=\scA_i'*\scA_i^{-1}$. Corollary~\ref{cor:different-path-formula} implies that the effect on the link surgery hypercube is to replace $F^{-K_i}$ (the sum of the hypercube maps ranging over all sublinks of $L$ containing $-K_i$) with the map $(\id+\scV_i^{-1}\cA_{[K_i]}\circ \Phi_{w_i}) \circ F^{-K_i}$.

The map $(\id+\scV_i^{-1}\cA_{[K_i]} \circ \Phi_{w_i})$ is viewed as an endomorphism of the $(\ell-1)$-dimensional subcube of $\cC_{\Lambda}^1(L,\scA)\subset \cC_{\Lambda}(L,\scA)$ which has $K_i$-component $1$. On this subcube, the arc $\scA_i'$ intersects only beta curves, and $\scA_i$ intersects only alpha curves.  By Lemma~\ref{lem:homology-action-versus-basepoint-action}, we have
\[
\cA_{[K_i]}\simeq \scU_i \Phi_{w_i}+\scV_i \Psi_{z_i},
\] 
as endomorphisms of $\cC_{\Lambda}^1(L,\scA)$. 
Hence
\[
\id+\scV_i^{-1}\Phi_{w_i}\cA_{[K_i]}=\id+\scV^{-1}_i\Phi_{w_i} (\scU_i \Phi_{w_i}+\scV_i \Psi_{z_i})\simeq \id+\Phi_{w_i} \Psi_{z_i},
\]
since $\Phi_{w_i}^2\simeq 0$. In particular, it follows that if we write
\[
\cC_{\Lambda}(L,\scA)\iso \Cone\left(
\begin{tikzcd}[column sep=2cm]
\cC^0_{\Lambda}(L,\scA)\ar[r, "F^{K_i}+F^{-K_i}"]&\cC^1_{\Lambda}(L,\scA)
\end{tikzcd}
\right)
\]
then $\cC_{\Lambda}(L,\scA')$ is homotopic to
\[
\Cone\left(
\begin{tikzcd}[column sep=3cm]
\cC^0_{\Lambda}(L,\scA)\ar[r, "F^{K_i}+(\id+\Phi_{w_i}\Psi_{z_i})F^{-K_i}"]&\cC^1_{\Lambda}(L,\scA)
\end{tikzcd}
\right).
\]
On the level of type-$D$ structures over $\cL$, we claim that the above complex corresponds to the tensor product of $\cX_{\Lambda}(L,\scA)^{\cL}$ with the transformer bimodule ${}_{\cK} \cT^{\cK}$. To see this, observe that the components of the box tensor differential of $\cX_{\Lambda}(L,\scA)^{\cL}\hatbox {}_{\cK} \cT^{\cK}$ which involve the map $\delta_2^1$ of ${}_\cK \cT^{\cK}$ contribute the internal differentials of $\cC^\veps_{\Lambda}(L,\scA)$ (for $\veps\in \{0,1\}$), as well as the maps $F^{K_i}$ and $F^{-K_i}$ to the mapping cone differential. The terms of the box tensor product differential which involve $\delta_4^1$ of ${}_{\cK} \cT^{\cK}$ contribute the term $\Phi_{w_i} \Psi_{z_i} F^{-K_i}$ to the above mapping cone differential (compare Lemma~\ref{lem:Phi-box-tensor-product}). This completes the proof. 
\end{proof}

\section{The pair-of-pants bimodules}
\label{sec:pair-of-pants}
In this section, we describe several bimodules ${}_{\cK|\cK} W_{\a\b, \b}^{\cK}$ and ${}_{\cK|\cK} W_{\a\b,\a}^{\cK}$. The bimodules are related by the transformer bimodule
\[
{}_{ \cK|\cK} W_{\a\b, \b}^{\cK} \hatbox {}_{\cK} \cT^{\cK}\simeq {}_{\cK|\cK} W_{\a\b, \b}^{\cK}.
\]
We call these the \emph{pair-of-pants} bimodules. Other choices of $\a$ and $\b$ in the subscripts can also be constructed by tensoring in copies ${}_{\cK} \cT^{\cK}$ to switch from $\a$ to $\b$ and vice-versa.

After defining these bimodules and proving several basic properties, we will prove the following:

\begin{thm}\label{thm:alt-pairing} Suppose that $L_1,L_2\subset S^3$ are framed links with distinguished components $K_1\subset L_1$ and $K_2\subset L_2$. Write $n=|L_1|$ and $m=|L_2|$. Let $\scA_{1,\a}$ and $\scA_{2,\b}$ be systems of arcs for $L_1$ and $L_2$. Suppose that the arc for $K_1$ is alpha parallel and the arc for $K_2$ is beta parallel. (The other arcs need not be alpha or beta parallel). Let $\scA_{\a}$ and $\scA_{\b}$ be the system of arcs for $L_1\# L_2$ which use an alpha-parallel (resp. beta-parallel) arc for $K_1$ (resp. $K_2$). Then
\[
\left(\cX_{\Lambda_1}(L_1,\scA_{1,\a})^{\cK\otimes \cL_{n-1}},  \cX_{\Lambda_1}(L_2,\scA_{2,b})^{\cK\otimes \cL_{m-1}}\right)\hatbox {}_{\cK|\cK} W_{\a\b,\a}^{\cK}\simeq \cX_{\Lambda_1+\Lambda_2}(L_1\# L_2, \scA_{\a})^{\cL_{n+m-1}}
\]
and
\[
\left(\cX_{\Lambda_1}(L_1,\scA_{1,\a})^{\cK\otimes \cL_{n-1}},  \cX_{\Lambda_1}(L_2,\scA_{2,\b})^{\cK\otimes \cL_{m-1}}\right)\hatbox {}_{\cK|\cK} W_{\a\b,\b}^{\cK}\simeq \cX_{\Lambda_1+\Lambda_2}(L_1\# L_2, \scA_{\b})^{\cL_{n+m-1}}.
\]
\end{thm}

\subsection{The $W_{\a\b,\a}$ and $W_{\a\b,\b}$ bimodules}

We now define the $W_{\a\b,\a}$ and $W_{\a\b,\b}$ bimodules. We consider $W_{\a\b,\b}$ first. This bimodule has a $\delta_3^1$ which is identical to the merge module $M$. Additionally, there is a $\delta_5^1$, determined by the following formula:
\begin{equation}
\delta_5^1(a|b, a'|b', 1|\tau,\tau|1, i_0)=i_1\otimes \scV^{-1}\d_{\scU}(ab) a' \left(\scU \d_{\scU}+\scV \d_{\scV}\right)(b') \tau. 
\label{eq:delta-5-1-Wr}
\end{equation}
In the above, $a,b,a',b'\in \ve{I}_1\cdot \cK
\cdot \ve{I}_1$. 

The module $W_{\a\b,\a}$ is similar, except has
\[
\delta_5^1(a|b, a'|b', 1|\tau,\tau|1, i_0)=i_1\otimes \scV^{-1}\d_{\scU}(ab) b' \left(\scU \d_{\scU}+\scV \d_{\scV}\right)(a') \tau.
\]
\begin{lem}\label{lem:Wrl-bimodules}
 ${}_{ \cK|\cK} W_{\a\b,\a}^{\cK}$ and ${}_{\cK|\cK} W_{\a\b,\b}^{\cK}$ satisfy the $DA$-bimodule structure relations. Furthermore, both are split Alexander bimodules when we give the incoming algebra factors the discrete partition.
\end{lem}
\begin{proof}  We consider the structure relations first. We take the perspective of mapping cone bimodules from Section~\ref{sec:mapping-cone-bimodules}. We begin by considering the merge module ${}_{\cK|\cK} M^{\cK}$. We define bimodules ${}_{\cK|\cK} [\ve{m}_\veps]^{\cK}$ for $\veps\in \{0,1\}$ by
\[
(\ve{I}_{\veps}\otimes \ve{I}_{\veps})\cdot  {}_{\cK|\cK}M^{\cK} \cdot \ve{I}_{\veps}.
\]
The bimodule structure relations are easily seen to be satisfied. There are two $DA$-bimodule morphisms
\[
\scS_*^1,\scT_*^1\colon {}_{\cK|\cK}[\ve{m}_0]^{\cK}\to {}_{\cK|\cK} [\ve{m}_1]^{\cK}
\]
as follows.  We set $\scS_j^1=0$ if $j\neq 3$. We set $\scS_3^1(1|\sigma,\sigma|1,i_0)=i_1\otimes \sigma$. The map $\scT_*^1$ is similar, but with $\tau$ replacing $\sigma$. The $DA$ bimodule endomorphisms $\scS_*^1$ and $\scT_*^1$ are easily seen to be cycles. Furthermore,
\[
{}_{\cK|\cK} M^{\cK}=\Cone\left(\scS_*^1+\scT_*^1\colon {}_{\cK|\cK} [\ve{m}_0]^{\cK}\to {}_{\cK|\cK} [\ve{m}_1]^{\cK}\right).
\]
Note that we may form four $DA$-bimodule endomorphisms $\phi_*^l$, $\phi^r_*$, $\psi^l_*$ and $\psi_*^r$ of ${}_{\cK|\cK} [\ve{m}_1]^{\cK}$ by adapting the construction from Section~\ref{sec:basepoint-actions-A_infty-morphisms}. Here $\phi_*^l$ differentiates with respect to $\scU$ on the left factor of $\cK\otimes \cK$, while  $\phi_*^r$ differentiates with respect to $\scU$ on the right factor. We may also view these maps as endomorphisms of ${}_{\cK|\cK} [\ve{m}_1]^{\cK}$ by boxing with the identity map on $[\ve{m}_1]$.
With respect to the above notation, we observe that
\begin{equation}
{}_{\cK|\cK}W_{\a\b, \a}^{\cK}=\Cone\left(
\begin{tikzcd}[column sep=6cm] {[\ve{m}_0]}\ar[r, "\scS_*^1+(\id +\scV^{-1}(\phi_*^l+\phi_*^r)\circ (\scU \phi_*^l+\scV \psi_*^l))\circ \scT_*^1"] & {[\ve{m}_1]}
\end{tikzcd}
\right)
\label{eq:W_l-mapping-cone-description}
\end{equation}
The map 
\[
(\id +\scV^{-1}(\phi_*^l+\phi_*^r) \circ (\scU \phi_*^l+\scV \psi_*^r))\circ \scT_*^1
\]
 is a cycle since it is the composition of cycles. Similarly, $\scS_*^1$ is a cycle by direct computation. Hence the $DA$-bimodule relations are satisfied for ${}_{\cK|\cK}W_{\a\b, \a}^{\cK}$. The same argument works for ${}_{\cK|\cK}W_{\a\b, \b}^{\cK}$.
 
 Finally, the claim about split continuity is proven by an easy modification of the proof for the merge module, from Lemma~\ref{lem:merge-continuous}.
\end{proof}

\subsection{Proof of Theorem~\ref{thm:alt-pairing}}

We now give a proof of Theorem~\ref{thm:alt-pairing}:

\begin{proof}[Proof of Theorem~\ref{thm:alt-pairing}]
The proof is based on our formula for changing the path in the link surgery formula Corollary~\ref{cor:different-path-formula}. It is easier to analyze the surgery hypercubes, instead of their interpretation as type-$D$ modules, so we focus on the statement in terms of hypercube complexes first.

 Letting $K_1\subset L_1$ and $K_2\subset L_2$ be the special components, and suppose that $\scA_{1,\a}$ and $\scA_{2,\b}$ are systems of arcs such that $K_1$ has an alpha parallel arc and $K_2$ has a beta-parallel arc. We use the meridional $\sigma$-basic systems considered in Section~\ref{sec:meridional}, which we assume are algebraically rigid. These were constructed in Lemma~\ref{lem:algebraic-rigid-general-basic-system}. By Theorem~\ref{thm:pairing-links}, we know that
\[
\cC_{\Lambda_1+\Lambda_2}(L_1\# L_2,\scA_{1,\a}\# \scA_{2,\b})
\]
\[
\iso
 \Cone\left(
\begin{tikzcd}[column sep=3.5cm]
\cC^0(L_1)\otimes \cC^0(L_2)
\ar[r, "{F^{K_1}|F^{K_2}+F^{-K_1}|F^{-K_2}}"]
& \cC^1(L_1)\otimes \cC^1(L_2)
\end{tikzcd}\right)
\]
(Recall that we are writing $F^{\pm K_i}$ for the sum of all surgery maps for oriented sublinks of $L_i$ which contain $\pm K_i$).

The above surgery hypercube is for the system of arcs $\scA_{1,\a}\# \scA_{2,\b}$. The arc in $\scA_{1,\a}\# \scA_{2,\b}$ for $K_1\# K_2$ is the co-core of the connected sum band. The present statement instead involves $\sigma$-basic systems of Heegaard diagrams for $L_1\# L_2$ which use an alpha or beta parallel arc on $K_1\# K_2$. If we want to change the arc to be alpha parallel, we may stack the above hypercube for $\cC_{\Lambda_1+\Lambda_2}(L_1\# L_2,\scA_{1,\a}\# \scA_{2,\b})$ with the hypercube for moving the basepoint in a loop following $K_2$ in the subcube $\cC^{1}(L_1)\otimes \cC^{1}(L_2)$.  We may compute this cube using Corollary~\ref{cor:different-path-formula}. The effect on the resulting surgery hypercube is to replace $F^{-K_1}|F^{-K_2}$ with the expression
\[
(\id+\scV^{-1}\Phi_{w}\circ \cA_{[K_2]})(F^{-K_1}|F^{-K_2})\colon \cC^1(L_1\# L_2)\to \cC^1(L_1\# L_2).
\]

We are using Proposition~\ref{prop:connected-sums} to identify $\cC^\veps(L_1\# L_2)\iso \cC^\veps(L_1)\otimes \cC^\veps(L_2)$, for $\veps\in \{0,1\}$. The differential on each $\cC^{\veps}(L_1\#L_2)$ is the tensor product differential, i.e. $\d_1\otimes \id+\id\otimes \d_2$. In particular, the map $\Phi_{w}$ on $\cC^{\veps}(L_1\#L_2)$ is intertwined with the map
\[
\Phi_{w_1}\otimes \id+\id\otimes \Phi_{w_2}
\]
under the connected sum isomorphism $\cC^1(L_1\#L_2)\iso \cC^1(L_1)\otimes \cC^1(L_2)$. Similarly, using the connected sum formulas for hypercubes in Propositions~\ref{prop:disjoint-unions-hypercubes-main} and~\ref{prop:connected-sums}, it is straightforward to see that with respect to the isomorphism   $\cC^1(L_1\# L_2)\iso \cC^1(L_1)\otimes \cC^1(L_2)$ the map $\cA_{[K_2]}$ is intertwined with $\id\otimes \cA_{[K_2]}$.   By Lemma~\ref{lem:homology-action-versus-basepoint-action} we see that
\[
\id\otimes \cA_{[K_2]}=\id\otimes (\scU \Phi_{w_2}+\scV \Psi_{z_2}),
\]
as endomorphisms of $\cC^\veps(L_1)\otimes \cC^{\veps}(L_2)$.
We conclude that
\[
\Phi_w\circ \cA_{[K_2]}=(\Phi_{w_1}\otimes \id+\id\otimes \Phi_{w_2})\circ \left(\id\otimes (\scU \Phi_{w_2}+\scV \Psi_{z_2})\right).
\]
The addition of this term is encoded exactly by the extra $\delta_5^1$ term in Equation~\eqref{eq:delta-5-1-Wr} (cf. Lemma~\ref{lem:Phi-box-tensor-product}), so the proof is complete.
\end{proof}

\subsection{Relation with the $AA$-identity bimodule}
\label{sec:AAA-identity}
We recall that the $AA$-identity bimodule ${}_{\cK|\cK} [\bI^{\Supset}]$ from Section~\ref{sec:AAidentity} admits an extension to a bimodule
\[
{}_{\cK| \cK} [\bI^{\Supset}]_{\bF[U]}.
\]

It follows from Theorem~\ref{thm:pairing-links} and Proposition~\ref{prop:equivalence-type-D-change-basepoint}  that we may use ${}_{\cK|\cK} [\bI^{\Supset}]$ to obtain a valid model of the link surgery formula when taking the connected sum of two links  (i.e. the resulting hypercube has the correct module actions on the remaining knot components). It is not clear from this argument whether tensoring with ${}_{\cK|\cK} [\bI^{\Supset}]_{\bF[U]}$ gives a valid model. In this section, we show that indeed tensoring with ${}_{\cK|\cK} [\bI^{\Supset}]_{\bF[U]}$ produces a valid model, via the following algebraic result:

\begin{prop}
\label{prop:W-I-Supset-relation} There is a homotopy equivalence
\[
{}_{\cK|\cK} W_{\a\b, \a}^{\cK}\hatbox {}_{\cK}[\cD_0]_{\bF[U]}\simeq {}_{\cK|\cK} [\bI^{\Supset}]_{\bF[U]}.
\]
The same holds if we replace $W_{\a\b, \a}$ with $W_{\a\b, \b}$.
\end{prop}

The above proposition will be verified in several steps.

\begin{lem}
\label{lem:basic-relations-s-t-phi-psi}
\begin{enumerate}
\item As morphisms from ${}_{\cK}[\ve{i}_0]^{\cK}$ to ${}_{\cK}[\ve{i}_1]^{\cK}$, we have
\[
\phi_*^1 \circ s_*^1\simeq s_*^1\circ \phi_*^1 \quad \text{and} \quad \psi_*^1 \circ s_*^1 \simeq s_*^1\circ \psi_*^1.
\]
Similarly
\[
t_*^1\circ \phi_*^1\simeq \scV^2\psi_*^1\circ t_*^1\quad \text{and} \quad t_*^1\circ \psi_*^1\simeq \scV^{-2} \phi_*^1 \circ t_*^1.
\]
\item As morphisms from ${}_{\cK|\cK}[\ve{m}_0]^{\cK}$ to ${}_{\cK|\cK}[\ve{m}_1]^{\cK}$ we have
\[
\phi_*^l \circ \scS_*^1\simeq \scS_*^1\circ \phi_*^l \quad \text{and} \quad \psi_*^l \circ \scS_*^1 \simeq \scS_*^1\circ \psi_*^l
\]
and
\[
\scT_*^1\circ \phi_*^l\simeq \scV^2\psi_*^l\circ \scT_*^1\quad \text{and} \quad \scT_*^1\circ \psi_*^l\simeq \scV^{-2} \phi_*^l \circ \scT_*^1.
\]
The same relations hold with $\phi^r_*$ and $\psi_*^r$ in place of $\phi_*^l$ and $\psi_*^l$. 
\end{enumerate}
\end{lem}
\begin{proof} We begin with two algebraic relations on $\cK$. If $f\in \ve{I}_1\cdot \cK \cdot \ve{I}_1$ and $g\in \ve{I}_0\cdot \cK\cdot \ve{I}_0$ are such that $\tau a=b \tau$ then
 \begin{equation}
 \tau\d_\scU(a)=\scV^2 \d_{\scV}(b) \tau \quad \text{and} \quad\tau  \d_{\scV}(a)=\scV^{-2} \d_{\scU}(b) \tau.
 \end{equation}
 We verify the above equations. If $a=\scU^i\scV^j$, then $b=\phi^\tau(a)=\scU^j \scV^{2j-i}$. Then
 \[
\tau \d_{\scU}(a)= \tau \d_{\scU}(\scU^i \scV^j)=\tau i \scU^{i-1} \scV^j=i\scU^j \scV^{2j-i+1}\tau= \scV^2 \d_{\scV}(b)\tau
 \]
 as claimed. Similarly
 \[
 \tau \d_{\scV}(a) =\tau j \scU^i \scV^{j-1}=j \scU^{j-1}\scV^{2j-2-i} \tau =\scV^{-2} \d_{\scU}(b) \tau.
 \]

 We now consider the equation $\phi_*^1\circ s_*^1\simeq s^1_* \circ \phi_*^1$. We define a morphism $h_*^1\colon {}_{\cK}[\ve{i}_0]^{\cK}\to {}_{\cK}[\ve{i}_1]^{\cK}$ via the formula 
 \[
 h_2^1(f \sigma,1)=1\otimes\d_{\scU}(f)\sigma.
 \]
 One checks easily that $\d_{\Mor}(h_*^1)=\phi_*^1\circ s_*^1+s_*^1\circ \phi_*^1$. The other relations with $t_*^1$ and $\phi_*^1$ and $\psi_*^1$ are proven similarly.
 
 We now consider the relations for $\scS_*^1\circ \phi_*^l\simeq  \phi_*^l\circ \scS_*^1$. Define $f_*^1:=\scS_*^1\circ \phi_*^l$ and note that $f_*^1$ has $f_4^1$ non-trivial, and this map satisfies
  \[
  f_4^1(\sigma|1, 1| \sigma, a|a', i_0)=i_1\otimes \sigma \d_{\scU}(a)a'.
  \] 
 Note that more generally,
 $f_4^1(a \sigma|a', b| b' \sigma, c|c', i_0)=i_1\otimes  a a' b b' \sigma \d_{\scU}(c) c'$, though we omit the extra possible coefficients on the $\sigma$ terms to simplify the notation.
Define $g_*^1:=\phi_*^l\circ \scS_*^1$ and note that $g_*^1$ has only $g_4^1$ non-trivial, and
  \[
  g_4^1(a|a', \sigma|1, 1| \sigma, i_0)=i_1\otimes   \d_{\scU}(a)a' \sigma.
  \]
  There is a third map of interest, $k_*^1$, which has
  \[
    k_4^1(\sigma|1,a|a', 1|  \sigma, i_0)=i_1\otimes   a' \sigma \d_{\scU}(a).
  \]
  A homotopy $h_*^1$ between $f_*^1$ and $k_*^1$ is given by
  \[
  h_3^1(a\sigma |a', b|b'\sigma,i_0)=i_1\otimes a a'  b' \sigma \d_{\scU}(b).
  \]
 A homotopy between $k_*^1$ and $g_*^1$ is constructed similarly. The other relations in the lemma statement are proven by minor modifications of the above argument.
\end{proof}

\begin{lem}
\label{lem:tensor-W-D-lambda}
There is a homotopy equivalence
\[
{}_{\cK|\cK} W_{\a\b, \a}^{\cK}\hatbox {}_{\cK}[\cD_0]_{\bF[U]}\simeq \left({}_{\cK}\cT^{\cK}, {}_{\cK} \cT^{\cK}\right) \hatbox{}_{\cK|\cK} [\bI^{\Supset}]_{\bF[U]}.
\]
 The same holds if we replace $W_{\a\b, \a}$ with $W_{\a\b, \b}$.
\end{lem}

\begin{rem} The above statement does not hold if we replace $W_{\a\b,\a}$ with $W_{\a\a,\a}$, $W_{\a\a,\b}$, $W_{\b\b,\a}$ or $W_{\b\b,\b}$. 
\end{rem}

\begin{proof}[Proof of Theorem~\ref{lem:tensor-W-D-lambda}] The argument is similar to the proof of Lemma~\ref{lem:tensor-transformer-solid-torus}. The key is to observe that we may define a homotopy $\d_{\scW}=\scU \d_{\scU}+\scV \d_{\scV}$, and that  $\d_{\scW}$ commutes with the action of $U=\scU\scV$. 

In our present context, we want to view both modules in the statement as $AA$-bimodule mapping cones. For $\veps\in \{0,1\}$, write
\[
{}_{\cK|\cK}[\ve{md}_{\veps}]_{\bF[U]}:=(\ve{I}_{\veps}\otimes \ve{I}_{\veps})\cdot {}_{\cK|\cK} [\bI^{\Supset}]_{\bF[U]}.
\]
Both of the bimodules in the statement may be viewed as a mapping cone of an $AA$-bimodule morphism from $[\ve{md}_{0}]$ to $[\ve{md}_1]$. 

Similar to Equation~\eqref{eq:W_l-mapping-cone-description}, we may view $W_{\a\b, \a}\hatbox \cD_0$ as the mapping cone of
\begin{equation}
\scS_*+(\id +\scV^{-1}(\phi_*^l+\phi_*^r) (\scU \phi_*^l+\scV \psi_*^l))\circ \scT_*.
\label{eq:mapping-cone-W_l-D_0}
\end{equation}
We observe that the morphism $H_{0,1,0}\colon [\ve{md}_1]\to [\ve{md}_1]$ given by $H_{0,1,0}(\xs)=\d_{\scW}(\xs)$ satisfies
\[
\d_{\Mor}(H_{0,1,0})=\scU\phi_*^l+\scU \phi^r_*+\scV\psi_*^l+\scV \psi_*^r.
\]
In particular, we obtain that
\[
\begin{split}
&\id+\scV^{-1}(\phi_*^l+\phi_*^r)(\scU \phi_*^l+\scV \psi_*^l)
\\
=&\id+\scV^{-1}\phi_*^l(\scU \phi_*^l+\scV \psi_*^l)+ \scV^{-1}\phi^r_*(\scU \phi_*^l+\scV \psi_*^l)\\
\simeq& \id+ \scV^{-1}\phi_*^l(\scU \phi_*^l+\scV \psi_*^l)+\scV^{-1}\phi^r_*(\scU \phi^r_*+\scV \psi^r_*)\\
\simeq& \id +\phi^l_*\psi^l_*+\phi_*^r \psi_*^r\\
\simeq& (\id+\phi_*^l \psi_*^l)(\id+\phi_*^r\psi^r_*).
\end{split}
\]
In the last equivalence, we are using the fact that
\[
\phi_*^l \psi_*^l\phi^r_*\psi^r_*\simeq 0
\]
since $\phi_*^l \psi_*^l\simeq \scV^{-1} \phi_*^l(\scU \phi_*^l +\scV \psi_*^l)$ (and similarly for $\phi^r_*\psi_*^r$) so
\[
\phi^l_* \psi^l_* \phi^r_*\psi^r_*\simeq \scV^{-2} \phi_*^l \phi_*^r(\scU \phi_*^l+\scV \psi_*^l)(\scU \phi^r_*+\scV\psi^r_*)\simeq \scV^{-2} \phi_*^l \phi_*^r(\scU \phi_*^l+\scV \psi_*^l)^2\simeq 0.
\]
Hence, the map in Equation~\eqref{eq:mapping-cone-W_l-D_0} is homotopy equivalent to
\[
\scS_*+(\id+\phi_*^l \psi_*^l)(\id+\phi_*^r\psi^r_*) \scT_*.
\]
Arguing similarly to the proof of Lemma~\ref{lem:basic-relations-s-t-phi-psi}, one sees that the above map is homotopic to the map whose cone is
\[
\left({}_{\cK} \cT^{\cK}, {}_{\cK} \cT^{\cK}\right)\hatbox {}_{\cK|\cK}[ \bI^{\Supset}]_{\bF[U]},
\]
completing the proof.
\end{proof}

\begin{lem}
\label{lem:Transformers-and-M}
 There is a homotopy equivalence
\[
\left({}_{\cK}\cT^{\cK}, {}_{\cK} \cT^{\cK}\right) \hatbox {}_{\cK|\cK} M^{\cK} \hatbox {}_{\cK} \cT^{\cK}\simeq {}_{\cK|\cK} M^{\cK}.
\]
\end{lem}

\begin{proof} The proof is similar to the proof of Lemma~\ref{lem:tensor-W-D-lambda}. We observe first by computing algebraically using the techniques of Lemma~\ref{lem:basic-relations-s-t-phi-psi} that $({}_{\cK}\cT^{\cK}, {}_{\cK} \cT^{\cK}) \hatbox {}_{\cK|\cK} M^{\cK} \hatbox {}_{\cK} \cT^{\cK}$ is homotopy equivalent to the mapping cone
\begin{equation}
\begin{tikzcd}[column sep=10cm]
{[\ve{m}_0]}\ar[r, "\scS_*^1+(\id+(\phi_*^l+\phi^r_*)(\psi^l_*+\psi_*^r))(\id+\phi_*^l\psi_*^l)(\id+\phi^r_*\psi_*^r)\scT_*^1"] &{[\ve{m}_1]}.
\end{tikzcd}
\label{eq:mapping-cone-tensor-products-with-transformers}
\end{equation}
One observes by direct computation that
\[
(\id+(\phi_*^l+\phi_*^r)(\psi_*^l+\psi_*^r))(\id+\phi_*^l\psi_*^l)(\id+\phi_*^r\psi_*^r)\simeq (\id+\phi_*^l\psi_*^r)(\id+\psi_*^l\phi_*^r).
\]

We define the mapping cone complexes
\[
{}_{\cK|\cK}X:=\begin{tikzcd}[column sep=6cm]
{[\ve{m}_0]}\ar[r, "\scS_*^1+(\id+\phi_*^l\psi^r_*)(\id+\psi_*^l\phi_*^r)\scT_*^1"] &{[\ve{m}_1]},
\end{tikzcd}
\]
 and
\[
{}_{\cK|\cK}Y:=\begin{tikzcd}[column sep=6cm]
{[\ve{m}_0]}\ar[r, "\scS_*^1+(\id+\psi_*^l\phi_*^r)\scT_*^1(\id+\phi_*^l\psi^r_*)"]&{[\ve{m}_1]}.
\end{tikzcd}
\]
We now claim that ${}_{\cK|\cK}X$ and ${}_{\cK|\cK} Y$ are homotopy equivalent. We construct a homotopy equivalence $F$ by way of the diagram
\begin{equation}
\begin{tikzcd} {}_{\cK|\cK} X^{\cK}\ar[d, "F"]\\
{}_{\cK|\cK} Y^{\cK}
\end{tikzcd}=
\begin{tikzcd}[column sep=8cm,row sep=1cm,labels=description]
{[\ve{m}_0]}
\ar[d, "\id+\phi_*^l\psi^r_*"]
\ar[dr,dashed]
\ar[r, "\scS_*^1+(\id+\phi_*^l\psi_*^r)(\id+\psi_*^l\phi_*^r)\scT_*^1"] &{[\ve{m}_1]}
\ar[d, "\id+\phi_*^l\psi_*^r"]
\\
{[\ve{m}_0]}\ar[r, "\scS_*^1+(\id+\psi_*^l\phi_*^r)\scT_*^1(\id+\phi_*^l\psi_*^r)"] &{[\ve{m}_1]}
\end{tikzcd}.
\label{eq:homotopy-equivalence-phipsi}
\end{equation}
The length 2 map above is a morphism which realizes the chain homotopies
\[
(\id+\phi_*^l\psi^r_*)^2(\id+ \psi_*^l\phi^r_*) \scT_*^1\simeq (\id+ \psi_*^l\phi_*^r) \scT_*^1 (\id+\phi_*^l \psi_*^r)^2, \quad \text{and}
\]
\[
[\id+\phi_*^l\psi_*^r, \scS_*^1]\simeq 0.
\]
Note that by a similar construction we can construct a map in the opposite direction
\[
\begin{tikzcd} {}_{\cK|\cK} Y^{\cK}\ar[d, "G"]\\
{}_{\cK|\cK} X^{\cK}
\end{tikzcd}=
\begin{tikzcd}[column sep=8cm,row sep=1cm,labels=description]
{[\ve{m}_0]}
\ar[d, "\id+\phi_*^l\psi^r_*"]
\ar[dr,dashed]
\ar[r, "\scS_*^1+(\id+\psi_*^l\phi_*^r)\scT_*^1(\id+\phi_*^l\psi_*^r)"] &{[\ve{m}_1]}
\ar[d, "\id+\phi_*^l\psi_*^r"]
\\
{[\ve{m}_0]}\ar[r, "\scS_*^1+(\id+\phi_*^l\psi_*^r)(\id+\psi_*^l\phi_*^r)\scT_*^1"] &{[\ve{m}_1]}
\end{tikzcd}.
\]
Note that since $(\id+\phi_*^l \psi_*^r)^2\simeq \id$ as a map from $[\ve{m}_i]$ to $[\ve{m}_i]$, it follows that $G\circ F$ is homotopic to a map $J\colon {}_{\cK|\cK} X^{\cK}\to {}_{\cK|\cK} X^{\cK}$ which takes the form
\[
\begin{tikzcd} {}_{\cK|\cK} X^{\cK}\ar[d, "J"]\\
{}_{\cK|\cK} X^{\cK}
\end{tikzcd}=
\begin{tikzcd}[column sep=8cm,row sep=1cm,labels=description]
{[\ve{m}_0]}
\ar[d, "\id"]
\ar[dr,dashed]
\ar[r, "\scS_*^1+(\id+\phi_*^l\psi_*^r)(\id+\psi_*^l\phi_*^r)\scT_*^1"] &{[\ve{m}_1]}
\ar[d, "\id"]
\\
{[\ve{m}_0]}\ar[r, "\scS_*^1+(\id+\phi_*^l\psi_*^r)(\id+\psi_*^l\phi_*^r)\scT_*^1"] &{[\ve{m}_1]}
\end{tikzcd}.
\]
Note that $J^2=\id$. Therefore, the map $F$ has a left homotopy inverse $J\circ G$ since $J\circ G\circ F\simeq \id$. One may modify the construction above to construct a right inverse to $F$ by similar reasoning. By basic algebra, the left and right inverses of $F$ must be homotopic.  Hence $F$ is a homotopy equivalence.

Finally, we use Lemma~\ref{lem:basic-relations-s-t-phi-psi} to see that
\[
(\id+\psi_*^l\phi_*^r)\scT_*^1(\id+\phi_*^l\psi^r_*)\simeq (\id+\psi_*^l\phi_*^r)^2\scT_*^1\simeq \scT_*^1.
\]
It follows that the $({}_{\cK} \cT^{\cK}, {}_{\cK} \cT^{\cK})\hatbox {}_{\cK|\cK} M^{\cK}\boxtimes {}_{\cK}\cT^{\cK}$ (shown in Equation~\eqref{eq:mapping-cone-tensor-products-with-transformers}) is homotopy equivalent to the mapping cone
\[
\Cone(\scS_*^1+\scT_*^1\colon {}_{\cK|\cK}[\ve{m}_0]^{\cK}\to {}_{\cK|\cK}[\ve{m}_1]^{\cK})
\]
which is the definition of ${}_{\cK|\cK} M^{\cK}$.
\end{proof}

Proposition~\ref{prop:W-I-Supset-relation} is now proven as follows. Lemma~\ref{lem:tensor-W-D-lambda} implies that
\[
{}_{\cK|\cK} W_{\a\b, \a}^{\cK}\hatbox {}_{\cK} [\cD_0]_{\bF[U]}\simeq ({}_{\cK} \cT^{\cK}, {}_{\cK} \cT^{\cK}) \hatbox {}_{\cK|\cK} M^{\cK} \hatbox {}_{\cK} [\cD_0]_{\bF[U]}.
\]
Lemmas~\ref{lem:transformer-involutive} and~\ref{lem:Transformers-and-M} imply that
\[
({}_{\cK} \cT^{\cK}, {}_{\cK} \cT^{\cK}) \hatbox {}_{\cK|\cK} M^{\cK} \hatbox {}_{\cK} [\cD_0]_{\bF[U]}\simeq {}_{\cK|\cK} M^{\cK} \hatbox{}_{\cK} \cT^{\cK} \hatbox {}_{\cK} [\cD_0]_{\bF[U]}.
\]
Lemma~\ref{lem:tensor-W-D-lambda} implies that the above is homotopy equivalent to
\[
{}_{\cK|\cK} M^{\cK}  \hatbox {}_{\cK} [\cD_0]_{\bF[U]},
\]
which is the definition of ${}_{\cK|\cK} [\bI^{\Supset}]_{\bF[U]}.$

\subsection{Relation with the bordered invariant of $S^1\times P$}

\label{sec:bordered-pair-of-pants}

In this section, we sketch why the bimodules ${}_{\cK|\cK}W_{\a\b,\b}^{\cK}$ and ${}_{\cK|\cK} W_{\a\b,\a}^\cK$ can naturally be interpreted as the bordered invariants for $S^1\times P$, where where $P$ denotes a 2-dimensional disk with 2 subdisks removed. In fact, this is a consequence of the connected sum formula and a computation from \cite{CZZSatellites} of the invariant for the identity mapping cylinder over the 2-torus.

\begin{prop}\label{prop:W-module-interpretation} The bimodules ${}_{\cK|\cK}W_{\a\b,\b}^{\cK}$ and ${}_{\cK|\cK} W_{\a\b,\a}^\cK$ are the bordered invariants for $S^1\times P$, for various choices of arc systems on the boundary components. 
\end{prop}
\begin{rem} The arc systems on the boundary are compatible with the following gluing convention:
\begin{enumerate}
\item We glue alpha bordered boundary components to alpha bordered boundary components, and vice-versa.
\item ${}_{\cK|\cK} \bI^{\Supset}$ has one alpha-bordered boundary and one beta-bordered boundary component.
\end{enumerate}
Since we switch from type-$D$ to type-$A$ by tensoring with ${}_{\cK|\cK} \bI^{\Supset}$, this convention means that a type-$A$ component of $\cK$ is ``alpha bordered'' if the arc system is beta-parallel, whereas a type-$D$ component of $\cK$ is ``alpha bordered'' if the corresponding arc system is alpha-parallel.
\end{rem}

\begin{proof}[Proof of Proposition~\ref{prop:W-module-interpretation}]
A surgery description of $S^1\times P$ is given in Figure~\ref{fig:30}. We can view the five component link appearing in this surgery description as being obtained as the connected sum of two copies of a three component link $J$, as shown in Figure~\ref{fig:30}. (The link $J$ is a connected sum of two Hopf links).  Write $\scA_{\a,\b}$ for an arc system on $J$ where the top component is alpha-parallel, and the bottom component is beta-parallel. The middle component can have either alpha or beta parallel arc
.
Write ${}_{\cK}\cX_0(J,\scA_{\a,\b})^{\cK}$ for the $DA$-bimodule where the left action is top-most component of $J$, and the right action is for the bottom-most component. The middle component is surgered out. (I.e. we take the type-$D$ module for the three component link, tensor ${}_{\cK} \cD_0$ to the middle component, and tensor ${}_{\cK|\cK} \bI^{\Supset}$ to the top component).  It is proven in \cite{CZZSatellites}*{Theorem~1.7} that
\begin{equation}
{}_{\cK}\cX_0(J,\scA_{\a,\b})^{\cK}\simeq {}_{\cK}\cX_0(J,\scA_{\b,\a})^{\cK}\simeq {}_{\cK}[\bI]^{\cK}.\label{eq:identity-cobordism}
\end{equation}
(This computation is obtained by tensoring two copies of the Hopf link complex which we compute in Section~\ref{sec:Hopf-links-comp} of the present paper).
It follows from our connected sum formula that if ${}_{\cK|\cK}\cX(L)^{\cK}$ denotes the surgery module for the 5 component link describing $S^1\times P$, then
\[
{}_{\cK|\cK} \cX(L)^{\cK}\simeq \left({}_{\cK}\cX_0(J,\scA_{\a,\b})^{\cK},{}_{\cK}\cX_0(J,\scA_{\b,\a})^{\cK}\right)\hatbox {}_{\cK|\cK} W_{\b\a, \a}^{\cK}.
\]
By Equation~\eqref{eq:identity-cobordism}, the above is homotopy equivalent to ${}_{\cK|\cK} W_{\b\a, \a}^{\cK},$ completing the proof.
\end{proof}

\begin{figure}[ht]
\centering
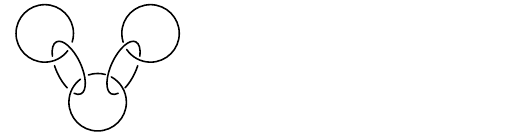
\caption{A Kirby diagram for $S^1\times P$. The components marked $0$ are surgered, while a tubular neighborhood of the unmarked components is removed. On the right, we describe the 5 component link as a connected sum of two copies of a 3-component link $J$.}
\label{fig:30}
\end{figure}

\section{The link surgery complex of the Hopf link}
\label{sec:Hopf-links-comp}

In this section, we compute the link surgery complex of the Hopf link. Our main result is Proposition~\ref{prop:Hopf-link-large-model}. 

\subsection{Computation of the surgery complex}
\label{sec:descent-maps-Hopf}

In this section, we compute the surgery hypercube for the Hopf link. We use the following model of the Hopf link:
 \begin{equation}
\cCFL(H)\iso \begin{tikzcd}[labels=description,row sep=1cm, column sep=1cm] \ve{a}& \ve{b}\ar[l, "\scU_2"] \ar[d, "\scU_1"]\\
\ve{c}\ar[r, "\scV_2"] \ar[u, "\scV_1"]& \ve{d}.
\end{tikzcd}
\label{eq:Hopf-def}
\end{equation}
This complex is realized by the diagram in Figure~\ref{fig:11}. This complex is for the \emph{negative} Hopf link, i.e., the two components have linking number $-1$. The complex of the positive Hopf link has a similar model. The techniques of this section may also be applied, essentially verbatim, to the positive Hopf link, though we focus on the negative one for concreteness.

The Alexander gradings $A=(A_1,A_2)$ and Maslov gradings $\gr=(\gr_{\ws},\gr_{\zs})$ of the generators are
\begin{equation}
\begin{split} A(\ve{a})&=(-\tfrac{1}{2}, \tfrac{1}{2}),\\
A(\ve{b})&=(-\tfrac{1}{2}, -\tfrac{1}{2}),\\
A(\ve{c})&=(\tfrac{1}{2}, \tfrac{1}{2}),\\
A(\ve{d})&= (\tfrac{1}{2}, -\tfrac{1}{2}),
\end{split}
\qquad \qquad\qquad 
\begin{split}
\gr(\ve{a})&=(0,0),\\
\gr(\ve{b})&=(-1,1),\\
\gr(\ve{c})&=(1,-1),\\
\gr(\ve{d})&=(0,0).
\end{split}
\label{eq:Alexander-grading-Hopf}
\end{equation}
Note that we are using the conventions where the top degree generator of $\HF^-(S^3,\ws)$ is in Maslov grading 0. See Remark~\ref{rem:grading-conventions}.

The main result of this section is the following result:
 \begin{prop}
 \label{prop:Hopf-link-large-model}
  Suppose that $\Lambda=(\lambda_1,\lambda_2)$ is an integral framing on the negative Hopf link $H$. Let $\scA_{\a\a}$ be a system of arcs for the Hopf link where both arcs are alpha-parallel. The link surgery hypercube $\cC_{\Lambda}(H,\scA_{\a\a})$ has the following maps (up to overall homotopy equivalence).
 \begin{enumerate}
  \item The maps $\Phi^{L_1}_{H,L_2}$ and $\Phi^{L_1}_{L_1,\emptyset}$ are the canonical inclusions of localization. The maps $\Phi^{-L_1}_{H,L_2}$ and $\Phi^{-L_1}_{L_1,\emptyset}$ are both given by the following formula:
  \[
 \Phi^{-L_1} =\begin{tikzcd}[row sep=0cm]
 \ve{a} \ar[r,mapsto]& \scV_1^{\lambda_1-1}\ve{d}\\
 \ve{b} \ar[r, mapsto]& 0\\
 \ve{c} \ar[r, mapsto]& \scV_1^{\lambda_1+1}\ve{b}+\scV_1^{\lambda_1}\scU_2\ve{c}\\
 \ve{d} \ar[r,mapsto]& \scV_1^{\lambda_1}\scU_2\ve{d}.
 \end{tikzcd}
  \]
 \item The maps $\Phi^{L_2}_{H,L_1}$ and $\Phi^{L_2}_{L_2,\emptyset}$ are the canonical inclusions of localization. The maps $\Phi^{-L_2}_{H,L_1}$ and $\Phi^{-L_2}_{L_2,\emptyset}$ are both given by the following formula:
\[
\Phi^{-L_2}=\begin{tikzcd}[row sep=0cm]
\ve{a}\ar[r, mapsto] &\scU_1\scV_2^{\lambda_2}\ve{a}\\
 \ve{b}\ar[r, mapsto] &0\\
 \ve{c}\ar[r, mapsto] &\scV_2^{\lambda_2+1}\ve{b}+\scU_1\scV_2^{\lambda_2} \ve{c}\\
 \ve{d}\ar[r, mapsto]&\scV_2^{\lambda_2-1}\ve{a}.
 \end{tikzcd}
\]
 \item The length 2 map $\Phi^{-H}_{H,\emptyset}$ is given by the following formula:
 \[
\Phi^{-H}_{H,\emptyset}=\begin{tikzcd}[row sep=0cm]
 \ve{a}\ar[r, mapsto] &\scV_1^{\lambda_1-2}\scV_2^{\lambda_2-1}\ve{c}\\
 \ve{b}\ar[r, mapsto] &0\\
  \ve{c}\ar[r, mapsto] &\scV_1^{\lambda_1-1} \scV_2^{\lambda_2}\ve{d}\\
 \ve{d}\ar[r, mapsto]&\scV_1^{\lambda_1-1} \scV^{\lambda_2-2}_2\ve{c}.
 \end{tikzcd}
 \]
 The length 2 maps for other orientations of the Hopf link vanish.
 \end{enumerate}
 \end{prop}

We now compute the link surgery hypercube maps for the Hopf link. For the purposes of computation, it is easier to compute the induced maps after quotienting variables, instead of after localizing. Since the link surgery hypercube maps are defined by computing holomorphic polygons on diagrams where we have deleted certain basepoints, this is easily seen to contain equivalent information as the ordinary surgery hypercube. Compare Lemma~\ref{lem:module-structure-surgery-hypercube-groups}.

As an example, if $L_1\subset H$, then the descent map from surgery hypercube takes the form
\[
\Phi^{-L_1}_{H,H\setminus L_1}\colon \cCFL(H)\to \scV^{-1}_1 \cdot \cCFL(H).
\]
The map $\Phi^{-L_1}_{H,H\setminus L_1}$ is completely determined by the choice of framing on $H$, and the induced homotopy equivalence on the level of quotients
\[
\tilde{\Phi}^{-L_1}_{H,H\setminus  L_1}\colon \cCFL(H)/(\scU_1-1)\to \cCFL(H)/(\scV_1-1).
\]
In this section, we write $\tilde{\Phi}^{\vec{N}}$ for the descent map on the level of quotients of $\cCFL(H)$.

\begin{lem}\label{lem:Hopf-descent}
Write $L_1$ and $L_2$ for the two components of the negative Hopf link. Up to overall equivalence, the maps for the Hopf link  are as follows:
\begin{enumerate}
\item The map 
\[
\tilde{\Phi}_{H,L_2}^{-L_1}\colon \cCFL(H)/(\scU_1-1)\to \cCFL(H)/(\scV_1-1)
\]
 is given by the formula
\[
\tilde{\Phi}_{H,L_2}^{-L_1}=\begin{tikzcd}[row sep=0cm]
\ve{a} \ar[r,mapsto]& \ve{d}\\
\ve{b} \ar[r, mapsto]& 0\\
\ve{c} \ar[r, mapsto]& \ve{b}+\scU_2 \ve{c}\\
\ve{d} \ar[r,mapsto]& \scU_2 \ve{d}.
\end{tikzcd}
\]
extended $\bF[\scU_2,\scV_2]$-equivariantly, and sending $\scV_1$ to $\scU_1$.
\item  The map 
\[
\tilde{\Phi}^{-L_1}_{L_1,\emptyset}\colon \cCFL(H)/(\scU_1-1,\scV_2-1)\to \cCFL(H)/(\scV_1-1,\scV_2-1)
\]
 is given by the same formula as $\tilde{\Phi}_{H,L_2}^{-L_1}$, except with $\scV_2$ set to 1.
 \item The map
 \[
 \tilde{\Phi}^{-L_2}_{H,L_1}\colon \cCFL(H)/(\scU_2-1)\to \cCFL(H)/(\scV_2-1)
 \]
 is given by the formula
 \[
\tilde{\Phi}_{H,L_1}^{-L_2}=
 \begin{tikzcd}[row sep=0cm]
 \ve{a}\ar[r, mapsto] &\scU_1 \ve{a}\\
 \ve{b}\ar[r, mapsto] &0\\
 \ve{c}\ar[r, mapsto] &\ve{b}+\scU_1 \ve{c}\\
 \ve{d}\ar[r, mapsto]&\ve{a}
 \end{tikzcd},
 \]
 extended equivariantly over $\bF[\scU_1,\scV_1]$, and sending $\scV_2$ to $\scU_2$.
 \item The map
 \[
 \tilde{\Phi}^{-L_2}_{L_2,\emptyset}\colon \cCFL(H)/(\scV_1-1, \scU_2-1)\to \cCFL(H)/(\scV_1-1,\scV_2-1)
 \]
 is given by the same formula as $\tilde{\Phi}_{H,L_1}^{-L_2}$, except with $\scV_1$ set to 1.
 \item The length 2 map  
 \[
 \tilde{\Phi}_{H,\emptyset}^{-H}\colon \cCFL(H)/(\scU_1-1,\scU_2-1)\to \cCFL(H)/(\scV_1-1,\scV_2-1)
 \]
 is given by the formula
 \[
\tilde{\Phi}_{H,\emptyset}^{-H}=\begin{tikzcd}[row sep=0cm]
 \ve{a}\ar[r, mapsto] &\ve{c}\\
 \ve{b}\ar[r, mapsto] &0\\
 \ve{c}\ar[r, mapsto] &\ve{d}\\
 \ve{d}\ar[r, mapsto]&\ve{c}
\end{tikzcd},
 \]
 extended to send $\scV_1$ to $\scU_1$ and $\scV_2$ to $\scU_2$.
 \item The length 1 maps with positively oriented components, e.g. $\tilde{\Phi}^{L_1}_{H,L_2}$, are the canonical quotient maps. The length 2 maps vanish for any orientation other than $-H$.
\end{enumerate}
\end{lem}
\begin{rem}
 As we will see in the proof of the above lemma, all of the maps in the surgery hypercube for $H$ are uniquely determined by the chain complex $\cCFL(H)$, except for the length 2 map $\tilde{\Phi}_{H,L_2}^{L_1}$, and up to homotopy there are two possible choices. We will use additional algebraic restrictions on the link surgery hypercube to determine which choice of length 2 map is correct.
\end{rem}
\begin{proof}[Proof of Lemma~\ref{lem:Hopf-descent}] We focus first on the length 1 maps.
Consider the complexes
\[
\cC/(\scU_1-1)=\begin{tikzcd}[labels=description,row sep=1cm, column sep=1cm] \ve{a}& \ve{b}\ar[l, "\scU_2"] \ar[d, "1"]\\
\ve{c}\ar[r, "\scV_2"] \ar[u, "\scV_1"]& \ve{d}.
\end{tikzcd}
\qquad \text{and} \qquad
\cC/(\scV_1-1)=\begin{tikzcd}[labels=description,row sep=1cm, column sep=1cm] \ve{a}& \ve{b}\ar[l, "\scU_2"] \ar[d, "\scU_1"]\\
\ve{c}\ar[r, "\scV_2"] \ar[u, "1"]& \ve{d}.
\end{tikzcd}
\]
The map $\tilde{\Phi}^{-L_1}_{H,L_2}$ is a homotopy equivalence between these two complexes, which is $\bF[\scU_2,\scV_2]$-equivariant, and sends $\scV_1$ to $\scU_1$. We perform the following changes of basis:
\begin{equation}
\cC/(\scU_1-1)\iso \begin{tikzcd}[labels=description,row sep=1cm, column sep=.3cm] 
\ve{a}& \ve{b} \ar[d, "1"]\\
\ve{c}+\scV_2\ve{b} \ar[u, "\scV_1+\scU_2\scV_2"]& \ve{d}+\scU_2 \ve{a}.
\end{tikzcd}
\quad \text{and} \quad 
\cC/(\scV_1-1)\iso \begin{tikzcd}[labels=description,row sep=1cm, column sep=.3cm] 
\ve{a}+\scV_2 \ve{d}& \ve{b}+\scU_2\ve{c} \ar[d, "\scU_1+\scU_2\scV_2"]\\
\ve{c}\ar[u, "1"]& \ve{d}
\end{tikzcd}
\label{eq:C^1*-basis-change}
\end{equation}
We may view each of the above complexes as the knot Floer complex for an unknot, with an extra free basepoint. In the following we write $\gr_{\ws}$ for the grading induced by the complete collection of basepoints consisting of the $\ws$-type link basepoints (i.e. the ones where the unknot intersects the Heegaard surface negatively), and all the free basepoints. We write $\gr_{\zs}$ for the grading induced by the complete collection consisting of the $\zs$-type basepoints and the free basepoints. We may easily compute the $(\gr_{\ws},\gr_{\zs})$-gradings of the all of the quotiented complexes, viewed in this manner:
\[
\cC/(\scU_1-1): \begin{tikzcd}[row sep=0cm, column sep=1.5cm]
\ve{a}\ar[r, mapsto, "{(\gr_{\ws},\gr_{\zs})}"] &(0,0)\\
\ve{b}\ar[r, mapsto] &(-1,1)\\
\ve{c}\ar[r, mapsto] &(-1,-1)\\
\ve{d}\ar[r, mapsto]&(-2,0)
\end{tikzcd}
\qquad \cC/(\scV_1-1):
 \begin{tikzcd}[row sep=0cm,column sep=1.5cm]
\ve{a}\ar[r, mapsto] &(0,-2)\\
\ve{b}\ar[r, mapsto] &(-1,-1)\\
\ve{c}\ar[r, mapsto] &(1,-1)\\
\ve{d}\ar[r, mapsto]&(0,0)
\end{tikzcd}
\]
\[
\cC/(\scU_2-1): \begin{tikzcd}[row sep=0cm,column sep=1.5cm]
\ve{a}\ar[r, mapsto] &(-2,0)\\
\ve{b}\ar[r, mapsto] &(-1,1)\\
\ve{c}\ar[r, mapsto] &(-1,-1)\\
\ve{d}\ar[r, mapsto]&(0,0)
\end{tikzcd}
\qquad \cC/(\scV_2-1): \begin{tikzcd}[row sep=0cm,column sep=1.5cm]
\ve{a}\ar[r, mapsto] &(0,0)\\
\ve{b}\ar[r, mapsto] &(-1,-1)\\
\ve{c}\ar[r, mapsto] &(1,-1)\\
\ve{d}\ar[r, mapsto]&(0,-2)
\end{tikzcd}
\]
The homotopy equivalence $\tilde{\Phi}^{-L_1}_{H,L_2}$ is grading preserving with respect to the above gradings, since it is the map for changing Heegaard diagrams. It is straightforward to see from Equation~\eqref{eq:C^1*-basis-change}
that there is a unique choice of non-zero chain map from $\cC/(\scU_1-1)$ to $\cC/(\scV_1-1)$ which is $(\gr_{\ws},\gr_{\zs})$-grading preserving, is $\bF[\scU_2,\scV_2]$-equivariant, and sends $\scV_1$ to $\scU_1$. (Note here we are using the $(\gr_{\ws},\gr_{\zs})$-bigrading obtained by viewing $\cC/(\scU_1-1)$ and $\cC/(\scV_1-1)$ as the complex of an unknot, shown above; we are not using the gradings naively inherited from the Hopf link). This map is the homotopy equivalence $\tilde{\Phi}_{H,L_2}^{-L_1}$ in the map in the statement.
 By a symmetric argument, one obtains the formula for $\tilde{\Phi}^{-L_2}_{H,L_1}$ in the statement.

Next, by the definition of a $\sigma$-basic system, the map $\tilde{\Phi}^{-L_1}_{L_1,\emptyset}$ is obtained by quotienting the domain and range of $\tilde{\Phi}^{-L_1}_{H,L_2}$ by $\scV_2-1$. The maps $\tilde{\Phi}^{-L_2}_{H,L_1}$ and $\tilde{\Phi}^{-L_2}_{L_2,\emptyset}$ are related by a similar relation.

We now consider the diagonal map of the cube. Firstly, since we are working with a basis system, the length 2 map vanishes except for the orientation $-H$ (i.e. all components oriented negatively). In our present notation, this is the map
\[
\tilde{\Phi}^{-H}_{H,\emptyset}\colon \cC/(\scU_1-1,\scU_2-1)\to \cC/(\scV_1-1,\scV_2-1).
\]

We define the map
\begin{equation}
E:=\tilde{\Phi}^{-L_1}_{L_1,\emptyset}\circ \tilde{\Phi}^{-L_2}_{H,L_1}+\tilde{\Phi}^{-L_2}_{L_2,\emptyset}\circ \tilde{\Phi}^{-L_1}_{H,L_2}. \label{eq:E-map}
\end{equation}
Using our computations of the length 1 maps of the cube, we compute $E$ to be
\[
E=\begin{tikzcd}[row sep=0cm]
\ve{a}\ar[r, mapsto] &\ve{a}+\ve{d}\\
\ve{b}\ar[r, mapsto] &0\\
\ve{c}\ar[r, mapsto] &(\scU_1+\scU_2)\ve{c}\\
\ve{d}\ar[r, mapsto]&\ve{a}+\ve{d},
\end{tikzcd}
\]
extended linearly by sending $\scV_1$ to $\scU_1$ and $\scV_2$ to $\scU_2$.

Up to an overall chain homotopy intertwining $\bF[\scV_1,\scV_2]$ and $\bF[\scU_1,\scU_2]$, there are two choices for a graded null-homotopy of the map $E$. These are the maps:
\begin{equation}
H_1:=\begin{tikzcd}[row sep=0cm]
\ve{a}\ar[r, mapsto] &\ve{c}\\
\ve{b}\ar[r, mapsto] &0\\
\ve{c}\ar[r, mapsto] &0\\
\ve{d}\ar[r, mapsto]&\ve{c},
\end{tikzcd}
\quad \text{and} \quad
H_2:=\begin{tikzcd}[row sep=0cm]
\ve{a}\ar[r, mapsto] &\ve{c}\\
\ve{b}\ar[r, mapsto] &0\\
\ve{c}\ar[r, mapsto] &\ve{d}\\
\ve{d}\ar[r, mapsto]&\ve{c},
\end{tikzcd}
\label{eq:two-homotopies-H1H2}
\end{equation}
To prove this, we first observe that $H_1$ and $H_2$ are easily checked to be valid homotopies. Any other valid homotopy will differ from $H_1$ by a $+1$ graded chain map from $\cC/(\scU_1-1,\scU_2-1)$ to $\cC/(\scV_1-1,\scV_2-1)$. However, it is easily checked that up to chain homotopy, there is exactly one non-zero $+1$ graded map from $\cC/(\scU_1-1,\scU_2-1)$ to $\cC/(\scV_1-1,\scV_2-1)$ which sends $\scV_i$ to $\scU_i$.  This is the map which sends $\ve{c}$ to $\ve{d}$, sends $\scV_i$ to $\scU_i$, and vanishes on other generators. Adding this map to $H_1$ yields $H_2$. Hence, $H_1$ and $H_2$ are the only possible choices, up to further homotopy. We prove in the subsequent Lemma~\ref{lem:length-2-map-Hopf} that $H_1$ is not a valid choice when we use the arc system $\scA_{\a\a}$, and hence the length 2 map of the surgery hypercube is $H_2$.
\end{proof}

We let $(S^2,\a,\b,\{w_1,w_2\}, \{z_1,z_2\})$ be the diagram shown in Figure~\ref{fig:11}, and we $p\in \a$ for the point shown therein. We use define a map $A_p$ by counting disks weighted by the difference in multiplicities on the two sides of the point $p$. Note that the map $A_p$ may be identified with the relative homology map considered in Section~\ref{sec:relative-homology-actions} if we pick a path $\lambda$ which connects a basepoint on $L_1$ to a basepoint on $L_2$ and satisfies $\lambda\cap \a=\{p\}$.

The map $A_p$ is defined on $\cCFL(H)$, and is easily computed by counting holomorphic bigons to be given by the formula
\[
A_p=
\begin{tikzcd}[row sep=1cm, column sep=1cm] 
\ve{a} \ar[d, "\scU_1"] \ar[r, "\scV_2"]& \ve{b}\\
\ve{c}\ar[r, "\scV_2", shift left] & \ve{d} \ar[l, "\scU_2", shift left] \ar[u, "\scV_1"].
\end{tikzcd}
=
 \begin{tikzcd}[row sep=0cm]
\ve{a}\ar[r, mapsto] &\scV_2 \ve{b}+\scU_1 \ve{c}\\
\ve{b}\ar[r, mapsto] &0\\
\ve{c}\ar[r, mapsto] &\scV_2\ve{d}\\
\ve{d}\ar[r, mapsto]&\scU_2 \ve{c}+\scV_1\ve{b}
\end{tikzcd}
\]

\begin{figure}[ht]
\centering
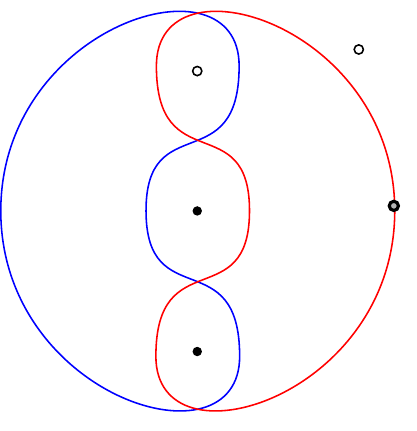
\caption{A diagram for the Hopf link with a special point $p\in \a$.}
\label{fig:11}
\end{figure}

The map $A_p$ has an extension to hypercubes, similar to the construction described in Section~\ref{sec:relative-homology-actions}. We can describe this by considering a formal endomorphism
\[
F_p\colon \a\to \a,
\]
and setting the hypercube action $\cA_p$ to be $\mu_2^{\Tw}(F_p,-).$

As in Lemma~\ref{lem:U-actions-homotopic} we can extend the above construction to give an endomorphism $\cA_p$ of the link surgery formula when we are using a $\sigma$-basic system of Heegaard diagrams for $H$ where the arc system consists of only alpha parallel arcs. In this case, the point-pushing diffeomorphisms never change $\a$. We consider this hypercube on 2-dimensional hypercube used to define all of the link surgery structure maps for negatively oriented sublinks of $H$. With these considerations in mind, we obtain a diagram of the following form:
\begin{equation}
\Cone(\cA_p)=\hspace{-.5cm}\begin{tikzcd}[
column sep={3.5cm,between origins},
row sep=.8cm,labels=description
]
\cC/(\scU_1-1,\scU_2-1)
	\ar[dd, "\tilde{\Phi}^{-L_2}"]
	\ar[dr,  "A_p"]
	\ar[rr, "\tilde{\Phi}^{-L_1}"]
	\ar[ddrr,dashed, "H"]
	\ar[dddrrr,dotted, "\omega_p"]
&&[-1.5cm]
\cC/(\scV_1-1,\scU_2-1)
	\ar[dd, "\tilde{\Phi}^{-L_2}"]
	\ar[dr,"A_p"]
	\ar[dddr,dashed, "h_p^{-L_2}"]
&
\\
&\cC/(\scU_1-1,\scU_2-1)
	\ar[rr,crossing over, "\tilde{\Phi}^{-L_1}"]
&&
\cC/(\scV_1-1,\scU_2-1)
	\ar[dd, "\tilde{\Phi}^{-L_2}"]
	\ar[from=ulll,dashed,crossing over, "h_p^{-L_1}"]
\\[1.8cm]
\cC/(\scU_1-1,\scV_2-1)
	\ar[rr, "\tilde{\Phi}^{-L_1}"]
	\ar[dr,"A_p"]
	\ar[drrr,dashed, "h_p^{-L_1}",,pos=.6]
&&
\cC/(\scV_1-1,\scV_2-1)
	\ar[dr, "A_p"]	
&
\\
&
\cC/(\scU_1-1,\scV_2-1)
	\ar[rr, "\tilde{\Phi}^{-L_1}"]
	\ar[from =uu, crossing over, "\tilde{\Phi}^{-L_2}"]
	\ar[from=uuul,dashed, crossing over, "h_p^{-L_2}"]
	&&
\cC/(\scV_1-1,\scV_2-1)
	\ar[from=uull,crossing over, dashed, "H"]
\end{tikzcd}
\label{eq:hypercube-A_p}
\end{equation}
In the above diagram, both maps labeled $h_p^{-L_2}$ are induced by a single map from $\cC/(\scU_2-1)$ to $\cC/(\scV_2-1)$. The map labeled $h_p^{-L_1}$ is similar.

The diagram in Equation~\eqref{eq:hypercube-A_p} does not satisfy the hypercube relations. Instead, if we view it as $\Cone(\cA_p)$, where $\cA_p$ is a morphism from the back face to the front face, then
\[
[\d, \cA_p]=(\scU_1\scV_1+\scU_2\scV_2)\cdot \id.
\]

We now analyze the maps labeled $h_p^{-L_i}$. Consider first $h_p^{-L_2}$. Both instances of this map is induced by a single map from $\cC/(\scU_2-1)$ to $\cC/(\scV_2-1)$, for which we also write $h_p^{-L_2}$. This maps preserves the Maslov bigradings, if we view $\cC/(\scU_2-1)$ and $\cC/(\scV_2-1)$ as the link Floer complexes for an unknot with an extra free basepoint. Extending the observation from the proof of Lemma~\ref{lem:Hopf-descent} that we used to compute the length 1 maps, there were exactly two grading preserving chain maps from $\cC/(\scU_2-1)$ to $\cC/(\scV_2-1)$ which are $\bF[\scU_1,\scV_1]$-equivariant, and send $\scV_2$ to $\scU_2$: the zero map, and the map labeled $\tilde{\Phi}^{-L_2}$. This observation constrains  $h_p^{-L_2}$ to be one of two maps, which differ by the map $\tilde{\Phi}^{-L_2}$. The same comment applies for $h_{p}^{-L_1}$.

We compute that directly
\[
[\tilde{\Phi}^{-L_1},A_p]= 0.
\]
Hence, by the above observation, we conclude that 
\begin{equation}
h_p^{-L_1}=s\cdot \tilde{\Phi}^{-L_1}\label{eq:h-maps-1}
\end{equation}
for some $s\in \bF$.

Similarly, we compute that
\[
[\tilde{\Phi}^{-L_2},A_p]= \begin{tikzcd}[row sep=0cm]
\ve{a}\ar[r, mapsto] &0\\
\ve{b}\ar[r, mapsto] &0\\
\ve{c}\ar[r, mapsto] &\scU_2 \ve{a}+\scU_1\ve{d}\\
\ve{d}\ar[r, mapsto]&0.
\end{tikzcd}
\]
Hence,
\begin{equation}
h_p^{-L_2}=(\ve{c}\mapsto \ve{b})+t\cdot \tilde{\Phi}^{-L_2},\label{eq:h-maps-2}
\end{equation}
for some $t\in \bF$.

Using the above observations, we are now able to compute the length 2 map of the link surgery hypercube of the Hopf link:

\begin{lem}
\label{lem:length-2-map-Hopf}
The length 3 hypercube relation in Equation~\eqref{eq:hypercube-A_p} is satisfiable with $H_2$, but not with $H_1$. Hence $H_2$ is the length 2 map of the Hopf link surgery hypercube.
\end{lem}
\begin{proof}For $i\in \{1,2\}$, let $C_i$ denote the sum of compositions of maps in~\eqref{eq:hypercube-A_p} which would appear in the length 3 relation, except for $[\d, \omega_p]$. Here, $C_i$ is computed with the map $H=H_i$. The hypercube relations are satisfiable with $H_i$ if and only if $C_i\simeq 0$.

We have determined the maps $h_p^{-L_2}$ and $h_{p}^{-L_1}$ up to a minor ambiguity, as shown in equations~\eqref{eq:h-maps-1} and~\eqref{eq:h-maps-2}. Note that changing the value of $s$ in $h_{p}^{-L_1}$  changes $C_i$ by addition of the map $E$ from~\eqref{eq:E-map}, which we already know to be null-homotopic. Changing $t$ by adding 1 has the same effect. Hence, it is sufficient to consider the case where $s=t=0$.

One easily computes
\[
C_1=[h^{-L_2}_p, \tilde{\Phi}^{-L_1}]+[H_1,A_p]= \begin{tikzcd}[row sep=0cm]
\ve{a}\ar[r, mapsto] &\ve{d}\\
\ve{b}\ar[r, mapsto] &0\\
\ve{c}\ar[r, mapsto] &\ve{b}+\scU_2 \ve{c}\\
\ve{d}\ar[r, mapsto]&\ve{d},
\end{tikzcd}
\]
which is not null-homotopic. Hence, $H_1$ is not a valid choice. It is easily checked that in fact
\[
[h^{-L_2}_p, \tilde{\Phi}^{-L_1}]+[H_2,A_p]=0.
\]
Hence, the diagonal map must be $H_2$.
\end{proof}

\begin{rem}
View the maps $H_1$ and $H_2$ as being $\bF[U_1,U_2]$-equivariant, where $U_i$ acts by $\scV_i$ on the domain, and $\scU_i$ on the codomain. The maps $H_1$ and $H_2$ are homotopic over $\bF[U_1]$ or $\bF[U_2]$, but not over $\bF[U_1,U_2]$. If we wanted to compute the Heegaard Floer homology of surgeries on the Hopf link, either $H_1$ or $H_2$ would give the same answer. Compare \cite{Liu2Bridge}*{Section~5.2}. We will see in Section~\ref{sec:alpha-beta-model} that using $H_1$ to build the link surgery hypercube gives $\cC_{\Lambda}(H,\scA_{\a\b})$ where $\scA_{\a\b}$ is a system of arcs where one arc is alpha-parallel and the other is beta-parallel.

 Since we will be taking tensor products, we need to know the length 2 map over the ring $\bF[U_1,U_2]$, and there are algebraically two distinct models.
\end{rem}

It remains to describe the proof of Proposition~\ref{prop:Hopf-link-large-model}:

\begin{proof}[Proof of Proposition~\ref{prop:Hopf-link-large-model}]
The proof is obtained by filling in powers of $\scV_i$ in the maps from Lemma~\ref{lem:Hopf-descent} according to Alexander grading changes. We leave this mostly to the reader, but we verify the claim for the map $\Phi_{H,L_2}^{-L_1}$. This map sends Alexander grading $\ve{s}$ to $\ve{s}+(\lambda_1,-1)$. Furthermore, it is $\bF[\scU_2,\scV_2]$ equivariant, and is $T$-equivariant with respect to $\bF[\scU_1,\scV_1]$. As an example, the equivariance properties and the computation of the quotient map in Lemma~\ref{lem:Hopf-descent} constrain $\Psi_{H,L_2}^{-L_1}(\ve{a})=\scV_1^N \ve{d}$ for some $N$. Since $A(\ve{a})=(-\tfrac{1}{2},\tfrac{1}{2})$ and $A(\ve{d})=(\tfrac{1}{2},-\tfrac{1}{2})$, we conclude that $N=\lambda_1-1$. Similar arguments hold for all of the other components.
\end{proof}

\subsection{The Hopf link complex as a type-$D$ module}

We now describe the type-$D$ module over $\cL_2=\cK\otimes \cK$ for the Hopf link complex. We write $\cH_{\Lambda}^{\cK\otimes \cK}$ for the type-$D$ module for the negative Hopf link with an alpha-parallel system of arcs. We recall that by definition, for each $\veps\in \bE_2$, we define $\cH\cdot \ve{I}_{\veps}$ to have the same generators over $\bF$ as a free $\bF[\scU_1,\scV_1,\scU_2,\scV_2]$-basis of $\cCFL(H)$. Hence, we write
\[
\cH_{\Lambda}^{\cK\otimes \cK}\cdot \ve{I}_{\veps}=\Span_\bF(\ve{a}_{\veps},\ve{b}_{\veps},\ve{c}_{\veps},\ve{d}_{\veps}).
\]
It is helpful to write $\cH_{\veps}$ for $\cH_{\Lambda}\cdot \ve{I}_{\veps}$. With this notation, we may decompose $\cH_{\Lambda}$ as 
\[
 \cH_{\Lambda}=
 \begin{tikzcd}[column sep=.3cm, row sep=.3cm]
\cH_{0,0} & \cH_{1,0}\\
 \cH_{0,1}& \cH_{1,1}
 \end{tikzcd}
 \]
We now describe the structure maps:

\begin{lem}
\label{lem:Hopf-type-D} Let $\cH_{\Lambda}^{\cK\otimes \cK}$ be the type-$D$ module associated to the negative Hopf link with an alpha-parallel system of arcs. The type-$D$ module $\cH_{\Lambda}^{\cK\otimes \cK}$ is as follows:
\[
\begin{tikzcd}[labels=description, row sep=2cm,column sep=3cm]
\cH_{0,0}
	\ar[loop left, "\d^1"] \ar[d,"{}^{L_2}\!F^1+{}^{-L_2}\! F^1"]
	\ar[r, "{}^{L_1}\! P^1+{}^{-L_1}\! P^1"]
	\ar[dr, dashed, "{}^{-H}\, \Omega^1"]
&
\cH_{1,0}
	\ar[loop right, "\d^1"]
	\ar[d, "{}^{L_2}\! G^1+{}^{-L_2}\! G^1"]\\
\cH_{0,1}
	\ar[loop left, "\d^1"]
	\ar[r, "{}^{L_1}\! Q^1+{}^{-L_1}\! Q^1"]
&
\cH_{1,1}
	\ar[loop right, "\d^1"]
  \end{tikzcd}
  \]
  where the maps are as follows:
  \begin{enumerate}
   \item The map $\d^1$ is the ordinary differential of the Hopf link complex. See Equation~\eqref{eq:Hopf-def}.
    \item The maps denoted ${}^{L_1} \!P^1$, ${}^{-L_1}\! P^1$, ${}^{L_1}\! Q^1$ and ${}^{-L_1}\! Q^1$ correspond to the length 1 surgery maps for $\pm L_1$. The maps ${}^{L_1}\!P^1$ and ${}^{L_1}\! Q^1$ share the same formula as each other, and the maps ${}^{-L_1} \! P^1$ and ${}^{-L_1}\! Q^1$ share the same formula as each other. They are given by the following formulas
    \begin{enumerate}
    \item\[
    {}^{L_1} \!P^1=\begin{tikzcd}[row sep=0cm]
    \ve{a} \ar[r,mapsto]& \ve{a}\otimes \sigma_1\\
    \ve{b} \ar[r, mapsto]& \ve{b}\otimes \sigma_1\\
    \ve{c} \ar[r, mapsto]& \ve{c}\otimes \sigma_1\\
    \ve{d} \ar[r,mapsto]& \ve{d}\otimes \sigma_1.
    \end{tikzcd}
    \]
   \item 
    \[
   {}^{-L_1} \!P^1 =\begin{tikzcd}[row sep=0cm]
   \ve{a} \ar[r,mapsto]& \ve{d}\otimes \scV_1^{\lambda_1-1}\tau_1\\
   \ve{b} \ar[r, mapsto]& 0\\
   \ve{c} \ar[r, mapsto]& \ve{b}\otimes \scV_1^{\lambda_1+1}\tau_1+\ve{c}\otimes \scU_2\scV_1^{\lambda_1}\tau_1\\
   \ve{d} \ar[r,mapsto]& \ve{d}\otimes \scU_2\scV_1^{\lambda_1}\tau_1.
   \end{tikzcd}
    \]
     \end{enumerate}
   \item The maps ${}^{L_2} \!F^1$ and ${}^{-L_2}\! F^1$ are the surgery maps for the component $L_2$,  with $L_1$ unreduced. They are given by the following formulas: 
   \begin{enumerate}
  \item 
  \[
  {}^{L_2}\!F^1=\begin{tikzcd}[row sep=0cm]
  \ve{a} \ar[r,mapsto]& \ve{a}\otimes \sigma_2\\
  \ve{b} \ar[r, mapsto]& \ve{b}\otimes \sigma_2\\
  \ve{c} \ar[r, mapsto]& \ve{c}\otimes \sigma_2\\
  \ve{d} \ar[r,mapsto]& \ve{d}\otimes \sigma_2.
  \end{tikzcd}
  \] 
   \item 
  \[
  {}^{-L_2}\! F^1=\begin{tikzcd}[row sep=0cm]
  \ve{a}\ar[r, mapsto] &\ve{a}\otimes \scU_1\scV_2^{\lambda_2}\tau_2\\
   \ve{b}\ar[r, mapsto] &0\\
   \ve{c}\ar[r, mapsto] &\ve{b}\otimes \scV_2^{\lambda_2+1}\tau_2+ \ve{c}\otimes \scU_1\scV_2^{\lambda_2}\tau_2\\
   \ve{d}\ar[r, mapsto]&\ve{a}\otimes \scV_2^{\lambda_2-1}\tau_2
   \end{tikzcd}
  \]
  The maps ${}^{L_2}\! G^1$ and ${}^{-L_2} \! G^1$ are the surgery maps $L_2$, with $L_1$ already reduced. They are given by the same formulas.
   \end{enumerate}
   \item The map ${}^{-H}\! \Omega^1$ is the length 2 surgery map for $-H=-L_1\cup -L_2$. This map takes the following form
   \[
   {}^{-H} \! \Omega^1=\begin{tikzcd}[row sep=0cm]
   \ve{a}\ar[r, mapsto] &\ve{c}\otimes \scV_1^{\lambda_1-2}\scV_2^{\lambda_2-1}\tau_1\tau_2\\
   \ve{b}\ar[r, mapsto] &0\\
    \ve{c}\ar[r, mapsto] &\ve{d}\otimes\scV_1^{\lambda_1-1} \scV_2^{\lambda_2} \tau_1\tau_2\\
   \ve{d}\ar[r, mapsto]&\ve{c}\otimes \scV_1^{\lambda_1-1}\scV_2^{\lambda_2-2}\tau_1\tau_2.
   \end{tikzcd}
   \]
   \end{enumerate}
\end{lem}
We leave the verification of the above lemma to the reader as it is mostly a restatement of Lemma~\ref{prop:Hopf-link-large-model} using the notation from Section~\ref{sec:link-surgery-A-infty}.

\subsection{The Hopf link with arc system $\scA_{\a\b}$}
\label{sec:alpha-beta-model}

In this section, we compute the type-$D$ module of the Hopf link complex when we use an arc system such that $L_1$ is alpha-parallel, and $L_2$ is beta-parallel. We write $\bar\cH_{\Lambda}^{\cK\otimes \cK}$ for the associated type-$D$ module.

\begin{prop}
\label{prop:bar-H-computation-type-D}
 The type-$D$ module $\bar \cH_{\Lambda}^{\cK\otimes \cK}$ is identical to the module $\cH_{\Lambda}^{\cK\otimes \cK}$, except that we delete the term $\ve{c}\mapsto \ve{d}\otimes \scV_1^{\lambda_1-1}\scV_2^{\lambda_2}\tau_1\tau_2$ from $\delta^1$.
\end{prop}
\begin{proof} We may use Theorem~\ref{thm:transformer-bimodule} and the transformer bimodule to change one of the arcs from alpha-parallel to beta-parallel. Accordingly, we have that
\[ \bar\cH_{\Lambda}^{\cK_1\otimes \cK_2}=\cH_{\Lambda}^{\cK_1\otimes \cK_2}\boxtimes {}_{\cK_2} \cT^{\cK_2}.
\]
We claim that the above tensor product description of $\bar \cH_{\Lambda}^{\cK_1\otimes \cK_2}$ gives exactly the module structure described in the statement (i.e. with structure maps as in Lemma~\ref{lem:Hopf-type-D}, except with the additional length 2 term $\ve{c}\mapsto \ve{d}\otimes\scV_1^{\lambda_1-1}\scV_2^{\lambda_2}\tau_1\tau_2$) . This is verified by a direct computation. We compute $\cH^{\cK_1\otimes \cK_2}_{\Lambda}\hatbox {}_{\cK_2} \cT^{\cK_2}$. The module ${}_{\cK_2} \cT^{\cK_2}$ has a structure map term of the form $\delta_2^1(a,1)=1\otimes a$, for all $a$. These contribute all of the differentials from $\cH^{\cK_1\otimes \cK_2}_{\Lambda}$ to $\bar{\cH}_{\Lambda}^{\cK_1\otimes \cK_2}$. Additionally, ${}_{\cK_2} \cT^{\cK_2}$ has a $\delta_4^1$ term, which is given by 
\[
\delta_4^1(a,b,c \tau,i_0)=i_1\otimes \d_{\scU}(a)\d_{\scV}(b) c \tau.
\] 
 The structure map diagram which adds the extra length 2 differential is the following:
\[
\begin{tikzcd}[labels=description, row sep=.5cm]
\ve{c}_{00}
\ar[d]
& i_0
\ar[dddd]\\
\Phi^{-L_2}
	\ar[dddr]
\ar[d]&\\
\d_{\cC}
	\ar[ddr]
	\ar[d]
&\\
\Phi^{-L_1}
	\ar[dr]
	\ar[dd]
&\\
\,&\delta_4^1\ar[dr]
	\ar[d]
\\
\ve{d}_{11}&i_1& \scV_1^{\lambda_1-1}\scV_2^{\lambda_2}\tau_1\tau_2.
\end{tikzcd}
\]
In the above, $\Phi^{-L_2}$, $\d_\cC$ and $\Phi^{-L_1}$ denote specific summands of the structure map $\delta^1$ on $\cH_{\Lambda}^{\cK_1\otimes \cK_2}$. The specific sequence of generators contributing to the above diagram is the following:
\[
\begin{tikzcd}
\ve{c}_{00} \ar[r, "\Phi^{-L_2}"]& \ve{c}_{01}\otimes \scU_1 \scV^{\lambda_2} \tau_2\ar[r, "\d_\cC"]&[-.3cm] \ve{d}_{01}\otimes \scV_2\otimes \scU_1 \scV_2^{\lambda_2}\tau_2\ar[r, "\Phi^{-L_1}"] & \ve{d}_{11} \otimes \scU_2 \scV_1^{\lambda_1}\tau_1\otimes\scV_2\otimes \scU_1 \scV_2^{\lambda_2}\tau_2.\end{tikzcd}
\]
(Here, we omit all terms which evaluate to zero under $\bI\otimes \delta_4^1$). Applying $\delta_4^1$, we obtain the extra term
\[
\ve{c}_{00}\mapsto \ve{d}_{11}\otimes \scV_1^{\lambda_1} \tau_1 \scU_1 \scV_2^{\lambda_2} \tau_2=\ve{d}_{11}\otimes  \scV_1^{\lambda_1-1} \scV_2^{\lambda_2} \tau_1\tau_2,
\]
as claimed in the statement. Note that when evaluating $\delta_2^1$ and $\delta_4^1$ in the tensor product, we are using the extension of scalars construction from Section~\ref{sec:extension-of-scalars} to view ${}_{\cK_1} \cT^{\cK_1}$ as a $DA$-bimodule over $(\cK_1\otimes \cK_2,\cK_1\otimes \cK_2)$. 
\end{proof}

\begin{rem} In our proof of Proposition~\ref{prop:Hopf-link-large-model}, we observed that all of the maps in the Hopf link complex were uniquely determined except for the length 2 map, and there were two choices, up to overall homotopy. We wrote $H_1$ and $H_2$ for these maps. We proved that $H_2$ was the correct choice when the system of arcs was alpha-parallel. We observe that the above model of $\bar \cH_{\Lambda}^{\cL_2}$ corresponds exactly to the other choice of homotopy.
\end{rem}

\section{Minimal models for the Hopf link surgery complex}

\label{sec:minimal-model-Hopf}

In this section, we consider the link surgery complex for the Hopf link as a $DA$-bimodule. We do this via a tensor product
\[
{}_{\cK} \cH_{\Lambda}^{\cK}:=\cH_{\Lambda}^{\cK\otimes \cK}\hatbox {}_{\cK| \cK}[\bI^{\Supset}].
\]
In the above, just one algebra output of $\cH_{\Lambda}^{\cK\otimes \cK}$ is input into ${}_{\cK| \cK}[\bI^{\Supset}]$. 

We note that the number of generators in the complex $\cH^{\cK\otimes \cK}_{\Lambda}$ can be reduced by homotopy equivalence. To this end, we define the maximal ideal $\frm\subset \cK$, to be the one which is generated by $\scU,\scV\in \ve{I}_0\cdot \cK \cdot \ve{I}_0$, $\sigma,\tau\in \ve{I}_1\cdot \cK\cdot \ve{I}_0$ and $\scU\in \ve{I}_1\cdot \cK\cdot \ve{I}_1$. We make the following definition (compare \cite{EisenbudCommutativeAlgebra}*{pg. 476}):

\begin{define} 
\label{def:minimal}
We say that a type-$DA$ bimodule ${}_{\cK} \cX^{\cK}$ is \emph{minimal} if $\delta_1^1\colon \cX\to \cX\otimes \cK/\frm$ is zero.
\end{define}

In the Heegaard Floer literature, one frequently uses the term \emph{reduced} instead of minimal. In this section, we will explore a minimal model of the Hopf link complex ${}_{\cK} \cH_{\Lambda}^{\cK}$, for which we will write ${}_{\cK} \cZ^{\cK}$. We will also compute a minimal model of the complex ${}_{\cK} \bar \cH_{\Lambda}^{\cK}$ obtained by using both alpha and beta parallel arcs on the Hopf link.

\subsection{$DA$-perspectives on the Hopf link complex}
\label{sec:perspectives}
In this section, we describe how to view the Hopf link complex as a type-$DA$ module, and we compute a minimal model.  This perspective was first explored in \cite{HHSZDuals}. 

Write $H=L_1\cup L_2$ for the Hopf link.  Write $P_1=\bF[\scU_1,\scV_1]$ and $P_2=\bF[\scU_2,\scV_2]$, where $P_i$ contains the variables for $L_i$.  Let us write $\cC$ for $\cCFL(H)$.

We may view $\cC$ is as a type-$D$ module over $P_2$. We do this by viewing $\cC^{P_2}$ as freely generated by $\ve{a}$, $\ve{b}$, $\ve{c}$ and $\ve{d}$ over $P_1$. The structure map $\delta^1$ is shown below:
 \begin{equation}
\cC^{P_2}\iso \begin{tikzcd}[labels=description,row sep=1.4cm, column sep=1.4cm] \ve{a}[\scU_1,\scV_1]& \ve{b}[\scU_1,\scV_1]\ar[l,dashed, "1|\scU_2"] \ar[d,dashed, "\scU_1|1"]\\
\ve{c}[\scU_1,\scV_1]\ar[r,dashed, "1|\scV_2"] \ar[u,dashed, "\scV_1|1"]& \ve{d}[\scU_1,\scV_1].
\end{tikzcd}
\label{eq:Hopf-type-D}
\end{equation}
 An arrow from $\ve{x}$ to $\ve{y}$ labeled by $a|b$ indicates that $\delta^1(\xs)$ has a summand of $(a\cdot \ys)\otimes b$. 

As a type-$D$ module, $\cC^{P_2}$ is not minimal. We now describe a homotopy equivalent complex which is minimal. It is helpful to realize that the above description of $\cC^{P_2}$ may be further decomposed as a box tensor product, as follows. Let $\cE_2$ denote the exterior algebra on two generators
\[
\cE_2:=\Lambda^*(\phi_2,\psi_2).
\]
We consider the following type-$DD$ bimodule ${}^{\cE_2}\scK^{P_2}$. As a vector space, $\scK\iso \bF$. The type-$DD$ structure map is given by the formula
\[
\delta^{1,1}(1)=\phi_2|1|\scU_2+\psi_2|1|\scV_2.
\]
We may define a type-$A$ module $\cC_{\cE_2}$, generated freely over $P_1$ by $\ve{a}$, $\ve{b}$, $\ve{c}$ and $\ve{d}$, with module structure given by the following diagram
\[
\cC_{\cE_2}=\begin{tikzcd}[labels=description,row sep=1.4cm, column sep=1.4cm] \ve{a}[\scU_1,\scV_1]& \ve{b}[\scU_1,\scV_1]\ar[l, "\phi_2"] \ar[d,dashed, "\scU_1"]\\
\ve{c}[\scU_1,\scV_1]\ar[r, "\psi_2"] \ar[u,dashed, "\scV_1"]& \ve{d}[\scU_1,\scV_1].
\end{tikzcd}
\]
In the above, solid arrows denote the $m_2$ actions of $\cE_2$, while the dashed arrows denote the $m_1$ actions (i.e. internal differentials). To illustrate the notation, if $f\in \bF[\scU_1,\scV_1]$, then 
\[
m_2(f\cdot \ve{b}, \phi_2)=f\cdot \ve{a}\quad \text{and} \quad m_1(f\cdot\ve{b})=\scU_1\cdot f\cdot \ve{d}.
\]
 The action of $\psi_2$ is similar.

Clearly, there is an isomorphism of type-$D$ modules
\[
\cC^{P_2}\iso \cC_{\cE_2}\boxtimes {}^{\cE_2}\scK^{P_2}.
\]

 In the above, we have forgotten about the natural action of $P_1$, given by ordinary multiplication. In fact, we may incorporate additionally this action to obtain bimodules ${}_{P_1}\cC^{P_2}$ and ${}_{P_1}\cC_{\cE_2}$, such that
\[
{}_{P_1}\cC^{P_2}\iso {}_{P_1}\cC_{\cE_2}\boxtimes {}^{\cE_2}\scK^{P_2}.
\]
The module ${}_{P_1} \cC^{P_2}$ has only $\delta_1^1$ and $\delta_2^1$ non-trivial, and ${}_{P_1}\cC_{\cE_2}$ has only $m_{0,1,0}$, $m_{1,1,0}$ and $m_{0,1,1}$ non-trivial. In the next section, we explore minimal models.

\subsection{Minimal models for the Hopf link Floer complex}
\label{sec:minimal-models-Hopf-link}

We apply the homological perturbation lemma to obtain minimal models of the Hopf link bimodules from the previous section. We will define a $DA$-bimodule ${}_{P_1} Z^{P_2}$ which is homotopy equivalent to ${}_{P_1} \cC^{P_2}$ and which has $\delta_1^1=0$. The techniques of this section formalize the construction in \cite{HHSZDuals}.

First define the chain complex $\cC_0$ by forgetting about the $\cE_2$-action action on $\cC_{\cE_2}$. Over $\bF_2$, $\cC_0$ is homotopy equivalent to its homology, which we view as the vector space
\[
Z:= \ve{a} [\scU_1]\oplus  \ve{d} [\scV_1].
\]
Here, $ \ve{a}[\scU_1]$ denotes a copy of $\bF[\scU_1]$, generated by $\ve{a}$, and similarly for $ \ve{d} [\scV_1]$. There are natural maps
\[
i\colon Z\to \cC_0\quad \pi\colon \cC_0\to Z, \quad \text{and} \quad h\colon \cC_0\to \cC_0
\]
such that $i$ and $\pi$ are chain maps, $i\circ \pi=\id+[m_1,h]$, $\pi\circ i=\id$, $h\circ i=0$, $\pi\circ h=0$ and $h\circ h=0$. The map $i$ is given by 
\[
i(\scU_1^n\ve{a})=\scU_1^n \ve{a}\quad \text{and} \quad i(\scV_1^n\ve{d})=\scV_1^n \ve{d}.
\]
The map $\pi$ is given by
\[
\pi(\scU_1^n \scV_1^m\ve{a})=\begin{cases}\scU_1^n \ve{a}& \text{ if } m=0\\
0& \text{ otherwise},
\end{cases}
\quad \text{and} \quad
\pi(\scU_1^n \scV_1^m\ve{d})=\begin{cases}\scV_1^m \ve{d}& \text{ if } n=0\\
0& \text{ otherwise}.
\end{cases}
\]
 The map $h$ is given by
\[
h(\scU_1^n\scV_1^m\ve{a})=\begin{cases} \scU_1^n \scV_1^{m-1}\ve{c}  &\text{if } m\ge 1\\
0 & \text{otherwise,}
\end{cases}
\quad \text{and} \quad h(\scU_1^n\scV_1^m\ve{d})=\begin{cases} \scU_1^{n-1} \scV_1^{m}\ve{b}  &\text{if } n\ge 1\\
0 & \text{otherwise.}
\end{cases}
\]
The homological perturbation lemma for $A_\infty$-modules, Lemma~\ref{lem:homological-perturbation-modules}, endows $Z$ with the structure of a right $A_\infty$-module over the exterior algebra $\cE_2$, for which we write $Z_{\cE_2}$, which is $A_\infty$-homotopy equivalent to $\cC_{\cE_2}$. In fact, the homotopy equivalence is given explicitly by the homological perturbation lemma. The maps $i$, $\pi$ and $h$ extend to $A_\infty$-module morphisms $i_*$, $\pi_*$ and $h_*$. We box these morphisms with the identity map on ${}^{\cE_2} \scK^{P_2}$ to obtain maps of type-$D$ modules
\[
\Pi^1\colon \cC^{P_2}\to Z^{P_2}, \quad I^1\colon Z^{P_2}\to \cC^{P_2}\quad \text{and} \quad H^1\colon \cC^{P_2}\to \cC^{P_2}.
\]

\begin{lem}
\label{lem:homological-perturbation-type-Z-D} The morphisms $I^1$ and $\Pi^1$ are type-$D$ homomorphisms (i.e. $\d_{\Mor}(I^1)=0$ and $\d_{\Mor}(\Pi^1)=0$). Furthermore
\begin{enumerate}
\item\label{lem:DA-comp-claim-1} $\Pi^1\circ I^1=\id$
\item\label{lem:DA-comp-claim-2} $H^1\circ H^1=0$.
\item\label{lem:DA-comp-claim-3} $H^1\circ I^1=0$.
\item\label{lem:DA-comp-claim-4} $\Pi^1\circ H^1=0$. 
\item\label{lem:DA-comp-claim-5} $I^1\circ \Pi^1=\id+\d_{\Mor}(H^1)$.
\end{enumerate}
\end{lem}
\begin{proof}Lemma~\ref{lem:homological-perturbation-modules} (the homological perturbation lemma) implies that the stated formulas hold for $i_*$, $h_*$ and $\pi_*$. Boxing with the identity map preserves these relations since the algebra is $\cE_2$ (which is an associative algebra), so
\[
(f_*\boxtimes \bI_{\scK})\circ (g_*\boxtimes \bI_{\scK})=((f_*\circ g_*)\boxtimes \bI_{\scK}),
\]
by \cite{LOTBimodules}*{Remark~2.2.28}. Furthermore boxing with the identity is a chain map:
\[
\d_{\Mor}(f_*\boxtimes \bI_{\scK})=(\d_{\Mor}(f_*)\boxtimes \bI_{\scK}).
\]
See \cite{LOTBimodules}*{Lemma~2.3.3 (1)}. 
\end{proof}

It is enlightening to compute concrete formulas for the morphisms $\Pi^1$, $H^1$ and $I^1$:
\begin{lem}\label{lem:type-D-comp-v1}
\,
\begin{enumerate}
\item $I^1$ is given by $I^1(\xs)=\xs\otimes 1$, for $\ve{x}\in Z$.
\item $\Pi^1$ vanishes on $\Span(\ve{b},\ve{c})$, and satisfies:
\[
\begin{split}
\Pi^1(\scU_1^i \scV_1^j \ve{a})&=\begin{cases}
\scU_1^{i-j} \ve{a}\otimes \scU_2^{j}\scV_2^{j}& \text{if }i\ge j\\
\scV_1^{j-i-1} \ve{d} \otimes \scU_2^{i} \scV_2^{i+1}& \text{if } i<j
\end{cases}
\\
\Pi^1(\scU_1^i \scV_1^j \ve{d})&=\begin{cases} \scU_1^{i-j-1}\ve{a}\otimes \scU_2^{j+1} \scV_2^{j}& \text{if } i>j\\
\scV_1^{j-i}\ve{d}\otimes \scU_2^{i}\scV_2^i& \text{if } i \le j.
\end{cases}
\end{split}
\]
\item $H^1$ vanishes on $\Span(\ve{b},\ve{c})$, and maps $\Span(\ve{a},\ve{d})$ to $\Span(\ve{b},\ve{c})\otimes P_2$. It is given by the formula
\[
H^1(\scU_1^i\scV_1^j \ve{a})=\begin{cases}\scU_1^i \scV_1^{j-1} \ve{c}\otimes 1+ \scU_1^{i-1} \scV_1^{j-1}\ve{b}\otimes \scV_2+\cdots+
\scU_1^{i-j} \ve{b} \otimes \scU_2^{j-1} \scV_2^j& \text{if } j\le i\\
\scU_1^i \scV_1^{j-1} \ve{c}\otimes 1+\scU_1^{i-1} \scV_1^{j-1}\ve{b}\otimes \scV_2+\cdots +\scV_1^{j-i-1} \ve{c}\otimes \scU_2^{i} \scV_2^{i}& \text{if } j>i.
\end{cases}
\]
and
\[
H^1(\scU_1^i \scV_1^j \ve{d})=
\begin{cases}
\scU_1^{i-1} \scV_1^j\ve{b}\otimes 1+\scU_1^{i-1} \scV_1^{j-1} \ve{c}\otimes \scU_2+\cdots + \scU_1^{i-j-1} \ve{b}\otimes \scU_2^j \scV_2^j& \text{ if }i> j\\
\scU_1^{i-1} \scV_1^j\ve{b}\otimes 1+\scU_1^{i-1} \scV_1^{j-1} \ve{c}\otimes \scU_2+\cdots +\scV_1^{j-i}\ve{c}\otimes \scU_2^{i}\scV_2^{i-1}& \text{ if } i\le j.
\end{cases}
\]
\end{enumerate}
\end{lem}
\begin{proof} Lemma~\ref{lem:homological-perturbation-modules} gives a concrete formula for the $A_\infty$-action on $Z_{\cE_2}$. We will prove the first statement, and also compute $\Pi^1(\scU_1^i\scV_1^j \ve{a})$ when $i\ge j$ in order to illustrate the technique. We will leave the remaining cases to the reader.

Firstly, we begin with $I^1$. We are boxing with $\bI_{\scK}$, so suppose that $a_1|\cdots| a_n|1|b_n|\cdots |b_1$ is the output of repeated applications of $\delta^{1,1}$ on ${}^{\cE_2}\scK^{P_2}$. The homological perturbation lemma gives a recipe for $m_{n+1}^Z(\ve{x},a_1,\dots, a_n)$. See Figure~\ref{fig:homological-perturbation}. The recipe is to include $\xs$ into $\cC$ via $i$. We then apply $m_2(-,a_1)$, then $h$, then $m_2(-,a_2)$, then $h$, and so forth, until one applies $m(-,a_n)$. Then we apply $\pi$. The algebra output is the product $b_n\cdots b_1$. The $m_2$ action of $\cE_2$ on $\cC$ vanishes on the image of $i$, so there are no contributions unless $n=0$. The formula follows.

We now compute $\Pi^1(\scU_1^i \scV_1^j \ve{a})$ for $i\ge j$. We begin at $\scU_1^i \scV_1^j \ve{a}\in \cC$. If $j=0$, then $h(\scU_1^i \scV_1^j\ve{a})=0$, so our only option is to apply $\pi$, which gives $\scU_1^i \ve{a}$ as claimed. If $j>0$, the only option is to apply $h$ to get $\scU_1^i \scV_1^{j-1}\ve{c}$. Then we apply $m_2(-,\psi_2)$ which gives $\scU_1^i \scV_1^{j-1}\ve{d}\otimes \scV_2$. We then apply $h$ and $m_2(-,\phi_2)$ and we get $\scU_1^{i-1} \scV_1^{j-1}\ve{a}\otimes \scU_2\scV_2$. We repeat this procedure until we cannot apply $h$ or $m_2$ any more. In the case that $i\ge j$, the final term in this sequence will be $\scU_1^{i-j} \ve{a}\otimes \scU_2^{j}\scV_2^{j}$. All of the remaining claims follow a similar analysis.
\end{proof}

Lemma~\ref{lem:homological-perturbation-type-Z-D} allows us to apply the homological perturbation lemma of $DA$-bimodules, Lemma~\ref{lem:homological-perturbation-DA-modules}, which equips $Z$ with a $DA$-bimodule structure ${}_{P_1} Z^{P_2}$, and supplies morphisms of $DA$-bimodules
\[
\Pi_*^1\colon {}_{P_1} \cC^{P_2}\to {}_{P_1} Z^{P_2}, \quad I_*^1 \colon {}_{P_1} Z^{P_2}\to {}_{P_1} \cC^{P_2} \quad \text{and} \quad H_*^1\colon {}_{P_1} \cC^{P_2}\to {}_{P_1} \cC^{P_2}
\] 
which satisfy relations identical to Lemma~\ref{lem:homological-perturbation-type-Z-D}.

\begin{lem}\, \label{lem:Hopf-reduction-properties}
\begin{enumerate}
\item The structure maps $\delta_j^1$ on ${}_{P_1} Z^{P_2}$ vanish if $j\neq 2$.
\item The maps $\Pi_j^1$ and $H_j^1$ vanish unless $j=1$.
\item The map $I_j^1$ vanishes if $j>2$.
\item The maps $\delta_*^1$ and $I_*^1$ are strictly unital i.e. they vanish if $1$ is an input, except for $\delta_2^1$, which satisfies $\delta_2^1(1\otimes x)=x\otimes 1$. (Note that $H_j^1$ and $\Pi_j^1$ are trivially unital, since they are only non-trivial if $j=1$).
\end{enumerate}
\end{lem}

\begin{proof} The proofs of all statements are by explicit examination of the maps from the homological perturbation lemma.

We begin with the statements about $\delta_j^1$. Consider first a sequence of algebra elements $(a_n,\dots, a_1)$ in $P_1$. The rule for computing $\delta_j^1$ is to input via $I^1$, then apply $m_2$, then we apply pairs of $H^1$ followed by $m_2$ until we exhaust $(a_n,\dots, a_1)$, and then finally we apply $\Pi_1$. The algebra elements output are multiplied together by applying $\mu_2$ repeatedly. However applying $m_2$ does not move an elements position in $\cC$ (i.e. $\ve{a}$, $\ve{b}$, $\ve{c}$ or $\ve{d}$). The map $H^1$ does change the generator, and it maps $\Span(\ve{a},\ve{d})$ to $\Span(\ve{c},\ve{b})\otimes P_2$, and vanishes on $\Span(\ve{c},\ve{b})$. If we apply another $m_2$, we remain in $\Span(\ve{c},\ve{b})$. We note that $\Pi^1$ and $H^1$ both vanish on $\Span(\ve{c},\ve{b})$. In particular, the only terms making non-trivial contribution are those with one $\delta_2^1$, and no $H^1$ term. The same argument shows that $\Pi_k^1$ and $H_k^1$ vanish if $k>1$. 

To compute the map $I_{j+1}^1(a_j,\dots, a_1, \xs)$, the recipe is to first include $\xs$ into $\cC\otimes P_2$ via $I^1$. This includes $\xs$ into $\Span(\ve{a},\ve{d})\otimes P_2$. We then apply $\delta_2^1(a_1,-)$. This preserves $\Span(\ve{a},\ve{d})\otimes P_2$. Next we apply $H^1$, which maps $\Span(\ve{a},\ve{d})\otimes P_2$ to $\Span(\ve{b},\ve{c})\otimes P_2$. Any further applications of  $\delta_2^1(a_n,-)$ followed by $H^1$ would map to zero. Hence we may have $I_2^1$, but not $I_j^1$ for $j>2$.

The final statement about being strictly unital follows from the homological perturbation lemma. 
\end{proof}

It is helpful to explicitly compute $\delta_2^1$ on ${}_{P_1} Z^{P_2}$. Since there is no $\delta_j^1$ for $j>2$ by Lemma ~\ref{lem:Hopf-reduction-properties}, we compute only $\delta_2^1(\scU_1,-)$ and $\delta_2^1(\scV_1,-)$.

\begin{lem}
\label{lem:minimal-model-B-1} The map $\delta_2^1$ on ${}_{P_1} Z^{P_2}$ satisfies the following.
\begin{enumerate}
\item $\delta_2^1(\scU_1,\scU_1^n\ve{a})=\scU_1^{n+1}\ve{a}\otimes 1$.
\item $
\delta_2^1(\scV_1,\scU_1^n\ve{a})=\begin{cases}
\scU_1^{n-1}\ve{a}\otimes \scU_2\scV_2& \text{ if } n>0\\
\ve{d}\otimes \scV_2& \text{ if } n=0.
\end{cases}$
\item $\delta_2^1(\scV_1, \scV_1^m \ve{d})=\scV_1^{m+1} \ve{d}\otimes 1$.
\item $
\delta_2^1(\scU_1, \scV_1^m \ve{d})=\begin{cases}\scV_1^{m-1}\ve{d}\otimes \scU_2 \scV_2 &\text{ if } m>0\\
\ve{a}\otimes \scU_2& \text{ if } m=0.
\end{cases}$
\end{enumerate}
\end{lem}

Finally, the complex $\scV_1^{-1} \cC$ will also be important to understand. Similarly to $\cC$, this the complex $\scV_1^{-1} \cC$ can be reduced in size via a homotopy equivalence. We note that $\scV_1^{-1}\cC$ is the $DA$-bimodule
\[
\scV_1^{-1}\cC\iso
 \begin{tikzcd}[labels=description,row sep=1.4cm, column sep=1.4cm] 
 {\ve{a}[\scU_1,\scV_1,\scV_1^{-1}]}
& {\ve{b}[\scU_1,\scV_1,\scV_1^{-1}]}
	\ar[l, "1|\scU_2"]
	\ar[d, "\scU_1|1"]
 \\
{\ve{c}[\scU_1,\scV_1,\scV_1^{-1}]}
	 \ar[r,"1|\scV_2"]
	 \ar[u, "\scV_1|1"]
& 
{\ve{d}[\scU_1,\scV_1,\scV_1^{-1}]}.
\end{tikzcd}
\]
In the above, all arrows denote $\delta_1^1$. The actions of $\delta_2^1$ are given by ordinary polynomial multiplication. Since $\scV_1$ is invertible, the arrow labeled $\scV_1|1$ and the two generators $\ve{a}$ and $\ve{c}$ may be completely canceled from the complex. Hence the above is homotopy equivalent to a $DA$-bimodule whose underlying type-$D$ structure is shown below:
\[
\begin{tikzcd}\ve{b}[\scU_1,\scV_1,\scV_1^{-1}]\ar[r, "\scU_1|1"]&\ve{d}[\scU_1,\scV_1,\scV_1^{-1}] \end{tikzcd}.
\]
This complex above may be further reduced as a type-$D$ structure, giving a minimal model, for which we write ${}_{\scV_1^{-1}P_1}W^{P_2}$.  As a vector space $W\iso \ve{d}[\scV_1,\scV_1^{-1}]$.  The same homological perturbation argument as before equips $W$ with bimodule structure ${}_{\scV_1^{-1}P_1} W^{P_2}$ which has $\delta_{j}^1=0$ unless $j=2$. For completeness we record the actions:

\begin{lem}\label{lem:minimal-model-B-2} The bimodule ${}_{\scV_1^{-1} P_1} W^{P_2}$ has the following action:
\[
\delta_2^1(\scU_1^i \scV_1^j, \scV_1^n\ve{d})=\scV_1^{j+n-i}\ve{d}\otimes \scU_2^{i}\scV_2^{i}.
\]
\end{lem}

We leave the proofs of Lemma~\ref{lem:minimal-model-B-1} and ~\ref{lem:minimal-model-B-2} to the reader.

\subsection{Minimal models over $\cK$}
\label{sec:more-minimal-Hopf}
In this section, we describe a minimal model of the $DA$-bimodule ${}_{\cK_1} \cH_{(\lambda_1,0)}^{\cK_2}$. We will denote the minimal model by
\[
{}_{\cK_1}\cZ_{(\lambda_1,0)}^{\cK_2}.
\]

The existence of the minimal model follows from the homological perturbation lemma for hypercubes of $DA$-bimodules. From the description of ${}_{\cK_1}\cH_{(\lambda_1,0)}^{\cK_2}$ it is clear that there is a filtration by the cube $\bE_2$. In Section~\ref{sec:minimal-models-Hopf-link}, we described minimal models of $\cCFL(H)^{\bF[\scU_2,\scV_2]}$ and $\scV_1^{-1} \cCFL(H)^{\bF[\scU_2,\scV_2]}$. Note that we may identify the underlying type-$D$ module $\cH_{\veps}^{\cK_2}$ with one of these type-$D$ modules, for $\veps\in \bE_2$. Hence, the construction from Section~\ref{sec:minimal-models-Hopf-link} may be viewed as giving morphisms of type-$D$ structures
\[
\Pi_{\veps}^1\colon \cH_{\veps}^{\cK_2}\to \cZ_{\veps}^{\cK_2}, \quad I_{\veps}^1 \colon \cZ_{\veps}^{\cK_2}\to \cH_{\veps}^{\cK_2}\quad \text{and} \quad H_{\veps}^1\colon \cH_{\veps}^{\cK_2}\to \cH_{\veps}^{\cK_2}
\]
which induce a homotopy equivalence of type-$D$ structures, such that furthermore the algebraic assumptions of the homological perturbation lemma are satisfied. We obtain a homotopy equivalent hypercube of $DA$-bimodules ${}_{\cK_1} \cZ_{(\lambda_1,0)}^{\cK_2}$ by applying the homological perturbation lemma for hypercubes of $DA$-bimodules, Lemma~\ref{lem:homological-perturbation-DA-hypercube}. We now schematically sketch the induced type-$D$ structure map:
\[
\delta_1^1=\begin{tikzcd}[labels=description, row sep=1.1cm,column sep=.8cm]
 \cZ_{0,0} \ar[d,"{}^{L_2}\! f_1^1+{}^{-L_2}\! f_1^1"] & \cZ_{1,0}\ar[d, "{}^{L_2}\! g_1^1+{}^{-L_2}\! g_1^1"]\\
  \cZ_{0,1}  &\cZ_{1,1} 
  \end{tikzcd}\qquad \qquad 
\delta_2^1=
\begin{tikzcd}[labels=description, row sep=1.1cm,column sep=2.3cm]
 \cZ_{0,0} \ar[loop left, "m_2^1"] \ar[r,"{}^{L_1}\! p_2^1+{}^{-L_1}\! p_2^1"] & \cZ_{1,0} \ar[loop right, "m_2^1"]
 \\
  \cZ_{0,1} \ar[r, "{}^{L_1}\! q_2^1+{}^{-L_1}\! q_2^1"]  \ar[loop left, "m_2^1"] &\cZ_{1,1} \ar[loop right, "m_2^1"] 
  \end{tikzcd} 
\]
\[
\delta_3^1=\begin{tikzcd}[labels=description, row sep=1.1cm,column sep=1cm]
 \cZ_{0,0} \ar[dr, "{}^{-H}\! \omega_3^1"]  & \cZ_{1,0}\\
  \cZ_{0,1}   &\cZ_{1,1} 
  \end{tikzcd} 
\]

\begin{prop}\label{prop:Z_lambda-bimodule} Give the negative Hopf link framing $(\lambda_1,0)$. The structure maps of the minimal model ${}_{\cK_1} \cZ_{(\lambda_1,0)}^{\cK_2}$ are as follows:
\begin{enumerate}
\item The maps $m_2^1$ are the same as the $\delta_2^1$ maps in Lemmas~\ref{lem:minimal-model-B-1} and~\ref{lem:minimal-model-B-2}.
\item \label{num:DA-module-Hopf-formulas-2}The maps ${}^{L_1}\! p_2^1$ and ${}^{L_1}\! q_2^1$ are given by the same formulas as each other, as are ${}^{-L_1}\! p_2^1$ and ${}^{-L_1}\! q_2^1$. They are determined by the following formulas:
\begin{enumerate}
\item ${}^{L_1}\!p_2^1(\sigma_1,\scU_1^i \ve{a})=\scV_1^{-i-1} \ve{d} \otimes \scU_2^i \scV_2^{i+1} $ and ${}^{L_1}\! p_2^1(\sigma_1,\scV_1^j\ve{d})=\scV_1^j \ve{d}\otimes 1$.
\item ${}^{-L_1}\! p_2^1(\tau_1,\scU_1^i\ve{a})=\scV_1^{-i-1+\lambda_1} \ve{d}\otimes 1$ and ${}^{-L_1}\! p_2^1(\tau_1,\scV_1^j \ve{d})=\scV_1^{j+\lambda_1} \ve{d}\otimes \scU_2^{j+1} \scV_2^j.$
\item ${}^{L_1} \!p_2^1(\tau_1,-)=0$ and ${}^{-L_1} \! p_2^1(\sigma_1,-)=0$.
\end{enumerate}
\item The maps for $\pm L_2$ are as follows:
\begin{enumerate}
\item ${}^{L_2}\! f_1^1(\scU_1^i\ve{a})=\scU_1^i \ve{a}\otimes \sigma_2$ and ${}^{L_2}\! f_1^1(\scV_1^j \ve{d})=\scV_1^j \ve{d}\otimes \sigma_2$.
\item ${}^{-L_2} \!f_1^1(\scU_1^i \ve{a})=\scU_1^{i+1} \ve{a} \otimes \tau_2$ and
\[
{}^{-L_2}\! f_1^1(\scV_1^j\ve{d})=
\begin{cases}
\ve{a}\otimes \scV_2^{-1}\tau_2 &\text{ if } j=0\\
\scV_1^{j-1} \ve{d}\otimes \tau_2& \text{ if } j>0.
\end{cases}
\]
\item  ${}^{L_2}\!g_1^1(\scV_1^i \ve{d})=\scV_1^i \ve{d}\otimes \sigma_2$.
\item ${}^{-L_2}\!g_1^1(\scV_1^i \ve{d})=\scV_1^{i-1}\ve{d}\otimes \tau_2.$
\end{enumerate}
\item The map ${}^{-H}\!\omega_3^1$ is determined by the relations
\[
\begin{split}
{}^{-H}\! \omega_3^1(\tau_1,\scV_1^m, \scU^i_1 \ve{a})&=\min(i+1,m)\scV_1^{\lambda_1+m-i-2}\ve{d}\otimes \scU_2^{m-1}\scV_2^{m-1}\tau_2\\
{}^{-H}\! \omega_3^1(\tau_1,\scU_1^n,\scU_1^i\ve{a})&=0
\\
{}^{-H}\! \omega_3^1(\tau_1, \scU_1^n, \scV_1^j\ve{d})&=\min(n,j)\scV_1^{j-n+\lambda_1-1}\ve{d}\otimes \scU_2^{j-1}\scV_2^{j-2}\tau_2\\
{}^{-H}\! \omega_3^1(\tau_1,\scV_1^m,\scV_1^j\ve{d})&=0,
\end{split}
\]
and that ${}^{-H}\!\omega_3^1$ vanishes if an algebra input is a multiple of $\sigma_1$. The map ${}^{-H}\! \omega_3^1$ also vanishes on pairs of algebra elements with other configurations of idempotents.
\end{enumerate}
\end{prop}

\begin{rem}
We have not enumerated a complete list of the structure maps. Additional powers of $\scU_1$ and $\scV_1$ may be added to arguments in the above maps. There are no terms of $\delta_{j+1}^1$ for $j>2$ however. A more minimal list could also have been made by specifying the length 1 maps on only $\ve{a}$ and $\ve{d}$. As an example of several relations which are forced by the $DA$ bimodule relations, we have
\[
\begin{split}
{}^{-H}\! \omega_{3}^1(\tau_1, \scU_1^{i} \scV_1^j, \scU_1^n \ve{a})&={}^{-H} \!\omega_3^1(\tau, \scV_1^j, \scU_1^{i+n}\ve{a})\quad \text{and}\\
{}^{-H}\! \omega_{3}^1(\tau_1, \scU_1^{i} \scV_1^j, \scV_1^n \ve{d})&={}^{-H}\! \omega_3^1(\tau, \scU_1^i, \scV_1^{j+n}\ve{d}) 
\end{split}
\]
\end{rem}

\begin{proof}
All of these computations are performed algorithmically using the homological perturbation lemma for hypercubes of $DA$-bimodules, Lemma~\ref{lem:homological-perturbation-DA-hypercube}. For the first two sets of equations the computation is essentially straightforward. We first apply the map $I^1_{\veps}$, then we apply a length 1 map of the cube (either with no algebra input, as for the maps labeled $f_1^1$ and $g_1^1$, or with an algebra input of $\sigma_1$ or $\tau_1$, as for the maps $p_2^1$ and $q_2^1$).  We leave these computations to the reader, as they are straightforward.

The map ${}^{-H}\! \omega_3^1$ is more interesting. There is exactly one configuration of morphism graph which gives a non-trivial evaluation. This occur for elements $\xs_{0,0}$ in idempotent $\veps=(0,0)$. The structure graph is shown below:
\begin{equation}
\begin{tikzcd}[row sep=.4cm] b\tau_1 \ar[ddddrr,bend left=4] & a \ar[ddr, bend left=8]& \ve{x}_{0,0} \ar[d]&\,\\
&&I^1\ar[d] \ar[dddddr, bend left=10]\\
&&\delta_2^1\ar[d] \ar[rdddd, bend left=10]\\
&&H^1\ar[d] \ar[rddd,bend left=11]&\,\\
&&{}^{-H} \Omega_2^1\ar[d] \ar[ddr,bend left=12]\\
&&\Pi^1\ar[dd] \ar[dr,bend left=14]\\
&&\, &\,\mu_2\ar[d]\\
&&\,&\,
\end{tikzcd}
\label{eq:homological-perturbation-omega-31}
\end{equation}
In the above, the map ${}^{-H} \Omega_2^1$ denotes the component of $\delta_2^1$ on $\cH_{\Lambda}^{\cK_1\otimes \cK_2}\boxtimes {}_{\cK_1|\cK_1} [\bI^{\Supset}]$ contributed by the map component ${}^{-H} \Omega^1$ of $\delta^1$ of $\cH_{\Lambda}^{\cK_1\otimes \cK_2}$. Concretely, it is given by the formula
\[
\Omega_2^1(b \tau_1,-)=\begin{tikzcd}[row sep=0cm]
   \scU^i_1 \scV^j_1 \ve{a}_{00}\ar[r, mapsto] &b\scV_1^{\lambda_1-2} \phi^\tau(\scU_1^i \scV_1^j)\ve{c}_{11}\otimes \scV_2^{-1}\tau_2\\
   \scU^i_1 \scV^j_1 \ve{b}_{00}\ar[r, mapsto] &0\\
    \scU^i_1 \scV^j_1 \ve{c}_{00}\ar[r, mapsto] &b\scV_1^{\lambda_1-1}\phi^\tau(\scU_1^i \scV_1^j) \ve{d}_{11}\otimes \tau_2\\
   \scU^i_1\scV^j_1\ve{d}_{00}\ar[r, mapsto]&b\scV_1^{\lambda_1-1}\phi^\tau(\scU^i_1 \scV^j_1)\ve{c}_{11}\otimes \scV_2^{-2}\tau_2.
   \end{tikzcd}
\]

We focus on the case that $b=i_1$. Consider first the case that $\ve{x}=\scU_1^i\ve{a}$. If $a=\scU_1^n$, then $\delta_2^1(\scU_1^n,\scU_1^i\ve{a})=\scU_1^{i+n}\ve{a}\otimes 1$, however $H^1$ vanishes on this element, so we conclude that
\[
{}^{-H}\!\omega_3^1(\tau_1,\scU_1^n,\scU_1^i\ve{a})=0.
\]
 The same argument implies that ${}^{-H}\!\omega_3^1(\tau_1,\scV_1^m,\scV_1^j\ve{d})=0$.
 
We now consider the case that $\ve{x}=\scU_1^i \ve{a}$ and $a=\scV_1^m$. In this case $H^1(\scU_1^i\scV_1^m\ve{a})$ is a sum involving both $\ve{c}$ and $\ve{b}$. The next term in  Equation~\eqref{eq:homological-perturbation-omega-31} is an application of ${}^{-H} \Omega_2^1$, which vanishes on multiples of $\ve{b}$. In particular, we need only consider the terms of $H^1(\scU_1^i\scV_1^m\ve{a})$ involving $\ve{c}$. There are $\min(i+1,m)$ such terms. They are
\[
\scU_1^i \scV_1^{m-1}\ve{c}\otimes 1+ \scU_1^{i-1}\scV_1^{m-2}\ve{c}\otimes \scU_2\scV_2+\scU_1^{i-2}\scV_1^{m-3}\ve{c}\otimes \scU_2^2\scV_2^2+\cdots.
\]
(The sum is over all terms in the sequence above where $\scU_1$ and  $\scV_1$ both have nonnegative exponent). We then apply ${}^{-H} \Omega_2^1$, $\Pi^1$, and then multiply the outgoing algebra elements by repeatedly applying $\mu_2$. It is easy check that that the application of ${}^{-H} \Omega_2^1$, $\Pi^1$, and $\mu_2$ on each of the above summands coincide, so we will only consider their evaluation on $\scU_1^i \scV_1^{m-1} \ve{c}\otimes 1$. Applying ${}^{-H} \Omega_2^1$ gives
\[
 \scV_1^{\lambda_1-1}\phi^\tau(\scU_1^i \scV_1^{m-1}) \ve{d}\otimes \tau_2\otimes 1=\scU_1^{m-1} \scV_1^{2m+\lambda_1-3-i}\ve{d}\otimes \tau_2\otimes 1.
\]
Applying $\Pi^1$ to the above generator, and then multiplying the algebra elements and the coefficient $\min(i+1,m)$ gives
\[
\omega_3^1( \tau_1, \scV_1^m, \scU_1^i \ve{a})=\min(i+1,m)\scV_1^{m-i-2+\lambda_1}\ve{d}\otimes \scU_2^{m-1}\scV_2^{m-1}\tau_2,
\]
which is the stated formula.

We now consider ${}^{-H}\!\omega_3^1(\tau_1,\scU_1^n,\scV_1^j \ve{d})$. The terms of $H^1(\scU_1^n \scV_1^j \ve{d})$ which involve $\ve{c}$ are
\[
\scU_1^{n-1} \scV_1^{j-1}\ve{c}\otimes \scU_2+\scU_1^{n-2}\scV_1^{j-2}\ve{c}\otimes \scU_2^2 \scV_2+\cdots.
\]
As before, the sum contains all such elements in the sequence which have nonnegative powers of both $\scU_1$ and $\scV_1$.
There are $\min(n,j)$ terms in this sum. As before, it is sufficient to evaluate ${}^{-H} \Omega_2^1$ and $\Pi^1$ only on the first term, and then multiply the result by $\min(n,j)$. Applying ${}^{-H}\Omega_2^1$, we obtain
\[
\scV_1^{\lambda_1-1}\phi^\tau(\scU_1^{n-1}\scV_1^{j-1})  \ve{d}\otimes \tau_2\otimes \scU_2=\scU_1^{j-1} \scV_1^{2j-2+\lambda_1-n}\ve{d}\otimes \tau_2\otimes \scU_2.
\]
Applying $\Pi^1$ and $\id\otimes \mu_2$, and multiplying by the coefficient $\min (n,j)$ gives
\[
\omega_3^1(\tau_1, \scV_1^m, \scU_1^i \ve{a})=\min (n,j)\scV_1^{j-1-n+\lambda_1}\ve{d}\otimes \scU_2^{j-1}\scV_2^{j-2}\tau_2,
\]
which proves the statement.
\end{proof}

\begin{rem} Recall that we write ${}_{\cK_1} \bar \cH_{\Lambda}^{\cK_2}$ for the Hopf link complex obtained by using one alpha-parallel arc and one beta-parallel arc. The associated type-$D$ module over $\cK\otimes \cK$ is computed in Proposition~\ref{prop:bar-H-computation-type-D}. Write ${}_{\cK_1} \bar \cZ_{\Lambda}^{\cK_2}$ for the minimal model of the associated type-$DA$ bimodule. we leave it to the reader to verify that ${}_{\cK_1} \bar \cZ_{(\lambda_1,0)}^{\cK_2}$ has an identical description to ${}_{\cK_1} \cZ_{(\lambda_1,0)}^{\cK_2}$ except that we omit the $\delta_3^1$ term.
\end{rem}

We now prove that ${}_{\cK_1}\cH_{(\lambda_1,0)}^{\cK_2}$ and ${}_{\cK_1}\cZ_{(\lambda_1,0)}^{\cK_2}$ are both Alexander modules, and are furthermore homotopy equivalent.

\begin{lem}\label{lem:Hopf-minimal-continuity}\,
\begin{enumerate}
\item The $DA$-bimodule structure maps on ${}_{\cK_1} \cH_{(\lambda_1,0)}^{\cK_2}$ and ${}_{\cK_1} \cZ_{(\lambda_1,0)}^{\cK_2}$ are continuous.
\item All of the maps involved in the homotopy equivalences described in Proposition~\ref{prop:Z_lambda-bimodule} are continuous. 
\end{enumerate}
\end{lem}
\begin{proof} The $DA$-bimodule structure maps on ${}_{\cK_1} \cH_{(\lambda_1,0)}^{\cK_2}$ are continuous by virtue of the facts that the map $\delta^1$ on $\cH_{(\lambda_1,0)}^{\cK_1\otimes \cK_2}$ being continuous, since $\cH_{(\lambda_1,0)}$ is a finitely generated vector space, and that ${}_{\cK_1} \cH_{(\lambda_1,0)}^{\cK_2}$ is obtained by a tensor product of $\cH_{(\lambda_1,0)}^{\cK_1\otimes \cK_2}$ with ${}_{\cK_1| \cK_1}[\bI^{\Supset}]$.

The structure maps on ${}_{\cK_1} \cZ_{(\lambda_1,0)}^{\cK_2}$ and also the equivalence with ${}_{\cK_1}\cH_{(\lambda_1,0)}^{\cK_2}$ are given by the homological perturbation lemma. These maps are finite compositions of the maps $\Pi^1$, $H^1$, $I^1$, as well as the map $\delta_2^1$ of ${}_{\cK_1} \cH_{(\lambda_1,0)}^{\cK_2}$. Hence, it is sufficient to show that each of these maps is continuous with respect to the appropriate topology. The map $I^1$ is obviously continuous, since it is given by $\ve{x}\mapsto \ve{x}\otimes 1$.

Consider the map $\Pi^1$ applied to elements in  $\ve{I}_0\cdot\cH_{(\lambda_1,0)}\cdot \ve{I}_0$. This map is computed in Lemma~\ref{lem:type-D-comp-v1}. As an example, consider $\Pi^1(x)$ when $x= \scU^i_1 \scV^j_1\ve{a}$ for $i\ge j\ge 0$. In this case,
\[
\Pi^1(\scU^i_1 \scV^j_1\ve{a} )=\scU_1^{i-j}\ve{a} \otimes \scU_2^j \scV_2^j.
\]
Given a finite set of $S\subset \N$ and some $n\in \N$, we wish to show that all but finitely many $\scU_1^i\scV_1^j\ve{a}$ are mapped into $\Span(\scU_1^s \ve{a}\otimes \scU_2^i\scV_2^j: s\in \N\setminus S \text{ or } j\ge n\}$. This is the case, since of course there are only finitely many $i,j\ge 0$ such that $0\le j\le n$ and $i-j\in S$. A similar computation holds for the rest of $\ve{I}_0\cdot \cH_{(\lambda_1,0)}\cdot \ve{I}_0$, so $\Pi^1$ is continuous in these idempotents.  Essentially the same argument applies for the other idempotents of $\cH$.
 
 The map $H^1$ is verified to be continuous by a very similar argument.
\end{proof}

\begin{figure}[h]
\adjustbox{scale=.9}{
\begin{tikzcd}[labels=description, column sep=1.3cm, row sep=0cm]
\cdots
	\ar[r, bend left, "\scV|U"]
&\scU^2 \ve{a}
	\ar[r, bend left, "\scV|U"]
	\ar[l, bend left, "\scU|1"]
& \scU \ve{a}
	\ar[r, bend left, "\scV|U"]
	\ar[l, bend left, "\scU|1"]
&\ve{a}
	\ar[r, "\scV|\scV", bend left]
	\ar[l, "\scU|1", bend left]
&\ve{d}
	\ar[r, bend left, "\scV|1"]
	\ar[l, bend left,"\scU|\scU"]
&\scV \ve{d}
	\ar[r, bend left, "\scV|1"]
	\ar[l, bend left,"\scU|U"]
& \cdots \,\,\, \ve{E}_{01}
	\ar[l, bend left,"\scU|U"]
\\[2cm]
\cdots
	\ar[r, bend left, "\scV|U"]
&\scU^2 \ve{a}
	\ar[r, bend left, "\scV|U"]
	\ar[l, bend left, "\scU|1"]
	\ar[d, "\substack{\sigma|U^2\scV \\ \tau|1}"]
	\ar[u, gray,"\sigma"]
	\ar[ul, gray, "\tau"]
& \scU \ve{a}
	\ar[r, bend left, "\scV|U"]
	\ar[l, bend left, "\scU|1"]
	\ar[d, "\substack{\sigma|U\scV \\ \tau|1}"]
	\ar[u, gray," \sigma"]
	\ar[ul, gray, "\tau"]
&\ve{a}
	\ar[r, "\scV|\scV", bend left]
	\ar[l, "\scU|1", bend left]
	\ar[d, "\substack{\sigma|\scV\\ \tau|1}"]
	\ar[u, gray," \sigma"]
	\ar[ul, gray, "\tau"]
& \ve{d} 
	\ar[r, bend left, "\scV |1"]
	\ar[l, bend left, "\scU|\scU"]
	\ar[d, "\substack{\sigma|1 \\ \tau|\scU}"]
	\ar[u,gray,"\sigma"]
	\ar[ul, gray, "\scV^{-1}\tau"]
&\scV \ve{d}
	\ar[r, bend left, "\scV|1"]
	\ar[l, bend left,"\scU|U"]
	\ar[d, "\substack{\sigma|1 \\ \tau|U\scU}"]
	\ar[u,gray,"\sigma"]
	\ar[ul, gray, "\tau"]
&\cdots \,\,\, \ve{E}_{00}
	\ar[l, bend left,"\scU|U"]
	\ar[ul, gray, "\tau"]
\\[2cm]
\cdots
	\ar[r, bend left, "\scV|1"]
&\scV^{-3}\ve{d}
	\ar[r, bend left, "\scV|1"]
	\ar[l, bend left, "\scV^{-1}|1"]
	\ar[loop below,looseness=20, "U|U"]
& \scV^{-2}\ve{d}
	\ar[r, bend left, "\scV|1"]
	\ar[l, bend left, "\scV^{-1}|1"]
	\ar[loop below,looseness=20, "U|U"]
&\scV^{-1}\ve{d}
	\ar[r, "\scV|1", bend left]
	\ar[l, "\scV^{-1}|1", bend left]
	\ar[loop below,looseness=20, "U|U"]
&\scV^0\ve{d}
	\ar[r, bend left, "\scV|1"]
	\ar[l, bend left,"\scV^{-1}|1"]
	\ar[loop below,looseness=20, "U|U"]
&\scV^1\ve{d}
	\ar[r, bend left, "\scV|1"]
	\ar[l, bend left,"\scV^{-1}|1"]
	\ar[loop below,looseness=20, "U|U"]
& \cdots \,\,\, \ve{E}_{10}
	\ar[l, bend left,"\scV^{-1}|1"]
\end{tikzcd}
}

\vspace{1cm}

\adjustbox{scale=.9}{
\begin{tikzcd}[labels=description, column sep=1.2cm, row sep=0cm]
\cdots
	\ar[r, bend left, "\scV|U"]
&\scU^2\ve{a}
	\ar[r, bend left, "\scV|U"]
	\ar[l, bend left, "\scU|1"]
	\ar[d, "\substack{\sigma|U^2\scV\\ \tau|1 }"]
&\scU\ve{a}
	\ar[r, bend left, "\scV|U"]
	\ar[l, bend left, "\scU|1"]
	\ar[d, "\substack{\sigma|U\scV\\ \tau|1 }"]
&\ve{a}
	\ar[r, "\scV|\scV", bend left]
	\ar[l, "\scU|1", bend left]
	\ar[d, "\substack{\sigma|\scV\\ \tau|1 }"]
&\ve{d}
	\ar[r, bend left, "\scV|1"]
	\ar[l, bend left,"\scU|\scU"]
	\ar[d, "\substack{\sigma|1\\ \tau|\scU }"]
&\scV \ve{d}
	\ar[r, bend left, "\scV|1"]
	\ar[l, bend left,"\scU|U"]
	\ar[d, "\substack{\sigma|1\\ \tau|U\scU }"]
& \cdots\,\,\, \ve{E}_{01}
	\ar[l, bend left,"\scU|U"]
\\[2cm]
\cdots
	\ar[r, bend left, "\scV|1"]
&\scV^{-3}\ve{d}
	\ar[r, bend left, "\scV|1"]
	\ar[l, bend left, "\scV^{-1}|1"]
	\ar[loop below,looseness=20, "U|U"]
& \scV^{-2}\ve{d}
	\ar[r, bend left, "\scV|1"]
	\ar[l, bend left, "\scV^{-1}|1"]
	\ar[loop below,looseness=20, "U|U"]
&\scV^{-1}\ve{d}
	\ar[r, "\scV|1", bend left]
	\ar[l, "\scV^{-1}|1", bend left]
	\ar[loop below,looseness=20, "U|U"]
&\scV^0\ve{d}
	\ar[r, bend left, "\scV|1"]
	\ar[l, bend left,"\scV^{-1}|1"]
	\ar[loop below,looseness=20, "U|U"]
&\scV^{1}\ve{d}
	\ar[r, bend left, "\scV|1"]
	\ar[l, bend left,"\scV^{-1}|1"]
	\ar[loop below,looseness=20, "U|U"]
& \cdots \,\,\, \ve{E}_{11}
\ar[l, bend left,"\scV^{-1}|1"]
\\[2cm]
 \cdots
	\ar[r, bend left, "\scV|1"]
&\scV^{-3}\ve{d}
	\ar[r, bend left, "\scV|1"]
	\ar[l, bend left, "\scV^{-1}|1"]
	\ar[loop below,looseness=20, "U|U"]
	\ar[u,gray, bend right=35, "\sigma"]
	\ar[ul, "\tau",gray]
& \scV^{-2}\ve{d}
	\ar[r, bend left, "\scV|1"]
	\ar[l, bend left, "\scV^{-1}|1"]
	\ar[loop below,looseness=20, "U|U"]
	\ar[u,gray, bend right=35, "\sigma"]
	\ar[ul, "\tau",gray]
&\scV^{-1}\ve{d}
	\ar[r, "\scV|1", bend left]
	\ar[l, "\scV^{-1}|1", bend left]
	\ar[loop below,looseness=20, "U|U"]
	\ar[u,gray, bend right=35, "\sigma"]
	\ar[ul, "\tau",gray]
&\scV^{0}\ve{d}
	\ar[r, bend left, "\scV|1"]
	\ar[l, bend left,"\scV^{-1}|1"]
	\ar[loop below,looseness=20, "U|U"]
	\ar[u,gray, bend right=35, "\sigma"]
	\ar[ul, "\tau",gray]
&\scV^{1}\ve{d}
	\ar[r, bend left, "\scV|1"]
	\ar[l, bend left,"\scV^{-1}|1"]
	\ar[loop below,looseness=20, "U|U"]
	\ar[u,gray, bend right=35, "\sigma"]
	\ar[ul, "\tau",gray]
& \cdots \,\,\, \ve{E}_{10}
	\ar[l, bend left,"\scV^{-1}|1"]
	\ar[ul, "\tau",gray]
\end{tikzcd}
}
\caption{The $DA$-bimodule of the negative Hopf link ${}_{\cK} \bar{\cZ}_{(0,0)}^{\cK}$. The gray arrows denote $\delta_1^1$. Subscripts on algebra elements indicating link components are omitted. This coincides with the bimodule ${}_{\cK} \cZ_{(0,0)}^{\cK}$ except for the lack of the $\omega_3^1$ differential.}
\label{fig:Hopf-link-diagram-big}
\end{figure}

\subsection{Comparison with the Eftekhary-Hedden-Levine model}
\label{sec:comparison-EHL}

We now compare our Hopf link complex with the dual knot formulas of Hedden-Levine \cite{HeddenLevineSurgery} and Eftekhary \cite{EftekharyDuals}.

If $K\subset S^3$, we will write $\HLE_n(K)^{\bF[\scU,\scV]}$ for the complex described in the introduction of \cite{HeddenLevineSurgery}, which is a model for $\cCFL(S^3_n(K), \mu)^{\bF[\scU,\scV]}$ where $\mu$ is a dual of $K$ inside of the Dehn surgery.
\begin{prop} If $K$ is a knot in $S^3$, then there is a canonical isomorphism
\[
\HLE_n(K)^{\bF[\scU,\scV]}\iso\cX_n(K)^{\cK}\hatbox {}_{\cK} \cZ^{\bF[\scU,\scV]}.
\]
\end{prop}

 \begin{proof}[Proof sketch] Since we do not need this result for any later results, we will not spell out all details. Instead, we will sketch several important details from which the interested reader can easily work out the rest of the argument.
 
  Both complexes are mapping cone complexes with similar structures. We will abbreviate our complex by $\bX^\mu_n(K)$. We write
 \[
 \HLE_n(K)^{\bF[\scU,\scV]}\iso \Cone(v'+h_n'\colon \bA^-_{\HLE}(K)\to \bB^-_{\HLE}(K))\qquad \text{and}
 \]\[
  \bX_n^{\mu}(K)^{\bF[\scU,\scV]}\iso \Cone(v^\mu+h_n^\mu\colon \bA^\mu(K)\to \bB^{\mu}(K)).
 \]
 We will only consider the claim when $n=1$. Furthermore, we will only show that
 \begin{equation}
 \bA_{\HLE}^-(K)\iso \bA^\mu(K)
\label{eq:HLE-isomorphic-ours} \end{equation}
 as (infinitely generated) type-$D$ modules over $\bF[\scU,\scV]$. Most of the remaining details are straightforward extensions of the ideas we present.

  We follow the description of $\HLE_{+1}(K)$  given by Hedden and Levine \cite{HeddenLevineSurgery}. They focus on a version of the knot Floer complex which is denoted $\CFK^\infty(K)$. This takes the form of a free chain complex over $\bF[U,U^{-1}]$ which is filtered by $\Z\oplus \Z$. The generators are of the form $[\xs, i,j]$ where $\xs\in \bT_{\a}\cap \bT_\b$ and $A(\xs)=j-i$. The variable $U$ acts by $U\cdot [\xs,i,j]= [\xs,i-1,j-1]$. The components $i$ and $j$ are the two components of the $\Z\oplus \Z$-filtration. There is additionally a Maslov grading $\gr_{\ws}([\xs,i,j])=\gr_{\ws}(\xs)+2i$. Given such a filtered chain complex, we can recover the chain complex $\cCFK(K)$ by replacing each $\bF[U,U^{-1}]$ basis element $[\xs, i,j]$ with a single generator $\xs$, viewed as having Alexander grading $j-i$ and $\gr_{\ws}$-grading $\gr_{\ws}([\xs,i,j])-2i$.

 We will write $\bA_{\HLE}^\infty(K)$ for the infinity version of the Hedden-Levine model. (We will later reformulate this to get the version which is a type-$D$ module over $\bF[\scU,\scV]$). By definition, 
\[
\bA^\infty_{\HLE}(K)\iso \prod_{s\in \Z} A_s^\infty (K),
\]
for some finitely generated $\Z\oplus \Z$-filtered complexes $A_s^\infty(K)$, as follows. For each $\xs$ (an intersection point for a Heegaard diagram of $K$), the generators are of the form $[\xs,i,j]$
where $A(\xs)=j-i$, where $A(\xs)$ is Alexander grading from $\cCFK(K)$. We will write $\xs_{i,j}$ for $[\xs,i,j]$. They describe two filtrations $\cI$ and $\cJ$ on these generators given by the formulas
\[
\cI(\xs_{i,j})=\max(i,j-s)\qquad \cJ(\xs_{i,j})=\max(i-1,j-s)+s.
\]
Then the $\Z\oplus \Z$-filtration of $\xs_{i,j}$ is $(\cI(\xs_{i,j}), \cJ(\xs_{i,j}))$. They also define a Maslov grading 
\[
\gr_{\ws}(\xs_{i,j})=\gr_{\ws}(\xs)+2i+\frac{(2s-1)^2-1}{4}.
\]
See \cite{HeddenLevineSurgery}*{Equations~(1.5)--(1.10)}, noting that we are setting $s=s_l$, $k=1$ and $d=1$ in their formulas.

We now rewrite $\cI$ and $\cJ$ as follows:
\begin{equation}
\cI(\xs_{i,j})=i+\max(0,j-i-s)\quad \text{and} \quad \cJ(\xs_{i,j})=\max(i-j+s-1,0)+j.
\label{eq:HL-Alexander-grading}
\end{equation}
 It is helpful to consider the two cases $A(\xs)\ge s$ and $A(\xs)<s$ separately. Recalling that $A(\xs)=j-i$,  we compute using Equation~\eqref{eq:HL-Alexander-grading} that
\[
(\cI,\cJ)(\xs_{i,j})=
\begin{cases}
(j-s,j)& \text{ if } A(\xs)\ge s\\
(i,i+s-1)& \text{ if } A(\xs)<s.
\end{cases}
\]

We now modify the above algebraic generators to get a type-$D$ module over $\bF[\scU,\scV]$. We do this using the same procedure as we described to go from $\CFK^\infty(K)$ to $\cCFK^-(K)$. We write $\bA^-_{\HLE}(K):=\prod_{s\in \Z} A_s^-(K)$ for the type-$D$ module constructed from $\bA^\infty_{\HLE}(K)$ in this manner. For each generator $\xs$ in $\cCFK^-(K)$, we get a generator $\xs'_s$ of $A_s^-(K)$ with Alexander grading
\begin{equation}
A(\xs'_s)=\begin{cases} s& \text{ if } A(\xs)\ge s,\\
s-1 & \text{ if } A(\xs)<s.
\end{cases}
\label{eq:Alexander-HLE-model}
\end{equation}
We have, additionally, that
\begin{equation}
\gr_{\ws}(\xs_s')=\cG_{\ws}(\xs_{i,j})-2 \cI(\xs_{i,j})=\gr_{\ws}(\xs)+2\min(0,s-A(\xs))+\frac{(2s-1)^2-1}{2}. \label{eq:grw-grading}
\end{equation}

We now consider the powers of $\scU$ and $\scV$ which appear in the differential. Recall that in general, if $\xs$ and $\ys$ are generators of a knot Floer complex $\cCFK(K)$ (for a knot $K$ in an integer homology 3-sphere) and there is a differential from $\xs$ to $\ys$ which is weighted by $\scU^i\scV^j$, then
\begin{equation}
j-i=A(\xs)-A(\ys)\quad \text{and} \quad \gr_{\ws}(\ys)=\gr_{\ws}(\xs)-1+2i.
\label{eq:weights-differential-general}
\end{equation}

We now compare this with the type-$D$ structure $\bA^\mu(K)$.  By construction, this type-$D$ structure is the tensor product of $\cCFK(K)^{\bF[\scU,\scV]}$ with the $DA$-bimodule shown below:
 \[
 \begin{tikzcd}[labels=description,column sep=1.5cm, row sep=1.5cm]
 \cdots
 &[-1.5cm]
 \scU^2 \ve{a}
 	\ar[r, bend left=20, "\scV|U"]
 	\ar[from=r, bend left=20, "\scU|1"]
 &
 \scU \ve{a}
 	\ar[r, bend left=20, "\scV|U"]
 	\ar[from=r, bend left=20, "\scU|1"]
 &
 \ve{a}
 	\ar[r, bend left=20, "\scV|\scV"]
 	\ar[from=r, bend left=20, "\scU|\scU"]
 &
 \ve{d}
 	\ar[r, bend left=20, "\scV|1"]
 	\ar[from=r, bend left=20, "\scU|U"]
 &
  \scV \ve{d}
 	\ar[r, bend left=20, "\scV|1"]
 	\ar[from=r, bend left=20, "\scU|U"]
 &
 \scV^2 \ve{d}
 &[-1.5cm]
 \cdots
 \end{tikzcd}
 \]
 We can give $\bA^\mu(K)$ a similar description to $\bA^-(K)$. The generators of $\bA^\mu(K)$ are of the form $\xs\otimes \ve{y}$, where $\xs$ is a generator of $\cCFK^-(K)$, and $\ys$ is of the form $\scU^i \ve{a}$ or $\scV^i \ve{d}$. We define an Alexander grading on generators via the formula 
 \[
 A'(\scU^i\ve{a})=-i\quad \text{and} \quad A'(\scV^i \ve{d})=i+1.
 \]
 We then define $A^\mu_s(K)\subset \bA^\mu(K)$ to be the $\bF$-span of pairs $\xs\otimes \ys$ where
 \[
 A(\xs)+A'(\ys)=s.
 \]
 It is straightforward to see that $A^\mu_s(K)$ is a subcomplex of $\bA^\mu$. We claim that there is an isomorphism of type-$D$ modules
 \[
 A_s^\mu(K)^{\bF[\scU,\scV]}\iso A^-_s(K)^{\bF[\scU,\scV]}.
 \]
 Note that there is a canonical bijective correspondence between $\bF[\scU,\scV]$-generators of these complexes, since the generators of both complexes are bijectively identified with generators of $\cCFK(K)$. Therefore we have a canonical isomorphism of vector spaces between these modules. We claim that this isomorphism intertwines the differential. To see this, note that the differentials on both complexes are also identified with the the ordinary differential of $\cCFK(K)$, except with powers of $\scU$ and $\scV$ changed. Therefore it suffices to show that the powers of $\scU$ and $\scV$ appearing in the two differentials coincide.
 
 We consider a differential in $\cCFK(K)$ from $\xs$ to $\ys$, weighted by $a\in \bF[\scU,\scV]$. We consider its induced weight in $A_s^\mu(K)$ and $A_s'(K)$. Suppose that in $A_s^\mu(K)$, it is weighted by an algebra element $a^\mu$, and suppose that in $A_s'(K)$, it is weighted by an algebra element $a'$. Our goal is to show that $a^\mu=a'$.

 Using Equations~\eqref{eq:Alexander-HLE-model} and ~\eqref{eq:weights-differential-general}, we observe that
 \begin{equation}
 A(a')=\begin{cases} 0& \text{ if } A(\xs),A(\ys) \ge s,\\
 0& \text{ if } A(\xs),A(\ys)<s,\\
 1& \text{ if } A(\xs)\ge s \text{ and } A(\ys)<s,\\
 -1& \text{ if } A(\xs)<s \text{ and } A(\ys)\ge s.
 \end{cases}
 \label{eq:cases-A(a')}
 \end{equation}
 Next, we compute using Equations~\eqref{eq:grw-grading} and~\eqref{eq:weights-differential-general} that
  \begin{equation}
  \gr_{\ws}(a')=\gr_{\ws}(\xs)-\gr_{\ws}(\ys)-1+2\min(0,s-A(\xs))-2\min(0,s-A(\ys))\label{eq:gr_w-grading-a'}
  \end{equation}
  Note that $\gr_{\ws}(a')$ and $A(a')$ uniquely determine $a'$.

  Next, we compute $a^\mu$ and show that it is equal to $a'$. It is helpful to break the argument into four cases, parallel to Equation~\eqref{eq:cases-A(a')}. We will consider two of the four cases, and leave the rest to the reader.
  
   We consider first the case that $A(\xs), A(\ys)\ge 0$. The corresponding generators of $A_s^\mu(K)$ are of the form
   \[
   \xs\otimes \scU^{A(\xs)-s}\ve{a}\quad \text{and} \quad \ys\otimes \scU^{A(\ys)-s}\ve{a}.
   \]
    We can write
   \[
   a=U^{-\gr_{\ws}(a)/2} \scV^{A(a)}=U^{-(\gr_{\ws}(\xs)-\gr_{\ws}(\ys)-1)/2} \scV^{A(\xs)-A(\ys)}.
   \]
   Note that 
   \[
   \delta_2^1(a,\scU^{A(\xs)-s}\ve{a})=\scU^{A(\ys)-s}\ve{a}\otimes U^{-(\gr_{\ws}(\xs)-\gr_{\ws}(\ys)-1)/2+A(\xs)-A(\ys)}. 
   \]
   (Some care must be taken to verify the above formula when $A(\xs)-A(\ys)$ is negative, but the formula holds regardless of the sign of $A(\xs)-A(\ys)$). Hence
   \[
   a^\mu=U^{-(\gr_{\ws}(\xs)-\gr_{\ws}(\ys)-1)/2+A(\xs)-A(\ys)}.
   \]
   We observe that $A(a^\mu)=0=A(a')$. Also
   \[
   \gr_{\ws}(a')=\gr_{\ws}(\xs)-\gr_{\ws}(\ys)-1+2(A(\ys)-A(\xs)),
   \]
   which is the same as $\gr_{\ws}(a^\mu).$
   
   We now consider the third case in Equation~\eqref{eq:cases-A(a')}, where $A(\xs)\ge s$ and $A(\ys)<s$. Since $A(\xs)\ge s$, the corresponding generator of $A_s^\mu(K)$ is of the form $\xs\otimes \scU^{A(\xs)-s}\ve{a}$. Similarly the generator of $A_s^\mu(K)$ corresponding to $\ys$ will be of the form $\ys\otimes \scV^{s-A(\ys)-1}\ve{d}$. We write
   \[
   a=U^{-(\gr_{\ws}(\xs)-\gr_{\ws}(\ys)-1)/2} \scV^{A(\xs)-A(\ys)}.
   \]
   In this case, we have $A(\xs)-A(\ys)>0$. Therefore the we can read $a^\mu$ as follows. Firstly, we have a factor of $U^{-(\gr_{\ws}(\xs)-\gr_{\ws}(\ys)-1)/2}$. Then, we have an additional factor which is obtained by composing all of the right moving arrows from $\scU^{A(\xs)-s}\ve{a}$ to $\scV^{s-A(\ys)-1}\ve{d}$. This will give us factors of $U^{A(\xs)-s} \scV$. Therefore
   \[
   a^\mu=U^{-(\gr_{\ws}(\xs)-\gr_{\ws}(\ys)-1)/2+A(\xs)-s}\scV.
   \]
   This must coincide with $a'$, since it has the same $(\gr_{\ws},A)$-bigrading by Equations~\eqref{eq:cases-A(a')} and~\eqref{eq:gr_w-grading-a'}.
   
   It remains to verify that $a'=a^\mu$ in the cases that $A(\xs),A(\ys)<s$ and when $A(\xs)<s$ and $A(\ys)\ge s$. The argument follows from the same line of reasoning as the two previously analyzed cases. We leave the details to the reader.
   
   The analysis of $\bB^\mu(K)$ and the maps $v^\mu$ and $h_n^\mu$ follows from similar, albeit somewhat tedious, reasoning, so we leave the details to the interested reader.
    \end{proof}

 \section{Examples and basic properties}
 In this section, we perform several example computations. In Section~\ref{sec:blow-ups} we compute the effect of adding a $\pm 1$ framed meridian to a link component. We show that it corresponds to a simple bimodule. In Section~\ref{sec:solid-tori}, we compute the type-$D$ module for a $p/q$-framed solid torus. We show that the type-$D$ module for a $p/q$-framed solid torus recovers the rational surgeries formula of Ozsv\'{a}th and Szab\'{o} \cite{OSRationalSurgeries}. Additionally, we show that the type-$D$ modules for solid tori recovers the surgery exact triangle by exhibiting a homotopy equivalence
 \[
 \cD^\cK_{\infty}\simeq \Cone(f^1\colon \cD^{\cK}_{n}\to \cD^{\cK}_{n+1})
 \]
 for all $n\in \Z$. 
 
 \subsection{Meridional Dehn twists}
 \label{sec:blow-ups}
 
 In this section, we consider the effect of adding a $\pm 1$ framed meridian to a link component. Adding a meridian may be encoded by taking the connected sum with a Hopf link, which has a predictable algebraic effect.
 
 We first define our algebraic candidate, denoted ${}_{\cK}\cB_{\pm 1}^{\cK}$. These are the bimodules for two very simply algebra endomorphisms of $\cK$.   As an $(\ve{I},\ve{I})$-module, $\cB_{\pm 1}\iso \ve{I}$ with the natural $\ve{I}$-action. The structure map $\delta_1^1$ vanishes. If $a\in \ve{I}_j\cdot \cK\cdot \ve{I}_j$ for either $j=0$ or $j=1$, then we set $\delta_2^1(a, i)=i\otimes a$ for $i\in \ve{I}$.
   Additionally, we set
  \[
  \delta_2^1(\sigma,i_0)=i_1\otimes \sigma\quad \text{and} \quad \delta_2^1(\tau,i_0)=i_1\otimes \scV^{\pm 1} \tau.
  \]

  We write ${}_{\cK}\frM_{\pm 1}^{\cK}$ for the type-$D$ module
  \begin{equation}
  {}_{\cK}\frM_{\pm 1}^{\cK}=\cH_{\Lambda(\pm 1)}^{\cK_0\otimes \cK_1}\hatbox {}_{\cK_0}\cD_0 \hatbox {}_{\cK_1|\cK}W_{\a\b, \a}^{\cK},\label{eq:meridian-bimodule}
  \end{equation}
  where $W_{\a\b, \a}$ is the pair-of-pants module from Section~\ref{sec:pair-of-pants} and
  \[
 \Lambda(\pm 1)=(\pm 1,0).
  \] In the above, $\cK_0$ and $\cK_1$ denote the algebras associated to different components of the Hopf link. We write $\cK_0$ for the algebra associated to the component of the Hopf link which becomes the new meridian (and is framed by $\pm 1$).

  By the pairing theorem and Theorem~\ref{thm:alt-pairing}, adding a $\pm 1$ framed meridian to a component $K\subset L$ and tensoring ${}_{\cK}\cD_0$ into the factor of the meridian has the effect of tensoring with the bimodule ${}_{\cK} \frM_{\pm 1}^{\cK}$. 
  
  \begin{prop}\label{prop:meridional-blow-up} Let ${}_{\cK}\frM_{\pm 1}^{\cK}$ denote the $DA$-bimodule in Equation~\eqref{eq:meridian-bimodule} for adding a $\pm 1$ framed meridian. Then
  \[
  {}_{\cK}\frM_{\pm 1}^{\cK}\simeq {}_{\cK}\cB_{\mp 1}^{\cK}.
  \]
  \end{prop}
 
 The proof is a direct computation, which we do in several steps.   Firstly, we observe the very basic isomorphism
  \[
  {}_{\cK_0} \cD_0\iso \cD_0^{\cK_2}\hatbox {}_{\cK_0| \cK_2} [\bI^{\Supset}].
  \]
 (The subscripts on different copies of $\cK$ are meant to indicate only how the tensor product is formed).
 Hence, we may manipulate several terms in Equation~\eqref{eq:meridian-bimodule} as follows:
 \begin{equation}
 \begin{split}
 \cH_{\Lambda(\pm 1)}^{\cK_0\otimes \cK_1}\hatbox {}_{\cK_0}\cD_0\simeq &\cD_0^{\cK_2} \hatbox (\cH_{\Lambda(\pm 1)}^{\cK_0\otimes \cK}\hatbox {}_{\cK_0| \cK_2} [\bI^{\Supset}])\\
 :=& \cD_0^{\cK_2}\hatbox {}_{\cK_2} \cH_{\Lambda(\pm 1)}^{\cK_1}\\
 \simeq& \cD_0^{\cK_2} \hatbox {}_{\cK_2} \cZ_{\Lambda(\pm 1)}^{\cK_1}.
 \end{split}
 \label{eq:meridional-dehn-twist-manipulation-1}
 \end{equation}
 Therefore, combining Equations~\eqref{eq:meridian-bimodule} and~\eqref{eq:meridional-dehn-twist-manipulation-1}, we see that
 \begin{equation}
 \begin{split}
 {}_{\cK} \frM_{\pm 1}^{\cK}&=\cH_{\Lambda(\pm 1)}^{\cK_0\otimes \cK_1}\hatbox {}_{\cK_0}\cD_0 \hatbox {}_{\cK_1|\cK}W_{\a\b, \a}^{\cK}\\
 &\simeq  \cD_0^{\cK_2} \hatbox {}_{\cK_2} \cZ_{\Lambda(\pm 1)}^{\cK_1} \hatbox {}_{\cK_1|\cK}W_{\a\b, \a}^{\cK}
 \end{split}
 \label{eq:meridional-dehn-twist-manipulation-3}
 \end{equation}
  On the other hand, we observe by direct computation that
 \begin{equation}
 {}_{\cK} \cB_{\mp 1}^{\cK} \simeq \cD^{\cK_0}_{\mp 1} \hatbox {}_{\cK_0| \cK} W_{\a\b, \a}^{\cK}.
 \label{eq:meridional-dehn-twist-manipulation-2}
 \end{equation}
 
Comparing Equations~\eqref{eq:meridional-dehn-twist-manipulation-3} and~\eqref{eq:meridional-dehn-twist-manipulation-2}, we observe that to prove Proposition~\ref{prop:meridional-blow-up}, it suffices to show
 \begin{equation}
 \cD_0^{\cK_1} \hatbox {}_{\cK_1} \cZ_{\Lambda(\pm 1)}^{\cK}\simeq \cD_{\mp 1}^{\cK}. 
 \label{eq:Hopf-link-related-to-D-meridional}
 \end{equation}

Let us write
 \[
 \cZ_{\Lambda}^{\cK}:= \cD_0^{\cK_1}\hatbox{}_{\cK_1} \cZ_{\Lambda(\pm 1)}^{\cK}.
 \]
 
 We note that the type-$D$ module $\cZ_{\Lambda(\pm 1)}^{\cK}$ is not minimal, in the sense of Definition~\ref{def:minimal}. Using the notation of Proposition~\ref{prop:Z_lambda-bimodule}, the complex $\cZ_{\Lambda(\pm 1)}^{\cK}$ may be written as
 \begin{equation}
 \cZ_{\Lambda(\pm 1)}^{\cK}=
 \begin{tikzcd}[row sep=1.3cm, column sep=3cm, labels=description]
 \cZ_{0,0}^{\cK}\ar[r,"{p_2^1(\sigma_1+\tau_1,-)}"] \ar[d,"f_1^1"]&\cZ_{1,0}^{\cK}\ar[d," g_1^1" ]
 \\
 \cZ_{0,1}^{\cK} \ar[r,"{q_2^1(\sigma_1+\tau_1,-)}"]& \cZ_{1,1}^{\cK}
 \end{tikzcd}
 \label{eq:Z_Lambda^K-diagram}
 \end{equation}
 In the above, we are writing $f_1^1$ for the map ${}^{L_2}\! f_1^1+{}^{-L_2}\! f_1^1$, and similarly for all of the other maps. Note that ${}^{-H}\!\omega_3^1$ does not make any contribution.

Our argument will go by way of viewing the above diagram as a mapping cone from the top line to the bottom line, and then applying the homological perturbation lemma. 
Let us write
\[
\cQ^{\cK}_{\Lambda(\pm 1),\veps}=\Cone(
\begin{tikzcd}[column sep=2.5cm] \cZ^{\cK}_{0,\veps} \ar[r, "{p_2^1(\sigma_1+\tau_1,-)}"]& \cZ^{\cK}_{1,\veps}
\end{tikzcd}).
\]
 We will write 
 \[
 J^1_{\Lambda(\pm 1)}\colon \cQ^{\cK}_{\Lambda(\pm 1);0}\to \cQ^{\cK}_{\Lambda(\pm 1);1}
 \]
  for the total morphism from the top line to the bottom line of Equation~\eqref{eq:Z_Lambda^K-diagram} (i.e. $f_1^1\oplus g_1^1$), so that $\cZ_{(\pm 1,0)}^{\cK}$ is $\Cone(J_{\Lambda(\pm 1)}^{1})$. 

For $\veps\in \{0,1\}$, write $\ve{i}_{\veps}^\cK$ for the type-$D$ module module which is concentrated in idempotent $\ve{I}_\veps$, has a single generator, and has vanishing $\delta^1$. Equation~\eqref{eq:Hopf-link-related-to-D-meridional}, and consequently Proposition~\ref{prop:meridional-blow-up}, follows from the subsequent two lemmas:

\begin{lem}\label{lem:Q-1} For $\veps\in \{0,1\}$, there are maps
\[
I^1_\veps\colon \ve{i}^{\cK}_{\veps}\to \cQ_{\Lambda(-1);\veps}^{\cK},\quad \Pi^1_\veps\colon  \cQ^{\cK}_{\Lambda(-1);\veps}\to \ve{i}_{\veps}^{\cK},\quad \text{and} \quad H^1_\veps\colon \cQ^{\cK}_{\Lambda(-1);\veps}\to \cQ^{\cK}_{\Lambda(-1);\veps}
\]
such that $I^1_\veps$ and $\Pi^1_\veps$ are cycles, and $I^1_\veps \circ \Pi^1_\veps=\id+\d_{\Mor}(H^1_\veps)$ and $\Pi^1_\veps\circ I^1_\veps=\id$. Furthermore, 
\begin{equation}
\Pi^1_1\circ J_{\Lambda(-1)}^1\circ I^1_0=(i_0\mapsto i_1\otimes (\sigma+\scV\tau)). \label{eq:Lambda(-1)-key-equation}
\end{equation}
Here $i_\veps$ denotes the generator of $\ve{i}_{\veps}^{\cK}$. In particular, $\cZ_{\Lambda(-1)}^{\cK}$ is homotopy equivalent to the solid torus module $\cD_{+ 1}^{\cK}$ defined in Section~\ref{sec:solid-torus-module}.
\end{lem} 
\begin{proof} As a first step we record the mapping cone of $\cQ^{\cK}_{\Lambda(-1);\veps}$. This is the type-$D$ module
\[
\begin{tikzcd}[labels=description, column sep={1.5 cm,between origins}, row sep=1cm]
\cdots&\scU_1^{3} \ve{a} \ar[d,"\scU_2^3\scV_2^4"]\ar[dl]
& \scU_1^2 \ve{a} \ar[d,"\scU_2^2 \scV_2^3"] \ar[dl,"1"]
& \scU_1^1 \ve{a} \ar[dl,  "1"] \ar[d,"\scU_2\scV_2^2"]
& \ve{a} \ar[d,"\scV_2"]\ar[dl, "1"] 
& \ve{d} \ar[d,"1"] \ar[dl,"\scU_2"]
& \scV_1\ve{d} \ar[d,"1"]\ar[dl, "\scU_2^2\scV_2"]
& \scV_1^2 \ve{d} \ar[d,"1"]\ar[dl,"\scU_2^3 \scV_2^2"]
& \scV_1^3\ve{d} \ar[d,"1"]\ar[dl,"\scU_2^4 \scV_2^3"]
&\cdots \ar[dl]
\\
\cdots
&\scV^{-4}_1\ve{d}
& \scV^{-3}_1\ve{d}
& \scV^{-2}_1\ve{d}
& \scV^{-1}_1\ve{d}
& \ve{d}
& \scV_1\ve{d}
& \scV_1^2 \ve{d}
& \scV_1^3\ve{d}
&\cdots
\end{tikzcd}
\]
We define the type-$D$ module $N^\cE$, where $\cE$ is the exterior algebra on one generator, $\theta$. The module $N$ has a single generator, and differential $\delta^1(1)=1\otimes \theta$. Then $\cQ^{\cK}_{\Lambda(-1),\veps}$ is obtained by boxing $N^{\cE}$ with the $DA$ bimodule over $(\cE,\cK)$ shown below:
\[
\begin{tikzcd}[labels=description, column sep={1.5 cm,between origins}, row sep=1.3cm]
\cdots&\scU_1^{3} \ve{a} \ar[d,dashed,"\theta|\scU_2^3\scV_2^4"] \ar[dl]
& \scU_1^2 \ve{a} \ar[d,dashed,"\theta|\scU_2^2 \scV_2^3"] \ar[dl,"1"]
& \scU_1^1 \ve{a} \ar[dl, "1"] \ar[d,dashed,"\theta|\scU_2\scV_2^2"]
& \ve{a} \ar[d,dashed,"\theta|\scV_2"]\ar[dl,  "1"] 
& \ve{d} \ar[d,"1"] \ar[dl,dashed,"\theta|\scU_2"]
& \scV_1\ve{d} \ar[d,"1"]\ar[dl,dashed, "\theta|\scU_2^2\scV_2"]
& \scV_1^2 \ve{d} \ar[d,"1"]\ar[dl,dashed,"\theta|\scU_2^3 \scV_2^2"]
& \scV_1^3\ve{d} \ar[d,"1"]\ar[dl,dashed,"\theta|\scU_2^4 \scV_2^3"]
&\cdots
\ar[dl]
\\
\cdots &\scV^{-4}_1\ve{d}
& \scV^{-3}_1\ve{d}
& \scV^{-2}_1\ve{d}
& \scV^{-1}_1\ve{d}
& \ve{d}
& \scV_1\ve{d}
& \scV_1^2 \ve{d}
& \scV_1^3\ve{d}
&\cdots
\end{tikzcd}
\]
In the above diagram, the solid lines denote the $\delta^1$ action, while the dashed lines denote $\delta_2^1(\theta,-)$. A solid arrow from $\xs$ to $\ys$ with a $1$ means that $\delta^1(\xs)$ has a summand of $\ys\otimes 1$. If we ignore the type-$A$ action, then the type-$D$ module is equivalent to the one spanned by $\scV_1^{-1} \ve{d}$ with vanishing $\delta^1$. There is a canonical inclusion map $i^1$ and projection map $\pi^1$. A homotopy $h^1$ is defined by moving backwards along the solid arrows. Clearly $i^1$ and $\pi^1$ are type-$D$ homomorphisms, and
\[
i^1\circ \pi^1=\id+\d_{\Mor}(h^1),\quad \pi^1\circ i^1=\id.
\]
Similarly, $\pi^1\circ h^1=0$, $h^1\circ h^1=0$, and $h^1\circ i^1=0$. By the homological perturbation lemma, Lemma~\ref{lem:homological-perturbation-DA-modules}, $i^1$, $\pi^1$ and $h^1$ extend to a homotopy equivalence of $DA$ bimodules, which we denote by $i_*^1$, $\pi_*^1$ and $h_*^1$. We define the type-$D$ module morphisms in the statement by
\[
I^1_{\veps}=\bI_{N}\boxtimes i_*^1,\quad \Pi^1_{\veps}=\bI_{N}\boxtimes \pi^1_*, \quad \text{and} \quad H^1_{\veps}=\bI_{N}\boxtimes h_*^1. 
\]
We leave it to the reader to check that the above expressions (which involve infinite sums on the completion) determine continuous morphisms.

All of the remaining claims in the statement are clear, except Equation~\eqref{eq:Lambda(-1)-key-equation}.  We compute that $I^1_0$ maps $i_0$ to $\scV^{-1}_1\ve{d}\otimes 1$. By the computation of Proposition~\ref{prop:Z_lambda-bimodule}, the map $J_{\Lambda(-1)}^1$ sends this to
\[
\scV_1^{-2}\ve{d}\otimes \tau_2+\scV_1^{-1}\ve{d}\otimes \sigma_2.
\]
Then $\Pi^1_1$ maps this to
\[
i_1\otimes \scV_2\tau_2+i_1\otimes \sigma_2.
\]
Note that the application of $\Pi^1_1$ to $\scV_1^{-2}\ve{d}\otimes \tau_2$ is slightly subtle. Indeed the recipe from the homological perturbation lemma is to move backwards along the arrow from $\ve{a}$ to $\scV_1^{-2}\ve{d}$ and move forward along the arrow from $\ve{a}$ to $\scV_1^{-1}\ve{d}$, while picking up a factor of $\scV_2$. The proof is complete.
\end{proof}

We now consider the claim for the framing $\Lambda(+1)$:

\begin{lem}\label{lem:Q+1} There are maps
\[
I^1_\veps\colon \ve{i}_\veps^{\cK}\to \cQ_{\Lambda(+1);\veps}^{\cK},\quad \Pi^1_\veps\colon  \cQ^{\cK}_{\Lambda(+1);\veps}\to  \ve{i}_{\veps}^{\cK},\quad \text{and} \quad H^1_\veps\colon \cQ^{\cK}_{\Lambda(+1);\veps}\to \cQ^{\cK}_{\Lambda(+1);\veps}
\]
such that $I^1_\veps$ and $\Pi^1_\veps$ are type-$D$ homomorphisms, and $I^1_\veps \circ \Pi^1_\veps=\id+\d_{\Mor}(H^1_\veps)$ and $\Pi^1_\veps\circ I^1_\veps=\id$. Furthermore, 
\[
\Pi^1_1\circ J_{\Lambda(+1)}^1\circ I^1_0=(i_0\mapsto i_1\otimes (\sigma+\scV^{-1}\tau)).
\]
Here $i_\veps$ denotes the generator of $\cQ_{+1,\veps}^{\cK}$. In other words, $\cZ_{+1}^{\cK}\simeq \cD_{-1}^{\cK}.$
\end{lem}
\begin{proof} The proof is very much the same as the proof with framing $-1$.  We write down the complex $\cQ_{\Lambda(+1);\veps}^{\cK}$ below:
\[
\begin{tikzcd}[labels=description, column sep={1.5 cm,between origins}, row sep=1cm]
\cdots \ar[dr]&\scU_1^{3} \ve{a} \ar[d,"\scU_2^3\scV_2^4"] \ar[dr, "1"]
& \scU_1^2 \ve{a} \ar[d,"\scU_2^2 \scV_2^3"] \ar[dr,"1"]
& \scU_1^1 \ve{a} \ar[dr,  "1"] \ar[d,"\scU_2\scV_2^2"]
& \ve{a} \ar[d,"\scV_2"]\ar[dr, "1"] 
& \ve{d} \ar[d,"1"] \ar[dr,"\scU_2"]
& \scV_1\ve{d} \ar[d,"1"]\ar[dr, "\scU_2^2\scV_2"]
& \scV_1^2 \ve{d} \ar[d,"1"]\ar[dr,"\scU_2^3 \scV_2^2"]
& \scV_1^3\ve{d} \ar[d,"1"]\ar[dr,"\scU_2^4 \scV_2^3"]
&\cdots
\\
\cdots
&\scV^{-4}_1\ve{d}
& \scV^{-3}_1\ve{d}
& \scV^{-2}_1\ve{d}
& \scV^{-1}_1\ve{d}
& \ve{d}
& \scV_1\ve{d}
& \scV_1^2 \ve{d}
& \scV_1^3\ve{d}
&\cdots
\end{tikzcd}
\]
Similar to the framing $-1$ case, we realize the above diagram as the box tensor product of $N^{\cE}$ with the following $DA$-bimodule over $(\cE,\cK)$:
\[
\begin{tikzcd}[labels=description, column sep={1.5 cm,between origins}, row sep=1.3cm]
\cdots\ar[dr]&\scU_1^{3} \ve{a} \ar[d,dashed,"\theta|\scU_2^3\scV_2^4"] \ar[dr, "1"]
& \scU_1^2 \ve{a} \ar[d,dashed,"\theta|\scU_2^2 \scV_2^3"] \ar[dr,"1"]
& \scU_1^1 \ve{a} \ar[dr,  "1"] \ar[d,dashed,"\theta|\scU_2\scV_2^2"]
& \ve{a} \ar[d,dashed,"\theta|\scV_2"]\ar[dr, "1"] 
& \ve{d} \ar[d,"1"] \ar[dr,dashed,"\theta|\scU_2"]
& \scV_1\ve{d} \ar[d,"1"]\ar[dr,dashed, "\theta|\scU_2^2\scV_2"]
& \scV_1^2 \ve{d} \ar[d,"1"]\ar[dr,dashed,"\theta|\scU_2^3 \scV_2^2"]
& \scV_1^3\ve{d} \ar[d,"1"]\ar[dr,dashed, "\theta|\scU_2^4 \scV_2^3"]
&\cdots
\\
\cdots
&\scV^{-4}_1\ve{d}
& \scV^{-3}_1\ve{d}
& \scV^{-2}_1\ve{d}
& \scV^{-1}_1\ve{d}
& \ve{d}
& \scV_1\ve{d}
& \scV_1^2 \ve{d}
& \scV_1^3\ve{d}
&\cdots
\end{tikzcd}
\]
We forget about the left $\cE$-action and simplify the type-$D$ structure. We build a homotopy equivalence with the module $\cW_{\veps}^{\cK}$, by picking type-$D$ module maps $i^1$, $\pi^1$ and $h^1$. The map $i^1$ sends $i_\veps$ to $(\ve{a}+\ve{d})\otimes 1$. The map $\pi^1$ requires a choice. It is either $(\ve{a}\mapsto i_\veps\otimes 1, \ve{d}\mapsto 0)$, or $(\ve{a}\mapsto 0, \ve{d}\mapsto i_{\veps}\otimes 1)$. The choice of $h^1$ depends on our choice of $\pi^1$. If $\pi^1$ sends $\ve{a}$ to $i_{\veps}$, then $h^1$ maps $\ve{d}$ (on the bottom row) to $\ve{d}\otimes 1$ (on the top row). If $\pi^1$ maps $\ve{d}$ to $i_{\veps}\otimes 1$, then $h^1$ maps $\ve{d}$ (bottom row) to $\ve{a}\otimes 1$ (top row). Arbitrarily, pick $\pi^1$ to map $\ve{d}$ to $i_\veps\otimes 1$. Away from the central region, the map $h^1$ just maps backwards along the arrows marked with a 1. One easily verifies that $\pi^1$, $i^1$ and $h^1$ satisfy the assumptions of the homological perturbation lemma.

The homological perturbation lemma now gives us extensions $\pi^1_*$, $i^1_*$ and $h^1_*$. Boxing these with $\bI_{N}$ gives homotopy equivalences between $\cQ^{\cK}_{\Lambda(+1);\veps}$ and $\ve{i}_{\veps}^{\cK}$ described in the statement. Note that the maps $I^1_\veps$, $\Pi^1_\veps$ and $H^1_\veps$ sometimes involve infinite sums, however it is straightforward to verify that the induced maps are continuous.

It remains to perform the computation involving $J_{\Lambda(+1)}^1$. We easily compute
\[
I^1_0(i_0)=\cdots+\scU_1 \ve{a}\otimes \scV_2+\ve{a}\otimes 1+\ve{d}\otimes 1+\scV_1\ve{d}\otimes \scU_2+\cdots
\]
 Using Part~\eqref{num:DA-module-Hopf-formulas-2} of Proposition~\ref{prop:Z_lambda-bimodule}, we compute that the map $J_{\Lambda(+1)}^1$ sends the above to
\[
\begin{split}
\cdots+\scU_1 \ve{a}\otimes \sigma_2\scV_2+\ve{a}\otimes \sigma_2 +\ve{d}\otimes \sigma_2+\scV_1\ve{d}\otimes \sigma_2 \scU_2+\cdots\\
\cdots+\scU_1^2 \ve{a}\otimes \tau_2\scV_2+\scU_1\ve{a}\otimes \tau_2+ \ve{a}\otimes \scV_2^{-1}\tau_2+\ve{d}\otimes \tau_2\scU_2+\cdots
\end{split}
\]
The map $\Pi^1_1$ evaluates the above $i_1\otimes (\sigma_2+\scV_2^{-1} \tau_2)$ (regardless of which choice we made in the construction of $\pi^1$). Note, we are using here that $\tau_2 \scU_2=\scV_2^{-1} \tau_2$. The proof is complete
\end{proof}

\subsection{The type-$D$ invariants of solid tori}

\label{sec:solid-tori}

We now describe the type-$D$ module for a solid torus with $p/q\in \Q\cup \{\infty\}$ framing. The most convenient description is to take a standard unknot complement (with standard meridian and 0-framed longitude) and perform $-q/p$ surgery to a meridian of $U$. See Figure~\ref{fig:27}. Additionally, we will consider the $\infty$-framed solid torus, which we view as being obtained by performing 0-surgery to a meridian of $U$.

\begin{figure}[ht]
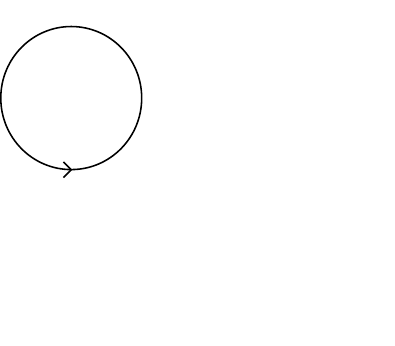
\caption{A slam dunk on a $p/q$-framed solid torus (top row), and an $\infty$-framed solid torus (bottom). The component $U$ with an arrow on it denotes the component which we associate to  $\cK$. }
\label{fig:27}
\end{figure}

We now define our candidate type-$D$ modules. We begin by recalling that by definition the type-$D$ module for an $n$-framed solid torus $\cD_n^{\cK}$ is the same as the surgery complex for an $n$-framed unknot, and hence has two generators $\xs^0$ and $\xs^1$, which live in idempotent $0$ and $1$, respectively. The structure map is given by
\[
\delta^1(\xs^0)=\xs^1\otimes (\sigma+\scV^n \tau).
\]

Generalizing this, we now define the candidate $p/q$-framed solid torus module $\cD_{p/q}^{\cK}$, where $p,q$ are coprime and $q>0$.
We define $\cD^{\cK}_{p/q}\cdot \ve{I}_\veps$ to be spanned by generators $\ve{x}_{0}^\veps,\dots, \ve{x}_{q-1}^\veps$, for $\veps\in \{0,1\}$. The structure map is given by the following formula
\[
\delta^1(\xs^0_i)=\xs^1_{(i+p) \Mod q }\otimes \scV^{\lfloor (i+p)/q \rfloor} \tau+\xs^1_i\otimes \sigma.
\]
The examples $\cD_{\pm 1/3}^{\cK}$ are shown in Figure~\ref{fig:D13}.
\begin{figure}[ht]
\[
\begin{tikzcd}[labels=description, column sep=1cm, row sep=2cm]
\ve{x}_0^0
	\ar[d, "\sigma"]
&
\ve{x}_1^0
&
\ve{x}_2^0
	\ar[d, "\sigma"]
	\ar[dll, "\scV\tau", pos=.15]
\\
\ve{x}_0^1
&
\ve{x}_1^1
	\ar[from=ul, "\tau",crossing over]
	\ar[from=u, "\sigma", crossing over, pos=.6]
&
\ve{x}_2^1
	\ar[from=ul, "\tau", crossing over]
\end{tikzcd}
\qquad\qquad
\begin{tikzcd}[labels=description, column sep=1cm, row sep=2cm]
\ve{x}_0^0
	\ar[d, "\sigma"]
	\ar[drr, "\scV^{-1}\tau", pos=.15]
&
\ve{x}_1^0
&
\ve{x}_2^0
	\ar[d, "\sigma"]
\\
\ve{x}_0^1
	\ar[from=ur, "\tau",crossing over]
&
\ve{x}_1^1
	\ar[from=ur, "\tau",crossing over]
	\ar[from=u, "\sigma", crossing over, pos=.6]
&
\ve{x}_2^1
\end{tikzcd}
\]
\caption{The modules $\cD_{1/3}^\cK$ (left) and $\cD_{-1/3}^{\cK}$ (right).}
\label{fig:D13}
\end{figure}

\begin{rem}
 The bimodule $\cD_{p/q}^{\cK}$ recovers the rational surgery mapping cone complex $\bX_{p/q}(K)$ of Ozsv\'{a}th and Szab\'{o} \cite{OSRationalSurgeries} in the sense that if $K\subset S^3$ is a knot, then
 \[
 \cD_{p/q}^{\cK}\hatbox {}_{\cK}\cX_{0}(K)\iso \bX_{p/q}(K).
 \]
\end{rem}

Finally, we define the type-$D$ module $\cD^\cK_\infty$ for the $\infty$-framed solid torus. This module $\cD_\infty^\cK\cdot \ve{I}_0=\{0\}$ and $\cD_\infty^{\cK}\cdot \ve{I}_1=\Span_{\bF}( \ve{x},\ve{y})$. The structure map is given by
\[
\delta^1(\xs)=\ys\otimes (1+\scV).
\]

The link surgery formula naturally produces a type-$D$ module $\cX_{p/q}^{\cK}$ for the $p/q$-framed solid torus, as follows.  We may write $p/q$ as a continued fraction expansion 
 \[
p/q= [a_n,\dots, a_1]^-=a_n-\frac{1}{a_{n-1}-\frac{1}{a_{n-2}-\cdots }}.
\]
We may define $\cX_{p/q}^{\cK}$ as the link surgery complex of a linear plumbing with weights $a_n,\dots, a_1$. This produces a type-$D$ module over $\cL_n$. We tensor 0-framed solid tori to the components labeled $a_{n-1},\dots, a_1$. The remaining component, weighted $a_n$, corresponds to the type-$D$ algebra action of $\cX_{p/q}^{\cK}$.

\begin{prop} The type-$D$ module for a $p/q$-framed solid torus $\cX_{p/q}^{\cK}$ is homotopy equivalent to $\cD_{p/q}^{\cK}$. 
\end{prop}
\begin{proof}
 The proof is by induction. By considering the continued fraction decomposition above, we will assume that the formula holds for $p/q$, and use this to show that it also holds for 
\[
m-\frac{1}{p/q}=m-\frac{q}{p},
\]
when $m\in \Z$.

Note that increasing the framing by $m$ may be performed by tensoring with $m$ meridional Dehn twist bimodules ${}_{\cK} \cB_{\pm 1}^{\cK}$ from Section~\ref{sec:blow-ups}. This tensor product is easy to understand, and indeed it is easy to check that
\[
\cD_{p/q}^{\cK}\hatbox {}_{\cK}\cB_{+1}^{\cK}\iso\cD_{p/q+1}^{\cK}
\]
In particular, it suffices to show that if the claim is true for some $p/q$, then it is true for $-q/p$. 

 On the algebraic level, it is sufficient to show that
\[
\cD_{-q/p}^{\cK}\simeq \cD_{p/q}^{\cK}\hatbox {}_{\cK} \cZ_{(0,0)}^{\cK},
\]
where $ {}_{\cK} \cZ_{(0,0)}^{\cK}$ is the Hopf link bimodule.

 In this case, we compute the type-$D$ module in a similar manner to our computation of the meridional Dehn twist module from Section~\ref{sec:blow-ups}. We will only consider the case that $p<0$ and $q>0$. The case that $p>0$ and $q>0$ follows from a straightforward modification of this argument.
 
   We compute $\cD^{\cK}_{p/q}\hatbox {}_{\cK} \cZ^{\cK}_{(0,0)}\cdot \ve{I}_0$ first.  The tensor product has generators 
 \[
 x_j^i:=\ve{x}_j^0\scU_1^i \ve{a}\quad \text{and} \quad y_j^i:=\ve{x}_j^0\scV_1^i \ve{d},
 \]
 ranging over $i\in \N$ and $j\in \Z/q$, and generators
  \[
 z_j^i:=\ve{x}_j^1\scV^i_1 \ve{d}
 \]
  ranging over $i\in \Z$ and $j\in \Z/q$. 
 We display the complex in Figure~\ref{fig:complex-rational-handlebody}.
 
\begin{figure}[ht]
\[
\begin{tikzcd}[column sep={.7cm,between origins}, row  sep=3cm,labels=description]
&
\cdots
&x_{2}^1
	\ar[d, "\scU\scV^2"]
&x_{3}^1
	\ar[d, "\scU\scV^2"]
	\ar[dl,dashed]
&x_{4}^1
	\ar[d, "\scU\scV^2"]
	\ar[dl,dashed]
&\,
&x_0^0
	\ar[d, "\scV"]
	\ar[dll,dashed]	
&
x_1^0
	\ar[dl,dashed]
	\ar[d,"\scV"]
&x_{2}^0
	\ar[d, "\scV"]
	\ar[dl,dashed]
&x_{3}^0
	\ar[d, "\scV"]
	\ar[dl,dashed]
&x_{4}^0
	\ar[d, "\scV"]
	\ar[dl,dashed]
&\,
&y_0^0
	\ar[d,dashed]
	\ar[dll, "\scU"]
&
y_1^0
	\ar[dl, "\scU"]
	\ar[d,dashed]
&y_{2}^0
	\ar[d,dashed]
	\ar[dl, "\scU"]
&y_{3}^0
	\ar[d,dashed]
	\ar[dl, "\scU"]
&y_{4}^0
	\ar[d,dashed]
	\ar[dl, "\scU"]
&\,
&y_0^1
	\ar[d,dashed]
	\ar[dll, "\scU^2\scV"]
&
y_1^1
	\ar[dl, "\scU^2\scV"]
	\ar[d,dashed]
&y_{2}^1
	\ar[d,dashed]
	\ar[dl, "\scU^2\scV"]
&\cdots
\\
&\cdots
&z_{2}^{-2}
&z_{3}^{-2}
&z_{4}^{-2}
&\,
&z_0^{-1}&
z_1^{-1}
&z_{2}^{-1}
&z_{3}^{-1}
&\boxed{z_{4}^{-1}} \ar[llllll, dotted, "\tau", bend left]
&\,
&z_0^{0}&
z_1^{0}
&z_{2}^0
&z_{3}^0
&z_{4}^0
&\,
&z_0^1&
z_1^1
&z_{2}^1
&\cdots
\end{tikzcd}
\]
\caption{The tensor product $\cD_{-1/5}^{\cK}\hatbox {}_{\cK} \cZ_{(0,0)}^{\cK} \cdot \ve{I}_{\veps}$, for either $\veps\in \{0,1\}$. The boxed generators are the generators of the minimal model of $\cD_{p/q}^{\cK}\hatbox {}_{\cK} \cZ_{(0,0)}^{\cK}\cdot \ve{I}_\veps$. The dashed arrows are weighted by $1$. The dotted arrow indicates a term of $\delta^1$ which maps from idempotent $0$ to idempotent 1. (Note that technically this arrow goes from the copy of this complex in idempotent 0 to the copy in idempotent 1).} 
\label{fig:complex-rational-handlebody} 
\end{figure}

Let us write $|p|=a\cdot q+r$, where $a\ge 0$ and $0\le r\le q-1$. The minimal model of $\cD^{\cK}_{p/q}\hatbox {}_{\cK} \cZ^{\cK}_{(0,0)}\cdot \ve{I}_0$  is generated by the $|p|$ generators
\[
\ws_0^0= z_{q-1}^{-1},\ws_1^0=z_{q-2}^{-1}\cdots ,\ws_{|p|-1}^0=z_{q-r}^{-a-1}.
\]
We think of the above generators as all of the generators between $z_{q-r}^{-a-1}$ and $z_{q-1}^{-1}$ if we arrange the generators in a row as in the bottom row of Figure~\ref{fig:complex-rational-handlebody}. (Recall here we are assuming that $p<0$ and $q>0$).  The same argument works for $\cD_{p/q}^{\cK}\hatbox {}_{\cK} \cZ_{(0,0)}^{\cK}\cdot \ve{I}_1$, and we write $\ws_0^1,\dots, \ws_{|p|-1}^1$ for these generators.

This gives the stated identification of the generators of the module. It remains now to understand $\delta^1$. This is obtained via homological perturbation (cf. Lemma~\ref{lem:homological-perturbation-DA-hypercube}) by first applying the inclusion map from $\cD_{-q/p}^{\cK}\cdot \ve{I}_0$ into $\cD_{p/q}^{\cK}\hatbox {}_{\cK} \cZ_{(0,0)}^{\cK}\cdot \ve{I}_0$, applying $\delta^1$ of $\cD_{p/q}^{\cK}\hatbox {}_{\cK} \cZ_{(0,0)}^{\cK}$ and then projecting  $\cD_{p/q}^{\cK}\hatbox {}_{\cK} \cZ_{(0,0)}^{\cK}\cdot \ve{I}_1$ to $\cD_{-q/p}^{\cK}\cdot \ve{I}_1$. 

The is clearly a summand of $\delta^1$ which sends $\ws_i^0$ to $\ws_i^1\otimes \sigma$, for each $i\in \Z/q$. There is a remaining term of $\delta^1$ terms which contributes $\tau$. We call that via homological perturbation theory, these terms are obtained by first including, applying $\delta^1$, and then projecting. After including, the map $\delta^1$ moves a generator of $\cD_{p/q}^{\cK}\hatbox {}_{\cK} \cZ_{(0,0)}^{\cK}$ to the left by $q$ places. The outgoing algebra element is 1. We then apply the projection map. If the application of $\delta^1$ sent $\ws_i^0$ to $\ws_j^1\otimes \tau$ for $\ws_j^1\in \{\ws_{0}^1,\dots, \ws_{p-1}^1\}$, then the projection map sends $\ws_j^1$ to $\ws_j^1\otimes 1$. In general, the projection map will involve traveling to the generators $\ws_0^1,\dots, \ws_{p-1}^1$ via a zig-zag of arrows in the larger complex $\cD_{p/q}^{\cK}\hatbox {}_{\cK} \cZ_{(0,0)}^{\cK}$ (more precisely, such a zig-zag may be encoded by the homological perturbation lemma, similar to our argument from Section~\ref{sec:blow-ups}). We pick up one power of $\scV$ for each zig-zag we have to do to get back to $\ve{w}_{(i+q)\Mod p}^1$. The number of zig-zags that we have to do to map $\delta^1(\ws_{i}^1)$ to $\ws_{(i+q)\Mod p}^1$ is exactly ${\lfloor (i+q)/(-p) \rfloor}$. This coincides with $\cD_{q/(-p)}^{\cK}$, completing the proof.
\end{proof}

\begin{prop}
 The type-$D$ module for an $\infty$-framed solid torus is $\cD_\infty^{\cK}$.
\end{prop}
\begin{proof}
We compute the tensor product 
\[
\cD_\infty^\cK:=\cD_0^{\cK}\hatbox {}_{\cK} \cZ_{(0,0)}^{\cK}.
\]
We note that for both $\veps=0,1$, we that $\cD_0\hatbox {}_{\cK} \cZ_{(0,0)}^{\cK}\cdot \ve{I}_{\veps}$ is isomorphic to the complex
\[
\begin{tikzcd}[labels=description]
 \cdots
&\scU_1 \ve{a}
	\ar[d, "1+\scU\scV^2"]
& \ve{a} 
	\ar[d, "1+\scV"]
& \ve{d}
	\ar[d, "1+\scU"]
& \scV_1\ve{d}
	\ar[d, "1+\scU^2\scV"]
&
\cdots
\\
\cdots
& \scV_1^{-2} \ve{d}
& \scV_1^{-1} \ve{d}
& \ve{d}
& \scV_1\ve{d}
&\cdots.
\end{tikzcd}
\]
In idempotent $\ve{I}_0$, this complex is acyclic. For example $\Cone(1+\scU: \ve{d}\to \ve{d})$ is acyclic because $1+\scU$ is a unit in the completion of $\ve{I}_0\cdot \cK\cdot \ve{I}_0$. Its inverse is given by $\sum_{i=0}^\infty \scU^i$. In idempotent $\ve{I}_1$, each summand is acyclic except for $\Cone(1+\scV: \ve{a}\to \scV_1^{-1} \ve{d})$. Deleting each acyclic subcomplex yields  $\cD_\infty^\cK$, completing the proof. 
\end{proof}

There is a related bimodule which is useful for applications, defined by
\[
{}_{\cK} \cD_\infty^{\bF[U]}:= \cD_\infty^{\cK} \hatbox {}_{\cK| \cK}[\bI^{\Supset}]^{\bF[U]}.
\]
The following lemma is helpful:
\begin{lem}
\label{lem:infinity-framed-type-A}
The $DA$-bimodule ${}_{\cK} \cD_\infty^{\bF[U]}$ is homotopy equivalent to the $DA$-bimodule with a single generator $\ve{z}^1$, concentrated in idempotent $\ve{I}_1$, such that $\delta_2^1(\scU^i\scV^j,\ve{z}^1)=\ve{z}^1\otimes U^i$.
\end{lem}
\begin{proof} Write ${}_{\cK}\cZ_\infty^{\bF[U]}$ for the small model in the statement. We view the completion of ${}_{\cK} \cD_\infty^{\bF[U]}$ as being the vector space
\[
\bF\llsquare \scV,\scV^{-1}\rrsquare \langle \xs^1\rangle\oplus \bF\llsquare \scV,\scV^{-1}\rrsquare \langle \ys^1\rangle.
\]
We have
\[
\delta_1^1(\scV^t \xs^1)=(1+\scV)\scV^t \ys^1\otimes 1.\]
Furthermore
\[
\delta_2^1(i_1\scV^n U^m i_1, \scV^t\xs^1)=\scV^{n+t}\xs^1\otimes U^m,
\]
and similarly for $\ys^1$. Apply the forgetful functor to the type-$A$ actions on ${}_{\cK} \cD_\infty^{\bF[U]}$ and ${}_{\cK} \cZ_{\infty}^{\bF[U]}$. Define type-$D$ morphisms 
\[
i^1\colon \cZ_\infty^{\bF[U]}\to \cD_\infty^{\bF[U]},\quad \pi^1\colon \cD_{\infty}^{\bF[U]}\to  \cZ_\infty^{\bF[U]}\quad \text{and} \quad h^1 \colon \cD_\infty^{\bF[U]}\to  \cD_\infty^{\bF[U]}
\]
via the following formulas. We set
\[
i^1(\zs^1)=\left(\sum_{i\in \Z} \scV^i\right)\xs^1\otimes 1, \quad \pi^1(\scV^i \xs^1)=\begin{cases} \zs^1\otimes 1& \text{ if } i=0\\
0& \text{ otherwise}.
\end{cases}
\]
Additionally, we set
\[
\pi^1(\scV^i \ys^1)=0 \quad \text{and} \quad h^1(\scV^i\ys^1)=
\begin{cases}\sum_{j<i} \scV^j \xs^1\otimes 1& \text{ if } i\le 0\\
\sum_{j\ge i} \scV^j\xs^1\otimes 1& \text{ if } i>0.
\end{cases}
\]
We leave it to the reader to check that the above maps are continuous, and furthermore
\[
\pi^1\circ i^1=\id,\quad i^1\circ \pi^1=\d_{\Mor}(h^1),\quad h^1\circ h^1=0,\quad \pi^1\circ h^1=0, \quad \text{and} \quad h^1\circ i^1=0.
\]
Applying the homological perturbation lemma from Lemma~\ref{lem:homological-perturbation-DA-modules} induces a type-$DA$ Alexander bimodule structure on $\cZ_\infty^{\bF[U]}$ which is homotopy equivalent to ${}_{\cK} \cD_{\infty}^{\bF[U]}$. This bimodule structure is easily checked to coincide with our definition of ${}_{\cK} \cZ_{\infty}^{\bF[U]}$, completing the proof. 
\end{proof}

\subsection{The surgery exact triangle}

\label{sec:exact-triangle}
We now show that our bordered invariants recover the surgery exact triangle, in the following sense:

\begin{prop}
\label{prop:exact-triangle} There is a type-$D$ morphism $f^1\colon \cD_n^{\cK}\to \cD_{n+1}^{\cK}$ such that
\[
\Cone(f^1\colon \cD_n^{\cK}\to \cD_{n+1}^{\cK})\simeq \cD_\infty^{\cK},
\]
for any $n\in \Z$.
\end{prop}
\begin{rem} This is unsurprising since the knot and link surgery formulas are proven using surgery exact triangles, though it highlights the similarities with the original bordered theory, cf. \cite{LOTBordered}*{Section~11.2}.
\end{rem}
\begin{rem}\label{rem:cyclic-reorder}
 It is straightforward to adapt our proof to construct morphisms and homotopy equivalences $\cD^{\cK}_{n+1}\simeq \Cone(g^1\colon \cD_\infty^\cK\to \cD_n^{\cK})$ and $\cD^{\cK}_{n}\simeq \Cone(h^1\colon \cD_{n+1}^{\cK}\to \cD_{\infty}^{\cK})$.
\end{rem}
\begin{proof}
Write $\xs_{0}, \xs_1$ for the generators of $\cD_n^{\cK}$ and $\ys_{0}$, $\ys_1$ for the generators of $\cD_{n+1}^{\cK}$. We define $f^1$ via the formula 
\[
f^1(\xs_{\veps})= \ys_{\veps}\otimes (1+\a)
\]
for $\veps\in \{0,1\}$, where
\[
\a=\sum_{s\ge 1} (\scU^s+\scV^s) U^{s(s-1)/2}=(\scU+\scV)+(\scU^2+\scV^2)U+(\scU^3+\scV^3)U^3+\cdots.
\]
(The element $\a$ is discovered by writing $\b=1+\a$ as $\b=\sum_{i\in \Z} \b_i$ where $\b_i$ has Alexander grading $i$, and then recursively solving the equation $\scV \phi^\tau(\b_i)=\b_{i+1}$ starting with $\b_0=1$).

It is straightforward to see that $\d_{\Mor}(f^1)=0$. Write
\[
\Cone(f^1)=
\begin{tikzcd}[column sep=2cm, row sep=1.5cm,labels=description]
\ve{x}_0\ar[r, "1+\a"]
\ar[d, "\sigma+\scV^n \tau"] & \ve{y}_0 \ar[d, "\sigma+\scV^{n+1}\tau"]\\
\ve{x}_1 \ar[r, "1+\a"]& \ve{y}_1
\end{tikzcd}.
\]
Since $1+\a$ is a unit in the completion of $\ve{I}_0\cdot \cK\cdot \ve{I}_0$, the above complex is homotopy equivalent to the bottom row, $\Cone(1+\a: \ve{x}_1\to \ve{y}_1)$. To construct a homotopy equivalence between $\Cone(1+\a: \ve{x}_1\to \ve{y}_1)$ and $\cD_{\infty}^{\cK}$, it is sufficient to show that $1+\a=(1+\scV)u$ for some unit $u$ in the completion of $\ve{I}_1\cdot\cK\cdot \ve{I}_1$. To see this, we recall that $\scU=U \scV^{-1}$. We use this expression for $\scU$ in our formula for $\a$ and then regroup terms based over the powers of $U$, and we obtain the following expression:
\[
1+\a=\sum_{s\ge 1} (\scV^s+\scV^{-s+1}) U^{s(s-1)/2}.
\]
We observe that 
\[
(\scV^s+\scV^{-s+1})U^{s(s-1)/2}=(1+\scV^{2s-1})\scV^{-s+1} U^{s(s-1)/2},
\]
which clearly has a factor of $(1+\scV)$, for all $s\ge 0$. Hence, $1+\a=(1+\scV) (1+Uf)$ for some series $f$, and $(1+U f)$ is a unit in the completion of $\ve{I}_1\cdot \cK \cdot\ve{I}_1$, completing the proof.
\end{proof}

\subsection{Finite generation}
\label{sec:finite-generation}

In this section we prove a basic property of the surgery modules that is useful in applications.

\begin{prop}
\label{prop:finite-generation}
 Suppose that $L\subset S^3$ is a link with framing $\Lambda$, and $L=K_1\cup \cdots \cup K_\ell$. Suppose that $0<j<\ell$ and write $L_{1,\dots, j}=K_1\cup \cdots \cup K_j$ and $L_{j+1,\dots, \ell}=K_{j+1}\cup \cdots \cup K_\ell$. Let $\cX_{\Lambda}(L_{1,\dots, j},L_{j+1,\dots, \ell})^{\cL_{\ell-j}}$ denote the type-$D$ module obtained by tensoring ${}_{\cK} \cD_0$ into each algebra component for $L_{1,\dots, j}$. Then $\cX_{\Lambda}(L_{1,\dots, j},L_{j+1,\dots, \ell})^{\cL_{\ell-j}}$ is homotopy equivalent to a finitely generated type-$D$ module. 
\end{prop}

Our argument is similar to Manolescu and Ozsv\'{a}th's truncation procedure \cite{MOIntegerSurgery}*{Section~10}. We will prove the claim in several steps. If $j<\ell$, we write $\Lambda_{(j)}$ for the induced framing matrix on $L_{1,\dots, j}$. We will first verify the claim when $\Lambda_{(j)}$ is positive definite. We will then use the exact triangle from Section~\ref{sec:exact-triangle} to verify the claim for arbitrary framings. 

If $0<j<\ell$ and we decompose $L$ as $L=L_{1,\dots, j}\cup L_{j+1,\dots, \ell}$, then we may consider the subcube
\[
\cX_{\Lambda}^{(*,0)}(L_{1,\dots, j}, L_{j+1,\dots, \ell})^{\cL_{\ell-j}}
\]
generated by points of the cube $\bE_\ell$ which have $L_{j+1,\dots, \ell}$ components 0. We will view this as a type-$D$ module over the algebra $\bF[\scU_{j+1},\scV_{j+1},\dots, \scU_\ell,\scV_\ell]$ which is continuous with respect to the topology induced by \[
\ve{E}_0\cdot \cL_{\ell-j}\cdot \ve{E}_0\iso \bF[\scU_{j+1},\scV_{j+1},\dots, \scU_\ell,\scV_\ell].
\]

Recall from Section~\ref{sec:framing} that the solid torus module ${}_{\cK} \cD_0$ may be decomposed as a tensor product
\[
{}_{\cK} [\cD_0]^{\bF[U]}\hatbox{}_{\bF[U]} \bF[U].
\]
Hence there is a type-$D$ module $\cX_{\Lambda}(L_{1,\dots, j}, L_{j+1,\dots, \ell})^{\bF[U_1,\dots, U_j]\otimes \cL_{n-j}}$, from which we get $\cX_{\Lambda}(L_{1,\dots, j},L_{j+1,\dots, \ell})^{\cL_{n-j}}$ by tensoring with ${}_{\bF[U_1,\dots, U_j]}\bF[U_1,\dots, U_j]$.

\begin{lem}
\label{lem:finite-generation-positive-definite} Suppose $L=K_1\cup\cdots \cup K_\ell$ is a link in $S^3$, and $0\le j<\ell$. Suppose that $\Lambda_{(j)}$ is positive definite. Then
\[
\cX_{\Lambda}^{(*,0)}(L_{1,\dots, j}, L_{j+1,\dots, \ell})^{\bF[U_1,\dots, U_j,\scU_{j+1},\scV_{j+1},\dots, \scU_\ell,\scV_\ell]}
\]
is homotopy equivalent to a finitely generated type-$D$ module. Furthermore, the maps appearing in the homotopy equivalence may be taken to be continuous if we equip $\cX_{\Lambda}^{(*,0)}(L_{1,\dots, j}, L_{j+1,\dots, \ell})$ with the product topology over a monomial basis, and we equip \[
\bF[U_1,\dots, U_j,\scU_{j+1},\scV_{j+1},\dots, \scU_\ell,\scV_\ell]\]
 with the $(U_1,\dots, U_\ell)$-adic topology.
\end{lem}

\begin{rem} The fact that the morphisms in the homotopy equivalence may be taken to be continuous with respect to the $(U_1,\dots, U_\ell)$-adic topology, as opposed to being continuous with respect to the $(U_1,\dots, U_{j},\scU_{j+1},\scV_{j+1},\dots, \scU_\ell,\scV_\ell)$-adic topology, allows us to use the homotopy equivalence constructed in the lemma for the complex at each subcube $\bE_{j}\times \{\veps\}$, for all $\veps\in \bE_{n-j}$.  
\end{rem}
\begin{proof} 
Our proof will be by induction on $j$. The case that $j=0$ is automatic since $\cX_{\Lambda}^{(*,0)}(\emptyset,L_{1,\dots, \ell})\iso \cCFL(L)$, which is a finitely generated type-$D$ module over $\bF[\scU_1,\scV_1,\dots, \scU_\ell,\scV_\ell]$. 

 We define a shifted Alexander $A_j$-grading on $\cX_{\Lambda}^{(*,0)}(L_{1,\dots, j}, L_{j,\dots, \ell})$, as follows. We set
\begin{equation}
\omega(\ve{s})=-(\lk(K_1,K_j),\dots, \lk(K_{j-1},K_j)) \Lambda_{(j-1)}^{-1} (s_1,\dots, s_{j-1})^T.
\label{eq:def:omega}
\end{equation}
In the above, $\Lambda_{(j-1)}^{-1}$ is the inverse matrix (over $\Q$) and $T$ denotes transpose. Also, if $\ve{s}\in \bH(L)$ we are writing $\ve{s}=(s_1,\dots, s_\ell)$. On the link surgery formula, we define the $\omega$-shifted Alexander grading 
\[
A_j^\omega(\ve{s}):=s_j+\omega(\ve{s}).
\]

We observe that if $i\neq j$, then
\begin{equation}
\omega(\ve{s}+\Lambda_{L,-K_i})-\omega(\ve{s})=-\lk(K_i,K_j), \label{eq:omega-formula-i-neq-j}
\end{equation}
where $\Lambda_{L,-K_i}$ (as defined in Section~\ref{sec:statement-of-link-surgery}) consists of the vector whose $j$-th component is $\lk(K_i,K_j)$ if $i\neq j$, and whose $i$-th component is $\lambda_i$, the framing of $K_i$. Equation~\eqref{eq:omega-formula-i-neq-j} follows from the fact that the $i$-th column of $\Lambda_{(j-1)}$ is obtained by deleting the last component of $\Lambda_{L,-K_i}$. Equation~\eqref{eq:omega-formula-i-neq-j} verifies that if $\vec{M}\subset L_{1,\dots, j}$ is a sublink which does not contain $-K_j$, then $\Phi^{\vec{M}}$ preserves $A_j^\omega$. 

 On the other hand we compute directly from Equation~\eqref{eq:def:omega} and the definition of $\Lambda_{L,-K_j}$ that
\begin{equation}
\begin{split}
&\lambda_j+\omega(\ve{s}+\Lambda_{L,-K_j})-\omega(\ve{s})\\
=&\lambda_j-(\lk(K_1,K_j),\dots, \lk(K_{j-1},K_j)) \Lambda_{(j-1)}^{-1} (\lk(K_1,K_j),\dots, \lk(K_{j-1},K_j))^T.
\end{split}
\label{eq:shift-lambda-j'}
\end{equation}
Note that the above is quantity, which we denote by $\lambda'_j$, is the Schur complement of $\Lambda_{(j-1)}$ inside of $\Lambda_{(j)}$. In particular, $\Lambda_{(j)}$ is positive definite if and only $\Lambda_{(j-1)}$ is positive definite and $\lambda'_j>0$. It follows from Equation~\eqref{eq:shift-lambda-j'} that  if $-K_j\subset \vec{M}\subset L_{1,\dots,j}$, then $\Phi^{\vec{M}}$ shifts $A_j^\omega$ by $ \lambda'_j$.

 We may thus decompose the complex $\cX^{(*,0)}_{\Lambda}(L_{1,\dots, j}, L_{j+1,\dots, \ell})$ as a sum of staircases
\[
\begin{tikzcd}[row sep=1.5cm, column sep=1cm, labels=description]
\cdots
	\ar[dr, "F^{-K_j}"]
 &
\cA_{s-\lambda'_j}^\omega
	\ar[d, "F^{K_j}"]
	\ar[dr, "F^{-K_j}"]
&
\cA_s^\omega
	\ar[d, "F^{K_j}"]
	\ar[dr, "F^{-K_j}"]
&
\cA_{s+\lambda'_j}^\omega
	\ar[d, "F^{K_j}"]
	\ar[dr, "F^{-K_j}"]
&\cdots
\\
\cdots &
\cB_{s-\lambda'_j}^\omega
&
\cB_{s}^\omega&
\cB_{s+\lambda'_j}^\omega
&\cdots
\end{tikzcd}
\]
The vertical direction is the direction of $K_j$. Here we are writing $F^{-K_j}$ for the hypercube maps for all oriented sublinks $\vec{M}\subset L_{1,\dots, j}$ which contain $-K_j$, and we define $F^{K_j}$ similarly. In the above, we are writing $\cA^{\omega}_{s}\subset \cC_{\Lambda}^{(*,0)}(L_{1,\dots, j}, L_{j+1,\dots, \ell})$ for the subspace of $\bF[U_1,\dots, U_j,\scU_{j+1},\scV_{j+1},\dots, \scU_\ell,\scV_\ell]$-module generators with $K_j$-component 0 and Alexander $A_j^\omega$-grading $s\in \Q$. The subspace $\cB^\omega_{s}$ is similar, but has $K_j$-component 1.
Since $\bH(L_{1,\dots, j})/\Lambda_{(j)}$ is finite, there are only finitely many such staircases. Here, the subscript $s$ of $\cA_s^\omega$ and $\cB_s^\omega$ denotes their shifted Alexander grading.

Standard arguments (see \cite{MOIntegerSurgery}*{Lemma~10.1}) imply the following:
\begin{enumerate}[ref=$h$-\arabic*, label=($h$-\arabic*)]
\item\label{h-staircase-1} If $s$ is sufficiently large, the map $F^{K_j}$ is a homotopy equivalence between $\cA_{s}^\omega$ and $\cB_{s}^\omega$.
\item\label{h-staircase-2} If $s$ is sufficiently negative, the map $F^{-K_j}$ is a homotopy equivalence between $\cA^\omega_{s}$ and $\cB_{s+\lambda_j'}^\omega$.
\end{enumerate}

Note that we may view $\cA_s^{\omega}\otimes \bF[U_j]$ as being a subspace of $\cX_{\Lambda}^{(*,0)}(L_{1,\dots, j-1},L_{j+1,\dots, \ell})\otimes \bF[\scU_{j},\scV_j]$ in $A^\omega_j$ Alexander grading $s$. Similarly we may view $\cB_s^{\omega}\otimes \bF[U_j]$ as being the subspace of $\cX_{\Lambda}^{(*,0)}(L_{1,\dots, j-1},L_{j+1,\dots, \ell})\otimes \bF[\scU_{j},\scV_j,\scV_j^{-1}]$ in $A_j^\omega$-Alexander grading $s$. By induction, $\cX_{\Lambda}^{(*,0)}(L_{1,\dots, j-1}, L_{j+1,\dots, \ell})$ is homotopy equivalent to a finitely generated type-$D$ module over $\bF[U_1,\dots, U_{j-1},\scU_{j},\scV_j,\dots, \scU_\ell,\scV_\ell]$ with respect to the topologies in the statement of the lemma.

In particular, we may replace $\cX_{\Lambda}^{(*,0)}(L_{1,\dots, j-1},L_{j,\dots, \ell})$ with a finitely generated model $\cC'$, and also replace the subcomplexes $\cA_s^{\omega}$ and $\cB_s^{\omega}$ with the finite dimensional subspaces of $\cC'\otimes \bF[\scU_{j},\scV_{j}]$ and $\cC'\otimes \bF[\scU_j,\scV_{j},\scV_{j}^{-1}]$ which form an $\bF[U_j]$ basis in Alexander grading $s$. Properties \eqref{h-staircase-1} and~\eqref{h-staircase-2} are preserved by homotopy equivalence, and hence still persist on the finitely generated models. Abusing notation, we write $\cA_s^\omega$ and $\cB_s^{\omega}$ also for the finitely generated models.

Using the same logic as for the ordinary knot surgery mapping cone formula \cite{OSIntegerSurgeries}, each such staircase above admits a finitely generated truncation. We need additionally to show that the maps appearing in a homotopy equivalence with the finite truncation are continuous with respect to the topologies in the main statement. These maps may be described easily using homological perturbation theory. Compare the proof of Lemma~\ref{lem:Q-1}. Since $\cC'$ is finitely generated, for $s\gg 0$, each generator of $\cA_s^\omega$ must be a multiple of $\scV_j^{r(s)}$ for some function $r\colon \Q\to \N$ such that $r(s)\to \infty$ as $s\to \infty$. Similarly for $s\ll 0$, each generator of $\cA_s^\omega$ must be a multiple of $\scU_j^{l(s)}$ for some $l\colon \Q\to \N$ such that $l(s)\to \infty$ as $s\to -\infty$. Using the equivariance of the maps $F^{K_j}$ and $F^{-K_j}$ from Lemma~\ref{lem:module-structure-surgery-hypercube-morphisms}, we see that for $s\ll 0$, the map $F^{K_j}$ sends $\cA_s^{\omega}$ into the ideal $(U_j^{l(s)})$. Similarly, for $s \gg 0$, $F^{-K_j}$ maps $\cA_s^\omega$ into the ideal $(U_j^{r(s)})$. This is sufficient to show that the maps appearing in the homotopy equivalence are continuous as stated.

Since $\cX_{\Lambda}^{(*,0)}(L_{1,\dots,j}, L_{j+1,\dots, j})$ decomposes as a sum of finitely many staircases, each of which is homotopy equivalent to a finitely generated truncation, as above, the direct sum is homotopy equivalent to a finitely generated complex, completing the proof.
\end{proof}

We now prove a basic lemma concerning homological algebra.

\begin{lem}
\label{lem:reduce-homotopic-U-actions} Write $\cR=\bF[U_1,\dots, U_{j-1}, \scU_{j+1},\scV_{j+1},\dots, \scU_\ell,\scV_\ell]$, where $0<j<\ell$. Write $\cR[U_j]$ for $\cR\otimes \bF[U_j]$. Suppose that $\cX^{\cR[U_j]}$
is a finitely generated type-$D$ module which admits a relative $\Q\times \Q$-valued $(\gr_{\ws},\gr_{\zs})$-bigrading. Suppose that $U_j\simeq \scU_\ell\scV_\ell$ as type-$D$ morphisms (i.e. there is a morphism $J^1$ satisfying $\d_{\Mor}(J^1)=\id\otimes (U_j+\scU_\ell\scV_\ell)$). Then
\[
\left(\cX^{\cR[U_j]}\boxtimes {}_{\bF[U_j]} \bF[U_j]\right){}^{\cR}
\]
 is homotopy equivalent to a finitely generated type-$D$ module over $\cR$. 
\end{lem}
\begin{rem} The condition on the $(\gr_{\ws},\gr_{\zs})$-bigrading is likely not absolutely necessary, though it simplifies the discussion on completions.
\end{rem}
\begin{proof} Our proof uses homological perturbation theory. We view $\cX^{\cR[U_j]}$ as a being obtained from a $DA$-bimodule ${}_{\cE}\cW^{\cR[U_j]}$ where $\cE$ denotes the exterior algebra on one generator. We define $\delta_{k+1}^1(\theta,\dots, \theta, \xs)$ on $\cW$ to consist of the components of $\delta^1$ of $\cX$ which are weighted by an algebra element of total $\cR$-degree equal to $k$. Here, the $\cR$ degree of a monomial denotes the sum of powers of variables from $\cR$ appearing (but not of powers of $U_j$). It is straightforward to see that $\cW$ satisfies the $DA$-bimodule structure relations. 

Note that $\delta_1^1$ on $\cW$ consists of exactly the terms of the differential of $\cX$ which are weighted by elements of $\bF[U_j]$.

Let $N^\cE$ be the rank 1 type-$D$ module with $\delta^1(1)=1\otimes \theta$, so that 
\[
N^\cE\boxtimes {}_{\cE}\cW^{\cR[U_j]}=\cX^{\cR[U_j]}.
\]

We consider the type-$D$ module $\cY^{\cR[U_j]}$ obtained by applying the forgetful functor to
\[
{}_{\cE} \cW^{\cR[U_j]}.
\]
The type-$D$ module $\cY^{\cR[U_j]}$ has $\delta^1$ with algebra elements in $\bF[U_j]$, i.e., we may view it as a finitely generated free chain complex over $\bF[U_j]$. Setting all variables except $U_j$ equal to 0, the map $J^1$ induces a map $h^1$ on $\cY^{\cR[U_j]}$ which satisfies $\d_{\Mor}(j^1)=\id\otimes U_j$. The classification theorem for finitely generated chain complexes over a PID implies that we may find a basis of $\cY$ so that $\cY^{\cR[U_j]}$ decomposes as a direct sum of 1-step complexes (i.e. generators with vanishing $\delta^1$), and 2-step complexes (i.e. summands generated by two generators with $\delta^1(\ys)=0$ and $\delta^1(\xs)=\ys\otimes \a$). Our assumptions about the existence of a relative bigrading imply that $\a$ may always be taken to be a power of $U_j$. The fact that $\d_{\Mor}(j^1)=\id \otimes U_j$ implies that there are no 1-step summands, and each 2-step summand is of the form $\delta^1(\xs)=\ys\otimes U_j$.

We define the finite dimensional subspace $\cQ\subset \cW\boxtimes \bF[U_j]$ as follows. Enumerate the generators of the $U_j^1$-weighted 2-step complexes as $\xs_t$ and $\ys_t$, so that $\delta^1(\xs_t)=\ys_t\otimes U_j$, for $t$ in some finite set.
We define $\cQ$ to be the $\bF$ span of $\ys_t\otimes 1$, ranging over all $t$. We view $\cQ$ as being a finitely generated type-$D$ module over $\cR$ with vanishing differential.  There are maps forming a homotopy equivalence of type-$D$ modules
\[
\pi^1\colon \left(\cW\boxtimes \bF[U_j]\right)^{\cR}\to \cQ^\cR,\qquad  i^1\colon \cQ^\cR\to\left(\cW\boxtimes \bF[U_j]\right)^{\cR},
\]
\[
\text{and} \quad h^1\colon \left(\cW\boxtimes \bF[U_j]\right)^{\cR}\to \left(\cW\boxtimes \bF[U_j]\right)^{\cR},
\]
which satisfy the assumptions of the homological perturbation lemma stated in Lemma~\ref{lem:homological-perturbation-DA-modules}. 
The map $\pi^1$ sends $\ys_t\otimes U_j^s$ to $\ys_t\otimes 1$ if $s=0$, and is zero on all generators. The map $i^1$ sends $\ys_t$ to $(\ys_t\otimes 1)\otimes 1$. The map $h^1$ sends $\ys_t\otimes U_j^s$ to $(\ys_t\otimes U_j^{s-1})\otimes 1$ if $s>0$, and $h^1$ vanishes on all other generators. Applying the homological perturbation lemma endows $\cQ^{\cR}$ with a left action of $\cE$. And also yields a homotopy equivalence ${}_{\cE} \cQ^{\cR}\simeq {}_{\cE} \cW^{\cR[U_j]}\boxtimes {}_{\bF[U_j]}\bF[U_j]$.
 Note that the induced $A_\infty$-module structure is operationally bounded by grading considerations. Finally, boxing with the identity morphism of $N^{\cE}$ gives a homotopy equivalence as in the statement.
\end{proof}

We are now able to prove finite generation in general:

\begin{proof}[Proof of Proposition~\ref{prop:finite-generation}]
Suppose that $L=K_1\cup \cdots \cup K_\ell$ is a link in $S^3$, and $\Lambda$ is an integral framing on $L$. Let $0\le j<\ell$. It is sufficient to show that
\[
\cX^{(*,0)}_{\Lambda}(L_{1,\dots, j}, L_{j+1,\dots, \ell})^{\bF[\scU_{j+1},\scV_{j+1},\dots, \scU_{\ell}, \scV_\ell]}
\]
is homotopy equivalent to a finitely generated type-$D$ module over $\bF[\scU_{j+1},\scV_{j+1},\dots, \scU_{\ell}, \scV_\ell]$. This is sufficient since we may use this model at each point of the cube $\bE_{n-j}$ to obtain a finitely generated model of $\cX_{\Lambda}(L_{1,\dots, j}, L_{j+1,\dots, \ell})^{\cL_{n-j}}$.

By Proposition~\ref{prop:equivalence-type-D-change-basepoint}, changing the arc system on the components $K_1,\dots,  K_j$ does not change the homotopy type of $\cX_{\Lambda}(L_{1,\dots, j}, L_{j+1,\dots, \ell})^{\cL_{n-j}}$. Hence, we assume all of the arcs for $K_1,\dots, K_j$ are alpha-parallel.

Lemma~\ref{lem:finite-generation-positive-definite} verifies that if $\Lambda_{(j)}$ is positive definite, then
\[
\cX^{(*,0)}_{\Lambda}(L_{1,\dots, j}, L_{j+1,\dots, \ell})^{\bF[U_1,\dots, U_j,\scU_{j+1},\scV_{j+1},\dots, \scU_{\ell}, \scV_\ell]}
\]
is homotopy equivalent to a finitely generated type-$D$ module. Furthermore, since $\Lambda_{(j)}$ is positive definite (in particular non-singular), it is straightforward to see that this complex admits a relative $\Q\times \Q$-valued $(\gr_{\ws},\gr_{\zs})$-bigrading (compare \cite{MOIntegerSurgery}*{Section~9.3}).
Lemma~\ref{lem:U-actions-homotopic} implies that each $U_1,\dots, U_j$ is chain homotopic to $\scU_{\ell}\scV_\ell$ as a type-$D$ endomorphism on this subcube. (Note that on this subcube, the $K_\ell$-direction is never incremented, so the arc for $K_\ell$ is irrelevant). Lemma~\ref{lem:reduce-homotopic-U-actions} now implies that $\cX_{\Lambda}^{(*,0)}(L_{1,\dots, j}, L_{j+1,\dots, \ell})^{\bF[\scU_{j+1},\scV_{j+1},\dots, \scU_{\ell},\scV_\ell]}$ is homotopy equivalent to a finitely generated type-$D$ module.

Since ${}_{\cK} \cD_\infty^{\bF[U]}$ is homotopy equivalent to a finitely generated type-$DA$ module by Lemma~\ref{lem:infinity-framed-type-A}, we observe that the claim still holds if we replace any number of framings in a positive definite $\Lambda_{(j)}$ with $+\infty$. The mapping cone of two finitely generated complexes is finitely generated, so the surgery exact triangle from Proposition~\ref{prop:exact-triangle} (cf. Remark~\ref{rem:cyclic-reorder}), implies that $\cX_{\Lambda}(L_{1,\dots, j}, L_{j+1,\dots, \ell})^{\cL_{n-j}}$ is homotopy equivalent to finitely generated type-$D$ module for any framing, concluding the proof.
 \end{proof}

\appendix
\section{Splicing operations}

\label{appendix}
In this section, we describe some topology related to splicing knot complements. Everything in this section is well-known, and we only include it as a reference.

\subsection{Splices and connected sums}

Suppose that $K_1$ and $K_2$ are knots in $S^3$. The complements of $K_1$ and $K_2$ are manifolds with torus boundary. The following is well known, though we give a proof for the convenience of the reader. Compare \cite{GordonSatellite}*{Section~7}.

\begin{lem} Suppose that $\lambda_1$ and $\lambda_2$ are integral framings on knots $K_1$ and $K_2$ in $S^3$, and let $[\lambda_1]$ and $[\lambda_2]$ denote the induced elements of $H_1(\d S^3\setminus N(K_1))$ and $H_1(\d S^3\setminus N(K_2))$. Then $S^3_{\lambda_1+\lambda_2}(K_1\# K_2)$ is diffeomorphic to the 3-manifold obtained by gluing $S^3\setminus N(K_1)$ to $S^3\setminus N(K_2)$ using the diffeomorphism which sends
\[
[\lambda_1]\mapsto -[\lambda_2]\quad \text{and} \quad [\mu_1]\mapsto [\mu_2].
\]
\end{lem}
\begin{proof} Consider the manifold $Z$ obtained by gluing the complements of $K_1$ and $K_2$ as described above. We may describe $Z$ alternatively by first attaching a 3-dimensional 1-handle to connect $S^3\setminus \nu(K_1)$ and $S^3\setminus \nu(K_2)$, and then attaching two 2-handles and one 3-handle. The first 2-handle is attached along $\mu_1*\bar{\mu}_2$ (where $\bar{\mu}_2$ denotes $\mu_2$ with orientation reverses), and the second 2-handle is attached along $\lambda_1* \lambda_2$. In both cases, the curves are concatenated across the 1-handle.

We observe that joining $S^3\setminus \nu(K_1)$ and $S^2\setminus \nu(K_2)$ with the 1-handle and the 2-handle tracing $\mu_1* \bar \mu_2$ is the same as gluing $S^3\setminus \nu(K_1)$ and $S^3\setminus \nu(K_2)$ along meridians of $K_1$ and $K_2$. This gives $S^3\setminus \nu(K_1\# K_2)$. Gluing the second 2-handle and then the 3-handle is the same as performing Dehn surgery on $K_1\# K_2$ with framing $\lambda_1+\lambda_2$. 
\end{proof}

\begin{rem}
 The argument above also holds for arbitrary knots in 3-manifolds, as long as the framings are Morse.
\end{rem}

\begin{rem}\label{rem:connected-sum-links}
 The same argument works for links. Suppose that $L_1$ and $L_2$ are links in $S^3$ with framings $\Lambda_1$ and $\Lambda_2$. Suppose that $K_1$ and $K_2$ are two chosen components of $L_1$ and $L_2$, respectively. Then
\[
S^3_{\Lambda_1+\Lambda_2}(L_1\# L_2)\iso M_{\Lambda_1}(L_1,K_1)\cup_{\phi} M_{\Lambda_2}(L_2,K_2),
\]
where $M_{\Lambda_i}(L_i,K_i)$ is the manifold obtained by surgering along $L_i- K_i$, and removing a neighborhood of $K_i$. The diffeomorphism $\phi$ which identifies $\d N(K_1)$ with $\d N(K_2)$ sends $\lambda_1$ to $-\lambda_2$, and $\mu_1$ to $\mu_2$. The framing $\Lambda_1+\Lambda_2$ is obtained by summing the framings for $K_1$ and $K_2$, and leaving the remaining framings unchanged.
\end{rem}

\subsection{The Hopf link and changes of parametrization}

We may naturally view the Hopf link as a link which changes the boundary parametrization. There are two ways to do this:

\begin{enumerate}
\item (Blow up) If $K$ is the special component, and we connect sum with a Hopf link, and leave $K$ the special component. If we give the other component of the Hopf link framing $\pm 1$, then we change the framing on $\d N(K)$ by a Dehn twist along the meridian.
\item (Rotating layer) Noting that the complement of the Hopf link is $\bT^2\times [0,1]$, we may use the description from the previous remark to see the effect of connect summing the Hopf link to the special component, and using the new Hopf link component as the special component. The effect is to change the boundary parametrization by composing with one of the two maps:
\[
\mu\mapsto \lambda\quad \text{and} \quad \lambda\mapsto -\mu, \quad \text{ or }
\]
\[
\mu \mapsto -\lambda \quad \text{and} \quad \lambda\mapsto \mu.
\]
These two cases are distinguished by the sign of the Hopf link which is used in the connected sum.
\end{enumerate}

\bibliographystyle{custom}
\def\MR#1{}
\bibliography{biblio}

\begin{bibdiv}
\begin{biblist}

\bib{AtiyahMacdonald}{book}{
      author={Atiyah, Michael},
      author={Macdonald, Ian},
       title={Introduction to commutative algebra},
   publisher={Addison-Wesley Publishing Co., Reading, Mass.-London-Don Mills,
  Ont.},
        date={1969},
      review={\MR{0242802}},
}

\bib{Beilinson-Tensors}{article}{
      author={Beilinson, Alexander},
       title={Remarks on topological algebras},
        date={2008},
        ISSN={1609-3321},
     journal={Mosc. Math. J.},
      volume={8},
      number={1},
       pages={1\ndash 20, 183},
      review={\MR{2422264}},
}

\bib{CZZSatellites}{unpublished}{
      author={Chen, Daren},
      author={Zemke, Ian},
      author={Zhou, Hugo},
       title={L-space satellite operations and knot {F}loer homology},
        date={2024},
        note={e-print, arXiv 2412.05755 [math.GT]},
}

\bib{EftekharyDuals}{unpublished}{
      author={Eftekhary, Eaman},
       title={{H}eegaard {F}loer homology and {M}orse surgery},
        date={2006},
        note={e-print, {arXiv:0603171} [math.GT]},
}

\bib{EftekharySplices}{article}{
      author={Eftekhary, Eaman},
       title={Floer homology and splicing knot complements},
        date={2015},
        ISSN={1472-2747},
     journal={Algebr. Geom. Topol.},
      volume={15},
      number={6},
       pages={3155\ndash 3213},
         url={https://doi.org/10.2140/agt.2015.15.3155},
      review={\MR{3450759}},
}

\bib{EisenbudCommutativeAlgebra}{book}{
      author={Eisenbud, David},
       title={Commutative algebra},
      series={Graduate Texts in Mathematics},
   publisher={Springer-Verlag, New York},
        date={1995},
      volume={150},
        ISBN={0-387-94268-8; 0-387-94269-6},
         url={https://doi.org/10.1007/978-1-4612-5350-1},
        note={With a view toward algebraic geometry},
      review={\MR{1322960}},
}

\bib{GordonSatellite}{article}{
      author={Gordon, C.~McA.},
       title={Dehn surgery and satellite knots},
        date={1983},
        ISSN={0002-9947},
     journal={Trans. Amer. Math. Soc.},
      volume={275},
      number={2},
       pages={687\ndash 708},
         url={https://doi.org/10.2307/1999046},
      review={\MR{682725}},
}

\bib{GrothendieckNuclear}{incollection}{
      author={Grothendieck, Alexander},
       title={Produits tensoriels topologiques et espaces nucl\'{e}aires},
        date={1954},
   booktitle={S\'{e}minaire {B}ourbaki, {V}ol. 2},
   publisher={Soc. Math. France, Paris},
       pages={Exp. No. 69, 193\ndash 200},
}

\bib{HanselmanSplices}{article}{
      author={Hanselman, Jonathan},
       title={Splicing integer framed knot complements and bordered {H}eegaard
  {F}loer homology},
        date={2017},
        ISSN={1663-487X},
     journal={Quantum Topol.},
      volume={8},
      number={4},
       pages={715\ndash 748},
}

\bib{HHSZExact}{unpublished}{
      author={Hendricks, Kristen},
      author={Hom, Jennifer},
      author={Stoffregen, Matthew},
      author={Zemke, Ian},
       title={Surgery exact triangles and involutive {H}eegaard {F}loer
  homology},
        date={2020},
        note={e-print, {arXiv:2011.00113} [math.GT]},
}

\bib{HHSZNaturality}{unpublished}{
      author={Hendricks, Kristen},
      author={Hom, Jennifer},
      author={Stoffregen, Matthew},
      author={Zemke, Ian},
       title={Naturality and functoriality in involutive {H}eegaard {F}loer
  homology},
        date={2021},
        note={e-print, \url{2201.12906} [math.GT]},
}

\bib{HHSZDuals}{unpublished}{
      author={Hendricks, Kristen},
      author={Hom, Jennifer},
      author={Stoffregen, Matthew},
      author={Zemke, Ian},
       title={A dual knot mapping cone formula for involutive {H}eegaard
  {F}loer homology},
        date={2022},
        note={e-print, \url{2205.12798} [math.GT]},
}

\bib{HK_Homological_Perturbation}{article}{
      author={Huebschmann, Johannes},
      author={Kadeishvili, Tornike},
       title={Small models for chain algebras},
        date={1991},
        ISSN={0025-5874},
     journal={Math. Z.},
      volume={207},
      number={2},
       pages={245\ndash 280},
         url={https://doi.org/10.1007/BF02571387},
      review={\MR{1109665}},
}

\bib{HLSplices}{article}{
      author={Hedden, Matthew},
      author={Levine, Adam~Simon},
       title={Splicing knot complements and bordered {F}loer homology},
        date={2016},
        ISSN={0075-4102},
     journal={J. Reine Angew. Math.},
      volume={720},
       pages={129\ndash 154},
}

\bib{HeddenLevineSurgery}{unpublished}{
      author={Hedden, Matthew},
      author={Levine, Adam~Simon},
       title={A surgery formula for knot {F}loer homology},
        date={2019},
        note={e-print, {arXiv:1901.02488} [math.GT]},
}

\bib{HMInvolutive}{article}{
      author={Hendricks, Kristen},
      author={Manolescu, Ciprian},
       title={Involutive {H}eegaard {F}loer homology},
        date={2017},
     journal={Duke Math. J.},
      volume={166},
      number={7},
       pages={1211\ndash 1299},
}

\bib{Kadeishvili_Ainfinity}{article}{
      author={Kadeishvili, T.~V.},
       title={The algebraic structure in the homology of an {$A(\infty
  )$}-algebra},
        date={1982},
        ISSN={0132-1447},
     journal={Soobshch. Akad. Nauk Gruzin. SSR},
      volume={108},
      number={2},
       pages={249\ndash 252 (1983)},
      review={\MR{720689}},
}

\bib{KellerNotes}{article}{
      author={Keller, Bernhard},
       title={Introduction to {$A$}-infinity algebras and modules},
        date={2001},
        ISSN={1532-0081},
     journal={Homology Homotopy Appl.},
      volume={3},
      number={1},
       pages={1\ndash 35},
}

\bib{KellerNotesAddendum}{article}{
      author={Keller, Bernhard},
       title={Addendum to: ``{I}ntroduction to {$A$}-infinity algebras and
  modules''},
        date={2002},
        ISSN={1532-0081},
     journal={Homology Homotopy Appl.},
      volume={4},
      number={1},
       pages={25\ndash 28},
         url={https://doi.org/10.4310/hha.2002.v4.n1.a2},
      review={\MR{1905779}},
}

\bib{KLTSplicing}{article}{
      author={Karakurt, \c{C}a\u{g}ri},
      author={Lidman, Tye},
      author={Tweedy, Eamonn},
       title={Heegaard {F}loer homology and splicing homology spheres},
        date={2021},
        ISSN={1073-2780},
     journal={Math. Res. Lett.},
      volume={28},
      number={1},
       pages={93\ndash 106},
         url={https://doi.org/10.4310/MRL.2021.v28.n1.a4},
      review={\MR{4247996}},
}

\bib{KontsevichSoibelman}{incollection}{
      author={Kontsevich, Maxim},
      author={Soibelman, Yan},
       title={Homological mirror symmetry and torus fibrations},
        date={2001},
   booktitle={Symplectic geometry and mirror symmetry ({S}eoul, 2000)},
   publisher={World Sci. Publ., River Edge, NJ},
       pages={203\ndash 263},
         url={https://doi.org/10.1142/9789812799821_0007},
      review={\MR{1882331}},
}

\bib{LefschetzAT}{book}{
      author={Lefschetz, Solomon},
       title={Algebraic {T}opology},
      series={American Mathematical Society Colloquium Publications, Vol. 27},
   publisher={American Mathematical Society, New York},
        date={1942},
      review={\MR{0007093}},
}

\bib{LipshitzCylindrical}{article}{
      author={Lipshitz, Robert},
       title={A cylindrical reformulation of {H}eegaard {F}loer homology},
        date={2006},
     journal={Geom. Topol.},
      volume={10},
       pages={955\ndash 1097},
}

\bib{Liu2Bridge}{unpublished}{
      author={Liu, Yajing},
       title={{H}eegaard {F}loer homology of surgeries on two-bridge links},
        date={2014},
        note={e-print, {arXiv:1402.5727} [math.GT]},
}

\bib{Loday}{article}{
      author={Loday, Jean-Louis},
       title={The diagonal of the {S}tasheff polytope},
        date={2011},
      volume={287},
       pages={269\ndash 292},
}

\bib{LOTDoubleBranchedI}{article}{
      author={Lipshitz, Robert},
      author={Ozsv\'{a}th, Peter~S.},
      author={Thurston, Dylan~P.},
       title={Bordered {F}loer homology and the spectral sequence of a branched
  double cover {I}},
        date={2014},
        ISSN={1753-8416},
     journal={J. Topol.},
      volume={7},
      number={4},
       pages={1155\ndash 1199},
         url={https://doi.org/10.1112/jtopol/jtu012},
      review={\MR{3286900}},
}

\bib{LOTBimodules}{article}{
      author={Lipshitz, Robert},
      author={Ozsv\'{a}th, Peter~S.},
      author={Thurston, Dylan~P.},
       title={Bimodules in bordered {H}eegaard {F}loer homology},
        date={2015},
        ISSN={1465-3060},
     journal={Geom. Topol.},
      volume={19},
      number={2},
       pages={525\ndash 724},
         url={https://doi.org/10.2140/gt.2015.19.525},
      review={\MR{3336273}},
}

\bib{LOTDoubleBranchedII}{article}{
      author={Lipshitz, Robert},
      author={Ozsv\'{a}th, Peter~S.},
      author={Thurston, Dylan~P.},
       title={Bordered {F}loer homology and the spectral sequence of a branched
  double cover {II}: the spectral sequences agree},
        date={2016},
        ISSN={1753-8416},
     journal={J. Topol.},
      volume={9},
      number={2},
       pages={607\ndash 686},
         url={https://doi.org/10.1112/jtopol/jtw003},
      review={\MR{3509974}},
}

\bib{LOTBordered}{article}{
      author={Lipshitz, Robert},
      author={Ozsvath, Peter~S.},
      author={Thurston, Dylan~P.},
       title={Bordered {H}eegaard {F}loer homology},
        date={2018},
        ISSN={0065-9266},
     journal={Mem. Amer. Math. Soc.},
      volume={254},
      number={1216},
       pages={viii+279},
         url={https://doi.org/10.1090/memo/1216},
      review={\MR{3827056}},
}

\bib{LOTDiagonals}{unpublished}{
      author={Lipshitz, Robert},
      author={Ozsv\'{a}th, Peter},
      author={Thurston, Dylan},
       title={Diagonals and {A}-infinity tensor products},
        date={2020},
        note={e-print, {arXiv:2009.05222} [math.GT]},
}

\bib{LipshitzSarkarSplit}{unpublished}{
      author={Lipshitz, Robert},
      author={Sarkar, Sucharit},
       title={Khovanov homology detects split links},
        date={2019},
        note={e-print, {arXiv:1910.04246} [math.GT]},
}

\bib{MOIntegerSurgery}{unpublished}{
      author={Manolescu, Ciprian},
      author={Ozsv{\'a}th, Peter~S.},
       title={Heegaard {F}loer homology and integer surgeries on links},
        date={2010},
        note={e-print, {arXiv:1011.1317} [math.GT]},
}

\bib{MarklShnider}{article}{
      author={Markl, Martin},
      author={Shnider, Steve},
       title={Associahedra, cellular {$W$}-construction and products of
  {$A_\infty$}-algebras},
        date={2006},
        ISSN={0002-9947},
     journal={Trans. Amer. Math. Soc.},
      volume={358},
      number={6},
       pages={2353\ndash 2372},
}

\bib{NemethiAR}{article}{
      author={N\'{e}methi, Andr\'{a}s},
       title={On the {O}zsv\'{a}th-{S}zab\'{o} invariant of negative definite
  plumbed 3-manifolds},
        date={2005},
        ISSN={1465-3060},
     journal={Geom. Topol.},
      volume={9},
       pages={991\ndash 1042},
         url={https://doi-org.prx.library.gatech.edu/10.2140/gt.2005.9.991},
      review={\MR{2140997}},
}

\bib{NemethiLattice}{article}{
      author={N\'{e}methi, Andr\'{a}s},
       title={Lattice cohomology of normal surface singularities},
        date={2008},
        ISSN={0034-5318},
     journal={Publ. Res. Inst. Math. Sci.},
      volume={44},
      number={2},
       pages={507\ndash 543},
         url={https://doi.org/10.2977/prims/1210167336},
      review={\MR{2426357}},
}

\bib{Ni-homological}{article}{
      author={Ni, Yi},
       title={Homological actions on sutured {F}loer homology},
        date={2014},
     journal={Math. Res. Lett.},
      volume={21},
      number={5},
       pages={1177\ndash 1197},
}

\bib{OSKnots}{article}{
      author={Ozsv\'ath, Peter},
      author={Szab\'o, Zolt\'an},
       title={Holomorphic disks and knot invariants},
        date={2004},
     journal={Adv. Math.},
      volume={186},
      number={1},
       pages={58\ndash 116},
}

\bib{OSDisks}{article}{
      author={Ozsv\'ath, Peter},
      author={Szab\'o, Zolt\'an},
       title={Holomorphic disks and topological invariants for closed
  three-manifolds},
        date={2004},
     journal={Ann. of Math. (2)},
      volume={159},
      number={3},
       pages={1027\ndash 1158},
}

\bib{OSProperties}{article}{
      author={Ozsv{\'a}th, Peter~S.},
      author={Szab{\'o}, Zolt{\'a}n},
       title={Holomorphic disks and three-manifold invariants: properties and
  applications},
        date={2004},
     journal={Ann. of Math. (2)},
      volume={159},
      number={3},
       pages={1159\ndash 1245},
}

\bib{OSTriangles}{article}{
      author={Ozsv\'ath, Peter},
      author={Szab\'o, Zolt\'an},
       title={Holomorphic triangles and invariants for smooth four-manifolds},
        date={2006},
     journal={Adv. Math.},
      volume={202},
      number={2},
       pages={326\ndash 400},
}

\bib{OSLinks}{article}{
      author={Ozsv\'ath, Peter},
      author={Szab\'o, Zolt\'an},
       title={Holomorphic disks, link invariants and the multi-variable
  {A}lexander polynomial},
        date={2008},
     journal={Algebr. Geom. Topol.},
      volume={8},
      number={2},
       pages={615\ndash 692},
}

\bib{OSIntegerSurgeries}{article}{
      author={Ozsv\'{a}th, Peter~S.},
      author={Szab\'{o}, Zolt\'{a}n},
       title={Knot {F}loer homology and integer surgeries},
        date={2008},
        ISSN={1472-2747},
     journal={Algebr. Geom. Topol.},
      volume={8},
      number={1},
       pages={101\ndash 153},
         url={https://doi.org/10.2140/agt.2008.8.101},
      review={\MR{2377279}},
}

\bib{OSRationalSurgeries}{article}{
      author={Ozsv\'ath, Peter},
      author={Szab\'o, Zolt\'an},
       title={Knot {F}loer homology and rational surgeries},
        date={2011},
     journal={Algebr. Geom. Topol.},
      volume={11},
      number={1},
       pages={1\ndash 68},
}

\bib{OSBorderedKauffmanStates}{article}{
      author={Ozsv\'{a}th, Peter},
      author={Szab\'{o}, Zolt\'{a}n},
       title={Kauffman states, bordered algebras, and a bigraded knot
  invariant},
        date={2018},
        ISSN={0001-8708},
     journal={Adv. Math.},
      volume={328},
       pages={1088\ndash 1198},
         url={https://doi.org/10.1016/j.aim.2018.02.017},
      review={\MR{3771149}},
}

\bib{OSBorderedHFK}{unpublished}{
      author={Ozsv\'{a}th, Peter},
      author={Szab\'{o}, Zolt\'{a}n},
       title={Algebras with matchings and knot {F}loer homology},
        date={2019},
        note={e-print {arXiv:1912.01657} [math.GT]},
}

\bib{Positelski-Linear}{article}{
      author={Positelski, Leonid},
       title={Exact categories ot topological vector spaces with linear
  topology},
        date={2020},
        note={e-print, arXiv:2012.15431 [math.CT]},
}

\bib{RasmussenKnots}{thesis}{
      author={Rasmussen, Jacob},
       title={{F}loer homology and knot complements},
        type={Ph.D. Thesis},
        date={2003},
        note={{arXiv:math/0306378}},
}

\bib{SarkarMovingBasepoints}{article}{
      author={Sarkar, Sucharit},
       title={Moving basepoints and the induced automorphisms of link {F}loer
  homology},
        date={2015},
     journal={Algebr. Geom. Topol.},
      volume={15},
      number={5},
       pages={2479\ndash 2515},
}

\bib{SeidelFukaya}{book}{
      author={Seidel, Paul},
       title={Fukaya categories and {P}icard-{L}efschetz theory},
      series={Zurich Lectures in Advanced Mathematics},
   publisher={European Mathematical Society (EMS), Z\"{u}rich},
        date={2008},
        ISBN={978-3-03719-063-0},
         url={https://doi.org/10.4171/063},
      review={\MR{2441780}},
}

\bib{StasheffKn}{book}{
      author={Stasheff, James~Dillon},
       title={Homotopy associativity of {H}-spaces},
   publisher={ProQuest LLC, Ann Arbor, MI},
        date={1961},
        note={Thesis (Ph.D.)--Princeton University},
}

\bib{SaneblidzeUmble}{article}{
      author={Saneblidze, Samson},
      author={Umble, Ronald},
       title={Diagonals on the permutahedra, multiplihedra and associahedra},
        date={2004},
        ISSN={1532-0081},
     journal={Homology Homotopy Appl.},
      volume={6},
      number={1},
       pages={363\ndash 411},
}

\bib{ZemGraphTQFT}{unpublished}{
      author={Zemke, Ian},
       title={Graph cobordisms and {H}eegaard {F}loer homology},
        date={2015},
        note={e-print, {arXiv:1512.01184} [math.GT]},
}

\bib{ZemQuasi}{article}{
      author={Zemke, Ian},
       title={Quasistabilization and basepoint moving maps in link {F}loer
  homology},
        date={2017},
        ISSN={1472-2747},
     journal={Algebr. Geom. Topol.},
      volume={17},
      number={6},
       pages={3461\ndash 3518},
         url={https://doi.org/10.2140/agt.2017.17.3461},
      review={\MR{3709653}},
}

\bib{ZemAbsoluteGradings}{article}{
      author={Zemke, Ian},
       title={Link cobordisms and absolute gradings on link {F}loer homology},
        date={2019},
        ISSN={1663-487X},
     journal={Quantum Topol.},
      volume={10},
      number={2},
       pages={207\ndash 323},
}

\bib{ZemCFLTQFT}{article}{
      author={Zemke, Ian},
       title={Link cobordisms and functoriality in link {F}loer homology},
        date={2019},
        ISSN={1753-8416},
     journal={J. Topol.},
      volume={12},
      number={1},
       pages={94\ndash 220},
         url={https://doi.org/10.1112/topo.12085},
      review={\MR{3905679}},
}

\bib{ZemLattice}{unpublished}{
      author={Zemke, Ian},
       title={The equivalence of lattice and {H}eegaard-{F}loer homology},
        date={2021},
        note={e-print, arXiv 2111.14962 [math.GT]},
}

\bib{ZemExactTriangle}{unpublished}{
      author={Zemke, Ian},
       title={A general {H}eegaard {F}loer surgery formula},
        date={2023},
        note={e-print, arXiv 2308.15658 [math.GT]},
}

\end{biblist}
\end{bibdiv}
\end{document}